\documentclass[a4paper,10pt,onecolumn]{amsbook}
\pdfoutput=1
\usepackage[utf8]{inputenc}
\usepackage[english]{babel}
\usepackage[T1]{fontenc}
\usepackage[svgnames,rgb]{xcolor}

\usepackage{svg}
\usepackage{amsfonts}
\usepackage{amssymb}
\usepackage{amsmath}
\usepackage{amsthm}
\usepackage{mathcomp}
\usepackage{textcomp}
\usepackage{bm}
\usepackage{cancel}
\usepackage{mathtools, cuted}
\usepackage[margin=1cm]{caption}
\allowdisplaybreaks{}
\usepackage{tikz}
\usetikzlibrary{matrix, arrows, shapes}
\usepackage{tikz-cd}
\usepackage{enumitem}
\setlist[itemize]{leftmargin=*}
\usepackage[draft,backgroundcolor=red!5,textsize=small,linecolor=red!30]{todonotes} %
\usepackage{unicode-mario}
\usepackage{mathpartir}
\usepackage[official]{eurosym}
 \usepackage{unicode-mario}

\newcommand{\biobj}[2]{{\protect\substack{ #1 \\ #2 }}}

\usepackage[pagebackref,breaklinks]{hyperref}
\definecolor{nordred}{HTML}{bf616a}
\definecolor{nordblue}{HTML}{81a1c1}
\definecolor{norddarkblue}{HTML}{5e81ac}
\definecolor{nordgreen}{HTML}{a3be8c}
\definecolor{nordnight}{HTML}{4c566a}
\definecolor{mygrey}{HTML}{9daaa3}
\hypersetup{
  pdftitle={PhD Thesis},
  pdfauthor={Román},
  pdfkeywords={optics, lens, tambara module},
  colorlinks = true,
  breaklinks=true,
  linktocpage,
  linkcolor = nordnight,
  citecolor = nordblue,
  urlcolor = nordnight,
}
\usepackage[nameinlink]{cleveref}

\usepackage[anythingbreaks]{breakurl}
    \expandafter\def\expandafter\UrlBreaks\expandafter{\UrlBreaks\do\a%
    \do\b\do\c\do\d\do\e\do\f\do\g\do\h\do\i\do\j\do\k\do\l\do\m\do\n%
    \do\o\do\p\do\q\do\r\do\s\do\t\do\u\do\v\do\w\do\x\do\y\do\z\do\&}

\long\def\ignore#1{}
\usepackage{adjustbox}
\usepackage{supertabular}
\usepackage{longtable}

\makeatletter
\newcommand{\nicelinktarget}[1]{\Hy@raisedlink{\Hy@raisedlink{\hypertarget{#1}{}}}}
\makeatother
\newcommand\defining[2]{\nicelinktarget{#1}{\color{black}{#2}}}
\newcommand\multicategory{\hyperlink{linkmulticategory}{multicategory}}
\newcommand\multicategories{\hyperlink{linkmulticategory}{multicategories}}

\newcommand\Multicategories{\hyperlink{linkmulticategory}{Multicategories}}
\newcommand\multifunctors{\hyperlink{linkmultifunctor}{multifunctors}}
\newcommand\Mult{\hyperlink{linkMult}{\mathbf{Mult}}}

\newcommand\Set{\hyperlink{linkSet}{\mathbf{Set}}}

\newcommand\zentre{\mathcal{Z}}

\newcommand\malleableMulticategory{\hyperlink{linkMalleableMulticategory}{malleable multicategory}}
\newcommand\malleableMulticategories{\hyperlink{linkMalleableMulticategory}{malleable multicategories}}

\newcommand\MalleableMulticategories{\hyperlink{linkMalleableMulticategory}{Malleable multicategories}}
\newcommand\coherentMulticategory{\hyperlink{linkcohmulticategory}{malleable multicategory}}
\newcommand\coherentMulticategories{\hyperlink{linkcohmulticategory}{malleable multicategories}}

\newcommand\mMult{\hyperlink{linkmMult}{\mathbf{mMult}}}

\newcommand{\symmetricMulticategory}{\hyperlink{linkSymmetricMulticategory}{symmetric multicategory}}
\newcommand{\symmetricMulticategories}{\hyperlink{linkSymmetricMulticategory}{symmetric multicategories}}

\newcommand{\cartesianMonoidalCategories}{\hyperlink{linkCartesianMonoidal}{cartesian monoidal categories}}
\newcommand{\cartesianMonoidalCategory}{\hyperlink{linkCartesianMonoidal}{cartesian monoidal category}}
\newcommand{\CartesianMonoidalCategories}{\hyperlink{linkCartesianMonoidal}{Cartesian monoidal categories}}

\newcommand{\partialMarkovCategory}{\hyperlink{linkDiscretePartialMarkovCategory}{partial Markov category}}

\newcommand{\discretePartialMarkovCategory}{\hyperlink{linkDiscretePartialMarkovCategory}{discrete partial Markov category}}
\newcommand{\discretePartialMarkovCategories}{\hyperlink{linkDiscretePartialMarkovCategory}{discrete partial Markov categories}}

\newcommand{\effectAlgebra}{\hyperlink{linkEffectAlgebra}{effect algebra}}
\newcommand{\effectAlgebras}{\hyperlink{linkEffectAlgebra}{effect algebras}}

\newcommand{\EffectAlgebras}{\hyperlink{linkEffectAlgebra}{Effect algebras}}

\newcommand{\dinaturalEquivalence}{\hyperlink{linkdinaturality}{dinatural equivalence}}
\newcommand{\dinaturality}{\hyperlink{linkdinaturality}{dinaturality}}
\newcommand\Contour[1]{\mathsf{Contour}{#1}}

\newcommand\decomp{\operatorname{decomp}}
\newcommand\bcomp{\ensuremath{\mathbin{\bm{;}}}}

\newcommand\MonCat{\mathbf{MonCat}}
\newcommand\MonCatStr{\hyperlink{linkStrictMonoidalFunctor}{\ensuremath{\mathbf{MonCat}_{\mathsf{Str}}}}}
\newcommand\Duals{\mathsf{Duals}}
\newcommand\Polar{\mathsf{Polar}}
\newcommand\monoidalCategory{\hyperlink{linkMonoidalCategory}{monoidal category}}
\newcommand\monoidalCategories{\hyperlink{linkMonoidalCategory}{monoidal categories}}

\newcommand\MonoidalCategories{\hyperlink{linkMonoidalCategory}{Monoidal categories}}
\newcommand\strictMonoidalCategory{\hyperlink{linkStrictMonoidalCategory}{strict monoidal category}}
\newcommand\strictMonoidalCategories{\hyperlink{linkStrictMonoidalCategory}{strict monoidal categories}}

\newcommand\StrictMonoidalCategories{\hyperlink{linkStrictMonoidalCategory}{Strict monoidal categories}}
\newcommand\nonStrictMonoidalCategory{\hyperlink{linkNonStrictMonoidalCategory}{non-strict monoidal category}}
\newcommand\stringDiagram{\hyperlink{linkStringDiagram}{string diagram}}
\newcommand\stringDiagrams{\hyperlink{linkStringDiagram}{string diagrams}}

\newcommand\StringDiagrams{\hyperlink{linkStringDiagram}{String diagrams}}

\newcommand\symmetricMonoidalCategory{\hyperlink{linkSymmetricMonoidal}{symmetric monoidal category}}
\newcommand\symmetricMonoidalCategories{\hyperlink{linkSymmetricMonoidal}{symmetric monoidal categories}}
\newcommand\SymmetricMonoidalCategories{\hyperlink{linkSymmetricMonoidal}{Symmetric monoidal categories}}

\newcommand\noidalCategory{\hyperlink{linknoidal}{noidal category}}

\newcommand\premonoidalCategory{\hyperlink{linkpremonoidal}{premonoidal category}}
\newcommand\premonoidalCategories{\hyperlink{linkpremonoidal}{premonoidal categories}}

\newcommand\PremonoidalCategories{\hyperlink{linkpremonoidal}{Premonoidal categories}}

\newcommand\effectfulCategory{\hyperlink{linkeffectful}{effectful category}}
\newcommand\effectfulCategories{\hyperlink{linkeffectful}{effectful categories}}
\newcommand\EffectfulCategories{\hyperlink{linkeffectful}{Effectful categories}}

\newcommand{\colorPre}[1]{{\color{nordred}{#1}}}
\newcommand{\colorMon}[1]{{\color{norddarkblue}{#1}}}
\newcommand{\pure}{\colorMon{pure}}
\newcommand{\effectful}{\colorPre{effectful}}
\newcommand{\MON}{\colorMon{\mathrm{Mon}}}
\newcommand{\MONRUN}{\colorMon{\mathrm{Mon}}_\colorPre{\mathrm{Run}}}
\newcommand\R{\colorPre{R}}
\newcommand\EFF{\colorPre{\mathrm{Eff}}}

\newcommand{\Run}{\colorPre{\operatorname{Run}}}
\newcommand{\obj}[1]{{#1}_{\mathrm{obj}}} %

\newcommand{\Braid}{\operatorname{Braid}_\R}

\newcommand\promonoidal{\hyperlink{linkpromonoidal}{promonoidal}}
\newcommand\promonoidalCategory{\hyperlink{linkpromonoidal}{promonoidal category}}
\newcommand\promonoidalCategories{\hyperlink{linkpromonoidal}{promonoidal categories}}

\newcommand\PromonoidalCategories{\hyperlink{linkpromonoidal}{Promonoidal categories}}
\newcommand\promonoidalFunctor{\hyperlink{linkPromonoidalFunctor}{promonoidal functor}}
\newcommand\promonoidalFunctors{\hyperlink{linkPromonoidalFunctor}{promonoidal functors}}

\newcommand{\monoidalMulticategory}{\hyperlink{linkMonoidalMulticategory}{monoidal multicategory}}

\newcommand{\MonoidalMulticategories}{\hyperlink{linkMonoidalMulticategory}{Monoidal multicategories}}
\newcommand{\monoidalMulticategories}{\hyperlink{linkMonoidalMulticategory}{monoidal multicategories}}

\newcommand\produoidal{\hyperlink{linkProduoidalCategory}{produoidal}}
\newcommand\symmetricProduoidal{\hyperlink{linkSymmetricProduoidal}{symmetric produoidal}}
\newcommand\produoidalCategory{\hyperlink{linkProduoidalCategory}{produoidal category}}

\newcommand\produoidalCategories{\hyperlink{linkProduoidalCategory}{produoidal categories}}
\newcommand\ProduoidalCategories{\hyperlink{linkProduoidalCategory}{Produoidal categories}}
\newcommand\symmetricProduoidalCategory{\hyperlink{linkSymmetricProduoidal}{symmetric produoidal category}}

\newcommand\symmetricProduoidalCategories{\hyperlink{linkSymmetricProduoidal}{symmetric produoidal categories}}

\newcommand{\pDuo}{\hyperlink{linkProduoidalCategory}{\ensuremath{\mathbf{ProDuo}}}}
\newcommand{\ProDuo}{\hyperlink{linkProduoidalCategory}{\ensuremath{\mathbf{ProDuo}}}}
\newcommand{\npDuo}{\hyperlink{linkNormalProduoidalCategory}{\ensuremath{\mathbf{nProDuo}}}}
\newcommand{\symProDuo}{\hyperlink{linkSymmetricProduoidal}{\ensuremath{\mathbf{symProDuo}}}}

\newcommand{\nSymProduo}{\hyperlink{linkSymmetricProduoidal}{\ensuremath{\mathbf{nSymProDuo}}}}
\newcommand{\produoidalFunctor}{\hyperlink{linkProduoidalFunctor}{produoidal functor}}
\newcommand{\produoidalFunctors}{\hyperlink{linkProduoidalFunctor}{produoidal functors}}

\newcommand\duoidalCategory{\hyperlink{linkDuoidalCategory}{duoidal category}}

\newcommand\duoidalCategories{\hyperlink{linkDuoidalCategory}{duoidal categories}}
\newcommand\DuoidalCategories{\hyperlink{linkDuoidalCategory}{Duoidal categories}}
\newcommand\normalDuoidalCategory{\hyperlink{linkNormalDuoidal}{normal duoidal category}}

\newcommand\normalDuoidalCategories{\hyperlink{linkNormalDuoidal}{normal duoidal categories}}

\newcommand\Prom{\hyperlink{linkProm}{\mathbf{Prom}}}

\newcommand\Splice[1]{\mathsf{Splice}{#1}}

\newcommand\id{\mathrm{id}}
\newcommand\Subd{\mathbf{D}_{\leq 1}}
\newcommand{\polygraph}{\hyperlink{linkPolygraph}{polygraph}}
\newcommand{\polygraphs}{\hyperlink{linkPolygraph}{polygraphs}}

\newcommand{\PolyGraph}{\hyperlink{linkPolygraph}{\mathbf{PolyGraph}}}
\newcommand{\polygraphCouple}{\hyperlink{linkpolygraphcouple}{effectful polygraph}}
\newcommand{\polygraphCouples}{\hyperlink{linkpolygraphcouple}{effectful polygraphs}}

\newcommand{\polycategory}{\hyperlink{linkPolycategory}{polycategory}}

\newcommand{\polycategories}{\hyperlink{linkPolycategory}{polycategories}}
\newcommand{\Polycategories}{\hyperlink{linkPolycategory}{Polycategories}}

\newcommand{\malleablePolycategory}{\hyperlink{linkMalleablePolycategory}{malleable polycategory}}

\newcommand{\malleablePolycategories}{\hyperlink{linkMalleablePolycategory}{malleable polycategories}}

\newcommand{\polyfunctors}{\hyperlink{linkPolyfunctor}{polyfunctors}}

\newcommand{\ProStarAut}{\mathbf{ProStar}}
\newcommand{\ProStar}{\mathbf{ProStar}}
\newcommand{\ProstarAutonomousCategories}{\hyperlink{linkProstar}{Prostar autonomous categories}}

\newcommand{\prostarAutonomousCategories}{\hyperlink{linkProstar}{prostar autonomous categories}}
\newcommand{\prostarAutonomousCategory}{\hyperlink{linkProstar}{prostar autonomous category}}
\newcommand{\prostarFunctor}{\hyperlink{linkProstarFunctor}{prostar functor}}

\newcommand{\kleisli}{\mathsf{Kleisli}}
\newcommand{\Kleisli}{\mathsf{Kleisli}}

\newcommand{\processTheory}{\hyperlink{linkProcessTheory}{process theory}}

\newcommand{\processTheories}{\hyperlink{linkProcessTheory}{process theories}}

\newcommand{\Cat}{\mathbf{Cat}}

\newcommand{\bicategory}{\hyperlink{linkBicategory}{bicategory}}
\newcommand{\bicategories}{\hyperlink{linkBicategory}{bicategories}}

\newcommand{\binoidalCategories}{\hyperlink{linknoidalcategory}{binoidal categories}}

\newcommand{\egg}{\mathsf{egg}}
\newcommand{\white}{\mathsf{white}}
\newcommand{\shell}{\mathsf{shell}}
\newcommand{\sugar}{\mathsf{sugar}}
\newcommand{\paste}{\mathsf{paste}}
\newcommand{\thickPaste}{\mathsf{thick\ paste}}
\newcommand{\yolk}{\mathsf{yolk}}
\newcommand{\whiskedWhites}{\mathsf{whisked\ whites}}
\newcommand{\mascarpone}{\mathsf{mascarpone}}
\newcommand{\cream}{\mathsf{cream}}
\newcommand{\crack}{\operatorname{crack}}
\newcommand{\beat}{\operatorname{beat}}
\newcommand{\stir}{\operatorname{stir}}
\newcommand{\whisk}{\operatorname{whisk}}
\newcommand{\fold}{\operatorname{fold}}
\newcommand{\swap}{\operatorname{swap}}
\newcommand{\discard}{\operatorname{discard}}

\newcommand{\physicalDuoidalCategory}{\hyperlink{linkPhysicalDuoidal}{physical duoidal category}}

\newcommand{\physicalDuoidalCategories}{\hyperlink{linkPhysicalDuoidal}{physical duoidal categories}}
\newcommand{\PhysicalDuoidalCategories}{\hyperlink{linkPhysicalDuoidal}{Physical duoidal categories}}

\newcommand{\physicalMonoidalMulticategory}{\hyperlink{linkPhysicalMonoidalMulticategory}{physical monoidal multicategory}}
\newcommand{\physicalMonoidalMulticategories}{\hyperlink{linkPhysicalMonoidalMulticategory}{physical monoidal multicategories}}

\newcommand{\PhysicalMonoidalMulticategories}{\hyperlink{linkPhysicalMonoidalMulticategory}{Physical monoidal multicategories}}

\newcommand{\wShuf}{\hyperlink{linkShufflingWords}{\ensuremath{\mathsf{wShuf}}}}

\newcommand{\ShufflingWords}{\hyperlink{linkShufflingWords}{Shuffling words}}

\newcommand{\pShuf}{\hyperlink{linkPolarShuffle}{\ensuremath{\mathsf{pShuf}}}}
\newcommand{\polarShuffle}{\hyperlink{linkPolarShuffle}{polar shuffle}}
\newcommand{\polarShuffles}{\hyperlink{linkPolarShuffle}{polar shuffles}}

\newcommand{\PolarShuffles}{\hyperlink{linkPolarShuffle}{Polar shuffles}}

\newcommand{\SHF}{\ensuremath{\operatorname{\mbox{\textsc{shf}}}}}
\newcommand{\COM}{\ensuremath{\operatorname{\mbox{\textsc{lnk}}}}}
\newcommand{\LNK}{\ensuremath{\operatorname{\mbox{\textsc{lnk}}}}}
\newcommand{\SPW}{\ensuremath{\operatorname{\mbox{\textsc{spw}}}}}
\newcommand{\NOP}{\ensuremath{\operatorname{\mbox{\textsc{nop}}}}}

\newcommand{\messageTheory}{\hyperlink{linkMessageTheory}{message theory}}

\newcommand{\messageTheories}{\hyperlink{linkMessageTheory}{message theories}}
\newcommand{\MessageTheories}{\hyperlink{linkMessageTheory}{Message theories}}

\newcommand\mydots{\makebox[0.8em][c]{.\hfil.\hfil.}}

\newcommand{\Proc}{\hyperlink{linkProc}{\ensuremath{\mathsf{Proc}}}}

\newcommand{\hypergraph}{\hyperlink{linkHypergraph}{hypergraph}}

\newcommand{\hypergraphs}{\hyperlink{linkHypergraph}{hypergraphs}}

\newcommand{\StringSigma}{\hyperlink{linkSymmetricStringDiagram}{\mathsf{String}_\sigma}}
\newcommand{\SymMonCatStr}{\hyperlink{linkSymMonCatStr}{\mathbf{SymMonCat}_\mathsf{Str}}}

\newcommand\return{\mathsf{return}}
\newcommand\Ins{\mathsf{Ins}}
\newcommand{\var}[1]{\mathsf{#1}}

\newcommand{\adjointMonoid}{\hyperlink{linkAdjointMonoid}{adjoint monoid}}

\newcommand{\adjointMonoids}{\hyperlink{linkAdjointMonoid}{adjoint monoids}}

\newcommand{\profunctor}{\hyperlink{linkProfunctor}{profunctor}}
\newcommand{\profunctors}{\hyperlink{linkProfunctor}{profunctors}}

\newcommand{\Profunctors}{\hyperlink{linkProfunctor}{Profunctors}}

\newcommand\sequentialUnit{\hyperlink{linkProduoidalComponents}{sequential unit}}

\newcommand\parallelUnit{\hyperlink{linkProduoidalComponents}{parallel unit}}
\newcommand\parallelUnits{\hyperlink{linkProduoidalComponents}{parallel units}}

\newcommand\sequentialSplit{\hyperlink{linkProduoidalComponents}{sequential join}}

\newcommand\SequentialSplits{\hyperlink{linkProduoidalComponents}{Sequential joins}}

\newcommand\parallelSplit{\hyperlink{linkProduoidalComponents}{parallel join}}

\newcommand\ParallelSplits{\hyperlink{linkProduoidalComponents}{Parallel joins}}

\newcommand\splicedMonoidalArrows{\hyperlink{linkProduoidalSplice}{spliced monoidal arrows}}
\newcommand\SplicedMonoidalArrows{\hyperlink{linkProduoidalSplice}{Spliced monoidal arrows}}
\newcommand\produoidalSplice{\hyperlink{linkProduoidalSplice}{produoidal splice}}
\newcommand\MonoidalContour{\hyperlink{linkMonoidalContour}{Monoidal contour}}
\newcommand\monoidalContour{\hyperlink{linkMonoidalContour}{monoidal contour}}

\newcommand\monoidalLens{\hyperlink{linkMonoidalLens}{monoidal lens}}

\newcommand\monoidalLenses{\hyperlink{linkMonoidalLens}{monoidal lenses}}
\newcommand\MonoidalLenses{\hyperlink{linkMonoidalLens}{Monoidal lenses}}

\newcommand\bisplice[2]{\langle #1 ⨾ \square ⨾ #2 \rangle}

\newcommand\trisplice[3]{\langle #1 ⨾ \square ⨾ #2 ⨾ \square ⨾ #3 \rangle}

\newcommand\Sobocinski{Soboci\'{n}ski}

\newcommand{\mSplice}[1]{\hyperlink{linkMonoidalSplice}{\mathsf{mSplice}(#1)}}
\newcommand{\mSpliceF}{\hyperlink{linkMonoidalSplice}{\mathsf{mSplice}}}
\newcommand{\mContour}[1]{\hyperlink{linkMonoidalContour}{\mathsf{mContour}(#1)}}
\newcommand{\mContourF}{\hyperlink{linkMonoidalContour}{\mathsf{mContour}}}

\newcommand{\normalProduoidalCategories}{\hyperlink{linkNormalProduoidalCategory}{normal produoidal categories}}
\newcommand{\normalProduoidalCategory}{\hyperlink{linkNormalProduoidalCategory}{normal produoidal category}}
\newcommand{\NormalProduoidalCategories}{\hyperlink{linkNormalProduoidalCategory}{Normal produoidal categories}}

\newcommand{\normalSymmetricProduoidalCategory}{\hyperlink{linkNormalSymmetricProduoidalCategory}{normal symmetric produoidal category}}

\newcommand{\promonad}{\hyperlink{linkPromonad}{promonad}}

\newcommand{\monoidalContext}{\hyperlink{linkMonoidalContext}{monoidal context}}

\newcommand{\monoidalContexts}{\hyperlink{linkMonoidalContext}{monoidal contexts}}

\newcommand\mLens{\mathsf{mLens}}
\newcommand\smLens{\hyperlink{linkSymmetricMonoidalLens}{\mathsf{smLens}}}

\newcommand{\symmetricMonoidalLens}{\hyperlink{linkSymmetricMonoidalLens}{symmetric monoidal lens}}
\newcommand{\SymmetricMonoidalLenses}{\hyperlink{linkSymmetricMonoidalLens}{Symmetric monoidal lenses}}
\newcommand{\symmetricMonoidalLenses}{\hyperlink{linkSymmetricMonoidalLens}{symmetric monoidal lenses}}

\newcommand\pbiobj[2]{\ensuremath{\left(\biobj{#1}{#2}\right)}}
\newcommand{\NOR}{\hyperlink{linkProduoidalNormalization}{\mathsf{Nor}}}
\newcommand{\sNOR}{\hyperlink{linkSymmetricProduoidalNormalization}{\mathsf{sNor}}}
\newcommand{\withPoint}[2]{#2_{#1}}

\newcommand{\starp}{\hyperlink{linkpromonad}{\star}}
\newcommand{\unitp}[1]{{#1}^{\hyperlink{linkpromonad}{\circ}}}

\newcommand{\Session}{\hyperlink{linkSessions}{\mathsf{Session}}}
\newcommand{\Msg}{\hyperlink{linkSessions}{\mathsf{Msg}}}
\newcommand{\inProc}{\mathsf{inProc}}

\newcommand{\Send}[1]{{#1}^{∘}}
\newcommand{\Get}[1]{{#1}^{•}}

\newcommand{\Alice}{\ensuremath{\mathsf{Alice}}}
\newcommand{\Bob}{\ensuremath{\mathsf{Bob}}}
\newcommand{\Eve}{\ensuremath{\mathsf{Eve}}}
\newcommand{\Stage}{\ensuremath{\mathsf{Stage}}} %
\newcommand\blackComultiplication{
\tikzset{every picture/.style={line width=0.85pt}} %
\begin{tikzpicture}[x=0.75pt,y=0.75pt,yscale=-0.4,xscale=0.4,baseline=-15pt,rotate=90,transform shape]
\draw    (40,39.92) -- (27,40) ;
\draw    (50,50.01) .. controls (40.47,50.32) and (41.08,47.82) .. (40,39.92) ;
\draw    (40,39.92) .. controls (42.58,30.82) and (40.87,29.92) .. (50,30.01) ;
\draw  [fill={rgb, 255:red, 0; green, 0; blue, 0 }  ,fill opacity=1 ] (36,39.92) .. controls (36,37.71) and (37.79,35.92) .. (40,35.92) .. controls (42.21,35.92) and (44,37.71) .. (44,39.92) .. controls (44,42.13) and (42.21,43.92) .. (40,43.92) .. controls (37.79,43.92) and (36,42.13) .. (36,39.92) -- cycle ;
\end{tikzpicture}
}
\newcommand\iconbcm{\blackComultiplication}

\newcommand\blackComonoidUnit{
\tikzset{every picture/.style={line width=0.85pt}} %
\begin{tikzpicture}[x=0.75pt,y=0.75pt,yscale=-0.4,xscale=0.4,baseline=-15pt,rotate=90,transform shape]
  \tikzset{every picture/.style={line width=0.85pt}} %
  \draw  (40,39.92) -- (27,40);
  \draw  [fill={rgb, 255:red, 0; green, 0; blue, 0 }  ,fill opacity=1 ] (36,39.92) .. controls (36,37.71) and (37.79,35.92) .. (40,35.92) .. controls (42.21,35.92) and (44,37.71) .. (44,39.92) .. controls (44,42.13) and (42.21,43.92) .. (40,43.92) .. controls (37.79,43.92) and (36,42.13) .. (36,39.92) -- cycle ;
\end{tikzpicture}
}

\newcommand\iconbcu{\blackComonoidUnit}

\newcommand\blackUnit{
\tikzset{every picture/.style={line width=0.85pt}} %
\begin{tikzpicture}[x=0.75pt,y=0.75pt,yscale=-0.4,xscale=0.4,baseline=8pt,rotate=270,transform shape]
\draw    (40,39.92) -- (27,40) ;
\draw  [fill={rgb, 255:red, 0; green, 0; blue, 0 }  ,fill opacity=1 ] (36,39.92) .. controls (36,37.71) and (37.79,35.92) .. (40,35.92) .. controls (42.21,35.92) and (44,37.71) .. (44,39.92) .. controls (44,42.13) and (42.21,43.92) .. (40,43.92) .. controls (37.79,43.92) and (36,42.13) .. (36,39.92) -- cycle ;
\end{tikzpicture}
}
\newcommand\iconbu{\blackUnit}

\newcommand\whiteUnit{
\tikzset{every picture/.style={line width=0.85pt}} %
\begin{tikzpicture}[x=0.75pt,y=0.75pt,yscale=-0.4,xscale=0.4,baseline=8pt,rotate=270,transform shape]
\draw    (40,39.92) -- (27,40) ;
\draw  [fill={rgb, 255:red, 255; green, 255; blue, 255 }  ,fill opacity=1 ] (36,39.92) .. controls (36,37.71) and (37.79,35.92) .. (40,35.92) .. controls (42.21,35.92) and (44,37.71) .. (44,39.92) .. controls (44,42.13) and (42.21,43.92) .. (40,43.92) .. controls (37.79,43.92) and (36,42.13) .. (36,39.92) -- cycle ;
\end{tikzpicture}
}
\newcommand\iconwu{\whiteUnit}

\newcommand\whiteMultiplication{
\tikzset{every picture/.style={line width=0.85pt}} %
\begin{tikzpicture}[x=0.75pt,y=0.75pt,yscale=-0.4,xscale=0.4,baseline=8pt,rotate=270,transform shape]
\draw    (40,39.92) -- (27,40) ;
\draw    (50,50.01) .. controls (40.47,50.32) and (41.08,47.82) .. (40,39.92) ;
\draw    (40,39.92) .. controls (42.58,30.82) and (40.87,29.92) .. (50,30.01) ;
\draw  [fill={rgb, 255:red, 255; green, 255; blue, 255 }  ,fill opacity=1 ] (36,39.92) .. controls (36,37.71) and (37.79,35.92) .. (40,35.92) .. controls (42.21,35.92) and (44,37.71) .. (44,39.92) .. controls (44,42.13) and (42.21,43.92) .. (40,43.92) .. controls (37.79,43.92) and (36,42.13) .. (36,39.92) -- cycle ;
\end{tikzpicture}
}
\newcommand\iconwm{\whiteMultiplication}

\newcommand\whiteComultiplication{
\tikzset{every picture/.style={line width=0.85pt}} %
\begin{tikzpicture}[x=0.75pt,y=0.75pt,yscale=-0.4,xscale=0.4,baseline=-15pt,rotate=90,transform shape]
\draw    (40,39.92) -- (27,40) ;
\draw    (50,50.01) .. controls (40.47,50.32) and (41.08,47.82) .. (40,39.92) ;
\draw    (40,39.92) .. controls (42.58,30.82) and (40.87,29.92) .. (50,30.01) ;
\draw  [fill={rgb, 255:red, 255; green, 255; blue, 255}  ,fill opacity=1 ] (36,39.92) .. controls (36,37.71) and (37.79,35.92) .. (40,35.92) .. controls (42.21,35.92) and (44,37.71) .. (44,39.92) .. controls (44,42.13) and (42.21,43.92) .. (40,43.92) .. controls (37.79,43.92) and (36,42.13) .. (36,39.92) -- cycle ;
\end{tikzpicture}
}
\newcommand\iconwcm{\whiteComultiplication}

\newcommand\whiteComonoidUnit{
\tikzset{every picture/.style={line width=0.85pt}} %
\begin{tikzpicture}[x=0.75pt,y=0.75pt,yscale=-0.4,xscale=0.4,baseline=-15pt,rotate=90,transform shape]
  \tikzset{every picture/.style={line width=0.85pt}} %
  \draw  (40,39.92) -- (27,40);
  \draw  [fill={rgb, 255:red, 255; green, 255; blue, 255}  ,fill opacity=1 ] (36,39.92) .. controls (36,37.71) and (37.79,35.92) .. (40,35.92) .. controls (42.21,35.92) and (44,37.71) .. (44,39.92) .. controls (44,42.13) and (42.21,43.92) .. (40,43.92) .. controls (37.79,43.92) and (36,42.13) .. (36,39.92) -- cycle ;
\end{tikzpicture}
}
\newcommand\iconwcu{\whiteComonoidUnit}

\newcommand\iconAgent{\includegraphics[scale=0.24]{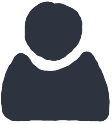}}
\newcommand\iconPredictor{\includegraphics[scale=0.20]{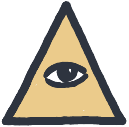}}
\newcommand\iconBoxes{\includegraphics[scale=0.24]{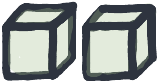}} 

\theoremstyle{plain}
\newtheorem{theorem}{Theorem}[chapter]
\newtheorem{proposition}[theorem]{Proposition}

\newtheorem{lemma}[theorem]{Lemma}

\newtheorem{conjecture}[theorem]{Conjecture}
\newtheorem{corollary}[theorem]{Corollary}
\theoremstyle{definition}
\newtheorem{definition}[theorem]{Definition}

\theoremstyle{remark}
\newtheorem{remark}[theorem]{Remark}
\newtheorem{example}[theorem]{Example}

\begingroup\makeatletter
  \@for\theoremstyle:=definition,remark,plain
  \do{\expandafter\g@addto@macro\csname th@\theoremstyle\endcsname{\addtolength\thm@preskip\parskip}}
\endgroup

\numberwithin{theorem}{section}


\usepackage[numbers,sort&compress]{natbib}

\usepackage{array}
\newcolumntype{H}{>{\setbox0=\hbox\bgroup}c<{\egroup}@{}}

\makeatletter
\def\l@subsection{\@tocline{1}{0pt}{2pc}{}{}}
\makeatother

\title{Monoidal Context Theory}
\author[Román]{Mario Román\\
\quad \\
{\small \textsc{Tallinn University of Technology, PhD Thesis,}\\
Defended on the 16th November 2023, Tallinn \\
Supervisor: {Pawe{\l} Soboci{\'n}ski}\\
Opponents: Guy McCusker and Paul-André Melliès\\
}}

\begin{document} 

\begin{abstract}
  We universally characterize the produoidal category of monoidal lenses over a monoidal category.
  In the same way that each category induces a cofree promonoidal category of spliced arrows, each monoidal category induces a cofree produoidal category of monoidal spliced arrows; monoidal lenses are the free normalization of the cofree produoidal category of monoidal spliced arrows.

  We apply the characterization of symmetric monoidal lenses to the analysis of multi-party message-passing protocols. We introduce a minimalistic axiomatization of message passing -- message theories -- and we construct combinatorially the free message theory over a set. Symmetric monoidal lenses are the derivations of the free message theory over a symmetric monoidal category.

  \ \quad \\

  \begin{center}
    \textsc{Monoidiliste Kontekstide Teooria}
  \end{center}

  \noindent\textsc{Kokkuvõte.}  Karakteriseerime monoidiliste läätsede produoidilise kategooria
  universaalomaduse abil. Nii nagu iga kategooria indutseerib
  pleissitud noolte kovaba promonoidilise kategooria, indutseerib
  monoidiline kategooria monoidiliste pleissnoolte kovaba produoidilise
  kategooria; monoidilised läätsed on monoidiliste pleissnoolte kovaba
  produoidilise kategooria vaba normalisatsioon.
  
  Kasutame sümmeetriliste monoidiliste läätsede karakterisatsiooni mitme
  osapoole sõnumiedastusprotokollide analüüsimiseks. Toome sisse
  sõnumiedastuse minimalistliku aksiomatisatsiooni – sõnumiteooriaid – ja
  konstrueerime vaba sõnumiteooria etteantud hulgal.
  Sümmeetrilised monoidilised läätsed on sümmeetrilise monoidilise
  kategooria vaba sõnumiteooria tuletised.
\end{abstract}

\maketitle
\newpage

\subsection*{Acknowledgements}
I would like to thank my PhD advisor, Pawel Sobocinski. Pawel has an exceptional ability to separate the scientifically promising ideas from the noise; he gave me the support, encouragement and freedom to pursue the research on this thesis. Pawel always said he wanted to replicate in Tallinn the atmosphere of Bob Walters' group in Sydney and I am particularly thankful for the result. I am also very grateful to Nicoletta Sabadini, for her advice and for sharing her encyclopedic knowledge of both automata and the history of Como.

Most ideas were cultivated at group meetings, and I want to thank Ed, Chad, Clémence, Nathan, Diana, Fosco, Elina, Amar, Cole, Philipp, Ekaterina, Niccolò, Michele, Andrea, and the rest of the \emph{Tarkvarateaduse Instituut} for all the math and time we shared. I am very grateful to Niels, for his contagious enthusiasm and much useful feedback on this thesis. Special thanks go to Matt for his attention to detail and mathematical elegance, great discussions and ideas, and equally great book recommendations.

I learned and enjoyed a lot on short but productive visits to Pisa, Como, Oxford and Paris, and I want to thank Filippo, Alessandro, Vladimir, Louis and Davidad for many insightful discussions during this thesis. I had the privilege of having Giovanni, James, and Dylan as coauthors and I learned a lot from each one of them. 

I thank the constant support of my parents, my brother Víctor, and my friends; I especially thank David and Esperanza for finding the best cafés in Granada. I thank Anna, Paolo, Enrico and Andrea for the time at the lake. Finally, I thank Elena: for all the fun we had writing each joint paper, and for all the happiness, math and drawings of these four years.

\setcounter{tocdepth}{3}
\tableofcontents

\clearpage{}%

\section*{Preface}
Understanding and correctly designing intelligent and explainable systems could be both, if we get it right, one of the most beneficial human advancements; and, if we get it wrong, an existential risk for humanity \cite{ord20:precipice}. Humanity's need for languages and formalisms for trustworthy complex systems is now an urge. 

Mathematics may possibly be the only right tool for this; but mathematics has not always been concerned with complex and interconnected systems. John von Neumann, talking about the intelligent and complex system that is the human brain, famously noted that

\begin{quote}
    the outward forms of our mathematics are not absolutely relevant from the point of view of evaluating what the mathematical or logical language truly used by the central nervous system is. However, the above remarks about reliability and logical and arithmetical depth prove that whatever the system is, it cannot fail to differ considerably from what we consciously and explicitly consider as mathematics.
    
     --  John Von Neumann, \emph{The Computer and The Brain} \cite{vonneumann20:computerandbrain}.
\end{quote}

Meanwhile, when we try to describe big interconnected networks with linear algebra, geometry and calculus, even with all of our achievements, we seem to miss the point: things get extremely complicated, computationally intractable, humanly unimaginable; and we declare our defeat, we resort to vague analogies, and we ask an impenetrable pile of linear algebra to be our oracle.

This does not need to be our strategy: mathematics and computer science do not advance with bigger computations; they advance with new conceptual understanding.
The past century saw the rise of conceptual mathematics and theoretical computer science -- the kind of mathematics that took seriously the most elementary notions and cultivated them to tame complex abstractions and systems \cite{lawvere09:conceptual,lawvereinterview,grothendieck85:recoltes}. Slowly but surely, %
the development of the conceptual theory of categories has brought us to a point where we can forget about comforting but vague analogies and start talking about complex systems formally and scientifically.

This thesis is part of the ongoing effort to find better languages and reasoning tools for science, epistemology, causality and probability: both intuitive graphical syntaxes for humans to reason with, and formal languages for computers, linked by a trusted and transparent mathematical formalism.

\newpage

\section*{Introduction}

\subsection*{Processes and Diagrams}

Processes come intuitively to us; descriptions of processes arose independently all across science and engineering, in the form of diagrams, flowcharts or prose. We reason with them and we depict them all the time, but that does not mean that we always know how to interpret them: many diagrams in computer science and elsewhere do not have clear formal semantics, so we relegate them to serving merely as sources of intuition and inspiration. 

\begin{quote}
  The notation has been found very useful in practice as it greatly simplifies the appearance of complicated tensor or spinor equations, the various interrelations expressed being discernable at a glance. Unfortunately the notation seems to be of value mainly for private calculations because it cannot be printed in the normal way. 
  -- Penrose and Rindler, \emph{Spinors and Spacetime} \cite{penrose:kissingerquote}
\end{quote}

Diagrams deserve better: we can lift diagrams from mere intuitions to mathematical structures; we can defend the legitimate and exceptional conceptual mathematics we now have to talk about processes and diagrams. This thesis follows the framework of \symmetricMonoidalCategories{}. Processes that pass resources around and that compose sequentially and in parallel form \symmetricMonoidalCategories{}; diagrams that depict these processes are no less than a sound and complete formal syntax for \symmetricMonoidalCategories{} (e.g. \Cref{fig:one-time-pad-correctnessIntro}).

We will develop formal syntaxes for the compositional description of process, in particular for -- but not restricted to -- probabilistic, effectful and non-classical processes. We make use of category theory as a foundational tool: category theory allows us to characterize a syntactic construction as the one generating a universal semantics object and, at the same time, it provides a robust classification framework for mathematical structures.

\begin{figure}[ht]
  \centering
  \includegraphics[scale=0.35]{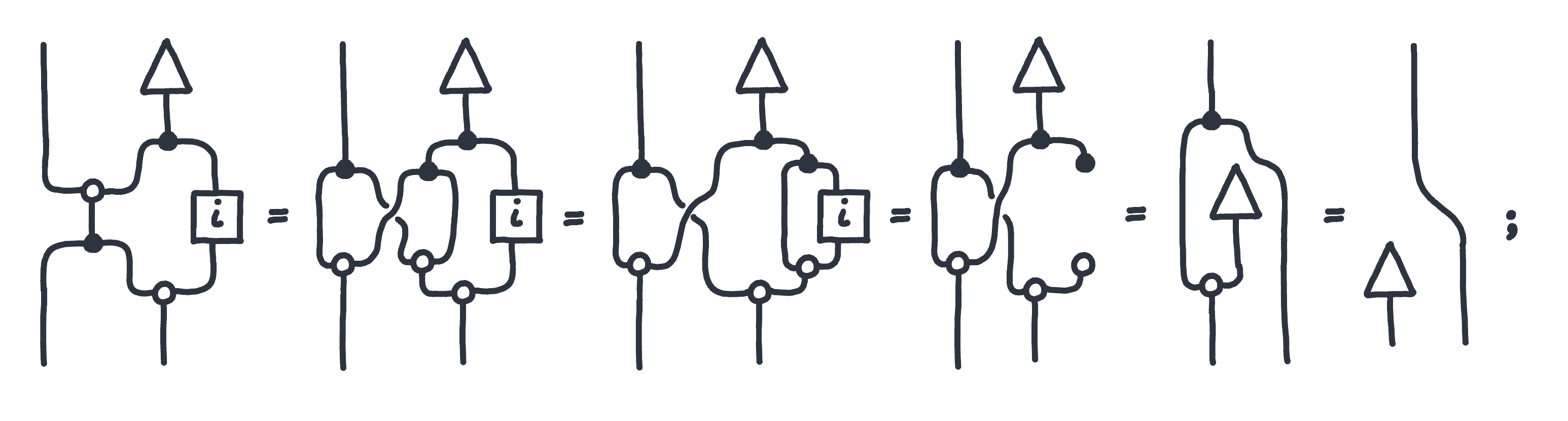}
  \caption{String-diagrammatic correctness proof for the One-time pad protocol (\Cref{prop:correctness}, {{\cite{broadbent22:crypto}}}).}
  \label{fig:one-time-pad-correctnessIntro}
\end{figure}

\subsection*{Algebra and Duoidal Algebra}
The main technical idea of this thesis is natural: in the same way that the analysis of classical algebraic theories required the development of \multicategories{} -- and more precisely, of cartesian multicategories and Lawvere theories -- the analysis of process theories, which are themselves two-dimensional algebraic theories, requires the development of \monoidalMulticategories{} and \duoidalCategories{}.

\Multicategories{}, or colored operads, are mathematical structures that describe algebraic theories. In 1963, Lawvere introduced a categorical approach to universal algebra \cite{lawvere63:functorial}: a theory can be captured by the cartesian \multicategory{} containing all of its derived operations, and this notion is invariant to the specific primitive operations we choose to present the theory. This idea opens the field of functorial semantics: theories are categories, models are functors, and homomorphisms are natural transformations. More importantly, Lawvere's thesis gives a robust account of classical algebra that can be modified to suit our needs: the same framework can be employed for deductive systems \cite{lambek:deductive}, higher-order algebra \cite{lambek1986a}, relational algebra \cite{pavlovic17}, or partial algebra \cite{di2021functorial}. 

How does it apply to \processTheories{}? \MonoidalCategories{} and \multicategories{} are not structured enough for the task of describing 2-dimensional structures themselves: we need \duoidalCategories{} and \produoidalCategories{} \cite{street12:linking}. Intermediate algebraic expressions with variables are not complete expressions; they are only \emph{contexts} into which we can plug values, and context is of central importance in computer science: we model not only processes but also the environment in which they act. While the algebra of 1-dimensional context is commonplace in applications like parsing \cite{mellies22:parsing},
the same concept was missing for 2-dimensional syntaxes, which are still less frequent in computer science \cite{uustalu18:sequent,earnshaw22}. 

\begin{figure}[ht]
  \centering
  \includegraphics[scale=0.35]{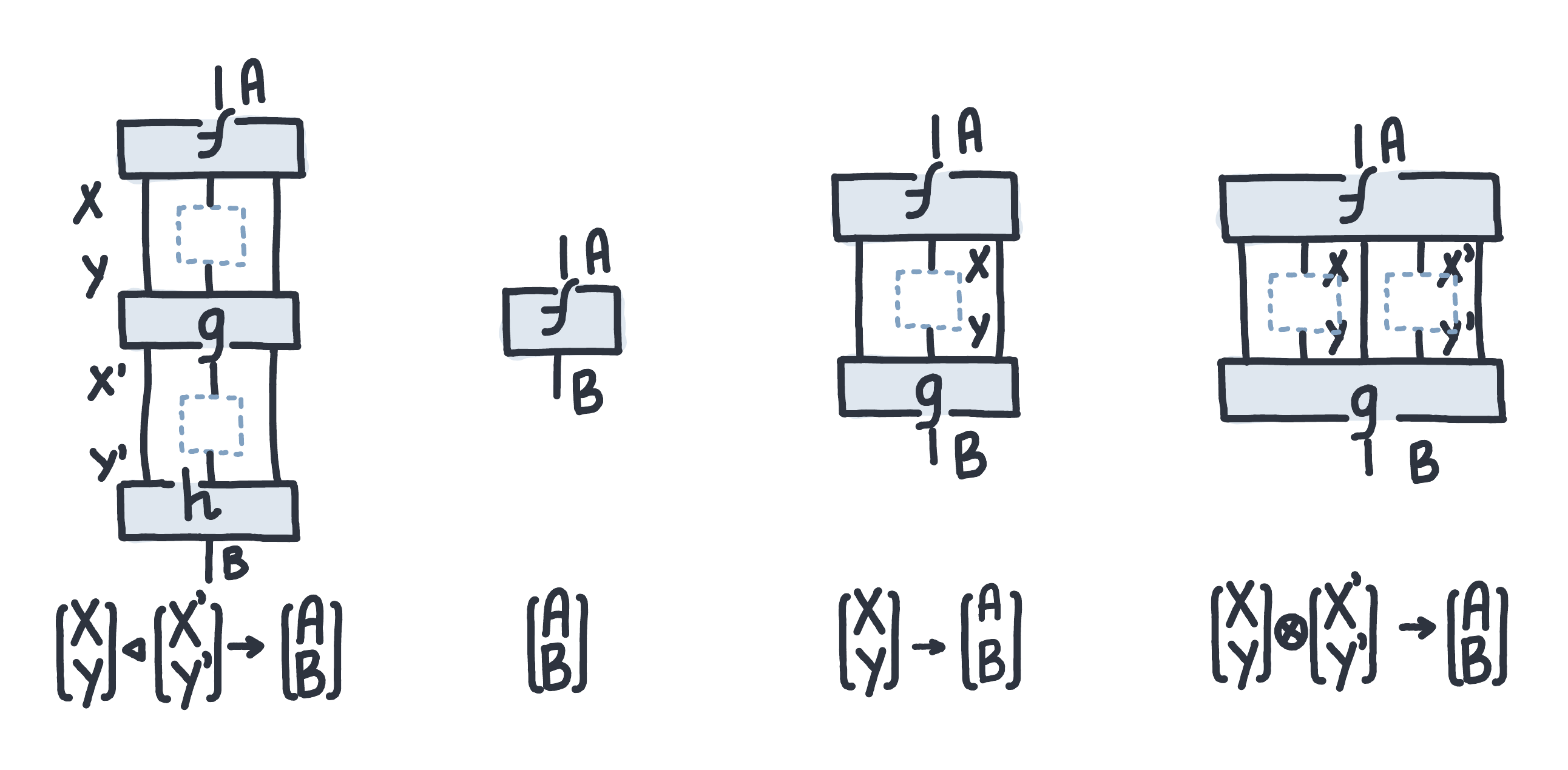}
  \caption{A depiction of monoidal lenses, or incomplete processes.}
  \label{fig:monoidalLensesIntro}
\end{figure}

\DuoidalCategories{} are well-known and there is a reasonable body of literature primarily concerned with applications in pure algebra and algebraic topology \cite{aguiar10:monoidal,street12:linking}; but the usage of \duoidalCategories{} to study processes is less frequent: two notable examples are the treatment of commutativity in the work of Garner and López Franco \cite{garner16}, and the study of ``compositional dependencies'' in the recent work of Spivak and Shapiro \cite{shapiro22:duoidal}.
In this text, \duoidalCategories{} and \monoidalMulticategories{} allow us to postulate axioms for modularity and message passing; these axioms apply to any \symmetricMonoidalCategory{}, or any \processTheory{}.

\subsection*{Fundamental Structures for Message Passing}
This main idea has an immediate consequence that we explore in the second part of this thesis: we can now develop an algebra for incomplete processes and their communication. While \emph{concurrent} software has been intensively studied since the early 60s, the theoretical research landscape remains quite fragmented: we do not have a satisfactory understanding of the underlying mathematical principles of concurrency, and the proliferation of models has not helped us understand how they relate. Indeed, Abramsky~\cite{abramsky06:concurrency} argued in 2006 that we simply do not know what the fundamental structures of concurrency are.

A way to identify such principles and arrive at more canonical models is to look for logical or universal properties. An example of the former is the discovery of and work on Curry-Howard style connections between calculi for concurrency and fragments of linear logic, which led to the development of session types~\cite{honda93,dezani09}.
We take the latter route: departing from \monoidalCategories{} and their theory of context, we universally characterize a minimalistic axiomatization of message passing in process theories.

Concurrent message passing assumes two principles: interleaving and polarization. Polarization is a categorical technique to construct dualities; and in message passing, it constructs the duality between sending and receiving \cite{cockett07:polarized,nester21,mellies21:asynchronous}. Interleaving is well-known in concurrency, and it models the ability of multiple processes to advance in parallel by mixing their global effects: imagine multiple processes determined by a sequence of statements; their concurrent execution may shuffle these statements in any possible order -- the only requirement is to preserve the relative order of statements within any single process. We will not only propose a minimalistic axiomatization of message passing from these two principles, but we will also characterize the universal structures for message passing on a \processTheory{}.

Briefly, we assume polarized types, $X^{•}$ and $X^{∘}$, that correspond to \emph{sending} and \emph{receiving}; and ordered lists of types describe sessions. Our axioms ask that \emph{(i)} a sending port can be linked to a receiving port; \emph{(ii)} echoing allows us to receive and then send; \emph{(iii)} sequences of actions can be interleaved by a shuffling $τ$; and \emph{(iv)} there exists a no-operation that does nothing.

\begin{figure}[ht]
  \hspace{-3em}
  \begin{minipage}{0.3\textwidth}
  \begin{mathpar}
      \inferrule*[Right=(com)]
      {Γ, X^{•}, X^{∘}, Δ}
      {Γ, Δ}
  \end{mathpar}
  \end{minipage}
  \begin{minipage}{0.2\textwidth}
  \begin{mathpar}
      \inferrule*[Right=(spw)]
      {\ }
      {X^{∘}, X^{•}}
  \end{mathpar}
  \end{minipage}
  \begin{minipage}{0.3\textwidth}
    \begin{mathpar}
        \inferrule*[Right=(shf${}_{\tau}$)]
        {Γ \\ Δ}
        {τ(Γ,Δ)}
    \end{mathpar}
    \end{minipage}
  \begin{minipage}{0.15\textwidth}
    \begin{mathpar}
        \inferrule*[Right=(nop)]
        {\ }
        {()}
    \end{mathpar}
    \end{minipage}
  \caption{Type-theoretic presentation of a message theory.}
  \label{fig:type-messageIntro}
\end{figure}

This is a naive logic of message passing, but its strength is that it can be characterized mathematically using \duoidalCategories{} and, more concretely, \physicalMonoidalMulticategories{}, which we introduce. This paves the way to an adjunction that characterizes the free \messageTheory{} on top of any process theory. The idea is simple but powerful: in order to construct \messageTheories{}, we need to add global effects for \emph{sending} and \emph{receiving} to our process theories \cite{orchard16:effects}; \Cref{th:sessions-vs-processes} notices that the diagrams for resulting effectful process theories can be wired precisely in the ways that the minimalistic logic of message passing prescribes.

\begin{figure}[ht]
  \centering
  \includegraphics[scale=0.30]{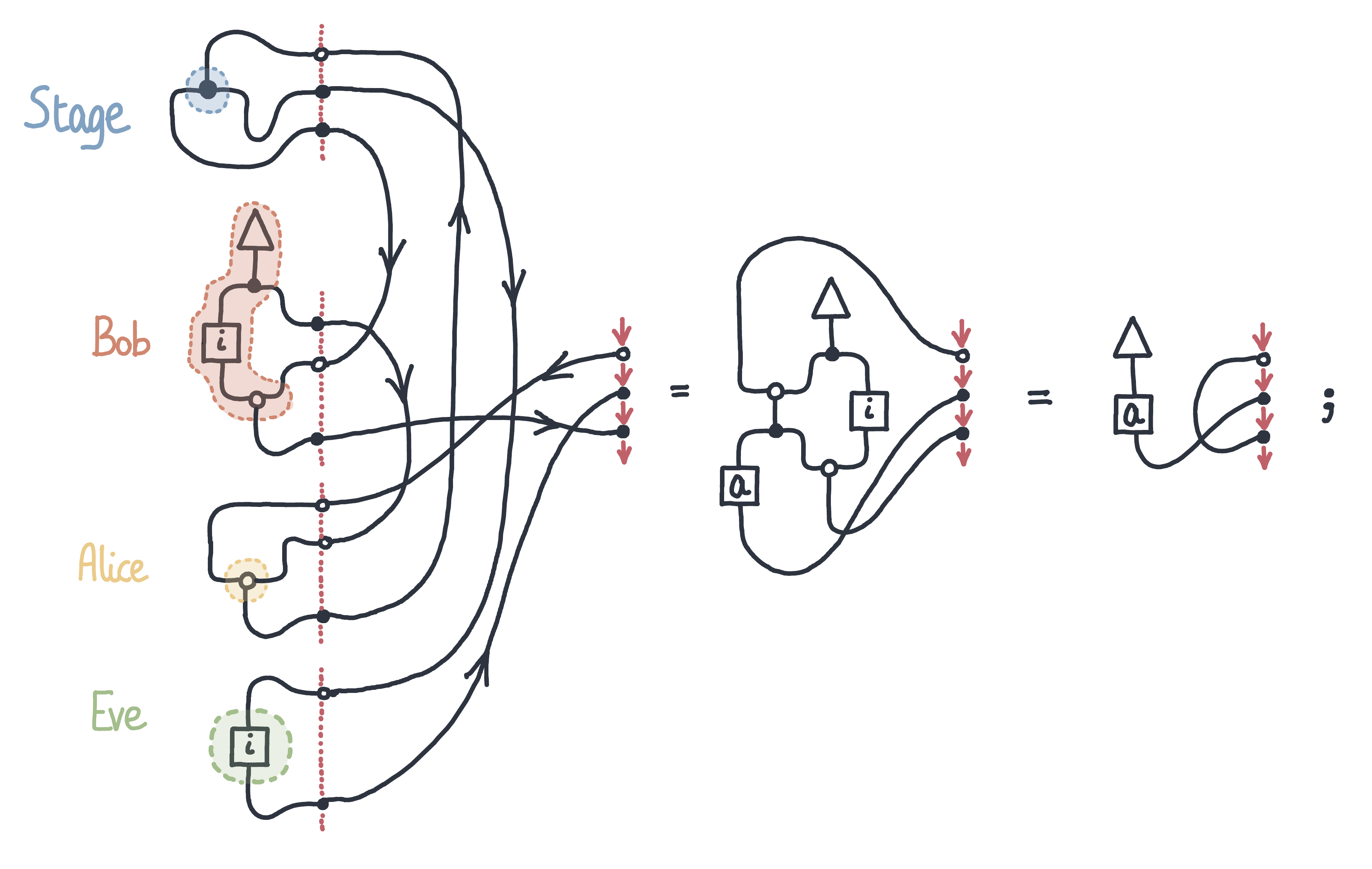}
  \caption{One-time pad protocol, split in four actors, mixed with a shuffle.}
  \label{fig:opt-sessionIntro}
\end{figure}

This means that the only addition to our process are two global effects (sending and receiving), that we depict using special red wires in the string diagrams. Each party in a session will have one of these red wires, and the logic of message passing allows us to combine them together. For instance, if the one-time pad protocol consists of a party (say, \Alice{}) sending a message to another party (say, \Bob{}), with an attacker (say, \Eve{}), sharing a \Stage{} that only allows broadcasting of messages; then these are four parties that connect together (\Cref{fig:opt-sessionIntro}).

\subsection*{Global Effects}
It remains then to explain the idea of global effects. Most imperative programming languages assume that there exist a \emph{global state} that the program affects. Full parallelism is not possible when two programs need to change this global state in a specific order: they could run into \emph{race conditions} \cite{huffman1954synthesis}.

However, mathematical theories of processes often assume no global state; processes do not interact with each other except when it is explicit that they do. This property is called \emph{purity} in some functional programming languages \cite{haskellreport} and that makes it easier to reason with them. The problem is that even pure functional programming languages need some techniques to change global state, and mathematical structures like \emph{monads} \cite{moggi91} or \emph{arrows} \cite{hughes00} achieve precisely this -- they take a pure theory and endow it with global effects.

Effects, monads and arrows create \premonoidalCategories{} \cite{power02,heunen06:arrows}. These are not \monoidalCategories{}, but Alan Jeffrey \cite{jeffrey1997:premonoidal} still introduced a string diagrammatic calculus for them: it is similar to the string diagrammatic calculus of monoidal categories, but it adds a red wire to control effects. This thesis proves that the extra red wire ensures a sound and complete graphical calculus for \premonoidalCategories{}.

\subsection*{Monoidal Context Theory}
All these ideas align to produce a theory of contexts, or incomplete processes, in \monoidalCategories{}. Each monoidal category can generate a \premonoidalCategory{} with the global effects of sending and receiving. The \stringDiagrams{} of this new \premonoidalCategory{} can be combined using the logic of \messageTheories{}, and in fact, they form the free \messageTheory{} on top of the original process theory: we can use them to reason and decompose multi-party processes in arbitrary process theories.

\newpage

\section*{Overview}

\subsection*{Chapter 1: Process Theories}
\Cref{chapter:monoidalProcess} is an introduction to \monoidalCategories{} and their string diagrammatic syntax. \Cref{sec:monoidalCategories} defines \strictMonoidalCategories{} in terms of process theories and introduces their string diagrams. \Cref{sec:symmetricMonoidalCategories} defines their \emph{symmetric} counterpart and their type theory in terms of do-notation, while \Cref{sec:non-strict-monoidals,sec:strings-bicategories} extend \stringDiagrams{} to non-strict monoidal categories and bicategories, variants that we will employ later.

\Cref{sec:premonoidal-categories} is an introduction to \premonoidalCategories{} and \effectfulCategories{}. \Cref{sec:runtime} gives their string diagrammatic calculus and proves its soundness and completeness. Finally, \Cref{sec:linearity} studies linearity, copying and discarding in terms of \monoidalCategories{}. This concludes a basic treatment of processes in terms of \monoidalCategories{}.

\subsection*{Chapter 2: Context Theory}
\Cref{chapter:compositional-algebra} introduces \profunctors{}, in \Cref{sec-profunctors}, and \multicategories{}, in \Cref{sec-multicategories}, as the mathematical tools to analyze decomposition. \Profunctors{} provide a canonical equivalence relation, \dinaturality{}, that we use whenever we study decomposition; in fact, it brings us to consider \malleableMulticategories{} in \Cref{sec-malleable-multicategories}. \Cref{sec:splice-contour-adjunction} presents the splice-contour adjunction between a category and its \malleableMulticategory{} of incomplete terms, or contexts.

\subsection*{Chapter 3: Monoidal Context Theory}
\Cref{chapter:monoidal-context-theory} brings context theory to the monoidal setting. \Cref{section-duoidal-categories} and \Cref{sec-normal-duoidal-categories} introduce \duoidalCategories{} and \normalDuoidalCategories{}. The duoidal counterpart of \malleableMulticategories{} are \produoidalCategories{} and we introduce their splice-contour adjunction in \Cref{sec:produoidalDecomposition}. The idempotent normalization monad of \produoidalCategories{} is constructed in \Cref{sec:normalization}, and it is used in \Cref{sec:monoidal-lenses} to normalize monoidal spliced arrows and obtain a universal characterization of \monoidalLenses{}.

\subsection*{Chapter 4: Monoidal Message Passing}
\Cref{chapter:monoidal-message-passing} starts defining \messageTheories{} in \Cref{sec:message-theories}. \Cref{sec:physical-monoidal-multicategories-shufflings} studies its categorical semantics in terms of \physicalMonoidalMulticategories{}. \Cref{sec:polarization} introduces polarization and opens the way for \Cref{sec:polar-shuffles} to define \polarShuffles{} and prove that they form a free polarized \monoidalMulticategory{}.
\Cref{sec:processes-sessions} constructs an adjunction between process theories and \messageTheories{}.

\begin{figure}[ht]
  \centering
  \includegraphics[scale=0.25]{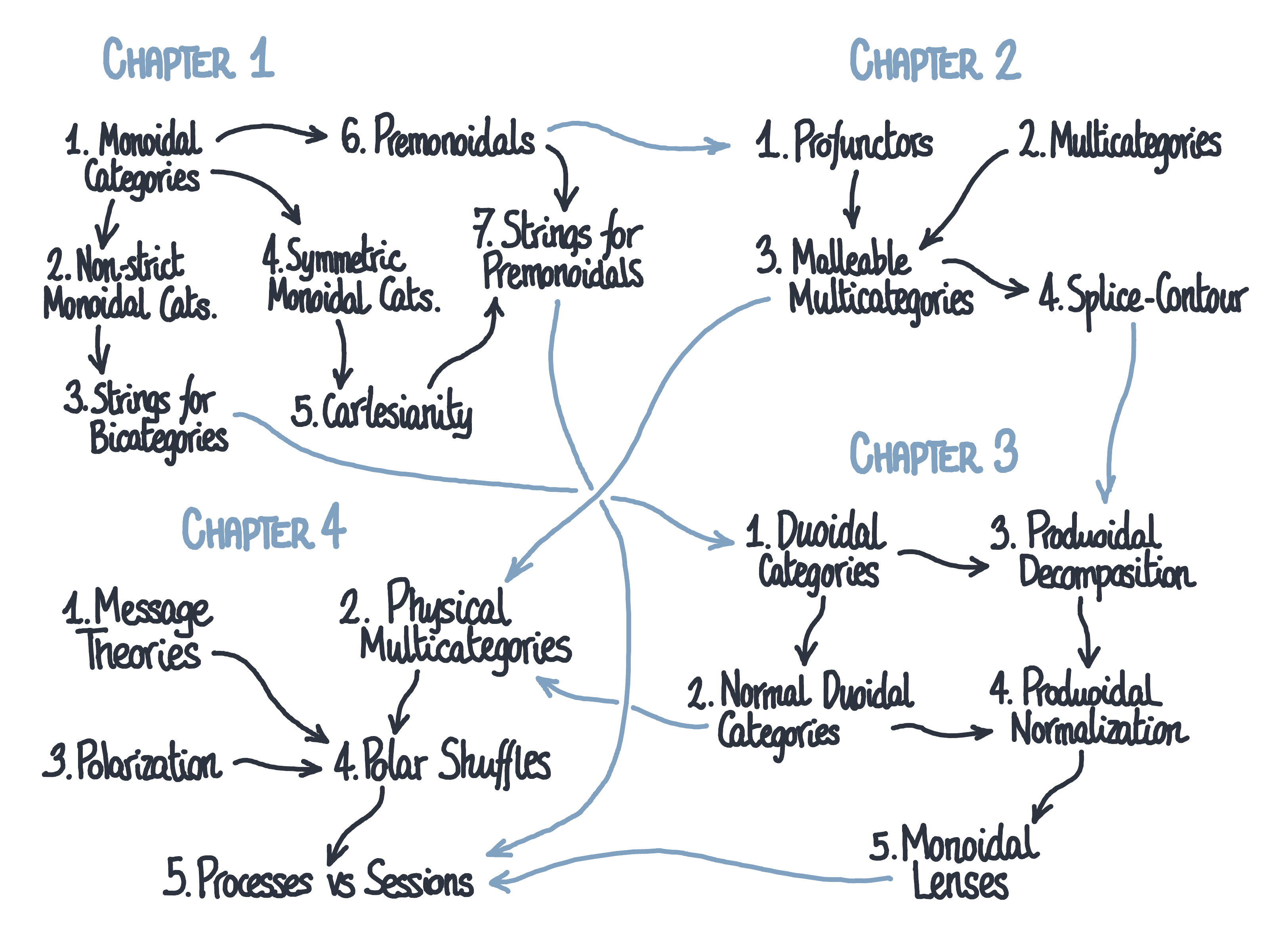}
  \caption{Chapter dependencies.}
  \label{fig:chapter-dependencies}
\end{figure}

\newpage
\section*{Contributions}

The main results of this thesis are \Cref{th:monoidalContextsAreANormalization} and  \Cref{th:sessions-vs-processes}. They universally characterize, in two different ways, the produoidal structure of incomplete diagrams: the former is used for a theory of monoidal context, the latter is used for message passing.

The definition of \emph{\messageTheory{}} (\Cref{def:messagetheory,def:messagefunctor,prop:soonerlater}) is novel. There does not seem to be literature specifically on \physicalMonoidalMulticategories{} (\Cref{def:physical-monoidal-multicategory}) nor on the observation that shuffles form the free one (\Cref{prop:shuffling-free-physical-monoidal-multicategory}) -- even when, admittedly, these are all variations on the idea of \physicalDuoidalCategories{} and an old result by Grabowski \cite{grabowski81:partial}. We give a different presentation of polarization in \monoidalCategories{} (\Cref{prop:polarization}), we discuss the problems of polarization in \monoidalCategories{} (\Cref{prop:shuffling-limit}) and we propose a solution describing polarization in \physicalMonoidalMulticategories{} (\Cref{def:physicalMonoidalMulticategoryDuals}). 
The definitions of \polarShuffle{} (\Cref{def:polarshuffle,def:polarlist}) and their \physicalMonoidalMulticategory{} (\Cref{th:polar:physicalmonoidal}) are new contributions, as it is its proposed characterization as a free polarized \physicalMonoidalMulticategory{} (\Cref{th:polar-shuffles-are-the-free-polarized-physical-monoidal-multicategory}). Our main contribution is the final adjunction between sessions and processes (\Cref{th:sessions-vs-processes}).

\DuoidalCategories{} are well-known, but we write down some observations about coherence in \Cref{prop:duoidalCoherencefails1} and we contribute the definition of the physical tensor (\Cref{def:physicaltensor}). Our main contribution is not only the monoidal splice-contour adjunction (\Cref{prop:produoidalSpliceContour}); the adjunctions between \produoidalCategories{} and \normalProduoidalCategories{}, and between \symmetricProduoidalCategories{} and \emph{physical} \produoidalCategories{}, with the construction of an idempotent monad \Cref{th:sym:freeNormalProduoidal,th:normalizationIdempotent}, are contributions to pure category theory.
\Cref{th:monoidalContextsAreANormalization} consitutes the first universal characterization of the whole \produoidalCategory{} of lenses.

Even when do-notation is well-known, a categorical treatment like the one in \Cref{th:donotation} seemed to be missing from the literature; it is based in an exposition of string diagrams that is unusual in that it takes adjunctions as the main construction (\Cref{th:pseudo-monoidal-categories}).
The string diagrams for \premonoidalCategories{} and \effectfulCategories{} are a new formalization (\Cref{theorem:runtime-as-a-resource}) that is detailed in other papers by this author \cite{roman:promonads-string-diagrams}.
We propose a new way of seeing coend calculus (\Cref{sec:pointedcoendcalculus}) that is used briefly in this thesis but that is more extensively explained in other papers by this author \cite{openDiagrams}.
The only contribution that we claim while translating the splice-contour adjunction to \promonoidalCategories{} is realizing their characterization as \malleableMulticategories{} (\Cref{prop:equivalencePromonoidalMalleable}), which is admittedly a new spin on the usual characterization as closed multicategories.

\subsection*{Literature}
The following is the list of publications authored or coauthored during the preparation of this thesis. 
As is customary in mathematics, we list authors in alphabetical order.
\begin{enumerate}
  \item Bryce Clarke, Derek Elkins, Jeremy Gibbons, Fosco Loregi{\`{a}}n, Bartosz
  Milewski, Emily Pillmore, and Mario Rom{\'{a}}n. \emph{Profunctor optics, a categorical update.} Accepted at {\em Compositionality}, preprint abs/2001.07488, 2020, \cite{ClarkeRoman20:ProfunctorOptics}.
  \item Mario Román. Open diagrams via coend calculus. {\em Applied Category Theory 2020. Electronic Proceedings in Theoretical Computer Science}, 333:65–78, Feb 2021, \cite{openDiagrams}.
  \item Guillaume Boisseau, Chad Nester, and Mario Román. Cornering optics. In {\em Applied Category Theory 2022}, Preprint abs/2205.00842, 2022, \cite{boisseaunester:corneringoptics}.
  \item Mario Rom{\'{a}}n. Promonads and string diagrams for effectful categories. In Jade Master and Martha Lewis, editors, {\em Proceedings Fifth
    International Conference on Applied Category Theory, {ACT} 2022, Glasgow,
    United Kingdom, 18-22 July 2022}, volume 380 of {\em {EPTCS}}, pages
    344--361, 2022, \cite{roman:promonads-string-diagrams}.
  \item Elena~Di Lavore, Alessandro Gianola, Mario Rom{\'{a}}n, Nicoletta Sabadini, and
  Pawel Sobocinski. A canonical algebra of open transition systems. In Gwen Sala{\"{u}}n and Anton Wijs, editors, {\em Formal Aspects of Component Software - 17th International Conference, {FACS} 2021, Virtual Event, October 28-29, 2021, Proceedings}, volume 13077 of {\em Lecture Notes in Computer Science}, pages 63--81. Springer, 2021, \cite{diLavore21:feedback}.
  \item Elena~Di Lavore, Alessandro Gianola, Mario Rom{\'{a}}n, Nicoletta Sabadini, and
  Pawel Sobocinski. Span(graph): a canonical feedback algebra of open transition systems. {\em Softw. Syst. Model.}, 22(2):495--520, 2023 \cite{diLavore:spangraph}.
  \item James Hefford and Mario Rom{\'{a}}n. Optics for premonoidal categories. {\em Applied Category Theory 2023}, abs/2305.02906, 2023 \cite{hefford23:optics-premonoidal}.
  \item Elena Di Lavore, Giovanni de Felice, and Mario Román. Monoidal streams for dataﬂow programming. In Proceedings of the \emph{37th Annual ACM/IEEE Symposium on Logic in Computer Science, LICS ’22}, New York, NY, USA, 2022. Association for Computing Machinery. Kleene Award to the best student paper.
  \item Dylan Braithwaite and Mario Román. Collages of string diagrams. {\em Applied Category Theory 2023, preprint arXiv:2305.02675}, 2023 \cite{braithwaite23:collages}.
  \item Elena~Di Lavore and Mario Rom{\'{a}}n. Evidential decision theory via partial Markov categories. In {\em {Logic In Computer Science (LICS'23)}}, pages 1--14, 2023 \cite{dilavore:evidentialdecision}.
  \item Matt Earnshaw, James Hefford, and Mario Román. The Produoidal Algebra of Process Decomposition, 2023. \emph{In Peer-Review}, \cite{produoidal23}. 
\end{enumerate}

\emph{The Produoidal Algebra of Process Decomposition} is the main unpublished work (currently in peer-review) that guides the writing of the main chapter of this thesis. It develops the universal characterization of monoidal lenses and forms the basis of \Cref{chapter:monoidal-context-theory} and \Cref{chapter:monoidal-message-passing}.
\emph{Promonads and String Diagrams for Effectful Categories}, adapted, was used as the basis of \Cref{sec:premonoidal-categories} and \Cref{sec:runtime}.

\clearpage{}%

\chapter{Monoidal Process Theory}
\label{chapter:monoidalProcess}
\section*{Monoidal Process Theory}
This chapter gives an overview of \monoidalCategories{}, their variants and their syntaxes. \MonoidalCategories{} are our framework of choice for \emph{process theories}: we claim that the minimalistic axioms of \monoidalCategories{} capture what a process theory is and we assume them for the rest of the thesis.

\Cref{sec:monoidalCategories} recalls \monoidalCategories{} and their \stringDiagrams{}. \Cref{sec:non-strict-monoidals} shows that the same axioms and syntax apply to \emph{non-strict} \monoidalCategories{} and \Cref{sec:strings-bicategories} extends them to \bicategories{}, which we will briefly use later.
\Cref{sec:symmetricMonoidalCategories} presents our definitive notion of \processTheory{}: \symmetricMonoidalCategories{}. \SymmetricMonoidalCategories{} have two syntaxes that are not commonly presented together: a string diagrammatic syntax in terms of \hypergraphs{} and a term theoretic syntax -- Hughes' \emph{do-notation} \cite{hughes00}. We argue that these two syntaxes further justify \symmetricMonoidalCategories{} as a natural setting for processes.

There is a final concept that has been traditionally left out of \monoidalCategories{}: computational effects. We argue in \Cref{sec:runtime,sec:premonoidal-categories} that, far from being a problem that requires an extension of \monoidalCategories{}, as usually assumed, computational effects can still use the same diagrammatic syntax of \stringDiagrams{}. This will be crucial for the next chapters in \emph{message passing}: messages will constitute a computational effect, but our results in this chapter allow us to model them without having to leave the syntax of \monoidalCategories{}.

\clearpage{}%

\section{Monoidal Categories}
\label{sec:monoidalCategories}
\subsection{Strict Monoidal Categories}

\MonoidalCategories{} are an algebra of processes, with minimal axioms. The definition of \monoidalCategory{} -- and this thesis -- follow a particular tradition of conceptual mathematics: \emph{category theory}. Category theory aims to extract mathematical structures in an abstract and general form.
As one such structure, \monoidalCategories{} are permissive: process theories like quantum maps and Markov kernels form \monoidalCategories{} \cite{abramsky09:categoricalquantum,heunenvicary19:categoriesquantum,fritz:markov2020,cho:jacobs:disintegration2019}; and even relations among sets or the homomorphisms of modules over a ring form \monoidalCategories{} \cite{bonchi18,aluffi21:algebra}.
We start by reinterpreting MacLane’s axioms for a \monoidalCategory{} \cite{macLane71:workingMathematician} in terms of processes.

\begin{definition}
  \defining{linkStrictMonoidalCategory}{}
  \defining{linkMonoidalCategory}{}
  A \textbf{strict monoidal category} $ℂ$ consists of a monoid of \emph{objects}, or resources, $(ℂ_{obj}, ⊗, I)$, and a collection of \emph{morphisms}, or processes, $ℂ(X; Y)$, indexed by an input $X ∈ ℂ_{obj}$ and an output $Y ∈ ℂ_{obj}$. A \emph{strict monoidal category} is endowed with operations for the sequential and parallel composition of processes, respectively
  \begin{align*}
    (⨾) &፡ ℂ(X;Y) × ℂ(Y; Z) → ℂ(X;Z), \\
    (⊗) &፡ ℂ(X;Y) × ℂ(X';Y') → ℂ(X ⊗ X'; Y ⊗ Y'), 
  \end{align*}
  and a family of \emph{identity} morphisms, $\id_X ∈ ℂ(X;X)$. \StrictMonoidalCategories{} must satisfy the following axioms.
  \begin{enumerate}
    \item Sequencing is unital, $f ⨾ \id_Y = f$ and $\id_X ⨾ f = f$.
    \item Sequencing is associative, $f ⨾ (g ⨾ h) = (f ⨾ g) ⨾ h$.
    \item Tensoring is unital, $f ⊗ \id_I = f$ and $\id_I ⊗ f = f$.
    \item Tensoring is associative, $f ⊗ (g ⊗ h) = (f ⊗ g) ⊗ h$.
    \item Tensoring and identities interchange, $\id_A ⊗ \id_B = \id_{A ⊗ B}$.
    \item Tensoring and sequencing interchange, $$(f ⨾ g) ⊗ (f' ⨾ g') = (f ⊗ f') ⨾ (g ⊗ g').$$
  \end{enumerate}
\end{definition}

\begin{remark}[Process theories]
  \label{remark:process-theories}
  Objects are also known as \emph{types} or \emph{resources} \cite{coeckeFS16}. If $X$ and $Y$ are both resources, it is reasonable to assume their joint occurrence is also a resource, $X ⊗ Y$; this joining operation, called \emph{tensor} $(⊗)$, must be unital with the empty resource $I$.
  Morphisms represent \emph{transformations} or \emph{processes}. If we have a process transforming $X$ into $Y$ and a process transforming $Y$ into $Z$, we can \emph{sequence} them $(⨾)$ and create a process that transforms $X$ into $Z$. The process that does nothing, the identity $(\id)$, is neutral for sequential composition. Similarly, transforming $X$ into $Y$ and transforming $X'$ into $Y'$ gives a way of transforming the joint object $X ⊗ X'$ into $Y ⊗ Y'$. Whenever we accept these basic constructions and axioms, we end up with strict monoidal categories.
\end{remark}

Once we have accepted these basic axioms, the next sections develop a syntax for monoidal categories: \emph{string diagrams}. \StringDiagrams{} are an intuitive syntax for process that is sound and complete for the previous axioms.

\subsection{Some Words on Syntax}

What makes a mathematical syntax practical? Different syntaxes highlight different aspects of a proof, and we consider better those that make the more bureaucratic steps invisible. Syntaxes are an explicit construction of the free mathematical object with some algebraic structure; what makes them efficient is \emph{how} we construct them.

For instance, how to prove that, in a group, the inverse of a multiplication is the reversed multiplication of the inverses? Usually, we simply observe that $$(x ⋅ y) ⋅ (y^{-1} ⋅ x^{-1}) = \cancel{x} ⋅ \cancel{y ⋅ y}^{-1} ⋅ \cancel{x}^{-1} = e;$$ 
that is, a simple computation checks that each letter is cancelled by its inverse. But we could be more bureaucratic and argue that the correct proof is, actually,
\begin{align*}
  (x ⋅ y) ⋅ (y^{-1} ⋅ x^{-1})\quad 
  \overset{(i)}{=}\quad & x ⋅ (y ⋅ (y^{-1} ⋅ x^{-1})) \\
  \overset{(ii)}{=}\quad &  x ⋅ ((y ⋅ y^{-1}) ⋅ x^{-1}) \\
  \overset{(iii)}{=}\quad &  x ⋅ (e ⋅ x^{-1}) \\
  \overset{(iv)}{=}\quad &  x ⋅ x^{-1} \\
  \overset{(v)}{=}\quad &  e.
\end{align*}
This proof uses associativity (\emph{i, ii}), the definition of inverse (\emph{iii, v}), and unitality (\emph{iv}). 
What makes these two proofs different? We can argue that, implicitly, they are using different syntaxes, constructed in different ways \cite{shulman:catlog}.

The bureaucratic syntax implicitly assumes the tautological construction of a free group. The free group on a set is generated by the elements of the set, the binary multiplication $(⋅)$, the unit $(e)$, and the inverse unary operator $({}^{-1})$; then, it is quotiented by associativity, unitality, and the inverse axioms. Tautological constructions only allow bureaucratic proofs -- but we can do better.

How does one construct free objects non-tautologically? The usual strategy is to first show that some combinatorial structure possesses the desired algebraic structure (say, it forms a group with some selected elements). This combinatorial structure will be as simple as possible, will relegate most steps to computation, and will use minimal quotienting. The result that makes this recipe work is freeness: the fact that it defines an adjunction (say, there exists a unique map to any group with some elements). 

More concretely, in our example, we know of a better classical construction of the free group: reduced words. \emph{Reduced words} are lists containing some generators and their inverses. The only condition is that they cannot contain ocurrences of a generator followed by its inverse: they get automatically cancelled out.

\begin{definition}
  Given a set $A$, the \emph{reduced words} over it, $\mathsf{Word}(A)$, are lists of polarized elements of $A$ -- that is, $a$ or $a^{-1}$ for each $a ∈ A$ -- not containing the substrings $a a^{-1}$ or $a^{-1} a$ for any element $a ∈ A$. 
\end{definition}

\begin{definition}
  The multiplication of two reduced words is inductively defined: if the first word is empty, then the multiplication is defined to be the second, $e ⋅ w₂ = w₂$; however, if the first word consists of a letter and a word, $aw₁$ or $a^{-1}w₁$, then we consider two cases: we first compute $w₁ ⋅ w₂$ by induction; if this word starts by the inverse of the first letter, $a^{-1} ⋅ w'$ or $a ⋅ w'$, then they both reduce and the multiplication is $w' ⋅ w₂$, otherwise, we just append the first letter, $a(w₁ ⋅ w₂)$.
\end{definition}
\begin{remark}
  It is non-trivial to prove that this multiplication is associative: the effort we put in here is the ease we get in return every time we use the syntax. We spare the reader this proof and we focus only on showcasing the syntax.
\end{remark}

\begin{proposition}
  The inverse of a multiplication is the reversed multiplication of the inverses.
\end{proposition}
\begin{proof}
  Reduced words form a group, in fact, the free group over some generators. In the group of reduced words, $(xy) ⋅ (y^{-1}x^{-1}) = e$ holds by definition. Because of freeness, there is a unique group homomorphism mapping this equality to any two elements of any other group.
\end{proof}

The core of this argument has been to construct, combinatorially, a left adjoint $\mathsf{Words} ፡ \Set → \mathbf{Group}$ to the forgetful functor $\mathsf{Forget} ፡ \mathbf{Group} → \Set$. 
This thesis will use adjoints as a more compositional way to discuss syntax. Let us start with the first of these syntaxes: string diagrams for monoidal categories.

\subsection{String Diagrams of Strict Monoidal Categories}
\label{sec:string-diagrams-of-strict-monoidal-categories}

\MonoidalCategories{} have a sound and complete syntax in terms of \stringDiagrams{} \cite{joyal91:geometryOfTensorCalculus}, which is the one we will use during this text.
We may prefer the classical axioms of \monoidalCategories{} when proving that some category is indeed monoidal, but proving equalities in a \monoidalCategory{} is easier using deformations of \stringDiagrams{} -- we will not need to remember the formulas.
Accepting \stringDiagrams{} and deformations as a criterion for equality is equivalent to accepting the axioms of \strictMonoidalCategories{}: whenever we accept one, we accept the other.

\begin{figure}[ht]
  \centering
  \includegraphics[scale=0.45]{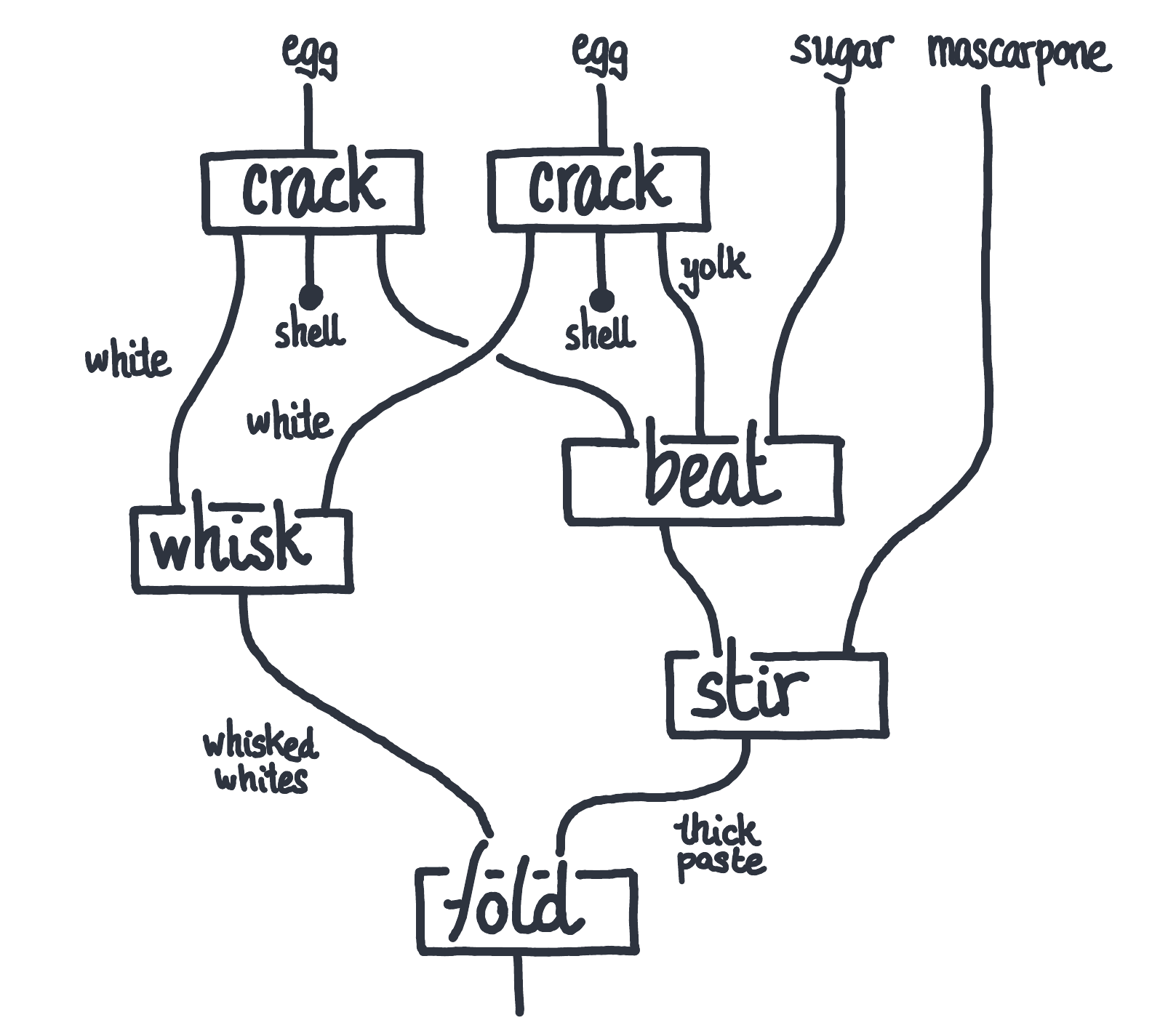}
  \caption{Process of preparing a Crema di Mascarpone, adapted from Sobocinski.}
  \label{fig:mascarpone}
\end{figure}

A first example of this syntax describing a process is in \Cref{fig:mascarpone} {{\cite{sobocinski13:graphicalLinearAlgebra}}}.
\StringDiagrams{} construct an adjunction between a category of \polygraphs{} and a category of \strictMonoidalCategories{}.

\begin{definition}
  \defining{linkPolygraph}{}
  A \emph{\polygraph{}}  $𝓖$ (analogue of a \emph{multigraph} \cite{shulman:catlog}) is given by a set of objects, $𝓖_{obj}$, and a set of arrows $𝓖(A_{0},\dots,A_{n};B_{0},\dots,B_{m})$ for any two sequences of objects $A_{0},\dots,A_{n}$ and $B_{0},\dots,B_{m}$.
  A morphism of \polygraphs{} $f ፡ 𝓖 \to 𝓗$ is a function between their object sets, $f_{\mathrm{o}} ፡ 𝓖_{obj} → 𝓗_{obj}$,
  and a family of functions between their corresponding morphism sets for any two sequences of objects
  \[f ፡ %
    𝓖(A_{0},\dots,A_{n};B_{0},\dots,B_{m}) →
    𝓗(f_{\mathrm{o}}A_{0},\dots,f_{\mathrm{o}}A_{n};f_{\mathrm{o}}B_{0},\dots,f_{\mathrm{o}}B_{m}).\]
  Polygraphs with polygraph homomorphisms form a category, $\PolyGraph$.
\end{definition}

\begin{definition}
  \defining{linkStrictMonoidalFunctor}
  A strict \emph{monoidal functor}, $F ፡ ℂ → 𝔻$, is a monoid homorphism between their object sets, $F_{obj} ፡ ℂ_{obj} → 𝔻_{obj}$, and an assignment taking any morphism $f ∈ ℂ(X; Y)$ to a morphism $F(f) ∈ 𝔻(FX; FY)$. A functor must preserve sequential composition, $F(f ⨾ g) = F(f) ⨾ F(g)$; parallel composition, $F(f ⊗ g) = F(f) ⊗ F(g)$; and identities, $F(\id) = \id$. \StrictMonoidalCategories{} with strict monoidal functors form a category, $\MonCatStr$.
\end{definition}

\begin{definition}
  \defining{linkStringDiagram}{}
  A \emph{string diagram} over a \emph{polygraph} $𝓖$ (or \emph{progressive plane graph} in the work of Joyal and Street \cite[Definition 1.1]{joyal91:geometryOfTensorCalculus}) is a graph $Γ$ embedded in the squared interval such that
  \begin{enumerate}
    \item the boundary of the graph touches only the top and the bottom of the square, $δ Γ \subseteq \{0,1\} × [0,1]$;
    \item and the second projection is injective on each component of the graph without its vertices, $Γ - Γ₀$; this makes it acyclic and progressive.
  \end{enumerate}
  We call to the components of $Γ - Γ₀$ \emph{wires}, $W$; we call the vertices of the graph \emph{nodes}, $Γ₀$.
  Wires must be labelled by the objects of the \polygraph{}, $o ፡ W → 𝓖_{obj}$, nodes must be labelled by the generators of the \polygraph{}, $m ፡ Γ₀ → 𝓖$; and each node must be connected to wires exactly typed by the objects of its generator -- a string diagram must be well-typed.
\end{definition}

\begin{lemma}
  String diagrams over a \polygraph{} $𝓖$ form a \monoidalCategory{}, which we call $\mathsf{String}(𝓖)$. This determines a functor, 
  $$\mathsf{String} ፡ \PolyGraph → \MonCatStr.$$
\end{lemma}
\begin{proof}[Proof sketch]
  The objects of the category are lists of objects of the \polygraph{}, which we write as $[X₀ , \dots , Xₙ]$, for $Xᵢ ∈ 𝓖_{obj}$. These form a (free) monoid with concatenation and the empty list. 
  
  Morphisms $[X₀ , \dots , Xₙ] → [Y₀ , \dots , Yₘ]$ are \stringDiagrams{} over the \polygraph{} $𝓖$ such that \emph{(i)} the ordered list of wires that touches the upper boundary is typed by $[X₀ , \dots , Xₙ]$, and \emph{(ii)} the ordered list of wires that touches the lower boundary is typed by $[Y₀ , \dots , Yₘ]$.
  \begin{figure}[ht]
    \centering
    \includegraphics[scale=0.45]{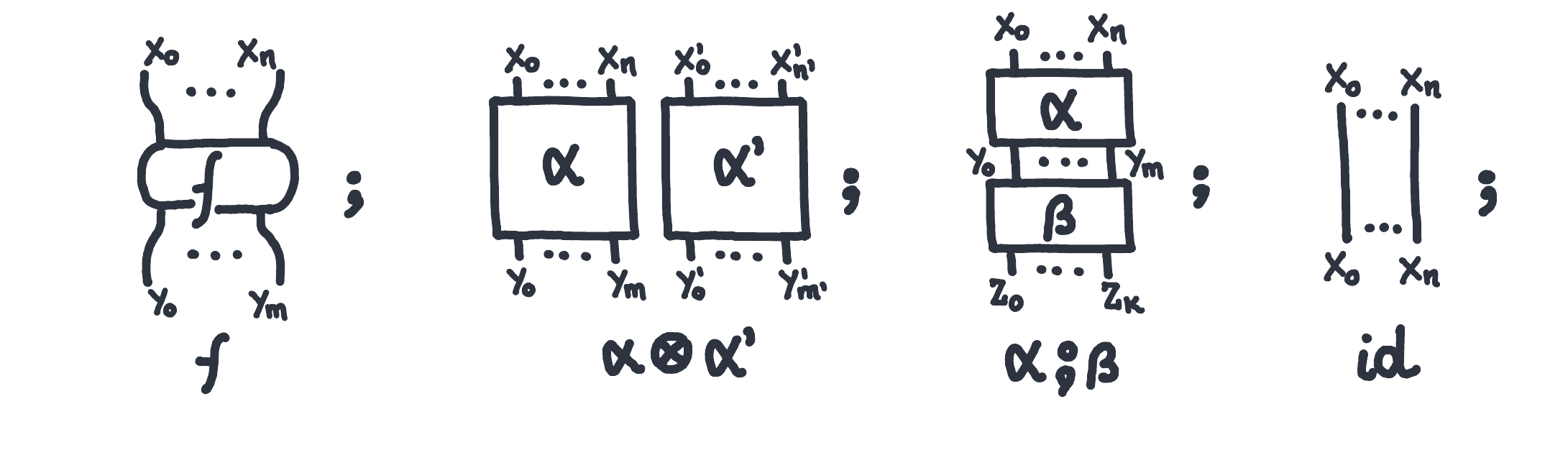}
    \caption{Strict monoidal category of string diagrams.}
    \label{fig:string-diagrams-mon-cat}
  \end{figure}

  \Cref{fig:string-diagrams-mon-cat} describes the operations of the category.
  The parallel composition of two diagrams $α ፡ [X₀ , \dots , Xₙ] → [Y₀ , \dots , Yₘ]$ and $α' ፡ [X'_0 , \dots , X'_{n'}] → [Y'_0 , \dots , Y'_{m'}]$ is their horizontal juxtaposition. 
  The sequential composition of two diagrams $α ፡ [X₀ , \dots , Xₙ] → [Y₀ , \dots , Yₘ]$ and $β ፡ [Y₀ , \dots , Yₘ] → [Z₀ , \dots , Zₖ]$ is the diagram obtained by vertical juxtaposition linking the outputs of the first to the inputs of the second.
  The identity on the object $[X₀ , \dots , Xₙ]$ is given by a diagram containing $n$ identity wires labelled by these objects. 
\end{proof}

\begin{lemma}
  Forgetting about the sequential and parallel composition defines a functor from \monoidalCategories{} to \polygraphs{}, $$\mathsf{Forget} ፡ \MonCatStr → \PolyGraph.$$
\end{lemma}
\begin{proof}
  Any \monoidalCategory{} $ℂ$ can be seen as a \polygraph{} $\mathsf{Forget}(ℂ)$ where the edges are determined by the morphisms, 
  $$\mathsf{Forget}(ℂ)(A_{0},\dots, A_{n};B_{0}, \dots, B_{m}) = ℂ(A₀ ⊗ … ⊗ Aₙ, B₀ ⊗ … ⊗ Bₘ),$$ 
  and we forget about composition and tensoring. It can be checked, by its definition, that any strict monoidal functor induces a homomorphism on the underlying \polygraphs{}.
\end{proof}

\begin{theorem}[Joyal and Street, {{\cite[Theorem 2.3]{joyal91:geometryOfTensorCalculus}}}]
  There exists an adjunction between \polygraphs{} and \strictMonoidalCategories{}, $\mathsf{String} ⊣ \mathsf{Forget}$.
  Given a \polygraph{} $𝓖$, the free \strictMonoidalCategory{} $\mathsf{String}(𝓖)$ is the \strictMonoidalCategory{} that has as morphisms the \stringDiagrams{} over the generators of the \polygraph{}; the underlying \polygraph{} determines the right adjoint.
\end{theorem}

\subsection{Example: Crema di Mascarpone}
  \label{sec:cremadimascarpone}
  This first example shows how to construct morphisms in a \monoidalCategory{}.
  The theory for preparing crema di mascarpone contains the following resources,
  $$\{ \mathsf{egg}, \mathsf{white}, \mathsf{yolk}, \mathsf{shell}, \mathsf{whisked\ white}, \mathsf{sugar}, \mathsf{mascarpone}, \mathsf{paste}, \mathsf{thick\ paste}, \mathsf{crema} \}.$$
  These resources are the objects of the \polygraph{} also containing the following seven generators, as in \Cref{fig:mascarpone-signature}.
  \begin{enumerate}
    \item $\crack ፡ \egg → \white ⊗ \shell ⊗ \yolk$,
    \item $\beat ፡ \yolk ⊗ \yolk ⊗ \sugar → \paste$,
    \item $\stir ፡ \paste ⊗ \mascarpone → \thickPaste$,
    \item $\whisk ፡ \white ⊗ \white → \whiskedWhites$,
    \item $\fold ፡ \whiskedWhites ⊗ \thickPaste → \cream$,
    \item $\swap ፡ \yolk ⊗ \white → \white ⊗ \yolk$,
    \item $\discard ፡ \shell → {\normalfont\textsc{i}}$.
  \end{enumerate}
  All these resources form a \polygraph{} $𝓒$. Thanks to the adjunction between \polygraphs{} and \monoidalCategories{}, deciding how to interpret the resources and the generators of the \polygraph{} in any \monoidalCategory{}, $𝓒 → \mathsf{Forget}(𝔻)$, is the same as creating a strict monoidal functor that interprets string diagrams in that category, $\mathsf{String}(𝓒) → 𝔻$. In particular, it interprets \Cref{fig:mascarpone} as a morphism in the \monoidalCategory{} $𝔻$.
  \begin{figure}[ht]
    \centering
    \includegraphics[scale=0.35]{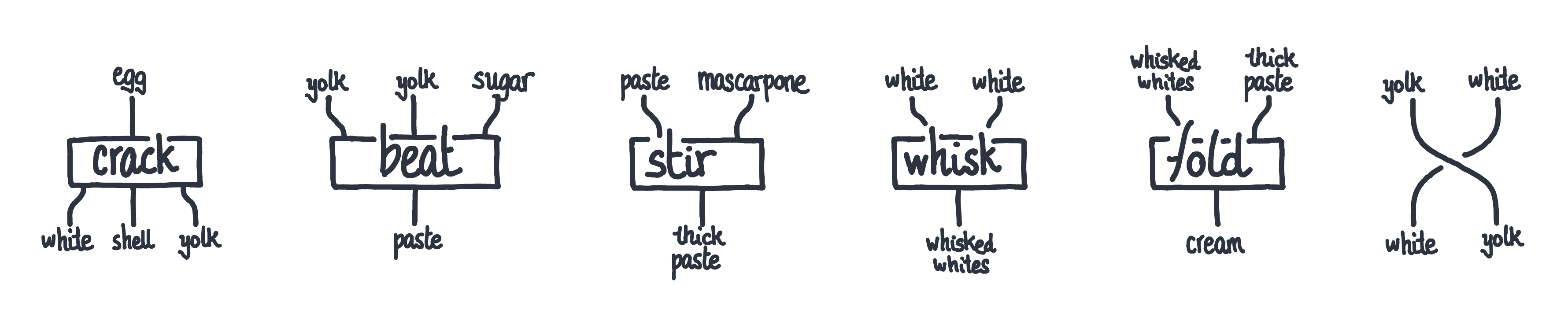}
    \caption{Polygraph for the theory of mascarpone.}
    \label{fig:mascarpone-signature}
  \end{figure}

  Usually, discarding and swapping are better understood as global structure with which all of the objects of the category are endowed. Even better than asking for the last generator in \Cref{fig:mascarpone-signature}, we will ask for the ability to copy or to swap resources freely. This is done via \cartesianMonoidalCategories{} or \symmetricMonoidalCategories{}, respectively.

\subsection{Bibliography}
  \MonoidalCategories{}, together with their coherence theorem, were first introduced by MacLane \cite{maclane63:natural,maclane78}, explicit equivalence theorems are given later by Joyal and Street \cite{joyal:braided}. \StringDiagrams{} as progressive graphs were introduced by Joyal \cite{joyal91:geometryOfTensorCalculus}. 
  Our presentation follows the \emph{resource theories} of Coecke, Fritz and Spekkens \cite{coeckeFS16}.
  
  Petri nets and process calculi are alternative mathematical approaches to \processTheory{}; at the same time, they are arguably particular cases of \monoidalCategories{} \cite{sobocinski10:representationspetrinets}.
  \MonoidalCategories{} as circuits and processes are explicitly pioneered by Sabadini, Walters, Carboni and Street \cite{sabadini95:bicategoriesOfProcesses,carboni87:cartesianbicategories}.
  The string diagrammatic recipe for ``crema di mascarpone'' is an adaptation of a blog post by Sobocinski \cite{sobocinski13:graphicalLinearAlgebra}.
  We also follow the ideas of Shulman on categorical logic \cite{shulman:catlog}.

  \newpage
  \clearpage{}%
\clearpage{}%
\section{Non-Strict Monoidal Categories}
\label{sec:non-strict-monoidals}

We have just argued for the axioms of \strictMonoidalCategories{} as our process theories and \stringDiagrams{} as our syntax. The bad news is that there exists a technicality preventing many interesting examples from ever forming a strict \monoidalCategory{}; the good news is that neither the axioms of \monoidalCategories{} nor our \stringDiagrams{} need to change at all to accommodate this technicality: this curious phenomenon is possible thanks to MacLane's \emph{strictification and coherence} results \cite{macLane71:workingMathematician}. This section introduces non-strict, or general \monoidalCategories{}, and it immediately details how MacLane's results ratify the axioms of \strictMonoidalCategories{}.

\subsection{Non-Strictness}
The axioms of \strictMonoidalCategories{} are enough to study a broader class of mathematical structures: non-strict monoidal categories.  

What is a non-strict monoidal category? In a \monoidalCategory{}, the tensor $(⊗)$ is not required to be associative or unital. Because of how we usually construct our definitions, it is not always the case that $X ⊗ (Y ⊗ Z) = (X ⊗ Y) ⊗ Z$. Consider the theory of sets and functions, with the cartesian product $(×)$ as the tensor. If we define
$$X × Y = \{ (x,y) \,\mid\, x ∈ X, y ∈ Y \},$$
then, simply, $X × (Y × Z) \neq (X × Y) × Z$. However, it is still the case that there exist two functions $X × (Y × Z) → (X × Y) × Z$ and $(X × Y) × Z → X × (Y × Z)$ and that these functions are mutual inverses for sequential composition. Whenever this happens in a category, we say these two objects are \emph{isomorphic} $(≅)$ and the morphisms are \emph{isomorphisms}. In a \monoidalCategory{}, the tensor is not associative and unital but it is still associative and unital up to an isomorphism. This may sound like a minor technicality, but it makes many examples fail to form a \monoidalCategory{}.

\begin{definition}[]
   \defining{linkmonoidalcategory}{}
   \defining{linkNonStrictMonoidalCategory}{}
   A \textbf{monoidal category}, $(ℂ, ⊗, I, α, λ, ρ)$, is a category $ℂ$ equipped with a pair of functors 
   $(⊗) ፡ ℂ × ℂ → ℂ,\mbox{ and } I ∈ ℂ$, called tensor and unit respectively. and three families of isomorphisms called the \emph{coherence maps}:
   \begin{enumerate}
     \item the associator $α_{X,Y,Z} ፡ (X ⊗ Y) ⊗ Z ≅ X ⊗ (Y ⊗ Z)$, 
     \item the left unitor $λ_{A} ፡ I ⊗ A ≅ A$ and 
     \item the right unitor $ρ_{A} ፡ A ⊗ I ≅ A$.
   \end{enumerate}
   These three families of maps must be natural, meaning they commute with other well-typed morphisms of the monoidal category. Moreover, these must be such that every formally well-typed equation between coherence maps holds.
\end{definition}

\begin{proposition}
  A \strictMonoidalCategory{} is precisely a \monoidalCategory{} where $α$, $λ$, and $ρ$ are identities.
\end{proposition}
\begin{proof}
  The naturality of the coherence maps, whenever these are identities, is the same as the associativity and unitality of the tensor (3 and 4). Functoriality of the tensor is the same as the interchange axioms (5 and 6); while the functoriality of the unit is trivially true.
\end{proof}

\subsection{Coherence}

The conditions of the definition of a \nonStrictMonoidalCategory{} may seem too strong: we are asking that a wide family of equations (all the formally well-typed ones) hold. Fortunately, the \emph{coherence theorem} shows that simply checking two families of equations is enough.

\begin{theorem}[MacLane, {{\cite{macLane71:workingMathematician}}}]
  The three different families of coherence maps $(α,λ,ρ)$ satisfy all formally well-typed equations between them whenever they satisfy the \emph{triangle} and \emph{pentagon} equations: 
  \begin{enumerate}
    \item $α_{X,I,Y} ⨾ (\id_{X} ⊗ λ_{Y}) = ρ_{X} ⊗ \id_{Y}$, and 
    \item $(α_{X,Y,Z} ⊗ \id) ⨾ α_{X,Y ⊗ Z, W} ⨾ (\id_X ⊗ α_{Y,Z,W}) = α_{X ⊗ Y,Z,W} ⨾ α_{X,Y,Z ⊗ W}.$
  \end{enumerate}
\end{theorem}

Finally, the theorem that allows \strictMonoidalCategories{} to talk about non-strict monoidal categories is the \emph{strictification theorem} that says that any \monoidalCategory{} is monoidally equivalent to a strict one. This does determine an adjunction with extra structure, making it a 2-adjunction \cite{campbell19:strictification}.

\begin{theorem}[Joyal and Street, {{\cite{joyal91:geometryOfTensorCalculus}}}]
  \label{th:equivalentStrictOne}
  Any \monoidalCategory{} is equivalent via a strong monoidal functor to a strict one. There exists a 2-adjunction between the category of monoidal categories and strong monoidal functors and the category of strict monoidal categories with strict monoidal functors. Moreover, the unit of this 2-adjunction is an equivalence.
\end{theorem}

\subsection{String Diagrams of Monoidal Categories}
While it is true that we can construct the free \strictMonoidalCategory{} on a \polygraph{}, it is not true yet that we know how to construct the free \monoidalCategory{} (the non-strict one) over a \polygraph{}; in fact, this seems to be impossible.
This could be misread as saying that string diagrams are not equally sound and complete for \monoidalCategories{}. Nothing is further from the truth: even if there are now multiple ways of interpreting a \stringDiagram{} in a \monoidalCategory{}, these are essentially equal -- they define isomorphic functors.

\begin{theorem}
  \label{th:pseudo-monoidal-categories}
  There is a pseudoadjunction between the locally discrete 2-category of polygraphs and the 2-category of monoidal categories, strong monoidal functors and monoidal natural transformations.
\end{theorem}
\begin{proof}[Proof sketch]
  This pseudoadjunction arises as a combination of two different 2-adjunctions. The first one is the adjunction between \polygraphs{} and \strictMonoidalCategories{} we have studied before. The second one is described in \Cref{th:equivalentStrictOne}, and its unit defines an equivalence: every \monoidalCategory{} is equivalent to a strict one.

  \begin{figure}[ht]
    \centering
    \includegraphics[scale=0.35]{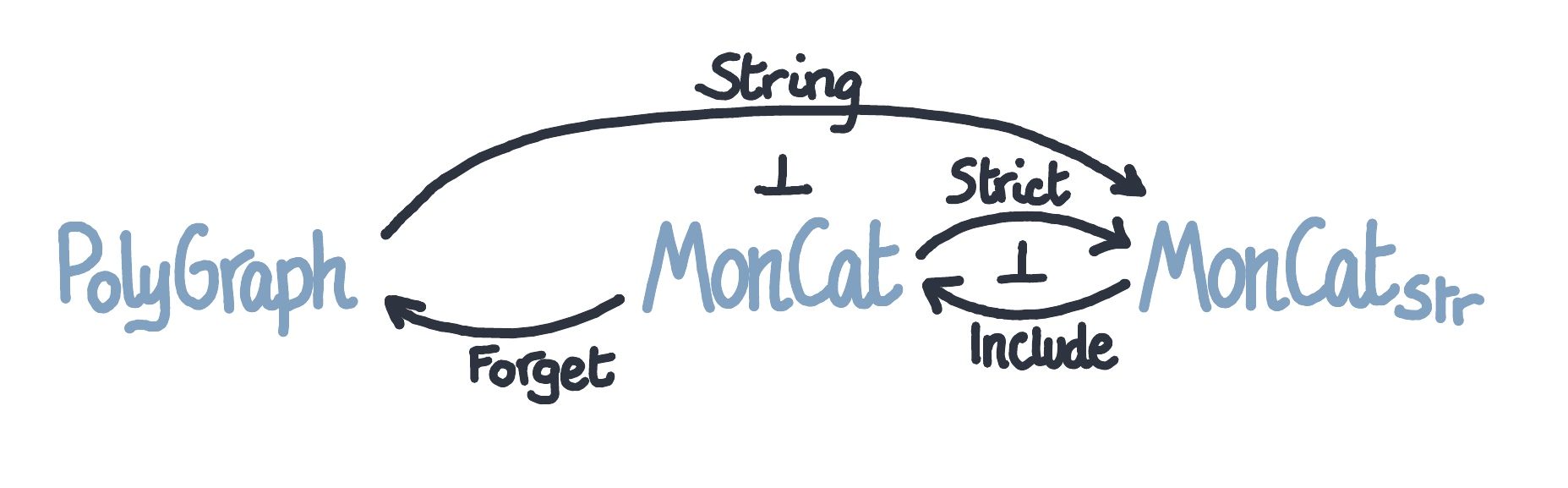}
    \caption{Pseudoadjunctions between polygraphs and monoidal categories.}
    \label{fig:pseudoadjunctions-monoidal}
  \end{figure}

  It is well known that each time that we have two adjunctions in this disposition we can reduce one along the other, provided that the unit of the former is invertible (see \Cref{prop:reducing}). In this case, the unit is not exactly invertible but merely an equivalence: as a consequence, we obtain not another 2-adjunction but merely a pseudoadjunction. This concludes the proof.
\end{proof}

\subsection{Bibliography}
The coherence results go back to MacLane, Joyal and Street \cite{maclane63:natural,joyal91:geometryOfTensorCalculus}; Hermida arrived at the same result via \multicategories{} \cite{hermida01:coherent,hadzihasanovic18:weak}; and there is a more modern account by Becerra detailing the 2-adjunction between strict and non-strict \monoidalCategories{} \cite{becerra23:strictification}, which Campbell studies for \bicategories{} \cite{campbell19:strictification}. 

\newpage
\clearpage{}%
\clearpage{}%

\newpage
\section{String Diagrams of Bicategories}
\label{sec:strings-bicategories}

Bicategories are the second extension of \monoidalCategories{} that we will employ during the text. If \monoidalCategories{} were well-suited to reason about process theories, \bicategories{}, one level up, are well-suited to reason about categories themselves; however, their string diagrammatic syntaxes are very close. 

String diagrams of \monoidalCategories{} can be easily extended to \bicategories{} if we allow ourselves to color the regions. Coloring the regions simply constrains which objects can be tensored: this algebraic structure is a \bicategory{}, also known as a \emph{weak 2-category}. In a \bicategory{}, two objects must coincide along a boundary to be tensored.

\begin{definition}
  \defining{linktwocategory}{}
  A strict \emph{2-category} $𝔹$ consists of a collection of \emph{objects}, or 0-cells, $𝔹_{obj}$, and a category of \emph{morphisms} or 1-cells between any two objects, $𝔹(A;B)$. A strict \emph{2-category} is endowed with operations for the parallel composition of 1-cells,
  \begin{align*}
    (\bcomp) &፡ 𝔹(A; B) × 𝔹(B; C) → 𝔹(A; C), \\
    (I_A) &፡ 𝔹(A;A),
  \end{align*}
  that are associative and unital both on objects and morphisms, meaning that
  $(X \bcomp Y) \bcomp Z = X \bcomp (Y \bcomp Z)$, and $I_A \bcomp X = X = X \bcomp I_B$.
  Bicategories must satisfy the following axioms, making parallel composition a functor:
  \begin{enumerate}
    \item parallel composition is unital, $f \bcomp \id = f$, and $\id \bcomp f = f$;
    \item parallel composition is associative, $f \bcomp (g \bcomp h) = (f \bcomp g) \bcomp h$;
    \item compositions are unital, $\id \bcomp \id = \id$;
    \item compositions interchange, $(f ⨾ g) \bcomp (f' ⨾ g') = (f \bcomp f') ⨾ (g \bcomp g')$.
  \end{enumerate}
\end{definition}
\begin{remark}
  A single-object strict 2-category is exactly a \strictMonoidalCategory{}.
\end{remark}

\subsection{String diagrams of 2-categories}
Let us briefly comment on how the \stringDiagrams{} of \monoidalCategories{} extend to \bicategories{}. We repeat the same definitions and the same theorems, just taking care of matching the boundaries this time.

\begin{definition}
  A \emph{bigraph}, or \emph{2-graph}, $𝓑$ is given by a set of objects, $𝓑_{obj}$; a set of arrows between any two objects, $𝓑(A; B)$; and a set of 2-arrows between any two paths of arrows, $𝓑(X₀,…,Xₙ; Y₀,…,Yₘ)$.

  A bigraph homomorphism, $f ፡ 𝓐 → 𝓑$, is a function between their object sets, $f_o ፡ 𝓐_{obj} → 𝓑_{obj}$; a family of functions between their corresponding arrow sets, $f ፡ 𝓐(A;B) → 𝓑(f(A),f(B))$; and a family of functions between their corresponding 2-arrow sets,
  $$f ፡ 𝓐(X₀,…,Xₙ; Y₀,…,Yₘ) → 𝓑(fX₀,…,fXₙ; fY₀,…,fYₘ).$$
  Bigraphs with bigraph homomorphisms form a category, $\mathbf{BiGraph}$.
\end{definition}

\begin{definition}
  A string diagram over a bigraph $𝓐$ is a string diagram over the polygraph formed by arrows and 2-arrows, additionally satisfying that each region is labelled by an object of the bigraph, and in such a way that any wire is labelled by an arrow connecting the labels of the two regions.
\end{definition}

\begin{theorem}
  There is an adjunction between bicategorical graphs and 2-categories with strict 2-functors between them.
  The left adjoint is given by colored string diagrams over the bigraph.
\end{theorem}

\subsection{Bicategories}
Strict monoidal categories have a weak analogue that still shares the same syntax -- (weak, or non-strict) monoidal categories. In the same way, strict 2-categories have a weak analogue that shares the same syntax: weak 2-categories, sometimes called bicategories.

\begin{definition}
  \defining{linkBicategory}{}
  A \emph{bicategory} $(𝔹,\bcomp,I,α,λ,ρ)$ is a collection of 0-cells, $𝔹_{obj}$, together with a category $𝔹(A;B)$ between any two 0-cells, $A,B ∈ 𝔹_{obj}$, and functors 
  \begin{align*}
    (\bcomp) ፡ 𝔹(A; B) × 𝔹(B; C) → 𝔹(A; C), \mbox{ and } I_A ፡ 𝔹(A;A),
  \end{align*}
  that are associative and unital up to isomorphism, meaning that there exist natural isomorphisms describing associativity $α_{X,Y,Z} ፡ (X ⊗ Y) ⊗ Z ≅ X ⊗ (Y ⊗ Z)$, left unitality $λ_{X} ፡ I_{A} ⊗ X ≅ X$ and right unitality $ρ_{X} ፡ X ⊗ I_{B} ≅ B$.
\end{definition}

\begin{conjecture}
  There is a pseudoadjunction between the locally discrete 2-category of bicategorical graphs and the 2-category of bicategories, pseudofunctors and icons (see Campbell, Garner and Gurski's work for the higher structure of the strictification adjunction \cite{campbell19:strictification,garner09low}).
\end{conjecture}

\subsection{Example: Adjunctions}
We exemplify the usage of string diagrams for bicategories in an abstract definition of adjunctions. We then use string diagrams to prove a theorem about adjunctions.

\begin{definition}
  The \emph{theory of a duality} in a \bicategory{} contains two 0-cells $A$ and $B$;
  it contains two 1-cells between them, $L ፡ A → B$ and $R ፡ B → A$, and it contains two 2-cells, $ε ፡ L ⨾ R → I$ and $η ፡ I → R ⨾ L$, that satisfy $(\id ⊗ η) ⨾ (ε ⊗ \id) = \id$ and $(η ⊗ \id) ⨾ (\id ⊗ ε) = \id$.
    \begin{figure}[!ht]
	  \centering
	  \scalebox{0.3}{\includegraphics{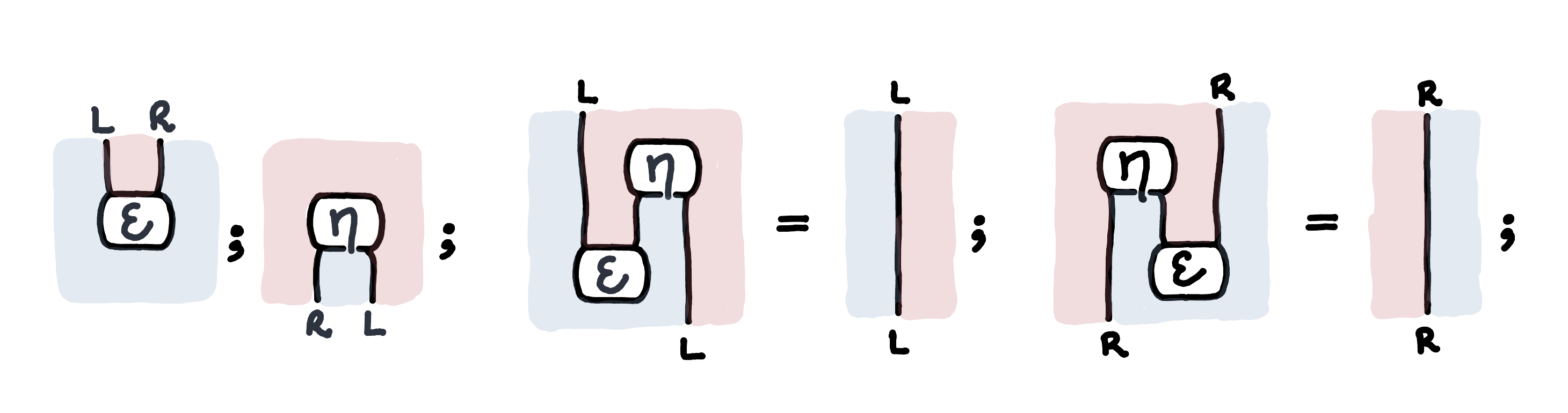}}
	  \caption{Theory of a duality.}
	  \label{fig:duality-theory}
	\end{figure}
\end{definition}

\begin{remark}
  Adjunctions are dualities in the bicategory of categories, functors, and natural transformations.  
\end{remark}

\begin{proposition}[Reducing an adjunction]
	\label{prop:reducing}
	Let $F ፡ 𝔸 \to ℂ$ and $H ⨾ U ፡ ℂ \to 𝔸$ determine an adjunction $(F,H ⨾ U,\eta,\varepsilon)$ and let $P ፡ 𝔹 \to ℂ$ determine a second adjunction $(P,H, u, c)$ such that the unit $u ፡ I \to P ⨾ H$ is a natural isomorphism (as in \Cref{fig:reducingadjoint1}). Then, $F ⨾ H$ is left adjoint to $U$.
	\begin{figure}[!ht]
	  \centering
	  \scalebox{0.3}{\includegraphics{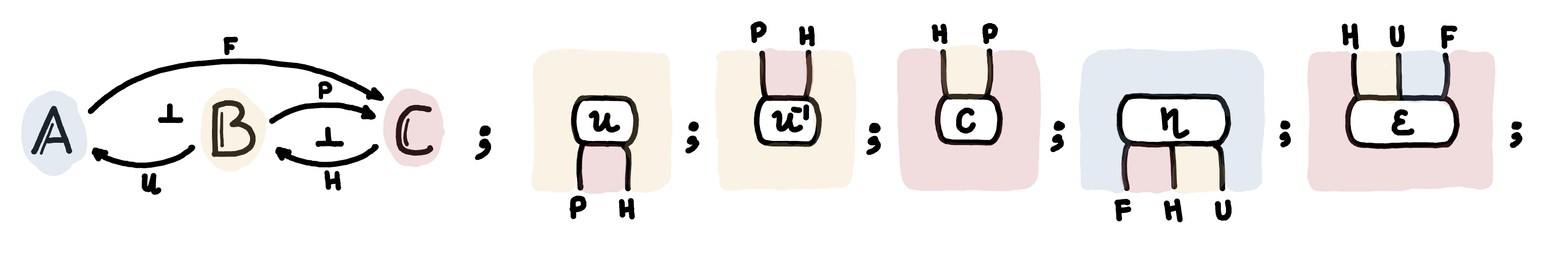}}
	  \caption{Setting for reducing an adjunction.}
	  \label{fig:reducingadjoint1}
	\end{figure}
  \end{proposition}
  \begin{proof}
	  We employ the string diagrammatic calculus of bicategories for the bicategory of categories, functors and natural transformations \cite{marsden14:category}. We define the morphisms in \Cref{fig:reducingunitcounit} to be the unit and the counit of the adjunction. We then prove that they satisfy the snake equations in \Cref{fig:snake1,fig:snake2}.
	  \begin{figure}[!ht]
		  \centering
		  \scalebox{0.3}{\includegraphics{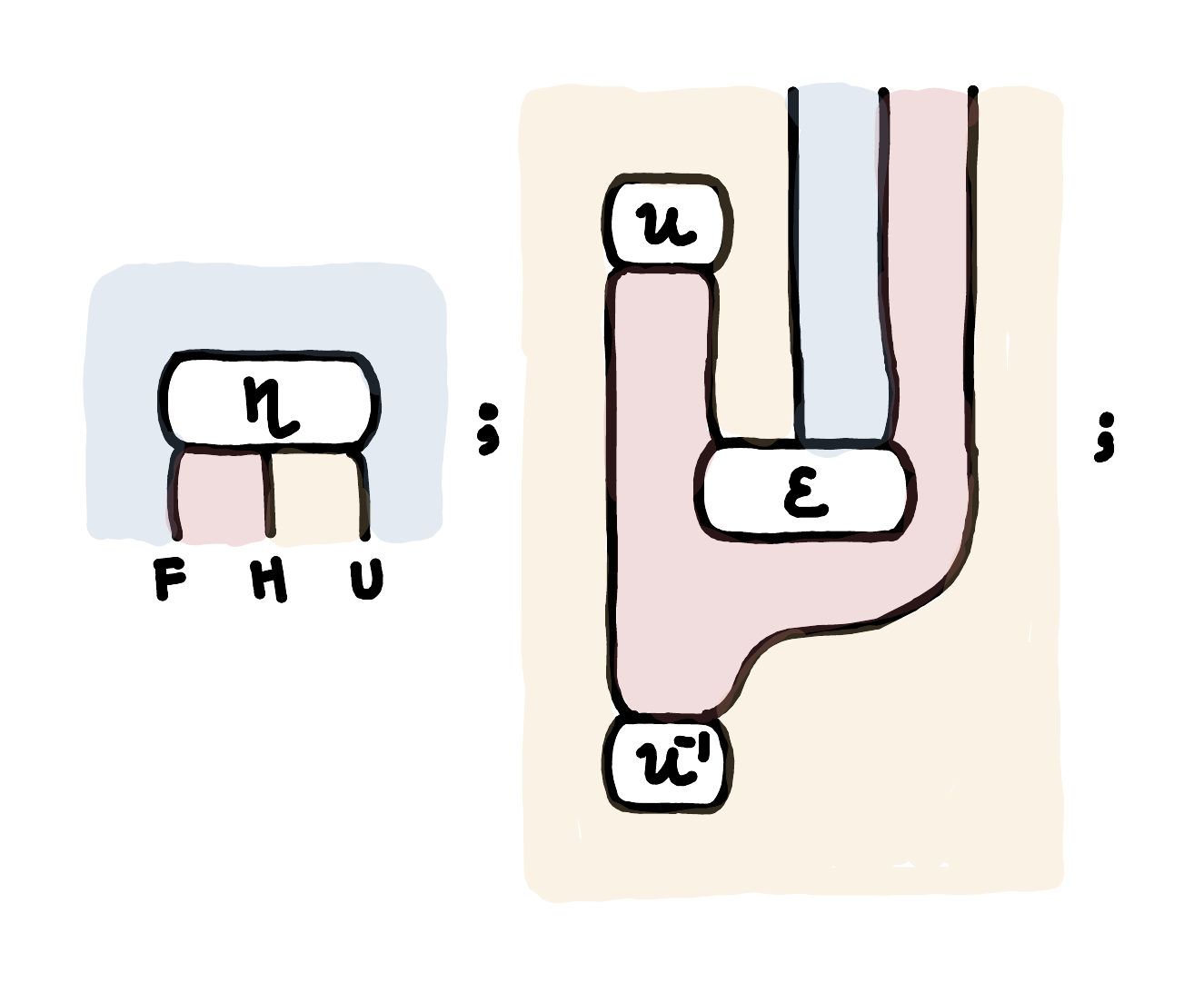}}
		  \qquad
		  \scalebox{0.3}{\includegraphics{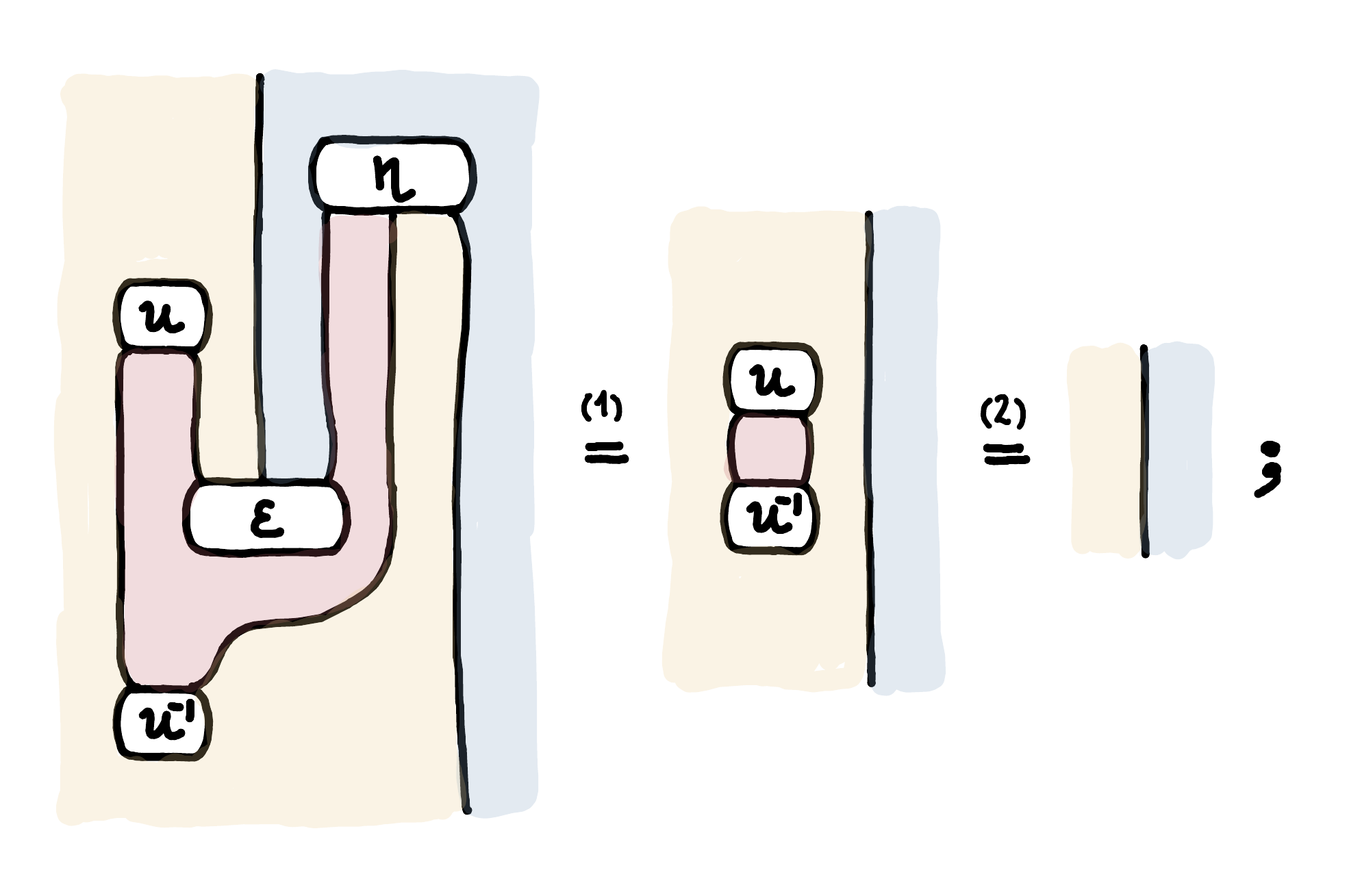}}
		  \caption{Unit and counit of the reduced adjunction 
		  (left). First snake equation (right).}
		  \label{fig:reducingunitcounit}
		  \label{fig:snake1}
	  \end{figure}

	  In the first snake equation, in \Cref{fig:snake1}, we use \emph{(i)} that there is a duality $(\eta,\varepsilon)$, and \emph{(ii)} that $u$ is invertible. In the second snake equation, in \Cref{fig:snake2}, we use \emph{(i)} that there is a duality $(u,c)$, \emph{(ii)} that $u$ is invertible, \emph{(iii)} that there is a duality $(u,c)$, again; and \emph{(iv)} that there is a duality $(\eta,\varepsilon)$. \qedhere
	  
	  \begin{figure}[!ht]
		  \centering
		  \scalebox{0.3}{\includegraphics{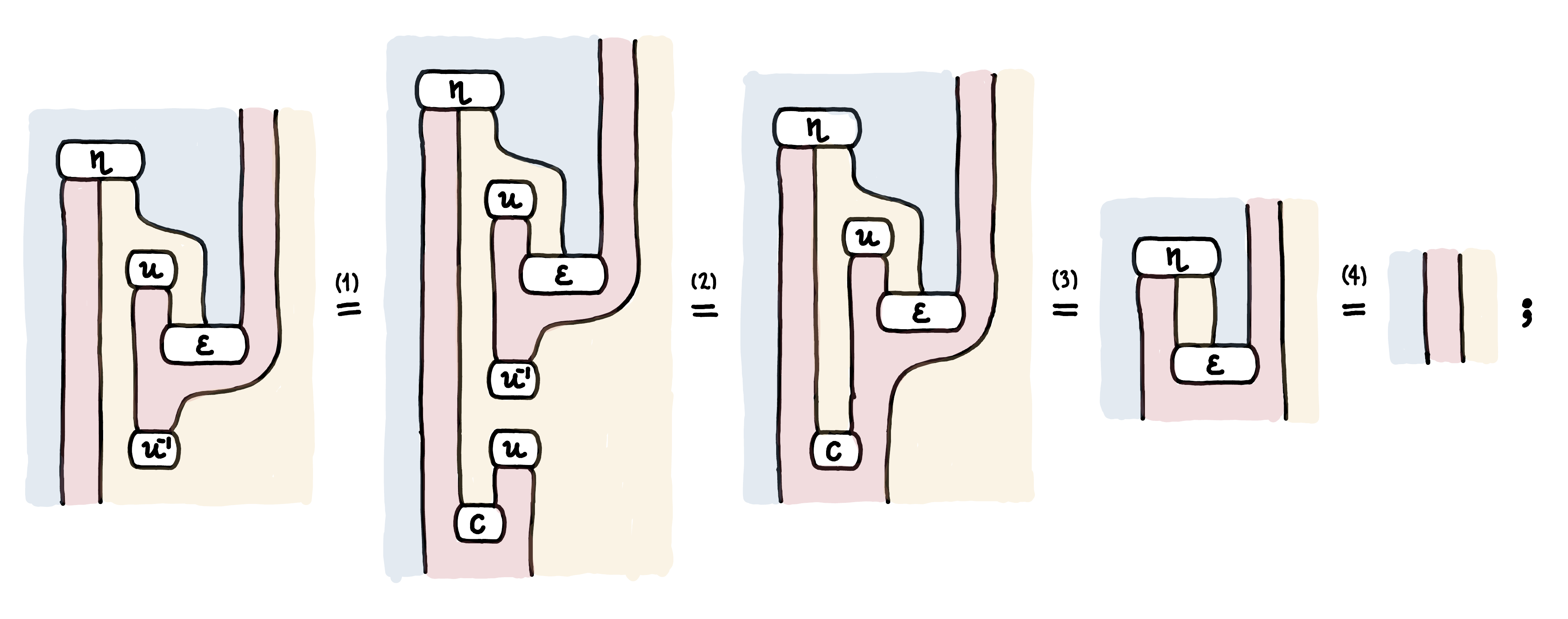}}
		  \caption{Second snake equation.}
		  \label{fig:snake2}
	  \end{figure}
  \end{proof}

\subsection{Bibliography}
String diagrams for bicategories are usually interpreted as the Poincaré dual of \emph{globular pasting diagrams}, which were used by Bénabou since the introduction of \bicategories{} \cite{benabou67}.
Marsden uses string diagrams for \bicategories{} to study basic formal category theory \cite{marsden14:category}.

\newpage
\clearpage{}%
\clearpage{}%
\section{Symmetric Monoidal Categories and Do-Notation}
\label{sec:symmetricMonoidalCategories}

Setting strictness asside, an extra axiom sharpens \monoidalCategories{} for the study of process theories: \emph{symmetry}. Symmetry states that the position that resources occupy is not important: $A ⊗ B$ is worth the same as $B ⊗ A$. Assuming symmetry simplifies process syntax drammatically -- in fact, it also enables a new syntax that mimics imperative programming -- and it is arguably an axiom that we need before we can really talk of \emph{processes}.

This section introduces \symmetricMonoidalCategories{}, their specialized string diagrams in terms of \hypergraphs{} \cite{bonchi16:rewriting}, and their do-notation syntax \cite{hughes00}.

\subsection{Commutative Monoidal Categories}
So far, the meaning of the tensor $(⊗)$ in process theories has arguably been too general: so far, we have asked $A ⊗ B$ to represent the juxtaposition of resources -- resources form a monoid. Why not a commutative monoid?
If we interpret objects as bags of resources, it seems clear that a commutative monoid would be more appropriate. However, imposing commutativity naively fails catastrophically: the individuality of each resource disappears \cite{meseguer90:petriaremonoids}.

\begin{definition}
  A \emph{commutative monoidal category} is a strict monoidal category $(ℂ,⊗,I)$ where objects form a commutative monoid, $A ⊗ A' = A' ⊗ A$ for each $A,A' ∈ ℂ_{obj}$, and the tensor of morphisms is also commutative, $f ⊗ f' = f' ⊗ f$ for any two $f ፡ A → B$ and $f' ፡ A' → B'$.
\end{definition}

\begin{proposition}
  In a commutative monoidal category, resources do not have individuality: it does not matter to which of them we apply a transformation, and not even the order in which we apply them. More formally, 
  $$(f ⨾ g) ⊗ \id = (f ⊗ g) = \id ⊗ (g ⨾ f)$$
  for each two transformations of the same resource, $f ፡ X → X$ and $g ፡ X → X$.
\end{proposition}

This may be useful in specific applications. Indeed, it is one of the crucial ideas behind the formalization of Petri nets as monoids in the work of Meseguer and Montanari \cite{meseguer90:petriaremonoids} -- commutativity represents the ``collective token philosophy'' of Petri nets \cite{baez21:categoriesnets}. However, for our purposes, we will need a more refined notion that does not destroy the individuality of our resources: this notion is given by \symmetricMonoidalCategories{}.

\subsection{Symmetric Monoidal Categories}
\defining{linkProcessTheory}{}
\emph{\SymmetricMonoidalCategories{}} do not assume that the monoid of objects is commutative; they only assume that there is a family of processes that allow us to reorder resources, making it commutative ``up to an invertible process''. In practice, this means that even when $X ⊗ Y \neq Y ⊗ X$, there exists a process $σ_{X,Y} ፡ X ⊗ Y → Y ⊗ X$ that is invertible. These are our definitive notion of \emph{process theory}.

\begin{definition}
  \defining{linkSymmetricMonoidal}{}
  A strict \emph{symmetric monoidal category} (or a \emph{permutative category}) is a strict \monoidalCategory{} $(ℂ, ⊗, I)$ endowed with a family of maps $σ_{X,Y} ፡ X ⊗ Y → Y ⊗ X$ that satisfy the following equations describing how to
  \begin{enumerate}
    \item swap nothing, $σ_{I,X} = \id_X = σ_{X,I}$;
    \item swap resources on the left, $σ_{X,Y ⊗ Z} = (σ_{X,Y} ⊗ \id_Z) ⨾ (\id_Y ⊗ σ_{X,Z})$;
    \item swap resources on the right, $σ_{X ⊗ Y, Z} =  (\id_X ⊗ σ_{Y,Z}) ⨾ (σ_{X,Z} ⊗ \id_Y)$;
    \item reverse a swap with a swap, $σ_{X,Y} ⨾ σ_{Y,X} = \id_X ⊗ \id_Y$;
    \item swap and apply transformations, $σ_{X,Y} ⨾ (f ⊗ g) = (g ⊗ f) ⨾ σ_{X',Y'}$;
    \item and swap in any order, 
    $$(σ_{X,Y} ⊗ \id) ⨾ (\id ⊗ σ_{X,Z}) ⨾ (σ_{Y,Z} ⊗ \id) = (\id ⊗ σ_{Y,Z}) ⨾ (σ_{X,Z} ⊗ \id) ⨾  (\id ⊗ σ_{X,Y}).$$
  \end{enumerate}
\end{definition}

The first three axioms are especially important for clarifying all the rest: they say that the swapping process of any two objects in a freely generated \monoidalCategory{} is determined by the swapping process of the generators. This already allow us to have a first string diagrammatic calculus for \symmetricMonoidalCategories{}: the swap on the generators is represented by wires crossing; the swap on arbitrary objects is constructed from it; the rest of the axioms are better understood in terms of string diagrams (\Cref{fig:permutative-axioms}).

  \begin{figure}[ht] 
    \centering
    \includegraphics[scale=0.35]{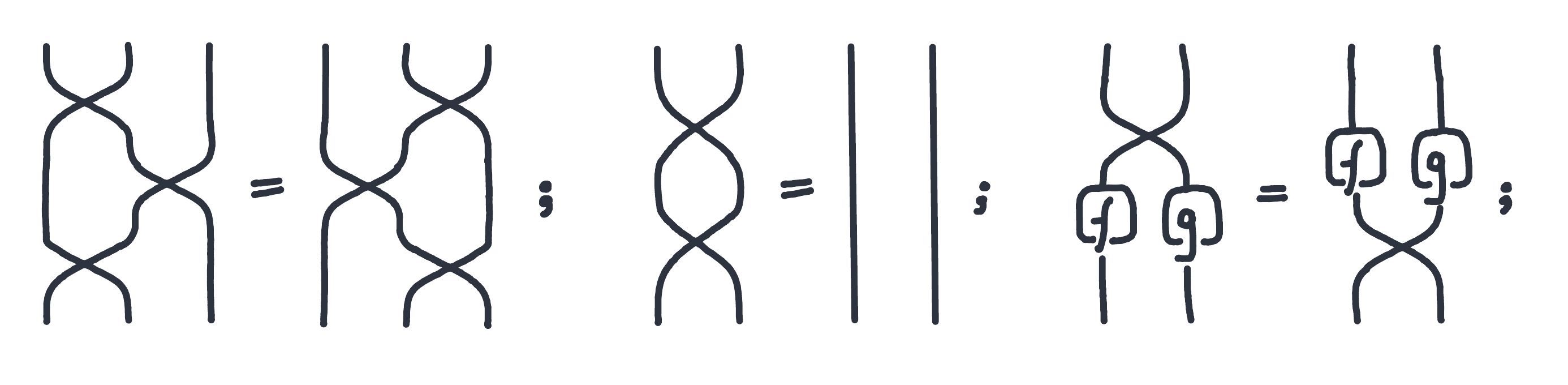} 
    \caption{Theory of strict symmetric monoidal categories.}
    \label{fig:permutative-axioms}
  \end{figure}

However, this is an inefficient syntax: it forces us to explicitly deal with the axioms of the swap. A better syntax would make them transparent, and state that the only thing that matters in a \stringDiagram{} for \symmetricMonoidalCategories{} is where the wires are ultimately connected -- that is, we only care about the underlying \hypergraph{}. A detailed presentation of the string diagrams of \symmetricMonoidalCategories{} as \hypergraphs{} is in the work of Bonchi, Sobocinski, Zanasi, and others \cite{bonchi14,bonchi16:rewriting}.

\begin{definition}
  \defining{linkHypergraph}{}
  A \emph{hypergraph} $(V,E)$ consists of a set of nodes, $V$, and a set of directed hyperedges $E$ connecting lists of vertices to lists of vertices, that is, $e ፡ [v₁,…,vₙ] → [w₁,…wₘ]$ for each $e ∈ E$. We say a \hypergraph{} is acylic if contains no loops.
\end{definition}

\begin{definition}
  A \emph{hypergraph labelled over a \polygraph{}} $𝓖$, is a \hypergraph{} $(V,E)$ such that each vertex $v ∈ V$ is assigned an object of the \polygraph{}, $l(v) ∈ 𝓖_{obj}$, and each hyperedge $e ፡ [v₁,…,vₙ] → [w₁,…wₘ]$ is assigned an edge 
  $$l(e) ፡ l(v₁),…,l(vₙ) → l(w₁),…,l(wₘ),$$ preserving the type of its vertices.
\end{definition}

\begin{definition}
  A \emph{symmetric string diagram} from $[X₁,…,Xₙ]$ to $[Y₁,…,Yₘ]$ is an acyclic \hypergraph{} $(V,E)$ labelled by a \polygraph{} $𝓖$, such that each vertex appears exactly once as an input and exactly once as an output, and endowed with two distinguished unlabelled hyperedges: 
  $$\mbox{the input }i ፡ [] → [x₁,…,xₙ], \mbox{ and the output } o ፡ [y₁,…,yₘ] → [],$$ 
  typed by $l(x₁) = X₁, … , l(xₙ) = Xₙ$ and $l(y₁) = Y₁, … , l(yₘ) = Yₘ$.
\end{definition}

\begin{proposition}
  \defining{linkSymmetricStringDiagram}{}
  Symmetric string diagrams over a \polygraph{} $𝓖$ form a \symmetricMonoidalCategory{}, $\StringSigma(𝓖)$.
\end{proposition}
\begin{figure}[ht] 
  \centering
  \includegraphics[scale=0.35]{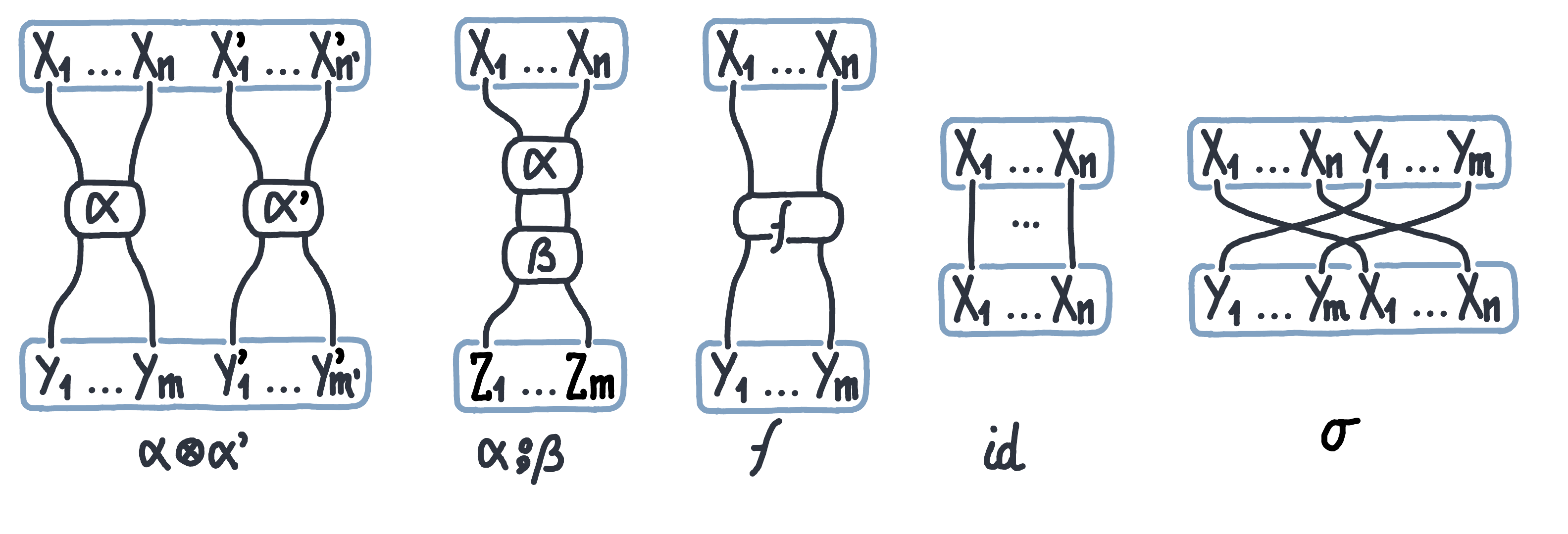} 
  \caption{Symmetric monoidal category of string diagrams.}
  \label{fig:symmetric-monoidal-string-diagrams}
\end{figure}
\begin{proof}
  A summary of the construction is in \Cref{fig:symmetric-monoidal-string-diagrams}: wires are vertices and hyperedges are nodes.
  Let us describe the category. The objects are lists of objects of the \polygraph{}.
  Tensoring of string diagrams of type $[X₁,\mydots,Xₙ] → [Y₁,\mydots,Yₘ]$ and $[X₁',\mydots,Xₙ'] → [Y₁',\mydots,Yₘ']$ is defined by the disjoint union of the hypergraphs -- merging input and output edges -- into a string diagram $[X₁,\mydots,Xₙ,X₁',\mydots,Xₙ'] → [Y₁,\mydots,Yₘ,Y₁',\mydots,Yₘ']$.
  Composition of string diagrams of type $[X₁,\mydots,Xₙ] → [Y₁,\mydots,Yₘ]$ and $[Y₁,\mydots,Yₘ] → [Z₁,\mydots,Zₚ]$ is constructed by glueing the vertices along the output edge of the first and the input edge of the second, which disappear producing a string diagram $[X₁,…,Xₙ] → [Z₁,…,Zₚ]$. Generators are included as single hyperedges labelled by them. The identity, $[X₁,…,Xₙ] → [X₁,…,Xₙ]$, consists only of the input and output edges, connected by a list of vertices.  Symmetries are defined by twisting the connections of the input and output edges. Finally, we can check that string diagrams satisfy the axioms of \symmetricMonoidalCategories{}.
\end{proof}

\begin{definition}
  \defining{linkSymMonCatStr}{}
  A \emph{strict symmetric monoidal functor} between two strict symmetric monoidal categories, $F ፡ ℂ → 𝔻$, is a strict monoidal functor that moreover preserves symmetries: $F(σ) ፡ F(X) ⊗ F(Y) → F(Y) ⊗ F(X)$ is the symmetry on $F(X)$ and $F(Y)$. Strict \symmetricMonoidalCategories{} and strict symmetric monoidal functors form a category, $\mathbf{SymMonCat}_\mathsf{Str}$.
\end{definition}

\begin{proposition}
  The construction of symmetric string diagrams extends to a functor from \polygraphs{} to \symmetricMonoidalCategories{}, 
  $$\StringSigma ፡ \PolyGraph{} → \SymMonCatStr.$$
\end{proposition}

\begin{proposition}
  There exists a forgetful functor from strict \symmetricMonoidalCategories{} to \polygraphs{} that takes the objects as the vertices of the \polygraph{} and the morphisms as the edges, 
  $$\mathsf{Forget} ፡ \SymMonCatStr → \PolyGraph.$$
\end{proposition}

\begin{theorem}
  String diagrams for \symmetricMonoidalCategories{} form the free strict \symmetricMonoidalCategory{} over a \polygraph{}:
  there exists an adjunction $\StringSigma{} \dashv \mathsf{Forget}$.
\end{theorem}
\begin{proof}[Proof sketch]
  We have already shown that string diagrams form a symmetric monoidal category. It only remains to show that there exists a unique strict symmetric monoidal functor to any strict symmetric monoidal category. The assignment is determined by the fact that each string diagram is constructed from the generators and the operations of a symmetric monoidal category, the difficulty is in showing that this assignment is well-defined. We refer the reader to the work of Bonchi and others \cite{bonchi16:rewriting}.
\end{proof}

The syntax of symmetric string diagrams as \hypergraphs{} is more efficient: to check equality, only the connectivity of the wires matters, and we no longer need to track the specific blocks forming the diagram.

\subsection{Do-Notation}

There is a second practical syntax for \symmetricMonoidalCategories{} that links \stringDiagrams{} to programming: Hughes' \emph{arrow do-notation} \cite{hughes00,paterson01:arrows}. It comes from \emph{functional programming}, but it is precisely a representation of \emph{imperative programming}. The main idea is that, in a string diagram, we can label the wires by variable names, and simply declare which nodes take which inputs and outputs to reconstruct the \stringDiagram{}. In a certain sense, this is the graph encoding of a \stringDiagram{}, but it closely resembles an imperative program. 

\begin{example}[Crema di Mascarpone in Do-notation]
  Consider the same process for ``crema di mascarpone'' that we detailed in \Cref{sec:cremadimascarpone}.
  This time, we can directly assume that we are in a \symmetricMonoidalCategory{}. The translation of the string diagram of \Cref{fig:mascarpone} is the following code in \Cref{fig:code-mascarpone}.
  \begin{figure}[ht] 
    \centering
    \includegraphics[scale=0.15]{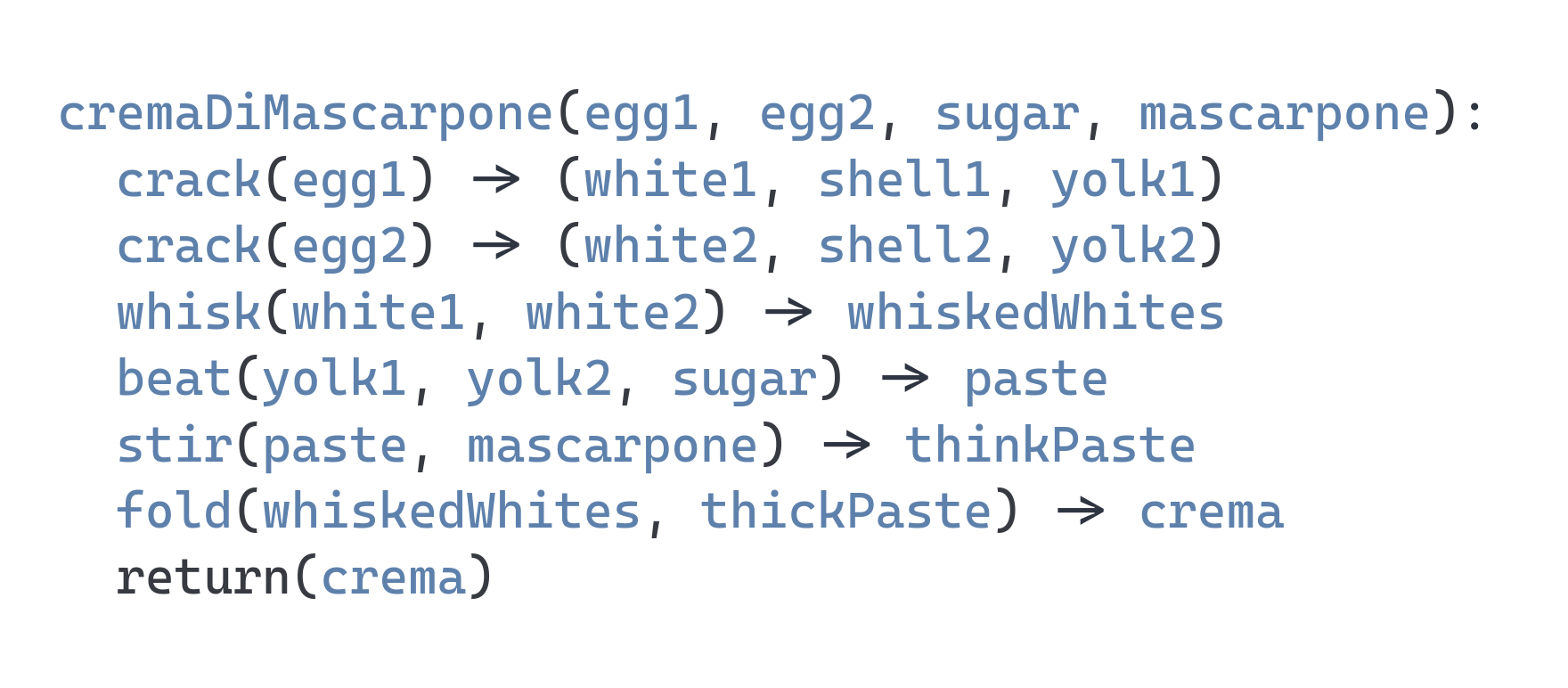} 
    \caption{Do-notation recipe for Crema di Mascarpone.}
    \label{fig:code-mascarpone}
  \end{figure}
\end{example}

Let us formalize do-notation in the style of a type-theory: we will work with variables, we consider them to be unique and we work implicitly up to $α$-equivalence (or renaming of variables). Our notion of signature is again that of a \polygraph{}: our basic types will be the objects of the polygraph, and we will have a rule for each one of the generators in the \polygraph{}.

\begin{definition}
  A \emph{derivation} in the proof theory of do-notation over a \polygraph{} is defined inductively to be either
  \begin{enumerate}
    \item a single $\mathsf{return}$ statement, given by a permutation; or
    \item the application of a generator $f ፡ A₀,…,Aₙ → B₀,…,Bₘ$, given by a choice of generator and an insertion of variables, followed by a derivation.
  \end{enumerate}

\begin{figure}[ht!]
  \begin{mathpar}
    \inferrule* [Right=(Return)]
    { }
    {\var{a_0} {:} A_0, … , \var{a_n} {:} A_n ≫_{τ} () ⊢ \var{\return(a_0,…,a_n)} : A_0 ⊗ … ⊗ A_n}
  \end{mathpar}
  \begin{mathpar}
    \inferrule*[Right=($f$)]
    {\var{b₀} {:} B₀, \dots , \var{bₘ} {:} Bₘ, Γ ⊢ \var{t} : Δ\hspace{11em}}
    {\var{a₀} {:} A₀, \dots , \var{aₙ} {:} Aₙ ≫_{τ} Γ ⊢ \var{f(a₀,…,aₙ) → b₀,…,bₘ ⨾ t} : Δ}
  \end{mathpar}
  \caption{Do-notation for symmetric monoidal categories.}
  \end{figure}
\end{definition}

Before continuing, then, it is important to understand what an \emph{insertion} is. An insertion captures how many ways we have of inserting some $n$ new terms into a list of $m$ terms. The new $n$ terms can be freely permuted, but the list of $m$ terms must preserve their relative order. This is the combinatorial structure that will track how variables are used in a derivation.

\begin{definition}
  We define the family of insertions of $n$ terms into $m$ terms, $\Ins(n,m)$, inductively. There exists a single way of inserting zero terms into any list of terms, $\mathsf{Ins}(0,m) = 1$; inserting $n+1$ terms into a list of $m$ terms amounts to choosing the position of the first among $m+1$ possible choices, and inserting the rest of the terms,
  $$\Ins(n+1,m) = (m+1) × \Ins(n,m+1).$$
  As a consequence, the number of possible insertions is $\Ins(n,m) = (m+n)!/m!$.

  We write $a₁,…,aₙ ≫_{τ} Γ$ to refer to the list resulting from the insertion of the variables $a₁,…,aₙ$ into the list $Γ$, of length $m$, according to the insertion $τ ∈ \Ins(n,m)$.
\end{definition}

\begin{remark}
  Accordingly, the only information that a return statement may carry is that of an insertion $\Ins(n,0)$, which is equivalently a permutation of the $n$ elements that are being returned. The information carried by a generator statement is not only a generator $n$-to-$m$ but also an insertion $\Ins(n,\#Γ)$ of $n$ variables on the context of the derivation.
\end{remark}

\begin{example}
  Let us provide an example of the correspondence between the different notations: a Rosetta's stone translating between string diagrams, terms of do-notation, and their corresponding derivations (\Cref{fig:rosettaexample}). 
  \begin{figure}[ht]
    \begin{minipage}{0.3\textwidth}
      \includegraphics[scale=0.35]{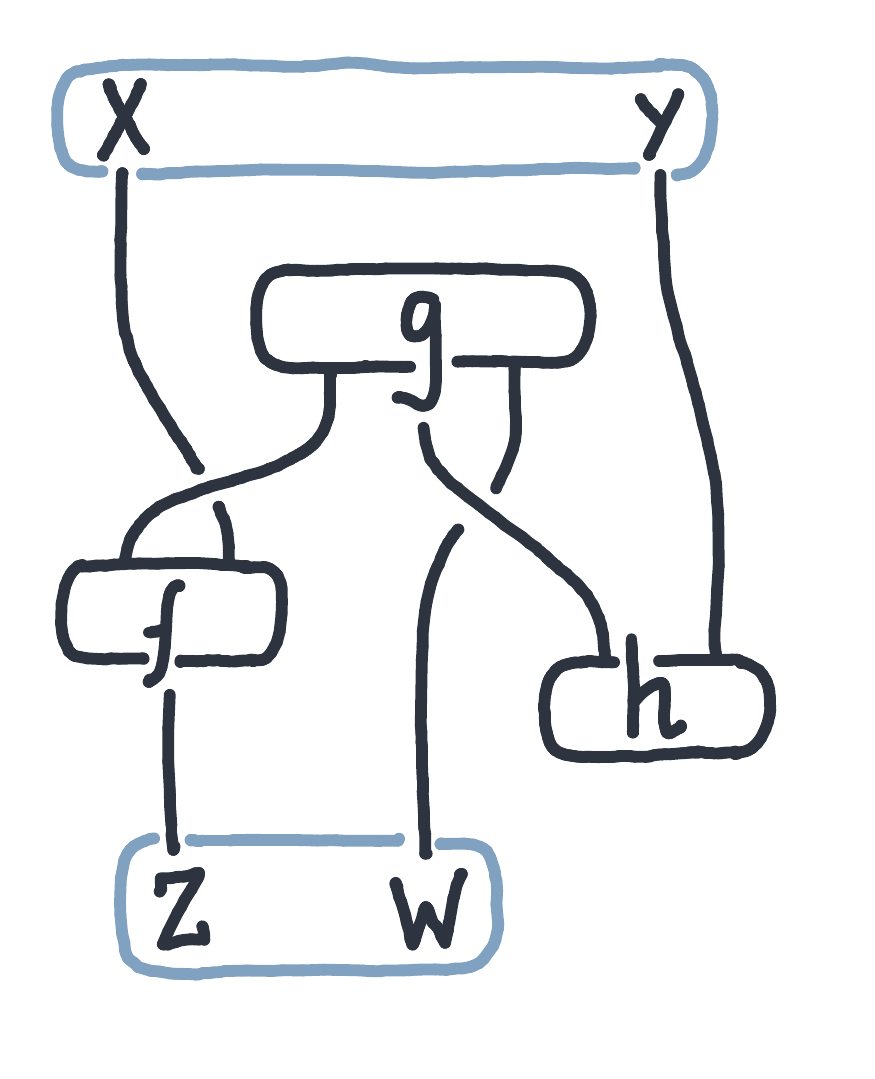}
    \end{minipage}
    \begin{minipage}{0.65\textwidth}
      \begin{mathpar}
        \inferrule*{ }{
          \inferrule*
          { \var{z}{:}Z, \var{w}{:} W ⊢ \var{\return(z,w)} : Z ⊗ W}
          {\inferrule*{ 
            { \begin{aligned} \var{z} {:} Z,\var{b} {:} B,\var{w} {:} W, \var{y} {:} Y  ⊢\ & \var{h(b,y) → () ⨾} 
              \hspace{5em} \\[-5pt] 
              & \var{\return(z,w)} : Z ⊗ W \end{aligned}} }
            { \inferrule* { { \begin{aligned} \var{a} {:} A, \var{b} {:} B, \var{w} {:} W, \var{x} {:} X, \var{y} {:} Y  ⊢\ 
              & \var{f(a,x) → z\, ⨾} \hspace{5em}\\[-5pt]
              & \var{h(b,y) → ()\, ⨾} \\[-5pt] 
              & \var{\return(z,w)} : Z ⊗ W \end{aligned}} }  
            { { \begin{aligned} \var{x} {:} X, \var{y} {:} Y  ⊢\ 
                & \var{g() → (a,b,w)\, ⨾} \\[-5pt]
                & \var{f(a,x) → z\, ⨾} \\[-5pt]
                & \var{h(b,y) → ()\, ⨾} \\[-5pt] 
                & \return(z,w) : Z ⊗ W \end{aligned}} } }
          }}
      \end{mathpar}
    \end{minipage}
    \caption{String diagram and derivation of a term.}
    \label{fig:rosettaexample}
  \end{figure}
\end{example}

\subsection{Symmetry in Do-notation}

At this point, using insertions may seem complicated: why not simply assume an exchange rule that allows us to permute variables freely? The problem we would encounter is that exchanges introduce redundancy: there would be multiple ways of writing the same term, depending on where we place the symmetries (Shulman describes the same problem for a different notation \cite{shulman:catlog}). This is not a catastrophic problem -- we could still quotient them out appropriately -- but it would make the construction much more complicated than simply dealing with the combinatorial structure of insertions upfront.

The better solution is to have exchange appear as a derived, admissible rule, rather than a primitive.

\begin{proposition}
  Exchange is admissible in Do-notation for symmetric monoidal categories. Any derivation $Γ,\var{x},\var{y},Δ ⊢ \var{t} : X$ admits a derivation $Γ,\var{y},\var{x},Δ ⊢ \var{t} : X$.
\end{proposition}
\begin{proof}
  We proceed by structural induction. The base case is a single $\return$ statement, written as $\var{a₀,…,x,y,…,aₙ} ≫_σ () ⊢ \var{\return(a₀,…,x,…,y,…aₙ)} : X$. Permuting $\var{x,y}$ in the insertion $τ$ gives us a new insertion $τ_{\var{xy}}$ deriving the same statement under a different context
  $$\var{a₀,…,y,x,…,aₙ} ≫_{τ_{\var{xy}}} () ⊢ \var{\return(a₀,…,x,…,y,…aₙ)} : X.$$

  Consider now an application of a generator, $$\var{a₀},…,\var{aₙ} ≫_{τ} Γ ⊢ \var{f(a₀,…,aₙ) →  b₀,…,bₘ ⨾ term} : X.$$ There are two possible cases here: (1) if any of $\var{x,y}$ is among the inserted variables, $\var{a₀,…,aₙ}$, then we may simply exchange them by changing the order in which they are inserted; (2) if not, then $Γ = Γ',\var{x},\var{y},Δ'$, and we apply the induction hypothesis over the derivation $Γ',\var{x},\var{y},Δ' ⊢ \var{term} : X$.
\end{proof}

Indeed, this is enough to ensure that terms do correspond to derivations: we may simply write a term and there is a unique way it could have been extracted from the context.
\begin{proposition}
  Do-notation terms in a given context have a unique derivation.
\end{proposition}
\begin{proof}
  By structural induction, a term is either a single $\return$ or an application of some generator $(\var{f})$.
  Any single term $\var{x₀,…,x_n} ⊢ \var{\return(a₀,…,aₙ)}$ has a unique derivation: namely, the one that inserts the $\var{a}$'s into the $\var{x}$'s permuting them in the only possible order. The insertion $τ$ must be the only one making $τ(\var{xᵢ}) = \var{aᵢ}$.

  Consider now an application of a generator, $Δ ⊢ \var{f(a₀,…,aₙ) → b₀,…,bₘ ⨾ t} : X$. We know that it must come from $Δ = (\var{a₀,…,aₙ} ≫_{τ} Γ)$, and this forces $Γ = (Δ - \{\var{a₀,…,aₙ}\})$ and the value of the insertion $τ$: it is the only one that turns $Γ$ into $Δ$. Once $Γ$ has been determined, we apply the induction hypothesis to get the derivation of $\var{b₀,…,bₘ},Γ ⊢ \var{t}:X$.
\end{proof}

\subsection{Quotienting Do-notation}

Our logic is freely constructed, but it is not yet a logic of \monoidalCategories{}: technically, it misses interchange. As it stands, it is actually a logic for symmetric \premonoidalCategories{}, which we will study later. We shall only add the following rule -- the interchange rule -- in order to convert it into a calculus of \symmetricMonoidalCategories{}.

\emph{Interchange rule.} Consider a derivation of the following term, where $\var{b₀,…,bₘ}$ and $\var{c₀,…cₚ}$ are two lists of \emph{distinct} variables, meaning that no $\var{bᵢ}$ appears in $\var{cⱼ}$ -- and conversely, no $\var{dᵢ}$ appears in $\var{aⱼ}$. Then, we can interchange the two first statements of the term.
\begin{figure}[ht]
  \begin{minipage}{0.3\textwidth}
    \begin{align*}
    Γ ⊢\ 
    & \var{f(a₀,…,aₙ) → b₀,…,bₘ}\, ⨾ \\
    & \var{g(c₀,…,cₚ) → d₀,…,dₚ}\, ⨾ \\
    & \var{term} : Δ
  \end{align*}
  \end{minipage}
  \quad $\equiv$ \quad
  \begin{minipage}{0.3\textwidth}
    \begin{align*}
      Γ ⊢\ 
      & \var{g(c₀,…,cₚ) → d₀,…,dₚ\, ⨾} \\
      & \var{f(a₀,…,aₙ) → b₀,…,bₘ\, ⨾} \\
      & \var{term} : Δ
    \end{align*}    
  \end{minipage}
  \caption{Interchange rule.}
\end{figure}

\begin{proposition}
  We show this is well defined.
  Whenever the left hand side of the interchange rule is a valid term, and variables are distinct, the right hand side is also a valid term.
\end{proposition}
\begin{proof}
  First, we reason that the term on the left-hand side must have a derivation tree of the following form.
  \begin{mathpar}
    \inferrule*{\var{d₀,…,d_q,b₀,…,bₘ},Δ ⊢ \var{term} : C \hspace{6em}}
  {\inferrule*{ \hspace{1em} { \begin{aligned} 
    \var{b₀,…,bₘ},(\var{c₀,…,cₚ} ≫_{ρ} Δ) ⊢\ & \var{g(c₀,…,cₚ) → (d₀,…,d_q) \, ⨾}  \\ & \var{term} : C 
  \end{aligned} } \hspace{0em} }
      { { \begin{aligned} \var{a₀,…,aₙ} ≫_{σ} (\var{c₀,…,cₚ} ≫_{ρ} Δ) ⊢\ & 
        \var{f(a₀,…,aₙ) → (b₀,…,bₘ)} \, ⨾\\ & 
        \var{g(c₀,…,cₚ) → (d₀,…,d_q)} \, ⨾ \\ & 
        \var{term} : C \end{aligned} } }
    }
  \end{mathpar}
  We now argue for this. In the second line of the derivation we use that, because of variables being distinct, the only possible context $\var{(c₀,…,cₚ)} ≫_{τ} Ψ$ we can use must factor as $\var{b₀,…,bₘ,(c₀,…,cₚ)} ≫_{ρ} Δ$ for some $Δ$ and $ρ$. We can then check by the form of the rules that the final context $Γ$ must be of the form $\var{a₀,…,aₙ} ≫_σ (\var{c₀,…,cₚ} ≫_ρ Δ)$ for some insertion $σ$. 
  All this means that the following derivation is also valid.
  \begin{mathpar}
    \inferrule*{\var{b₀,…,bₘ, d₀,…,d_q},  Δ ⊢ \var{term} : X \hspace{7em}}
  {\inferrule*{ { \begin{aligned} 
    \var{d₀,…,d_q} ,(\var{c₀,…,cₚ} ≫ Δ) ⊢\ & 
    \var{f(a₀,…,aₙ) → b₀,…,bₘ} \, ⨾  \\ &
    t : X \end{aligned} } }
      { { \begin{aligned} 
        \var{a₀,…,aₙ} ≫ (\var{c₀,…,cₚ} ≫ Δ) ⊢\ 
        & \var{g(c₀,…,cₚ) → (d₀,…,d_q)}\, ⨾ \\ 
        & \var{f(a₀,…,aₙ) → (b₀,…,bₘ)}\, ⨾ \\
        & \var{term} : X \end{aligned} } }
    }
  \end{mathpar}
  This derivation proves that the right side term is also valid.
\end{proof}

Finally, we can start proving that do-notation terms are a syntax for \symmetricMonoidalCategories{}: they form the free strict \symmetricMonoidalCategory{} over a signature.

\begin{lemma}
  Do-notation terms over a \polygraph{} $𝓖$, quotiented by the interchange rule, define a \monoidalCategory{}, $\mathsf{Do}(𝓖)$. This construction induces a functor $\mathsf{Do} ፡ \PolyGraph → \SymMonCatStr$.
\end{lemma}
\begin{proof}[Proof sketch]
  We will describe the operations of this strict \symmetricMonoidalCategory{}. Let us start by \textbf{composition}: consider derivations $Γ ⊢ \var{t} ፡ B₁ ⊗ … ⊗ Bₘ$ and $\var{b₁}{:}B₁,…,\var{bₘ}{:}Bₘ ⊢ \var{s} ፡ Ψ$. We will construct, by induction over $\var{t}$, a derivation $Γ ⊢ \mathsf{comp}(t,s) ፡ Ψ$, that will define their \emph{composition}. If $\var{t}$ is a $\return$ statement, then $Γ$ is a permutation of $\var{b₁}{:}B₁,…,\var{bₘ}{:}Bₘ$ and, by exchange, we obtain $Γ ⊢ s ፡ Ψ$ and we define it as the composition.
  Whenever $\var{t}$ consists of a generator followed by a derivation, $Γ ⊢ t ፡ Ψ$ must be of the form
  $$(\var{x₁},…,\var{xₙ}) ≫ Γ' ⊢ \var{f(x₁,…,xₙ) → y₁,…,yₘ ⨾ t'} : Ψ,$$
  and we define $\mathsf{comp}(t,s)$ to mean $\var{f(x₁,…,xₙ) → y₁,…,yₘ ⨾ \mathsf{comp}(t',s)}$, using the $y₁,…,yₘ,Γ' ⊢ t' ፡ B₁ ⊗ … ⊗ Bₘ$ we just obtained. Put simply, we have removed the last $\return$ statement from one of the derivations, taking care of permutations, and then concatenated both (see \Cref{fig:do-notation-category}).

  Let us now define \textbf{tensoring}: consider derivations $Γ ⊢ t ፡ Δ$ and $Γ' ⊢ t' ፡ Δ'$. We start by noticing that, if we have a derivation $Γ ⊢ \var{t} ፡ Δ$, we can construct, by induction, a derivation $Γ,\var{z} {:}Z ⊢ \var{t}_z ፡ Δ ⊗ Z$ for any $\var{z}{:}Z$: the return case consists of adding an extra variable to the permutation, the generator case is not affected by the presence of an extra variable. Repeating this reasoning and using exchange, we can obtain terms $Γ,Γ' ⊢ \var{t}_{Γ'} ፡ Δ ⊗ Γ'$ and $Δ,Γ' ⊢ \var{t'}_{Δ} ፡ Δ ⊗ Δ'$ that we can compose into $Γ,Γ' ⊢ \mathsf{comp}(\var{t}_{Γ'},\var{t'}_{Δ'}) ፡ Δ ⊗ Δ'$. The order of composition does not matter because of the \emph{interchange law}. Put simply, we write the two terms one after the other, taking care not to mix the variables (see \Cref{fig:do-notation-category}, for a graphical intuition).

  \begin{figure}[!ht]
    \centering
    \includegraphics[scale=0.35]{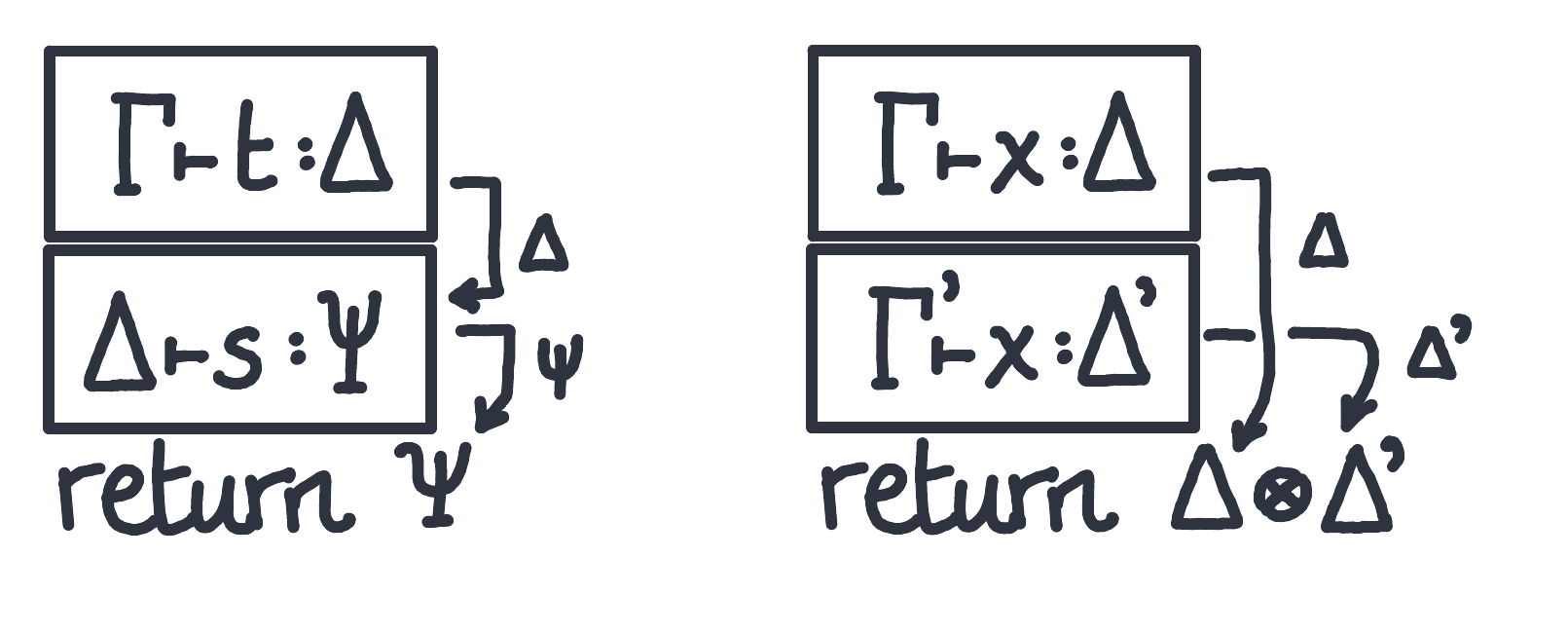}
    \caption{Composition and tensoring in do-notation.}
    \label{fig:do-notation-category}
  \end{figure}

  Finally, we must define \textbf{identities}, \textbf{generators} and \textbf{symmetries}. These are the following three terms.
  \begin{align*}
     \var{a₁,…,aₙ} &⊢ \var{\return(a₁,…,aₙ)} \\
     \var{a₁,…,aₙ} &⊢ \var{f(a₁,…,aₙ) → (b₁,…,bₘ) ⨾ \return(b₁,…,bₘ)} \\
     \var{a₁,…,aₙ,b₁,…bₘ} &⊢ \var{\return(b₁,…,bₘ,a₁,…,aₙ)}
  \end{align*}
  We can check that these operations define a category and a monoidal category. The interchange law of monoidal categories is extracted from the interchange law that we imposed in do-notation.
\end{proof}

\begin{theorem}
  \label{th:donotation}
  There is an adjunction between \polygraphs{} and the category of strict \symmetricMonoidalCategories{} given by do-notation terms, $\mathsf{Do} ⊣ \mathsf{Forget}$. Moreover, do-notation terms and string diagrams are naturally isomorphic.
\end{theorem}
\begin{proof}[Proof sketch]
  We have already proven that do-notation constructs a \symmetricMonoidalCategory{} over a \polygraph{}. We need to show that there is a unique map out of the category constructed by do-notation that commutes with any assignment of the \polygraph{} to a \monoidalCategory{}.

  The first part would be to prove the initiality of the syntax. We can do so by structural induction: any term is either a return statement or a composition of a generator with a symmetry and a term. In the first case, the symmetry determined by the return statement must be mapped to the corresponding symmetry in any \symmetricMonoidalCategory{}; in the second case, the image of the generator is determined, and the rest of the term is determined by structural induction. In conclusion, the image of any do-notation term is determined in any \symmetricMonoidalCategory{} over which we map the \polygraph{} of generators. 
  
  The core of the proof is in showing that this unique possible assignment is well-defined: the only equality imposed on do-notation is the interchange law, but this law corresponds, under the assignment, to the interchange law of \symmetricMonoidalCategories{}.

  Finally, both functors, 
  $$\mathsf{Do}, \StringSigma ፡ \PolyGraph → \SymMonCatStr$$ 
  have been found to be left adjoints to the same forgetful functor. This implies they are isomorphic functors.
\end{proof}

\subsection{Example: the XOR Variable Swap}
  Let us provide an example of the formal usage of both do-notation and \stringDiagrams{}, by reasoning about a simple process.
  In imperative programming languages, swapping the value of two variables usually requires a third temporary variable.
  However, the XOR variable swap algorithm uses the properties of the exclusive-or operation (XOR, $\oplus$) to exchange variables without needing a temporary variable. Let $\var{xor(x,y) = (x ⊕ y, y)}$. The code is in \Cref{fig:code-xor-variable-exchange}.
  \begin{figure}[!ht]
    \centering
    \includegraphics[scale=0.15]{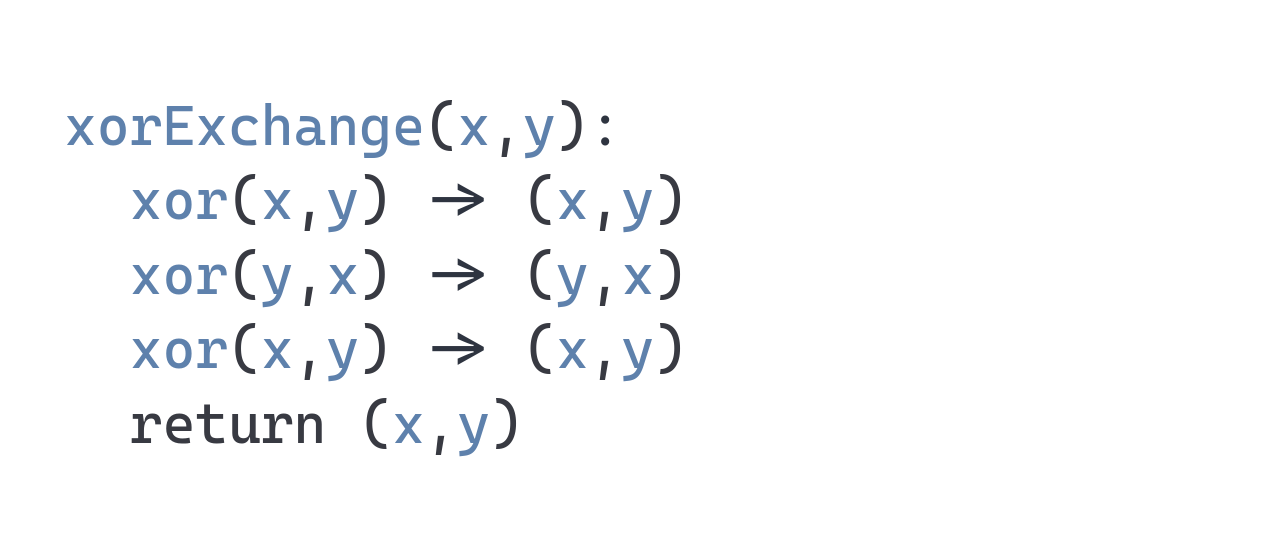}
    \caption{XOR variable exchange.}
    \label{fig:code-xor-variable-exchange}
  \end{figure}

  The property that makes this algorithm possible is the \emph{nilpotency} of the XOR operation: $x \oplus x = 0$ for any n-bit word $x ∈ 2^n$. This means that we can prove the correctness of the XOR variable exchange in the abstract setting of nilpotent bialgebras. In fact, consider a polygraph $𝓧$ with a single object and the generators depicted in \Cref{fig:signature-bialgebra}.
  \begin{figure}[ht]
    \centering
    \includegraphics[scale=0.35]{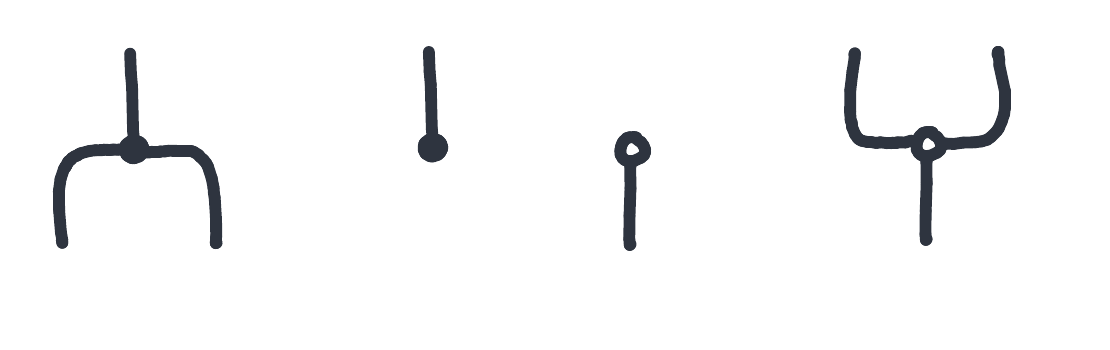}
    \caption{Signature for a bialgebra.}
    \label{fig:signature-bialgebra}
  \end{figure}

  We now want to impose a set of equations, $E \subseteq \mathsf{String}(𝓧) × \mathsf{String}(𝓧)$, on top of this signature. This can also be done via the adjunction: the equations give two maps $E → \mathsf{Forget}(\mathsf{String}(𝓧))$, or equivalently, $\mathsf{String}(E) → \mathsf{String}(𝓧)$. The coequalizer of the latter two describes the universal \monoidalCategory{} with some generators and satisfying some equations. Back to our example, the theory of a nilpotent bialgebra satisfies the following equations in \Cref{fig:nilpotent-bialgebra}.
  \begin{figure}[ht]
    \centering
    \includegraphics[scale=0.35]{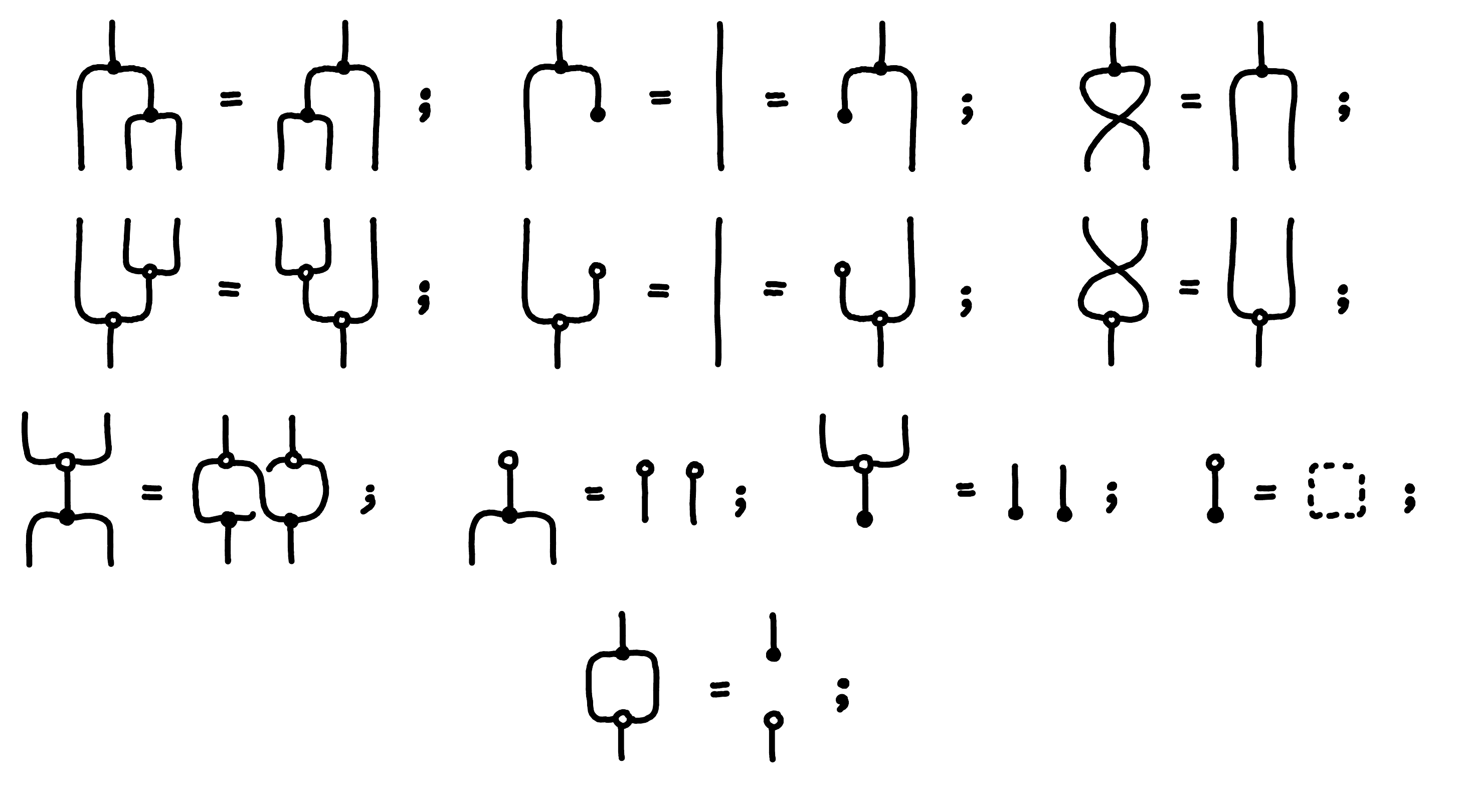}
    \caption{Theory of a nilpotent bialgebra.}
    \label{fig:nilpotent-bialgebra}
  \end{figure}

  Given any nilpotent bialgebra in any strict \symmetricMonoidalCategory{}, there exists a unique monoidal functor from the \stringDiagrams{} quotiented by these equations to that signature.

  \begin{proposition}
    Let a nilpotent bialgebra in a \symmetricMonoidalCategory{}. The XOR variable exchange algorithm is equal to the swap morphism.
  \end{proposition}
  \begin{proof}
  In the theory of nilpotent bialgebras over a \symmetricMonoidalCategory{}, the following equation in \Cref{fig:xor-variable-exchange} holds. 
    \begin{figure}[ht]
      \centering
      \includegraphics[scale=0.3]{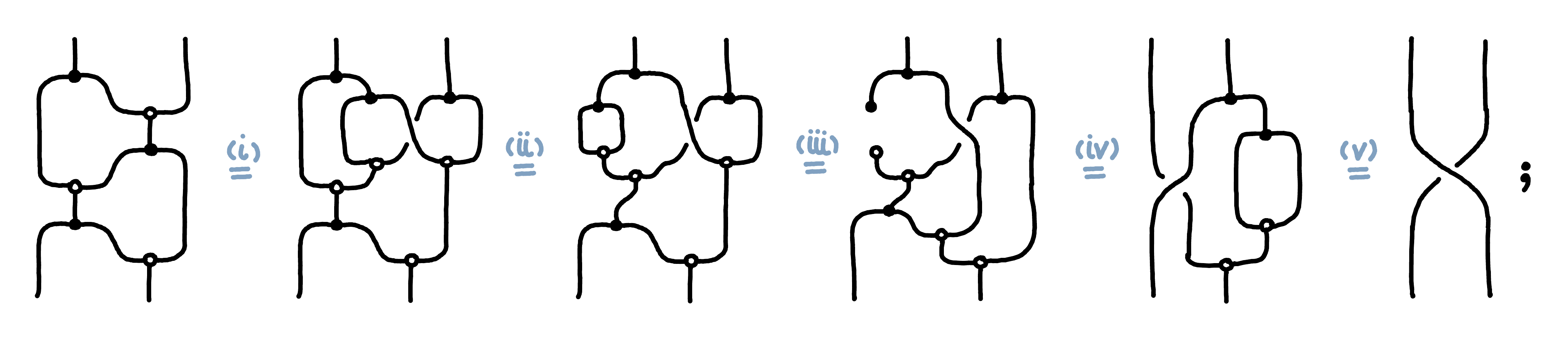}
      \caption{XOR variable exchange.}
      \label{fig:xor-variable-exchange}
    \end{figure}
    
  The left hand side represents the XOR variable exchange while the right hand side represents swapping the contents of two variables. We have shown both are equal.
  \end{proof}
  
\subsection{Bibliography}
\SymmetricMonoidalCategories{} and their hexagon coherence equations were already stated by MacLane \cite{maclane63:natural,macLane71:workingMathematician}, Bénabou defined \emph{commutations} on a monoidal category and an abstract notion of commutative monoid \cite{benabou68:structures}. String diagrams for \symmetricMonoidalCategories{} are already described by Joyal and Street \cite{joyal91:geometryOfTensorCalculus,selinger2010survey}; we follow the representation in terms of hypergraphs by Bonchi, Gadducci, Kissinger, Sobocinski and Zanasi \cite{bonchi16:rewriting}. The XOR example was known to Erbele \cite{erbele14:control} and Sobocinski. The idea of using a Rosetta stone to translate between categories and logics comes from Baez and Stay \cite{baezstay10:rosetta}.

Do-notation comes from the Haskell programming language \cite{haskellreport}, where it takes semantics in a strong promonad \cite{hughes00,heunen06:arrows} (also known as an \emph{arrow} \cite{paterson03:arrowscomputation}). We have studied here an adaptation to the monoidal setting. Our presentation of do-notation follows the style of Shulman's categorical logic \cite{shulman:catlog}.
\clearpage{}%
\clearpage{}%
\section{Cartesianity: Determinism and Totality}
\label{sec:linearity}

Our \processTheories{} are, by default, \emph{linear on resources}: every variable must be used exactly once.
This may seem like a limitation, but it is a more general case that can be particularized into the classical case when necessary:
we say that a \processTheory{} -- a \monoidalCategory{} -- is \emph{classical} or \emph{cartesian} whenever it has processes  representing copying and discarding and satisfying suitable axioms.

Non-classical theories can become so in two ways: either because they do not allow copying, or because they do not allow discarding. Theories without copying model stochasticity and non-determinism: running a computation twice is different from just running it once and assuming its result will be the same next time. Theories without discarding model partiality: even if we do not care about the result, we cannot assume anymore that any computation will succed. 

\subsection{Cartesian Monoidal Categories}
\CartesianMonoidalCategories{} give a universal property to their tensor: the tensor of two objects, $A × B$, is such that  pair of maps to $A$ and $B$ are in precise correspondence to single map to $A × B$. This universal property, in some sense, ensures that the tensor contains nothing more and nothing less than its two constituent parts; this is what characterizes classical theories.

\begin{definition}
  \emph{Cartesian monoidal categories} are \monoidalCategories{}, $(ℂ,×,1)$, such that 
  \begin{enumerate}
    \item each tensor, $A × B$, is endowed with projections, $π₁ ፡ A × B → A$ and $π₂ ፡ A × B → B$, that make it a product: for each object $X ∈ ℂ_{obj}$ and any pair of morphisms, $f ፡ X → A$ and $g ፡ X → B$, there exists a unique $\langle f,g \rangle ፡ X → A × B$ such that
    $$\langle f,g \rangle ⨾ π₁ = f \quad\mbox{ and }\quad \langle f,g \rangle ⨾ π₂ = g;$$
    \item the unit, $1$, is terminal: for each object $X ∈ ℂ_{obj}$ there exists a unique morphism $π ፡ X → 1$.
  \end{enumerate}
\end{definition}

Fox's theorem is a characterisation of classical theories, \cartesianMonoidalCategories{}, in terms of the existence of a uniform cocommutative comonoid structure (copy and delete) on all objects of a \monoidalCategory{}.

\begin{theorem}[Fox, \cite{fox76}] 
  \label{th:fox}
  A \symmetricMonoidalCategory{} $(ℂ,\otimes,I)$ is cartesian monoidal if and only if every object $X \in ℂ$ has a commutative comonoid structure $(X,ε_{X},δ_{X})$, every morphism of the category $f \colon X \to Y$ is a comonoid homomorphism, and this structure is uniform across the \monoidalCategory{}, meaning that
  $ε_{X \otimes Y} = ε_{X} \otimes ε_{Y}$, that
  $ε_{I} = \id$, that $δ_{I} = \id$ and that
  $δ_{X \otimes Y} = (δ_{X} \otimes δ_{Y}) ⨾ (\id \otimes \sigma_{X,Y} \otimes \id)$.
\end{theorem}

\begin{figure}[h]
  \includegraphics[scale=0.4]{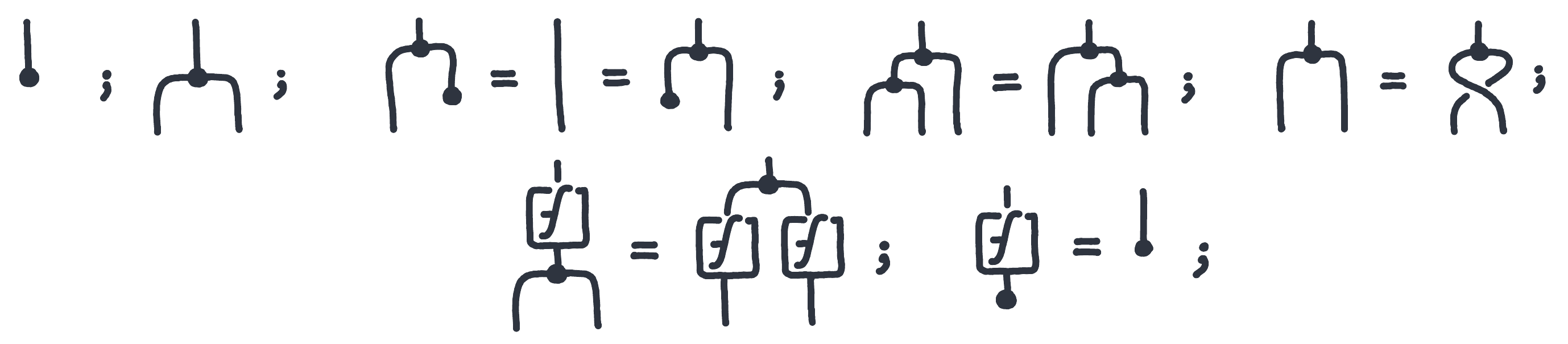}
  \caption{Theory of cartesian categories.}
  \label{fig:theorycartesian}
\end{figure}

Fox's characterization has a direct translation to string diagrams: the first conditions impose a natural commutative comonoid structure on each generator (\Cref{fig:theorycartesian}); the last conditions state that the structure on all the objects follows from that of the generators.

We can add a slight improvement. Most sources ask the comonoid structure in Fox's theorem (\Cref{th:fox}) to be cocommutative~\cite{fox76,fong2019supplying}.
However, cocommutativity and coassociativity of the comonoid structure are implied by the fact that the structure is uniform and natural. We present an original refined version of Fox's theorem.

\begin{theorem}[Refined Fox's theorem]
  \label{th:refinedfoxappendix}
  A \symmetricMonoidalCategory{}, ($ℂ$, $⊗$, $I$), is cartesian monoidal if and only if every object $X ∈ ℂ$ has a counital comagma structure $(X,ε_{X},δ_{X})$, or $(X,\iconbcu_{X},\iconbcm_{X})$, every morphism of the category $f ፡ X → Y$ is a comagma homomorphism, and this structure is uniform
  across the monoidal category: meaning that
  $ε_{X ⊗ Y} = ε_{X} ⊗ ε_{Y}$,
  $ε_{I} = \id$, $δ_{I} = \id$ and
  $δ_{X ⊗ Y} = (δ_{X} ⊗ δ_{Y}) ; (\id ⊗ \sigma_{X,Y} ⊗ \id)$.
\end{theorem}
\begin{proof}
  We prove that such a comagma structure is necessarily coassociative and cocommutative.
  Note that any comagma homomorphism $f ፡ A → B$ must satisfy $δ_{A} ⨾ (f ⊗ f) = f ⨾ δ_{B}$.
  In particular, $δ_{X} ፡ X → X ⊗ X$ must itself be a comagma homomorphism
  (see \Cref{figure:comultiplication}), meaning that
  \begin{equation}\label{eq:comultiplicationhomomorphism}
    δ_{X} ⨾ (δ_X ⊗ δ_{X}) =
    δ_{X} ⨾ δ_{X ⊗ X} =
    δ_{X} ⨾ (δ_X ⊗ δ_{X}) ⨾ 
    (\id ⊗ \sigma_{X,Y} ⊗ \id),
  \end{equation}
  where the second equality follows by uniformity.
  \begin{figure}[h]
    \includegraphics[scale=0.4]{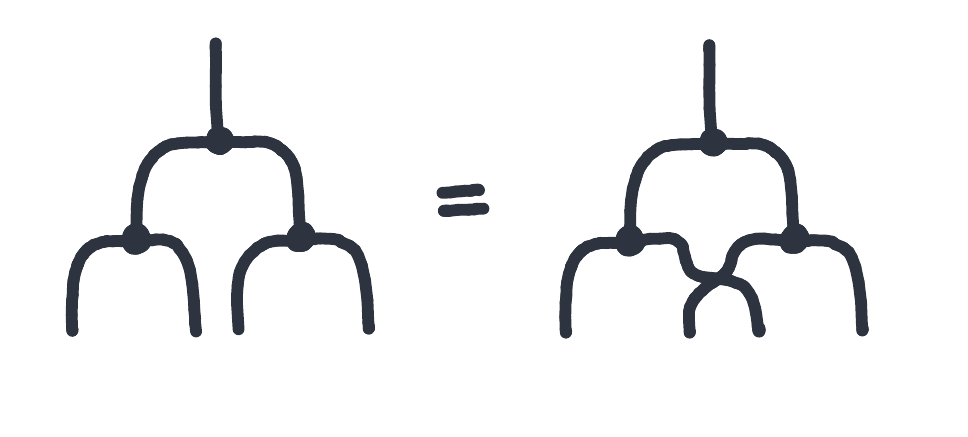}
    \caption{Comultiplication is a comagma homomorphism.}
    \label{figure:comultiplication}
  \end{figure}

  Now, we prove cocommutativity (\Cref{string:cocommutative}): composing both sides of \Cref{eq:comultiplicationhomomorphism} with $(ε_{X} ⊗ \id ⊗ \id ⊗ ε_{X})$ discards the two external outputs and gives $δ_{X} = δ_{X} ⨾ \sigma_{X}$.
  \begin{figure}[h]
    \includegraphics[scale=0.35]{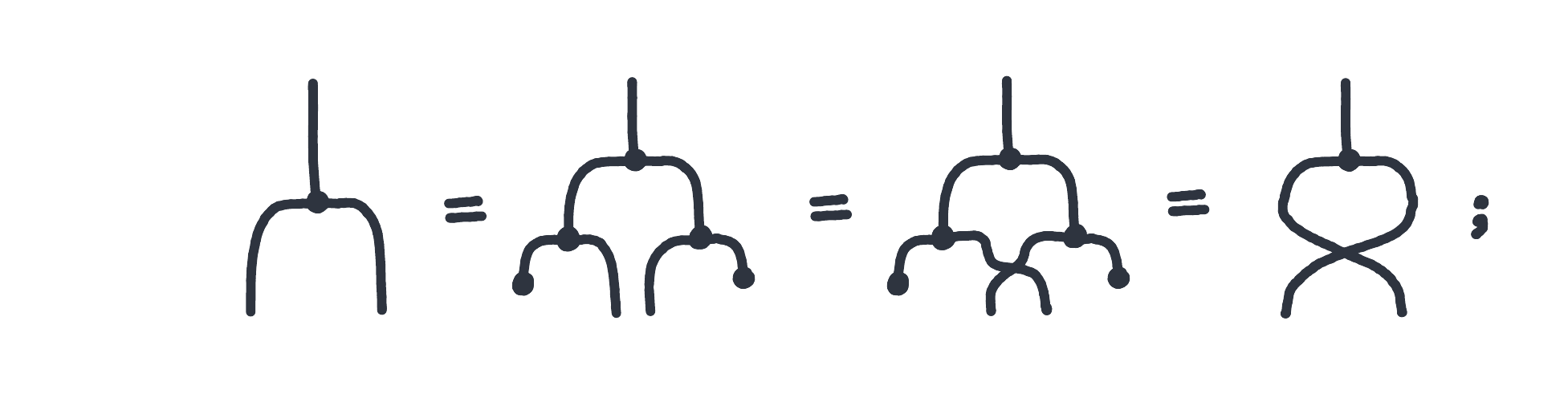}
    \caption{Cocommutativity}
    \label{string:cocommutative}
  \end{figure}
  Then, we prove coassociativity (\Cref{string:coassociativity}): composing both sides of \Cref{eq:comultiplicationhomomorphism} with $(\id ⊗ ε_{X} ⊗ \id ⊗ \id)$ discards one of the middle outputs and gives $δ_{X} ⨾ (\id ⊗ δ_{X}) = δ_{X} ⨾ (δ_{X} ⊗ \id)$.
\begin{figure}
  \includegraphics[scale=0.4]{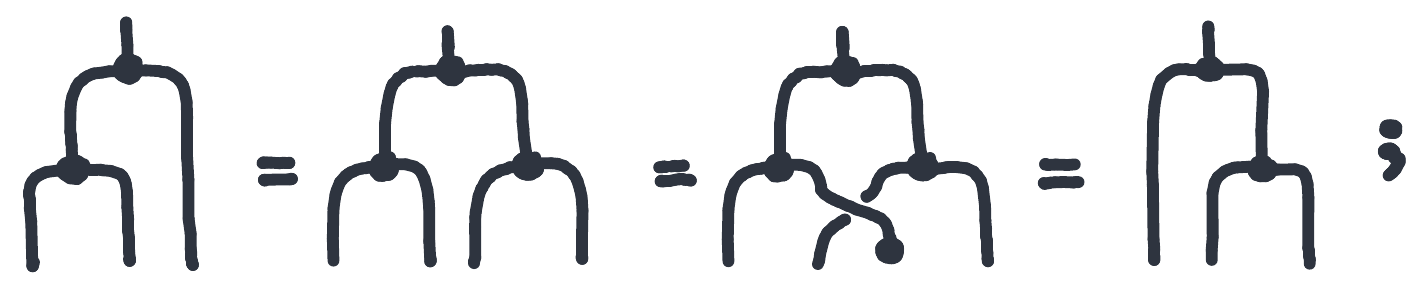}
  \caption{Coassociativity}
  \label{string:coassociativity}
\end{figure}

A coassociative and cocommutative comagma is a cocommutative comonoid. We can then apply the classical form of Fox's theorem (\Cref{th:fox}).
\end{proof} 

\subsection{Partial Markov Categories}
If process theories were all cartesian, we could use the commutative comonoid structure on every object to simplify calculations. However, most of the theories that pose a challenge to computer science (like stochastic processes, partial processes, or quantum maps) are not cartesian. The rest of this thesis will not assume cartesianity: let us give a good example and motivation for doing this.

Cartesianity in a category with copy and discard processes can be divided in two concepts: \emph{determinism} and \emph{totality}. All pure functions, for instance, are deterministic and total, but stochastic functions are not deterministic (tossing a coin twice is different from tossing it once and copying the result twice), and partially computable functions are not total (because even if we do not care about the output, they could diverge and make the whole process wait).

\begin{definition}
  In a \symmetricMonoidalCategory{} with uniform commutative comonoids, $(ℂ,⊗,I,\iconbcm,\iconbcu, σ)$, a morphism $f ፡ X → Y$ is \emph{deterministic} if it can be copied, $f ⨾ \iconbcm_{Y} = \iconbcm_{X} ⨾ (f ⊗ f)$, and it is \emph{total} (or \emph{causal}) if it can be discarded, $f ⨾ \iconbcu_Y = \iconbcu_X$. Moreover, we say it is \emph{quasitotal} (or \emph{quasicausal}), if it can be copied on the side and discarded, $f = \iconbcm_{X} ⨾ ((f ⨾ \iconbcu_Y) ⊗ f)$. See \Cref{figure:deterministictotal}.
\end{definition}
\begin{figure}[h]
  \includegraphics[scale=0.35]{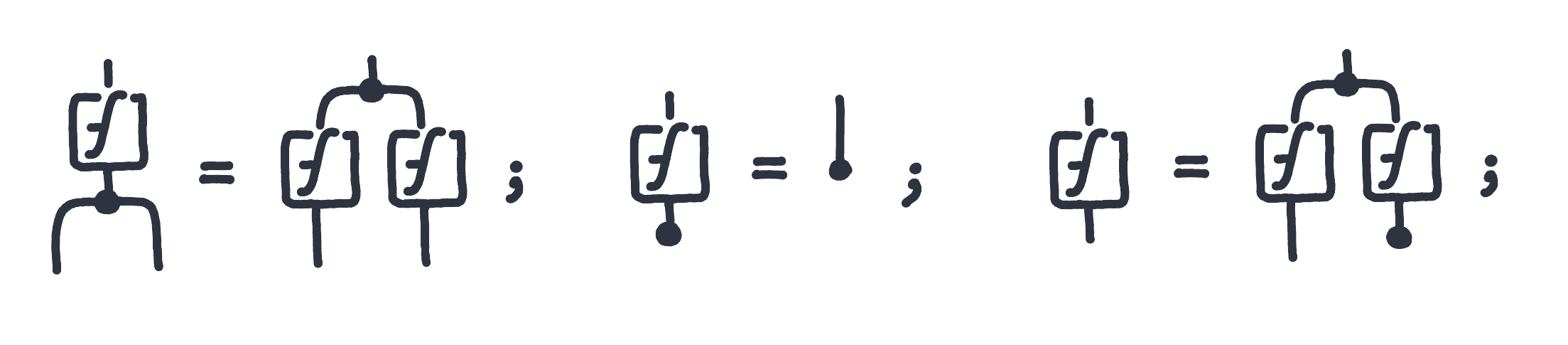}
  \caption{Deterministic, total, and quasitotal morphisms.}
  \label{figure:deterministictotal} 
\end{figure}

Let us provide a single theory where these three assumptions fail: the theory of \discretePartialMarkovCategories{} \cite{dilavore:evidentialdecision}, which we will use to study partial stochastic functions. Apart from copy $(\iconbcm) ፡ X → X ⊗ X$ and discard $(\iconbcu) ፡ X → I$, we also consider a comparator morphism $(\iconwm) ፡ X ⊗ X → X$. Copying, discarding and comparing interact as a \emph{partial Frobenius algebra} \cite{di2021functorial}.

\begin{definition}
  \defining{linkDiscretePartialMarkovCategory}{}
  A \emph{discrete partial Markov category} is is a \symmetricMonoidalCategory{} $(ℂ,⊗,I)$ such that every object has a partial Frobenius monoid structure $(\iconbcm,\iconbcu,\iconwm)$ that satisfies the axioms in \Cref{figure:comultiplication} and uniformity, meaning that 
  \begin{enumerate}
    \item $(\iconbcm)_{X ⊗ Y} = (\iconbcm_X ⊗ \iconbcm_Y) ⨾ (\id_X ⊗ σ ⊗ \id_Y)\mbox{ and }(\iconbcm)_I = \id$;
    \item $(\iconwm)_{X ⊗ Y} = (\id_X ⊗ σ ⊗ \id_Y) ⨾ (\iconwm_X ⊗ \iconwm_Y) \mbox{ and }(\iconwm)_I = \id$;
    \item $(\iconbcu)_{X ⊗ Y} = (\iconbcu_X ⊗ \iconbcu_Y) \mbox{ and }(\iconbu)_I = \id$.
  \end{enumerate}
  \begin{figure}[h]
    \includegraphics[scale=0.35]{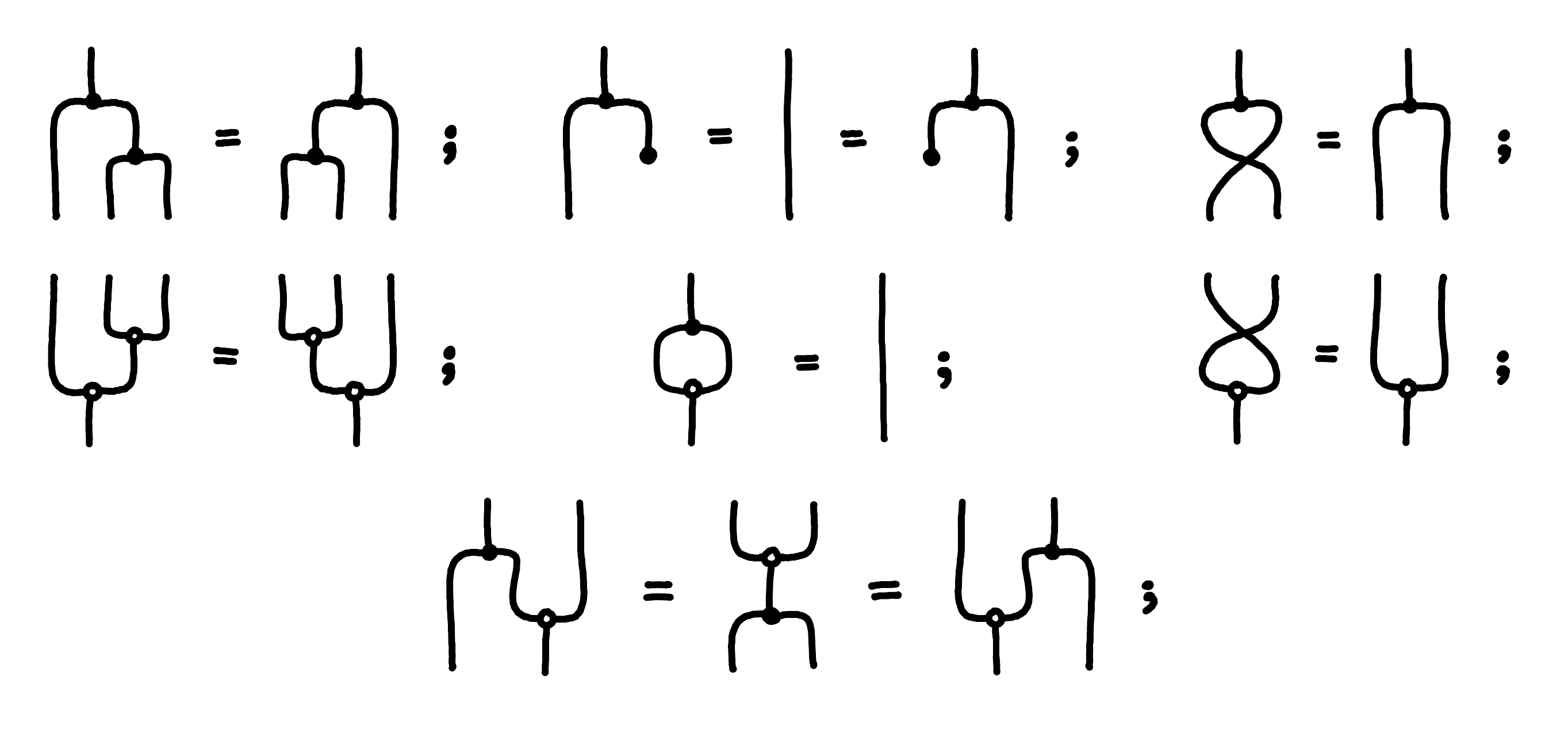}
    \caption{Theory of a partial Frobenius algebra.}
    \label{figure:comultiplication}
  \end{figure}
\end{definition}

\begin{proposition}
  A subdistribution on a set $X$ is a function $d ፡ X → [0,1]$ that is non-zero on a finite number of elements and that adds up to less or equal than one, 
  $$\sum_{d(x) > 0} d(x) \leq 1.$$
  Subdistributions form a monad, and the Kleisi category of the subdistribution monad forms a \discretePartialMarkovCategory{} \cite{dilavore:evidentialdecision}.
\end{proposition}
\begin{proof}
  Let us write $\Subd(X)$ for the set of subdistributions over a set. We claim that this extends to a monad in the category of sets and functions $\Subd ፡ \Set → \Set$. The multiplication $μ ፡ \Subd(\Subd(X)) → \Subd(X)$ is defined by $μ(ψ)(x) = \sum\nolimits_{ψ(d) > 0} d(x)$, and the unit $η ፡ X → \Subd(X)$ is defined by the \emph{Dirac's delta}, $η(x₀)(x) = [x = x₀]$, which is valued to $1$ whenever $x = x₀$ and is valued to $0$ otherwise.

  The Kleisli category of this monad has morphisms the \emph{partial stochastic channels} $f ፡ X → \Subd(Y)$. We write $f(y|x) ∈ [0,1]$ for the value of $f(x)$ on the input $y ∈ Y$, capturing the usual notation for conditionals in probability. Under this notation, composition on the Kleisli category becomes
  $$(f ⨾ g)(z|x) = \sum_{y ∈ Y} f(y|x) \cdot g(z|y).$$
  While tensoring is $(f ⊗ f')(y,y'|x,x') = f(y|x) \cdot f'(y'|x')$.
  The copy morphism is defined by $(\iconbcm)(x,y|z) = [x = y = z]$; the discard morphism is defined by $(\iconbcu)(|x) = 1$; and the comparator is given by $(\iconwm)(x|y,z) = [x = y = z]$. It is direct to check that these satisfy the axioms of a partial Frobenius algebra.
\end{proof}

\begin{remark}[Effect algebras]
  \defining{linkEffectAlgebra}{}
  The set of partial stochastic channels between two sets, $X → \Subd(Y)$, forms a particular algebraic structure known as an \emph{effect module}.

  An \emph{effect algebra} \cite{foulis94:effect,jacobs15:new,deWetering21:interval} is a set $E$ with a partial binary operation $(⊕) ፡ E × E → E$, a unary operation $(•)^{⊥} ፡ E → E$, and a constant $0 ∈ E$. We write $x ⊥ y$ whenever $x ⊕ y$ is well-defined and we write $1 = 0^{⊥}$. The effect algebra must satisfy
  \begin{enumerate}
    \item Commutativity, $x ⊕ y = y ⊕ x$, where $x ⊥ y$ implies $y ⊥ x$.
    \item Unitality, $x ⊕ 0 = x = 0 ⊕ x$, where $x ⊥ 0$.
    \item Associativity, $x ⊕ (y ⊕ z) = (x ⊕ y) ⊕ z$, where having both $y ⊥ z$ and $x ⊥ (y ⊕ z)$ implies both $x ⊥ y$ and $(x ⊕ y) ⊥ z$.
    \item Complementarity, $x^{⊥}$ is unique such that $x^{⊥} ⊕ x = 1$ and $x ⊥ x^{⊥}$.
  \end{enumerate}
  An \effectAlgebra{} homomorphism $f ፡ E → F$ must satisfy $f(1) = 1$, and $x ⊥ y$ must imply $f(x) ⊥ f(y)$, with $f(x ⊕ y) = f(x) ⊕ f(y)$. \EffectAlgebras{} form a \symmetricMonoidalCategory{}.

  The  unit interval, $[0,1]$, forms an \effectAlgebra{} with the binary sum (whenever it is contained on the interval), the unary complement $x^{⊥} = 1 - x$, and the zero; moreover, the unit interval $[0,1]$ with multiplication forms a monoid in the \monoidalCategory{} of \effectAlgebras{}. 
  
  Partial stochastic channels also form an effect algebra, with the same pointwise operations, but moreover, we can multiply them by a `scalar' from the unit interval $[0,1]$: they form a module in the \monoidalCategory{} of effect algebras. The structure of effect module interplays well with composition and tensoring of partial stochastic channels; so we will employ it later in \Cref{sec:casestudy}.
\end{remark}

\begin{remark}[Bayesian inversions]
  The \emph{Bayesian inversion} of a stochastic channel $g ፡ X → Y$ with respect to a distribution $f ፡ I → X$ is the stochastic channel $g^{\dagger}_f ፡ Y → X$ classically defined by
  $$g^{\dagger}_f(x|y) = \frac{g(y|x) \cdot f(x)}{\sum\nolimits_{x_• ∈ X} g(y|x_•) \cdot f(x_•)}.$$
  Bayesian inversions can be defined synthetically in any partial Markov category \cite{cho:jacobs:disintegration2019}. The Bayesian inversion of a morphism $g ፡ X → Y$ with respect to $f ፡ I → X$ is a morphism $g^{\dagger}_f ፡ Y → X$ satisfying the equation in \Cref{figure:bayesianinversion}, which translates to the above when the partial Markov category is that of subdistributions.
  \begin{figure}[h]
    \includegraphics[scale=0.35]{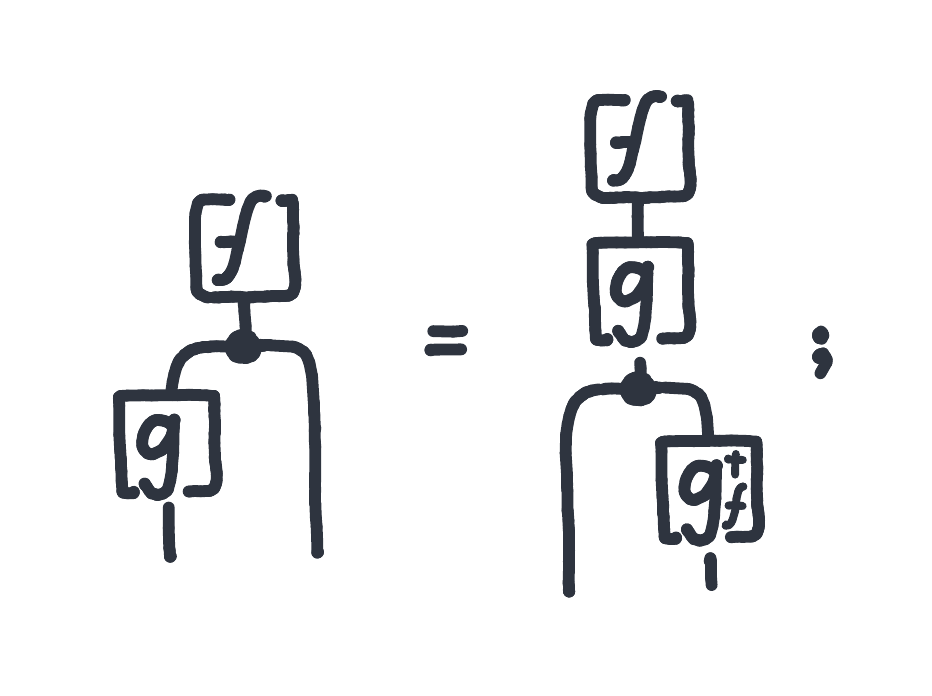}
    \caption{Bayesian inversion.}
    \label{figure:bayesianinversion}
  \end{figure}

The definition of a \emph{\partialMarkovCategory{}} includes asking a quasitotal Bayesian inversion for every morphism, with respect to any other morphism: this is the notion of \emph{quasitotal conditional} \cite{dilavore:evidentialdecision}, which we will not need in this thesis.
\end{remark}

Finally, let us justify that this \processTheory{} is enough to capture many of the features of classical probability theory. We will prove a synthetic version of Bayes theorem using the syntax of discrete partial Markov categories. The classical Bayes theorem prescribes that, after observing the output of a prior distribution through a channel, we should update our posterior distribution to be the Bayesian inversion of the channel with respect to the prior distribution, evaluated on the observation.

\begin{theorem}[Synthetic Bayes' Theorem]\label{th:bayes}
  In a discrete partial Markov category, observing a deterministic $x ፡ I → X$ from a prior distribution \(f ፡ I → A\) through a channel \(g ፡ A → X\) is the same, up to scalar, as the Bayesian inversion evaluated on the observation, $g^{\dagger}_f(x) ፡ I → A$.
\end{theorem}    
\begin{proof}
  \begin{figure}[h]
    \includegraphics[scale=0.35]{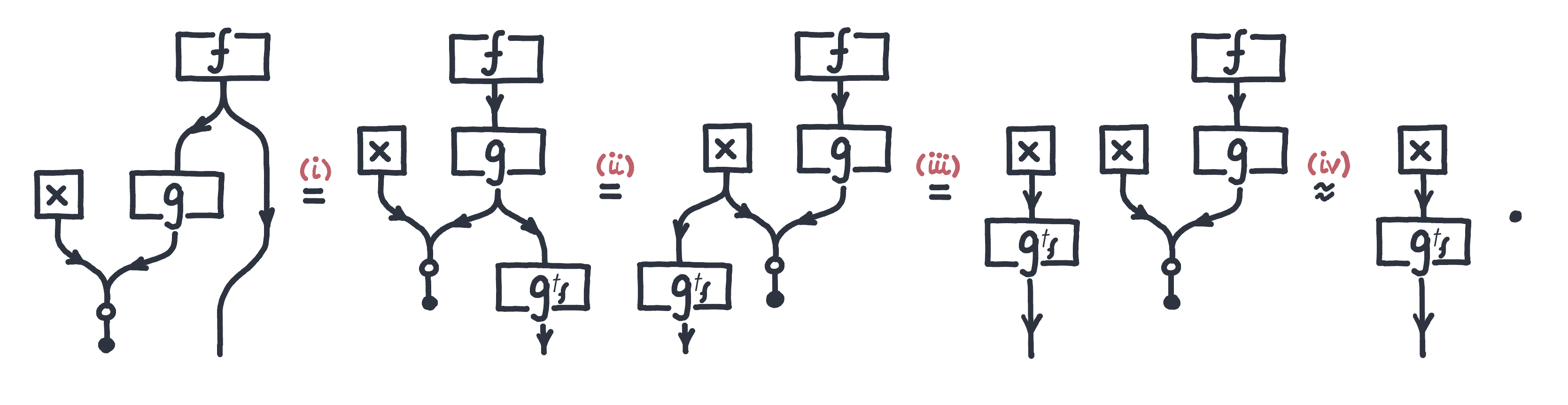}
    \caption{Bayes theorem.}
    \label{figure:bayestheorem}
\end{figure}
We employ string diagrams (\Cref{figure:bayestheorem}). Equalities follow from \emph{(i)} the definition of Bayesian inversion, \emph{(ii)} the partial Frobenius axioms, and \emph{(iii)} determinism of $y$.
\end{proof}

\subsection{Bibliography}
Fox's theorem, in its original formulation, is the construction of a right adjoint to the forgetful functor from \cartesianMonoidalCategories{} to \MonoidalCategories{}. This right adjoint is given by the category of cocommutative comonoids over a \monoidalCategory{} \cite{fox76}. The version here presented is esentially equivalent; in fact, it is called ``Fox's theorem'' in the work of Bonchi, Seeber and Sobocinski \cite{bonchi18}. 

The categorical approach to probability theory based on Markov categories is due to Fritz~\cite{fritz:markov2020} and prior work of Cho and Jacobs \cite{cho:jacobs:disintegration2019}. Multiple results of classical probability theory have been adapted to the framework of Markov categories by multiple authors~\cite{fritz2020infinite,fritz2021probability,fritz2019probability,fitz2021deFinetti}.
Markov categories have been further applied for formalising Bayes networks and other kinds of probabilistic and causal reasoning in categorical terms~\cite{fong_2013, jacobs2020logical, jacobs2021causal}. Their partial counterpart and the application to decision theory was introduced in joint work with Di Lavore \cite{dilavore:evidentialdecision}. \EffectAlgebras{} are due to Foulis and Bennett \cite{foulis94:effect}; Jacobs employed them for a probabilistic categorical logic \cite{jacobs15:new}; van de Wetering \cite{deWetering21:interval} characterized the unit interval in terms of \effectAlgebras{}.\clearpage{}%
\clearpage{}%
\section{Premonoidal Categories}
\label{sec:premonoidal-categories}

It might seem that \monoidalCategories{} are limited to pure imperative programming without effects. After all, the \emph{interchange law} seems to imply that the order in which two independent are executed does not matter. This is true, but again, category theory has a solution for us: \premonoidalCategories{}.

Category theory has two successful applications that are rarely combined: monoidal string diagrams \cite{joyal91:geometryOfTensorCalculus} and programming semantics \cite{moggi91}.
We use string diagrams to talk about quantum transformations \cite{abramsky09:categoricalquantum}, relational queries \cite{bonchi18}, and even computability \cite{pavlovic13}; at the same time, proof nets and the geometry of interaction \cite{girard89,blute96} have been widely applied in computer science \cite{abramsky02,hoshino14}.
On the other hand, we traditionally use monads and comonads, Kleisli categories and \premonoidalCategories{} to explain effectful functional programming \cite{hughes00,jacobs09,moggi91,power99:freyd,uustalu2008comonadic}. Even if we traditionally employ Freyd categories with a cartesian base \cite{power02}, we can also consider non-cartesian Freyd categories \cite{staton13}, which we call \emph{effectful categories}.

This section introduces \premonoidalCategories{} and \effectfulCategories{}. The next section will study their string diagrams in terms of \monoidalCategories{}, reducing them to a particular consideration on top of the theory of \monoidalCategories{}.

\subsection{Premonoidal Categories}
\label{sec:premonoidalcats}
\PremonoidalCategories{} are \monoidalCategories{} without the \emph{interchange law},
$(f ⊗ \id) ⨾ (\id ⊗ g) \neq (\id ⊗ g) ⨾ (f ⊗ \id)$.
This means that we cannot tensor any two arbitrary morphisms, $(f ⊗ g)$, without explicitly stating which one is to be composed first, $(f ⊗ \id) ⨾ (\id ⊗ g)$ or $(\id ⊗ g) ⨾ (f ⊗ \id)$, and the two compositions are not equivalent (\Cref{fig:noninterchange}).
  \begin{figure}[h]
    \centering
    \includegraphics[scale=0.6]{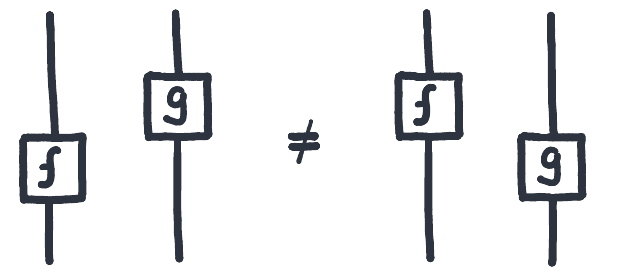}
    \caption{The interchange law does not hold in a premonoidal category.}
    \label{fig:noninterchange}
  \end{figure}

In technical terms, the tensor of a \premonoidalCategory{} $(⊗) ፡ ℂ \times ℂ → ℂ$ is not a functor, but only what is called a \emph{sesquifunctor}: independently functorial in each variable. Tensoring with any identity is itself a functor $(\bullet ⊗ \id) ፡ ℂ → ℂ$, but there is no functor $(\bullet ⊗ \bullet) ፡ ℂ \times ℂ \to ℂ$.

A practical motivation for dropping the interchange law can be found when describing transformations that affect a global state.
These effectful processes should not interchange in general, because the order in which we modify the global state is meaningful.
For instance, in the Kleisli category of the \emph{writer monad}, $(\Sigma^{\ast} \times \bullet) ፡ \Set \to \Set$ for some alphabet $\Sigma ∈ \Set$, we can consider the function $\mathsf{print} ፡ \Sigma^{\ast} \to \Sigma^{\ast} \times 1$. The order in which we ``print'' does matter (\Cref{fig:writermonad}).
  \begin{figure}[ht]
    \centering
    \includegraphics[scale=0.6]{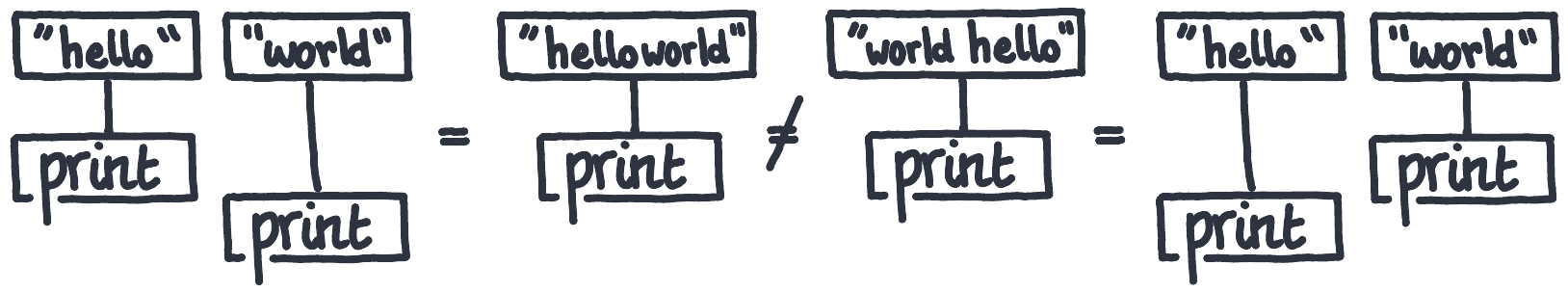}
    \caption{Writing does not interchange.}
    \label{fig:writermonad}
  \end{figure}

  Not surprisingly, the paradigmatic examples of \premonoidalCategories{} are the Kleisli categories of Set-based monads $T ፡ \Set \to \Set$  (more generally, of strong monads), which fail to be monoidal unless the monad itself is commutative \cite{guitart1980tenseurs,power97,power99:freyd,hedgesblog2019}.
  Intuitively, the morphisms are ``effectful'', and these effects do not always commute.

However, we may still want to allow some morphisms to interchange.
For instance, apart from asking the same associators and unitors of \monoidalCategories{} to exist, we ask them to be \emph{central}: which means that they interchange with any other morphism.
This notion of centrality forces us to write the definition of \premonoidalCategory{} in two different steps: first, we introduce the minimal setting in which centrality can be considered (\emph{\binoidalCategories{}} \cite{power99:freyd}) and then we use that setting to bootstrap the full definition of \premonoidalCategory{} with central coherence morphisms.

\begin{definition}[Binoidal category]
  \defining{linknoidalcategory}{}
  A \emph{binoidal category} is a category $ℂ$ endowed with an object $I \in ℂ$ and an object $A ⊗ B$ for each $A \in ℂ$ and $B \in ℂ$.
  There are functors
  $(A ⊗ \bullet) ፡ ℂ \to ℂ
    \mbox{, and }
    (\bullet ⊗ B) ፡ ℂ \to ℂ$
  that coincide on $(A ⊗ B)$. Note that $(\bullet ⊗ \bullet)$ is not being defined as a functor.
\end{definition}

Again, this means that we can tensor with identities (whiskering), functorially; but we cannot tensor two arbitrary morphisms: the interchange law stops being true in general.
The \emph{centre}, $\zentre(ℂ)$, is the wide subcategory of morphisms that do satisfy the interchange law with any other morphism.
That is, $f ፡ A \to B$ is \emph{central} if, for each $g ፡ A' \to B'$,
\begin{align*}
  & (f ⊗ \id_{A'}) ⨾ (\id_{B} ⊗ g)
   = (\id_{A} ⊗ g) ⨾ (f ⊗ \id_{B'}), \mbox{ and } \\
  & (\id_{A'} ⊗ f) ⨾ (g ⊗ \id_{B})
   = (g ⊗ \id_{A}) ⨾ (\id_{B'} ⊗ f).
\end{align*}

\begin{definition}
  \defining{linkpremonoidalcategory}
  A \emph{premonoidal category} is a \noidalCategory{} $(ℂ,⊗,I)$ together with the following coherence isomorphisms
  $\alpha_{A,B,C} ፡ A ⊗ (B ⊗ C) \to (A ⊗ B) ⊗ C$, $\rho_{A} ፡ A ⊗ I \to A$ and $\lambda_{A} ፡ I ⊗ A \to A$ which are central, natural \emph{separately at each given component}, and satisfy the pentagon and triangle equations.

  A \premonoidalCategory{} is \emph{strict} when these coherence morphisms are identities.
  A \premonoidalCategory{} is moreover \emph{symmetric} when it is endowed with a coherence isomorphism $\sigma_{A,B} ፡ A ⊗ B \to B ⊗ A$ that is central and natural at each given component and satisfies the symmetry condition and hexagon equations.
\end{definition}

\begin{remark}
  The coherence theorem of \monoidalCategories{} still holds for \premonoidalCategories{}: every premonoidal is equivalent to a strict one.
  We will construct the free strict \premonoidalCategory{} using \stringDiagrams{}.
  However, the usual \stringDiagrams{} for \monoidalCategories{} need to be restricted: in \premonoidalCategories{}, we cannot consider two morphisms in parallel unless any of the two is \emph{central}.
\end{remark}

\subsection{Effectful and Freyd Categories}
\label{sec:effectfulcats}

\PremonoidalCategories{} immediately present a problem: what are the premonoidal functors?
If we want them to compose, they should preserve the centrality of the coherence morphisms (so that the central coherence morphisms of $F ⨾ G$ are these of $F$ after applying $G$), but naively asking them to preserve all central morphisms rules out important examples~\cite{staton13}.
The solution is to explicitly choose some central morphisms that represent ``pure'' computations.
These do not need to form the whole centre: it could be that some morphisms considered  \emph{effectful} just ``happen'' to fall in the centre of the category, while we do not ask our functors to preserve them.
This is the well-studied notion of a \emph{non-cartesian Freyd category}, which we shorten to \emph{effectful monoidal category} or \emph{effectful category}.

\EffectfulCategories{} are \premonoidalCategories{} endowed with a chosen family of central morphisms.
These central morphisms are called \pure{} morphisms, contrasting with the general, non-central, morphisms that fall outside this family, which we call \effectful{}.

\begin{definition}
  \defining{linkeffectful}{}
  \defining{linkfreyd}{}
  An \emph{effectful category}
  is an identity-on-objects functor $\colorMon{𝕍} → \colorPre{ℂ}$ from a monoidal category $𝕍$ (the \pure{} morphisms, or ``values'') to a \premonoidalCategory{} $ℂ$ (the \effectful{} morphisms, or ``computations''), that strictly preserves all of the premonoidal structure and whose image is central. It is \emph{strict} when both are. A \emph{Freyd category} \cite{levy04:callbypush} is an effectful category where the \pure{} morphisms form a \cartesianMonoidalCategory{}.
\end{definition}

\EffectfulCategories{} solve the problem of defining premonoidal functors: a functor between \effectfulCategories{} needs to preserve only the \pure{} morphisms.
We are not losing expressivity: \premonoidalCategories{} are effectful with their centre, $\colorMon{\zentre}(ℂ) \to ℂ$. From now on, we study \effectfulCategories{}.

\begin{definition}[Effectful functor]
  \defining{linkeffectfulfunctor}
  Let $𝕍 → ℂ$ and $𝕎 → 𝔻$ be \effectfulCategories{}.
  An \emph{effectful functor} is a quadruple  $(F, F_{0},\varepsilon, \mu)$ consisting of a functor $F ፡ ℂ → 𝔻$ and a functor $F_{0} ፡ 𝕍 → 𝕎$ making the square commute,
  and two natural and pure isomorphisms $\varepsilon ፡ I' ≅ F(I)$ and $\mu ፡ F(A \otimes B) ≅ F(A) \otimes F(B)$ such that they make $F_{0}$ a monoidal functor. It is \emph{strict} if these are identities.
\end{definition}

When drawing \stringDiagrams{} in an \effectfulCategory{}, we shall use two different colours to declare if we are depicting either a value or a computation (\Cref{fig:prehelloworld}).

\begin{figure}[h]
  \centering
  \includegraphics[scale=0.6]{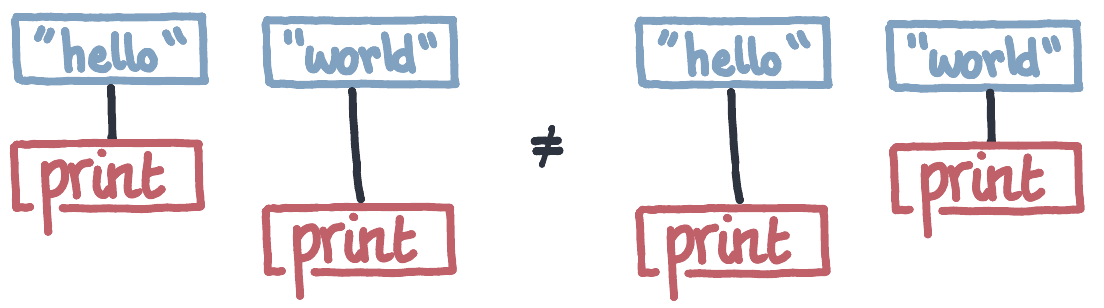}
  \caption{``Hello world'' is not ``world hello''.}
  \label{fig:prehelloworld}
\end{figure}

Here, the values \colorMon{``hello''} and \colorMon{``world''} satisfy the interchange law as in an ordinary monoidal category. However, the effectful computation \colorPre{``print''} does not need to satisfy the interchange law.
String diagrams like these can be found in the work of Alan Jeffrey \cite{jeffrey97}.
Jeffrey presents a clever mechanism to graphically depict the failure of interchange: all effectful morphisms need to have a control wire as an input and output.
This control wire needs to be passed around to all the computations in order, and it prevents them from interchanging.

  \begin{figure}[h]
    \centering
    \includegraphics[scale=0.6]{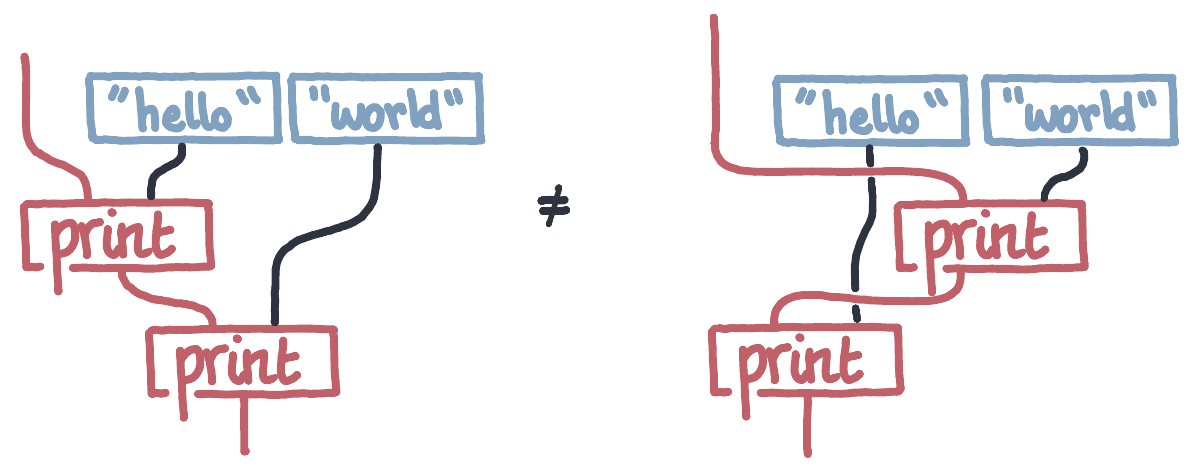}
    \caption{An extra wire prevents interchange.}
    \label{fig:helloworld}
  \end{figure}

Our interpretation of \monoidalCategories{} is as theories of resources.
We can interpret \premonoidalCategories{} as \monoidalCategories{} with an extra resource -- the \colorPre{``runtime''} -- that needs to be passed to all computations.
The next section promotes Jeffrey's observation into a theorem.

\begin{remark}
  After the next section, which reduces \premonoidalCategories{} to \monoidalCategories{}, the rest of this thesis deals mostly with \monoidalCategories{}.
  Why not study \premonoidalCategories{} instead of just \monoidalCategories{}? \PremonoidalCategories{} should be more general: they do not assume the interchange law, which is false in stateful systems.
  However, we give an argument for studying only \monoidalCategories{} in the next section: any \premonoidalCategory{} can be reinterpreted as a \monoidalCategory{} carrying an extra resource (an extra wire) representing the global state.
  We do not need to write a new theory of \premonoidalCategories{}: premonoidal categories are already \monoidalCategories{}, in which one wire is hidden. The following section makes this idea formal.
\end{remark}

\subsection{Bibliography}
This chapter follows closely the first part of ``Promonads and String Diagrams for Effectful Categories'', by this author \cite{roman:promonads-string-diagrams}.

\EffectfulCategories{} are the monoidal counterpart of a well-known notion: that of ``Freyd categories''. The name ``Freyd category'' sometimes assumes cartesianity of the pure morphisms, but it is also used for the general case; choosing to call ``effectful categories'' to the general case and reserving the name ``Freyd categories'' for the cartesian ones avoids this clash of nomenclature.
There exists also the more fine-grained notion of ``Cartesian effect category'' \cite{dumas11:cartesianEffect}, which generalizes Freyd categories and may further justify calling ``effectful category'' to the general case.

\newpage

\clearpage{}%
\clearpage{}%

\section{String Diagrams for Premonoidal Categories}
\label{sec:runtime}

\PremonoidalCategories{} give us a generalization of \monoidalCategories{} accounting for effectful computation but, at the same time, they do not need us to change our syntax yet.
String diagrams for \premonoidalCategories{} can be reduced to string diagrams for \monoidalCategories{}: they rely on the fact that the morphisms of the monoidal category freely generated over a \polygraph{} of generators are string diagrams on these generators, quotiented by topological deformations, as we saw in \Cref{sec:string-diagrams-of-strict-monoidal-categories}.

\subsection{Effectful Polygraphs}
We justify string diagrams for \premonoidalCategories{} by proving that the freely generated effectful category over a pair of \polygraphs{} (for pure and effectful generators, respectively) can be constructed as the freely generated \monoidalCategory{} over a particular \polygraph{} that includes an extra wire. In the same sense that \polygraphs{} are signatures for generating free \monoidalCategories{}, \polygraphCouples{} are signatures for generating free \effectfulCategories{}.

\begin{definition}
  \defining{linkpolygraphcouple}{}
  An \emph{\polygraphCouple{}} is a pair of \polygraphs{} $(𝓥,𝓖)$ sharing the same objects, $𝓥_{\mathrm{obj}} = 𝓖_{\mathrm{obj}}$.
  A morphism of \polygraphCouples{} $(u,f) ፡ (𝓥,𝓖) \to (𝓦,𝓗)$ is a pair of morphisms of \polygraphs{}, $u ፡ 𝓥 → 𝓦$ and $f ፡ 𝓖 → 𝓗$, such that they coincide on objects, $f_{\mathrm{obj}} = u_{\mathrm{obj}}$.
\end{definition}

\subsection{Adding Runtime}
  Recall from \Cref{sec:string-diagrams-of-strict-monoidal-categories} that there exists an adjunction between \polygraphs{} and \strictMonoidalCategories{}.
  Any monoidal category $ℂ$ can be seen as a \polygraph{}, $\mathsf{Forget}(ℂ)$, where the edges are morphisms 
  $$\mathsf{Forget}(ℂ)(A_{0},\dots, A_{n};B_{0}, \dots, B_{m}) = ℂ(A_{0} ⊗ \dots ⊗ A_{n};B_{0} ⊗ \dots ⊗ B_{m}),$$ and we forget about composition and tensoring.
  Given a \polygraph{} $𝓖$, the free strict monoidal category, which we will now write as $\MON(𝓖) = \mathsf{String}(𝓖)$, is the strict monoidal category that has as morphisms the string diagrams over the generators of the \polygraph{}.

  We will construct a similar adjunction between \polygraphCouples{} and \effectfulCategories{}. Let us start by formally adding the runtime to a free monoidal category.

\begin{definition}[Runtime monoidal category]
  \defining{linkruntimepolygraph}{}
  Let $(𝓥,𝓖)$ be an \polygraphCouple{}.
  Its \emph{runtime monoidal category}, $\MONRUN(𝓥,𝓖)$, is the monoidal category freely generated from adding an extra object -- the runtime, $\R$ -- to the input and output of every effectful generator in $𝓖$ (but not to those in $𝓥$), and letting that extra object be braided with respect to every other object of the category.

  In other words, it is the monoidal category freely generated by the following \polygraph{}, $\Run(𝓥,𝓖)$,
  (\Cref{fig:rungen}),
  assuming $A_{0},\dots,A_{n}$ and $B_{0},\dots,B_{m}$ are distinct from $\R$
  \begin{itemize}[leftmargin=2em]
    \item $\obj{\Run(𝓥,𝓖)} = \obj{𝓖} + \{ \R \} = \obj{𝓥} + \{ \R \}$,
    \item $\Run(𝓥,𝓖)(\R,A_{0},\dots,A_{n};\R, B_{0},\dots,B_{n}) = 𝓖(A_{0},\dots,A_{n}; B_{0},\dots,B_{n})$,
    \item $\Run(𝓥,𝓖)(A_{0},\dots,A_{n}; B_{0},\dots,B_{n}) = 𝓥(A_{0},\dots,A_{n}; B_{0},\dots,B_{n})$,
    \item $\Run(𝓥,𝓖)(\R,A_{0};A_{0},\R) = \Run(𝓥,𝓖)(A_{0},\R;\R,A_{0}) = \{\sigma\}$,
  \end{itemize}
  with $\Run(𝓥,𝓖)$ empty in any other case, and quotiented by the braiding axioms for $\R$ (\Cref{fig:runaxiom}).
  \begin{figure}[ht]
    \centering
    \includegraphics[scale=0.6]{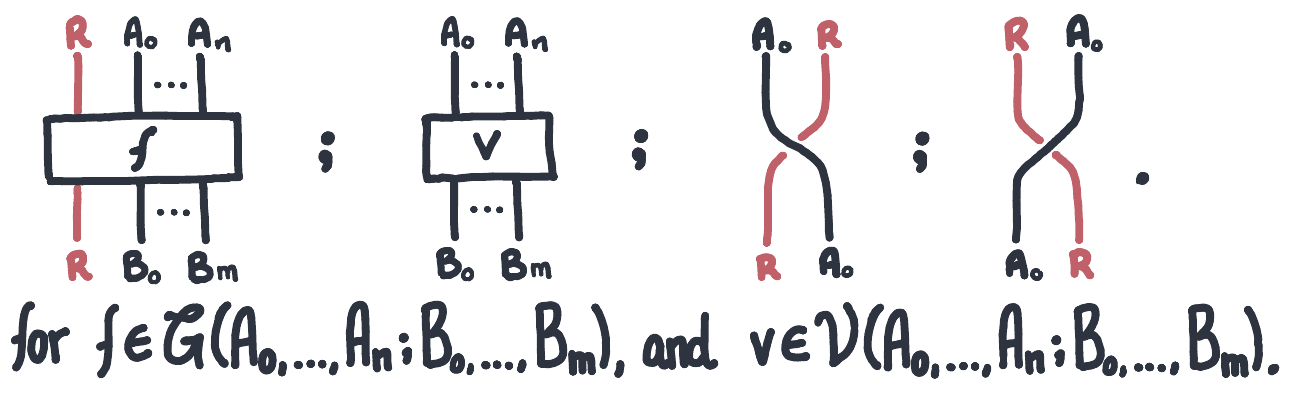}
    \caption{Generators for the runtime monoidal category.}
    \label{fig:rungen}
  \end{figure}
  \begin{figure}[ht]
    \centering
    \includegraphics[scale=0.6]{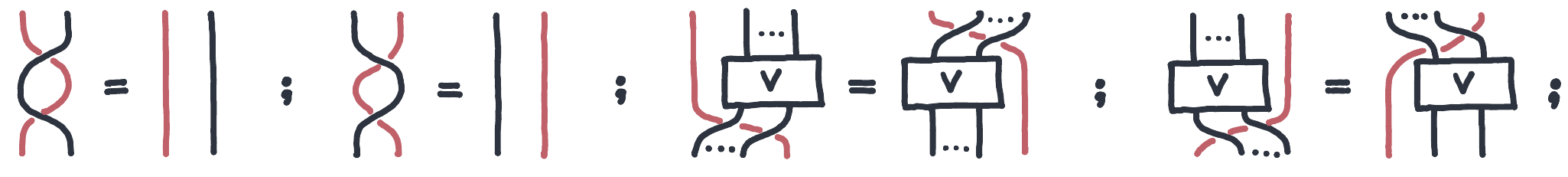}
    \caption{Axioms for the runtime monoidal category.}
    \label{fig:runaxiom}
  \end{figure}
\end{definition}

Somehow, we are asking the runtime $\R$ to be in the Drinfeld centre \cite{drinfeld10} of the monoidal category.
The extra wire that $\R$ provides is only used to prevent interchange, and so it does not really matter where it is placed in the input and the output.
We can choose to always place it on the left, for instance -- and indeed we will be able to do so -- but a better solution is to just consider objects ``up to some runtime braidings''.
This is formalized by the notion of \emph{braid clique}.

\begin{definition}[Braid clique]
  Given any list of objects $A_{0},\dots,A_{n}$ in $\obj{𝓥} = \obj{𝓖}$,
  we construct a \emph{clique}  \cite{trimble:coherence,shulman20182} in the category $\MONRUN(𝓥,𝓖)$: we consider the objects, $A_{0} ⊗ \dots ⊗ \R_{(i)} ⊗ \dots ⊗ A_{n}$, created by inserting the runtime $\R$ in all of the possible $0 \leq i \leq n+1$ positions; and we consider the family of commuting isomorphisms constructed by braiding the runtime,
  \[ \sigma_{i,j} ፡ A_{0} ⊗ \dots ⊗ \R_{(i)} ⊗ \dots ⊗ A_{n} \to  A_{0} ⊗ \dots ⊗ \R_{(j)} ⊗ \dots ⊗ A_{n}.\]
  We call this the \emph{braid clique}, $\Braid(A_{0},\dots,A_{n})$, on that list.
\end{definition}

\begin{definition}
  A \emph{braid clique morphism}, 
  $$f ፡ \Braid(A_{0},\dots,A_{n}) \to \Braid(B_{0},\dots,B_{m}),$$ 
  is a family of morphisms in the runtime monoidal category, $\MONRUN(𝓥,𝓖)$, from each of the objects of first clique to each of the objects of the second clique,
  \[f_{ik} ፡ A_{0} ⊗ \dots ⊗ \R_{(i)} ⊗ \dots ⊗ A_{n} \to
    B_{0} ⊗ \dots ⊗ \R_{(k)} ⊗ \dots ⊗ B_{m},\]
  that moreover commutes with all braiding isomorphisms, $f_{ij} ⨾ \sigma_{jk} = \sigma_{il} ⨾ f_{}$.
\end{definition}

A braid clique morphism  $f ፡ \Braid(A_{0},\dots, A_{n}) \to \Braid(B_{0},\dots, B_{m})$ is fully determined by \emph{any} of its components, by pre/post-composing it with braidings.
In particular, a braid clique morphism is always fully determined by its leftmost component $f_{00} ፡ \R ⊗ A_{0} ⊗ \dots ⊗ A_{n} \to \R ⊗ B_{0} ⊗ \dots ⊗ B_{m}$.

\begin{lemma}\label{lemma:eff-is-premonoidal}
  Let $(𝓥,𝓖)$ be an \polygraphCouple{}.
  There exists a \premonoidalCategory{}, $\EFF(𝓥,𝓖)$, that has objects the braid cliques, $\Braid(A_{0},\dots,A_{n})$, in the runtime monoidal category $\MONRUN(𝓥,𝓖)$, and as morphisms the braid clique morphisms between them. 
\end{lemma}
\begin{proof}
  First, let us give $\EFF(𝓥,𝓖)$ the structure of a category.
  The identity on $\Braid(A_{0},\dots,A_{n})$ is the identity on $\R ⊗ A$.
  The composition of a morphism $\R  ⊗ A \to \R ⊗ B$ with a morphism $\R ⊗ B \to \R ⊗ C$ is their plain composition in $\MONRUN(𝓥,𝓖)$.

  Let us now check that it is moreover a \premonoidalCategory{}.
  Tensoring of cliques is given by concatenation of lists, which coincides with the tensor in $\MONRUN(𝓥,𝓖)$. However, it is interesting to note that the tensor of morphisms cannot be defined in this way:
  a morphism $\R ⊗ A \to \R ⊗ B$ cannot be tensored with a morphism $\R ⊗ A' \to \R ⊗ B'$ to obtain a morphism $\R ⊗ A ⊗ A' \to \R ⊗ B ⊗ B'$.

  Whiskering of a morphism $f ፡ \R  ⊗ A \to \R ⊗ B$ is defined with braidings in the left case, $\R ⊗ C ⊗ A \to \R ⊗ C ⊗ B$, and by plain whiskering in the right case, $\R ⊗ A ⊗ C \to \R ⊗ B ⊗ C$, as depicted in \Cref{fig:whiskering}.
   \begin{figure}[!ht]
    \centering
    \includegraphics[scale=0.65]{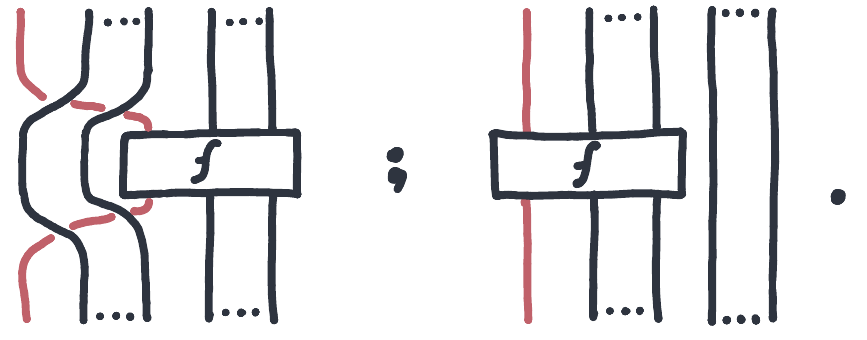}
    \caption{Whiskering in the runtime premonoidal category.}
    \label{fig:whiskering}
  \end{figure}
  Finally, the associators and unitors are identities, which are always natural and central.
\end{proof}

\begin{lemma}
  \label{lemma:identity-mon-eff}
  Let $(𝓥,𝓖)$ be an \polygraphCouple{}.
  There exists an identity-on-objects functor \(\MON(𝓥) \to \EFF(𝓥,𝓖)\) that strictly preserves the premonoidal structure and whose image is central. 
\end{lemma}
\begin{proof}
  A morphism $v \in \MON(𝓥)(A,B)$ induces a morphism $(\id_{\R} ⊗ v) \in \MONRUN(𝓥,𝓖)(\R ⊗ A,\R ⊗ B)$, which can be read as a morphism of cliques $(\id_{\R} ⊗ v) \in \EFF(𝓥,𝓖)(A,B)$.
  This is tensoring with an identity, which is indeed functorial.

  Let us now show that this functor strictly preserves the premonoidal structure.
  The fact that it preserves right whiskerings is immediate.
  The fact that it preserves left whiskerings follows from the axioms of symmetry (\Cref{fig:whiskeringeq}, left).
  Associators and unitors are identities, which are preserved by tensoring with an identity.
  \begin{figure}[ht]
    \centering
    \includegraphics[scale=0.52]{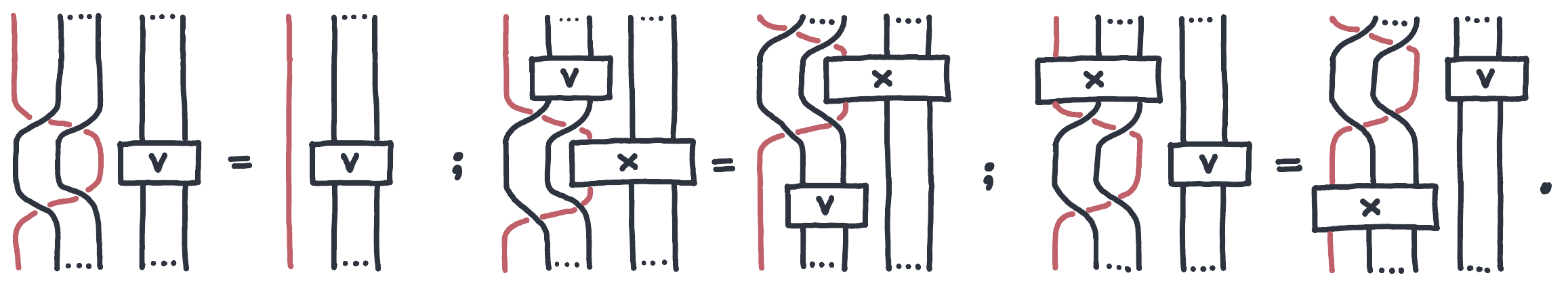}
    \caption{Preservation of whiskerings, and centrality.}
    \label{fig:centrality}    
    \label{fig:whiskeringeq}
  \end{figure}
  Finally, we can check by string diagrams that the image of this functor is central, interchanging with any given $x ፡ \R ⊗ C \to \R ⊗ D$ (\Cref{fig:centrality}, center and right).
\end{proof}

\begin{lemma}
  \label{lemma:freeness}
  Let $(𝓥,𝓖)$ be an \polygraphCouple{} and consider the \effectfulCategory{} determined by \(\MON(𝓥) \to \EFF(𝓥,𝓖)\).
  Let $𝕍 \to ℂ$ be a strict effectful category endowed with an \polygraphCouple{} morphism $F ፡ (𝓥,𝓖) \to \mathcal{U}(𝕍, ℂ)$.
  There exists a unique strict effectful functor from $(\MON(𝓥) \to \EFF(𝓥,𝓖))$ to $(𝕍 \to ℂ)$ commuting with $F$ as an \polygraphCouple{} morphism. 
\end{lemma}
\begin{proof}
  By freeness, there already exists a unique strict monoidal functor $H_{0} ፡ \MON(𝓥) \to 𝕍$ that sends any object $A \in \obj{𝓥}$ to $F_{obj}(A)$.
  We will show there is a unique way to extend this functor together with the hypergraph assignment $𝓖 \to ℂ$ into a functor $H ፡ \EFF(𝓥,𝓖) \to ℂ$.
  Giving such a functor amounts to give some mapping of morphisms containing the runtime $\R$ in some position in their input and output,
  \[f ፡ A_{0} ⊗ \dots ⊗ \R ⊗ \dots ⊗ A_{n} \to B_{0} ⊗ \dots ⊗ \R ⊗ \dots ⊗ B_{m}\]
  to morphisms $H(f) ፡ FA_{0} ⊗ \dots  ⊗ FA_{n} \to FB_{0} ⊗ \dots  ⊗ FB_{n}$ in $ℂ$, in a way that preserves composition, whiskerings, inclusions from $\MON(𝓥)$, and that is invariant to composition with braidings.
  In order to define this mapping, we will perform structural induction over the monoidal terms of the runtime monoidal category of the form
  $\MONRUN(𝓥,𝓖)(A_{0} ⊗ \dots ⊗ \R^{(i)} ⊗ \dots ⊗ A_{n}, \R ⊗ B_{0} ⊗ \dots  ⊗ \R^{(j)} ⊗  \dots ⊗ B_{m})$ and show that it is the only mapping with these properties (\Cref{fig:assignment}). 
  
  Monoidal terms in a strict, freely presented, monoidal category are formed by identities ($\id$), composition $(⨾)$, tensoring $(⊗)$, and some generators (in this case, in \Cref{fig:rungen}). Monoidal terms are subject to \emph{(i)} functoriality of the tensor, $\id ⊗ \id = \id$ and $(f ⨾ g) ⊗ (h ⨾ k) = (f ⊗ h) ⨾ (g ⊗ k)$; \emph{(ii)} associativity and unitality of the tensor, $f ⊗ \id_{I} = f$ and $f ⊗ (g ⊗ h) = (f ⊗ g) ⊗ h$; \emph{(iii)} the usual unitality, $f ⨾ \id = f$ and $\id ⨾ f = f$ and associativity $f ⨾ (g ⨾ h) = (f ⨾ g) ⨾ h$; \emph{(iv)} the axioms of our presentation (in this case, in \Cref{fig:runaxiom}).
  \begin{figure}[ht]
    \centering
    \includegraphics[scale=0.60]{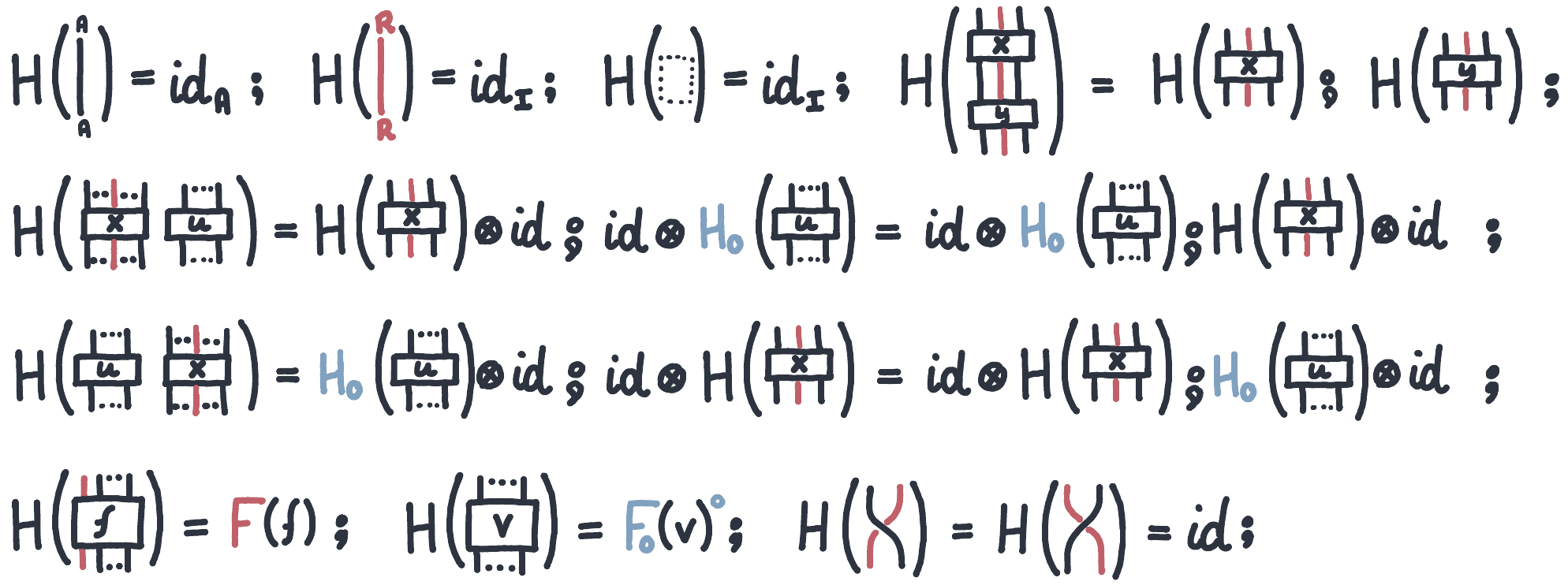}
    \caption{Assignment on morphisms, defined by structural induction on terms.}
    \label{fig:assignment}
  \end{figure}

  \begin{itemize}
    \item
      If the term is an identity, it can be \emph{(i)} an identity on an object $A \in \obj{(𝓥,𝓖)}$, in which case it must be mapped to the same identity by functoriality, $H(\id_{A}) = \id_{A}$; \emph{(ii)} an identity on the runtime, in which case it must be mapped to the identity on the unit object, $H(\id_{\R}) = \id_{I}$; or \emph{(iii)} an identity on the unit object, in which case it must be mapped to the identity on the unit, $H(\id_{I}) = \id_{I}$.

    \item
      If the term is a composition, $(f ⨾ g) ፡ A_{0} ⊗ \dots ⊗ \R ⊗ \dots ⊗ A_{n} \to C_{0} ⊗ \dots ⊗ \R ⊗ \dots ⊗ C_{k}$, it must be along a boundary of the form $B_{0} ⊗ \dots ⊗ \R ⊗ \dots ⊗ B_{m}$: this is because every generator leaves the number of runtimes, $\R$, invariant.
      Thus, each one of the components determines itself a braid clique morphism.
      We must preserve composition of braid clique morphisms, so we must map $H(f ⨾ g) = H(f) ⨾ H(g)$.

    \item
      If the term is a tensor of two terms,  $(x ⊗ u) ፡ A_{0} ⊗ \dots ⊗ \R ⊗ \dots ⊗ A_{n} \to B_{0} ⊗ \dots ⊗ \R ⊗ \dots ⊗ B_{m}$, then only one of them was a term taking $\R$ as input and output (without loss of generality, assume it to be the first one) and the other was not: again, by construction, there are no morphisms taking one $\R$ as input and producing none, or viceversa.
      We split this morphism into $x ፡ A_{0} ⊗\dots⊗ \R ⊗ \dots ⊗ A_{i-1} \to B_{0} ⊗ \dots ⊗ \R ⊗ \dots ⊗ B_{j-1}$ and $u ፡ A_{i} ⊗ \dots ⊗ A_{n} \to B_{j} ⊗ \dots ⊗ B_{m}$.

      Again by structural induction, this time over terms $u ፡ A_{i} ⊗ \dots ⊗ A_{n} \to B_{j} ⊗ \dots ⊗ B_{m}$,
      we know that the morphism must be either a generator in $𝓥(A_{i},\dots,A_{n};B_{j},\dots,B_{n})$ or a composition and tensoring of them. That is, $u$ is a morphism in the image of $\MON(𝓥)$, and it must be mapped according to the functor $H_{0} ፡ \MON(𝓥) \to 𝕍$.

      By induction hypothesis, we know how to map the morphism $x ፡ A_{0} ⊗ \dots ⊗  \R ⊗ \dots ⊗ A_{i-1} \to B_{0} ⊗ \dots ⊗ \R ⊗ \dots ⊗ B_{j-1}$.
      This means that, given any tensoring $x ⊗ u$, we must map it to $H(x ⊗ u) = (H(x) ⊗ \id) ⨾ (\id ⊗ H_{0}(u)) =  (\id ⊗ H_{0}(u)) ⨾ (H(x) ⊗ \id)$, where $H_{0}(u)$ is central.

    \item
      If the string diagram consists of a single generator, $f ፡ \R ⊗ A \to \R ⊗ B$, it can only come from a generator $f \in \Run(𝓥,𝓖)(\R,A_{0},\dots,A_{n};\R, B_{0},\dots,B_{m}) = 𝓖(A_{0},\dots,A_{n}; B_{0},\dots,B_{m})$,
      which must be mapped to $H(f) = F(f) \in ℂ(A_{0} ⊗ \dots ⊗ A_{n}, B_{0} ⊗ \dots ⊗ B_{m})$.
      If the string diagram consists of a single braiding, it must be mapped to the identity, because the want the assignment to be invariant to braidings.

  \end{itemize}

  Now, we need to prove that this assignment is well-defined with respect to the axioms of these monoidal terms. Our reasoning follows \Cref{fig:assignmentwell}.

  \begin{figure}
    \centering
    \includegraphics[scale=0.60]{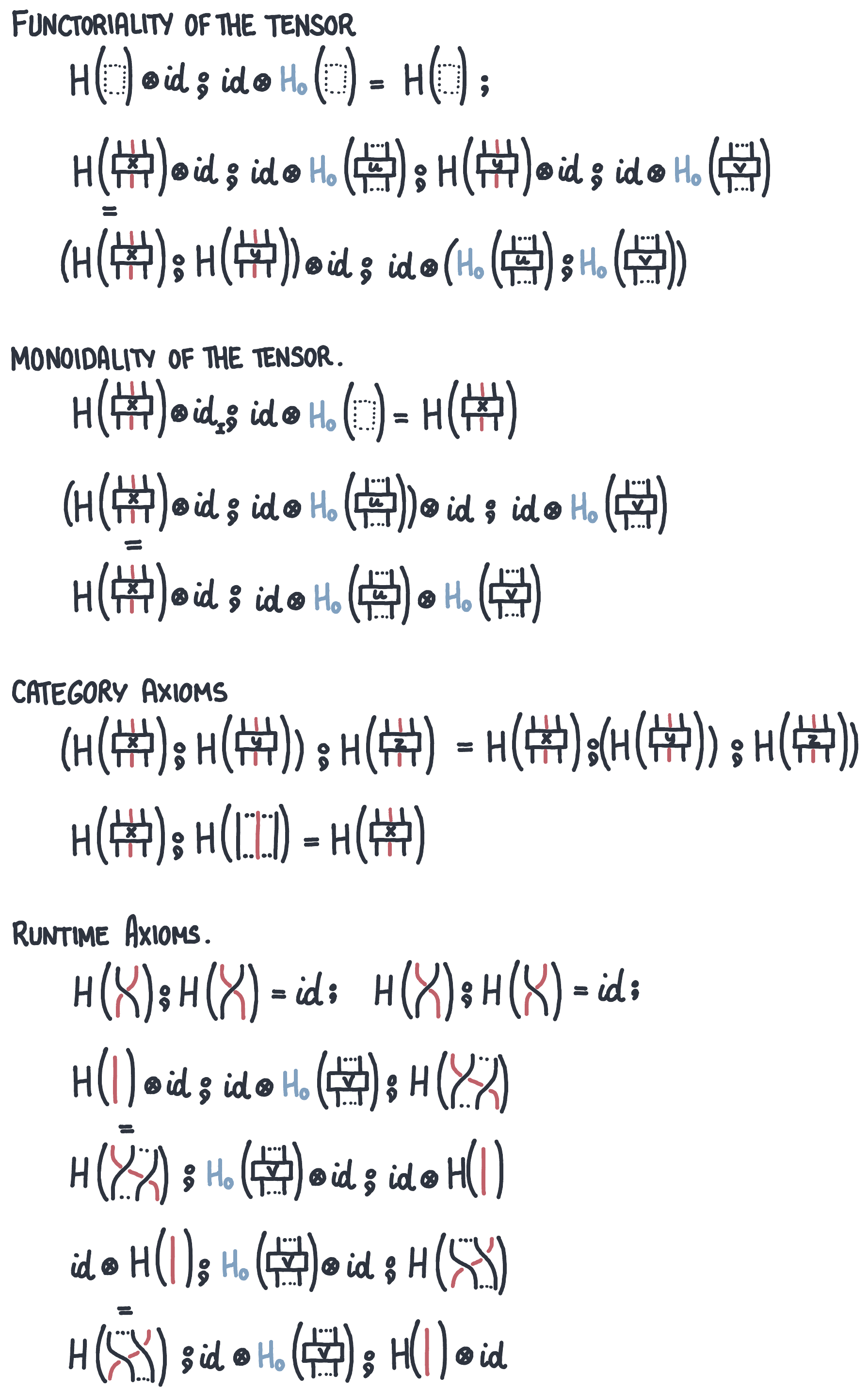}
    \caption{The assignment is well defined.}
    \label{fig:assignmentwell}
  \end{figure}

  \begin{itemize}
    \item
    The tensor is functorial. We know that $H(\id ⊗ \id) = H(\id)$, both are identities and that can be formally proven by induction on the number of wires.
    Now, for the interchange law, consider a quartet of morphisms that can be composed or tensored first and such that, without loss of generality, we assume the runtime to be on the left side. Then, we can use centrality to argue that
    \begin{align*}
      H((x ⊗ u) ⨾ (y ⊗ v)) &= (H(x) ⊗ \id) ⨾ (\id ⊗ H_{0}(u)) ⨾  (H(y) ⊗ \id) ⨾ (\id ⊗ H_{0}(v))
      \\&= ((H(x)⨾ H(y)) ⊗ \id) ⨾ (\id ⊗ (H_{0}(u) ⨾ H_{0}(v)))
      \\&= H((x ⨾ y) ⊗ (u ⨾ v)).
    \end{align*}

    \item
    The tensor is monoidal. We know that $H(x ⊗ \id_{I}) = (H(x) ⊗ \id_{I}) ⨾ (\id ⊗ \id_{I}) = H(x)$.
    Now, for associativity, consider a triple of morphisms that can be tensored in two ways and such that, without loss of generality, we assume the runtime to be on the left side.
    Then, we can use centrality to argue that
    \begin{align*}
      H((x ⊗ u) ⊗ v) 
      &=  (((H(x) ⊗ \id)  ⨾ (\id ⊗ H_{0}(u))) ⊗ \id) ⨾ \id ⊗ H_{0}(v)
      \\&= (H(x) ⊗ \id) ⨾ (\id ⊗ H_{0}(u) ⊗ H_{0}(v))
      \\&= H(x ⊗ (u ⊗ v))
    \end{align*}

    \item
    The terms form a category. And indeed, it is true by construction that $H(x ⨾ (y ⨾ z)) = H((x ⨾ y) ⨾ z)$ and also that $H(x ⨾ \id) = H(x)$ because $H$ preserves composition.

    \item
    The runtime category enforces some axioms. The composition of two braidings is mapped to the identity by the fact that $H$ preserves composition and sends both to the identity.  Both sides of the braid naturality over a morphism $v$ are mapped to $H_{0}(v)$; with the multiple braidings being mapped again to the identity.
  \end{itemize}
  Thus, $H$ is well-defined and it defines the only possible assignment and the only possible strict premonoidal functor.
\end{proof}

\begin{theorem}[Runtime as a resource]
  \label{theorem:runtime-as-a-resource}
  The free strict effectful category over an \polygraphCouple{} $(𝓥,𝓖)$ is $\MON(𝓥) \to \EFF(𝓥,𝓖)$. Its morphisms $A \to B$ are in bijection with the morphisms $\R ⊗ A \to \R ⊗ B$ of the runtime monoidal category,
  \[\EFF(𝓥,𝓖)(A,B) \cong \MONRUN(𝓥,𝓖)(\R ⊗ A, \R ⊗ B).\]
\end{theorem}
\begin{proof}
  We must first show that $\MON(𝓥) \to \EFF(𝓥,𝓖)$ is an effectful category.
  The first step is to see that $\EFF(𝓥,𝓖)$ forms a premonoidal category (\Cref{lemma:eff-is-premonoidal}).
  We already know that $\MON(𝓥)$ is a monoidal category: a strict, freely generated one.
  There exists an identity on objects functor, \(\MON(𝓥) \to \EFF(𝓥,𝓖)\), that strictly preserves the premonoidal structure and centrality (\Cref{lemma:identity-mon-eff}).

  Let us now show that it is the free one over the \polygraphCouple{} $(𝓥,𝓖)$.
  Let $𝕍 \to ℂ$ be an effectful category, with an \polygraphCouple{} map $F ፡ (𝓥,𝓖) \to \mathcal{U}(𝕍,ℂ)$.
  We can construct a unique effectful functor from $(\MON(𝓥) \to \EFF(𝓥,𝓖))$ to $(𝕍 → ℂ)$ giving its universal property (\Cref{lemma:freeness}).
\end{proof}

\begin{corollary}[String diagrams for effectful categories]
  We can use string diagrams for effectful categories, quotiented under the same isotopy as for monoidal categories, provided that we do represent the runtime as an extra wire that needs to be the input and output of every effectful morphism.
\end{corollary}

\subsection{Example: a Theory of Global State}
Let us provide an example of reasoning using the string diagrams for \premonoidalCategories{}. Imperative programs are characterized by the presence of a global state that can be mutated. Reading or writing to this global state constitutes an effectful computation: the order of operations that affect some global state cannot be changed. 
Let us propose a simple theory of global state and let us show that it is enough to capture the phenomenon of \emph{race conditions}.

\begin{definition}
  The \emph{theory of global state} is given by a single object $X$; two pure generators, $(\iconbcm) ፡ X → X ⊗ X$ and $(\iconbcu) ፡ X → I$, allowing copy and discard; and two effectful generators, $\mathsf{put} ፡ X → I$ and $\mathsf{get} ፡ I → X$, quotiented by the equations in \Cref{fig:state-laws}.
  \begin{figure}[ht]
    \centering
    \includegraphics[scale=0.35]{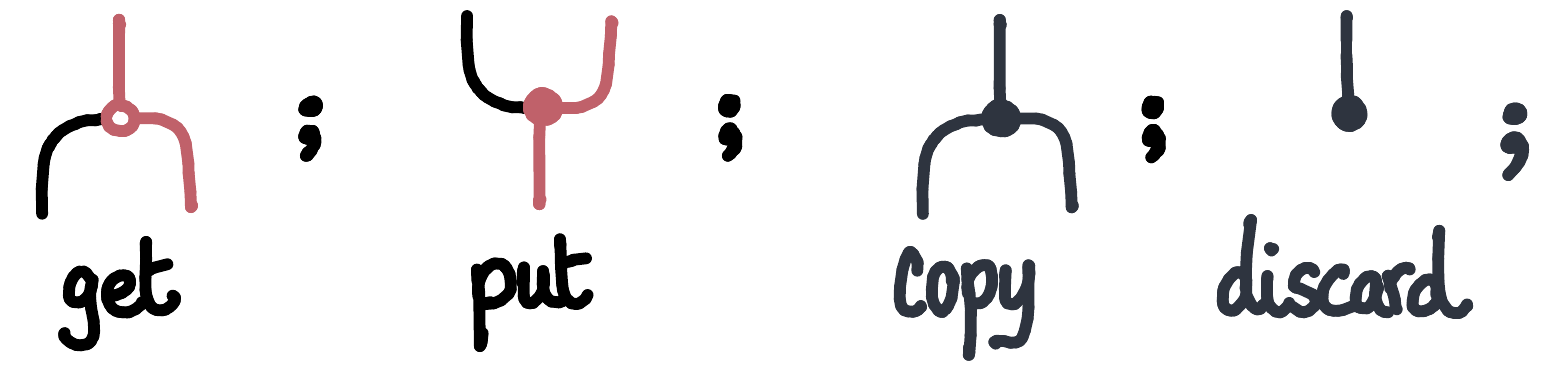}
    \caption{Generators of the theory of global state.}
    \label{fig:state-generators}
  \end{figure}
\end{definition}

These two $\mathsf{put}$ and $\mathsf{get}$ generators, without extra axioms, are enough for capturing what happens when a process can \emph{send} or \emph{receive} resources; in further sections, we will develop this theory. Right now, we are only concerned with the theory of a single \emph{global state}, accessed by a single process: we can impose some axioms that assert that the memory was not changed by anyone but this single process.
\begin{figure}[ht]
  \centering
  \includegraphics[scale=0.35]{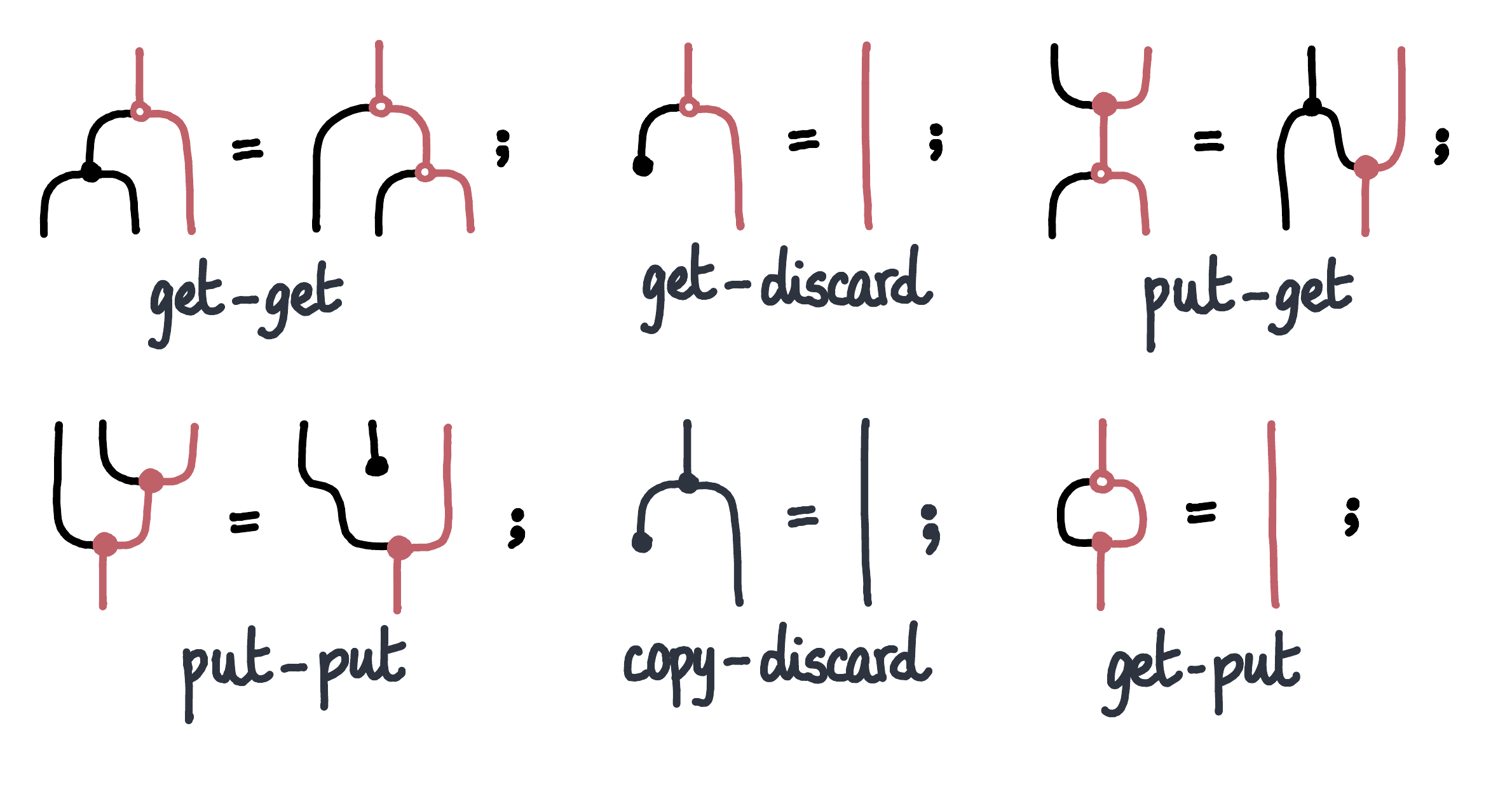}
  \caption{Axioms of the theory of global state.}
  \label{fig:state-laws}
\end{figure}

The equations in \Cref{fig:state-laws} say that: \emph{(i)} reading the global state twice gets us the same result, \emph{(ii)} reading the global state and discarding the result is the same as doing nothing, \emph{(iii)} writing something to the global state and then reading it is the same as keeping a copy of it, \emph{(iv)} writing twice to the global state keeps only the last thing that was written, \emph{(v)} copying and discarding a copy is doing nothing, and \emph{(vi)} reading something and immediately writing it to the global state is the same as doing nothing.                                          

\begin{proposition}[Race conditions]
  Concurrently mixing two processes that share a global state, $f$ and $g$, can produce four possible results: \emph{(i)} only the result of the first one is preserved, \emph{(ii)} only the result of the second one is preserved, or \emph{(iii,iv)} the composition of both is preserved, in any order, $f ⨾ g$ or $g ⨾ f$.
\end{proposition}
\begin{proof}
  We work in the theory of global state adding two processes, $f ፡ X → X$ and $g ፡ X → X$, that can moreover be discarded, meaning $f ⨾ ε = g ⨾ ε = ε$. We employ the string diagrams of \premonoidalCategories{}.
  \begin{figure}[ht]
    \centering
    \includegraphics[scale=0.35]{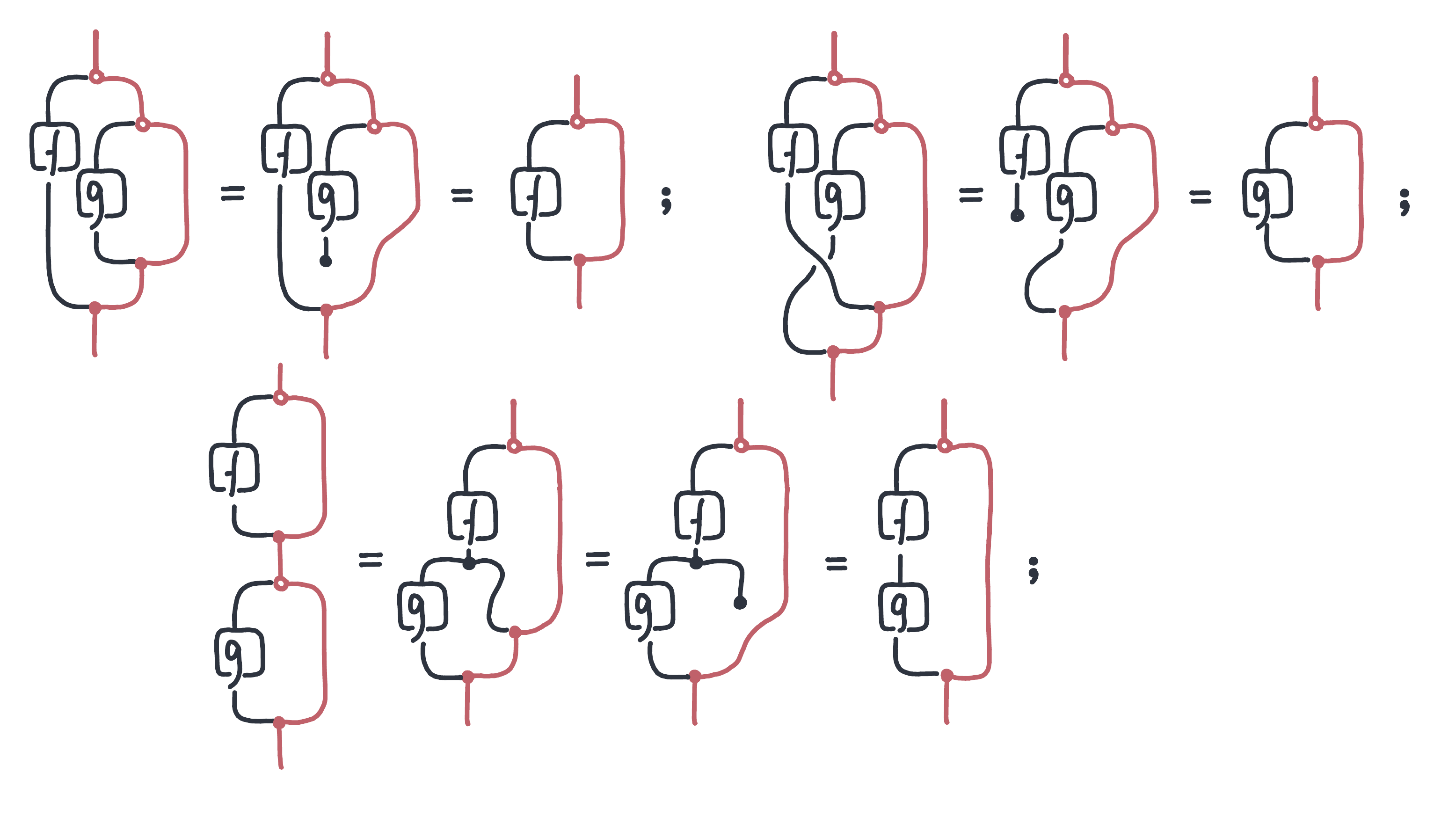}
    \caption{Race conditions in the theory of global state.}
    \label{fig:state-race}
  \end{figure}
  The first three diagrams in \Cref{fig:state-race} correspond to the first three cases, the last one is analogous to the third one.
\end{proof}

\subsection{Bibliography}
  Alan Jeffrey pioneered a string diagrammatic representation of programs using an extra wire to represent runtime \cite{jeffrey1997:premonoidal}, all the credit for this idea should go there.
  Staton and M{\o}gelberg \cite{mogelberg14} already showed how any \premonoidalCategory{} could be reinterpreted as the Kleisli category of a state promonad.

  We take these ideas a step further, showing that a particular syntax for \monoidalCategories{} can be used to talk about \premonoidalCategories{} as well. Our study may demystify \premonoidalCategories{}, or at least make them more accessible to a category theorist already interested in \monoidalCategories{}: \premonoidalCategories{} are simply \monoidalCategories{} with a hidden state object.

\newpage

\clearpage{}%

\chapter{Context Theory}
\label{chapter:compositional-algebra}
\section*{Context Theory}
Our goal in this thesis is to study how processes compose from their constituent parts while keeping for each one of these incomplete parts (each context) a meaning of its own: the meaning -- the semantics -- of the whole process is then determined by the semantics of each one of its parts, and how they compose. This is the principle of compositionality.

In the categorical framework, the structures that govern composition and decomposition are Lambek's \multicategories{}: categories where morphisms have multiple inputs that get composed into a single output. We will decompose \multicategories{} and compare them to \monoidalCategories{} using \profunctors{} and \dinaturality{}, which we claim to be the right mathematical tools to talk about abstract process composition in \Cref{sec-profunctors}.

We give a short exposition of \multicategories{} in \Cref{sec-multicategories}, and then introduce a type of \multicategory{} with the particular property that every transformation can be decomposed into smaller ones: \emph{\malleableMulticategories{}}, which we prove equivalent to \promonoidalCategories{} in \Cref{sec-malleable-multicategories}. Our main result in this section is a characterization of the \multicategory{} that governs how morphisms compose in a category -- the \malleableMulticategory{} of \emph{spliced arrows} -- as cofreely generated by the category, \Cref{sec:splice-contour-adjunction}. This result is a variant on a recent result by Melliés and Zeilberger \cite{mellies22:parsing}; and \Cref{chapter:monoidal-context-theory} will extend this theory for the first time to the setting of \monoidalCategories{}.

\clearpage{}%
\section{Profunctors and Coends}
\label{sec-profunctors}

\subsection{Profunctors}  
A \profunctor{} from a category $𝔸$ to a category $𝔹$ is a functor $P ፡ 𝔸^{op} × 𝔹 → \Set$ \cite{benabou00}.
\Profunctors{} describe families of processes indexed functorially by the objects of two different categories.
The canonical example of a profunctor is the one that returns the set of morphisms between two objects of the same category, $𝔸(•;•) ፡ 𝔸^{op} × 𝔸 \to \Set$. Profunctors, however, do not need to be restricted to a single category: this makes them useful to study the relation between processes of different categories.

Categorically, \profunctors{} can be seen as a categorification of the concept of \emph{relations}, functions $A × B → 2$.
Under this analogy, existential quantifiers correspond to \emph{coends}. This section will first introduce \profunctors{} (\Cref{def:profunctor}), then a naturality relation for them (\Cref{def:dinaturality}) and, finally, their composition using coends (\Cref{subsec:coend-calculus,sec:pointedcoendcalculus}). We connect them explicitly to process theories in \Cref{subsec:promonads}.

\begin{definition}
  \label{def:profunctor}
  \defining{linkProfunctor}{} \defining{linkprofunctor}{}
  A \emph{profunctor} $(P,≺,≻)$ between two categories, $𝔸$ and $𝔹$, is a family of sets, $P(A,B)$, indexed by objects $A ∈ 𝔸_{obj}$ and $B ∈ 𝔹_{obj}$, and endowed with jointly functorial left and right actions of the morphisms of the two categories $𝔸$ and $𝔹$, respectively. 

  Explicitly, types of these actions are 
  \begin{align*}
    & (≻) ፡ 𝔸(A';A) × P(A';B) → P(A;B) \\
    & (≺) ፡  P(A;B) × 𝔹(B;B') → P(A;B')
  \end{align*}
  These two actions must be compatible, $(f ≻ p) ≺ g = f ≻ (p ≺ g)$, they must preserve identities, $id ≻ p = p$, and $p ≺ id = p$, and they must preserve composition $(p ≺ f) ≺ g = p ≺ (f ⨾ g)$ and $f ≻ (g ≻ p) = (f ⨾ g) ≻ p$.

	More succinctly, a \profunctor{} $P ፡ 𝔸 → 𝔹$ is the same as a functor $P ፡ 𝔸^{op} × 𝔹 → \Set$. When presented as a family of sets with a pair of actions, profunctors have been sometimes called \emph{bimodules}.
\end{definition}

\begin{definition}[Parallel composition]
  Two \profunctors{}  $P ፡ 𝔸_{1}^{op} × 𝔹_{1} \to \Set$ and
  $Q ፡ 𝔸_{2}^{op} × 𝔹_{2} \to \Set$ compose \defining{linkparallel}{\emph{in parallel}}
  into a \profunctor{} $P × Q ፡ 𝔸^{op}_{1} × 𝔸^{op}_{2} × 𝔹_{1} × 𝔹_{2} \to \Set$
  defined by
  \[(P × Q)(A,A';B,B') = P(A;B) × Q(A';B').\]
\end{definition}

\begin{remark}
  We will consider \profunctors{} between product categories explicitly:
  a \defining{linkprofunctor}{\emph{profunctor}} $P ፡ 𝔸₀ × \mydots × 𝔸ₙ → 𝔹₀ × \mydots × 𝔹ₘ$ is a functor 
  $$P \colon 𝔸₀^{op} \mydots × 𝔸ₙ^{op} × 𝔹₀ × \mydots × 𝔹ₘ \to \mathbf{Set}.$$
  For our purposes, a profunctor $P(A₀, \mydots, Aₙ; B₀, \mydots, Bₘ)$ is a family of processes indexed by contravariant inputs $A₀,  \mydots, Aₙ$ and covariant outputs $B₀, \mydots, Bₘ$. The profunctor is endowed with jointly functorial left $(≻_0, \mydots, ≻_n)$ and right $(≺_0, \mydots, ≺_m)$ actions of the morphisms of $𝔸₀, \mydots, 𝔸ₙ$ and $𝔹₀, \mydots, 𝔹ₘ$, respectively \cite{benabou00,loregian2021}. We will simply use $(≺/≻)$ without any subscript whenever the input/output is unique.
\end{remark}

  Composing \profunctors{} \emph{sequentially} is subtle: the same processes could arise as the composite of different pairs of processes, so we need to impose an equivalence relation. Imagine we try to connect two different processes:
  $$p \in P(A_0,\mydots,A_n;B_0,\dots,B_m),\mbox{ and } q \in Q(C_0,\mydots,C_k;D_0,\dots,D_h);$$
and we have some morphism $f \colon B_i \to C_j$ that translates the i-th output port of $p$ to the j-th input port of $q$. Let us write $(ᵢ|ⱼ)$ for this connection operation. Note that we could connect them in two different ways: we could
\begin{enumerate}
  \item change \emph{the output of the first process} $p ≺_i f$ before connecting both, thus obtaining $(p ≺ᵢ f)\, {}_i|_j\, q$;
  \item or change \emph{the input of the second process} $f ≻_j q$ before connecting both, thus obtaining $p\, {}_i|_j\, (f ≻_j q)$.
\end{enumerate}
These are different descriptions, made up of two different components. However, they essentially describe the same process: they are \emph{dinaturally equal}.
Indeed, \profunctors{} are canonically endowed with this notion of equivalence \cite{benabou00,loregian2021}, precisely equating these two descriptions.
\Profunctors{}, and their elements, are thus composed \emph{up to dinatural equivalence}.

\subsection{Dinaturality and Composition}
Dinaturality is a canonical notion of equivalence for \profunctors{}: it arises naturally from the construction of the bicategory of profunctors, but it also has a good interpretation in terms of procesess.

\begin{definition}[Dinatural equivalence] 
  \label{def:dinaturality}
  \defining{linkdinaturality}{}
  For any functor $P ፡ ℂ^{op} × ℂ → \Set$, consider the set
  \[S_{P} = \sum_{M ∈ ℂ} P(M;M).\]
  \emph{Dinatural equivalence}, $(\sim)$, on the set $S_{P}$ is the smallest equivalence relation satisfying $(r ≻ p) \sim (p ≺ r)$ for each $p ∈ P(M;N)$ and each $r ∈ ℂ(N;M)$.
\end{definition}

Coproducts quotiented by \dinaturalEquivalence{} construct a particular form of colimit called a \emph{coend}. Under the process interpretation of \profunctors{}, taking a coend means \emph{plugging an output to an input} of the same type.

\begin{definition}[Coend]
  Let $P \colon ℂ^{op} × ℂ \to \Set$ be a functor.
  Its \emph{coend} is the coproduct of $P(M,M)$ indexed by $M ∈ ℂ$, quotiented by \dinaturalEquivalence{}.
  \[
     ∫^{M ∈ ℂ} P(M;M) \coloneqq \left(\sum_{M ∈ ℂ} P(M;M) \bigg/ \sim \right).
  \]
  That is, the coend is the colimit of the diagram containing a \emph{cospan} $P(M;M) \gets P(M;N) \to P(N;N)$
  for each $f \colon N \to M$.
\end{definition}

\begin{definition}[Sequential composition]
  Two \profunctors{} $P ፡ 𝔸^{op} × 𝔹 → \Set$ and $Q ፡ 𝔹^{op} × ℂ → \Set$ compose \defining{linksequential}{sequentially} into
  a profunctor $P ⋄ Q ፡ 𝔸^{op} × ℂ → \Set$ defined by
  $$
  (P ⋄ Q)(A;C) = ∫^{B ∈ 𝔹} P(A;B) × Q(B; C).
  $$ 
  The \defining{linkidprof}{hom-profunctor} $\hom ፡ 𝔸^{op} × 𝔸 → \Set$ that returns the set of morphisms between two objects is the unit for sequential composition. Sequential composition is associative up to isomorphism.
\end{definition}

\subsection{Coend Calculus}
\label{subsec:coend-calculus}
\emph{Coend calculus} is the name given to the algebraic manipulations of coends that prove isomorphisms or construct natural transformations between profunctors using the behaviour of \emph{coends}.
MacLane \cite{macLane71:workingMathematician} and Loregian \cite{loregian2021} give presentations of coend calculus.

\begin{proposition}[Yoneda reduction]
  \label{prop:yonedareduction}
  \defining{linkcoyoneda}{} 
  Let $ℂ$ be any category and let $F ፡ ℂ \to \Set$ be a functor; the following isomorphism holds for any given object $A ∈ ℂ_{obj}$.
  \[
    ∫^{X ∈ ℂ} ℂ(X;A) × FX ≅ FA.
  \]
  Following the analogy with classical analysis, the $\hom$ profunctor works as a Dirac's delta.
\end{proposition}

\begin{proposition}[Fubini rule]
  \defining{linkfubini}{} Coends commute between them; that is,
  there exists a natural isomorphism
  \label{prop:fubinirule}
  \[\begin{aligned}
    & ∫^{X_{1} ∈ ℂ} ∫^{X_{2} ∈ ℂ} P(X_{1},X_{2};X_{1},X_{2})
    ≅ & 
    & ∫^{X_{2} ∈ ℂ} ∫^{X_{1} ∈ ℂ}  P(X_{1},X_{2};X_{1},X_{2}).
  \end{aligned}\]
  In fact, they are both isomorphic to the coend over the product category,
  \[ ∫^{(X_{1}, X_{2}) ∈ ℂ × ℂ} P(X_{1},X_{2};X_{1},X_{2}). \]
  Following the analogy with classical analysis, coends follow the Fubini rule
  for integrals.
\end{proposition}

\subsection{The Point of Coend Calculus}
\label{sec:pointedcoendcalculus}

In the same way that regular logic links relations, a coend calculus expression is a list of \profunctors{} linked by some objects that are bound to a coend. Usually, the isomorphisms that we construct are never made explicit, and it is difficult for the reader to compute the precise map we constructed.

Fortunately, this has a straightforward solution. We propose to \emph{point} the coends: to write an profunctorial expression, $P$, together with the \emph{generic element} it computes, $\withPoint{p}{P}$. An expression of pointed coend calculus is a coend bounding some objects and a series of \emph{pointed profunctors}. For instance, we may write
\begin{align*} 
  & {\textstyle ∫^{M,N} \withPoint{f}{P(A ; M,N)} × \withPoint{g}{Q(M ; B)} × \withPoint{h}{ℂ(N; C)}},\mbox{\ \ instead of just \ \ } \\
  & {\textstyle ∫^{M,N}} P(A ; M,N) × Q(M ; B) × ℂ(N; C).
\end{align*}
Coends quotient expressions by dinaturality, meaning that any left action on the covariant occurrence of a bounded variable can be equivalently written as a right action on its contravariant occurrence. In terms of pointed profunctors, this means that
$$∫^N \withPoint{(f ≺ h)}{P(A;N)} × \withPoint{g}{Q(N;B)} = ∫^M \withPoint{f}{P(A;M)} × \withPoint{(h ≻ g)}{Q(M;B)}.$$
\begin{proposition}
  Let $ℂ$ be a category and let $F ፡ ℂ^{op} → \mathbf{Set}$ and $G ፡ ℂ → \mathbf{Set}$ be a presheaf and a copresheaf, respectively.
  The following are natural isomorphisms of pointed profunctors,
  \begin{align*}
    & ∫^X \withPoint{f}{ℂ(X;A)} × \withPoint{h}{F(X)}\ ≅\ \withPoint{(f ≻ h)}{F(A)}; \\
    & ∫^X \withPoint{f}{ℂ(A;X)} × \withPoint{h}{G(X)}\ ≅\ \withPoint{(h ≺ f)}{G(A)}.
  \end{align*} 
  We call these isomorphisms the ``pointed'' Yoneda reductions.
\end{proposition}

\begin{remark}
Using pointed coends, any derivation does also include the computation of the isomorphism it induces. As an example, compare the following with the usual coend derivation of a cartesian lens \cite{ClarkeRoman20:ProfunctorOptics},
\end{remark}

\begin{proposition}
  In a cartesian monoidal category, the pairs of morphisms $ℂ(A; M × X)$ and $ℂ(M × Y; B)$, quotiented by dinaturality, are in bijective correspondence with the pairs of morphisms $ℂ(A; M)$ and $ℂ(M × Y; B)$.
\end{proposition}
\begin{proof}
  A function is explicitly constructed by the following derivation.
  \begin{align*}
    & ∫^M \withPoint{f}{ℂ(A ; M × X)} × \withPoint{g}{ℂ(M × Y ; B)} \\ 
    & \quad ≅ \quad (\mbox{ \emph{by the adjunction} } Δ ⊣ ×) \\
    & ∫^M \withPoint{(f ⨾ π_1)}{ℂ(A ; M)} × \withPoint{(f ⨾ π_2)}{ℂ(A ; X)} × \withPoint{g}{ℂ(M × Y ; B)} \\
    & \quad ≅ \quad (\mbox{ \emph{by pointed Yoneda lemma} }) \\
    & \withPoint{(f ⨾ π_2)}{ℂ(A ; X)} × \withPoint{((f⨾ π₁) ⊗ id) ⨾ g}{ℂ(X × Y ; B)}.
  \end{align*}  
  The function mapping an equivalence class $[f,g]$ to $(f ⨾ π_2; (f⨾ π₁) ⊗ id)$ is a bijection because it has been constructed from composing bijections.

Indeed, in the first step, we have used that the adjunction $(Δ ⊣ ×)$ is given by postcomposition with projections and; in the second step, we use that the action on the last profunctor is defined as $h ≻ g = (h ⊗ id) ⨾ g$. The bijection has been explicitly constructed as sending the pair $(f;g)$ to $(f ⨾ π_2; ((f⨾ π₁) ⊗ id) ⨾ g)$.
\end{proof}

\subsection{Promonads}
\label{subsec:promonads}
Promonads are to profunctors what monads are to functors: to quip, a promonad is just a monoid in the category of endoprofunctors.
It may be then surprising to see that so little attention has been devoted to them, relative to their functorial counterparts.
The main source of examples and focus of attention has been the semantics of programming languages \cite{hughes00,paterson01:arrows,jacobs09}.
Strong monads are commonly used to give categorical semantics of effectful programs \cite{moggi91}, and the so-called \emph{arrows} (or \emph{strong promonads}) strictly generalize them: they coincide with our previous definition of \effectfulCategory{} \cite{heunen06:arrows}.

Part of the reason behind the relative unimportance given to promonads elsewhere may stem from precisely from that fact: promonads over a category can be shown in an elementary way to be equivalent to identity-on-objects functors from that category~\cite{loregian2021}. The explicit proof is, however, difficult to find in the literature, and so we include it here (\Cref{th:promonadidonobjs}).

Under this interpretation, promonads are new morphisms for an old category. We can reinterpret the old morphisms into the new ones in a functorial way. The paradigmatic example is again that of Kleisli or cokleisli categories of strong monads and comonads.
This structure is richer than it may sound, and we will explore it further during the rest of this text.

\begin{definition}\label{definition:promonad}
  \defining{linkpromonad}{}
  A \emph{promonad} $(P,\starp,\unitp{})$ over a category $ℂ$ is a \profunctor{} $P \colon ℂ^{op} × ℂ → \Set$ together with natural transformations for inclusion $(\unitp{})_{X,Y} \colon ℂ(X;Y) \to P(X;Y)$ and multiplication $(\defining{linkpromonadmultiplication}{\starp})_{X,Y} \colon P(X;Y) \times P(Y;Z) \to P(X;Z)$, and such that
  \begin{enumerate}[label=\roman*.]
    \item the right action is premultiplication, $\unitp{f} \starp p = f ≻ p$;
    \item the left action is postmultiplication, $p \starp \unitp{f} = p ≺ f$;
    \item multiplication is dinatural, $p \starp (f ≻ q) = (p ≺ f) \starp q$;
    \item and multiplication is associative, $(p_{1} \starp p_{2}) \starp p_{3} = p_{1} \starp (p_{2} \starp p_{3})$.
  \end{enumerate}
  Equivalently, promonads are promonoids in the double category of categories, where the
dinatural multiplication represents a transformation from the composition of the \profunctor{} $P$ with itself.
\end{definition}

\begin{lemma}[Kleisli category of a promonad]
  \label{lemma:kleisli}
  Every promonad $(P,\starp,\unitp{})$ induces a category with the same objects as its base category, but with hom-sets given by $P(•,•)$, composition given by $(\starp)$ and identities given by $(\unitp{\id})$.
  This is called its \emph{Kleisli category}, $\kleisli{(P)}$. Moreover, there exists an identity-on-objects functor $ℂ → \kleisli{(P)}$, defined on morphisms by the unit of the promonad. 
\end{lemma}

The converse is also true: every category $ℂ$ with an identity-on-objects functor from some base category $𝕍$ arises as the Kleisli category of a promonad.

\begin{theorem}
  \label{th:promonadidonobjs}
  Promonads over a category $ℂ$ correspond to identity-on-objects functors from the category $ℂ$.
  Given any identity-on-objects functor $i ፡ ℂ → 𝔻$ there exists a unique promonad over $ℂ$ having $𝔻$ as its Kleisli category: the promonad given by the profunctor $\hom_{𝔻}(i(\bullet),i(\bullet))$.
\end{theorem}

\subsection{Bibliography}
Coends, the Yoneda lemma, and their calculus, were introduced in MacLane's monograph \cite{macLane71:workingMathematician}. A more modern presentation of coend calculus and its applications is in the work of Loregian \cite{loregian2021}. This author has also written on the importance of pointed profunctors for open diagrams \cite{openDiagrams} and collages \cite{braithwaite23:collages}.

\newpage

\clearpage{}%
\clearpage{}%
\section{Multicategories}
\label{sec-multicategories}

\subsection{Multicategories}

Multicategories will provide an algebra for composing multiple pieces into one.
A \multicategory{} is like a category where every morphism has a list of inputs instead of a single one. 
A multicategory, $𝕄$, contains a set of objects, $𝕄_{obj}$, as a category does; but instead of a set of morphisms, $𝕄(X;Y)$, for every pair of objects $X,Y ∈ 𝕄_{obj}$, it will have a set of \emph{multimorphisms}, 
$$𝕄(X₁,\dots,Xₙ;Y),\mbox{ for each list of objects } X₁,\dots,Xₙ, Y ∈ 𝕄_{obj}.$$
As in sequent logic, it is easier to denote lists of objects by metavariables. For instance, we will use $Γ = X₁,\dots,Xₙ$ and write $𝕄(Γ;Y)$ for the set of multimorphisms $𝕄(X₁,\dots,Xₙ;Y)$.

\begin{definition}
  \label{def:multicategory}\defining{linkmulticategory}{}
  A \emph{multicategory}, $𝕄$, is a collection of objects, $𝕄_{obj}$, together with a collection of multimorphisms, $𝕄(Γ;Y)$, for each list of objects $Γ = X₀,\dots,Xₙ \in 𝕄_{obj}$ and each object $Y ∈ 𝕄_{obj}$.

  For each object $X$, there must be an identity multimorphism, $\mathrm{id}_X ∈ 𝕄(X;X)$. For each three lists of objects $Γ,Γ₁,Γ₂$ and each two objects $Y$ and $Z$, there must exist a composition operation (we omit superscripts when clear from the context),
  \[(⨾)_{i}^{Γ₁,Γ₂} \colon 
    𝕄(Γ ; Yᵢ) × 𝕄(Γ₁,Yᵢ,Γ₂ ; Z) → 𝕄(Γ₁,Γ,Γ₂;Z).\]
  
  Composition must be unital, meaning that $\mathrm{id}_X ⨾_X f = f$ and $f ⨾ \mathrm{id}_Y = f$ every time that the equation is fomally well-typed. Composition must be also associative, meaning that  $(h ⨾_X g) ⨾_Y f = h ⨾_X (g ⨾_Y f); \mbox{ and } 
  g ⨾_Y (h ⨾_X f) = h ⨾_X (g ⨾_Y f)$ must hold whenever they are formally well-typed, see \Cref{fig:assoc-multicategory}.
  \begin{figure}[h]
    \centering
    \includegraphics[scale=0.4]{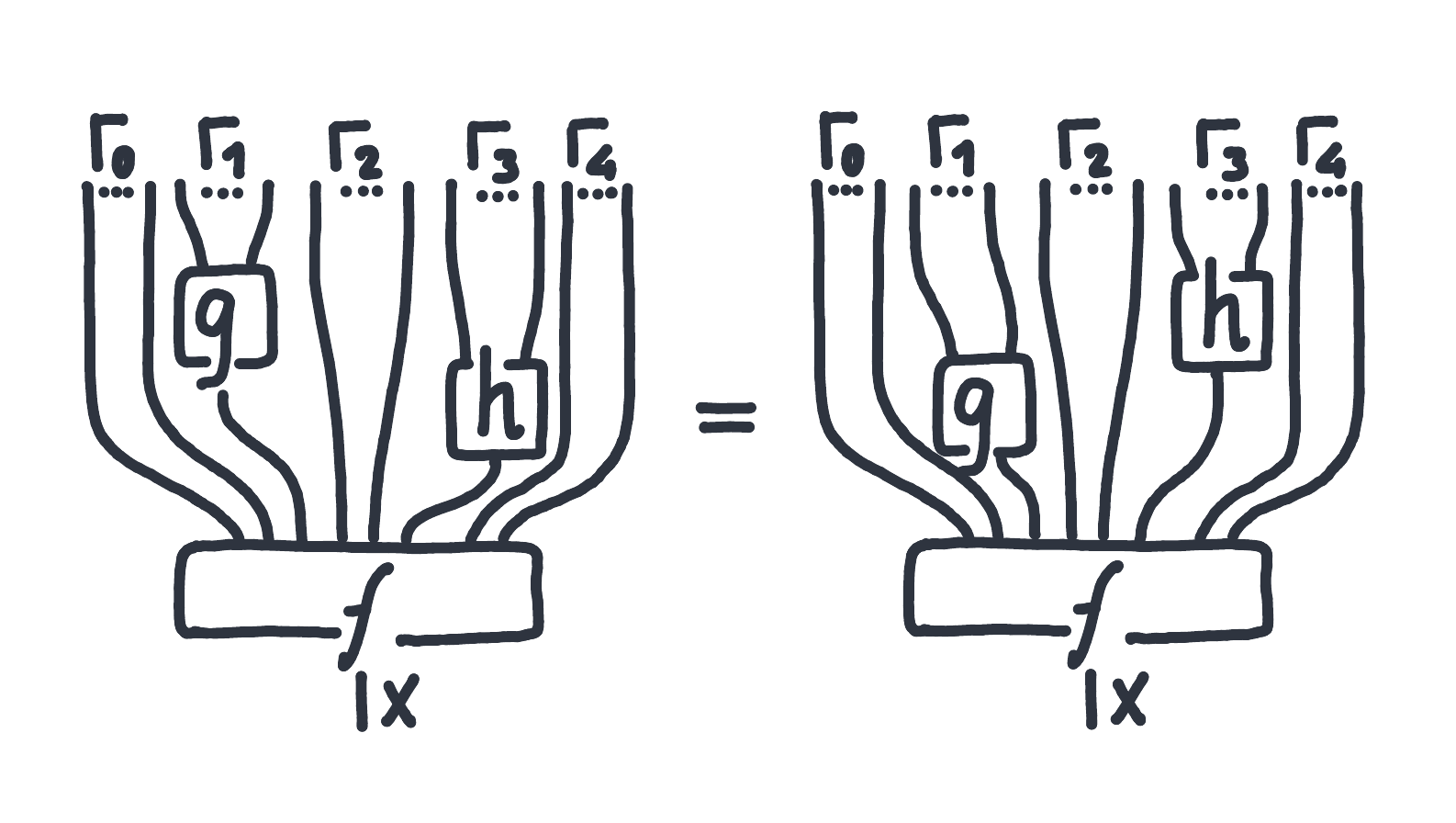}
    \caption{Associativity for a multicategory.}
    \label{fig:assoc-multicategory}
  \end{figure}
  
\end{definition}

\subsection{The Category of Multicategories}
In the same way categories are the first step towards the theory of functors and natural transformations, multicategories are the first step towards the theory of multifunctors and multinatural transformations.
In the same way the formal theory of categories is synthetised by the 2-category $\mathbf{Cat}$ of categories, functors and natural transformations; the study of multicategories is synthetised by the 2-category $\mathbf{MultiCat}$ of multicategories, multifunctors and multinatural transformations.

\begin{definition}
  \defining{linkmultifunctor}{}
  A \emph{multifunctor} between two \multicategories{}, $F ፡ 𝕄 → ℕ$, consists of an assignment on objects $F_{obj} ፡ 𝕄 → ℕ$ and an assignment on multimorphisms of any arity,
  $$Fₙ ፡ 𝕄(X₁,\dots,Xₙ; Y) → ℕ(F_{obj}X₁,\dots,F_{obj}Xₙ; F_{obj}Y),$$
  that preserves identities, $F₁(\id_X) = \id_{F_{obj}(X)}$, and composition of multimorphisms, $F_{n+m-1}(f ⨾_Y g) = Fₙ(f) ⨾_{F_{obj}(Y)} Fₘ(g)$.
\end{definition}

\begin{definition}
  A multinatural transformation $θ ፡ F → G$ between two \multifunctors{} $F, G ፡ 𝕄 → ℕ$ is given by a family of multimorphisms $θ_X ∈ ℕ(FX; GX)$ such that, for each multimorphism $f ∈ 𝕄(X₁,…,Xₙ;Y)$, the following naturality condition holds
  $$θ_{X_1} ⨾_1 … ⨾_{n-1} θ_{X_n} ⨾ G(f) = F(f) ⨾ θ_Y.$$
\end{definition}

\begin{proposition}
  \defining{linkMult}
  \Multicategories{} with \multifunctors{} between them form a category, $\Mult$.
\end{proposition}

\subsection{Application: Shufflings}
Let us exemplify \multicategories{} with an example that will become increasingly relevant in this thesis.
Shufflings are permutations that preserve the relative ordering of some blocks. We can always count shufflings combinatorially, but \multicategories{} provide the extra structure that allows us to track how different shufflings compose.

\begin{example}
  A \emph{shuffling} is a permutation of the elements of multiple blocks that preserves their relative ordering.
  The \multicategory{} of shufflings has objects the natural numbers and morphisms the shufflings, $σ ∈ \mathbf{Shuf}(p₀,\dots,pₙ;q)$ that reorganize $p₀,\dots,pₙ$ elements into $q = p₀ + \dots + pₙ$ without altering their internal ordering.
  
  \begin{figure}[ht]
    \centering
    \includegraphics[scale=0.22]{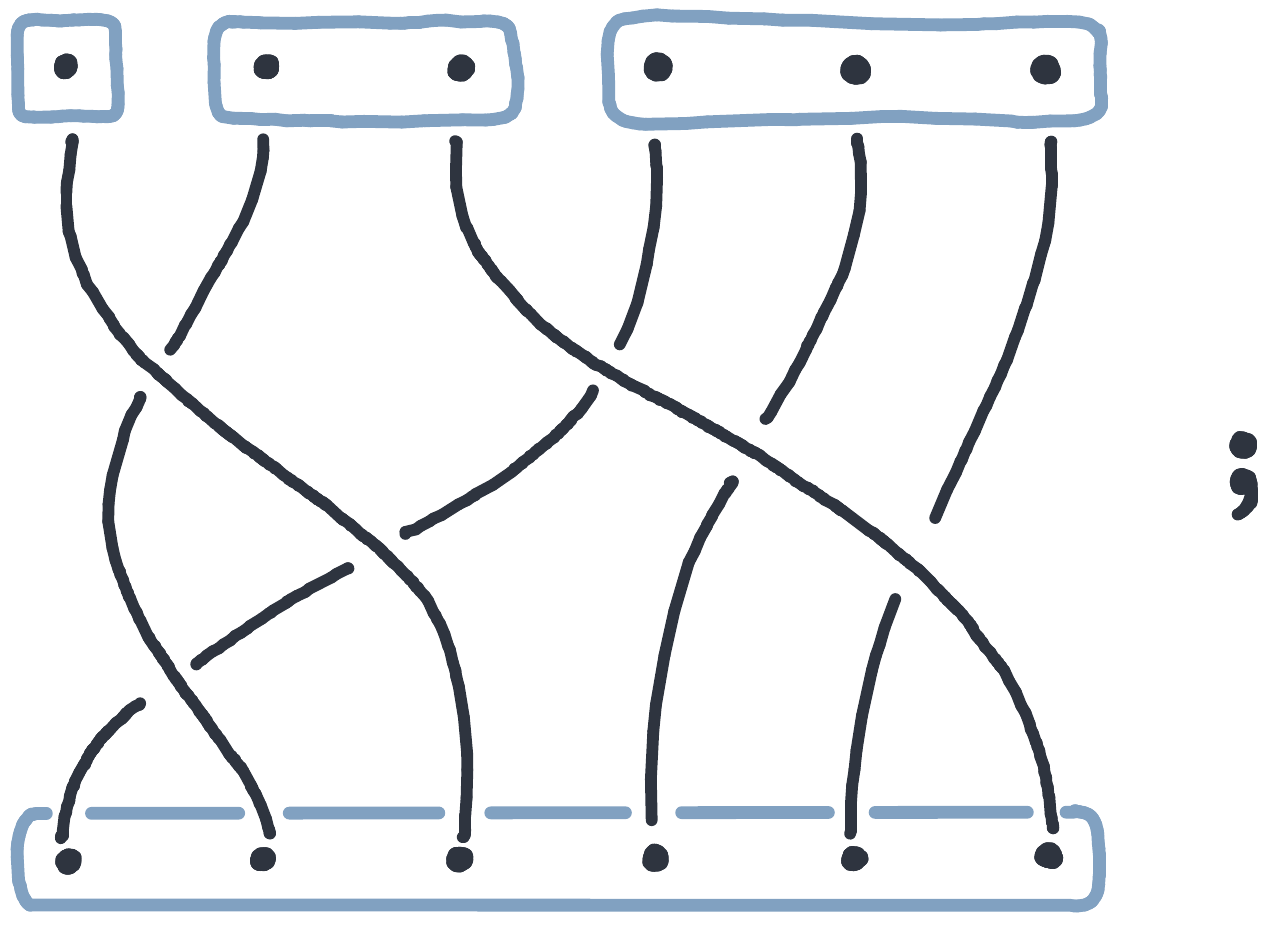}
    \caption{Example of a 1,2,3-shuffling.}
  \end{figure}

  More explicitly, the number of shufflings $\mathbf{Shuf}(p₀,\dots,pₙ; q)$ is given by a multinomial coefficient whenever $q = p₀ + \dots + pₙ$,
  $$\#\mathbf{Shuf}(p₀,\dots,pₙ; p₀ + \dots + pₙ) = \frac{(p₀ + \dots + pₙ)!}{p₀!\cdot \dots \cdot pₙ!},$$
  and it is zero in any other case.

  Shufflings exhibit a particular property that motivates our next section: \emph{malleability}. Any shuffling of $p₀$, $p₁$ and $p₂$ can be factored \emph{uniquely} in two different forms: we can first shuffle $p₀$ and $p₁$ and then shuffle the result, $p₀ + p₁$, with $p₂$; or we can first shuffle $p₁$ and $p₂$, and then shuffle $p₀$ with the result, $p₁ + p₂$. For instance, any $1,2,3$-shuffling splits uniquely into a $2,3$-shuffling followed by a $1,5$-shuffling, but also uniquely into a $1,2$-shuffling followed by a $3,3$-shuffling.

  \begin{figure}[ht]
    \centering
    \includegraphics[scale=0.22]{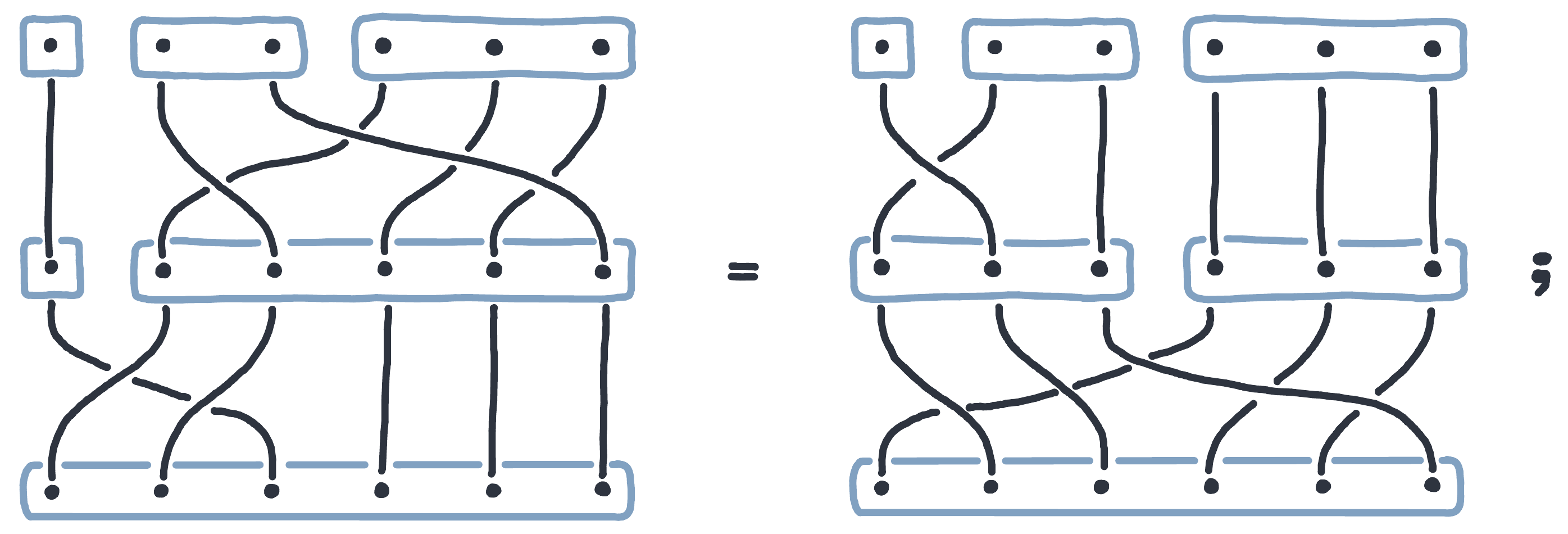}
    \caption{Two factorizations of the previous shuffling.}
  \end{figure}
  This is a global property: any shuffling can be uniquely factored into smallest shufflings, in whichever arrangement we pick. Any morphism of the \multicategory{} $\mathbf{Shuf}$ can be factored into any possible shape, uniquely. We say that the multicategories that satisfy this property are \emph{``malleable multicategories''}: shufflings form a \malleableMulticategory{}.
\end{example}

\section{Malleable Multicategories}
\label{sec-malleable-multicategories}

A \malleableMulticategory{} is a \multicategory{} where each morphism can be morphed uniquely into any possible shape. This means that there exist unique factorizations of each morphism into each one of the possible shapes. Formally, we will define \malleableMulticategories{} to have an invertible composition, up to the morphisms of some underlying category.

\begin{definition}
  \defining{linkUnderlyingCategoryOfAMulticategory}
  The unary morphisms of a \multicategory{} form a category \cite{shulman:catlog}.
  In other words, given a \multicategory{} $𝕄$, the underlying category, $𝕄^u$, has the same objects as the multicategory, $𝕄_{obj}^u = 𝕄_{obj}$, and morphisms defined from the unary multimorphisms of the \multicategory{}, $𝕄^u(X;Y) = 𝕄(X;Y)$. Composition and identities are exactly those of the \multicategory{}.
\end{definition}

\begin{remark}
  The multimorphisms of a \multicategory{} determine profunctors over the underlying category of the \multicategory.
  The underlying category acts on the multimorphisms by composition,
  \begin{align*}
    & (≻) \colon 𝕄^u(X;X') × 𝕄(Γ₁,X',Γ₂;Y) → 𝕄(Γ₁,X,Γ₂;Y), \\
    & (≺) \colon 𝕄(Γ;Y) × 𝕄^u(Y;Y') → 𝕄(Γ;Y').
  \end{align*}
  In any \multicategory{}, composition of multimorphisms is dinatural with respect to the underlying category.
  This follows from the associativity for \multicategories{},
  $$(f ≺ h) ⨾_{Xᵢ} g = (f ⨾_{X_i} h) ⨾_{X_i'} g = f ⨾_{X_i} (h ⨾_{X_i'} g) = f ⨾_{X_i} (h ≻ g).$$
  As a consequence, composition is well-defined under dinaturality. We define dinatural composition to be composition lifted to the equivalence classes of the dinaturality equivalence relation, which are written as a coend,
  $$\textstyle (⨾) \colon \left( ∫^{Y ∈ 𝕄} 𝕄(Γ;Y) × 𝕄(Γ₀,Y,Γ₁; Z)\right)  →  𝕄(Γ₀,Γ,Γ₁; Z).$$
\end{remark}

\begin{definition}
  \defining{linkcohmulticategory}{}\defining{linkcoherentmulticategory}{}\defining{linkMalleableMulticategory}{}
  A \emph{malleable multicategory} is a \multicategory{} where dinatural composition is invertible.
\end{definition}

\begin{proposition}
  \defining{linkmMult}
  \MalleableMulticategories{} with \multifunctors{} between them form a category, $\mMult$. This is a wide subcategory of the category of \multicategories{}.
\end{proposition}

\begin{remark}
  If a multicategory is malleable, we can reconstruct it up to isomorphism from its binary and nullary maps.
  When defining a \malleableMulticategory{}, it is usually easier to provide its binary, unary and nullary maps, and deduce from those the rest of the structure. The situation is similar in monoidal categories: we do not need to provide the n-ary tensor in order to define a monoidal category, we only provide the binary and unary tensors.

  This suggests that we will really work with a \emph{biased} version of \malleableMulticategories{}, one that privileges the binary and nullary tensors over the others. Biased malleable multicategories are better known as \emph{promonoidal categories}.
\end{remark}

\subsection{Promonoidal Categories}

In the same sense that multicategories provide an algebra for the composition of multiple pieces into one,
promonoidal categories provide an algebra for the \emph{coherent composition} of multiple pieces into one. A category $ℂ$ contains sets of \emph{morphisms}, $ℂ(X ; Y)$. In the same way, a promonoidal category $𝕍$ contains sets of \emph{joints}, $𝕍(X_0 ◁ X_1; Y)$, \emph{morphisms}, $𝕍(X; Y)$, and \emph{units}, $𝕍(N; X)$, where $N$ is the virtual tensor unit.
Joints, $𝕍(X₀ ◁ X₁; Y)$, represent a way of joining objects of type $X₀$ and $X₁$ into an objects of type $Y$. Morphisms, $𝕍(X; Y)$, as in any category, are transformations of $X$ into $Y$. Units, $𝕍(N;Y)$, are the atomic pieces of type $Y$.

These compositions must now be coherent. For instance, imagine we want to join $X₀$, $X₁$ and $X₂$ into $Y$. Joining $X₀$ and $X₁$ into something $(•)$,  and then joining that something $(•)$ and $X₂$ into $Y$, \emph{should be doable in essentially the same ways} as joining $X₁$ and $X₂$ into something $(•)$, and then joining $X₀$ and that something $(•)$ into $Y$. Formally, we are saying that,
$$
∫^{U} 𝕍(X₀ ⊲ X₁ ; U) × 𝕍(U ⊲ X₂; Y)  ≅  ∫^{V} 𝕍(X₁ ⊲ X₂; V) × 𝕍(X₀ ⊲ V; Y),
$$
and, in fact, we usually just write $𝕍(X₀ ⊲ X₁ ⊲ X₂; Y)$ for the set of such decompositions, even when it is only defined up to isomorphism.

\begin{definition}
  \defining{linkpromonoidal}{}\label{def:promonoidal}
  \PromonoidalCategories{} are the 2-monoids of the monoidal bicategory of \profunctors{}, which is equivalent to the following definition. A \emph{promonoidal category} is a category $𝕍(• ; •)$ endowed with two \profunctors{}
  $$
    𝕍(• ⊲ • ; •) ፡ 𝕍 × 𝕍 → 𝕍, \mbox{ and } 𝕍(𝖭; •) ፡ 1 → 𝕍.
  $$
  Equivalently, these are functors
  $$
    𝕍(• ⊲ • ; •) ፡ 𝕍^{op} × 𝕍 × 𝕍 → \mathbf{Set}, \mbox{ and } 𝕍(𝖭; •) ፡ 𝕍^{op} → \mathbf{Set}.
  $$
  Moreover, \promonoidalCategories{} must be endowed with the following natural isomorphisms,
  \begin{align*}
    𝕍(X₀ ⊲ X₁; •) ⋄ 𝕍(• ⊲ X₂ ; Y)    &≅  𝕍(X₁ ⊲ X₂; •) ⋄ 𝕍(X₀ ⊲ • ; Y);  \\
    𝕍(𝖭; •) ⋄ 𝕍(• ⊲ X; Y)   &≅  𝕍(X; Y); \\
    𝕍(𝖭; •) ⋄ 𝕍(X ⊲ •; Y)   &≅  𝕍(X; Y);
  \end{align*}
  called $α, λ, ρ$, respectively, and asked to satisfy the pentagon and triangle coherence equations, 
  $α ⨾ α = (α ⋄ \id) ⨾ α ⨾ (\id ⋄ α)$, and $(ρ ⋄ \id) = α ⨾ (λ ⋄ \id)$.
\end{definition}

\begin{definition}[Promonoidal functor]
  \label{def:promonoidalfunctor}\defining{linkPromonoidalFunctor}{}
  Let $𝕍$ and $𝕎$ be two \promonoidalCategories{}. A \emph{promonoidal functor} $F﹕ 𝕍 → 𝕎$ is a functor between the two categories together with natural transformations
  \begin{align*}
      & F_{⊲}﹕ 𝕍(X₀ ⊲ X₁; Y) → 𝕎(FX₀ ⊲ FX₁; FY),\ \mbox{ and } 
      & F_{𝖭} ﹕ 𝕍(𝖭; X) → 𝕎(𝖭; Y),
  \end{align*}
  that satisfy $λ ⨾ F_{map} = (F_{◁} × F_{N}) ⨾ λ$, $ρ ⨾ F_{map} = (F_{◁} × F_{N}) ⨾ ρ$, and $α ⨾ (F_{◁} × F_{◁}) ⨾ i = (F_{◁} × F_{◁}) ⨾ i ⨾ α$. 
\end{definition}

\begin{proposition}
  \defining{linkProm}{}
  \PromonoidalCategories{} with \promonoidalFunctors{} between them form a category, $\Prom$.
\end{proposition}

\subsection{Promonoidal Categories are Malleable Multicategories} 
In this section, we show that the category of \promonoidalCategories{} is equivalent to that of \coherentMulticategories{}. In this sense, the study of \malleableMulticategories{} is the study of \promonoidalCategories{}.

\begin{definition}[Underlying malleable multicategory]
  \label{def:underlyingMalleableMulticategory}
  Let $𝕍$ be a \promonoidalCategory{}. There is a \coherentMulticategory{}, $𝕍^m$, that has the same objects but multimorphisms defined by the elements of the \promonoidalCategory{}. By induction, we define
  \begin{align*}
    & \textstyle{𝕍^m(X₀,X₁,Γ; Y) = ∫^V 𝕍(X₀ ⊲ X₁; V) × 𝕍^m(V,Γ;Y)}, \\
    & 𝕍^m(X;Y) = 𝕍(X;Y), \\
    & 𝕍^m(;Y) = 𝕍(𝖭; Y).
  \end{align*}
  In other words, the multimorphisms are elements of the left-biased tree reductions of the promonoidal category, seen as a 2-monoid.
  Dinatural composition is then defined to be the unique map relating two tree expressions in a 2-monoid, which exists uniquely by coherence,
  $$(\mathrm{coh}) \colon \left(\textstyle{∫^{X ∈ 𝕍}} 𝕍(Γ;X) × 𝕍^m(Γ₀,X,Γ₁; Y) \right) → 𝕍^m(Γ₀,Γ,Γ₁;X).$$
  Coherence maps are isomorphisms, and so dinatural composition is invertible, making the multicategory coherent. 
  By coherence for pseudomonoids, composition must satisfy associativity and unitality.
\end{definition}

\begin{proposition}
  \label{prop:equivalencePromonoidalMalleable}
  The category of \promonoidalCategories{} and the category of \malleableMulticategories{} are equivalent with the functor  $(•)^m \colon \Prom → \mMult$ induced by the construction of the underlying \malleableMulticategory{} of a \promonoidalCategory{}.
  See a polycategorical analogue at \Cref{prop:equivalenceProstarMalleable}.
\end{proposition}
\begin{proof}
  First, let us show that a \promonoidalFunctor{}, $F ፡ 𝕍 → 𝕎$, induces a multifunctor, $F^m ፡ 𝕍^m → 𝕎^m$ between the underlying multicategories.
  On objects, we define it to be the same, $F^m_{obj} = F_{obj}$.
  On multimorphisms, we define the binary, unary and nullary cases using the promonoidal transformations: 
  $$F^u_0 = F_N,\mbox{ with }F^u_1 = F_{map}\mbox{ and }F_2^u = F_{⊲}.$$
  Then, we extend this definition to the n-ary case by induction, using $F^u_n = F_{⊲} × F_{n-1}^u$.
  We now verify that this assignment is functorial: the only remarkable case is that of checking that the inductive case preserves composition.
  \begin{align*}
    (F ⨾ G)^u_n & \overset{(i)}{=} (F ⨾ G)_{⊲} × (F ⨾ G)^u_{n-1} 
    \overset{(ii)}{=} (F_{⊲} ⨾ G_{⊲}) × (F^u_{n-1} ⨾ G^u_{n-1}) \\
                & \overset{(iii)}{=} (F_{⊲} × F^u_{n-1}) × (G_{⊲} × G^u_{n-1}) 
                \overset{(iv)}{=} F^u_n ⨾ G^u_n.
  \end{align*}
  This equation holds because \emph{(i)} of the inductive definition, \emph{(ii)} the composition of promonoidal functors and the inductive hypothesis, \emph{(iii)} the interchange law for functions, \emph{(iv)} and the inductive definition.

  We will now show that this is a fully faithful functor.
  It is hopefully clear that it is faithful because $F^u = G^u$ directly implies $F_{obj} = G_{obj}$, $F_{⊲} = G_{⊲}$, $F_N = G_N$ and $F_{map} = G_{map}$. Let us show that it is also full. Let $G \colon 𝕍_u → 𝕎_u$ be a multifunctor between \malleableMulticategories{}. 
  We will construct a \promonoidalFunctor{} $G^\star$ such that $G^{\star u} = G$. 
  We start by defining that $G^\star_N = G_0$, that $G^\star_{map} = G_1$ and that $G^\star_{⊲} = G_2$. Now, we need to prove that $(G^{★ u})ₙ = Gₙ$. This is definitionally true for binary, unary and nullary multimorphisms; then, by induction,
  $$Gₙ \overset{(i)}{=}
  \decomp{} ⨾ (G₂ × Gₙ) \overset{(ii)}{=}
  \decomp{} ⨾ (G₂^{\star u} × Gₙ^{\star u}) \overset{(iii)}{=}
  G^{\star u}_n.$$
  Here, we use \emph{(i)} malleability, \emph{(ii)} the induction hypothesis, and \emph{(iii)} the definition of the underlying multifunctor. We have shown that we have a fully faithful functor.

  Finally, we will show that $U ፡ \Prom → \mMult$ is essentially surjective. Given any malleable multicategory $𝕄$, we can define $M^{\flat}$ to be the \promonoidalCategory{} with the same objects and only binary, unary and nullary morphisms. We now note that there is an isomorphism, $𝕄 \cong 𝕄^{\flat u}$, that is the identity on objects. It is defined to use the invertible dinatural composition that exists by malleability,
  \begin{align*}
     𝕄_{\flat}^{u}(X₁,\dots,Xₙ;Y) & = \textstyle{∫^U} 𝕄(X₀,X₁;U) × 𝕄(U,X₂,\dots,Xₙ;Y) \\
   & \cong 𝕄(X₁, \dots, Xₙ; Y).
  \end{align*}
  This makes the functor fully faithful and essentially surjective, defining an equivalence of categories.
\end{proof}

\subsection{Bibliography}
What makes monoidal tensors universal? Products and coproducts have a universal property, but that is the exception and not the rule. Hermida's work \cite{hermida:representable} explains that tensors are universal because they represent a relevant multimap structure, a \emph{multicategory}. For instance, the \monoidalCategory{} of vector spaces with their tensor product represents functions linear in each variable: a linear function $A ⊗ B → C$ is the same as a multilinear function $A,B → C$.
Indeed, Lambek \cite{lambek:deductive} first introduced \multicategories{} as the underlying structure that unified Gentzen's sequents and multilinear maps.

\newpage
\clearpage{}%
\clearpage{}%
\section{The Splice-Contour Adjunction}
\label{sec:splice-contour-adjunction}

\subsection{Contour of a multicategory}
This last section characterizes the cofree \malleableMulticategory{} of \emph{spliced arrows}, which governs how incomplete morphisms can be nested.
Any \multicategory{} freely generates another category, its \emph{contour} \cite{mellies22:parsing}. This can be interpreted as the category that tracks the processes of decomposition that the \multicategory{} describes. The construction is particularly pleasant from the geometric point of view: it takes its name from the fact that it can be constructed by following the contour of the shape of the decomposition (\Cref{fig:contour-multimorphism}).

\begin{figure}[ht]
  \centering
  \includegraphics[scale=0.45]{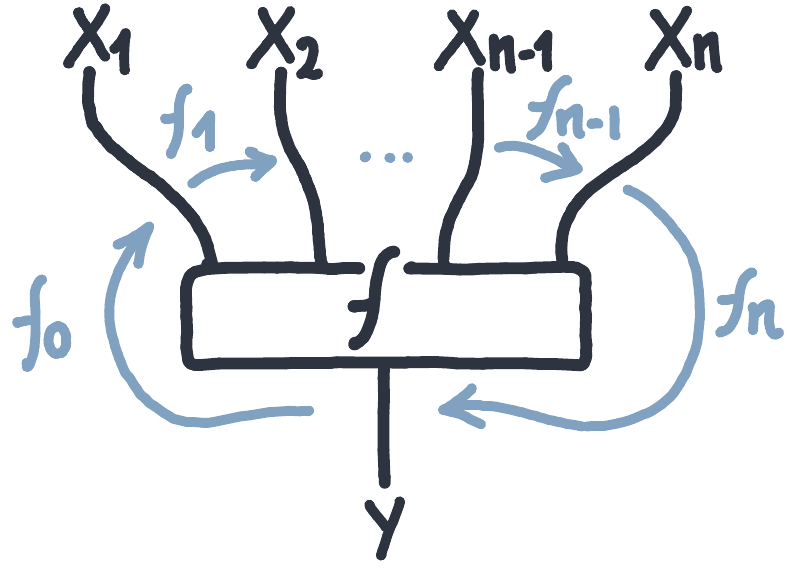}
  \caption{Contour of a multimorphism.}
  \label{fig:contour-multimorphism}
\end{figure}

\begin{definition}
  \label{defn:contour}
  \defining{linkContour}{}
  Let $𝕄$ be a \multicategory{}.
  Its contour, $\Contour(𝕄)$, is the category presented by 
  two polarized objects, $X^{∘}$ and $X^{∙}$, for each object $X ∈ 𝕄_{obj}$; 
  \begin{enumerate}
    \item for each multimorphism $f \in 𝕄(X₁,\dots,Xₙ; Y),$ the following generators,  
    \begin{align*} 
      f₀ ፡ Y^{∘} → X₁^{∘};
      f₁ ፡ X₁^{∙} → X₂^{∘};
      \quad … ; \quad
      f_{n-1} ፡ X_{n-1}^{∙} → Xₙ^{∘};
      f_{n} ፡ Xₙ^{∙} → Y^{∙};
    \end{align*}
    with only an $f₀ \colon Y^{∘} → Y^{∙}$ for the case $n = 0$;
    \item requiring contour to preserve identities, $(\id_{X})₀ = \id_{X^{∘}}$ and $(\id_X)₁ = \id_{X^{∙}}$;
    \item and requiring contour to preserve compositions, meaning that for each $f ∈ 𝕄(X₁,\dots,Xₙ; Yᵢ)$ and each $g ∈ 𝕄(Y₁,\dots,Yₘ; Z)$, the contour of their composition is defined by the following five cases
    $$(f ⨾_{X_i} g)_j = \left\{ 
      \begin{array}{ll}
        gⱼ & \mbox{ when } j < i, \\
        gᵢ ⨾ f₀ & \mbox{ when } j = i, \\
        f_{j-i} & \mbox{ when } i < j < i+n, \\
        f_n ⨾ g_{i+1} & \mbox{ when } j = i+n, \\
        g_{j-n+1} & \mbox{ when } i+n < j < n + m, \\
      \end{array}
    \right.$$
    with the special case $(f ⨾_{X_i} g)_i = gᵢ ⨾ f₀ ⨾ g_{i+1}$ whenever $n = 0$.
  \end{enumerate}
\end{definition}

A recent article by Melliès and Zeilberger \cite{mellies22:parsing} develops the notion of a context-free grammar over a category as a multicategorical functor to the multicategory of \emph{spliced arrows}. The multicategory of spliced arrows is a universal construction over a category that produces a multicategory of ``contexts'' over the category.

\subsection{Spliced Arrows}

The multicatgeory of spliced arrows is formed by arrows containing \emph{blanks} or \emph{holes} that could be filled to constuct a morphism. This \multicategory{} gives an algebraic theory of context for the category.

\begin{definition}[Spliced arrows]
  Let $ℂ$ be a category. The \multicategory{} of \emph{spliced arrows} has objects pairs of objects in $ℂ$, and the multimorphisms are given by sequences of arrows in $ℂ$ separated by $n$ gaps
  $$\Splice{ℂ} \left( 
    \biobj{X₁}{Y₁} , \dots , \biobj{Xₙ}{Yₙ} ;
    \biobj{X}{Y}
  \right) = ℂ(X;X₁) × \prod_{k=1}^{n-1} ℂ(Y_k, X_{k+1}) × ℂ(Yₙ;Y);$$
  which we write as $f₀ ⨾ □ ⨾ \dots ⨾ □ ⨾ fₙ$. That is, the sequence goes from $X$ to $Y$, with holes typed by $\{Xᵢ → Yᵢ\}_{0<i \leq n}$.
  Composition is defined by substitution
  \begin{align*}
    & (f₀ ⨾ □ ⨾ \dots ⨾ □ ⨾ fₙ) ≻_{i} (g₀ ⨾ □ ⨾ \dots ⨾ □ ⨾ gₘ) = \\
    & \qquad (g₀ ⨾ □ ⨾ \dots ⨾ gᵢ ⨾ f₀ ⨾ □ ⨾ \dots ⨾ □ ⨾ fₙ ⨾ g_{i+1} ⨾ \dots ⨾ □ ⨾ gₘ),
  \end{align*}
  and the identity morphism in $\left(\biobj{X}{Y} \right)$ is $id_X ⨾ □ ⨾ id_Y$.
\end{definition}

\begin{proposition}
  The \multicategory{} of spliced arrows is a \coherentMulticategory{}.
\end{proposition}
\begin{proof}
  We will show that dinatural composition is invertible by exhibiting an inverse to the composition operation, up to dinaturality. Consider a spliced arrow $h₀ ⨾ □ \cdots □ ⨾ h_{n+m-1}$ with $(n+m+1)$ holes; we can decompose $n$ of its holes at position $i$ with the following operation.
  \begin{align*}
    & \decomp_i^n(h₀ ⨾ □ \cdots □ ⨾ h_{n+m-1})  \\
    & \qquad = (h₀ ⨾ □ \cdots  hᵢ ⨾ □ ⨾ h_{i+n} \cdots  □ ⨾ h_{n+m-1})\ |\ (\id ⨾ □ ⨾ h_{i+1} \cdots h_{i+n-1} ⨾ □ ⨾ \id) \\
    & \qquad = (h₀ ⨾ □ \cdots  \id ⨾ □ ⨾ \id \cdots  □ ⨾ h_{n+m-1})\ |\ (hᵢ ⨾ □ ⨾ h_{i+1} \cdots h_{i+n-1} ⨾ □ ⨾ h_{i+n}).
  \end{align*}
  This operation is an inverse to dinatural composition. It follows by construction that $(≻)_i ⨾ \decomp_i = \id$, and we now check that $\decomp_i ⨾ (≻)_i$ is an identity up to dinaturality. Let $f₀ ⨾ □  \cdots □ ⨾ fₙ$ and $g₀ ⨾ □  \cdots □ ⨾ gₘ$ be two spliced arrows,
  \begin{align*}
    & \decomp_i^n((f₀ ⨾ □  \dots □ ⨾ fₙ) ≻_i (g₀ ⨾ □  \dots □ ⨾ gₘ)) \\
    & \ =\ \decomp_i^n(g₀ ⨾ □ \dots ⨾ gᵢ ⨾ f₀ ⨾ □ ⨾ \dots □ ⨾ fₙ ⨾ g_{i+1} ⨾ □ \cdots □ ⨾ gₘ)) \\
    & \ =\ (g₀ ⨾ \dots ⨾ gᵢ ⨾ f₀ ⨾ □ ⨾ fₙ ⨾ g_{i+1} ⨾ \dots ⨾ gₘ)\ |\ (\id ⨾ □ ⨾ f₁ ⨾ \dots ⨾ f_{n-1} ⨾ □ ⨾ \id) \\
    & \ =\ (g₀ ⨾ \dots ⨾ gᵢ ⨾ □ ⨾ g_{i+1} ⨾ \dots ⨾ gₘ)\ |\ (f₀ ⨾ □ ⨾ f₁ ⨾ \dots ⨾ f_{n-1} ⨾ □ ⨾ fₙ).
  \end{align*}
  We have shown that dinatural composition is invertible.
\end{proof}

\subsection{Splice-Contour Adjunction}
A first explicit account of splice-contour adjunction is due to Melliès and Zeilberger \cite{mellies22:parsing}. In later  joint work with Earnshaw and Hefford \cite{produoidal23}, we showed that this adjunction produced not only \multicategories{} but \coherentMulticategories{}.

\begin{theorem}
  \label{th:splice-right-contour}
  Splice is right adjoint to contour.
\end{theorem}
\begin{proof}
  Let $𝕄$ be a \multicategory{}. We will show that any multifunctor to a spliced arrow multicategory, $F ፡ 𝕄 → \Splice{(ℂ)}$, factors through a canonical multifunctor $T ፡ 𝕄 → \Splice(\Contour(𝕄))$ that is followed by a unique functor $F^\sharp ፡ \Contour(𝕄) → ℂ$.
  First, we construct  $T ፡ 𝕄 → \Splice(\Contour(𝕄))$, the multifunctor that sends any object $X$ to the pair of polarized objects $\left( \biobj{X^{∘}}{X^{•}} \right)$; and that sends any multimap $f ∈ 𝕄(X₀,\dots,Xₙ;Y)$ to the spliced arrow 
  $$f₀ ⨾ □ ⨾ \dots ⨾ □ ⨾ fₙ ∈ \Splice(ℂ)\left(\biobj{X_0^{∘}}{X_0^{∙}}, \dots, \biobj{X_n^{∘}}{X_n^{∙}} ; \biobj{Y^{∘}_{\phantom{0}}}{Y^{∙}_{\phantom{0}}}\right).$$
  We check now that $T$ is indeed a multifunctor: by construction, it sends $T(f ⨾_{X_i} g) = f ≻_i g$ and it sends $T(\id_X) = (\id_{X^{∘}} ⨾ □ ⨾ \id_{X^{∙}})$.
  
  We will now show that there exists a unique functor $F^\sharp ፡ \Contour{(𝕄)} → ℂ$ factoring the multifunctor $F$, such that $F = T ⨾ \Splice{(F^\sharp)}$. The contour is a category presented by some generators and equations: to define a functor from it, it suffices to define it on the generators and show that it preserves the equations of the presentation. We do so next.
  Consider the objects, for each $X ∈ 𝕄_{obj}$, assume $F(X) = \left(\biobj{A}{B}\right)$. We must have that 
    $$\Splice{(F^\sharp)}(T(X)) = \Splice{(F^\sharp)}\left(\biobj{X^{∘}}{X^{∙}}\right) = \left(\biobj{F^\sharp X^{∘}}{F^\sharp X^{∙}}\right) = \left(\biobj{A}{B}\right),$$
    which forces $F^\sharp(X^{∘}) = A$ and $F^\sharp(X^{∙}) = B$. Consider now the morphisms, for each $f \in 𝕄(X₀,\dots,Xₙ;Y)$. We must have that
    \begin{align*}
      \Splice{(F^\sharp)}(T(f))  
       = \Splice{(F^\sharp)}(f₀ ⨾ \dots ⨾ fₙ) 
       = F^\sharp f₀ ⨾ □ ⨾ \dots ⨾ □ ⨾ F^\sharp fₙ,
    \end{align*}
    which forces $F^\sharp(f_i) = F(f)_i$. This uniquely determines the value of $F^\sharp$ in all of the morphisms of the contour.
  We must finally check that that $F^\sharp$ satisfies the equations. We first notice that, by functoriality of $F$, we have
  $$F^\sharp((f ⨾_{X_i} g)_j) = F(f ⨾_{X_i} g)_j = (F(f) ⨾_{FX_i} F(g))_j,$$
  and, using this and the previous $F^\sharp(f_i) = F(f)_i$, we simply check the five cases of contour composition,
  \begin{align*}
    F^\sharp((f ⨾_{X_i} g)_j) = 
    \left\{ \begin{array}{l}
      F(g)_j \\
      F(g)_i ⨾ F(f)_0 \\
      F(f)_{j-i} \\
      F(f)_n ⨾ F(g)_{i+1} \\
      F(g_{j-n+1})
    \end{array}
    \right\} = 
    \left\{ \begin{array}{l}
      F^\sharp(gⱼ) \\
      F^\sharp(gᵢ) ⨾ F^\sharp(f₀) \\
      F^\sharp(f)_{j-i} \\
      F^\sharp(f)ₙ ⨾ F^\sharp(g)_{i+1} \\
      F^\sharp(g)_{j-n+1}
    \end{array}
    \right\}.
  \end{align*}
  This proves the existence of $F^\sharp$, but it also proves that it is the unique possible functor such that $F = T ⨾ \Splice(F^\sharp)$.
\end{proof}

\subsection{Promonoidal Splice-Contour}

We have commented on how any \coherentMulticategory{} induces a \promonoidalCategory{}.
The \coherentMulticategory{} of spliced arrows induces a \promonoidalCategory{} of spliced arrows.
This \promonoidalCategory{} is precisely the one that arises from the adjunction between any category and its opposite category in the monoidal bicategory of profunctors.

\begin{remark}
  The \promonoidal{} splice could be seen as a particular case of the more general multicategorical splice. 
  However, we will see that it is usually better behaved: technically, it is the 2-monoid arising from the 2-duality of a category with its opposite category, $ℂ ⊣ ℂ^{op}$, in the monoidal bicategory of profunctors. 
  We will not use this particular fact too much, but it will inspire its generalization. In the next chapter, we repeat a \emph{monoidal} version of the splice-contour adjunction that, by default, uses only the malleable version.
\end{remark}

\begin{proposition}[Promonoidal spliced arrows] \label{prop:promSplicedArrows}
  \defining{linkSplice}{}
  Let $ℂ$ be a category. The \promonoidalCategory{} of \emph{spliced arrows}, $\Splice{ℂ}$, has as objects pairs of objects of $ℂ$. It uses the following profunctors to define morphisms, splits and units.
  \begin{enumerate}
    \item $\Splice(ℂ) \left(\biobj{X}{Y} ; \biobj{A}{B} \right) = ℂ(A;X) × ℂ(Y,B);$ 
    \item $\Splice(ℂ)( \biobj{X}{Y} ◁ \biobj{X'}{Y'} ; \biobj{A}{B}) = ℂ(A;X) × ℂ(Y;X') × ℂ(Y';B);$
    \item $\Splice(ℂ)(N ; \biobj{A}{B}) = ℂ(A;B)$.
  \end{enumerate}
\end{proposition}

In other words, morphisms are \emph{pairs of arrows} $f \colon A \to X$ and $g \colon Y \to B$. 
Splits are \emph{triples of arrows} $f \colon A \to X$, $g \colon Y \to X'$ and $h \colon Y' \to B$. Units are simply \emph{arrows} $f \colon A \to B$. We use the following notation for
\begin{enumerate} 
  \item morphisms, $(f ⨾ □ ⨾ g) ፡ \pbiobj{X}{Y} → \pbiobj{A}{B}$;
  \item joins, $(f ⨾ □ ⨾ g ⨾ □ ⨾ h) ፡ \pbiobj{X}{Y} ⊲ \pbiobj{X'}{Y'} → \pbiobj{A}{B}$; 
  \item and units, $f ፡ 𝖭 → \pbiobj{A}{B}$.
\end{enumerate}

\begin{definition}
  The \emph{contour} of a \promonoidalCategory{} $ℙ$ is the category $\Contour{(ℙ)}$ presented by two polarized objects, $X^∘$ and $X^•$, for each object $X ∈ ℙ_{obj}$; and generated by the arrows that arise from \emph{contouring} the decompositions of the \promonoidalCategory{}.

  \begin{figure}[ht]
    \centering
    \includegraphics[scale=0.35]{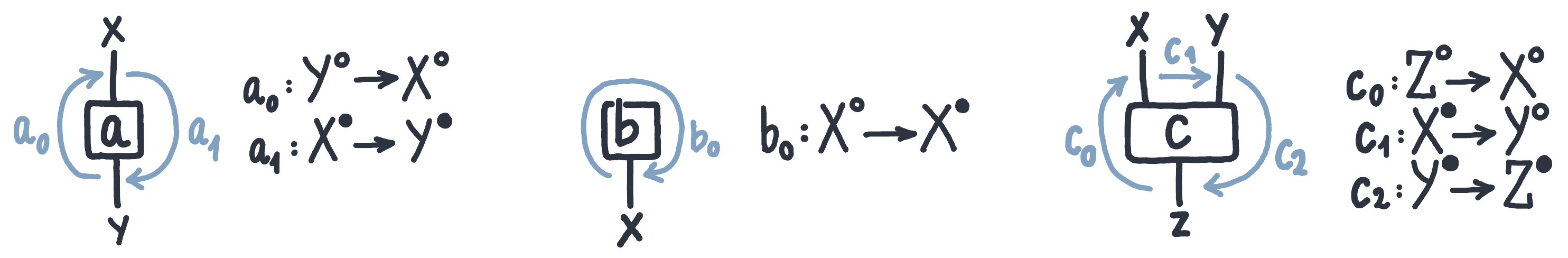}
    \caption{Contour of a promonoidal category.}
    \label{fig:contour-promonoidal}    
  \end{figure}
  
  Specifically, the contour is the category presented by the following generators, as depicted in \Cref{fig:contour-promonoidal}:
  \begin{enumerate}
    \item $a₀ ፡ Y^{∘} → X^{∘}$ and $a₁ ፡ X^{•} → Y^{•}$, for each morphism $a ∈ ℙ(X;Y)$;
    \item $b₀ ፡ X^{∘} → X^{•}$, for each unit $b ∈ ℙ(N;X)$;
    \item a triple of generators, $c₀ ፡ Z^{∘} → X^{∘}$, $c₁ ፡ X^{•} → Y^{∘}$ and $c₂ ፡ Y^{•} → Z^{•}$, for each split $c ∈ 
    ℙ(X ⊲ Y; Z)$.
  \end{enumerate}
  We impose several equations over these generators, all depicted in \Cref{fig:contour-promonoidal-equations}. The equations come from the decomposition of categories and they form the theory of contour.

  \begin{figure}[ht]
    \centering
    \includegraphics[scale=0.35]{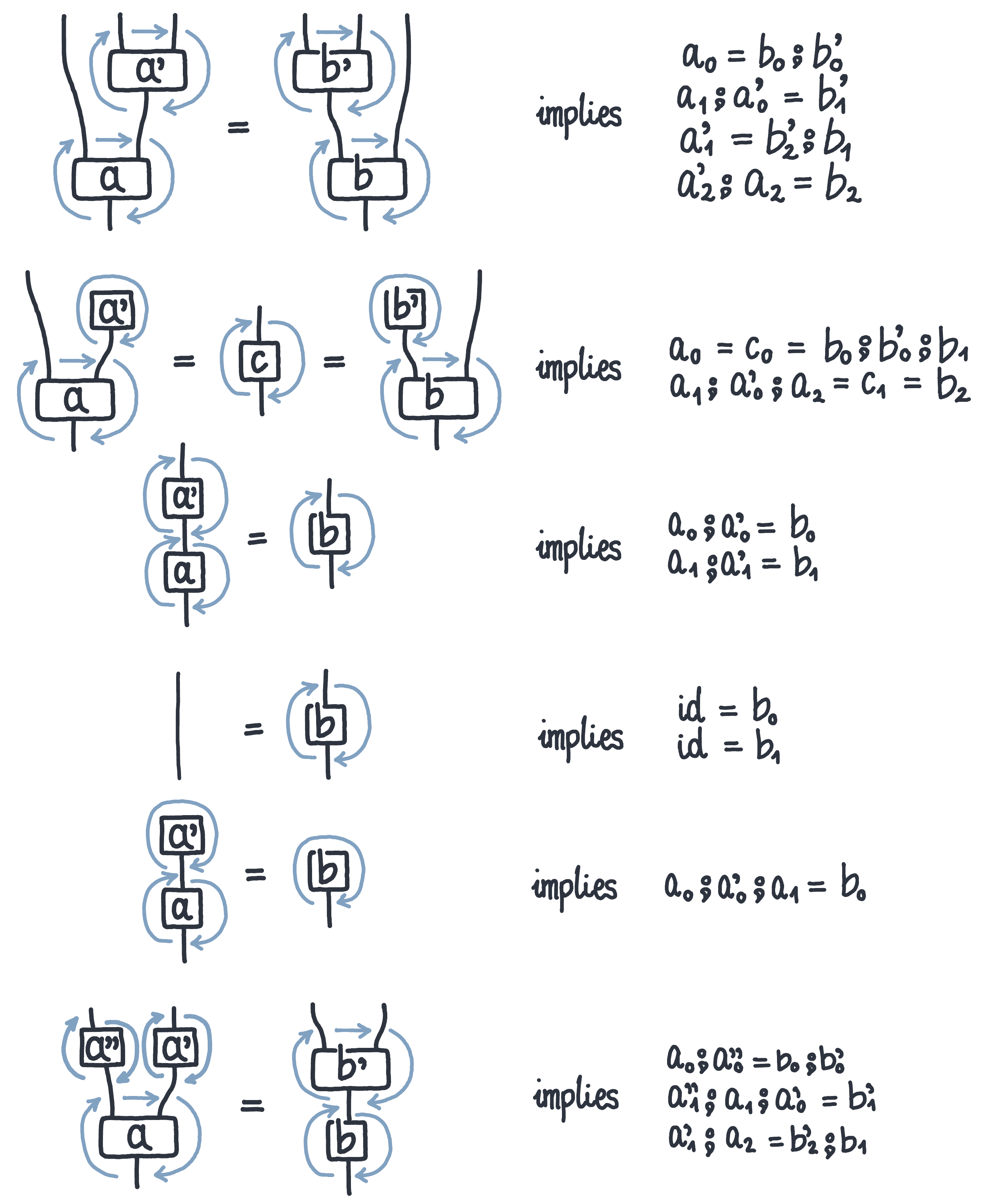}
    \caption{Theory of contour.}
    \label{fig:contour-promonoidal-equations}
  \end{figure}
  
\end{definition}

\newpage
\clearpage{}%

\chapter{Monoidal Context Theory}
\label{chapter:monoidal-context-theory}

\section*{Monoidal Context Theory}
This section develops a theory of context, or incomplete parts, for arbitrary \monoidalCategories{}.
In the same way that the theory of context for categories required \malleableMulticategories{} or \promonoidalCategories{}; the theory of context for \monoidalCategories{} will require \duoidalCategories{} and \produoidalCategories{}.

\DuoidalCategories{} combine a sequential tensor $(⊲)$ with a parallel tensor $(⊗)$; and it is well known that they can be used for process description \cite{garner16,shapiro22:duoidal}. However, lifting context theory to \monoidalCategories{} will come with a few technical surprises that we develop; the most important one is a normalization monad in the category of \produoidalCategories{}: this becomes a crucial step in creating a theory of \monoidalContext{} that allows incomplete morphisms to take any shape. The morphisms of the \produoidalCategory{} of contexts have been called \emph{lenses} and \emph{combs} in the literature, and we characterize them by a universal property.

We revisit the literature on \duoidalCategories{} and normalization in \Cref{section-duoidal-categories,sec-normal-duoidal-categories}. \ProduoidalCategories{} and their splice-contour adjunction are introduced in \Cref{sec:produoidalDecomposition}. We take a technical aside in \Cref{sec:normalization} to construct the normalization monad, and we immediately use it in \Cref{sec:monoidal-lenses}.

\clearpage{}%
\section{Duoidal categories}
\label{section-duoidal-categories}

\subsection{Duoidal Categories}
\DuoidalCategories{} result from the interaction of two \monoidalCategories{}.
By the Eckmann-Hilton argument, each time we have two monoids $(\ast,\circ)$ such that one is a monoid homomorphism over the other, $(a ∘ b) \ast (c ∘ d) = (a \ast c) ∘ (b \ast d)$, we know that both monoids coincide in a single commutative monoid.
  
However, an extra dimension helps us side-step the Eckmann-Hilton argument. If, instead of equalities or isomorphisms, we use directed morphisms, both monoids (which now may become 2-monoids) do not necessarily coincide, and the resulting structure is that of a \duoidalCategory{}.

\begin{definition}[Duoidal category]
  \defining{linkDuoidal}{}
  \defining{linkduoidal}{}
  \defining{linkDuoidalCategory}{}
    A \emph{duoidal category} \cite{aguiar10:monoidal} is a category $ℂ$ with two monoidal structures, $(ℂ,⊗, I,α,λ,ρ)$ and $(ℂ,◁, N, β, κ, ν)$ such that the latter distribute over the former.  In other words, it is endowed with a duoidal tensor, $(◁) \colon ℂ × ℂ → ℂ$, together with natural distributors
    \begin{align*}
    & ψ_2﹕ (X ◁ Z) ⊗ (Y ◁ W) → (X ⊗ Y) ◁ (Z ⊗ W), \\
    & ψ_0﹕ I → I ◁ I, \\
    & φ_2﹕ N ⊗ N → N, \quad\mbox{and}\\
    & φ_0 ﹕ I → N,
    \end{align*}
    satisfying the following coherence equations (\Cref{sec:coherencediagramsduoidal}, \Cref{cd:duoidal-coherence-assoc,cd:duoidal-coherence-unit,cd:duoidal-coherence-unit2,cd:duoidal-coherence-bimonoids,cd:duoidal-coherence-nandi}). 
    A duoidal category is \emph{strict} when both of its monoidal structures are.
\end{definition}

\begin{remark}
    In other words, the duoidal tensor and unit are lax monoidal functors for the first monoidal structure, which means that the laxators must satisfy the following equations.
    \begin{enumerate}
    \item $(ψ_2 ⊗ id) ⨾ ψ_2 ⨾ (α ◁ α) = α ⨾ (id ⊗ ψ_2) ⨾ ψ_2$, for the associator;
    \item $(ψ_0 ⊗ id) ⨾ ψ_2 ⨾ (λ ◁ λ) = λ$, for the left unitor; and
    \item $(id ⊗ ψ_0) ⨾ ψ_2 ⨾ (ρ ◁ ρ) = ρ$, for the right unitor;
    \item $α ⨾ (id ⊗ φ_2) ⨾ φ_2 = (φ_2 ⊗ id) ⨾ φ_2$, for the associator;
    \item $(φ_0 ⊗ id) ⨾ φ_2 = λ$, for the left unitor; and
    \item $(id ⊗ φ_0) ⨾ φ_2 = ρ$, for the right unitor.
    \end{enumerate}
\end{remark}

\subsection{Communication via Duoidals}
\label{sec:communication-duoidals}
The operations of a posetal duoidal structure can be interpreted as speaking about the communication of processes \cite{shapiro22:duoidal}. Let $(⊗,i)$ and $(⊲,n)$ form a duoidal structure on a poset. We read the elements of this poset as being processes and we interpret
\begin{enumerate}
  \item $x ⊗ y$ as ``$x$ and $y$ happen together, in parallel and independently'';
  \item $x ⊲ y$ as ``$y$ happens after $x$, and may depend on it'';
  \item $i$ as a process that ``interrupts communication'';
  \item $n$ as a process that ``does nothing'';
  \item $x → y$ as ``channels of $x$ are included in channels of $y$''.
\end{enumerate}
Under this interpretation, the rules of a \duoidalCategory{} say that
\begin{enumerate}
  \item $(x ⊲ y) ⊗ (z ⊲ w) → (x ⊗ z) ⊲ (y ⊗ w)$, adds intermediate communication;
  \item $i → i ⊲ i$, allows to interrupt an interrupted process;
  \item $n ⊗ n → n$, simplifies parallelism that does nothing; and
  \item $i → n$, allows new communications.
\end{enumerate}
We can already picture this interpretation in terms of \emph{communication diagrams}. Processes are boxes, and wires represent the information flow: a path from a process to another one means that the first can communicate to the second (\Cref{fig:duoidal-from-adjoint-monoid}).
\begin{figure}[!h]
  \centering
  \includegraphics[scale=0.3]{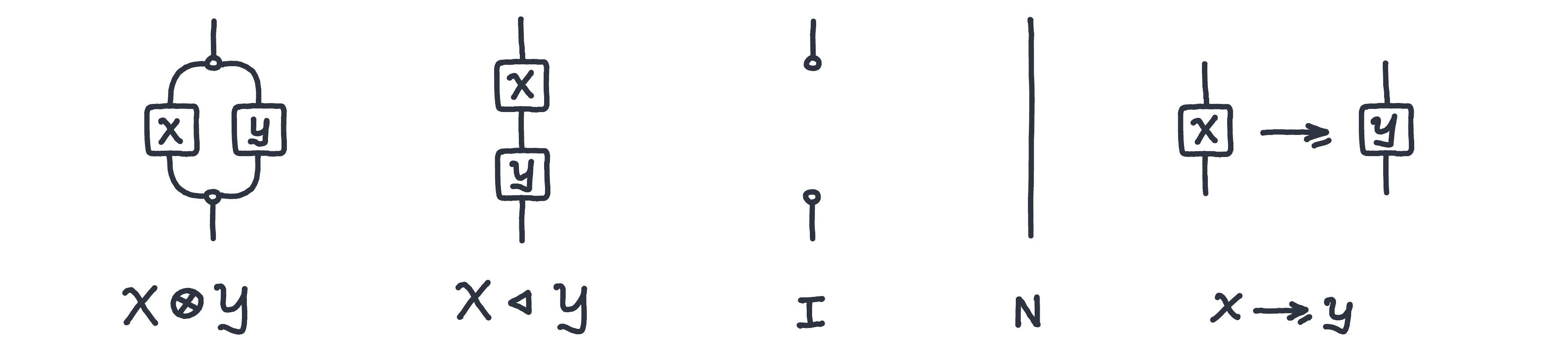}
  \caption{Communication diagrams.}
  \label{fig:duoidal-from-adjoint-monoid}
\end{figure}

The axioms of a \duoidalCategory{} impose the following transformations of communication diagrams (\Cref{fig:communication-diagram-axioms}). 

We can take these diagrams seriously: they are string diagrams of a monoidal bicategory where processes are endocells. The only structure that we ask of the single object generator is that of an \emph{adjoint monoid}.

\begin{figure}[!ht]
  \centering
  \includegraphics[scale=0.3]{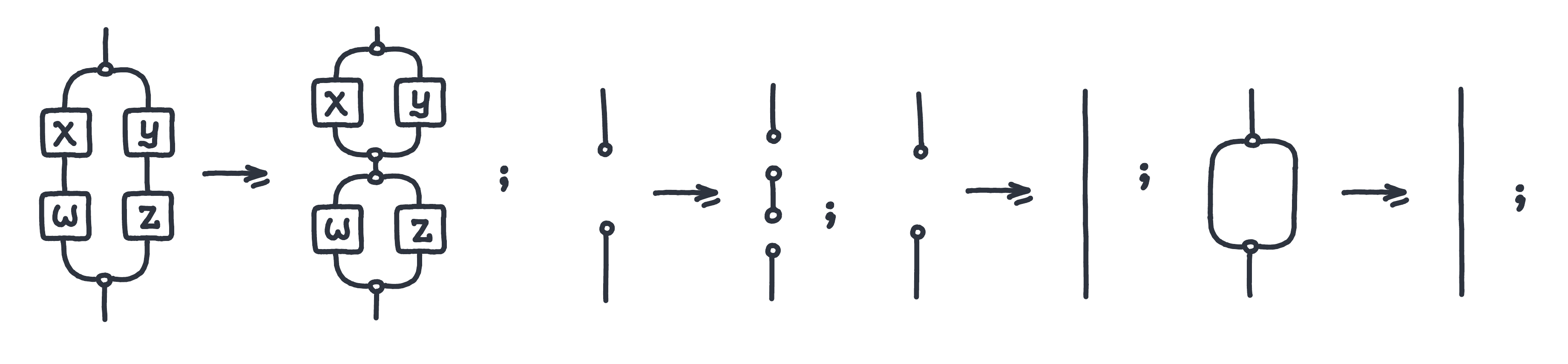}
  \caption{Communication diagram axioms.}
  \label{fig:communication-diagram-axioms}
\end{figure}

\subsection{Duoidals via adjoint monoids}

Let $(ℂ,⊗,I)$ be a monoidal category; let $(A,\iconwm{},\iconwu{})$ be a monoid and let $(B,\iconbcm{},\iconbcu{})$ be a comonoid. The set of morphisms $ℂ(A;B)$ forms a monoid with the operation of \emph{convolution}, $f \ast g = \iconbcm ⨾ (f ⊗ g) ⨾ \iconwm$ and the biunit, $e = \iconbcu{} ⨾ \iconwu{}$, as in \Cref{fig:convolution-monoid}.
\begin{figure}[!ht]
  \centering
  \includegraphics[scale=0.3]{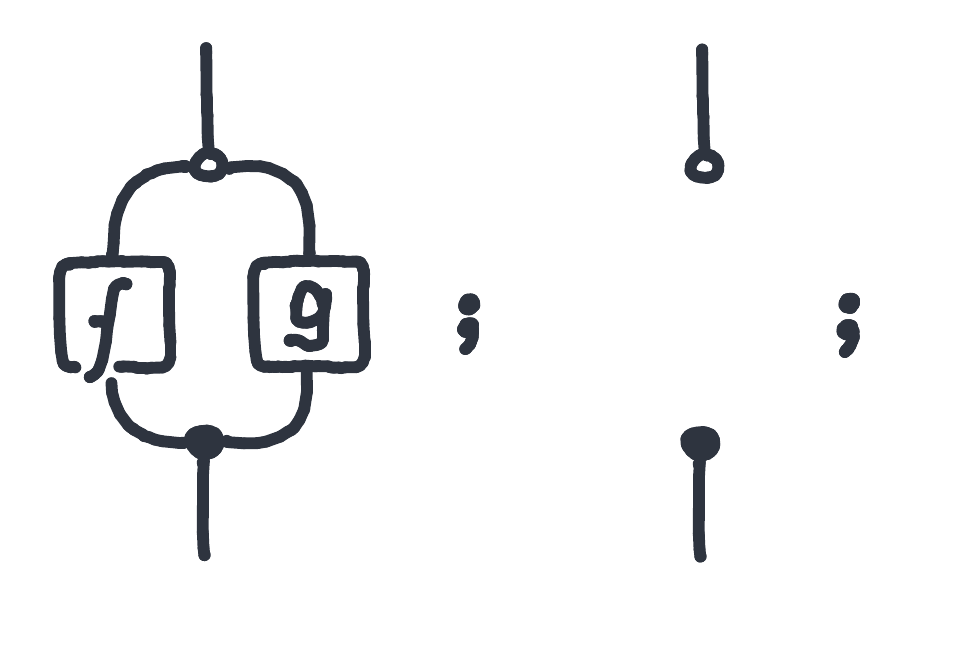}
  \caption{Convolution monoid.}
  \label{fig:convolution-monoid}
\end{figure}

Convolution and composition interact duoidally whenever the monoid and the comonoid are adjoint to each other. This is the notion of an \emph{adjoint monoid}.

\begin{definition}
  \defining{linkAdjointMonoid}{}
  An \emph{adjoint monoid}, in a monoidal bicategory $𝔹$, is an object endowed with both 2-monoid and 2-comonoid structure $(A,\iconwm,\iconwu,\iconwcu,\iconwcm)$, such that the multiplication is adjoint to the comultiplication $(\iconwm ⊣ \iconwcm)$ and the unit is adjoint to the counit $(\iconwu ⊣ \iconwcu)$.
\begin{figure}[!ht]
  \centering
  \includegraphics[scale=0.3]{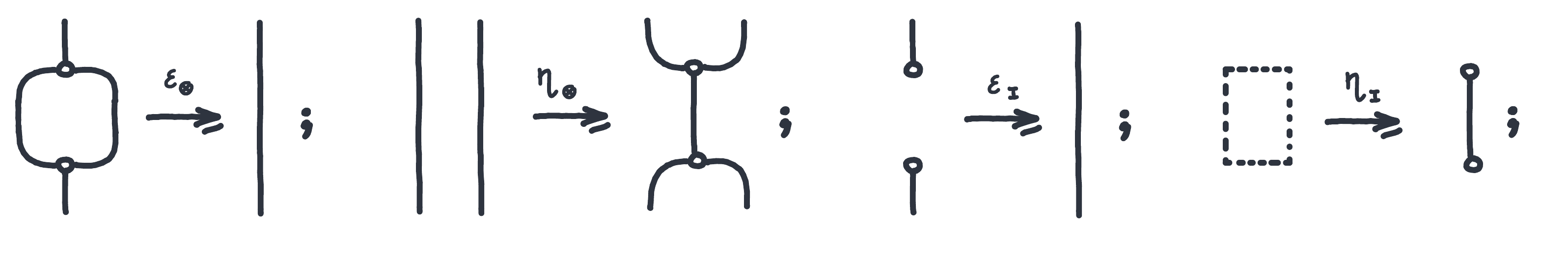}
  \caption{Adjoint monoid.}
  \label{fig:adjoint-monoid}
\end{figure}

This means that there exist 2-cells, $ε_{⊗} ፡ \iconwcm ⨾ \iconwm → \id$ and $η_{⊗} ፡ \id → \iconwcm ⨾ \iconwm$, witnessing the adjunction $(\iconwm ⊣ \iconwcm)$; and that there exist 2-cells, $ε_{I} ፡ \iconwcu ⨾ \iconwu → \id$ and $η_{I} ፡ \id → \iconwcu ⨾ \iconwu$ witnessing the adjunction $(\iconwu ⊣ \iconwcu)$, as in \Cref{fig:adjoint-monoid}.
\end{definition}

\begin{theorem}[Garner, López Franco {{\cite{garner16}}}]
  The endocells of an adjoint 2-monoid form a duoidal category with convolution and composition.
\end{theorem}
\begin{proof}
  Let $(A,\iconwm,\iconwu,\iconwcu,\iconwcm)$ be an \adjointMonoid{}, and let $X, Y ∈ 𝔹(A;A)$ be endocells. We define the sequential tensor as the composition, $X ⊲ Y = X \bcomp Y$, with the unit being the identity, $N = \mathbf{I}$. We define the parallel tensor as the convolution, $X ⊗ Y = \iconwcm ⨾ (X \boxtimes Y) ⨾ \iconwm$, with the unit being the pair of monoid units, $I = \iconwcu ⨾ \iconwu$. The duoidal interchangers are constructed out of the 2-cells of the adjoint monoid, taking \Cref{fig:communication-diagram-axioms} seriously as the string diagrams of a monoidal bicategory.
  \begin{enumerate}
    \item The first interchanger, $(X ⊲ W) ⊗ (Y ⊲ Z) → (X ⊗ Y) ⊲ (W ⊗ Z)$, is constructed from the tensor adjunction unit, $η_{⊗} ፡ \id → \iconwm ⨾ \iconwcm$;
    \item the second interchanger, $I → I ⊲ I$, is constructed from the unit of the unit adjunction, $η_{I} ፡ \id → \iconwu ⨾ \iconwcu$;
    \item the third interchanger, $I → N$, is constructed from the counit of the unit adjunction, $ε_{I} ፡ \iconwcu ⨾ \iconwu → \id$; and
    \item the fourth interchanger, $N ⊗ N → N$, is constructed from the tensor adjunction counit, $ε_{⊗} ፡ \iconwcm ⨾ \iconwm → \id$.
  \end{enumerate}
  Finally, we need to check that all of the structure diagrams commute. This is usually left to the reader \cite{garner16}; we can visualize the equations as surface diagrams.
\end{proof}

\begin{remark}[Day convolution]
  A particular case is the convolution of two parallel profunctors $P, Q ፡ ℂ → 𝔻$ between monoidal categories. A monoidal category determines an adjoint monoid in the monoidal bicategory of profunctors. Convolution is the operation that constructs a profunctor $P \ast Q ፡ ℂ → 𝔻$ defined by
  $$(P \ast Q)(A,B) = {\textstyle ∫}^{X,X',Y,Y'} ℂ(A, X ⊗ Y) × P(X,X') × Q(Y,Y') × 𝔻(X'⊗Y',B).$$
  Whenever we particularize to presheaves over a monoidal category $ℂ$, we recover \emph{Day convolution} of presheaves.
  $$(F \ast G)(A) = {\textstyle ∫}^{X,Y} ℂ(A, X ⊗ Y) × F(X) × G(Y).$$
  In particular, the endoprofunctors over any \monoidalCategory{} form a \duoidalCategory{}, this is the duoidal category that we study in \Cref{sec:produoidalDecomposition}.
\end{remark}

\subsection{Be Careful with Duoidal Coherence}
\MonoidalCategories{} possess a coherence theorem that determines that any two parallel morphisms constructed out of structure isomorphisms commute. In contrast, \duoidalCategories{} do not satisfy that same statement. This causes some confusion around coherence for \duoidalCategories{}. I bring an example of how this confusion may arise, hoping that it will help the interested reader and that it may further justify the importance of expository category theory.

We could be tempted to provide an alternative definition of \duoidalCategories{} that avoids asking for a bunch of commutative diagrams by simply asking that any formal such diagram commutes. In fact, this may possible for the \physicalDuoidalCategories{} studied by Spivak and Shapiro \cite{shapiro22:duoidal}, who comment that
\begin{quote}
  \emph{alternatively, duoidal categories can be defined by the two monoidal structures along with the generating structure maps [...] (4) natural in $a, b, c, d$ which satisfy equations guaranteeing that any two structure maps built from those in (4) between the same two expressions in $y_{⊗} , y_{⊲} , ⊗, ⊲$ are equal.}
\end{quote}

One of the first, most complete and comprehensive accounts of \duoidalCategories{} is the monograph by Aguiar and Mahajan \cite{aguiar10:monoidal}. It includes a passing comment that could suggest that this version can be proven correct. It says that
\begin{quote}
  \emph{``[...] if two morphisms $A → B$ are constructed out of the structure maps in $ℂ$ (including the structure constraints of the monoidal categories $(ℂ, \diamond, I)$ and $(ℂ, \star, J)$), then they coincide.''}
\end{quote}

However, intepreted literally and strictly, this turns out to not be true. Two parallel morphisms constructed out of the structure maps of a \duoidalCategory{} do not need to coincide.

\begin{proposition}\label{prop:duoidalCoherencefails1}
  There exist two different maps of type $I ⊲ I → I$ constructed out of the structure maps of a \duoidalCategory{}.
\end{proposition}
\begin{proof}
  We can consider two maps of type $I ⊲ I → I$, depending on which of the two parallel units we decide to convert to a sequential unit using the laxators. Explicitly, we are saying that $(I ⊲ φ_0) ⨾ ρ_{⊲}$ and $(φ_0 ⊲ I) ⨾ λ_{⊲}$ do not coincide; and the more intuitive string diagrams for bicategories for \adjointMonoids{} confirm this (\Cref{fig:countercoherence}). We will construct an explicit example of this phenomenon.
  \begin{figure}[!ht]
    \centering
    \includegraphics[scale=0.45]{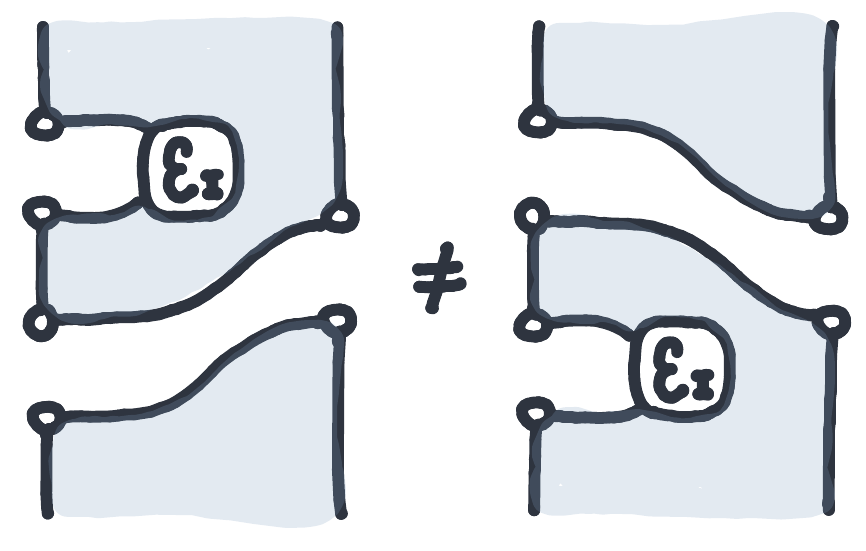}
    \caption{Bicategorical string diagrams for the two coherence maps.}
    \label{fig:countercoherence}
  \end{figure}

  Consider the \duoidalCategory{} of endoprofunctors over a \monoidalCategory{}. This is one of the first examples of \duoidalCategory{} described by Street \cite{street12:linking}; it is also described by Garner and López Franco \cite{garner16}, even when the axioms are not explicitly checked in print.
  In this category of endoprofunctors over $ℂ$,  parallel tensor is the profunctor 
  $I(X;Y) = ℂ(X;I) × ℂ(I;Y)$, and sequencing two of them gives
  \[(I ⊲ I)(X;Y) = \hom(X;I) × \hom(I;I) × \hom(I;Y).\]
  In this case, the two maps send the triple $(f,a,g)$ to $(f ⨾ a, g)$ and $(f, a ⨾ g)$, respectively. However, these two pairs do not need to be equal if $a ∈ \hom(I;I)$ is a non-identity morphism.
\end{proof}

\begin{example}[Graded spaces]
  We look for a more classical source of examples in the theory of \emph{graded spaces}.
  Let $(𝕍,⊗,I)$ be a \monoidalCategory{} with coproducts that are preserved by the tensor; let $(G,+,0)$ be a commutative monoid. We say that the functor category $[G,𝕍]$ is the category of \emph{$G$-graded $𝕍$-spaces}. This category has a rich structure; we highlight two of its tensor products: the \emph{pointwise} or \emph{Hadamard} tensor product
  \[(V ⊗ W)ₙ = Vₙ ⊗ Wₙ,\mbox{ for each } n ∈ G,\mbox{ with unit }𝕀ₙ = I;\]
  and the \emph{convolution} or \emph{Cauchy} tensor product
  \[(V • W)ₙ = \sum\nolimits_{k+m = n} Vₖ ⊗ Wₘ, \mbox{ with unit } \mathbf{1}ₙ = \mathbf{0}\mbox{ except for }\mathbf{1}₀ = I.\]
  These two tensors interact in a \duoidalCategory{} with a laxator as follows; see for instance the work of López Franco and Vasilakopoulou \cite{vasilakopoulou:20duoidal}.
  \[\sum_{k+m = n} Vₖ ⊗ Wₖ ⊗ Uₘ ⊗ Zₘ → \left( \sum_{k₁ + m₁ = n₁} V_{k₁} ⊗ U_{m₁} \right) ⊗ \left( \sum_{k₂ + m₂ = n₂} W_{k₂} ⊗ Z_{m₂} \right).\]
\end{example}

\begin{proposition}\label{prop:duoidalCoherencefails2}
  Dually, there exist two different maps of type $J → J ⊗ J$ constructed out of the structure maps of a \duoidalCategory{}.
\end{proposition}
\begin{proof}
  This follows from the previous \Cref{prop:duoidalCoherencefails1}, by considering the opposite \duoidalCategory{}. However, let us comment a second example \cite{aguiarEmail}. Consider the \duoidalCategory{} of graded spaces over a monoid $G$. The two maps, $𝕀 → 𝕀 ⊗ 1 → 𝕀 ⊗ 𝕀$ and $𝕀 → 1 ⊗ 𝕀 → 𝕀 ⊗ 𝕀$, correspond to inclusions of the vector space graded by $g ∈ G$ into the summand indexed by $(g,0)$ or $(0,g)$, respectively; these are different in general.
\end{proof}

In fact, the stronger statement of coherence does not seem to be used explicitly in any of these two texts, and the definition of \duoidalCategories{} as completely coherent strucutres is not usually found in the literature. Most authors, like Aguiar and Mahajan \cite{aguiar10:monoidal}, and Garner and López Franco \cite{garner16}, revert to the definition of \duoidalCategory{} as a 2-monoid in the monoidal bicategory of monoidal categories.

Aguiar and Mahajan \cite{aguiar10:monoidal} do actually point out that the expected coherence theorem should follow from the coherence theorem for lax monoidal functors. The confusion can arise if we do not realize that this coherence theorem does not actually prove that any two parallel maps coincide: in particular, coherence for lax monoidal functors does not prove that the two maps $F(I) ⊗ F(I) → F(I)$ coincide. In this case, however, the problem is better known -- it is mentioned by Malkiewich and Ponto \cite{malkiewich21:coherence}, who cite a short mention in the original proof by Lewis \cite{lewis06:coherence} and Kelly and Laplaza \cite{kelly80:coherence}.  

\subsection{Bibliography}
I thank Marcelo Aguiar and Swapneel Mahajan \cite{aguiarEmail} for their generosity, helping me confirm the problem and specially for providing the second counterexample; I thank Brandon Shapiro and David Spivak for helping me follow this same idea on their work. I thank Matt Earnshaw for sharing his knowledge of the literature on \duoidalCategories{}.

\newpage
\section{Normal Duoidal Categories}
\label{sec-normal-duoidal-categories}
\DuoidalCategories{} seem to contain too much structure: of course, we want to split things in two different ways, sequentially $(◁)$ and in parallel $(⊗)$; but that does not necessarily mean that we want to keep track of two different types of units, parallel $(I)$ and sequential $(N)$. The atomic components of our decomposition algebra could be the same, without having to care if they are \emph{atomic for sequential composition} or \emph{atomic for parallel composition}; when this is the case, we talk of \emph{normal duoidal categories}.

\begin{definition}
  \defining{linkNormalDuoidal}{}
  A \emph{normal duoidal category} is a \duoidalCategory{} in which the map $φ_0 ፡ I → N$ is an isomorphism.
\end{definition}

While \duoidalCategories{} are useful to track communication between processes; symmetric normal duoidal categories track \emph{dependencies} -- whether a process' input depends on the output of another -- structuring a dependency poset. This idea is explored by Garner and López Franco \cite{garner16} and Spivak and Shapiro \cite{shapiro22:duoidal}, and it has a formal counterpart in \Cref{th:posetsAreFreePhysical}.

Most \duoidalCategories{} we have seen so far -- and particularly those arising from \adjointMonoids{} -- have two different units.
There exists a well-known abstract procedure that, starting from some \duoidalCategory{}, constructs a new \duoidalCategory{} that is normal: both units are identified. This procedure is known as \emph{normalization}, and it can only be applied to \duoidalCategories{} with certain coequalizers preserved by the tensor.

\subsection{Normalization of duoidal categories}
Garner and López Franco construct the normalization of a well-behaved \duoidalCategory{}, using a new duoidal category of \emph{bimodules} \cite{garner16}.

\begin{remark}
Let $M$ be a bimonoid in the \duoidalCategory{} $(𝕍,⊗,I,◁,N)$, with maps $e ﹕ I → M$ and $m ﹕ M ⊗ M → M$; and with maps $u ﹕ M → N$ and $d ﹕ M → M ◁ M$. Consider now the category of $M^⊗$-bimodules. This category has a monoidal structure lifted from $(𝕍,◁,N)$:
\begin{enumerate}
    \item the unit, $N$, has a bimodule structure with 
      $$M⊗N⊗M \overset{u ⊗ \id ⊗ u}\longrightarrow  N⊗N⊗N \longrightarrow N;$$
    \item the sequencing of two $M^⊗$-bimodules is a $M^⊗$-bimodule with
    $$\begin{aligned}
        M&⊗(A◁B)⊗M \\
        &→ (M◁M)⊗(A◁B)⊗(M◁M) \\
        &→ (M⊗A⊗M)◁(M⊗B⊗M) → A◁B.
    \end{aligned}$$
\end{enumerate}
Moreover, whenever $𝕍$ admits reflexive coequalizers preserved by $(⊗)$, the category of $M^{⊗}$-bimodules is monoidal with the tensor of bimodules: the coequalizer
$$A ⊗ M ⊗ B \rightrightarrows A ⊗ B \twoheadrightarrow A ⊗_M B.$$
In this case $(\mathbf{Bimod}^{⊗}_M, ⊗_M, M, ◁, N)$ is a \duoidalCategory{}.
\end{remark}

\begin{theorem}[Normalization of a duoidal, {{\cite{garner16}}}]
  \label{thm:normalizationDuoidal}
    Let $(𝕍,⊗,I,◁,N)$ be a \duoidalCategory{} with reflexive coequalizers preserved by $(⊗)$. The category of $N$-bimodules is then a \normalDuoidalCategory{},
    $$\mathcal{N}(𝕍) = (\mathbf{Bimod}^{⊗}_N, ⊗_N, N, ◁, N).$$
    We call this category the \emph{normalization} of the \duoidalCategory{} $𝕍$.
\end{theorem}

\subsection{Physical duoidal categories}
The interaction of dependent and independent composition of \normalDuoidalCategories{} is a recurrent idea in physical models: categorical models of spacetime exhibit this structure \cite{hefford_spacetime,shapiro22:duoidal}; but it is also exhibited by parallel and sequentially composing programs \cite{sigal23:duoidally}; or more simply, by the category of partially ordered sets.

In most of these cases, the \normalDuoidalCategory{} has an extra property: the parallel tensor ($⊗$) is symmetric. This is what motivates the name \emph{physical duoidal category} for the $⊗$-symmetric \normalDuoidalCategories{}.

\begin{definition}
  \defining{linkPhysicalDuoidal}{}
  A \emph{physical duoidal category} is a \normalDuoidalCategory{} endowed with a \symmetricMonoidalCategory{} structure for its parallel tensor.
\end{definition}

  Posets are a canonical example of a \physicalDuoidalCategory{}: in fact, it is known that a subcategory of the category of posets and poset inclusions forms the free \physicalDuoidalCategory{} over a generator. In that precise sense, duoidal expressions are dependency tracking posets.

\begin{definition}
  The category of \emph{poset shapes}, $\mathbf{PosetSh}$, is the skeleton of the category of finite posets with bijective-on-objects monotone functions. Objects are isomorphism classes of finite posets, and morphisms are inclusions.
\end{definition}

\begin{proposition}
  \emph{Poset shapes} form a \physicalDuoidalCategory{}.
\end{proposition}
\begin{proof}
  The sequential tensor is constructed by sequentially joining the posets. Let $(P, ≤_{P})$ and $(Q, ≤_{Q})$ be two posets; their sequentiation, $P ⊲ Q$, is a poset that contains a copy of $P$, a copy of $Q$, and an edge $pᵢ ≤ qⱼ$ for each $pᵢ ∈ P$ and $qⱼ ∈ Q$; that is,
  $$P ⊲ Q = (P + Q, ≤_P + ≤_Q + \{ pᵢ ≤ qⱼ \mid pᵢ ∈ P, qⱼ ∈ Q\}).$$
  The parallel tensor, $P ⊗ Q$, is defined to be the disjoint union of posets,
  $P ⊗ Q = (P + Q, ≤_P + ≤_Q)$, which defines a symmetric monoidal structure. The empty poset is the unit for both sequential and parallel tensoring, making $\mathbf{PosetSh}$ a \physicalDuoidalCategory{}.
\end{proof}

The category of poset shapes is not posetal: there are, for instance, two possible inclusions of the discrete two-element poset into itself. This prompts us to label the nodes to indicate inclusions, as in \Cref{fig:posetinclusion}, but we work up to relabelling, or $α$-equivalence. 

\begin{figure}[!ht]
  \centering
  \includegraphics[scale=0.55]{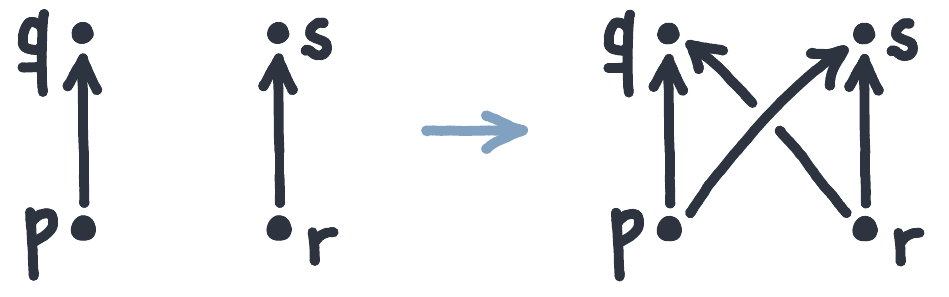}
  \caption{Poset inclusion.}
  \label{fig:posetinclusion}
\end{figure}

What makes this \physicalDuoidalCategory{} particularly relevant is that it contains the free \physicalDuoidalCategory{} over a generator. Every formal normal duoidal expression constructs a poset: we simply substitute each variable by the singleton poset and we interpret the expression in the \duoidalCategory{} of poset shapes. Every formal structure map between normal duoidal expressions corresponds to an inclusion; for instance, \Cref{fig:posetinclusion} documents the structure map $(p ⊲ q) ⊗ (r ⊲ s) → (p ⊗ r) ⊲ (q ⊗ s)$.

Formalizing this result needs a bit of care, though: while all formal physical duoidal expressions correspond to posets, not every poset shape corresponds to a physical duoidal expression. The poset shapes that arise from applying duoidal operations to the singleton poset are called \emph{expressible}, and we have a characterization result for them.

\begin{definition}
  Expressible poset shapes are those inductively constructed from
  \begin{enumerate}
    \item the empty poset, $𝖭$;
    \item the singleton poset, $\{A\}$;
    \item the union of posets, $P ⊗ Q = (P + Q, ≤_P + ≤_Q)$; and
    \item the sequencing of posets, $$P ⊲ Q = (P + Q, ≤_P + ≤_Q + \{ pᵢ ≤ qⱼ \mid pᵢ ∈ P, qⱼ ∈ Q\}).$$
  \end{enumerate}
  Expressible poset shapes form a full duoidal subcategory of the \physicalDuoidalCategory{} of poset shapes, $\mathbf{ExprSh}$.
\end{definition}

\begin{proposition}[Grabowski, {{\cite{grabowski81:partial}}}]
  Not every poset shape is expressible. In fact, a poset shape is not expressible if and only if it admits an inclusion of the \emph{Z-poset shape} defined by \Cref{fig:zposet}.
\end{proposition}

\begin{figure}[!ht]
  \centering
  \includegraphics[scale=0.55]{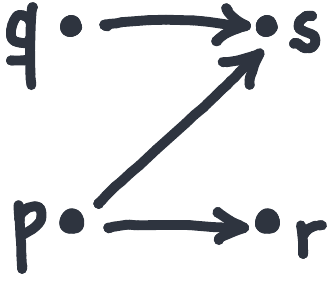}
  \caption{The Z poset shape.}
  \label{fig:zposet}
\end{figure}

\begin{theorem}[Shapiro and Spivak, {{\cite{shapiro22:duoidal}}}]
  \label{th:posetsAreFreePhysical}
  The \physicalDuoidalCategory{} $\mathbf{ExprSh}$ of expressible poset shapes is the free \physicalDuoidalCategory{} on a single object. There exists exactly one structure map between any two objects of the free \physicalDuoidalCategory{} for each inclusion of their associated expressible posets.
\end{theorem}

\begin{remark}[Coherence for normal duoidal categories]
  Coherence for \duoidalCategories{} needs some care: not any two morphisms between distinctly-typed expressions in the free \duoidalCategory{} are equal (\Cref{prop:duoidalCoherencefails1}). However, the previous theorem implies that the same statement is true for \normalDuoidalCategories{}.
\end{remark}
\begin{corollary}
  Any two morphisms between distinctly typed expressions in the free \duoidalCategory{} over a set of objects are equal.
\end{corollary}

\subsection{Physical Lax Tensor of a Physical Duoidal Category}
\label{sec:physicallaxtensor}
Let us recap our interpretation of \physicalDuoidalCategories{}: they track an underlying poset of dependencies. The sequential tensor, $X ⊲ Y$, says that $X$ occurs before $Y$, but $Y ⊲ X$ says that $Y$ occurs before $X$; consequently, it is not symmetric. The parallel tensor, $X ⊗ Y$, states that both $X$ and $Y$ occur independently. This final section of our introduction to \physicalDuoidalCategories{} shows what happens when we want to consider both $X$ and $Y$ but we do not know at all how they interact: the tensor that tracks this case is a derived operation, the \emph{physical tensor}, $X \boxtimes Y$.

The physical tensor simply says that both occur at some point: it does not impose independence, but it does not impose any particular dependency either. The physical tensor $X \boxtimes Y$ says that $X$ may occur before $Y$, or $Y$ before $X$, or both in parallel and in that case it does not matter how we regard the dependency. This is a tool that we will employ later to discuss a version of monoidal context that does not track dependency: \emph{wiring diagrams} (\cite{spivak13}, \Cref{conj:wiringdiagrams}).

\begin{remark}
  The binary physical tensor, $X \boxtimes Y$, is easy to define: it is the pushout of the two structure maps $X ⊗ Y → X ⊲ Y$ and $X ⊗ Y → Y ⊲ X$. However, unlike most tensors, its n-ary version cannot be deduced from its binary and nullary versions; the physical tensor is only a lax tensor.
\end{remark}

\begin{definition}[Leinster {{\cite{leinster04}}}]
  A \emph{lax monoidal category} is a category $ℂ$ endowed with a family of lax tensor n-fold tensor functors $(\boxtimes) ፡ ℂ^{n} → ℂ$ -- written as $X₁ \boxtimes … \boxtimes Xₙ$, with the 0-ary case $E$ -- and a family of associator natural transformations that unbias the application of the lax tensor,
  $$α ፡ \boxtimes_{i = 0}^{n} \left( \boxtimes_{j = 0}^{k_i} X^{i}_{j} \right) → X_{1}^{1} \boxtimes … \boxtimes X^{1}_{k_1} \boxtimes … \boxtimes X^{n}_1 \boxtimes … \boxtimes X^n_{k_n},$$
  such that all formally well-typed equations hold.
\end{definition}

\begin{definition}
  \label{def:physicaltensor}
  Let $(ℂ,⊗,I,⊲,N)$ be a \physicalDuoidalCategory{}. 
  The \emph{physical tensor}, $(\boxtimes)$, is an additional lax monoidal tensor, defined as the glueing of the sequential tensor $(⊲)$ along the parallel tensor $(⊗)$; that is, it is the pushout on the following family of structure maps, indexed by permutations
  $$l_{σ} ፡ X₁ ⊗ … ⊗ Xₙ → X_{σ1} ⊲ … ⊲ X_{σn},\mbox{ for } σ ∈ P(n).$$
\end{definition}

\begin{remark}
This only forms a lax tensor for a good reason. Consider the simpler case of three elements, $X \boxtimes Y \boxtimes Z$. This expression allows all of the possible six permutations to occur: \emph{(i)} $X ⊲ Y ⊲ Z$; \emph{(ii)} $X ⊲ Z ⊲ Y$; \emph{(iii)} $Y ⊲ X ⊲ Z$; \emph{(iv)} $Y ⊲ Z ⊲ X$; \emph{(v)} $Z ⊲ X ⊲ Y$; and \emph{(vi)} $Z ⊲ Y ⊲ X$. However, when we consider $(X \boxtimes Y) \boxtimes Z$, we are only allowing a certain subset of these cases to occur. Namely, only those where $Z$ does not happen between $X$ and $Y$: \emph{(i)} $X ⊲ Y ⊲ Z$; \emph{(ii)} $Y ⊲ X ⊲ Z$; \emph{(iii)} $Z ⊲ X ⊲ Y$; and \emph{(iv)} $Z ⊲ Y ⊲ X$.
This is why the physical tensor is only lax. We have two inclusions: 
$(X \boxtimes Y) \boxtimes Z → X \boxtimes Y \boxtimes Z$ and  $X \boxtimes (Y \boxtimes Z) → X \boxtimes Y \boxtimes Z$,
but these are not isomorphisms.
\end{remark}

\begin{proposition}
  The physical tensor defines a lax monoidal structure. Given any \physicalDuoidalCategory{} $(ℂ,⊗,⊲,N)$, the physical tensor defines a lax monoidal category $(ℂ,\boxtimes,N)$. 
\end{proposition}
\begin{proof}
  The definition of the laxator follows from the universal property of the pushout; the coherence equations hold by uniqueness of the maps constructed out of this universal property.
\end{proof}

\subsection{Bibliography}

The original monograph on \duoidalCategories{} is due to Aguiar and Mahajan \cite{aguiar10:monoidal} -- \duoidalCategories{} were originally known as \emph{``2-monoidal categories''}; Street first described multiple examples that we recall \cite{street12:linking}, and Garner and López Franco mention for the first time the connection to \adjointMonoids{} \cite{garner16}. The reason \duoidalCategories{} do not have a correspondence in lower dimensional algebra is the Eckmann-Hilton argument \cite[Theorem 1.12]{eckman61:structure}.

\PhysicalDuoidalCategories{} follow the definition and nomenclature of Spivak and Shapiro \cite{shapiro22:duoidal}; their work makes the case for interpreting them as expressing dependencies between processes and argues initiality of the category of expressible posets. It seems originally due to Grabowski \cite{grabowski81:partial} that expressible posets are precisely those not containing a Z, and Gischer recognized the lax interchange of \normalDuoidalCategories{} as subsumption of posets \cite{gischer88:equationalpomsets}. Even if the physical lax tensor does not seem to appear in the related literature, its definition and its consideration as a lax tensor follow the work of Leinster \cite{leinster04} and the abstraction of commutativity by Garner and López Franco \cite{garner16}. I thank Matt Earnshaw for multiple pointers to the literature.
\clearpage{}%
\clearpage{}%
\section{Produoidal Decomposition of Monoidal Categories}
\label{sec:produoidalDecomposition}

\subsection{Produoidal categories}

\defining{linkProduoidalComponents}
\ProduoidalCategories{}, first defined by Booker and Street \cite{bookerstreet13}, provide an algebraic structure for the interaction of sequential and parallel decomposition. A \produoidalCategory{} $𝕍$ not only contains \emph{morphisms}, $𝕍(X; Y)$, as in a category, but also \emph{sequential joints}, $𝕍(X_0 ◁ X_1; Y)$,  and \emph{sequential units}, $𝕍(N; X)$, provided by a \promonoidal{} structure; and \emph{parallel joints}, $𝕍(X_0 ⊗ X_1; Y)$ and \emph{parallel units}, $𝕍(I; X)$, provided by another \promonoidal{} structure.

These splits must be coherent. For instance, imagine we want to join  $X_0$, $X_1$ and $X_2$ (sequentially) into $Y$. Joining $X₀$ and $X₁$ into something $(•)$, and then joining that something with $X_2$ to produce $Y$ \emph{should be doable in essentially the same ways} as joining $X₁$ and $X₂$  into something $(•)$, and then joining that something with $X_0$ to produce $Y$. Formally, we are saying that
$$𝕍(X_0 ⊲ X₁ ; •) ⋄ 𝕍(• ⊲ X_2; Y) ≅ 𝕍(X₁ ⊲ X₂ ; •) ⋄ 𝕍(X₀ ⊲ • ; Y),
$$
and, in fact, we just write $𝕍(X_0◁ X_1 ◁ X_2; Y)$ for the set of such transformations. This is precisely what we ask for in a \promonoidal{} structure.

\begin{definition}[Produoidal category]
  \defining{linkproduoidal}{}\defining{linkProduoidalCategory}{}
  \label{def:produoidal}
  A \emph{produoidal category} is a category $𝕍$ endowed with two \promonoidal{} structures,
  $$\begin{gathered}
    𝕍(• ⊗ • ; •) ፡ 𝕍 × 𝕍 → 𝕍, \mbox{ and } 𝕍(I; •) ፡ 1 → 𝕍, \\
    𝕍(• ◁ • ; •) ፡ 𝕍 × 𝕍 → 𝕍, \mbox{ and } 𝕍(N; •) ፡ 1 → 𝕍,
  \end{gathered}$$
  such that one laxly distributes over the other.
  This is to say that it is endowed with the following natural \emph{lax interchangers}:
  \begin{enumerate}
    \item $ψ_2 ፡ 𝕍((X◁Y)⊗(Z◁W);•) → 𝕍((X⊗Z)◁(Y⊗W);•)$,
    \item $ψ_0 ፡ 𝕍(I;•) → 𝕍(I◁I;•)$,
    \item $φ_2 ፡ 𝕍(N⊗N;•) → 𝕍(N;•)$, and
    \item $φ_0 ፡ 𝕍(I;•) → 𝕍(N;•)$.
  \end{enumerate}
  Interchangers, together with unitors and associators, must satisfy coherence conditions (see \Cref{sec:coherencediagramsduoidal}). We denote by $\pDuo$ the category of \produoidalCategories{} and \produoidalFunctors{}.
\end{definition}

\begin{proposition}
  Let $𝕍$ be a \produoidalCategory{}, then its category of copresheaves, $[𝕍,\Set]$, is a \duoidalCategory{}.
\end{proposition}

\begin{remark}[Nesting profunctorial structures]
  Notation for nesting functorial structures, say $(◁)$ and $(⊗)$, is straightforward: we use expressions like $(X_1 ⊗ Y_1) ◁ (X_2 ⊗ Y_2)$ without a second thought. Nesting the profunctorial (or \emph{virtual}) structures $(⊲)$ and $(⊗)$ is more subtle: defining $𝕍(X ⊗ Y; •)$ and $𝕍(X ◁ Y; •)$ for each pair of objects $X$ and $Y$ does not itself define what something like $𝕍((X_1 ⊗ Y_1) ◁ (X_2 ⊗ Y_2); •)$ means. Recall that, in the profunctorial case, $X_1 ◁ Y_1$ and $X_1 ⊗ Y_1$ are not objects themselves: they are just names for the \profunctors{} $𝕍(X_1 ◁ Y_1; •)$ and $𝕍(X_1 ⊗ Y_1; •)$, which are not \emph{representable}.

  Instead, when we write $𝕍((X_1 ⊗ Y_1) ◁ (X_2 ⊗ Y_2); •)$, we formally mean the composition of \profunctors{} $𝕍(X_1 ⊗ Y_1; •_1) ⋄ 𝕍( X_2 ⊗ Y_2; •_2) ⋄ 𝕍(•_1 ◁ •_2 ; •)$. By convention, nesting profunctorial structures means \profunctor{} composition in this text.
\end{remark}

\begin{remark}
  Should we reverse the direction of the interchangers? Depending on the author, \promonoidalCategories{} and \produoidalCategories{} are reversed. It seems that both conventions have their advantages. The one we follow here \cite{day05:centres} makes intuitive sense: it follows the multicategorical and operadic point of view -- multiple ingredients produce a result. The opposite one \cite{produoidal23} gets the interchangers to be those of a \duoidalCategory{} and it becomes clear that there is a correspondence between \produoidalCategories{} and closed \duoidalCategories{} on preshaves.
\end{remark}

\subsection{Monoidal Contour of a Produoidal Category}

Any \produoidalCategory{} freely generates a monoidal category, its \emph{monoidal contour}. Contours form a monoidal category of paths around the decomposition trees of the \produoidalCategory{}.
Contours follow a pleasant geometric pattern where we follow the shape of the decomposition, both in the parallel and sequential dimensions, to construct both sequential and parallel compositions for a monoidal category.

\begin{definition}[Monoidal contour]
  \label{def:monoidalContour}
  \defining{linkMonoidalContour}{}
    The \emph{contour} of a \produoidalCategory{} $𝔹$ is the \monoidalCategory{} $\mContour{𝔹}$ presented by two polarized objects, $X^{∘}$ and $X^{•}$, for each object $X ∈ 𝔹_{\text{obj}}$; and generated by arrows that arise from \emph{contouring} both sequential and parallel decompositions of the \promonoidalCategory{}.
    \begin{figure}[ht]
      \centering
      \includegraphics[scale=0.35]{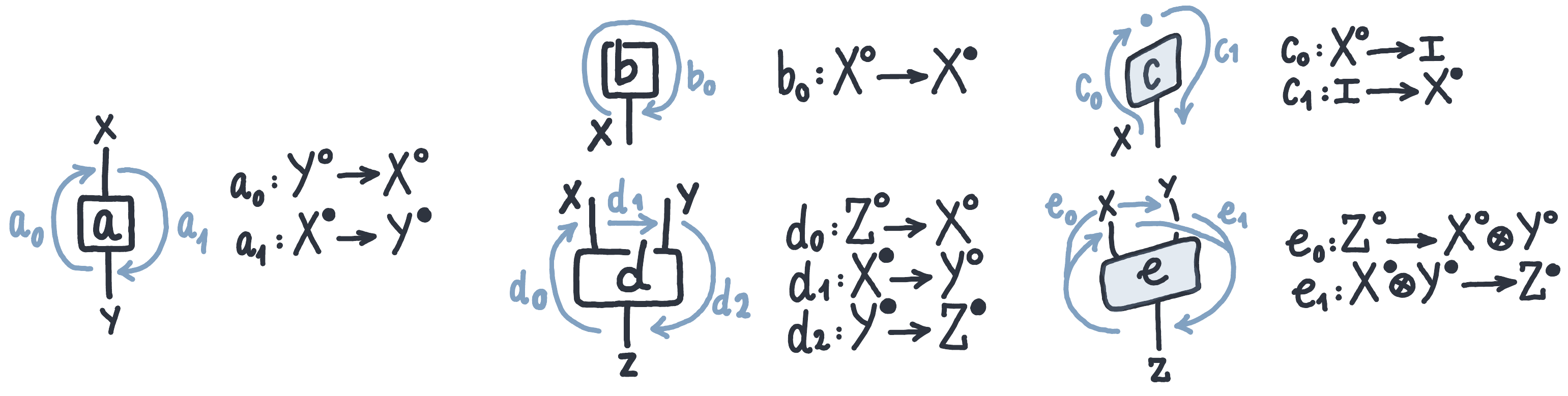}
      \caption{Generators of the monoidal category of contours.}
      \label{fig:monoidal-contour}
    \end{figure}
  \end{definition}

  Specifically, \monoidalContour{} is the \monoidalCategory{} presented by the following generators in \Cref{fig:monoidal-contour}:
    \begin{enumerate}
      \item $a_0 ፡ Y^{∘} → X^{∘}$ and $a_1 ፡ X^{•} → Y^{•}$, for each morphism $a ፡ X → Y$;
      \item $b_0 ፡ X^{∘} → X^{•}$, for each \sequentialUnit{}, $b ፡ 𝖭 → X$;
      \item $c_0 ፡ X^{∘} → I$ and $c_1 ፡ I → X^{•}$, for each \parallelUnit{}, $c ፡ I → X$;
      \item a triple of generators $d_0 ፡ Z^{∘} → X^{∘}$, $d_1 ፡ X^{•} → Y^{∘}$ and $d_2 ፡ Y^{•} → Z^{•}$, for each \sequentialSplit{} $d ፡ X ⊲ Y → Z$; and
      \item a pair of generators $e_0 ፡ Z^{∘} → X^{∘} ⊗ Y^{∘}$ and $e_1 ፡ X^{•} ⊗ Y^{•} → Z^{•}$ for each \parallelSplit{}, $e ፡ X ⊗  Y → Z$.
    \end{enumerate}

    We impose all the equations of the theory of contour. Additionally, we also impose all of the equations depicted in \Cref{fig:monoidal-contour-equation1}. Together, these form the theory of \monoidalContour{}, which adds to the theory of sequential contour a new monoidal dimension.
    \begin{figure}[ht]
        \centering
        \includegraphics[scale=0.26]{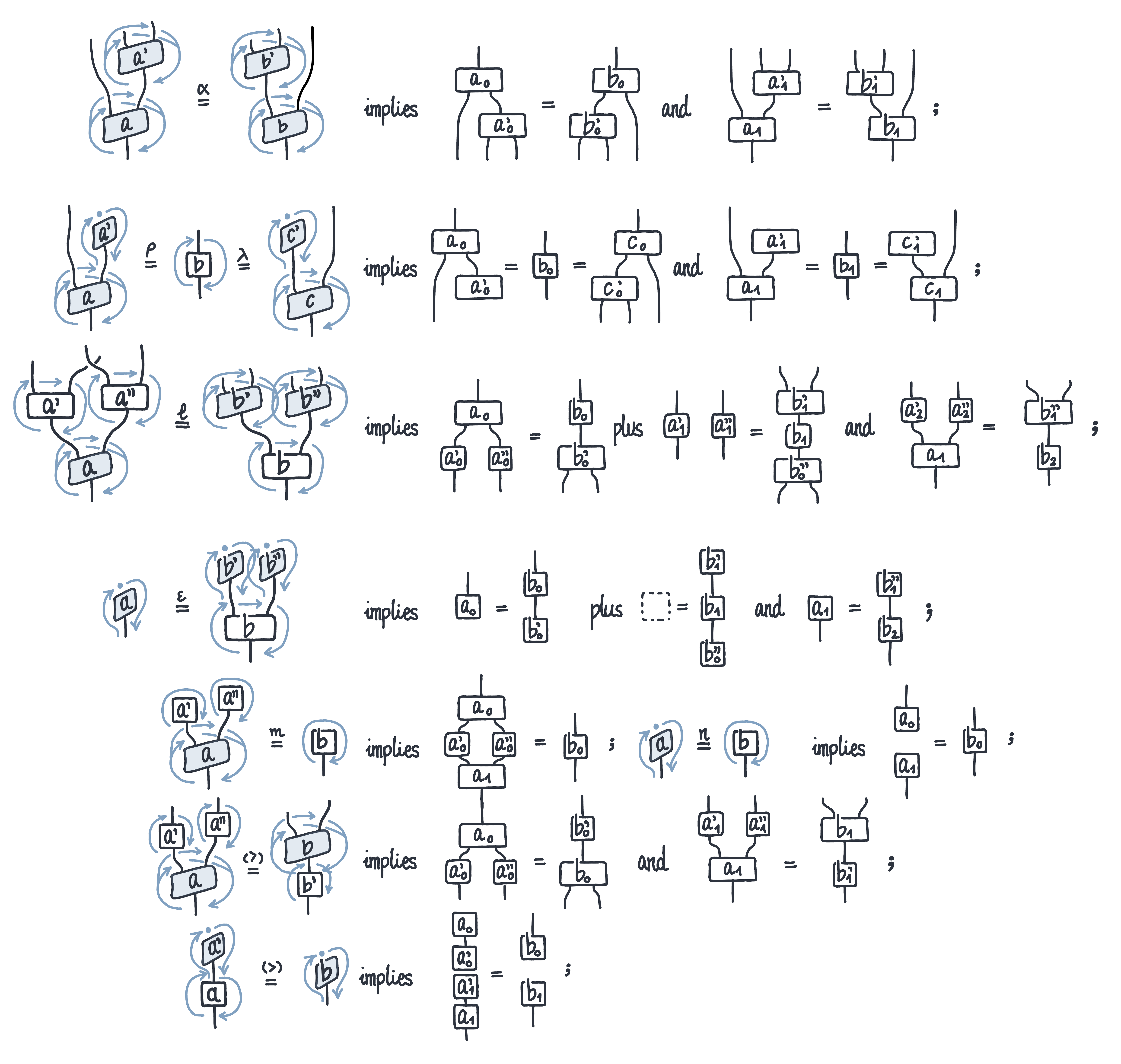}
        \caption{Extra equations for the theory of monoidal contour.}
        \label{fig:monoidal-contour-equation1}
        \label{fig:monoidalContourAssociativity}        
    \end{figure}

\begin{proposition} \label{prop:monoidalContourFunctor}
  \MonoidalContour{} extends to a functor 
  $$\mContourF : \pDuo → \MonCat{}.$$
\end{proposition}
\begin{proof}
  \Cref{def:monoidalContour} defines how the functor acts on objects. We define the action on \produoidalFunctors{}, the morphisms of the category of \produoidalCategories{}. Given a \produoidalFunctor{}, $F: 𝕍 → 𝕎$, let us define the strict monoidal functor $\mContour{F} : \mContour{𝕍} → \mContour{𝕎}$ by the following morphism of generators:
  \begin{enumerate}
    \item objects $X^{∘}$ and $X^{•}$ are mapped to $F(X)^{∘}$ and $F(X)^{•}$;
    \item for each $a ፡ X → Y$, the morphisms $a_0 : X^{∘} → X^{∘}, a_1 : X^{•} → Y^{•}$ are mapped to $F(a)_0$ and $F(a)_1$;
    \item for each $b ፡ I → X$, both $b_0 : X^{∘} → I$ and $b_1 : I → X^{•}$ are mapped to $F_I(b)_0$ and $F_I(b)_1$;
    \item for each $c ፡ 𝖭 → X$, the morphism $c_0 : X^{∘} → X^{•}$ is mapped to $F_{𝖭}(c)_0$;    
    \item for each $d ፡ X ⊲ Y → Z$, the morphisms $d_0 : Z^{∘} → X^{∘}$, $d_1 : X^{•} → Y^{∘}$ and $d_2 : Y^{•} → Z^{•}$ are mapped to $F_{◁}(d)_0$, $F_{◁}(d)_1$ and $F_{◁}(d)_2$;
    \item for each $e ፡ X ⊲ Y → Z$, the morphisms $e_0 : Z^{∘} → X^{∘} ⊗ Y^{∘}$, and $e₁ : X^{•} ⊗ Y^{•} → Z^{•}$ are mapped to $F_{⊲}(e)_0$, and $F_{◁}(e)₁$.
  \end{enumerate}
  To show that this defines a morphism of presentations, we need to prove that the assignment of generators preserves the equations of the theory of contour, in \Cref{defn:contour}. Because $F: 𝕍 → 𝕎$ is a \produoidalFunctor{}, the images of the generators do satisfy all of the contour equations of the target category. As a consequence, this assignment extends to a strict monoidal functor. 
  
  Finally, when $\id{_𝕍} : 𝕍 → 𝕍$ is an identity, the resulting functor is an identity because it is the identity on generators. Let $G : 𝕌 → 𝕍$ be another \produoidalFunctor{}, then $\mContour{G⨾F} = \mContour{G}⨾ \mContour{F}$ follows from the composition of \produoidalFunctors{}.
\end{proof}

\subsection{Produoidal Splice of a Monoidal Category}
\label{sec:produoidalSplice}

We want to go the other way around: given a \monoidalCategory{}, what is the \produoidalCategory{} that tracks the decomposition of arrows in that \monoidalCategory{}?
This subsection finds a right adjoint to the \monoidalContour{} construction: the \produoidalCategory{} of \emph{spliced monoidal arrows}.

\begin{figure}[ht]
  \centering
  \includegraphics[scale=0.35]{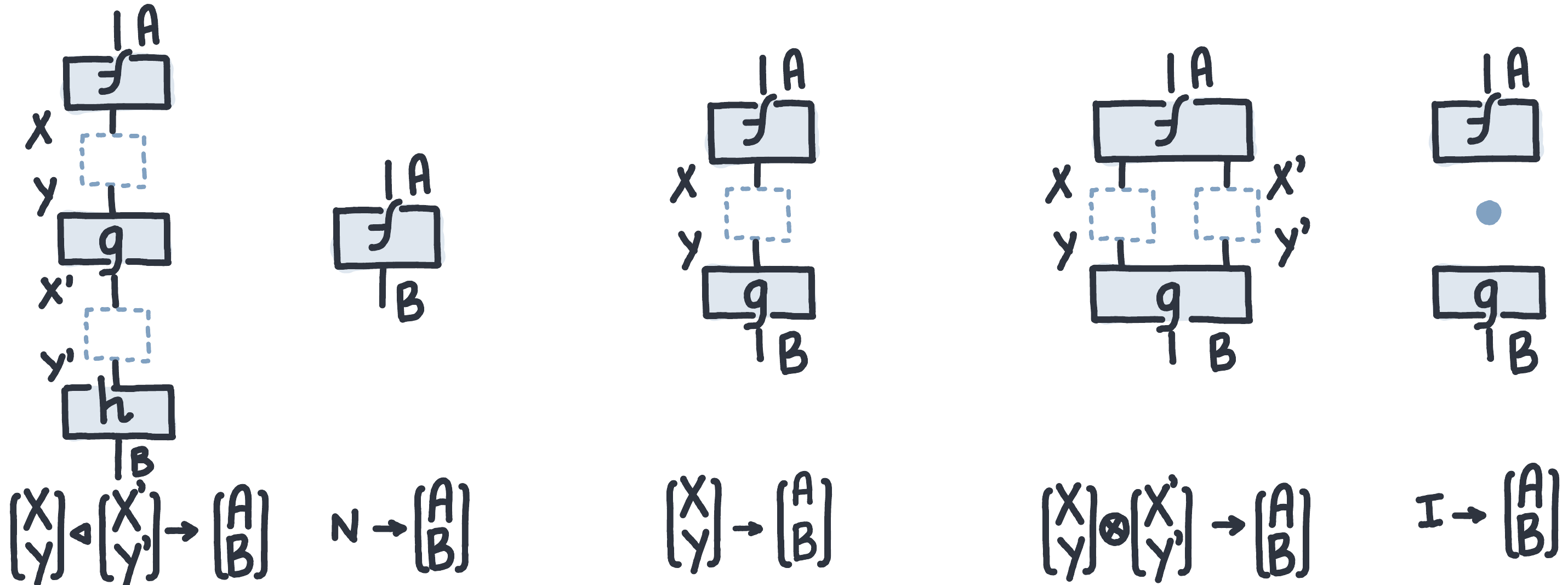}
  \caption{Spliced monoidal arrows.}
  \label{fig:monoidal-splice}
\end{figure}

\begin{definition}
  \defining{linkMonoidalSplice}{}\label{def:monoidalSplice}\label{eqs:splicedmonoidalarrows}
  Let $(ℂ,⊗,I)$ be a \monoidalCategory{}.
  Its \produoidalCategory{} of \emph{spliced monoidal arrows}, $\mSplice{ℂ}$, has objects formed by pairs, $\mSplice{ℂ}_{obj} = (ℂ^{op} × ℂ)_{obj}$, and is defined by the following profunctors, depicted in \Cref{fig:monoidal-splice}.
  \begin{enumerate}
    \item $\mSplice{ℂ} \left(\biobj{X}{Y} ;\biobj{A}{B} \right) = ℂ(A;X) × ℂ(Y,B)$,
    \item $\mSplice{ℂ}(\biobj{X}{Y} ◁ \biobj{X'}{Y'}; \biobj{A}{B}) = ℂ(A;X) × ℂ(Y;X') × ℂ(Y';B);$
    \item $\mSplice{ℂ}( \biobj{X}{Y} ⊗ \biobj{X'}{Y'}; \biobj{A}{B}) = ℂ(A;X ⊗ X') × ℂ(Y ⊗ Y';B);$
    \item $\mSplice{ℂ}(𝖭; \biobj{A}{B};) = ℂ(A;B);$
    \item $\mSplice{ℂ}(I; \biobj{A}{B}) = ℂ(A;I) × ℂ(I;B)$.
  \end{enumerate}
\end{definition}

  \begin{proposition}
    \label{ax:prop:spliceIsProduoidal}
    \SplicedMonoidalArrows{} indeed form a \produoidalCategory{}.
  \end{proposition}
  \begin{proof}[Proof sketch]
    The complete proof constructs all of the necessary natural isomorphisms using coend calculus. We refer to joint work of this author with Earnshaw and Hefford, where the equations are proven in full detail \cite{produoidal23}.
  \end{proof}

  \begin{remark}
    The produoidal algebra of spliced arrows is a natural construction: abstractly, we know that there exists a duoidal structure on the endomodules of any monoidal category \cite{day,street12:linking} -- monoidal spliced arrows form its explicitly constructed produoidal counterpart.
    What may be more surprising is that spliced arrows have themselves a universal property as part of an adjunction.
  \end{remark}
  
  \begin{theorem}\label{prop:produoidalSpliceContour}\label{prop:spliceIsProduoidal}\label{prop:monoidalSpliceFunctor}
    \SplicedMonoidalArrows{} form a \produoidalCategory{} with their \hyperlink{linkProduoidalComponents}{sequential and parallel splits, units}, and suitable coherence morphisms and laxators. \SplicedMonoidalArrows{} extend to a functor $\mSpliceF{} : \MonCat → \pDuo$.
    The \monoidalContour{} and the \produoidalSplice{} are left and right adjoints to each other, respectively.
  \end{theorem}
  \begin{proof}
    Monoidal contour $\mContour{𝔹}$ is presented by generators and equations: to specify a strict monoidal functor $\mContour{𝔹} → 𝕄$, it is enough to specify images of the generators and then prove that they satisfy the equations. 
    
    Let $(𝕄,⊗_M,I_M)$ be a \monoidalCategory{}. Then a strict monoidal functor $\mContour{𝔹} → 𝕄$ amounts to the following data satisfying some extra conditions.
    \begin{enumerate}
    \item For each object $X \in \obj{𝔹}$, a pair of objects $X^{∘}, X^{•} ∈ \obj{𝕄}$;
    \item for each element $f ፡ 𝖭 → X$, a morphism $f_0 ፡ X^{∘} → X^{•}$;
    \item for each unit $f ፡ I → X$, a choice of $f_0 ፡ X^{∘} → I$ and $f_1 ፡ I → X^{•}$;
    \item for each morphism $f ፡ X → Y$, a choice of $f_0 ፡ Y^{∘} → X^{∘}$ and $f_1 ፡ X^{•} → Y^{•}$;
    \item for each sequential split $f ፡ X ⊲ Y → Z$, a choice of morphisms $f_0 ፡ Z^{∘} → X^{∘}$, plus $f_1 ፡ X^{•} → Y^{∘} \mbox{ and } f_2 ፡ Y^{•} → Z^{•}$;
    \item for each parallel split $f ፡ X ⊗ Y → Z$, a choice of morphisms $f_0 ፡ Z^{∘} → X^{∘} ⊗ Y^{∘}\mbox{ and }f_1 ፡ X^{•} ⊗ Y^{•} → Z^{•}$.
    \end{enumerate}
  
    In order to construct a well-defined strict monoidal functor, the previous assignments must satisfy the following conditions for each one of the two \promonoidalCategories{}:
    \begin{enumerate}
      \item $α(a ⨾_1 b) = (c ⨾_2 d)$ in the \promonoidalCategory{} implies $a_0 ⨾ (b_0 ⊗ \id) = c_0 ⨾ (\id ⊗ d_0)$ and $(b_1 ⊗ \id) ⨾ a_1 = (\id ⊗ d_1) ⨾ c_1$ in the \monoidalCategory{};
      \item $λ(a ⨾_1 b) = c = ρ(d ⨾_2 e)$ in the \promonoidalCategory{} implies $a_0 ⨾ (b_0 ⊗ \id) = c_0 = d_0 ⨾ (\id ⊗ e_0)$ and  $(b_1 ⊗ \id) ⨾ a_1 = c_1 = (\id ⊗ e_1) ⨾ d_1$ in the \monoidalCategory{};
    \end{enumerate}
    Moreover, they must also satisfy the following conditions for the \produoidalCategory{}.
    \begin{enumerate}
      \item $ψ_2(a ｜ b ｜ c) = (d ｜ e ｜ f)$ in the \promonoidalCategory{} implies $a_0 ⨾ (b_0 ⊗ c_0)  = d_0 ⨾ e_0, b_1 ⊗ c_1 = e_1 ⨾ d_1 ⨾ f_0$ and $(b_2 ⊗ c_2) ⨾ a_1 = f_1 ⨾ d_2$ in the \monoidalCategory{};
      \item $ψ_0(a) = (b ｜ c ｜ d)$ in the \promonoidalCategory{} implies $a_0 = b_0 ⨾ c_0$, $\id = c_1 ⨾ b_1 ⨾ d_0$, and $a_1 = d_1 ⨾ b_2$ in the \monoidalCategory{};
      \item $φ_2(a ｜ b ｜ c) = d$ in the \promonoidalCategory{} implies $a_0 ⨾ (b_0 ⊗ c_0) ⨾ a_1 = d_0$ in the \monoidalCategory{};
      \item $φ_0(a) = b$ in the \promonoidalCategory{} implies $a_0 ⨾ a_1 = b_0$ in the \monoidalCategory{}.
    \end{enumerate}
    
    On the other hand, a \produoidalFunctor{} $𝔹 \to \mSplice{𝕄}$, also amounts to the following data. We will state it in multiple points and finally confirm that each one of these points has a correspondence on the first part of the proof, finishing the definition.
    \begin{enumerate}
    \item For each object $X \in \obj{𝔹}$, an object $(X^{∘}, X^{•}) \in \obj{\mSplice{𝕄}}$;
    \item for each element $f ፡ 𝖭 → X$, a morphism $f_0 ፡  𝖭 → \pbiobj{X^{∘}}{X^{•}};$
    \item for each unit, $f ፡ I → X$, a unit $⟨ f_0 \parallel f₁ ⟩ ፡ I → \pbiobj{X^{∘}}{X^{•}}$;
    \item for each morphism $f ፡ X → Y$, a splice $\bisplice{f_0}{f_1} ፡ \pbiobj{X^{∘}}{X^{•}} → \pbiobj{Y^{∘}}{Y^{•}}$;
    \item for each seq. join $f ፡ X ⊲ Y → Z$, a spliced arrow 
    $$\trisplice{f_0}{f_1}{f_2} ፡ \pbiobj{X^{∘}}{X^{•}} ⊲ \pbiobj{Y^{∘}}{Y^{•}} → \pbiobj{Z^{∘}}{Z^{•}};$$
    \item for each par. join $f ፡ X ⊗ Y → Z$, a spliced arrow 
    $$⟨ f_0 ⨾ □ ⊗ □ ⨾ f_1 ⟩ ፡ \pbiobj{X^{∘}}{X^{•}} ⊗ \pbiobj{Y^{∘}}{Y^{•}}→ \pbiobj{Z^{∘}}{Z^{•}}.$$
  \end{enumerate}
  
  The following conditions must hold for each one of the \promonoidalCategories{}. These correspond definitionally to the conditions for the two promonoidal structures we imposed before.
  \begin{enumerate}
    \item $α(a ｜ b) = (c ｜d)$ in the \produoidalCategory{} implies $α(Fa ｜Fb) = (Fc ｜ Fd)$ for the spliced arrows;
    \item $λ(a ｜ b) = c = ρ(d ｜ e)$ in the \produoidalCategory{} implies $λ(Fa ｜ Fb) = Fc = ρ(Fd ｜ Fe)$ for the spliced arrows.
  \end{enumerate}

  Finally, all the following conditions must also hold for the \produoidalCategory{}. These are definitionally equal to the conditions for the \produoidalCategory{} we imposed before.
  \begin{enumerate}
    \item $ψ_2(a ｜ b ｜ c) = (d ｜ e ｜ f)$ in the \produoidalCategory{} implies $$ψ_2(Fa ｜ Fb ｜ Fc) = (Fd ｜Fe ｜Ff);$$
    \item $ψ_0(a) = (b ｜ c ｜ d)$ in the \produoidalCategory{} implies $$ψ_0(Fa) = (Fa ｜ Fc ｜ Fd);$$
    \item $φ_2(a ｜ b ｜ c) = d$ in the \produoidalCategory{} implies $$φ_2(Fa ｜Fb ｜ Fc) = Fd;$$
    \item $φ_0(a) = b$ in the \produoidalCategory{} implies $φ_0(Fa) = Fb$.
  \end{enumerate}
  
  Each of these points is exactly equal by definition to the relative point in the first part of the proof: this establishes the desired adjunction. In other words, the definition of monoidal splice is precisely the one that makes this proof hold definitionally -- even when we gave multiple different characterizations of it.
  \end{proof}

  \subsection{A Representable Parallel Structure}

  A \produoidalCategory{} has two tensors, and neither is, in principle, representable. However, the cofree \produoidalCategory{} over a category we have just constructed happens also to have a representable tensor, $(⊗)$: spliced monoidal arrows form a \monoidalCategory{}.
  
  \begin{remark} %
    This means $\mSplice{ℂ}$ has the structure of a \emph{virtual duoidal category} \cite{nlab:duoidal} or \emph{monoidal multicategory}, defined by Aguiar, Haim and López Franco \cite{aguiar18} as a pseudomonoid in the cartesian monoidal 2-category of multicategories.
  \end{remark}

  \begin{proposition}
    \ParallelSplits{} and \parallelUnits{} of spliced monoidal arrows are representable \profunctors{}. Explicitly,
  \begin{align*}
    & \mSplice{ℂ}\left( \biobj{X}{Y} ⊗ \biobj{X'}{Y'}; \biobj{A}{B}\right) ≅ 
    \mSplice{ℂ}\left( \biobj{X ⊗ X'}{Y ⊗ Y'} ; \biobj{A}{B}\right),
    \mbox{ and } \\
    & \mSplice{ℂ}\left( I ; \biobj{A}{B}\right) ≅ 
    \mSplice{ℂ}\left( \biobj{I}{I} ; \biobj{A}{B} \right).
  \end{align*}
  \end{proposition}

  In fact, these sets are equal by definition. However, we argue that there is a reason to work in the full generality of \produoidalCategories{}: produoidal categories can always be \emph{normalized}. 
  
  \begin{remark}
    Normalization is a procedure to mix both tensors of a duoidal category, $(⊗)$ and $(\triangleleft)$, but not every \duoidalCategory{} has a normalization \cite{garner16}. 
    It is folklore that one loses nothing by regarding non-representable produoidal structures as representable \emph{duoidal structures on presheaves}, dismissing that they are moreover \emph{closed} \cite{day}; thus, one would expect only some produoidal categories to be normalizable -- after all, only some duoidal categories are.
    Against folklore, we prove that every \produoidalCategory{}, representable or not, has a \emph{universal normalization}, a normal \produoidalCategory{} which may be again representable or not. 
  \end{remark}

  \subsection{Bibliography}
  Motivated by language theory and the representation theorem of Chomsky and Schützenberger, Melliès and Zeilberger \cite{mellies22:parsing} were the first to present the multicategorical \emph{splice-contour} adjunction. We are indebted to their exposition, which we extend to the promonoidal and produoidal cases.  Our contribution is to show how \monoidalContexts{} arise from an extended \produoidal{} splice-contour adjunction; unifying these two threads.
  
  Street already noted that the endoprofunctors of a \monoidalCategory{} had a duoidal structure \cite{street12:linking}; Pastro and Street described a promonoidal structure on lenses \cite{pastro07} and Garner and López-Franco contributed a partial normalization procedure for duoidal categories \cite{garner16}.
  We build on top of this literature, putting it together, spelling out existence proofs, popularizing its produoidal counterpart and providing multiple new results and constructions that were previously missing (e.g. \Cref{prop:produoidalSpliceContour,prop:MonoidalContextProtensor,th:freeNormalProduoidal}).

  This section takes its main ideas from joint work of this author with Matt Earnshaw and James Hefford \cite{produoidal23}. Earnshaw and \Sobocinski{}~\cite{earnshaw22} have described a syntactic congruence on formal languages of string diagrams using \monoidalContexts{}.

  \newpage
  \section{Interlude: Produoidal Normalization}
  \label{sec:normalization}

  \subsection{Normal Produoidal Categories}
  \ProduoidalCategories{} seem to contain too much structure: of course, we want to split things in two different ways, sequentially $(◁)$ and in parallel $(⊗)$; but that does not necessarily mean that we want to keep track of two different types of units, parallel $(I)$ and sequential $(N)$. The atomic components of our decomposition algebra should be the same, without having to care if they are \emph{atomic for sequential composition} or \emph{atomic for parallel composition}. 
  
  \begin{remark}
    The monoidal spliced arrows we just introduced are a perfect example: if we simply want a hole in a string diagram, the type it may take depends on the wires we want to leave to each side (see \Cref{fig:centrality}).
    We would prefer to construct a new category -- a Kleisli category on top of this one -- where the bureaucracy of the units was already handled for us.
    \begin{figure}[ht]
      \centering
      \includegraphics[scale=0.35]{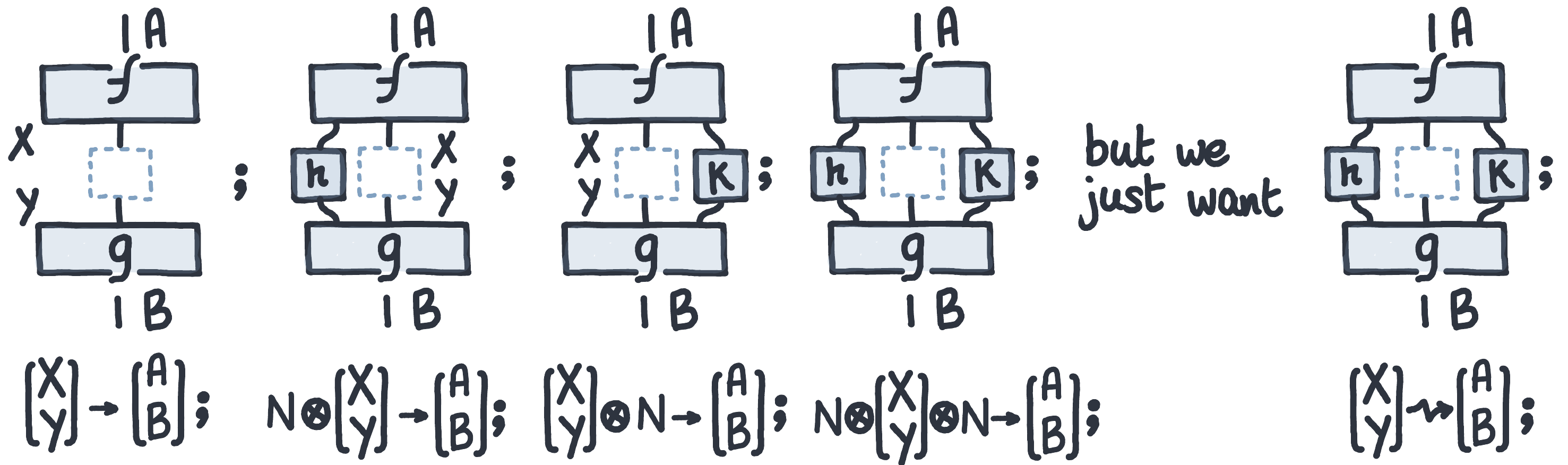}
      \caption{Multiple units complicate types.}
      \label{fig:motivate-normalized-splice}    
    \end{figure}
  \end{remark}

  \subsection{The Normalization Monad}
  Fortunately, there exists an abstract procedure that, starting from any \produoidalCategory{}, constructs a new produoidal category where both units are identified. This procedure is known as \emph{normalization}, and the resulting produoidal categories are called \emph{normal}.
  
  \begin{definition}[Normal produoidal category]
    \defining{linkNormalProduoidalCategory}{}
    A \emph{normal produoidal category} is a \produoidalCategory{} where the unit interchanger $n \colon 𝕍(I;•) → 𝕍(𝖭;•)$ is an isomorphism.

    \NormalProduoidalCategories{} form a category with produoidal functors between them, $\npDuo$. As a consequence, it is endowed with a fully faithful forgetful functor $𝓤 \colon \npDuo → \pDuo$.
  \end{definition}
  
  \begin{theorem}
    \label{th:normalizationProduoidal}\label{eqs:normalization}\defining{linkProduoidalNormalization}{}
      Let $𝕍_{⊗,I,◁,N}$ be a \produoidal{} category. The profunctor
      $$\NOR(𝕍)(•; •) = 𝕍(𝖭 ⊗ • ⊗ 𝖭; •)$$
      forms a \promonad{}. Moreover, the Kleisli category of this \promonad{} is a \normalProduoidalCategory{} with the following \profunctors{}. 
      \begin{enumerate}
        \item $\NOR(𝕍)(X ⊗_𝖭 Y; Z) = 𝕍(𝖭 ⊗ X ⊗ 𝖭 ⊗ Y ⊗ 𝖭; Z)$;
        \item $\NOR(𝕍)(X ⊲_𝖭 Y; Z) = 𝕍((𝖭 ⊗ X ⊗ 𝖭) ⊲ (𝖭 ⊗ Y ⊗ 𝖭); Z)$; and
        \item $\NOR(𝕍)(I_𝖭;X) = \NOR(𝕍)(𝖭_𝖭;X) = 𝕍(𝖭;X)$.
      \end{enumerate}
  \end{theorem}
  \begin{proof}
    Let us prove that $\NOR(𝕍)$ is a \promonad{}.
    We now define the multiplication and unit for the \promonad{}, $\NOR(𝕍)$. They are constructed out of laxators of the produoidal category $𝕍$ and Yoneda isomorphisms; thus, they must be associative and unital by coherence. The unit is defined by \emph{(i)} unitality of $𝕍$, \emph{(ii)} the laxator of $𝕍$, and \emph{(iii)} by definition of $\NOR(𝕍)$.
    $$𝕍(X; Y) ≅ 𝕍(I ⊗ X ⊗ I; Y) → 𝕍(𝖭 ⊗ X ⊗ 𝖭; Y) ≅ \NOR(𝕍)(X; Y).$$
    Multiplication is constructed as follows, using \emph{(i)} the definition of $\NOR(𝕍)$, \emph{(ii)} Yoneda reduction, \emph{(iii)} the laxators of $𝕍$, and \emph{(iv)} the definition of $\NOR(𝕍)$.
    \begin{align*}
      {\textstyle ∫}^{Y ∈ 𝕍} &\NOR(𝕍)(X; Y) × \NOR(𝕍)(Y; Z) =
      {\textstyle ∫}^{Y ∈ 𝕍} 𝕍(𝖭 ⊗ X ⊗ 𝖭; Y) × 𝕍(𝖭 ⊗ Y ⊗ 𝖭; Z) \\
      &≅ 𝕍(𝖭 ⊗ 𝖭 ⊗ X ⊗ 𝖭 ⊗ 𝖭; Y) → 𝕍(𝖭 ⊗ Y ⊗ 𝖭; Z) = \NOR(𝕍)(X; Z).
    \end{align*}
    Associativity and unitality follow from those of the \produoidalCategory{}.

    The second part of this proof will show that $\NOR(𝕍)$ is indeed a \produoidalCategory{}: we will construct its associators, unitors and laxators from those of $𝕍$ and Yoneda isomorphisms. The right unitor is constructed as follows; the left unitor is constructed in a similar way: \emph{(i)} by definition of $\NOR(𝕍)$, \emph{(ii)} by associativity, \emph{(iii)} by definition, \emph{(iv)} by Yoneda reduction, \emph{(v)} by definition, and \emph{(vi)} by unitality.
    \begin{align*}
    & {\textstyle ∫}^{M ∈ \NOR(𝕍)} \NOR(𝕍)(𝖭; M) × \NOR(𝕍)(X ⊗_{𝖭} M; Y) &\quad
    = & \\
    & {\textstyle ∫}^{M ∈ \NOR(𝕍)} \NOR(𝕍)(𝖭; M) × 𝕍(𝖭 ⊗ X ⊗ 𝖭 ⊗ M ⊗ 𝖭; Y) &\quad
    ≅ & \\
    & {\textstyle ∫}^{M ∈ \NOR(𝕍), P ∈ 𝕍} \NOR(𝕍)(𝖭; M) × 𝕍(𝖭 ⊗ M ⊗ 𝖭; P) ×  𝕍(𝖭 ⊗ X ⊗ P; Y) &\quad
    = & \\
    & {\textstyle ∫}^{M ∈ \NOR(𝕍), P ∈ 𝕍} \NOR(𝕍)(𝖭; M) × \NOR(𝕍)(M; P) × 𝕍(𝖭 ⊗ X ⊗ P; Y) &\quad
    ≅ & \\
    & {\textstyle ∫}^{P ∈ 𝕍} \NOR(𝕍)(𝖭; P) × 𝕍(𝖭 ⊗ X ⊗ P; Y)  &\quad
    = & \\
    & {\textstyle ∫}^{P ∈ 𝕍} 𝕍(𝖭; P) × 𝕍(𝖭 ⊗ X ⊗ P; Y)  &\quad
    ≅ & \\
    & 𝕍(𝖭 ⊗ X ⊗ 𝖭; Y).
    \end{align*}
    Let us now construct the associator in two steps: we will show that both sides of the following equation
    \begin{align*}
      & {\textstyle ∫}^{P ∈ \NOR(𝕍)} \NOR(𝕍)(Y ⊗_{𝖭} Z; P) × \NOR(𝕍)(X ⊗_{𝖭} P; A) \  ≅ \\
      & {\textstyle ∫}^{M ∈ \NOR(𝕍)} \NOR(𝕍)(X ⊗_{𝖭} Y; M) × \NOR(𝕍)(M ⊗_{𝖭} Z; A)
    \end{align*}
    are isomorphic to $𝕍(𝖭 ⊗ X ⊗ 𝖭 ⊗ Y ⊗ 𝖭 ⊗ Z ⊗ 𝖭; A)$. The first side is isomorphic by \emph{(i)} definition of $\NOR(𝕍)$, \emph{(ii)} associativity of $𝕍$, \emph{(iii)} definition of $\NOR(𝕍)$, \emph{(iv)} Yoneda reduction, \emph{(v)} definition of $\NOR(𝕍)$, and \emph{(vi)} associativity. The second side is analogous.
    \begin{align*}
      & {\textstyle ∫}^{M ∈ \NOR(𝕍)} \NOR(𝕍)(Y ⊗_{𝖭} Z; M) × \NOR(𝕍)(X ⊗_{𝖭} M; A) &\quad
      = & \\
      & {\textstyle ∫}^{M ∈ \NOR(𝕍)} \NOR(𝕍)(Y ⊗_{𝖭} Z; M) × 𝕍(𝖭 ⊗ X ⊗ 𝖭 ⊗ M ⊗ 𝖭; A)  &\quad
      ≅ & \\
      & {\textstyle ∫}^{M ∈ \NOR(𝕍), P ∈ 𝕍} \NOR(𝕍)(Y ⊗_{𝖭} Z; M) × 𝕍(𝖭 ⊗ M ⊗ 𝖭; P) × 𝕍(𝖭 ⊗ X ⊗ P; A) &\quad
      = & \\
      & {\textstyle ∫}^{M ∈ \NOR(𝕍), P ∈ 𝕍} \NOR(𝕍)(Y ⊗_{𝖭} Z; M) × \NOR(𝕍)(M; P) × 𝕍(𝖭 ⊗ X ⊗ P; A) &\quad
      ≅ & \\
      & {\textstyle ∫}^{P ∈ 𝕍} \NOR(𝕍)(Y ⊗_{𝖭} Z; P) × 𝕍(𝖭 ⊗ X ⊗ P; A) ×  &\quad
      = & \\
      & {\textstyle ∫}^{P ∈ 𝕍} 𝕍(𝖭 ⊗ Y ⊗ 𝖭 ⊗ Z ⊗ 𝖭; P) × 𝕍(𝖭 ⊗ X ⊗ P; A) &\quad
      ≅ & \\
      & 𝕍(𝖭 ⊗ X ⊗ 𝖭 ⊗ Y ⊗ 𝖭 ⊗ Z ⊗ 𝖭; A).
      \end{align*}
    Finally, let us construct the four interchangers that define the \produoidalCategory{}. Three of them are immediate: they are either identities or unitors: $\NOR(𝕍)(I_{𝖭},A) ≅ \NOR(𝕍)(I_{𝖭} ⊲ I_{𝖭}; A)$ is the first, $\NOR(𝕍)(𝖭 ⊗ 𝖭; A) ≅ \NOR(𝕍)(𝖭; A)$ is the second, and $\NOR(𝕍)(I_{𝖭};A) ≅ \NOR(𝕍)(𝖭_{𝖭};A)$ is the last one. We note here that this \produoidalCategory{} is natural because of these isomorphisms. Thus, we only need to construct the first interchanger morphism of a \produoidalCategory{},
    $$
    \NOR(𝕍)((X₁ ⊲ Y₁) ⊗ (X₂ ⊲ Y₂); A) \longrightarrow 
    \NOR(𝕍)((X₁ ⊗ X₂) ◁ (Y₁ ⊗ Y₂); A),
    $$
    which is defined by the following reasoning. Here, we abbreviate $𝖭 ⊗ X ⊗ 𝖭$ by $X^{⊗𝖭}$, and we apply \emph{(i)} the definition of the tensors of $\NOR(𝕍)$, \emph{(ii)} the interchanger of $𝕍$, \emph{(iii)} unitality in $𝕍$, and \emph{(iv)} the definition of tensors of $\NOR(𝕍)$.
    \begin{align*}
    & \NOR(𝕍)((X₁ ⊲_{𝖭} Y₁) ⊗_{𝖭} (X₂ ⊲_{𝖭} Y₂); A) && = \\
    & 𝕍(𝖭 ⊗ (X_1^{⊗𝖭} ⊲ Y_1^{⊗𝖭}) ⊗ 𝖭 ⊗ (X_2^{⊗𝖭} ⊲ Y_2^{⊗𝖭}) ⊗ 𝖭; A) & & \to  \\
    & 𝕍((𝖭 ⊗ X_1^{⊗𝖭} ⊗ 𝖭 ⊗ Y_1^{⊗𝖭} ⊗ 𝖭) ⊲ (𝖭 ⊗ X_2^{⊗𝖭} ⊗ 𝖭 ⊗ Y_2^{⊗𝖭} ⊗ 𝖭); A) & & \to  \\
    & 𝕍((𝖭 ⊗ X_1 ⊗ 𝖭 ⊗ Y_1 ⊗ 𝖭)^{⊗𝖭} ⊲ (𝖭 ⊗ X_2 ⊗ 𝖭 ⊗ Y_2 ⊗ 𝖭)^{⊗𝖭}; A) & & \to  \\
    & \NOR(𝕍)((X_1 ⊗_{𝖭} Y_1) ⊲_{𝖭} (X_2 ⊗_{𝖭} Y_2); A).
    \end{align*}
    The structure equations of the laxators follow from those of the base category $𝕍$. This finishes the construction of a \produoidalCategory{} on the Kleisli category of the \promonad{}.
  \end{proof}
  
  \begin{lemma}
    Normalization extends to an idempotent monad.
  \end{lemma}
  \begin{proof}
    The first part of the proof will show that $\NOR(𝕍)$ is the free \normalProduoidalCategory{} over $𝕍$ by constructing the a monad structure on top of the functor $\NOR ፡ \npDuo → \npDuo$. Let us construct the unit and the multiplication of the monad.
    The unit $\eta_{𝕍} : 𝕍 → \NOR(𝕍)$ is defined as the identity-on-objects functor associated to the \promonad{}; it and acts on morphisms by the unit of the promonad. Let us show that this is a \produoidalFunctor{} by constructing the following components; all of them use the map $I → 𝖭$ from the base \produoidalCategory{},
    \begin{enumerate}
      \item $η_{⊗} ፡ 𝕍(X ⊗ Y; A) → 𝕍(I ⊗ X ⊗ I ⊗ Y ⊗ I; A) → 𝕍(𝖭 ⊗ X ⊗ 𝖭 ⊗ Y ⊗ 𝖭; A)$;
      \item $η_{I} : 𝕍(I;A) → 𝕍(𝖭;A)$;
      \item $η_{⊲} : 𝕍(X ⊲ Y; A) → 𝕍((I ⊗ X ⊗ I) ⊲ (I ⊗ Y ⊗ I); A) → 𝕍((𝖭 ⊗ X ⊗ 𝖭) ⊲ (𝖭 ⊗ Y ⊗ 𝖭); A)$;
      \item $η_{𝖭} : 𝕍(𝖭; A) → 𝕍(𝖭;A)$ is simply an identity.
    \end{enumerate}
    these preserve laxators and coherence maps since they are constructed only from laxators and coherence maps.

    Let us construct now the multiplication of the monad, $μ_{𝕍} : \NOR(\NOR(𝕍)) → \NOR(𝕍)$, and show that it is an isomorphism, making it an idempotent monad. The underlying functor is identity on objects, and it acts on morphisms by 
    the normality of the already normalized \produoidalCategory{},
    $$\NOR(\NOR(𝕍))(X;Y) = \NOR(𝕍)(𝖭 ⊗_{𝖭} X ⊗_{𝖭} 𝖭; Y) ≅ \NOR(𝕍)(X;Y).$$
    The following components make this functor a \produoidalFunctor{}, they are constructed again from the normality of the already normalized \produoidalCategory{}:
    \begin{enumerate}
      \item $\mu_{⊗} : \NOR(\NOR(𝕍))(X ⊗_{𝖭𝖭} Y; A) = \NOR(𝕍)(𝖭 ⊗_𝖭 X ⊗_𝖭 𝖭 ⊗_𝖭 Y ⊗_𝖭 𝖭; A) ≅ \NOR(𝕍)(X ⊗_N Y;A)$;
      \item $\mu_{⊲} : \NOR(\NOR(𝕍))(X ⊲_{𝖭𝖭} Y; A) = \NOR(𝕍)((𝖭 ⊗_𝖭 X ⊗_𝖭 𝖭) ⊲_𝖭 (𝖭 ⊗_𝖭 Y ⊗_𝖭 𝖭); A) ≅ \NOR(𝕍)(X ⊲_𝖭 Y; A)$;
      \item $\mu_{I} : \NOR(\NOR(𝕍))(𝖭;A) = \NOR(𝕍)(𝖭;A)$;
      \item $\mu_{𝖭} : \NOR(\NOR(𝕍))(𝖭;A) = \NOR(𝕍)(𝖭;A)$,
    \end{enumerate}
    Finally we verify the monad laws. $η_{𝓝𝕍} ⨾ \mu_{𝕍}$ is an identity-on-objects; on morphisms, it applies left and right unitors followed by their inverses; as a consequence, its underlying functor is the identity. 
    The components of the natural transformations are also identities, since the interchanger $I → 𝖭$ is an identity for the normalized \produoidalCategory{} $\NOR(𝕍)$; they are otherwise composed of unitors followed by their inverses. The last step is to check the associativity of the monad, $\mu_{\NOR(𝕍)} ⨾ \mu_𝕍$ and $\NOR(\mu_{𝕍}) ⨾ \mu_𝕍$; this is simply the identity on objects, so we simply apply left and right unitors twice on morphisms and their components.
  \end{proof}
  
  \begin{lemma}
    \label{lemma:ax:exactlyOneAlgebra}
    A \produoidalCategory{} $𝕍$ has exactly one algebra structure for the normalization monad when it is normal, and none otherwise.
  \end{lemma}
  \begin{proof}
    Let $(F, F_{⊗}, F_I, F_{𝖭}, F_𝖭) ፡ \NOR(𝕍) → 𝕍$ be an algebra. This means that the following commutative diagrams with the unit and multiplication of the normalization monad must commute.
    \begin{center}
      \begin{tikzcd}
        𝕍 \arrow[r, "η"] \arrow[rd, "\mathrm{id}"', no head] & \NOR(𝕍) \arrow[d, "F"] &  & \NOR(\NOR(𝕍)) \arrow[r, "μ"] \arrow[d, "\NOR(F)"'] & \NOR(𝕍) \arrow[d, "F"] \\
        & 𝕍 && \NOR(𝕍) \arrow[r, "F"'] & 𝕍                 
        \end{tikzcd}  
    \end{center}
    Now, consider how the interchanger $ψ_0 ፡ 𝕍(I; •) → 𝕍(𝖭; •)$ is transported by these maps.
    \begin{center}
      \begin{tikzcd}
          & 𝕍(𝖭; •) \arrow[rd, "F_I"] \arrow[d, "\mathrm{id}"] & \\
          𝕍(I; •) \arrow[ru, "η_I"] \arrow[d, "ψ_0"'] \arrow[rr, "\mathrm{id}"', bend right] & 
          𝕍(𝖭; •) \arrow[rd, "F_N"'] & 
          𝕍(I; •) \arrow[d, "ψ_0"] \\
          𝕍(𝖭; •) \arrow[ru, "\mathrm{id}"'] \arrow[rr, "\mathrm{id}"', bend right] && 
          𝕍(𝖭; •)
      \end{tikzcd}
    \end{center}
    We conclude that $η_I = ψ_0$, but also that $F_𝖭 = \mathrm{id}$. As a consequence, $ψ_0$ is invertible and $F_I$ must be its inverse. We have shown that any \produoidalCategory{} that is an algebra for the normalization monad must be normal.
  
    We will now show that this already determines all of the functor $F$.  We know that $η_{⊗}, η_{◁}, η$ are isomorphisms because they are constructed from the unitors, associators, and the laxator $ψ_0$, which is an isomorphism in this case. This determines that $F_{⊗}, F_{⊲}, F$ must be their inverses. By construction, these satisfy all structure equations.
  \end{proof}

  \begin{theorem}[Free normal produoidal]
    \label{th:freeNormalProduoidal}\label{th:normalizationIdempotent}
    Normalization determines an adjunction between \produoidalCategories{} and \normalProduoidalCategories{},
    $$\NOR ፡ \ProDuo \rightleftharpoons \npDuo \colon \mathsf{Forget}.$$
    That is, $\NOR(𝕍)$ is the free \normalProduoidalCategory{} over $𝕍$.
  \end{theorem}
  \begin{proof}
    We know that the algebras for the normalization monad are exactly the \normalProduoidalCategories{} (\Cref{lemma:ax:exactlyOneAlgebra}). We also know that the normalization monad is idempotent (\Cref{th:normalizationIdempotent}). 
    This implies that the forgetful functor from its category of algebras is fully faithful, and thus, the algebra morphisms are exactly the produoidal functors. As a consequence, the canonical adjunction to the category of algebras of the monad is exactly an adjunction to the category of \normalProduoidalCategories{}. 
  \end{proof}

  \begin{remark}
    Garner and López Franco \cite{garner16} introduced a partial normalization procedure for duoidal categories. 
    We contribute a general normalization procedure for \produoidalCategories{} and we characterize it universally.
    Produoidal normalization behaves slightly better than duoidal normalization: it always succeeds, and we prove that it forms an idempotent monad (\Cref{th:normalizationIdempotent}).
    The technical reason for this improvement is that the original duoidal normalization required the existence of certain coequalizers in $𝕍$; produoidal normalization uses coequalizers in $\Set$.
  \end{remark}
  
  In the previous \Cref{sec:produoidalSplice}, we constructed the \produoidalCategory{} of \splicedMonoidalArrows{}, which distinguishes between morphisms and morphisms with a hole in the monoidal unit.
  This is because the latter hole splits the morphism in two parts.   
  Normalization equates both; it sews these two parts.
  In \Cref{sec:monoidalContexts}, we explicitly construct monoidal contexts, the normalization of \splicedMonoidalArrows{}. Before that, let us also introduce the symmetric version of normalization.
  
\subsection{Symmetric Normalization}

Normalization is a generic procedure that applies to any \produoidalCategory{}, it does not matter if the parallel join $(⊗)$ is symmetric or not. However, when $⊗$ happens to be symmetric, we can also apply a more specialized normalization procedure: \emph{symmetric normalization}.

\begin{definition}[Symmetric produoidal category]
  \defining{linkSymmetricProduoidal}{}
  \label{def:symmetricProduoidal}
  A \emph{symmetric produoidal category} is a \produoidalCategory{} $𝕍_{◁,N,⊗,I}$ endowed with a natural isomorphism
  $σ ፡ 𝕍(X ⊗ Y;Z) ≅ 𝕍(Y ⊗ X; Z)$ satisfying the symmetry and hexagon equations.
\end{definition}

\begin{theorem}
  \label{th:symNormalizationProduoidal}\defining{linkSymmetricProduoidalNormalization}{}
    Let $𝕍$ be a \symmetricProduoidal{} category. The profunctor
    $$\sNOR(𝕍)(• ; •) =𝕍(N ⊗ •; •)$$
    forms a \promonad{}. Moreover, the Kleisli category of this \promonad{} is a normal \symmetricProduoidal{} category with the following \profunctors{}.
    \begin{enumerate}
      \item $\sNOR(𝕍)(X ⊗_N Y; Z) = 𝕍(N ⊗ X ⊗ Y; Z)$; 
      \item $\sNOR(𝕍)(X ⊲_𝖭 Y; Z) = 𝕍((N ⊗ X) ◁ (N ⊗ Y); Z)$; and
      \item $\sNOR(𝕍)(I_𝖭; X) =\sNOR(𝕍)(𝖭_𝖭; X) = 𝕍(𝖭; X)$.
    \end{enumerate}
\end{theorem}

\begin{theorem}\label{th:sym:freeNormalProduoidal}
  Normalization determines an adjunction between \symmetricProduoidal{} and normal \symmetricProduoidal{} categories,
  $$\sNOR ፡ \symProDuo \rightleftharpoons \nSymProduo \colon 𝓤.$$
  That is, $\sNOR(𝕍)$ is the free normal \symmetricProduoidal{} category over $𝕍$.
\end{theorem}

\subsection{Bibliography}
Garner and López-Franco contributed a partial normalization procedure for duoidal categories \cite{garner16}, all of the credit for this elegant idea goes there. 

We contribute its produoidal counterpart. The reader could think that this is an automatic process: extending an argument by Day \cite{day}, \produoidalCategories{} could be understood as closed \duoidalCategories{} in some sense. However, we show that there are some technical differences that make this case important: in the work of Garner and López-Franco, not every \duoidalCategory{} has a normalization, and this prevents us from constructing a monad. We prove that produoidal normalization is always possible and defines an idempotent monad.

\newpage
\section{Monoidal Lenses}
\label{sec:monoidal-lenses}
\label{sec:monoidalContexts}

\MonoidalLenses{} -- the name we give to \monoidalContexts{} \cite{pickering17:profunctoroptics,riley2018categories,ClarkeRoman20:ProfunctorOptics} -- formalize the notion of an incomplete morphism in a monoidal category. The category of \monoidalLenses{} will have a rich algebraic structure: we shall be able to still compose contexts sequentially and in parallel and, at the same time, we shall be able to fill a context using another monoidal context. Perhaps surprisingly, then, the category of \monoidalLenses{} is not even monoidal.

We justify this apparent contradiction in terms of profunctorial structure: the category is not monoidal, but it does have two \promonoidal{} structures that precisely represent sequential and parallel composition. These structures form a \normalProduoidalCategory{}.
In fact, we show it to be the normalization of the \produoidal{} category of \splicedMonoidalArrows{}. This section constructs explicitly this \normalProduoidalCategory{} of \monoidalLenses{}.

\subsection{The Category of Monoidal Lenses}
A \monoidalLens{} -- an element of type $\pbiobj{X}{Y} → \pbiobj{A}{B}$ -- represents a process from $A$ to $B$ with a hole admitting a process from $X$ to $Y$. In this sense, \monoidalLenses{} are similar to \splicedMonoidalArrows{}. The difference with \splicedMonoidalArrows{} is that \monoidalLenses{} allow for communication to happen to the left and to the right of this hole.

\begin{definition}[Monoidal lens]
  \defining{linkmonoidalcontext}{}\defining{linkMonoidalLens}{}\defining{linkMonoidalContext}{}\label{def:monoidalcontext}
  \label{def:monoidallens}
  Let $(ℂ,⊗,I)$ be a \monoidalCategory{}. \emph{Monoidal lenses} are the elements of the \profunctor{}
  $$\mLens\left({\biobj{X}{Y}}; \biobj{A}{B}\right) = ∫^{M_1,M_2} ℂ(A;M_1⊗X⊗M_2) × ℂ(M_1⊗Y⊗M_2;B).$$
\end{definition}

In other words, a \emph{monoidal lens} from $A$ to $B$, \emph{with a hole} from $X$ to $Y$, is an equivalence class consisting of a pair of objects $M, N ∈ \obj{ℂ}$ and a pair of morphisms $f ∈ ℂ(A; M⊗X⊗N)$  and $g ∈ ℂ(M⊗Y⊗N;B)$, quotiented by dinaturality of $M$ and $N$.

\begin{definition}
  \defining{linkMonoidalLens}{}
  Let $(ℂ,⊗,I)$ be a \monoidalCategory{}.
  Its \normalProduoidalCategory{} of \emph{monoidal lenses}, $\mLens(ℂ)$, has objects formed by pairs, $\mLens(ℂ)_{obj} = (ℂ^{op} × ℂ)_{obj}$, and is defined by the following \profunctors{}. 
  \begin{enumerate}
    \item Morphisms are diagrams with a single typed hole. $$\mLens{ℂ} \left( \biobj{X}{Y} ; \biobj{A}{B}\right) =  ∫^{M_1,M_2} ℂ(A;M_1⊗X⊗M_2) × ℂ(M_1⊗Y⊗M_2;B),$$
    \item \SequentialSplits{} are diagrams with a pair of sequential holes. 
    $$\begin{aligned} \mLens{ℂ}(\biobj{X}{Y} ⊲ \biobj{X'}{Y'}; \biobj{A}{B};) =  ∫^{M_1,M_2}  &ℂ(A;M_1⊗X⊗M_2) × ℂ(M_1⊗Y⊗M_2;\\& M_3 ⊗X'⊗M_4) × ℂ(M_3⊗Y'⊗M_4; B);\end{aligned}$$
    \item \ParallelSplits{} are diagrams with a pair of parallel holes. 
    $$\begin{aligned}
      \mLens{(ℂ)}(\biobj{X}{Y} ⊗ \biobj{X'}{Y'}; \biobj{A}{B}) = ∫^{M₁,M₂,M₃} &ℂ(A; M₁ ⊗ X ⊗ M₂ ⊗ X' ⊗ M₃) × \\&  ℂ(M₁  ⊗ Y ⊗ M₂ ⊗ Y' ⊗ M₃; B)\end{aligned}$$
    \item Units are complete diagrams with no holes,
    $\mLens{(ℂ)}(𝖭;\biobj{X}{Y}) = ℂ(X;Y).$
  \end{enumerate}
  Reading the profunctorial notation can be unenlightening. 
  We provide the incomplete string diagrams for these profunctors in \Cref{fig:monoidalLenses} \cite{openDiagrams}.
  \begin{figure}[ht]
    \centering
    \includegraphics[scale=0.35]{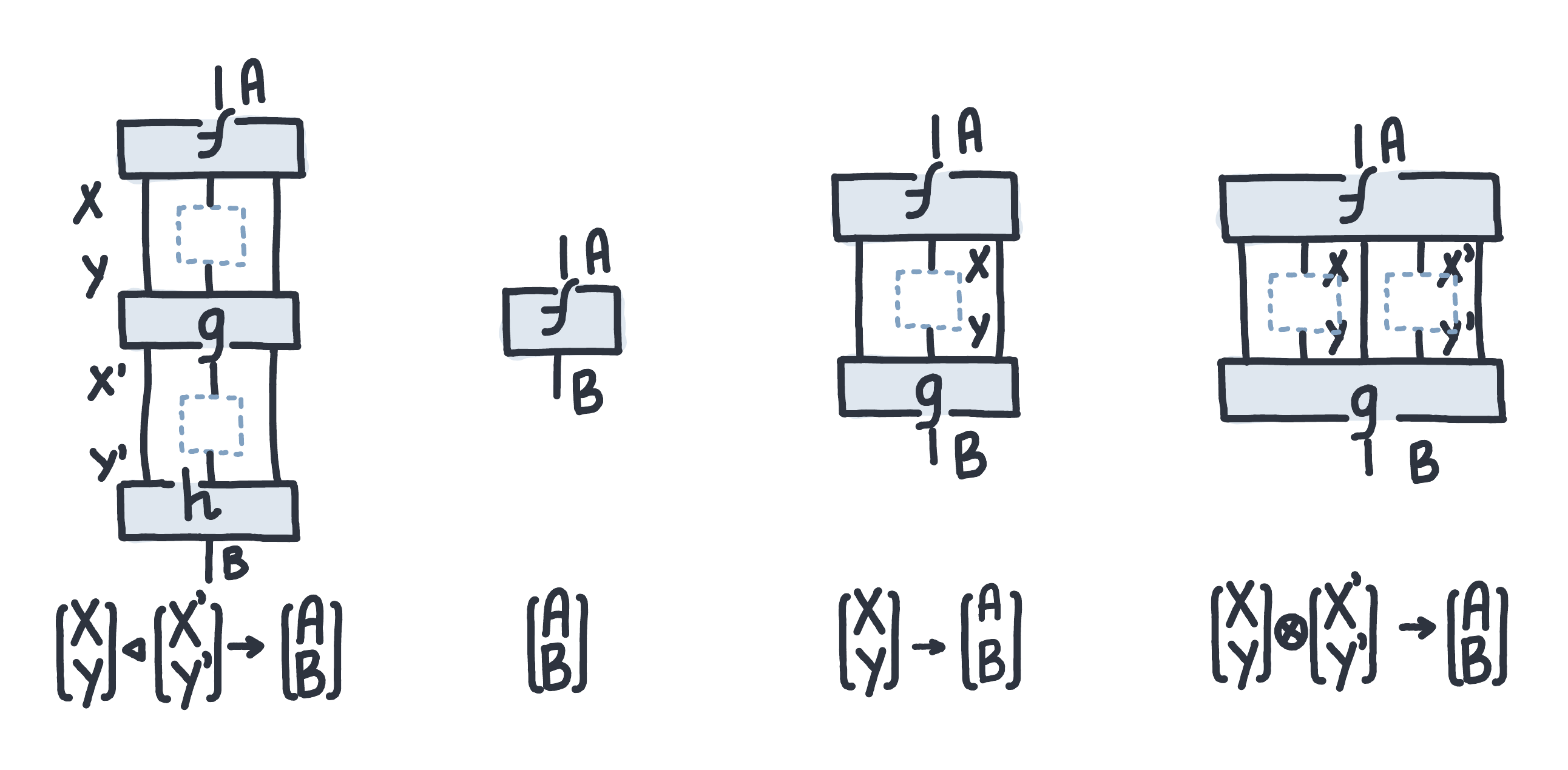}
    \caption{Monoidal lenses.}
    \label{fig:monoidalLenses}
  \end{figure}
\end{definition}

\begin{theorem}
  \label{prop:MonoidalContextProtensor}
  \label{th:monoidalContextsAreANormalization}
  The category of \monoidalLenses{} forms a normal \produoidalCategory{} with its units, sequential and parallel joins.
  \MonoidalLenses{} are the free normalization of the cofree \produoidalCategory{} over a category.
  In other words, \monoidalLenses{} are the normalization of \splicedMonoidalArrows{},
  $$\mLens{(ℂ)} ≅ \NOR(\mSplice{ℂ}).$$
\end{theorem}
\begin{proof}
The core of this result is in \Cref{th:freeNormalProduoidal}, which says that the normalization procedure yields the free normalization over a \produoidalCategory{}. It is only left to check that this \produoidalCategory{} of \monoidalLenses{} that we have explicitly constructed in this section is precisely the normalization of the \hyperlink{linkProduoidalSplice}{produoidal category of spliced arrows}. We do so for morphisms, the rest of the proof is similar; the proof shows that $\mSplice{ℂ} \left(  𝖭 ⊗ \biobj{X}{Y} ⊗ 𝖭 ; {\biobj{A}{B}} \right)$, the normalization of spliced monoidal arrows, is isomorphic to monoidal lenses, $\mLens(ℂ)\left( {\biobj{A}{B}};{\biobj{X}{Y}} \right)$. We employ the Yoneda lemma on both $V$ and $V'$.
  \begin{align*}
    & ∫^{U, V, U', V' ∈ ℂ} 
      ℂ(A; U ⊗ X ⊗ U') ×
      ℂ(V ⊗ Y ⊗ V'; B) × 
      ℂ\left(U ; V \right) ×
      ℂ\left(U' ; V' \right)
    & ≅ \\
    & ∫^{U,U' ∈ ℂ} 
      ℂ\left( A ; U ⊗ X ⊗ U'\right) × 
      ℂ(U ⊗ Y ⊗ U' ; B)
  \end{align*}
  The rest of the profunctors follow a similar reasoning.
\end{proof}

\subsection{Symmetric Monoidal Lenses}
\defining{linksymmetricmonoidalcontext}{}
\defining{linkSymmetricMonoidalLens}{}
\label{sec:monoidallenses}

A \symmetricMonoidalLens{} of type $\smLens{(ℂ)}(\biobj{X}{Y}; \biobj{A}{B})$ represents a process in a \symmetricMonoidalCategory{} with a hole admitting a process from $X$ to $Y$. 

\SymmetricMonoidalLenses{} are \monoidalLenses{}, but we stop caring where the hole is.
Again, the category of \symmetricMonoidalLenses{} has a rich algebraic structure; and again, most of this structure exists only virtually in terms of profunctors. In this case, though, the monoidal tensor \emph{does} indeed exist: contrary to \monoidalLenses{}, \symmetricMonoidalLenses{} form also a \monoidalCategory{}.
This is perhaps why applications of \monoidalLenses{} have grown popular in recent years~\cite{riley2018categories}, with applications in decision theory \cite{ghani:compositionalgametheory2018}, supervised learning \cite{cruttwell22:learning,fong19:lenses} and most notably in functional data accessing \cite{kmett12:lenslibrary,pickering17:profunctoroptics,boisseau2018you,ClarkeRoman20:ProfunctorOptics}.
The \promonoidal{} structure of optics was ignored, even when, after now identifying for the first time its relation to the monoidal structure of optics, we argue that it could be potentially useful in these applications: e.g. in multi-stage decision problems, or in multi-stage data accessors.

This section explicitly constructs the normal \symmetricProduoidalCategory{} of \emph{symmetric monoidal lenses}.
We describe it for the first time by a universal property: it is the free symmetric normalization of the cofree \produoidalCategory{}.

\begin{definition}
  \label{def:symmetricmonoidallens}
  \defining{linkSymmetricMonoidalLens}{}
  Let $(ℂ,⊗,I)$ be a symmetric \monoidalCategory{}.
  \emph{Symmetric monoidal lenses} are the elements of the profunctor
  $$\smLens{ℂ}\left(\biobj{X}{Y}; \biobj{A}{B} \right) = ∫^M ℂ(A ; M ⊗ X) × ℂ(M ⊗ Y ; B).$$
  In other words, a \emph{symmetric monoidal lens} from $A$ to $B$, \emph{with a hole} from $X$ to $Y$, is an equivalence class consisting of a pair of objects $M ∈ ℂ_{obj}$ and a pair of morphisms $f ∈ ℂ(A; M ⊗ X)$  and $g ∈ ℂ(M ⊗ Y;B)$, quotiented by \dinaturality{} of $M$.
\end{definition}
\begin{definition}
  \defining{linkSymmetricMonoidalLens}{}
  Let $(ℂ,⊗,I)$ be a \symmetricMonoidalCategory{}.
  Its \normalSymmetricProduoidalCategory{} of \emph{symmetric monoidal lenses}, $\smLens{(ℂ)}$, has objects formed by pairs, $\smLens{(ℂ)}_{obj} = (ℂ^{op} × ℂ)_{obj}$, and is defined by the following \profunctors{}. 
  \begin{enumerate}
    \item Morphisms are diagrams with a single typed hole. 
    $$\smLens{(ℂ)} \left(\biobj{X}{Y} ; \biobj{A}{B} \right) =  ∫^{M} ℂ(A;M ⊗X) × ℂ(M⊗Y;B),$$
    \item \SequentialSplits{} are diagrams with a pair of sequential holes. 
    $$\begin{aligned} 
      \smLens(ℂ)( \biobj{X}{Y} ⊲ \biobj{X'}{Y'} ; \biobj{A}{B}) =  ∫^{M_1,M_2}  &ℂ(A;M_1⊗X) × ℂ(M_1⊗Y;\\& M_2 ⊗X') × ℂ(M_2 ⊗ Y'; B);
    \end{aligned}$$
    \item \ParallelSplits{} are diagrams with a hole encompassing parallel wires. 
    $$\begin{aligned}
      \smLens(ℂ)(\biobj{X}{Y} ⊗ \biobj{X'}{Y'}; \biobj{A}{B}) = ∫^{M} &ℂ(A; M ⊗ X ⊗ X') ×  ℂ(M ⊗ Y ⊗ Y'; B).
    \end{aligned}$$
    \item Units are complete diagrams with no holes,
    $\smLens(ℂ)(𝖭; \biobj{A}{B}) = ℂ(A;B).$
  \end{enumerate}
  Reading the profunctorial notation can be unenlightening.
  We provide the incomplete string diagrams for these profunctors in \Cref{fig:monoidalLenses} \cite{openDiagrams}.
  The only substantial difference with \monoidalLenses{} is that we do not need to keep track of where the hole is placed.
  \begin{figure}[ht]
    \centering
    \includegraphics[scale=0.35]{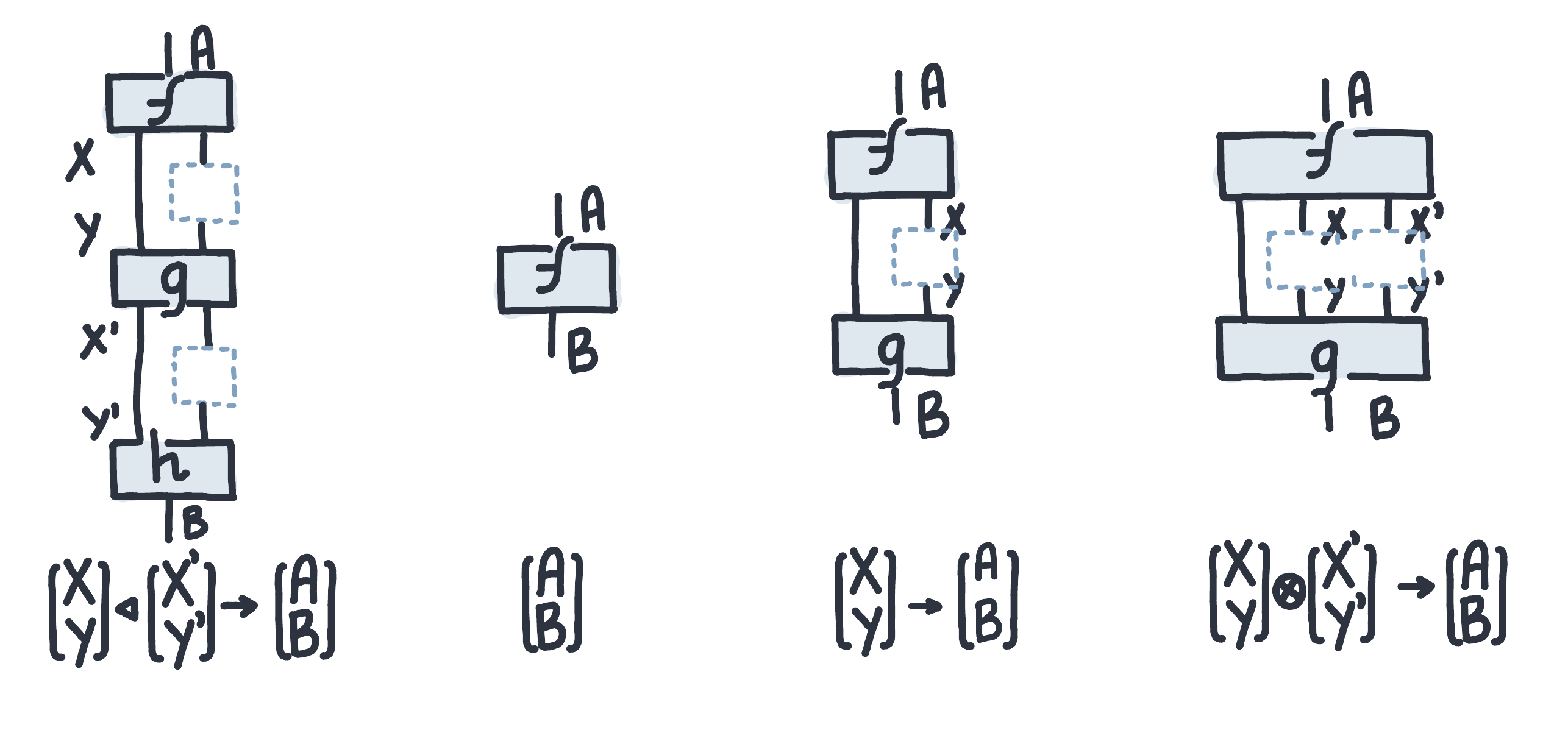}
    \caption{Symmetric monoidal lenses.}
    \label{fig:symmetricMonoidalLenses}
  \end{figure}

  The term ``\monoidalLenses{}'' has usually been reserved for the morphisms of this category; in the literature, the sequential splits  -- the lenses with multiple holes -- get a name from their distinctive shape: these are \emph{combs} or \emph{quantum combs} \cite{coeckeFS16}.
\end{definition}

\subsection{Towards Message Theories}
Lenses, or \emph{combs}, can be interpreted as incomplete morphisms, but also as morphisms that send and receive resources. The next chapter will exploit this intuition.

\begin{remark}[Session notation for combs]
  We will write $A^{∘} = \pbiobj{A}{I}$ and $B^{•} = \pbiobj{I}{B}$ for the objects of the symmetric \produoidalCategory{} of lenses that have a monoidal unit as one of its objects. Thanks to $A^{∘} ⊗ B^{•} = \pbiobj{A}{B}$, these are enough to express all objects.
\end{remark}

\begin{proposition}
  \label{prop:sessionNotation}
  Let $(ℂ,⊗,I)$ be a symmetric \monoidalCategory{}. There exist monoidal functors 
  $$(\Send{}) ፡ ℂ → \smLens(ℂ), \quad\mbox{ and }\quad (\Get{}) \colon ℂ^{op} → \smLens(ℂ).$$
  Moreover, they satisfy the following properties definitionally:
  $ℂ( \Get{A} ⊲ \Get{B} ; •) ≅ ℂ(\Get{A} ⊗ \Get{B}; •)$; 
  $\Send{(A ⊗ B)} = \Send{A} ⊗ \Send{B}$;
  $ℂ(\Send{A} ⊲ \Send{B}; •) ≅ ℂ(\Send{A} ⊗ \Send{B}; •)$; 
  $\Get{(A ⊗ B)} = \Get{A} ⊗ \Get{B}$; and 
  $ℂ(\Send{A} ⊲ \Get{B}; •) ≅ ℂ(\Send{A} ⊗ \Get{B}; •)$.
\end{proposition}
\begin{proof}  
  We define $\Send{f} = (f ⨾ □ ⨾ \id_I)$ and $\Get{g} = (\id_I ⨾ □ ⨾ g)$, and then check that compositions and tensoring of morphisms are compatible with composition and tensoring of \monoidalLenses{}, this is straightforward.
  Moreover, we can see that, by definition, 
  \begin{align*}
    & \Send{(A ⊗ B)} = \left(\biobj{A ⊗ B}{I}\right) = \left(\biobj{A}{I}\right) ⊗ \left(\biobj{B}{I}\right) = \Send{A} ⊗ \Send{B}, \quad\mbox{ and } \\ 
    &\Get{(A ⊗ B)} = \left(\biobj{I}{A ⊗ B}\right) = \left(\biobj{I}{A}\right) ⊗ \left(\biobj{I}{B}\right) = \Get{A} ⊗ \Get{B}.
  \end{align*}

  This proof appears with a different language in the work of Riley \cite[Proposition 2.0.14]{riley2018categories}. In fact, there, the combined identity-on-objects functor $({∘} × {•}) ፡ ℂ × ℂ^{op} → \smLens(ℂ)$ is shown to be monoidal.
\end{proof}

\begin{example}\label{ex:one-time-pad}
  Let us give a first example of how to employ combs for the description of concurrent protocols.
  Broadbent and Karvonen \cite{broadbent22:crypto} propose a formalization of the \emph{one-time pad} encryption protocol in a \symmetricMonoidalCategory{} endowed with a Hopf algebra with an integral.
  \begin{figure}[ht]
    \centering
    \includegraphics[scale=0.35]{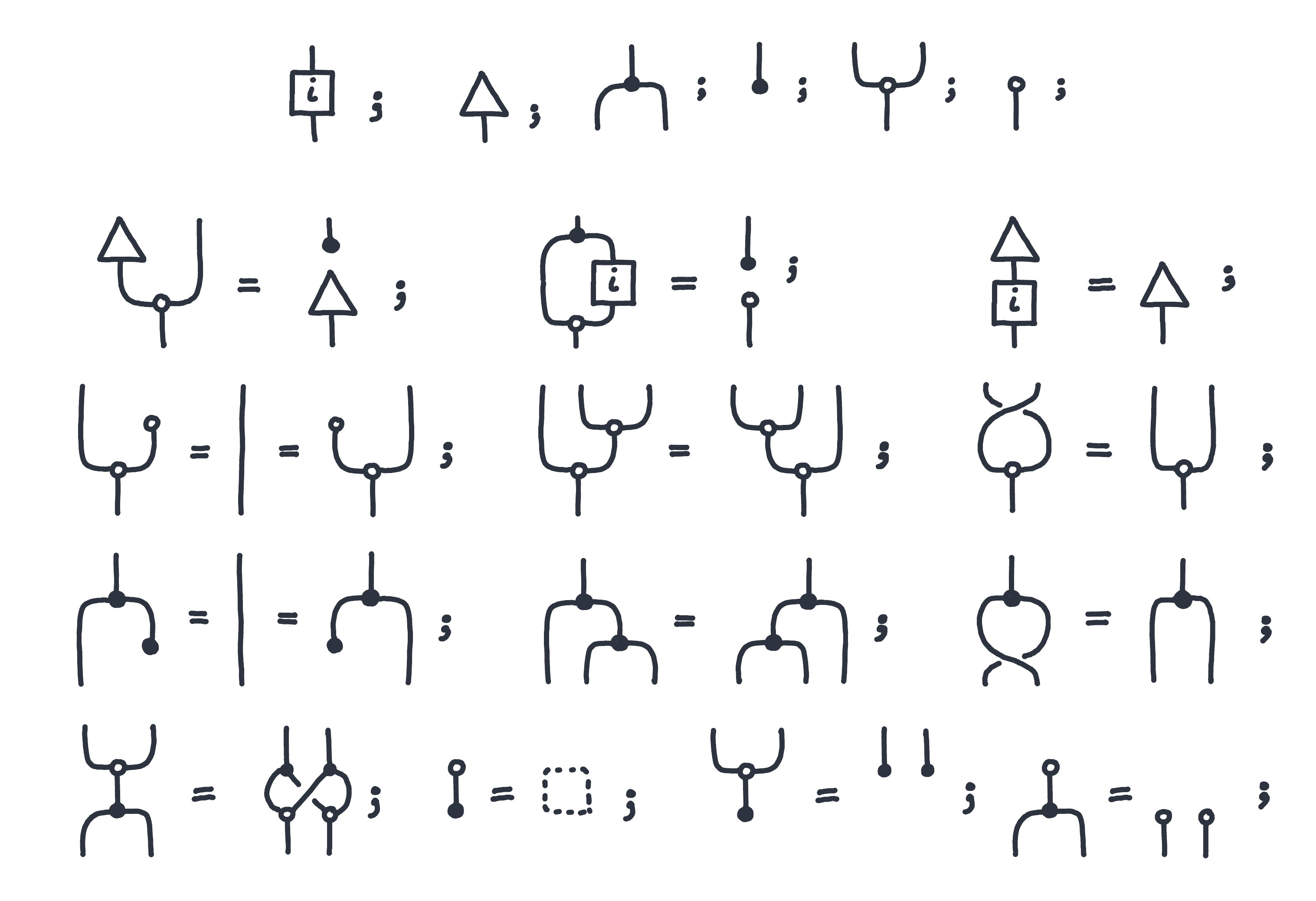}
    \caption{Theory of a Hopf algebra with an integral.}
    \label{fig:hopf-algebra-theory}    
  \end{figure}

  \begin{definition}
    A Hopf algebra with an integral, $(X,\iconbcm,\iconbcu,\iconwm,\iconwu,i,d)$, is a commutative bialgebra endowed with an \emph{antipode} map $i ፡ X → X$, representing inversion; and endowed with an \emph{integral map}, $d ፡ I → X$, representing a non-determined value, or pure noise. These must satisfy the equations in \Cref{fig:hopf-algebra-theory}.
  \end{definition}

  The one-time pad is a mathematically secure encryption technique. It works as follows: \emph{(i)} the two parties communicating -- say, \Alice{} and \Bob{} -- start by preparing some random bits and sharing them; \emph{(ii)} when the message is ready, \Alice{} applies bitwise XOR with the random bits to encrypt the message, and then broadcasts the encrypted message -- a potential attacker, \Eve{}, will receive this encrypted message; \emph{(iii)} finally, \Bob{} receives the encrypted message and applies again bitwise XOR with the random bits to decrypt the message (\Cref{fig:one-time-pad}).
  \begin{figure}[ht]
    \centering
    \includegraphics[scale=0.35]{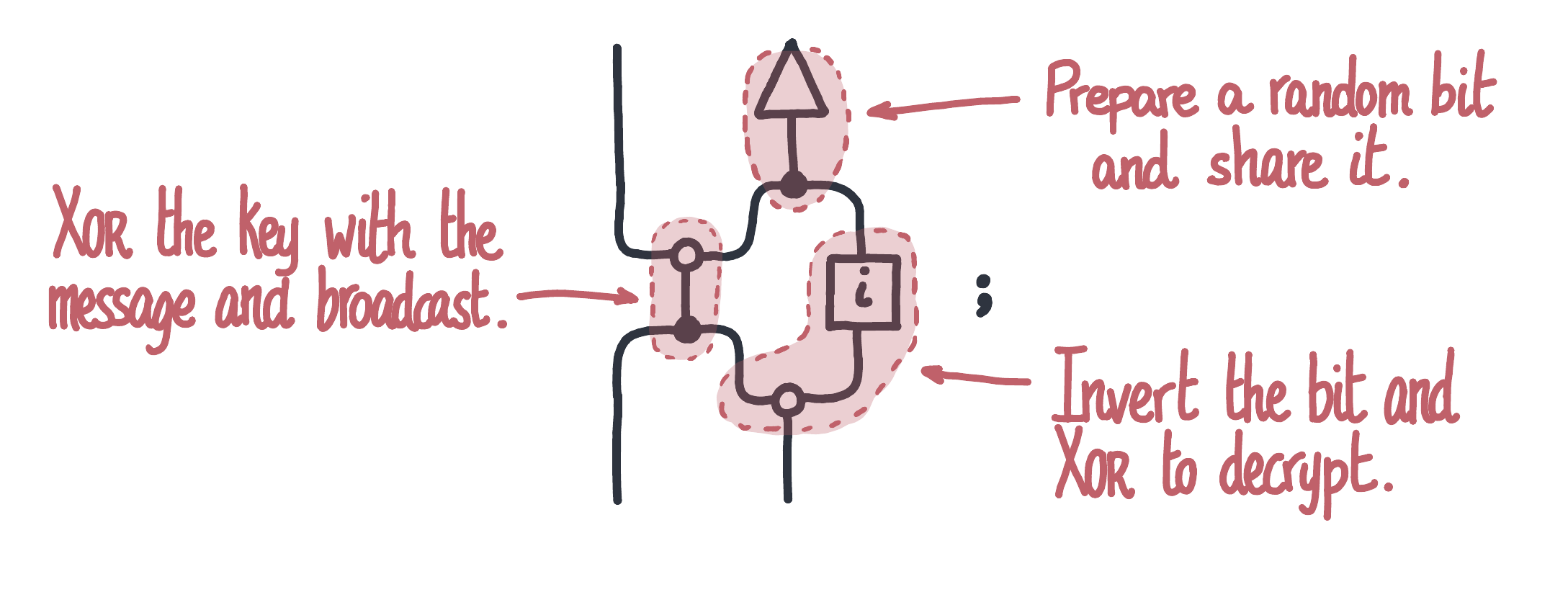}
    \caption{Description of the one-time pad.}
    \label{fig:one-time-pad}    
  \end{figure}

  \begin{proposition}
    \label{prop:correctness}
    The one-time pad is \emph{secure}, meaning that it is equal to the process that sends a message from \Alice{} to \Bob{} and outputs random noise through the attacker's channel.
  \end{proposition}
  \begin{proof}
    We repeat the proof from Broadbent and Karvonen \cite{broadbent22:crypto}.
    We employ string diagrams of \symmetricMonoidalCategories{}, in \Cref{fig:one-time-pad-correctness}, to show that the morphism is equal to an identity tensored by the integral of the Hopf algebra. 
    \begin{figure}[ht]
      \centering
      \includegraphics[scale=0.35]{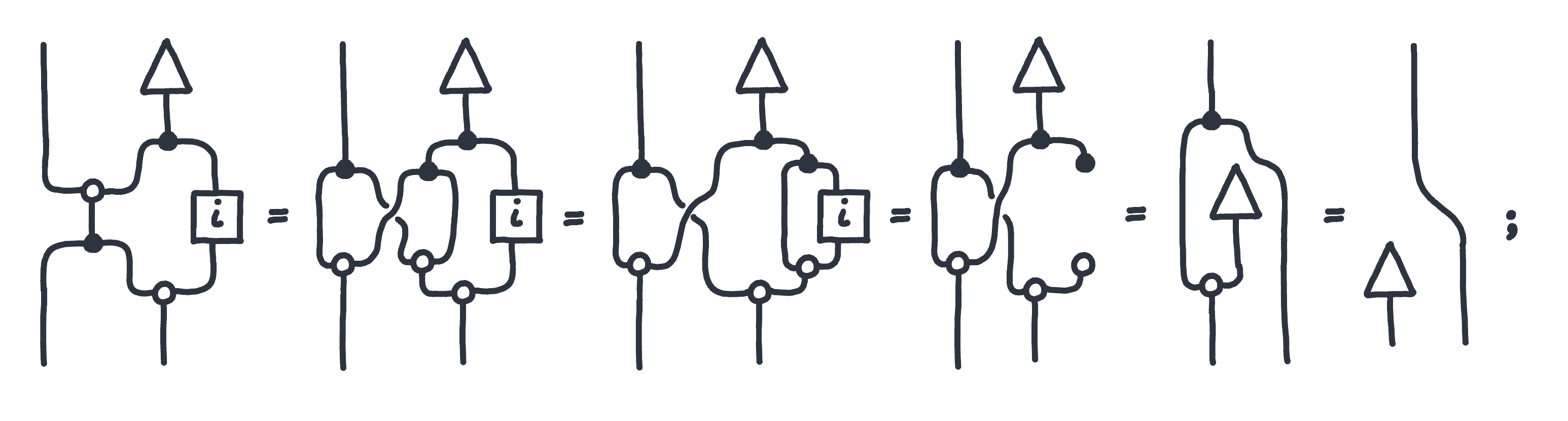}
      \caption{Correctness of the one-time pad.}
      \label{fig:one-time-pad-correctness}
    \end{figure}
  \end{proof}
\end{example}

The interesting part comes when we want to split the morphism into its different constituents: there should be a \emph{stage} where the three actors play; \Alice{} does not control the fact that the encrypted message will be broadcast; \Eve{}, the attacker, can only attack at the end; \Bob{} will need to keep a bit in memory. These considerations are part of the problem statement the one-time pad is solving. It is easy to come up with a morphism that connects an input to an output: the problem the one-time pad is solving is to do so on a stage that has been preset.

The components of the one time pad are not simply morphisms of a \monoidalCategory{} (\Cref{fig:one-time-pad-components}). They must be understood as \monoidalLenses{}. After this section, we can declare that a possible typing for these components is the following:
\begin{enumerate}
  \item $\Alice ፡ 𝖭 → \pbiobj{X ⊗ X}{X}$;
  \item $\Bob ፡ X^{•} ⊲ X^{∘} → \pbiobj{I}{X}$;
  \item $\Eve ፡ 𝖭 → \pbiobj{X}{X}$;
  \item $\mathsf{Stage} ፡ X^{∘} ⊲ X^{•} ⊲ X^{•} ⊲ X^{∘} ⊲ X^{•} ⊲ X^{•} ⊲ X^{∘} ⊲ X^{∘} → \pbiobj{X}{X ⊗ X}$;
\end{enumerate}

\begin{figure}[ht]
  \centering
  \includegraphics[scale=0.35]{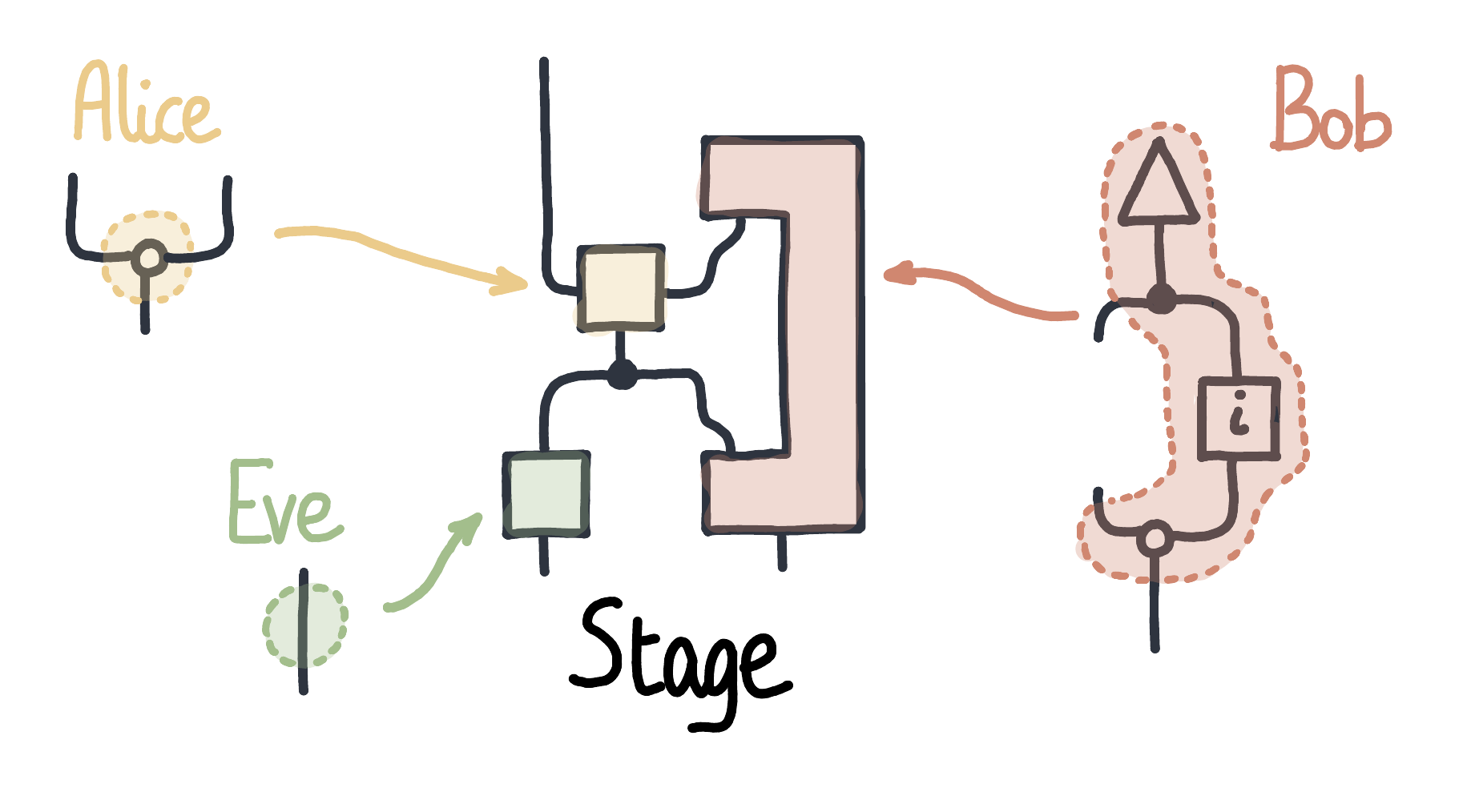}
  \caption{Components of the one-time pad.}
  \label{fig:one-time-pad-components}
\end{figure}

Still, at this stage it is not easy to talk about message passing with this syntax. It is true that the new produoidal types can track all the exchanges that happen along a boundary; but these types are tedious -- see, for instance, the long type of $\mathsf{Stage}$ -- and it is not clear how to compose them. What we are missing is a combinatorial description of the different ways we can combine elements of this produoidal algebra.

This is what the next chapter will solve: we will propose a combinatorial algebra of message passing and show that lenses, the normalized cofree produoidal algebra over a \symmetricMonoidalCategory{} have a second universal property -- they also constitute the free \messageTheory{}.

\subsection{Bibliography}
Lenses~\cite{foster07:bidirectional} are a notion of bidirectional transformation that can be cast in arbitrary monoidal categories.
The first mention of monoidal lenses separate from their classical database counterparts \cite{johnson2012lenses} is due to Pastro and Street \cite{pastro07}, who identify them as an example of a \promonoidalCategory{}. However, it was with a different monoidal structure \cite{riley2018categories} that they became popular in recent years, spawning applications not only in bidirectional transformations \cite{foster07:bidirectional} but also in functional programming \cite{pickering17:profunctoroptics,ClarkeRoman20:ProfunctorOptics}, open games \cite{ghani:compositionalgametheory2018}, polynomial functors \cite{niuspivak:polynomial} and quantum combs \cite{hefford_combs}.
Relating this monoidal category of lenses with the previous \promonoidalCategory{} of lenses was an open problem; and the \promonoidal{} structure was mostly ignored in applications.
We solve this problem, proving that lenses are a universal normal \symmetricProduoidal{} category (the \symmetricMonoidalLenses{}), which endows them with a novel algebra and a novel universal property.
This also extends work on the relation between \emph{incomplete diagrams}, \emph{comb-shaped diagrams}, and \emph{lenses} \cite{roman2020,openDiagrams}.

Lenses themselves have been applied to protocol specification~\cite{videlacapucci22}.
Spivak \cite{spivak13} also discusses the \multicategory{} of \emph{wiring diagrams}, later used for incomplete diagrams \cite{patterson21:wiringdiagrams} and related to lenses \cite{schultz20:dynamical}; we conjecture that this \multicategory{} of wiring diagrams is precisely the \produoidalCategory{} of lenses, once we stop tracking dependencies explicitly.

\begin{conjecture}
  \label{conj:wiringdiagrams}
  Each physical \produoidalCategory{} induces a \multicategory{} given by its physical lax tensor (\Cref{sec:physicallaxtensor}).
  The \multicategory{} of wiring diagrams \cite{spivak13} of \symmetricMonoidalCategories{} is the \multicategory{} induced by the \produoidalCategory{} of \monoidalLenses{}.
\end{conjecture}

The \promonoidalCategories{} we use can be seen as \multicategories{} with an extra coherence property. In this sense, we contribute the missing algebraic structure of the universal multicategory of \emph{wiring diagrams relative to a monoidal category}.

\newpage

\clearpage{}%

\chapter{Monoidal Message Passing}
\label{chapter:monoidal-message-passing}

\section*{Monoidal Message Passing}
This chapter develops message passing in \monoidalCategories{} following the theory of context we just constructed. We have already defined what incomplete morphisms in \monoidalCategories{} are: we will now study the structure of all their possible compositions. This includes not only the obvious operations of composition but any possible \emph{wiring} of a diagram that could combine them while respecting the acyclicity of string diagrams.

Studying seriously the combinatorial structure of \stringDiagram{} composition arrives at the same conclusion as axiomatizing a naive theory of message passing: \messageTheories{}. Indeed, we prove that the polarized shufflings that describe \stringDiagram{} composition have as algebras precisely the \messageTheories{}.

\Cref{sec:message-theories} introduces our minimalistic theory of message passing, with axioms that should hopefully be  acceptable to any reader. \Cref{sec:physical-monoidal-multicategories-shufflings} starts developing the categorical semantics for \messageTheories{}, based on \physicalMonoidalMulticategories{} (a variant of \duoidalCategories{}); it then shows that shufflings from the free \physicalMonoidalMulticategory{}. \Cref{sec:polarization} provides the second ingredient for this categorical semantics: polarization. We then combine both ingredients in \Cref{sec:polar-shuffles}: \polarShuffles{} form the combinatorial structure that combines incomplete string diagrams; \messageTheories{} are precisely the algebras of \physicalMonoidalMulticategory{} of \polarShuffles{}. This chapter ends with an adjunction between \messageTheories{} and \symmetricMonoidalCategories{}, which ensures that we can construct a free \messageTheory{} on top of any \symmetricMonoidalCategory{}.

\clearpage{}%
\section{Message Theories}
\label{sec:message-theories}

\subsection{Message Theories}
Message passing requires the interplay of at least two mathematical structures: the ability to \emph{interleave} events in time and the ability to connect a \emph{sender and a receiver}.
Let us propose a minimally axiomatized algebra of interleaving and sending/receiving: interleaving will correspond to a normal duoidal algebra, and sending/receiving will correspond to polarization.

\begin{definition}
  \defining{linkMessageTheory}{}
  \label{def:messagetheory}
  A \emph{message theory} $𝕄$ consists of a set of types, $𝕄_{obj}$ with extra structure: a \emph{send/receive session type} is a polarized list of types; for each session type, we have a collection of \emph{sessions} with that type,
  $$𝕄( X₁^{‽₁} , … , Xₙ^{‽ₙ} ), \mbox{ for each } X₁, …, Xₙ ∈ 𝕄_{obj}, \mbox{ and each polarization } ‽ᵢ ∈ \{{∘},{•}\}.$$
  A \messageTheory{} must contain operations for \emph{(i)} binary shuffling, \emph{(ii)} and nullary shuffling, \emph{(iii)} linking a sent message to immediately receive it, and \emph{(iv)} spawning a channel that receives a message and sends it immediately.
  \begin{enumerate}
    \item $\SHF_{σ} ፡ 𝕄(Γ) × 𝕄(Δ) → 𝕄({σ}(Γ,Δ))$, shuffling two processes;
    \item $\NOP ፡ 𝕄()$, a no-operation, doing nothing;
    \item $\LNK_x^{Γ;Δ} ፡ 𝕄(Γ, X^{•}, X^{∘}, Δ) → 𝕄(Γ, Δ)$, linking send to receive;
    \item $\SPW_x^{Γ;Δ} ፡ 𝕄(Γ, Δ) → 𝕄(Γ, X^{∘}, X^{•}, Δ)$, a receive to send channel.
  \end{enumerate}
  \MessageTheories{} may be better understood in the notation of a logic, as in \Cref{fig:type-message}. Types form a free polarized monoid; each term describes a possible communication protocol.
  \begin{figure}[ht]
    \begin{minipage}{0.20\textwidth}
      \begin{mathpar}
          \inferrule*[Right=(shf${}_{\sigma}$)]
          {Γ \qquad Δ}
          {[Γ,Δ]_{σ}}
      \end{mathpar}
      \end{minipage}
    \begin{minipage}{0.24\textwidth}
    \begin{mathpar}
        \inferrule*[Right=(lnk)]
        {Γ, X^{•}, X^{∘}, Δ}
        {Γ, Δ}
    \end{mathpar}
    \end{minipage}
    \begin{minipage}{0.30\textwidth}
    \begin{mathpar}
        \inferrule*[Right=(spw)]
        {Γ, Δ}
        {Γ, X^{∘}, X^{•}, Δ}
    \end{mathpar}
    \end{minipage}
    \begin{minipage}{0.15\textwidth}
      \begin{mathpar}
          \inferrule*[Right=(nop)]
          {\ }
          {\varepsilon}
      \end{mathpar}
      \end{minipage}
    \caption{Type-theoretic presentation of a message theory.}
    \label{fig:type-message}
  \end{figure} 

  A message theory must satisfy the following axioms: \emph{(i)} shuffles compose as in their symmetric \malleableMulticategory{}, where we write $(σ ⨾₁ τ)$ to be the same by associativity as $(τ' ⨾₂ σ')$, we write $(\ast)$ for the trivial shuffle, and we write $\tilde{σ}$ for the symmetric counterpart of $σ$; \emph{(ii)} linking is natural with respect to the shuffles; \emph{(iii)} spawning is natural with respect to the shuffles; \emph{(iv)} linking is dual to spawning; and \emph{(v)} independent linkings and spawnings commute.
  \begin{enumerate}
    \item[(1a)] $\SHF_{τ}(\SHF_{σ}(m₁m₂),m₃) = \SHF_{σ'}(m₁,\SHF_{τ'}(m₂,m₃))$;
    \item[(1b)] $\SHF_{\ast}(m, \NOP) = m$;
    \item[(1c)] $\SHF_{σ}(m₁,m₂) = \SHF_{\tilde{σ}}(m₂,m₁)$;
    \item[(2a)] $\SHF_{σ,τ}(\LNK_x^{Γ₁,Γ₂}(m₁), m₂) = \LNK_x^{Γ₁,Δ₁;Γ₂,Δ₂}(\SHF_{σ,x,τ}(m₁,m₂))$;
    \item[(2b)] $\SHF_{σ,τ}(m₁, \LNK_x^{Δ₁,Δ₂}(m₂)) = \LNK_x^{Γ₁,Δ₁;Γ₂,Δ₂}(\SHF_{σ,x,τ}(m₁,m₂))$;
    \item[(3a)] $\SHF_{σ,τ}(\SPW_x^{Γ₁,Γ₂}(m₁), m₂) = \SPW_x^{Γ₁,Δ₁;Γ₂,Δ₂}(\SHF_{σ,τ}(m₁,m₂))$;
    \item[(3b)] $\SHF_{σ,τ}(m₁, \SPW_x^{Δ₁,Δ₂}(m₂)) = \SPW_x^{Γ₁,Δ₁;Γ₂,Δ₂}(\SHF_{σ,τ}(m₁,m₂))$;
    \item[(4a)] $\LNK_{x}^{Γ,X^{∘};Δ}(\SPW_x^{Γ;X^{∘},Δ}(m)) = m$;
    \item[(4b)] $\LNK_{x}^{Γ;X^{•},Δ}(\SPW_x^{Γ,X^{•};Δ}(m)) = m$;
    \item[(5a)] $\LNK_{x}^{Γ₁;Γ₂YΓ₃}(\SPW_y^{Γ₁XΓ₂;Γ₃}(m)) = \SPW_y^{Γ₁Γ₂;Γ₃}(\LNK_x^{Γ₁;Γ₂Γ₃}(m))$;
    \item[(5b)] $\LNK_{y}^{Γ₁XΓ₂;Γ₃}(\SPW_x^{Γ₁;Γ₂YΓ₃}(m)) = \SPW_x^{Γ₁;Γ₂Γ₃}(\LNK_y^{Γ₁XΓ₂;Γ₃}(m))$;
    \item[(5c)] $\SPW_{x}^{Γ₁;Γ₂YΓ₃}(\SPW_y^{Γ₁Γ₂;Γ₃}(m)) = \SPW_y^{Γ₁XΓ₂;Γ₃}(\SPW_x^{Γ₁;Γ₂Γ₃}(m))$;
    \item[(5d)] $\LNK_{x}^{Γ₁;Γ₂Γ₃}(\LNK_y^{Γ₁XΓ₂;Γ₃}(m)) = \LNK_y^{Γ₁Γ₂;Γ₃}(\LNK_x^{Γ₁;Γ₂YΓ₃}(m))$.
  \end{enumerate}
  These axioms are again better understood in logic notation, as equations between derivations, see \Cref{fig:message-theory-axioms}.
  \begin{figure}[!ht]
    \centering
    \includegraphics[scale=0.3]{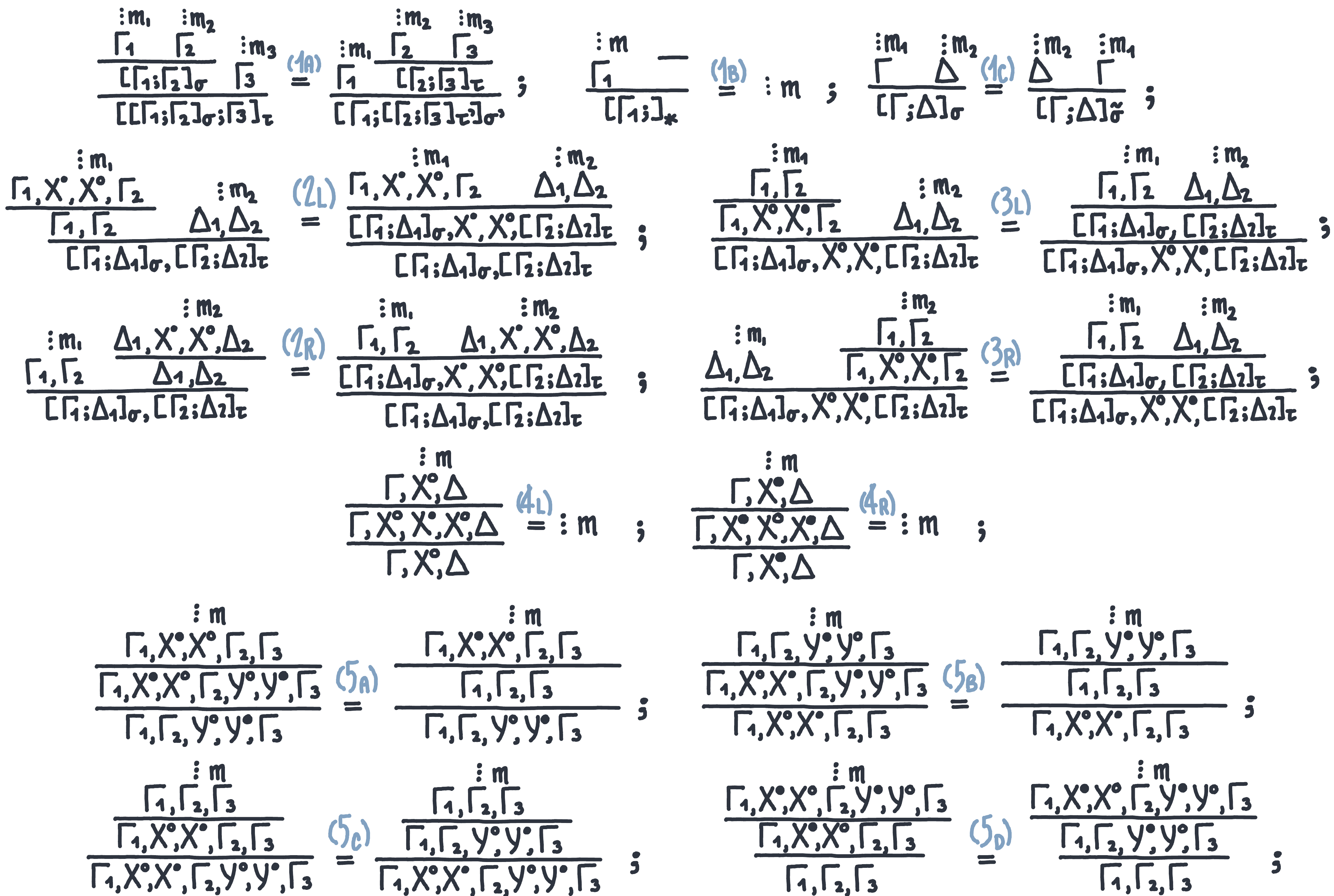}
    \caption{Axioms of a message theory.}
    \label{fig:message-theory-axioms}
  \end{figure}
\end{definition}

\begin{definition}
  \defining{linkMessageFunctor}{}
  \label{def:messagefunctor}
  A \emph{message functor} $F ፡ 𝕄 → ℕ$ between two \messageTheories{}, $𝕄$ and $ℕ$, is a function on objects $F_{obj} ፡ 𝕄_{obj} → ℕ_{obj}$ that extends to a family of functions on session sets,
  $$F ፡ 𝕄(X_1^{‽},…,X_n^{‽}) → ℕ(FX_1^{‽},…,FX_n^{‽}).$$
  This function must \emph{(i)} preserve shuffling, $F(\SHF_{σ}(f,g)) = \SHF_{σ}(Ff,Fg)$; \emph{(ii)} preserve spawning, $F(\SPW_x) =\SPW_{Fx}$; \emph{(iii)} connecting, $F(\COM_x(f)) = \COM_x(Ff)$; and \emph{(iv)} the no-operation, $F(\NOP) = \NOP$. \MessageTheories{} form a category $\mathbf{Msg}$ with message functors between them.
\end{definition}

\subsection{Properties of a Message Theory}
Our goal is to prove a coherence theorem for \messageTheories{}. Building up to this result, let us reason with \messageTheories{} to understand the basic properties that can be derived from the axioms.

\begin{proposition}
  \label{prop:shuffles-as-shuffles}
  \MessageTheories{} have a derived operation for each shuffling; moreover, these operations compose as in the \multicategory{} of shufflings.
  $$\mathsf{shuf}_{σ}^{Γ₁,…,Γₙ} ፡ 𝕄(Γ₁) × \dots × 𝕄(Γₙ) → 𝕄([Γ₁,…,Γₙ]_{σ})$$
\end{proposition}
\begin{proof}
  We have defined an operation for binary and nullary shufflings, and we have defined them to compose exactly as shufflings do. Because shufflings form a \malleableMulticategory{}, each n-ary shuffling can be recovered uniquely from the binary and nullary shufflings.
\end{proof}

\begin{proposition}
  \label{prop:soonerlater}
  A session can always send later and receive sooner, but it cannot send sooner nor receive later. Formally, there exist derived operations 
  \begin{align*}
    \mathsf{wait}_X^{Γ,Δ,Ψ} ፡ 𝕄(Γ,X^{•},Δ,Ψ) → 𝕄(Γ, Δ, X^{•},Ψ), \\
    \mathsf{rush}_X^{Γ,Δ,Ψ} ፡ 𝕄(Γ, Δ, X^{∘},Ψ) → 𝕄(Γ,X^{∘},Δ,Ψ).
  \end{align*}
\end{proposition}
\begin{proof}
  We can construct the derivation trees of both operations. They both spawn a new channel, shuffle its ends to the origin and target position, and they connect the channel.
  \begin{figure}[ht]
    \begin{minipage}{0.45\textwidth}
      \begin{mathpar}
          \inferrule*[Right=(shf)]
          {Γ, X^{•}, Δ, Ψ \\ \inferrule*[Right=(spw)]{\phantom{X}}{X^{∘}, X^{•}}}
          {\inferrule*[Right=(com)]
          {Γ, X^{•}, X^{∘}, Δ, X^{•}, Ψ}
          {Γ, Δ, X^{•}, Ψ}
          }
      \end{mathpar}
      \end{minipage}
      \begin{minipage}{0.45\textwidth}
        \begin{mathpar}
            \inferrule*[Right=(shf)]
            {Γ, Δ, X^{∘}, Ψ \\ \inferrule*[Right=(spw)]{\phantom{X}}{X^{∘}, X^{•}}}
            {\inferrule*[Right=(com)]
            {Γ, X^{∘}, Δ, X^{∘}, X^{•}, Ψ}
            {Γ,X^{∘}, Δ,  Ψ}
            }
        \end{mathpar}
    \end{minipage}
    \caption{Derivation of $\mathsf{wait}$ and $\mathsf{rush}$.}
    \label{fig:waitrush}
  \end{figure}

  The explicit construction is in \Cref{fig:waitrush}. Note how, thanks to polarization, it is not possible to use the same technique to \emph{send sooner} nor \emph{receive later}. In fact, we can reason by contradiction that sending sooner, $𝕄(Γ,Δ,X^{•},Ψ) → 𝕄(Γ, X^{•}, Δ, Ψ)$, is impossible: the only possible operations we can apply in an arbitrary message theory to an arbitrary session are \emph{shufflings with a spawned channel} and \emph{connections}; we require at least a connection to eliminate the $X^{•}$, but because the $X^{∘}$ must come from shuffling a spawned channel, the corresponding $X^{•}$ must be placed strictly after it.
\end{proof}

\begin{proposition}
  \label{prop:swaps-self-inverses}
  In particular, we can always swap the order of objects with the same polarity,
  \begin{align*}
    \mathsf{swap}^{Γ;X;Y;Δ}_{∘} ፡ 𝕄(Γ,X^{∘},Y^{∘},Δ) → 𝕄(Γ,Y^{∘},X^{∘},Δ), \\
    \mathsf{swap}^{Γ;X;Y;Δ}_{•} ፡ 𝕄(Γ,X^{•},Y^{•},Δ) → 𝕄(Γ,Y^{•},X^{•},Δ).
  \end{align*}
  These are self-inverses, forming not only braidings but symmetries in the underlying \monoidalCategory{} of positively or negatively polarized objects.
\end{proposition}
\begin{proof}
  Swaps are constructed from rushing and waiting (\Cref{prop:soonerlater}); explicitly,
  $$\mathsf{swap}^{Γ;X;Y;Δ}_{∘} = \mathsf{rush}_Y^{Γ,X^{∘},Ψ},\quad\mbox{ and }\quad\mathsf{swap}^{Γ;X;Y;Δ}_{•} = \mathsf{wait}_Y^{Γ,Y^{•},Ψ}.$$ 
  \begin{figure}[!ht]
    \centering
    \includegraphics[scale=0.4]{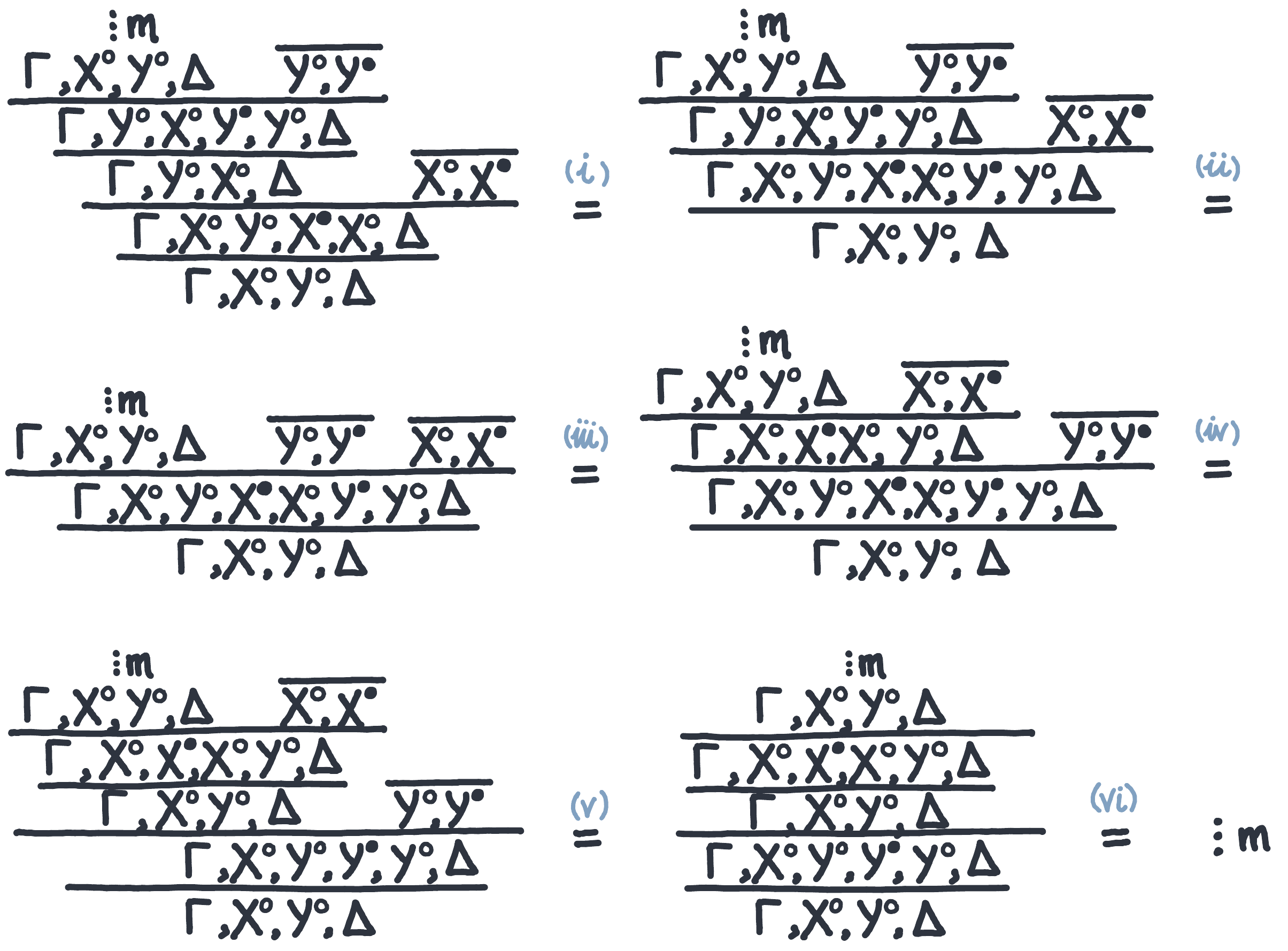}
    \caption{Swaps are self-inverses.}
    \label{fig:session-swaps}
  \end{figure}
  
  Let us prove that they are self-inverses; we do so with the negatively-polarized one (\Cref{fig:session-swaps}), the other case is analogous.
  We reason using naturality of linking, \emph{(i,iv)}, that shuffles compose as shuffles \emph{(ii,iii)} (\Cref{prop:shuffles-as-shuffles}), the interaction between shuffling and spawning \emph{(v)},  and the duality between spawning and linking \emph{(vi)}.
\end{proof}

\begin{proposition}
  \label{prop:linking-wherever}
  We can link sooner or later without changing the result. Formally, the equations in \Cref{fig:linking-wherever} hold.
  \begin{figure}[!ht]
    \centering
    \includegraphics[scale=0.4]{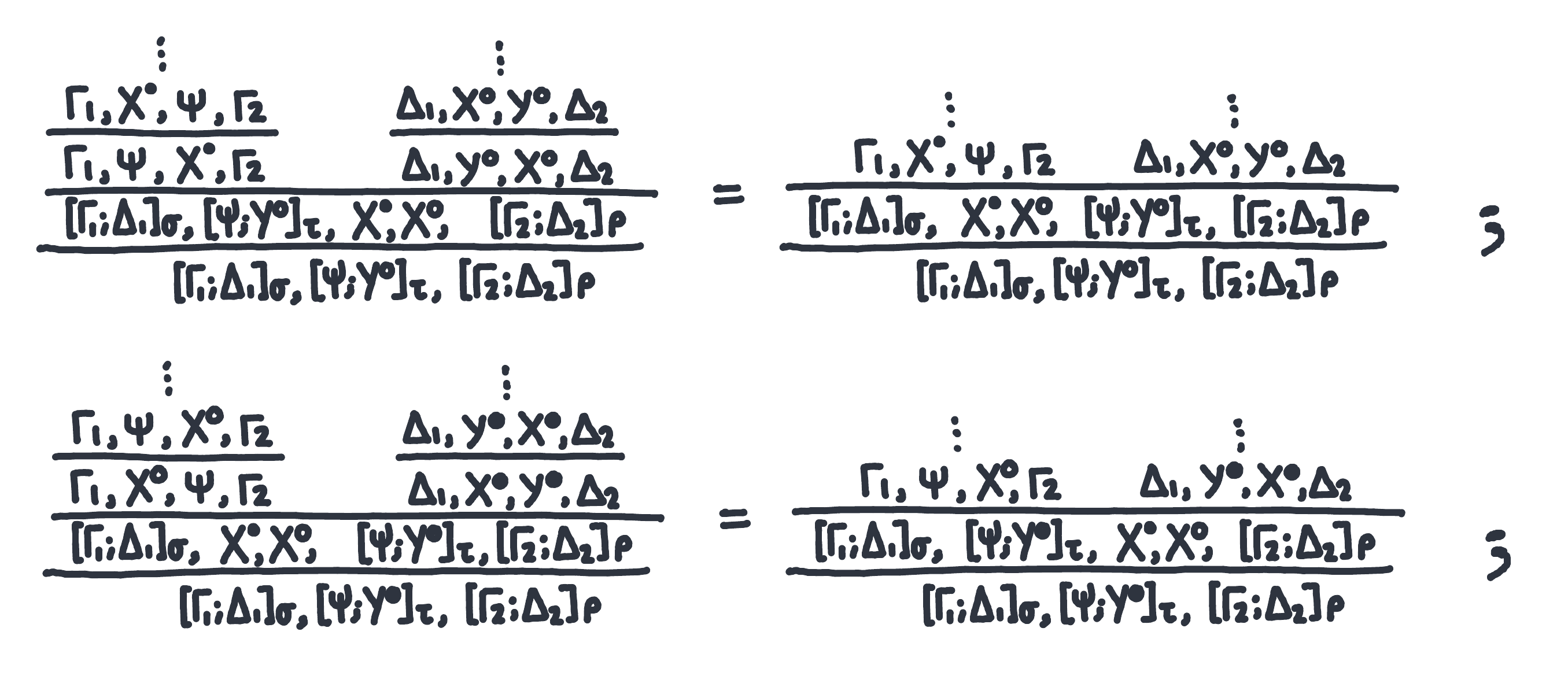}
    \caption{Linking sooner or later does not change the restult.}
    \label{fig:linking-wherever}
  \end{figure}
\end{proposition}
\begin{proof}
  We will prove the first one (\Cref{fig:linking-wherever-proof}), the second one follows analogously. 
  We reason using \emph{(i)} the definition of $\mathsf{wait}$; \emph{(ii,iii)} that shuffles compose as shuffles (\Cref{prop:shuffles-as-shuffles}); \emph{(iv)} naturality of linking; and \emph{(v)} that swaps are self-inverses (\Cref{prop:swaps-self-inverses}).
  \begin{figure}[!ht]
    \centering
    \includegraphics[scale=0.35]{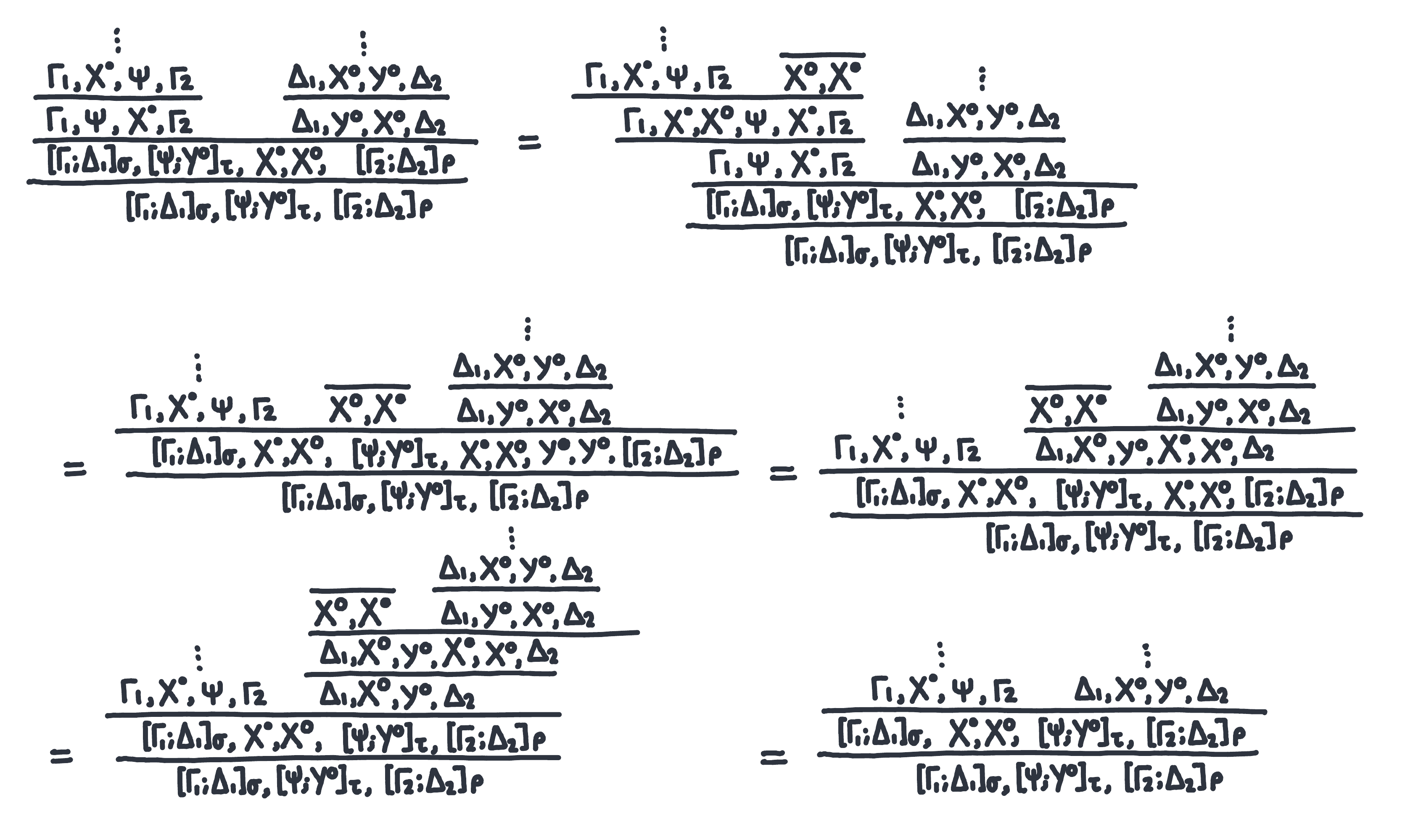}
    \caption{Proof of {{\Cref{prop:linking-wherever}}}.}
    \label{fig:linking-wherever-proof}
  \end{figure}
\end{proof}

\begin{proposition}
  Any spawning factors as the spawning of a single channel followed by a shuffling. Formally,
  $$\SPW_X^{Γ,Δ}(m) = \SHF_{Γ,X^{∘},X^{•},Δ}(m, \SPW_X^{;}(\NOP)).$$
\end{proposition}
\begin{proof}
  This is a direct consequence of Axioms (3a,3b).
\end{proof}

Theoretically, it would be possible to reason with \messageTheories{} at the level of derivations.
However, as we have done through this text, we will try to find better categorical semantics and a better combinatorial expression of the free \messageTheory{}.
We will be most interested in the categorical semantics of \messageTheories{} and how do they interplay with process theories, in the sense of \monoidalCategories{}. We introduce specialized semantics in terms of \physicalMonoidalMulticategories{}, which are another instance of the idea of partially representing a \physicalDuoidalCategory{}. For all of this, we will need a coherence theorem.

\subsection{Coherence for Message Theories}
\SymmetricMonoidalCategories{} are not \emph{perfectly coherent}: easily, we can find that there are two formally well-typed structure maps $A ⊗ A → A ⊗ A$, the identity and the swap. However, what is indeed true is that any two \emph{distinctly typed} structure maps in the free \symmetricMonoidalCategory{} are equal. ``Distinct typing'' is a notion that only makes sense in the free \symmetricMonoidalCategory{} over some generators; it means that the generators comprising the lists that are our objects appear only once with each variance: arrows $A ⊗ B ⊗ C → B ⊗ C ⊗ A$ are distinctly typed, but arrows $A ⊗ A → A ⊗ A$ are not, because $A$ appears twice with each variance.

Shufflings satisfy a similar form of coherence: there is a unique way of shuffling two words into a third one if these words are distinctly typed.
This section proves that \messageTheories{} satisfy the same form of coherence. It is not true that any two parallel formal arrows are equal in any \messageTheory{}: for instance, there are two ways of deriving $X^{∘},X^{∘}$ from $X^{∘}$ and $X^{∘}$. It is true, however, that there is a unique arrow for any distinctly typed domains and codomain.

\begin{theorem}
  \label{theorem:coherentMessageTheories}
  \MessageTheories{} are coherent. In the free \messageTheory{} over a set of objects, there is at most a single derivation between any distinctly typed premises and conclusion.
\end{theorem}
\begin{proof}
  Consider distinctly typed premises and conclusion. There must be three different classes of types in this derivation: \emph{(1)} those that appear twice on the conclusions with different polarity, \emph{(2)} those that appear twice on the premises with different polarity, and \emph{(3)} those that appear once in the premises and once in the conclusions with the same polarity.

  We will construct a non-unique normal form for derivations in a message theory, taking into account each one of these cases.
  \begin{enumerate}
    \item In the first case, the types must have been created by spawning a channel. Using naturality of spawning (Axioms 3a, 3b, 5c), we can move these spawning operations to be shuffled at the end of the derivation. Note that it is not true that we can move \emph{all} spawning to the end of the derivation, but if the types appear only in the conclusions, then we can always do so.
    \item In the second case, these variables must get linked to each other. We can always move the linkings to the end of the derivation (before spawning new objects) using again the naturality of the linkings (Axioms 2a, 2b, 5a).
    \item For the third case, only a shuffling, rushing and waiting can be involved (any linking or spawning that does not involve rushing or waiting has already been moved to the end). Rushing and waiting can be moved to the beginning of the derivation because of naturality of spawning and linking (Axioms 2a, 2b, 3a, 3b).
  \end{enumerate}
  All of this argues that we can always factor a derivation in a \messageTheory{} as \emph{(i)} rushing and waiting, \emph{(ii)} a shuffle, \emph{(iii)} linkings, \emph{(iv)} spawnings and shufflings of new variables; but we have not yet shown that this derivation is unique.

  Imagine that we know the premises $(Γ₁, …, Γₙ)$ and the conclusion $(Δ)$ of a derivation in a \messageTheory{}. Let us argue that there is at most a unique derivation between these two.
  \begin{enumerate}
    \item Removing the objects that appear twice in the conclusion $(Δ)$, we obtain a new conclusion $(Δ₁)$; there is, at most, a unique way of getting from $(Δ₁)$ to $(Δ)$ using spawnings and shufflings: spawning all the possible variables in any order and employing the only possible shuffle (using \Cref{prop:shuffles-as-shuffles}, axioms 1a, 1b, 1c); the order of spawning does not matter because shuffles are symmetric (Axioms 1c, 5c).
    \item Adding all of the objects that get linked and appear twice with different polarities on the premises, in any order, we obtain a new conclusion $(Δ₂)$. There is a unique way of getting from $(Δ₂)$ to $(Δ₁)$ because of the interchanging axioms of linking (Axioms 5b and 5d). The only obstruction to uniqueness here is that we could choose different conclusions $(Δ₂)$ depending on where they place the variables that will be linked; however, we have already shown that all possible choices lead to the same result (\Cref{prop:linking-wherever}).
    \item We are left with a shuffle and some rushings and waitings.  From the conclusion $(Δ₂)$ we obtain now a sequence of premises $(Γ₁',…,Γₙ')$ that are the same as the original premises $(Γ₁,…,Γₙ)$ with the only difference that the objects appear ordered as in the conclusion $(Δ₂)$. There is a unique possible shuffling from $(Γ₁',…,Γₙ')$ to $(Δ₂)$.
    \item Finally, for each premise, there will be rushings and waitings. Rushings and waitings interchange (\Cref{prop:swaps-self-inverses}), and so there is a unique way of going from the original premise $Γᵢ$ to the new premise $Γᵢ'$.
  \end{enumerate}
  A summary of the proof can be found in \Cref{fig:derivation-shape}: $Δ₁$ is determined from $Δ$; $Δ₂$ is not, but all the possible choices lead to the same result (\Cref{fig:linking-wherever}); $Γ₁',…,Γₙ'$ are determined from $Δ₂$ and $Γ₁,…,Γₙ$; we have argued that in each one of the steps of the proof there is a single possible derivation.
  \begin{figure}[!ht]
    \centering
    \includegraphics[scale=0.45]{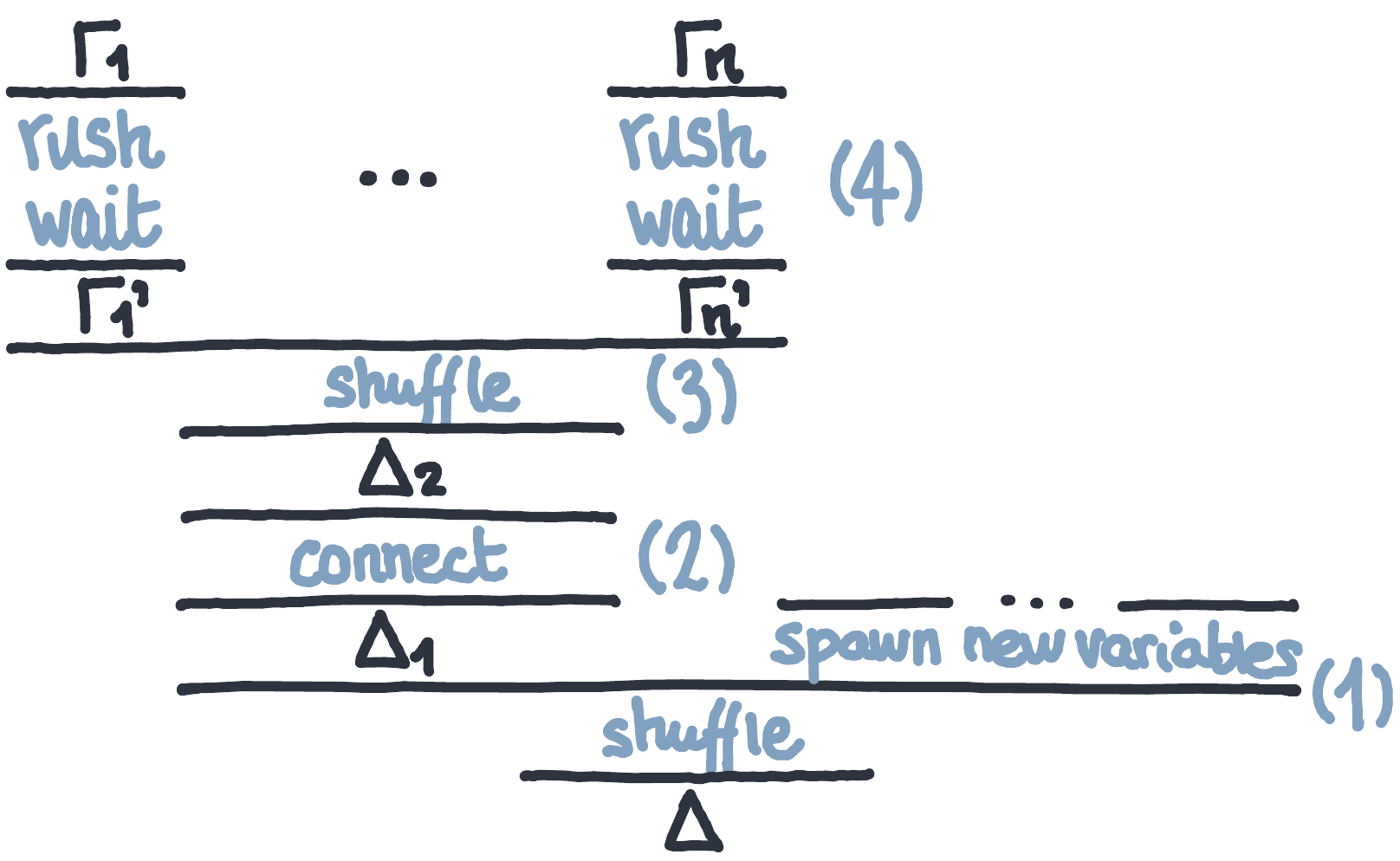}
    \caption{Schema for the proof of uniqueness.}
    \label{fig:derivation-shape}
  \end{figure}
\end{proof}

\subsection*{Bibliography}
Honda pioneered \emph{binary session types} in the 90s \cite{honda93}; and further work with Yoshida and Carbone extended them to the multi-party case \cite{honda08:sessionTypes}. Session types \cite{honda93,honda08:sessionTypes}
are the mainstay type formalism for communication protocols, and they have been extensively applied to the π-calculus \cite{sangiorgi01:picalculus}.
Our approach is not set up to capture all of the features of a fully-fledged session type theory \cite{kobayashi96:linearitypicalculus}. Explicitly, our framework of \messageTheories{} can be compared to asynchronous session types without choice.
Arguably, this makes it more general: it always provides a universal way of implementing send $(\Send{A})$ and receive $(\Get{A})$ operations in an arbitrary \processTheory{} represented by a \monoidalCategory{}.
For instance, recursion and the internal/external choice duality \cite{gay99,pierce93:subtyping} are not discussed, although they could be considered as extensions in the same way they are to monoidal categories: via trace \cite{hasegawa97} and linear distributivity \cite{cockett1997}.

\newpage
\section{Physical Monoidal Multicategories, and Shufflings}
\label{sec:physical-monoidal-multicategories-shufflings}

\PhysicalMonoidalMulticategories{} are \physicalDuoidalCategories{} \cite{spivak13} where the parallel tensor $(⊗)$ is not representable. They follow the same idea of \produoidalCategories{}, but they will be a better framework for message passing. The theory of \physicalMonoidalMulticategories{} will need of three ingredients: \emph{(i)} \symmetricMulticategories{}, which correspond to the symmetry of the parallel tensor $(⊗)$; \emph{(ii)} \monoidalMulticategories{}, which correspond to the half-representable sequential tensor $(⊲)$; and \emph{(iii)} normality, which we will need to reinterpret in this setting.

\subsection{Symmetric Multicategories} 
Symmetric multicategories \cite{baez98:higher,cruttwell09:unified,shulman:catlog} are to \multicategories{} what \symmetricMonoidalCategories{} are to \monoidalCategories{}.
Let us write $σ ∈ S(n)$ for an element of the permutation group on $n$ points. Let us write $σ₁ + σ₂ ∈ S(n + m)$ for the disjoint union of two permutations, $σ₁ ∈ S(n)$ and $σ₂ ∈ S(m)$. Let us write as $σ \wr (k₁, …, kₙ) ፡ k_{σ1} + … + k_{σn} → k_1 + … + k_n$ the thickening of a permutation $σ ∈ S(n)$ to a permutation $S(k_1 + \dots + k_n)$ that applies it \emph{considering each block} of $k_n$ elements separately.

\begin{definition}
  \defining{linkSymmetricMulticategory}{}
  A \emph{symmetric multicategory} is a \multicategory{} $𝕄$ together with the following family of functions
  $$σ^\ast ፡ 𝕄(X_{σ1}, … , X_{σn} ; Y) → 𝕄(X₁, …, Xₙ; Y),\mbox{ for each } σ ∈ S(n),$$
  that moreover satisfy the following axioms: \emph{(i)} functoriality, $(τ ⨾ σ)^\ast(f) = τ^\ast (σ^\ast (f))$ and $\id^\ast(f) = f$; \emph{(ii)} preservation of disjoint unions, $σ₁^\ast(f_1)⨾_1 … ⨾_{n-1} σ_n^\ast(fₙ) ⨾_n g = (σ₁ + … + σₙ)(f₁ ⨾_1 … ⨾_{n-1} fₙ ⨾_n g)$; and \emph{(iii)} naturality, $f₁ ⨾_1 … ⨾_{n-1} fₙ ⨾_n σ^\ast(g) = (σ \wr (k₁,…,kₙ))^\ast (f_{σ1} ⨾_1 … ⨾_{n-1} f_{σn} ⨾_n g),$ for any $fₙ$ having arity $kₙ$.
\end{definition}

In \physicalMonoidalMulticategories{}, symmetry appears with the non representable parallel tensor $(⊗)$; while the sequential tensor will still form a representable \monoidalCategory{}, this is the notion of \monoidalMulticategory{}.

\subsection{Monoidal Multicategories}
\Multicategories{} helped us describe typed algebras: the objects of the \multicategory{} were the types, and the multimorphisms were the operations we were allowed to use in that algebra. As we saw with \produoidalCategories{}, we need a second dimension if we want to study parallelism: what happens when the types themselves form a monoid? \MonoidalMulticategories{} are the monoidal version of algebra.

A \emph{monoidal category} was a 2-monoid on the 2-category of categories, functors and natural transformations. Analogously, a \emph{monoidal multicategory}, sometimes called a \emph{virtual duoidal category} \cite{nlab:duoidal} or \emph{monoidal operad}, is a 2-monoid of the 2-category $\Mult$ of multicategories, multifunctors, and multinatural transformations.

\begin{definition}
  \defining{linkMonoidalMulticategory}{}
  A \emph{monoidal multicategory} $(𝕄,⊲,N,m,i,⨾)$ is a \multicategory{} $(𝕄,⨾)$ with a tensor and a unit both on objects $(⊲) ፡ 𝕄_{obj} × 𝕄_{obj} → 𝕄_{obj}$ and $N ፡ 𝕄_{obj}$, and a tensor and unit on multimorphisms,
  \begin{align*}
    mₙ ፡ & 𝕄(X₁,\mydots, Xₙ; Y) × 𝕄(X'_1,\mydots, X'ₙ; Y) %
    → 𝕄(X_1 ⊲ X'_1, \mydots, Xₙ ⊲ X'ₙ; Y ⊲ Y'), \\
    n_n ፡ & 1 → 𝕄(N, \overset{n}\dots, N; N).
  \end{align*}
  These are associative and unital up to multinatural transformations $α_{X,Y,Z} ∈ 𝕄(X ⊲ (Y ⊲ Z); (X ⊲ Y) ⊲ Z)$, $λ_X ∈ 𝕄(I ⊲ X; X)$ and $ρ_X ∈ 𝕄(X ⊲ I; X)$, satisfying the pentagon and triangle equations.
\end{definition}

\begin{remark}
  Any \monoidalMulticategory{} has an underlying \monoidalCategory{}. 
  Multifunctors between representable multicategories are lax monoidal functors, and multinatural transformations are lax monoidal transformations. This implies that a representable monoidal multicategory is exactly a \duoidalCategory{} \cite{nlab:duoidal}.
\end{remark}

\subsection{Physical Monoidal Multicategories}

Bringing the notion of normality to the context of \monoidalMulticategories{} is not completely trivial: here, there is no map between the two units $I → N$. Instead, what we will have is a map between the hom-sets, given by the precomposition of the multimap $n₀ ∈ 𝕄(;N)$ that we have because of the monoidality of the unit.

\begin{definition}
  \label{def:physical-monoidal-multicategory}
  A \emph{physical monoidal multicategory} is a \symmetricMulticategory{} $(𝕄,⨾,\id)$ that is moreover monoidal forming a virtual duoidal category $(𝕄,⊲,N)$ and that contains an invertible composition with the unit map.
  That is, the following composition is an isomorphism,
  $$(n₀ ⨾ •)^{-1} ፡ 𝕄(X₀,\overset{n}{…},Xₙ ; Y) → 𝕄(X₀,…,N,…,Xₙ; Y).$$
\end{definition}

\begin{remark}
  We know the structure we need, but we will need to be able to compute with it.
  A better combinatorial description of \physicalMonoidalMulticategories{} will be possible once we characterize the free ones.
  The next section shows that the \physicalMonoidalMulticategory{} of shufflings is the free normal symmetric monoidal multicategory over a set of objects.
\end{remark}

\subsection{Shuffling}
Shuffling events is the principle on which we have based our definition of \messageTheories{}.
The first step towards endowing them with categorical semantics is to show that shuffling arises naturally in a categorical setting: if \physicalDuoidalCategories{} represented inclusions of posets, \physicalMonoidalMulticategories{} will represent inclusions of linear posets, which correspond to shuffling.

\begin{definition}
  \defining{linkShufflingWords}{}
  Let $A$ be an alphabet. The \monoidalMulticategory{} of \emph{shuffling words}, $\mathsf{wShuf}_A$, has as objects the set $A^\ast$ of words on the alphabet $A$. The multimorphisms $\mathsf{wShuf}_A(w₁,…,wₙ; w)$ are precisely the shufflings of the letters of the words $w₁,… , wₙ$ that result into the word $w$. The monoidal tensor is given by concatenation of words.
\end{definition}

\begin{remark}
  This \multicategory{} is not posetal: for instance, there exist two multimaps $\mathsf{wShuf}_A(a,a;aa)$ for any letter $a ∈ A$. More generally, the number of shuffles of a repeated letter is given by the binomial coefficients,
  \[\#\mathsf{wShuf}(a^{n_1},…,a^{n_k}; a^{n₁ + … + nₖ}) = \frac{(n₁ + … + nₖ)!}{n₁! … nₖ!}.\]
  However, it is true that, if each letter appears at most once on the component words, then the shuffle, if it exists, must be unique. For instance, there exists a unique shuffle $\mathsf{wShuf}(xy,wz; xwyz)$, and so we can refer to it without explicitly specifying it. In other words, there exists at most one morphism between formal distinctly-typed expressions; shuffles are coherent.
\end{remark}

\begin{remark}
  This \monoidalMulticategory{} is symmetric and normal. In other words, it is a \physicalMonoidalMulticategory{}.
\end{remark}

\begin{proposition}
  \label{prop:shuffling-free-physical-monoidal-multicategory}
  \ShufflingWords{} are the free \physicalMonoidalMulticategory{} on a set of objects.
\end{proposition}
\begin{proof}
  We have already shown how they form a \physicalMonoidalMulticategory{}; we need to show that it is the free one. Let $𝕍$ be any \physicalDuoidalCategory{}, for each map $A → 𝕍_{obj}$ there must exist a unique physical monoidal multifunctor $\wShuf_A → 𝕍$ that factors on objects through the former map.

  First, we argue that the physical monoidal multifunctor is forced. Any shuffle can be reconstructed from the operations of a \physicalMonoidalMulticategory{}, and so any shuffle must be mapped accordingly. For instance, if we want to reconstruct the shuffle $𝕍(xy,z; xzy)$, we multiply together: first $𝕍(x,i;x)$, then $𝕍(i,z;z)$, and finally $𝕍(y,i;y)$.
  
  Let us do this in general. Assume a shuffle $\wShuf(w₁,…,wₙ; w)$, where the words are $wⱼ = a_{j,1} … a_{j,k_j}$. We position each one of the letters in their position: for instance, if $a_{j,i}$ appears in the $j$th-word, we can use the normalization isomorphism 
  $$𝕍(a_{j,i};a_{j,i}) ≅ 𝕍(I,…,a_{j,i}^{(j)},…,I; a_{j,i})$$
  to position it in the $j$th input, via the identity multimorphism. Then, we can multiply all of the letters forming the shuffle, using monoidality, and obtain the desired shuffle in any \physicalMonoidalMulticategory{},
  \[ \prod_{a_{j,i} ∈ w} 𝕍(I,…,a_{j,i}^{(j)}, …,I ; a_{j,i}) \rightarrow 𝕍(w₀,…,wₙ; w).\]
  We have used exactly one factor for each letter of the resulting word.

  Finally, we need to prove that this assignment is well-defined and that it defines a physical monoidal multifunctor. These will be both a consequence of the same fact: there exists at most a unique shuffle between words with different letters and there exists at most a unique formally distinctly typed map between any poset shapes.
  To prove this, we use the characterization of the free duoidal category as poset shapes: there is a single formally well-typed map between any two poset shapes, so there is a single formally well-typed map between any two duoidal expressions; in particular, there exists at most one multimorphism between any words with distinct letters in a physical monoidal multicategory.
  This renders the map well-defined: there is at most one way of constructing any shuffle, and we have shown that it is possible. On the other hand, this also makes the map a physical monoidal multifunctor: all the equations are formal and thus they hold automatically.
\end{proof}

Shuffling takes care of the first and most important aspect of message theories: \messageTheories{} are algebras for the shuffling words \physicalMonoidalMulticategory{} of their objects. 
However, message theories do have an extra structure: each object is a left dual, naturally with respect to these shufflings -- the second ingredient for message theories is \emph{polarization}.

\subsection{Bibliography}

This formulation of \symmetricMulticategories{} is a particular case of Shulman's definition of \multicategory{} over a \emph{faithful cartesian club} \cite[\S 2]{shulman:catlog}. The theory of \symmetricMulticategories{} is less developed than the theory of \emph{cartesian multicategories}, which are well-known to correspond to Lawvere theories.
\MonoidalMulticategories{} are also not particularly explored, but Shulman has a note relating them directly to \duoidalCategories{} and proposing the name \emph{virtual duoidal category} for them \cite{nlab:duoidal}.

\newpage
\section{Polarization}
\label{sec:polarization}
Polarization is not a property of our structures, but more of a particular way of constructing free structures. This is, we do not say that a monoid is ``polarized''; but we have constructed free polarized monoids when talking about \messageTheories{}. Saying that a monoid is polarized would seem to imply that all of its elements have a sign, which is not what happens in the free polarized monoid: generators do appear with a sign, but the words on the monoid are combinations without a particular sign.

\begin{remark}[Polarization of a monoid]
  Given any monoid $M$, we can consider its polarization, $\mathsf{Polar}(M)$, to be the monoid generated by two copies of each element of $M$, denoted $a^{•} ∈ \mathsf{Polar}(M)$ and $a^{∘} ∈ \mathsf{Polar}(M)$, and quotiented by the equations $a^{•} b^{•} = (ab)^{•}$ and $a^{∘}b^{∘} = (ab)^{∘}$.
\end{remark}

Taking seriously this idea, we will not regard polarization as part of an algebraic structure (as it happens with \monoidalCategories{}). Instead, polarization will arise as a left adjoint. To quip, \emph{polarization is left adjoint to taking left adjoints}. Let us see first how this works in the context of monoidal categories; we will later extend it to \physicalMonoidalMulticategories{}.

\subsection{Monoidal Polarization}

Every \monoidalCategory{} has a notion of \emph{duality} inside it. It does not suffice to say that some objects \emph{have duals}: a duality is not only a property. Instead, we need to specify the maps that constitute the duality. Dualities in a \monoidalCategory{} form a \monoidalCategory{} themselves. Polarization is the left adjoint to taking this category of dualities.

\begin{definition}[Category of dualities]
  Let $(ℂ,⊗,I)$ be a monoidal category.
  We define the category of left duals, $\mathsf{Duals}(ℂ)$ to have objects the dualities $(L ⊣ R, ε, η)$ of the monoidal category and morphisms 
  $$(f_L,f_R) ፡ (L ⊣ R, ε, η) → (L' ⊣ R', ε', η')$$
  to be pairs of morphisms $f_L ፡ L → L'$ and $f_R ፡ R' → R$ such that the equations in \Cref{fig:morphism-of-adjunctions} hold.
  \begin{figure}[!ht]
    \centering
    \includegraphics[scale=0.4]{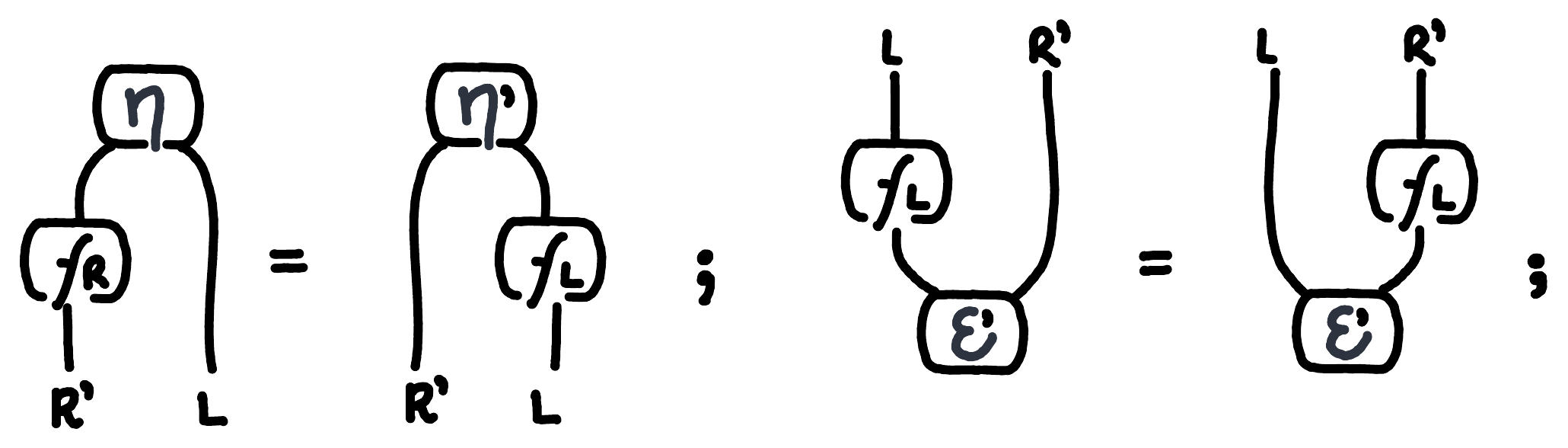}
    \caption{Morphism of adjunctions.}
    \label{fig:morphism-of-adjunctions}
  \end{figure}

  The category of left duals is a monoidal category where the tensor of two dualities, $(L ⊣ R, ε, η)$ and $(L' ⊣ R', ε', η')$, is defined to be
  $$(L ⊗ L' ⊣ R' ⊗ R, (\id ⊗ ε ⊗ \id)⨾ ε', η ⨾ (\id ⊗ η ⊗ \id)).$$
\end{definition}

Taking the left duals extends to an endofunctor $\mathsf{Duals} ፡ \mathbf{MonCat} → \mathbf{MonCat}$.
This endofunctor has a right adjoint $\mathsf{Polar} ፡ \mathbf{MonCat} → \mathbf{MonCat}$.

\begin{definition}[Polarization]
  Let $(ℂ,⊗,I)$ be a strict monoidal category.
  Its polarization, $\mathsf{Polar}(ℂ)$, is a monoidal category presented by the following data. 
  
  The polarization contains two objects, $A^{•}$ and $A^{∘}$, for each object of the original category $A ∈ ℂ_{obj}$, and quotiented by equalities for tensors $(A ⊗ B)^{•} = A^{•} ⊗ B^{•}$ and $(A ⊗ B)^{∘} = B^{∘} ⊗ A^{∘}$ and units $I^{•} = I^{∘} = I$.  
  It also contains a pair of morphisms $ε_A ፡ A^{•} ⊗ A^{∘} → I$ and $η_A ፡ I → A^{∘} ⊗ A^{•}$ constructing an adjunction $A^{•} ⊣ A^{∘}$: that is, $ε_A ⨾ η_A = \id$ and $η_A ⨾ ε_A = \id$. The duality is monoidal, satisfying both $ε_{A ⊗ B} = ε_A ⨾ ε_B$ and $η_{A ⊗ B} = η_B ⨾ η_A$, and finally $ε_I = η_I = \id_I$.

  The polarization contains two dual morphisms, $f^{•} ፡ A^{•} → B^{•}$ and $f^{∘} ፡ B^{∘} → A^{∘}$, for each morphism $f ፡ A → B$.  These are quotiented by equations explicitly asking for functoriality: $f^{•} ⨾ g^{•} = (f ⨾ g)^{•}$ and $\id^{•} = \id$, and also $f^{∘} ⨾ g^{∘} = (g ⨾ f)^{∘}$ and $\id^{∘} = \id$.  These are tensored as $(f ⊗ g)^{•} = f^{•} ⊗ g^{•}$ and $(f ⊗ g)^{∘} = g^{∘} ⊗ f^{∘}$.
\end{definition}

\begin{proposition}
  \label{prop:polarization}
  Polarization is left dual to taking left duals.
\end{proposition}
\begin{proof}
  We already know that $\Duals ፡ \MonCat → \MonCat$ is a functor; we only need to construct a universal arrow $η_{𝕄} ፡ 𝕄 → \Duals(\Polar(𝕄))$. The universal arrow will be given by $η(X) = (X^{•} ⊣ X^{∘}, ε_X, η_X)$, which uses the fact that $X^{•} ⊣ X^{∘}$ determines an adjunction. This assignment extends to a functor that sends a morphism $f ፡ X → Y$ to a morphism of dualities $(f^{•}, f^{∘}) ፡ (X^{•} ⊣ X^∘{}, ε, η) → (Y^{•} ⊣ Y^{∘}, ε, η)$, where $f^{∘} ፡ Y^{∘} → X^{∘}$ is the dual of $f^{•} ፡ X^{•} → Y^{•}$. This constructs the candidate universal arrow.

  Let us show that it is indeed universal. Consider a functor $H ፡ 𝕄 → \Duals(ℕ)$. We will show that there is a unique $H^\ast ፡ \Polar(𝕄) → ℕ$ such that $H = η_{𝕄} ⨾ \Duals(H^\ast)$. Let us pick an object $X ∈ 𝕄_{obj}$ that is sent to $H(X) = (A ⊣ A', ε_A, η_A)$.  To make this equation hold, it is necessary that $H^\ast(X^\circ) = A$ and $H^\ast(X^\bullet) = A'$, but also $H^\ast(ε_X) = ε_A$ and $H^\ast(η_X) = η_A$. 
  Let us pick a morphism $f ፡ A → B$ that is sent to $H(f) = (g_L,g_R) ፡ (A ⊣ A', ε_A, η_A) → (B ⊣ B', ε_B, η_B)$; this forces $H^\ast(f^{•}) = g_L$ and $H^{\ast}(f^{∘}) = g_R$.
  Because $(A ⊣ A', ε_A, η_A)$ forms a duality and the maps $g_L$ and $g_R$ are duals, this assignment satisfies all necessary equations. 
  This determines the value of $H^\ast$ on all of the generating maps while satisfying all of the equations of $\Polar(𝕄)$.
\end{proof}

\subsection{Monoidal Polarization is Not Enough}
An important insight of some works into the categorical semantics of message passing is the importance of polarization. Given this, one could expect that the polarization of a \processTheory{} is its corresponding \messageTheory{}.
Indeed, this chapter will claim that the polarization of a \symmetricMonoidalCategory{} exhibits the necessary structure to discuss message passing; however, at the same time, it will claim that the algebraic structure that allows it to do so is not that of a polarized monoidal category: it is that of a \messageTheory{}.

\begin{remark}
  Some simpler interleavings of events can be expressed using dualities in a \monoidalCategory{}. For instance, let us consider the polarization of a \symmetricMonoidalCategory{}: wires with the same sign can be swapped, but wires with different sign cannot. We can interleave two maps $f ፡ I → X^{∘} ⊲ Y^{∘} ⊲ Z^{•}$ and $g ፡ I → U^{∘} ⊲ V^{•}$ to produce a map $h ፡ I → X^{∘} ⊲ U^{∘} ⊲ V^{•} ⊲ Y^{∘} ⊲ Z^{•}$ as in \Cref{fig:shuffling-that-can} -- note how we use the adjunction to contort the wires and allow wires to pass over wires with different polarization.
  \begin{figure}[!ht]
    \centering
    \includegraphics[scale=0.32]{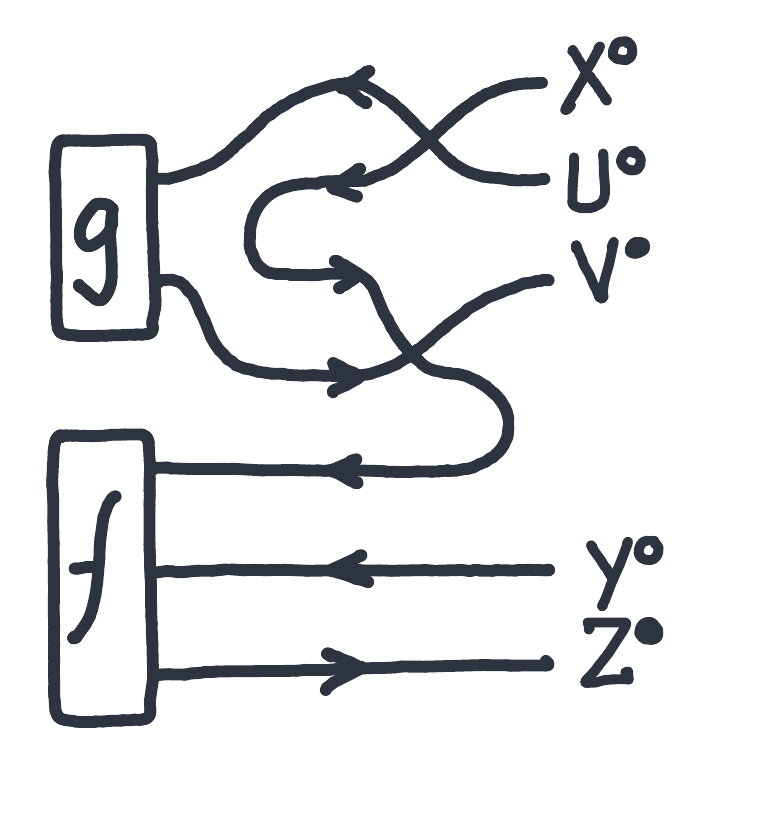}
    \caption{A shuffling, using polarization.}
    \label{fig:shuffling-that-can}
  \end{figure}

  In general, thanks to the dualities, a \emph{receiving} resource can be used later, $X^{∘} ⊲ A → A ⊲ X^{∘}$, and a \emph{sending} resource can be sent sooner, $A ⊲ X^{•} → X^{•} ⊲ A$, but not the other way around. These two laws govern which wires can pass over other wires. A natural question, then, is to ask if all possible interleavings can be expressed in the same way: polarizing objects, and using dualities.
\end{remark}

\begin{proposition}
  \label{prop:shuffling-limit}
  It is not the case that every shuffling of events can be expressed in a category with polarized objects and dualities.
\end{proposition}
\begin{proof}
  Consider two cells as in \Cref{fig:polar-not-enough}, with types 
  \begin{align*}
    f ፡ I → X^{•} ⊲ Y^{∘} ⊲ Z^{•} ⊲ W^{∘}, \mbox{ and }
    g ፡ I → A^{•} ⊲ B^{∘} ⊲ C^{∘} ⊲ D^{•} ⊲ E^{•} ⊲ F^{∘}
  \end{align*} 
  and imagine we want to interleave them into $A^{•} ⊲ B^{∘} ⊲ X^{•} ⊲ Y{∘} ⊲ C^{∘} ⊲ D^{•} ⊲ Z^{•} ⊲ W^{∘} ⊲ E^{•} ⊲ F^{∘}$ using only the dualities. There are two possible cases: \emph{(i)} we place $f$ before $g$ and we try to position the wires of $f$ correctly; or \emph{(ii)} we place $g$ before $f$ and we try to do the same.
  
  In the first case, we will find that we can pass the wire $X^{∘}$ over $A^{•} ⊲ B^{∘}$ -- but the same cannot be done with the wire $Y^{•}$. Therefore, we are forced to pass the wire $Y^{•}$ over $C^{∘} ⊲ D^{•} ⊲ E^{•} ⊲ F^{∘}$, using the dualities; but then there will be no way of moving the wire $Z^{•}$ to its place.
  \begin{figure}[!ht]
    \centering
    \includegraphics[scale=0.32]{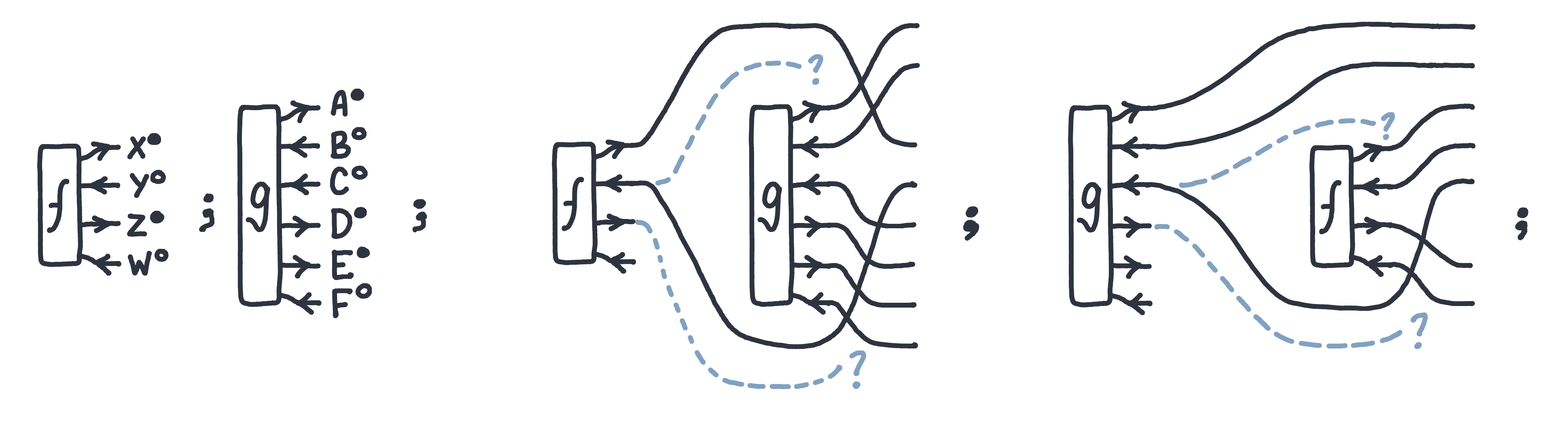}
    \caption{A shuffling that cannot be expressed.}
    \label{fig:polar-not-enough}
  \end{figure}

  In the second case, the first two wires, $A^{•} ⊲ B^{∘}$, are automatically correctly placed. The third, $C^{∘}$, must pass over $Z^{•} ⊲ W^{∘}$ because it could not do the same with $X^{•} ⊲ Y^{∘}$; so $D^{•}$ would be forced to do the same, but it cannot.
\end{proof}

Then, if this algebra is not what we are using about the polarization of a \monoidalCategory{}, what are we using exactly? It happens that the polarization of a \monoidalCategory{} posesses extra structure: its states (the morphisms without input) can always be split into morphisms that take negative objects on the left and produce positive objects on the right. This property is what makes it possible to express any interleaving of events.

With this in mind, the polarization of a \monoidalCategory{} seems interesting not only because it produces a category with duals (or, in the work of Nester \cite{nester21}, a proarrow equipment, or a \emph{cornering}), but because it constructs a full \messageTheory{}. We do not only care about the free polarization, we care about its \messageTheory{}.

\subsection{Polarization of a Physical Monoidal Multicategory}

Every \physicalMonoidalMulticategory{} has a notion of dual, inherited from that of its underlying monoidal category. The next section will construct the free polarized \physicalMonoidalMulticategory{} over a set of objects. Let us define here the right adjoint: the functor that picks the set of left adjoints of a \physicalMonoidalMulticategory{}.

\begin{definition}
  \label{def:physicalMonoidalMulticategoryDuals}
  Let $(𝕄,⊲,I)$ be a \physicalMonoidalMulticategory{}.  We define the set of left duals, $\Duals(𝕄)$, to have objects the dualities $(L ⊣ R, ε, η)$ where $ε ∈ 𝕄(L ⊲ R; I)$ and $η ∈ 𝕄(I ; R ⊲ L)$ are 1-to-1 multimorphisms satisfying the adjoint equations, $(ε ⊲ \id_L) ⨾ (\id_R ⊲ η) = \id_L$ and $(\id_R ⊲ ε) ⨾ (η ⊲ \id_R) = \id_R$.
\end{definition}

\subsection{Bibliography}
A categorical treatment of polarization appears in the work of Cockett and Seely \cite{cockett07:polarized}, which points to the connection with Abramsky-Jagadesaan games \cite{abramsky94:games}. ``Polarized category'' takes a different meaning there: it is a pair of categories endowed with a \profunctor{} between them. However, we do follow the same core idea: using walking adjunctions for sending and receiving.

Nester notices the importance of polarization for message passing \cite{nester21} via a single-object proarrow equipment; and all the credit for this idea should go there. This was later extended by Nester and Voorneveld \cite{voorneveld23} to include iteration and choice.
The reader will find a discussion of its relationship with lenses in a joint paper of this author with Nester and Boisseau \cite{boisseaunester:corneringoptics}.
The graphical calculus of proarrow equipments was described by Myers \cite{myers16}; we can reuse this calculus to explain polarization in \monoidalCategories{}.
Here, we prefer to avoid discussing polarization through proarrow equipments, noticing that this adjunction already works at the level of \monoidalCategories{}.

\newpage
\section{Polar Shuffles}
\label{sec:polar-shuffles}

\subsection{Polar Shuffles}
Polarization in a \physicalMonoidalMulticategory{} leads us to consider \polarShuffles{}: shufflings endowed with a polarization for each one of its elements. This section will show that \polarShuffles{} are the free polarized \monoidalMulticategory{}.

The objects of the category of shuffles could be seen as the finite ordinals: finite linear posets with inclusions on each other, preserving the ordering. Similarly, the objects of the category of \polarShuffles{} are polarized finite ordinals.

\begin{definition}
  \label{def:polarlist}
  A \emph{polar list} over a set of types $T$ is a list of types $X ∈ T^\ast$ endowed with a function $p ፡ X → \{ ∘, • \}$. 
  In any polar list, we consider the \emph{negative elements}, $X^{∘} = p^{-1}(∘)$, and the \emph{positive elements}, $X^{•} = p^{-1}(•)$. %
\end{definition}

\begin{definition}
  \defining{linkPolarShuffle}{}
  \label{def:polarshuffle}
  A \emph{polar shuffle} over a set of types $T$, from a multiset of polar lists $X₀,…,Xₙ$ to a single polar list $Y$, is a bijection
  \[f ፡ X₁^{•} + … + Xₙ^{•} + Y^{∘} → Y^{•} + X₁^{∘} + … + Xₙ^{∘}\]
  preserving the types and such that the directed graph containing all polar lists (as linear posets) and an edge $x → f(x)$ for each element in $X₁^{•} + … + Xₙ^{•} + Y^{∘}$, is acyclic.
\end{definition}

For instance, an untyped \polarShuffle{} (or typed over the singleton set) of shape $\pShuf({∘}•{∘}•, {∘}{∘}••, {∘}; {∘}{∘}••{∘})$, is given by the following acyclic graph in \Cref{fig:polar-shuffle-example}. In black, we depict the edges that come from the graph of a function. In blue, the edges that come from the linear finite posets.
\begin{figure}[!ht]
  \centering
  \scalebox{1}[1]{\includegraphics[scale=0.45]{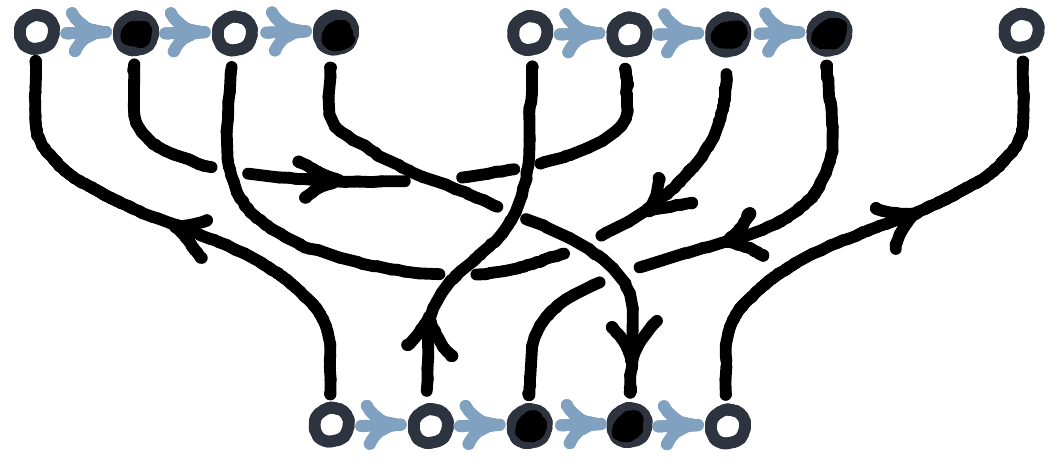}}
  \caption{A polar shuffle.}
  \label{fig:polar-shuffle-example}
\end{figure}

\subsection{Encoding of polar shuffles}

\PolarShuffles{} are ultimately graphs, and they can be encoded as such. We propose a notation suggestive of multiparty session calculi.

The \emph{session encoding} of \polarShuffles{} assigns a variable name to each polar list (say, $f, g, h, …$) and to each edge of the graph outside a polar list (say, $a,b,c,x,y,z,…$). The encoding of a \polarShuffle{} starts by declaring the list of edges incident to the output polar list, together with their polarization. Then, enclosed in braces, we write the polar lists and the edges that incide on them.

For instance, the encoding of the \polarShuffle{} in \Cref{fig:polar-shuffle-example} is in \Cref{fig:polar-shuffle-example-encoding}.
\begin{figure}[!ht]
  \centering
  \begin{minipage}{0.5\textwidth}
    \centering
    \includegraphics[scale=0.45]{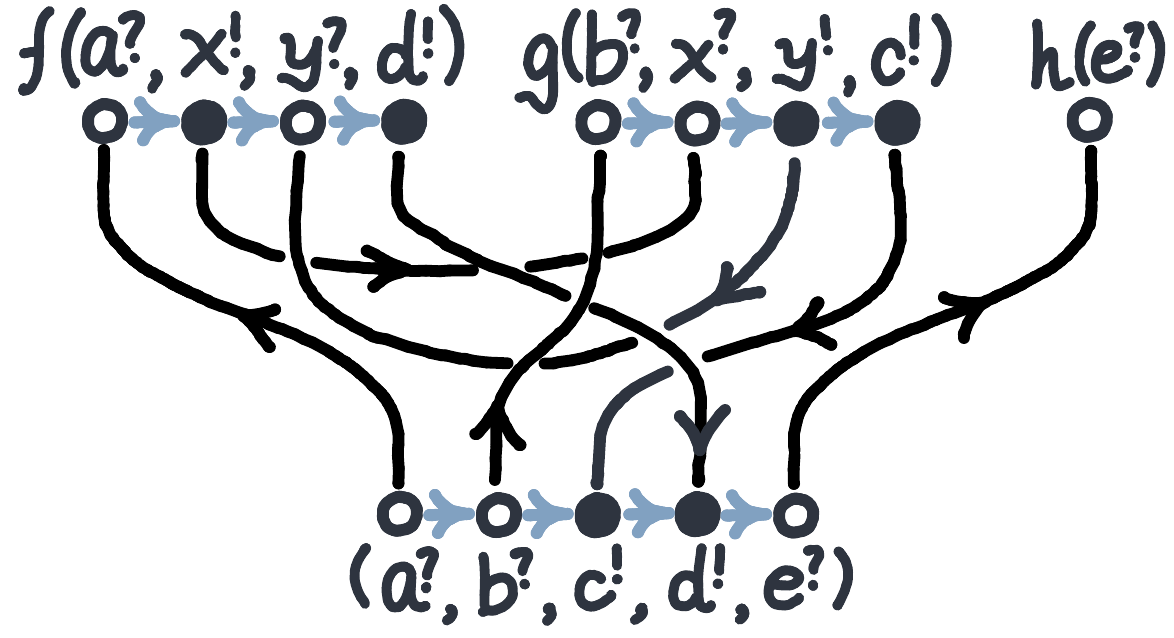}
  \end{minipage}
  \begin{minipage}{0.49\textwidth}
    \begin{verbatim}

      (a?,b?,c!,d!,e?) { 
        f(a?,x!,y?,d!),
        g(b,x?,y!,c!),
        h(e?) }
    \end{verbatim}
  \end{minipage}
  \caption{Encoding of a polar shuffle.}
  \label{fig:polar-shuffle-example-encoding}
\end{figure}

\begin{remark}
Parsing this notation requires checking whether the graph is acyclic. Checking if a graph is acyclic can be done in linear time on the sum of vertices and edges, $𝓞(v+e)$. The number of vertices in a \polarShuffle{} is the sum of the lengths of all of the polar lists involved, $e = \#X₁ + … + \#Xₙ + \#Y,$ and each vertex receives at most three edges, giving a bound $e \leq 2v$. This means that checking if a \polarShuffle{} is valid is linear in the length of its polar lists, $𝓞(\#X₁ + … + \#Xₙ + \#Y)$.
\end{remark}

The second implication of this encoding is that, if we label the vertices of a \polarShuffle{} with types, there exists at most one \polarShuffle{} with any distinctly typed sources and targets.

\begin{proposition}
  \label{prop:polar-shuffles-are-coherent}
  \PolarShuffles{} are coherent. There exists at most a single \polarShuffle{} between some distinctly typed polar lists: we say that a polar list is distinctly typed if each variable (each type) appears in the premises and the conclusion exactly twice, each time with a different variance.
\end{proposition}
\begin{proof}
  In this combinatorial structure, coherence works almost by definition.
  Note that a \polarShuffle{} is ultimately a bijection satisfying certain extra properties; but the distinct typing already forces where each element must be sent: to the only element with different variance but same type. Whether the \polarShuffle{} exists at all depends on whether it satisfies the acyclicity property.
\end{proof}

\subsection{The Multicategory of Polar Shuffles}

\PolarShuffles{} form a \multicategory{}, as their shape already suggests. The composition and the rest of the structure follow that of the category of shufflings. For instance, the composition of \polarShuffles{} is defined to be the substitution of the resulting polar list of a shuffle into the input polar list of another shuffle. See \Cref{fig:polar-shuffle-composition} for an example.

\begin{figure}[!ht]
  \centering
  \includegraphics[scale=0.40]{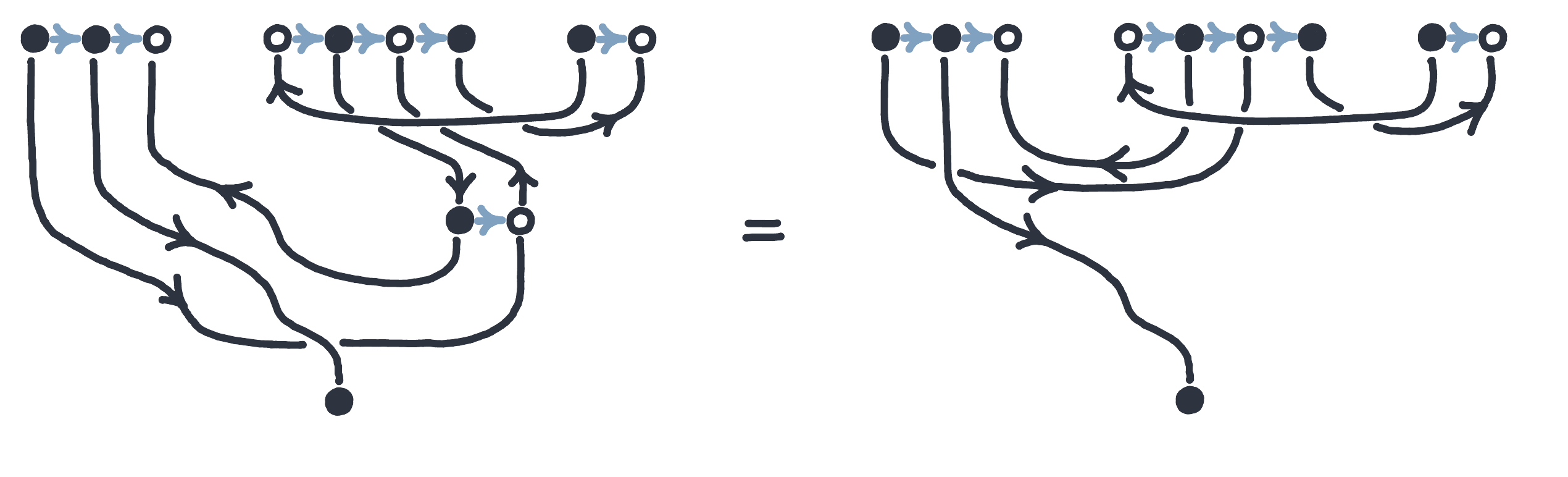}
  \caption{Composition of two polar shuffles.}
  \label{fig:polar-shuffle-composition}
\end{figure}

\begin{proposition}
  \label{prop:polarcomposition}
  The composition of two \polarShuffles{}, $s ∈ \pShuf(X₁,…,Xₙ;Y)$ and $t ∈ \pShuf(Y_1,…,Y,…,Yₘ;Z)$, along a polar list $Y$, is a \polarShuffle{} obtained by substituting the entire graph of the former into the polar list of the latter,
  $$s ⨾_Y t ∈ \pShuf(Y_1,…,X₁,…,Xₙ,…,Yₘ;Z).$$
  Substituting a \polarShuffle{} into the inputs of a \polarShuffle{} forms again a \polarShuffle{}. That is, composition is well-defined and preserves acyclicity.
\end{proposition}
\begin{proof}
  Composition happens across the resulting polar list of a \polarShuffle{}, which coincides with one of the input polar lists of another \polarShuffle{}; we say that this polar list is the \emph{border}. The proof argues that two acyclic graphs glued along a linear graph are again acyclic.

  Let us prove this by contradiction. Imagine there was a cycle in the multicategorical composition of \polarShuffles{}. It must contain edges on both components of the composition, simply because each component is itself acyclic. This means that the cycle should cross the border between both components of the \polarShuffle{} at least twice and always an even number of times -- it must take two different directions.

  \begin{figure}[!ht]
    \centering
    \includegraphics[scale=0.40]{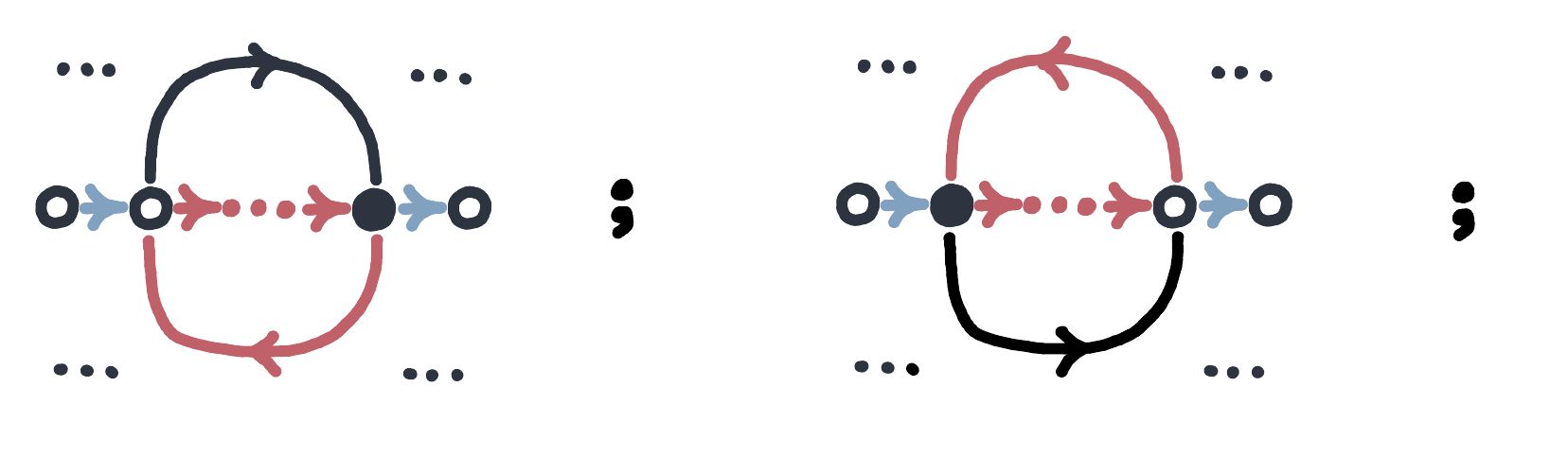}
    \caption{Composition along the borders of two polar shuffles.}
    \label{fig:polar-shuffle-border}
  \end{figure}

  Because the border is a linear poset, it must split the cycle (at least) in two parts, creating two undirected cycles to the two sides of the border. At least one of these two is forced to be directed, thus contradicting acyclicity on that side of the composition.  
\end{proof}

\begin{proposition}
  \label{prop:polartensoring}
  The tensoring of two polar shuffles of the same arity, 
  $$s ∈ \pShuf(X₀,…,Xₙ;Y)\mbox{ and }t ∈ \pShuf(X'₀,…,X'ₙ;Y'),$$
  is a polar shuffle on the pointwise concatenation of the polar lists 
  $$(s ⊲ t) ፡ \pShuf(X₀X'₀,…,XₙX'ₙ;YY'),$$
  defined by the disjoint union of the two acyclic graphs determining the polar shuffles.
  The tensoring of two polar shuffles is well-defined and it is again a polar shuffle.
  See \Cref{fig:polar-shuffle-parallel} for an example.
\end{proposition}
\begin{figure}[!ht]
  \centering
  \includegraphics[scale=0.40]{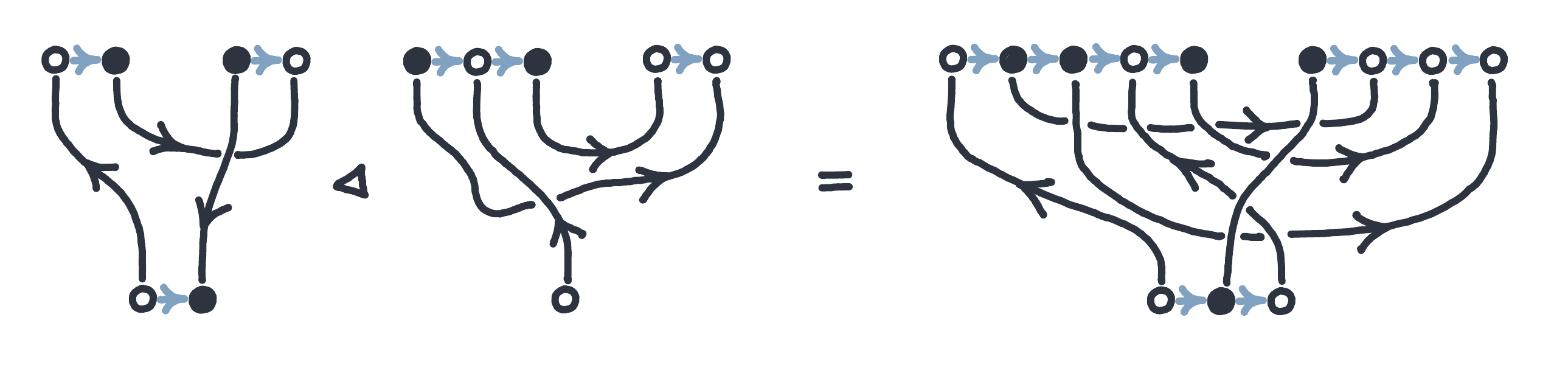}
  \caption{Parallel composition of polar shuffles.}
  \label{fig:polar-shuffle-parallel}
\end{figure}
\begin{proof}
  The graph of the tensoring is the disjoint union of two acyclic graphs, together with the edges determined by the polar lists. The directed edges coming from the polar lists always go from the first graph to the second; thus, they will not create cycles and the resulting graph with be acyclic.
\end{proof}

\begin{theorem}
  \label{th:polar:physicalmonoidal}
  Polar shuffles over a set of types are the morphisms of a \physicalMonoidalMulticategory{}, $\pShuf{}$, that has the polar lists as objects.
\end{theorem}
\begin{proof}
  We define the identity polar shuffle on a polar list to be the identity bijection linking each sign to itself. The identity polar shuffle is acyclic because the identity bijection preserves the linear ordering. We have already defined the composition and shown that it is acyclic in \Cref{prop:polarcomposition}. Associativity of composition follows from associativity of glueing graphs; unitality follows by construction. We have already defined the tensoring in \Cref{prop:polartensoring}, and the tensor on objects is the concatenation of polar lists. The unit for the tensoring is the empty polar list, and because it does not appear in a \polarShuffle{}, it makes the \monoidalMulticategory{} normal.
\end{proof}

\subsection{Message Theories are Algebras of Polar Shuffles}
What makes \polarShuffles{} relevant to the discussion of message-passing is that they promise us a better syntax for \messageTheories{}. Instead of thinking of the operations of a \messageTheory{} as generated by a few primitives satisfying equations, we can give them a combinatorial characterization in terms of \polarShuffles{}. This section shows that \messageTheories{} are precisely the algebras for the operations given by \polarShuffles{}.

\begin{definition}
  An algebra over a \multicategory{}, $(𝕄,⨾,\id)$, is the assignment of a set $S(X)$ to each object $X ∈ 𝕄_{ob}$, and the assignment of a function
  \[S(f) ፡ S(X₁) × … × S(Xₙ) → S(Y),\]
  for each multimorphism $f ∈ S(X₁, …, Xₙ; Y)$.  The assignment must preserve compositions, $S(f ⨾_X g) = S(f) ⨾_X S(g)$ and identities, $S(\id) = \id$. Alternatively, it is a multifunctor to the cartesian monoidal category of sets and functions.
\end{definition}

\begin{proposition}
  \label{prop:messageTheoriesAreAlgebrasOfFreePolarized}
  \MessageTheories{} are precisely the algebras of the free polarized \physicalMonoidalMulticategory{} over their respective sets of types. In other words, the derivations of a \messageTheory{} form the free polarized \physicalMonoidalMulticategory{} over its set of types.
\end{proposition}
\begin{proof}
  This will follow almost by definition. The definition of \messageTheories{} includes an operation for each shuffling, and these operations compose as shufflings (\Cref{prop:shuffles-as-shuffles} and Axioms 1a, 1b, and 1c); equivalently, this is saying that \messageTheories{} are in particular algebras of the multicategory of shufflings, the free \physicalMonoidalMulticategory{} over their set of types.

  The rest of the axioms of \messageTheories{} are exactly asking that each object is a left adjoint: the spawning and linking operations are describing the unit and the counit of the adjunction; the rest of the axioms are saying that: the unit of the adjunction is natural with respect to shufflings (Axioms 2a and 2b); the counit of the is natural with respect to shufflings (Axioms 3a and 3b); unit and counit satisfy the snake equations (4a and 4b); and they are natural with respect to each other (Axioms 5a, 5b, 5c and 5d).
\end{proof}

\begin{corollary}
  \label{cor:physicalMonoidalMulticat-are-coherent}
  Polarized \physicalMonoidalMulticategories{} are coherent; there exists at most one multimorphism between any distinctly typed objects of the free \physicalMonoidalMulticategory{} over some objects.
\end{corollary}
\begin{proof}
  This is now a consequence of \Cref{prop:messageTheoriesAreAlgebrasOfFreePolarized} and \Cref{theorem:coherentMessageTheories}. The derivations of a \messageTheory{} form the free polarized \physicalMonoidalMulticategory{}, but we have already shown that they are coherent.
\end{proof}

\begin{theorem}
  \label{th:polar-shuffles-are-the-free-polarized-physical-monoidal-multicategory}
  \PolarShuffles{} form the free polarized \physicalMonoidalMulticategory{} over a set of types.
\end{theorem}
\begin{proof}
  \PolarShuffles{} are coherent (\Cref{prop:polar-shuffles-are-coherent}), and polarized \physicalMonoidalMulticategories{} are coherent as well (\Cref{cor:physicalMonoidalMulticat-are-coherent}); this simplifies the proof because, to show that they coincide, we only need to show that a \polarShuffle{} between some types exists if an only if a multimorphism in the free polarized \physicalMonoidalMulticategory{} exists.

  If a multimorphism of a certain type exists in the free polarized \physicalMonoidalMulticategory{}, then it exists in all polarized \physicalMonoidalMulticategories{} and, in particular, in \polarShuffles{} (\Cref{th:polar:physicalmonoidal}).

  Let us prove the converse implication: if a \polarShuffle{} between some types does exist, then there is a multimorphism in the free polarized \physicalMonoidalMulticategory{} between these types. For this, we will use that a \polarShuffle can be always factored (not necessarily uniquely!) in the following way: \emph{(i)} we exchange positions inside each polar list to get them to their final relative position; \emph{(ii)} we use a series of spawnings, or polar shuffles $I → X^{∘} ⊲ X^{•}$; \emph{(iii)} we use a pure, non-polarized shuffle; and \emph{(iv)} some final linkings of type $X^{•} ⊲ X^{∘} → I$. This is easy to verify topologically: we can always `pull down the linkings' and `pull up' the spawnings, and we can always factor any shuffle in two parts -- a pure shuffle and a shuffle internal to each one of the components. For instance, consider the following example in \Cref{fig:factored-shuffle}, adapted from \Cref{fig:polar-shuffle-parallel} with an extra spawning.
  \begin{figure}[!ht]
    \centering
    \includegraphics[scale=0.40]{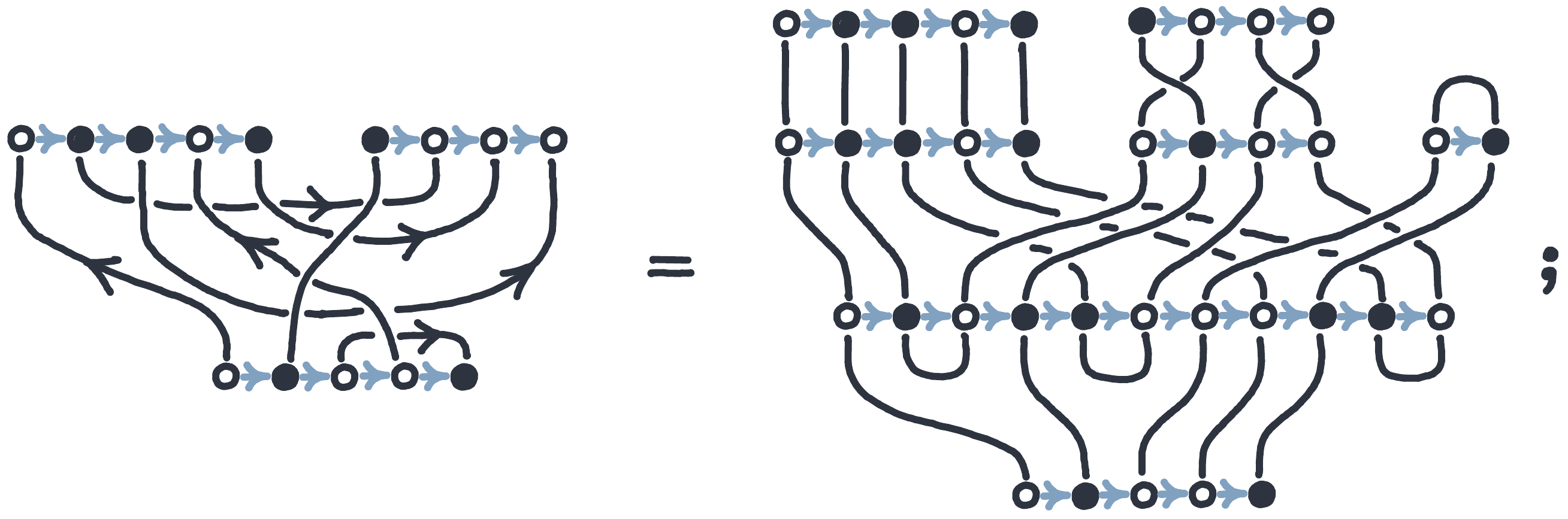}
    \caption{Factored polar shuffle.}
    \label{fig:factored-shuffle}
  \end{figure}

  Spawning, linking, shuffling, waiting and rushing are all operations on a \physicalMonoidalMulticategory{} (see \Cref{sec:message-theories}); so this proves that there is at least one multimorphism in the free \physicalMonoidalMulticategory{} with these types. We do not care about the specific choice of multimorphism because \polarShuffles{} are coherent (\Cref{prop:polar-shuffles-are-coherent}), and polarized \physicalMonoidalMulticategories{} are coherent as well (\Cref{cor:physicalMonoidalMulticat-are-coherent})
\end{proof}

\begin{corollary}
  \MessageTheories{} are the algebras of the \physicalMonoidalMulticategory{} of \polarShuffles{}.
\end{corollary}
\begin{proof}
  We will use that \polarShuffles{} are the free polarized \physicalMonoidalMulticategory{} (\Cref{th:polar-shuffles-are-the-free-polarized-physical-monoidal-multicategory}) and that \messageTheories{} are precisely the algebras of the free polarized \physicalMonoidalMulticategory{} over their objects (\Cref{prop:messageTheoriesAreAlgebrasOfFreePolarized}).
\end{proof}

\subsection{Bibliography}
The definition -- and the notation -- of polar shuffles takes inspiration from a different concept: Hughes' \emph{partial leaf functions} \cite{hughes12:simple}. Partial leaf functions are the $\mathbf{Int}$-construction -- the free compact closed category over a traced monoidal category -- applied to the category of finite sets and partial functions \cite{abramsky99}; here we follow a similar idea, but we work over finite sets and bijections, which are not traced. Not only the $\mathbf{Int}$-construction, but also the idea of shuffling two traces point back to game semantics \cite{mellies18:game}.

\newpage
\section{Processes versus Sessions}
\label{sec:processes-sessions}

\subsection{Processes of a message theory}
Inside of a \messageTheory{}, we call \emph{processes} the sessions that happen in two parts: \emph{(i)} first they ask for some indeterminate number of resources (possibly zero), $X₁^{∘},…Xₙ^{∘}$ and \emph{(ii)} then, they give some indeterminate number of resources (possibly zero), $Y₁^{•},…,Yₘ^{•}$. This simple definition builds a \symmetricMonoidalCategory{} inside any \messageTheory{}, and we will find a left adjoint to this construction.

\begin{proposition}
  \defining{linkProc}{}
  Let $𝕄$ be a \messageTheory{}. There exists a \symmetricMonoidalCategory{}, $\mathsf{Proc}(𝕄)$, whose objects are the lists of objects of the \messageTheory{}, $\Proc(𝕄)_{obj} = 𝕄^\ast_{obj}$, and whose morphisms are the sessions that first ask some resources and then provide some resources,
  $$\Proc(𝕄)(X₁ ⊗ … ⊗ Xₙ; Y₁ ⊗ … ⊗ Yₘ) = 𝕄(X^?_n,…X^?_1; Y^!_1,…,Y^!_m).$$
  Note how we reverse the order of inputs; this will make reasoning easier even if it is unnecessary: the monoidal category will be symmetric in any case.
\end{proposition}
\begin{proof}
  Let us define the composition and identities. Composition is given by the \polarShuffle{} that connects the middle outputs to inputs, see \Cref{fig:polar-process-composition}; because \messageTheories{} are algebras of \polarShuffles{}, such a \polarShuffle{} defines an operation, composition, that takes two sessions to a third one. 
  \begin{figure}[ht]
    \centering
    \includegraphics[scale=0.45]{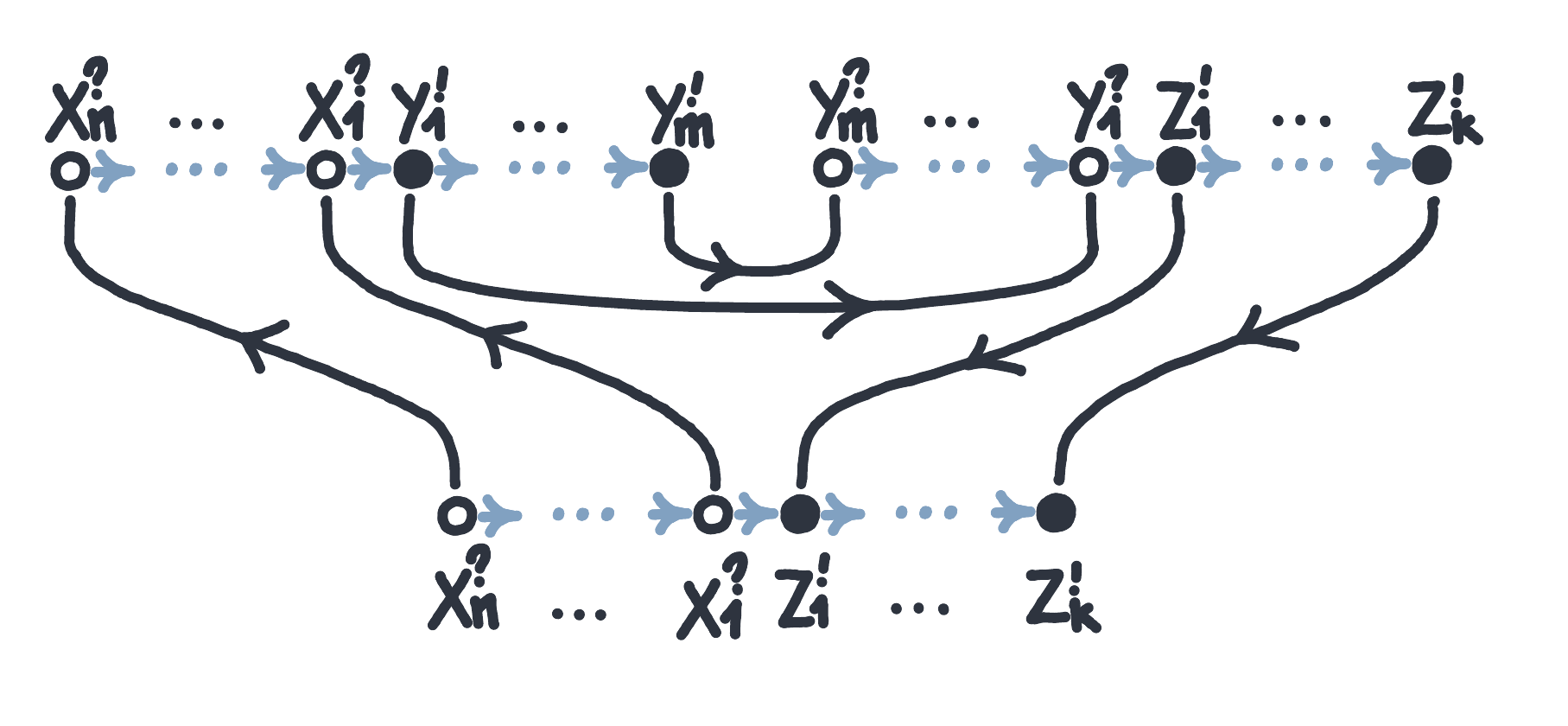}
    \caption{Composition of processes of a message theory.}
    \label{fig:polar-process-composition}
  \end{figure}

  Identities are created by spawning channels, see \Cref{fig:polar-process-identity}. Again, because \messageTheories{} are algebras of \polarShuffles{}, each \messageTheory{} must contain a constant given by the \polarShuffle{} that spawns a list of channels.
  Composition and identities are associative and unital: it can be checked from the definition that we are using the duality from spawning and connecting channels.
  \begin{figure}[ht]
    \centering
    \includegraphics[scale=0.45]{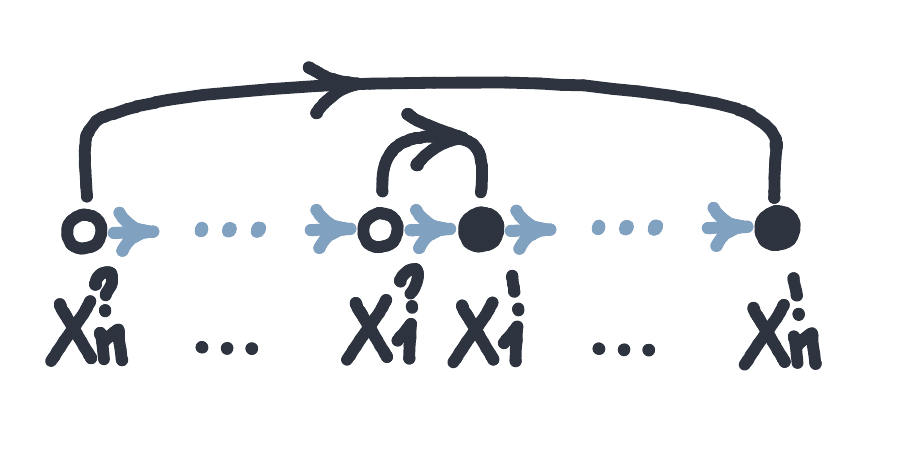}
    \caption{Identity process of a message theory.}
    \label{fig:polar-process-identity}
  \end{figure}

  Symmetries are given by the polar graph that spawns a channel for each one of the objects and then shuffles them so that the inputs and the outputs are divided in two blocks and positioned in reverse order, see \Cref{fig:polar-process-swap}.
  \begin{figure}[ht]
    \centering
    \includegraphics[scale=0.45]{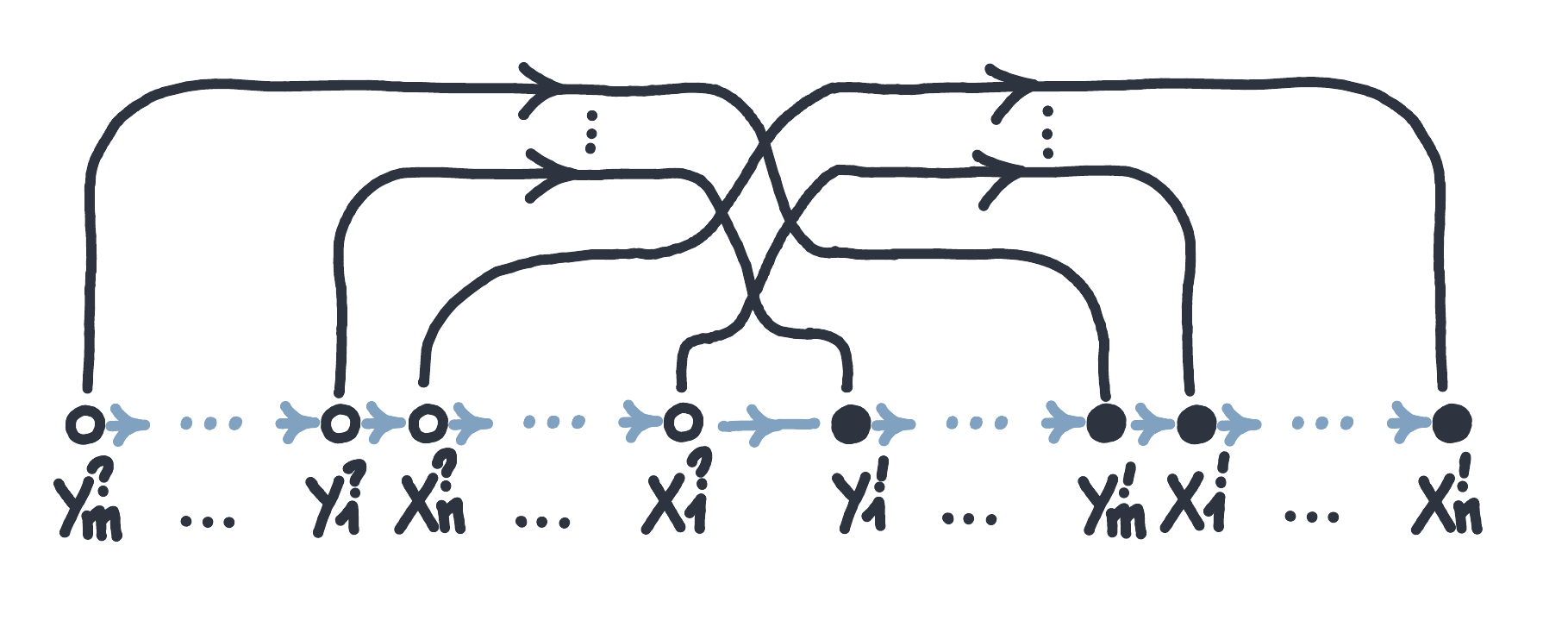}
    \caption{Symmetries of processes in a message theory.}
    \label{fig:polar-process-swap}
  \end{figure}

  Tensoring is given by the \polarShuffle{} that preserves all outputs and inputs but passes the outputs of a process pass the inputs of the other, see \Cref{fig:polar-process-tensor}. This operation is again associative and unital with the empty \polarShuffle{} that represents the monoidal unit.
  \begin{figure}[ht]
    \centering
    \includegraphics[scale=0.45]{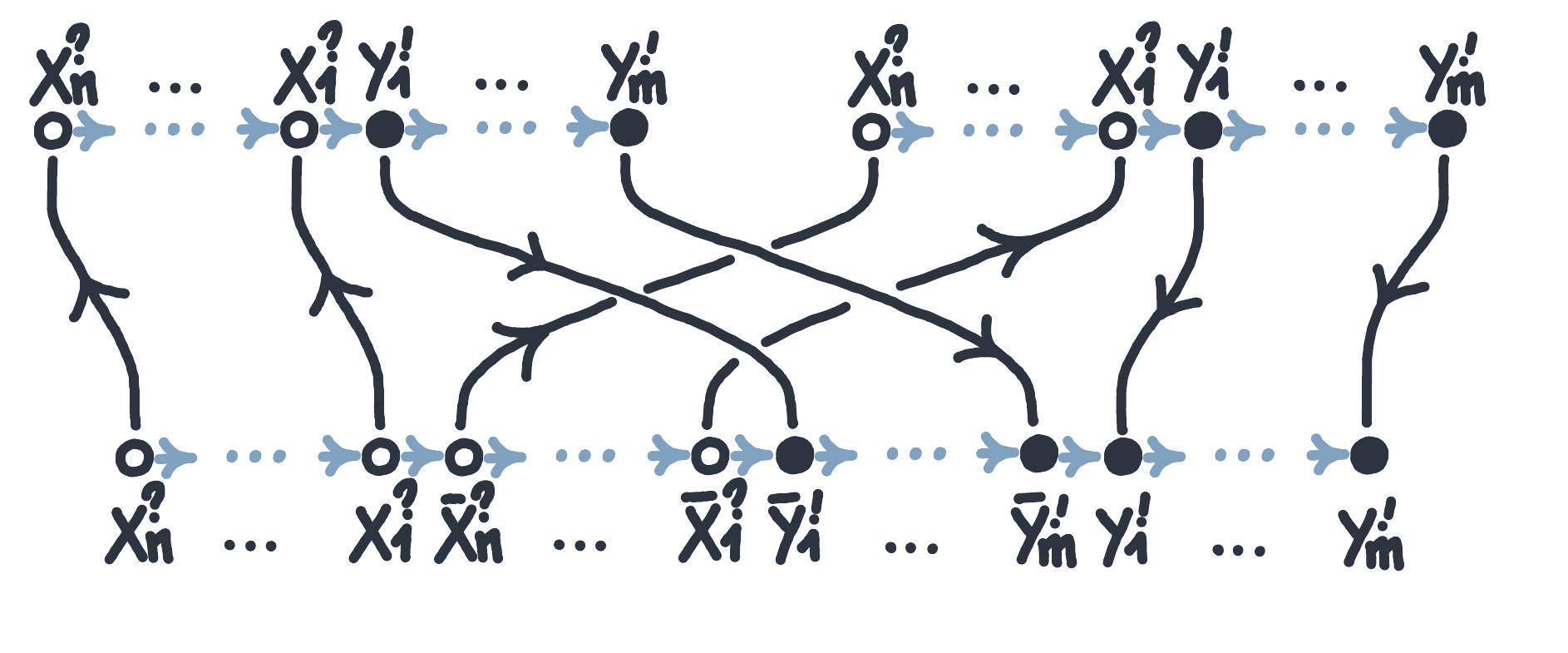}
    \caption{Tensor of processes in a message theory.}
    \label{fig:polar-process-tensor}
  \end{figure}
  
  It concludes the proof to check, by following the connections of the \polarShuffles{}, that tensoring behaves functorially with composition so that the interchange law of monoidal categories holds.
\end{proof}

\begin{proposition}
  The construction of the \symmetricMonoidalCategory{} of processes extends to a functor 
  $$\Proc ፡ \mathbf{Msg} → \SymMonCatStr.$$
\end{proposition}
\begin{proof}
  We will show that any message functor $F ፡ 𝕄 → ℕ$ induces a strict symmetric monoidal functor $\Proc(F) ፡ \Proc(𝕄) → \Proc(ℕ)$. Because the category of processes is freely generated on objects, it suffices to explain that the object $X ∈ 𝕄_{ob}$ is sent to the object $FX ∈ N_{ob}$. On morphisms, we already have a map
  $$F ፡ 𝕄(Xₙ^{?},…,X₁^{?},Y₁^{!},…,Yₘ^{!}) → ℕ(FXₙ^{?},…,FX₁^{?},FY₁^{!},…,FYₘ^{!})$$
  that gets reinterpreted as a map
  $$\Proc(𝕄)(X₁ ⊗ … ⊗ Xₙ;Y₁ ⊗ … ⊗ Yₘ) → \Proc(ℕ)(F(X₁ ⊗ … ⊗ Xₙ);F(Y₁ ⊗ … ⊗ Yₘ)).$$
  Composition, identities and tensoring are operations constructed as \polarShuffles{}, and so they must be preserved by a message functor; this means that $\Proc(F)$ becomes a strict monoidal functor.
\end{proof}

\subsection{Sessions of a process theory}

We will now construct message sessions over an arbitrary \processTheory{}, and we will do so in a minimalistic theory. Message passing  consists of two effects: \emph{sending} and \emph{receiving}. \PremonoidalCategories{} already are a framework for effectful computation in process theories, so we employ them here.

\begin{definition}
  \defining{linkSessions}{}
  \label{def:sessions}
  The \effectfulCategory{} of \emph{sessions} over a strict \symmetricMonoidalCategory{} $ℂ$ is the \effectfulCategory{} $ℂ → \Session(ℂ)$ generated by 
  \begin{enumerate}
    \item all of the morphisms of the original \monoidalCategory{}, $ℂ$, quotiented by the equations of the original monoidal category, as pure morphisms;
    \item and a pair of \emph{send} and \emph{receive} generators for each object $X ∈ ℂ_{obj}$ imposing no further equations. We write these generators as $(∘)_X ፡ X → I$ and $(•)_X ፡ I → X$.
  \end{enumerate}

  \begin{figure}[ht]
    \centering
    \includegraphics[scale=0.45]{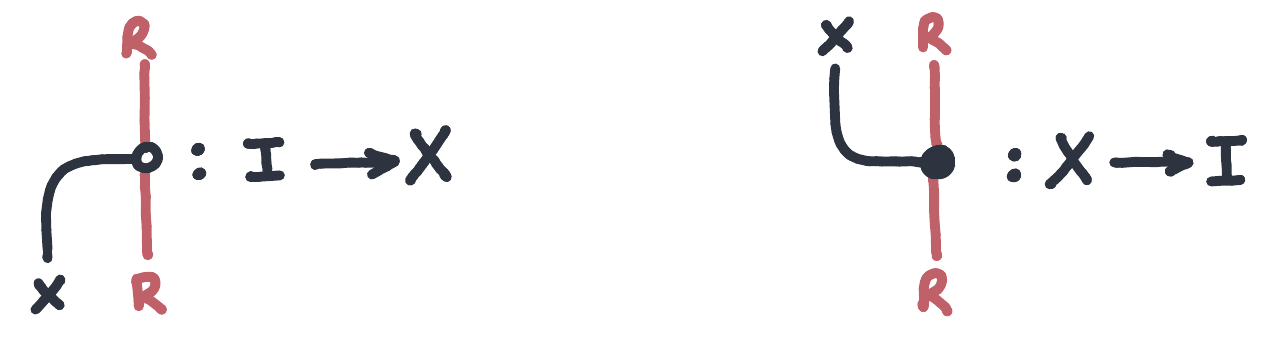}
    \caption{Session runtime generators.}
    \label{fig:session-runtime-generators}
  \end{figure}
\end{definition}

This naive theory of message passing on top of a \monoidalCategory{} may remind us of the combs that we were studying before: instead of using incomplete diagrams, we are marking the interchanges explicitly now. This intuition can be made formal: we will now prove that combs of type $X^{‽}_1 ⊲ … ⊲ X^{‽}_n → \pbiobj{A}{B}$ correspond to sessions $A → B$ where the events are exactly $X^{‽}_1, …, X^{‽}_n$, and happen in that specific order. First, note that we can define a sequence of events associated with a particular session.

\begin{definition}
  The \emph{sequence of events} of a session is the list of effectful generators obtained by following only the effectful wire on the diagram. Formally, it is defined by structural induction over the premonoidal category of sessions as follows.
  \begin{enumerate}
    \item It is the empty list for a pure morphism.
    \item It is invariant to whiskering.
    \item It contains a single element $[X^{∘}]$ for each generator $(∘)_X ፡ I → X$.
    \item It contains a single element $[X^{•}]$ for each generator $(•)_X ፡ X → I$.
    \item It concatenates the lists for a composition of morphisms.
  \end{enumerate}
  It becomes straightforward to check that the sequence of events is well-defined. We write $\Session(A;B)[X^{‽}_1, …, X^{‽}_n]$ for the set of sessions $A → B$ with a sequence of events $X^{‽}_1, …, X^{‽}_n$.
\end{definition}

\begin{proposition}[Combs are sessions]
  \label{prop:combsaresessions}
  Sessions from $A$ to $B$ with a sequence of events $X^{‽}_1, …, X^{‽}_n$ are in bijective correspondence with combs with the same events happening sequentially,
  $$\Session(A ; B)[X^{‽}_1, …, X^{‽}_n] ≅ \mLens\left( X^{‽}_1 ⊲ … ⊲ X^{‽}_n; \biobj{A}{B} \right).$$
\end{proposition}
\begin{proof}
  We proceed by structural induction over the presentation of the cteogry of sessions. The base case consists of a morphism that is pure: by definition, those are the combs $\mLens(𝖭; \biobj{A}{B})$.

  The inductive case considers a morphism $A → B$ with at least one first occurrence of the effectful generators, $(∘)_X$ or $(•)_X$. We assume without loss of generality that this is $(•)_X ፡ I → X$, so its sequence of events is $X^{•},Γ$ -- the other case is analogous.
  We consider it as a string diagram for a \monoidalCategory{}, quotiented by the equations of the base \monoidalCategory{} $ℂ$, the symmetries and only up to interchange by isotopy. 
  We may split the diagram into two parts (as in \Cref{fig:split-session}); we leave everything before the generator $(•)_X$ to one side, $f ∈ ℂ(A; X ⊗ M)$, and everything after the generator to the other side $g ∈ \Session(M,B)[Γ]$. This procedure constructs the following comb,
  $$∫^{M ∈ ℂ} \withPoint{f}{ℂ(A; M ⊗ X)} × \withPoint{g}{\Session(A ; B)[Γ]}.$$
  \begin{figure}[ht]
    \centering
    \includegraphics[scale=0.35]{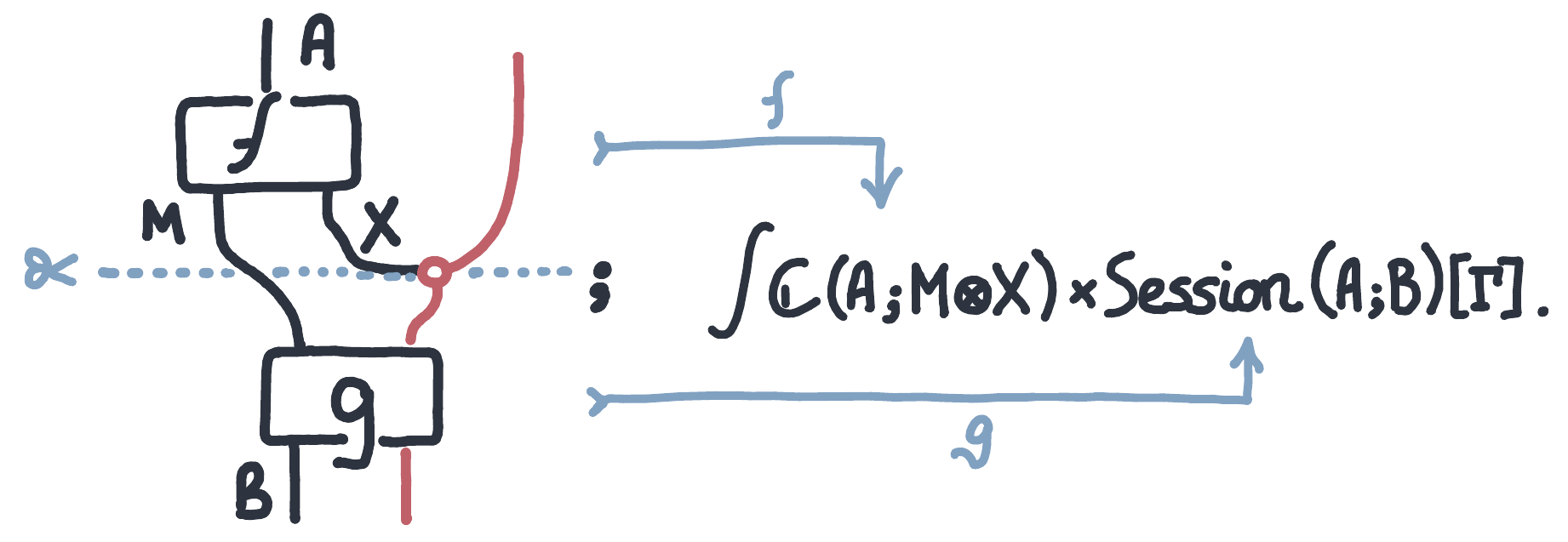}
    \caption{Splitting the diagram of a session.}
    \label{fig:split-session}
  \end{figure}
  
  Of course, the usual problem with this kind of definitions that split a string diagram is that we need to prove that they are well-defined. We need to show that this definition is invariant to isotopy; we do so by cases: \emph{(i)} if the isotopy interchanges two boxes before the cutting line, then there were two pure morphisms and it is captured by an equation of $ℂ$; \emph{(ii)} if the isotopy interchanges two boxes after the cutting line, then it is defining an equation of sessions; \emph{(iii)} if the isotopy interchanges two morphisms across the cutting line, then it is captured by dinaturality.
  
  At the same time, note that these are exactly all of the equations imposed to combs: those of the original \symmetricMonoidalCategory{} and \dinaturality{}; so the correspondence is bijective. Finally, by the induction hypothesis,
  \begin{align*}
     \Session(A;B)[X^{•},Γ] 
    & ≅ {\textstyle ∫}^{M ∈ ℂ}\ ℂ(A; X ⊗ M) × \Session(A;B)[Γ] \\
    & ≅ {\textstyle ∫}^{M ∈ ℂ}\  ℂ(A; X ⊗ M) × \mLens\left(Γ;\biobj{A}{B}\right) \\
    & ≅ \mLens\left(X^{•} ⊲ Γ;\biobj{A}{B}\right),
  \end{align*}
  which concludes the proof.
\end{proof}

The next step is to show that sessions over a \processTheory{} actually define a \messageTheory{}. For this, we will only need sessions with no inputs or outputs: we write $\Session[Γ]$ for $\Session(I;I)[Γ]$.

\begin{proposition}
  Sessions over a strict \symmetricMonoidalCategory{}, $\Session(ℂ)$, form a \messageTheory{}.
  This construction extends to a functorial assignemnt
  $$\Session ፡ \SymMonCatStr → \Msg.$$
\end{proposition}
\begin{proof}
  We will show that sessions over a strict \symmetricMonoidalCategory{}, $\Session(ℂ)$, form an algebra for the \multicategory{} of \polarShuffles{}. Consider a family of sessions, $sᵢ ∈ \Session(Γᵢ)$; and consider a \polarShuffle{}
  $$p ∈ \pShuf(Γ₁,…,Γₙ; Δ),$$
  we need to construct a new session of type $Δ$.

  The construction follows a topological intuition: consider the hypergraph defining the string diagrams of the sessions; and consider at the same time the acyclic graph defined by the \polarShuffle{}. We glue the string diagram of each $sᵢ$, along its runtime wire, to the inputs and outputs, $Γᵢ$, of the \polarShuffle{}, removing these nodes on the process. The crucial step happens now: we have a graph containing nodes of the \premonoidalCategory{} of sessions, and it has been constructed by gluing acyclic graphs along linear boundaries -- it must be acyclic, and it must be a string diagram.

  We define the composition of the sessions $sᵢ ∈ \Session(Γᵢ)$ along the \polarShuffle{} $p ∈ \pShuf(Γ₁,…,Γₙ; Δ),$ to be the session represented by the string diagram here obtained. This assignment preserves the composition of \polarShuffles{} -- which is also defined topologically by gluing -- and thus it determines an algebra.

  Finally, the assignment is functorial with respect to the base monoidal category because \emph{(i)} the construction of the sessions is functorial and \emph{(ii)} we only apply operations of a symmetric \premonoidalCategory{} when we build a string diagram over this \premonoidalCategory{} of sessions.
\end{proof}

\begin{example}
  Consider two morphisms, $f ፡ A → M ⊗ X$ and $g ፡ M ⊗ Y → B$, determining a session of type $[A^{∘},X^{•},Y^{∘},B^{•}]$; and consider two other morphisms, $h ፡ X ⊗ U → N$ and $k ፡ N → Y ⊗ V$, determining a session of type $[X^{∘},U^{∘},V^{•},Y^{•}]$.
  They can compose along the \polarShuffle{} we defined in \Cref{fig:polar-shuffle-example-encoding}; the result is in \Cref{fig:polar-process-swap}.
  \begin{figure}[ht]
    \centering
    \includegraphics[scale=0.35]{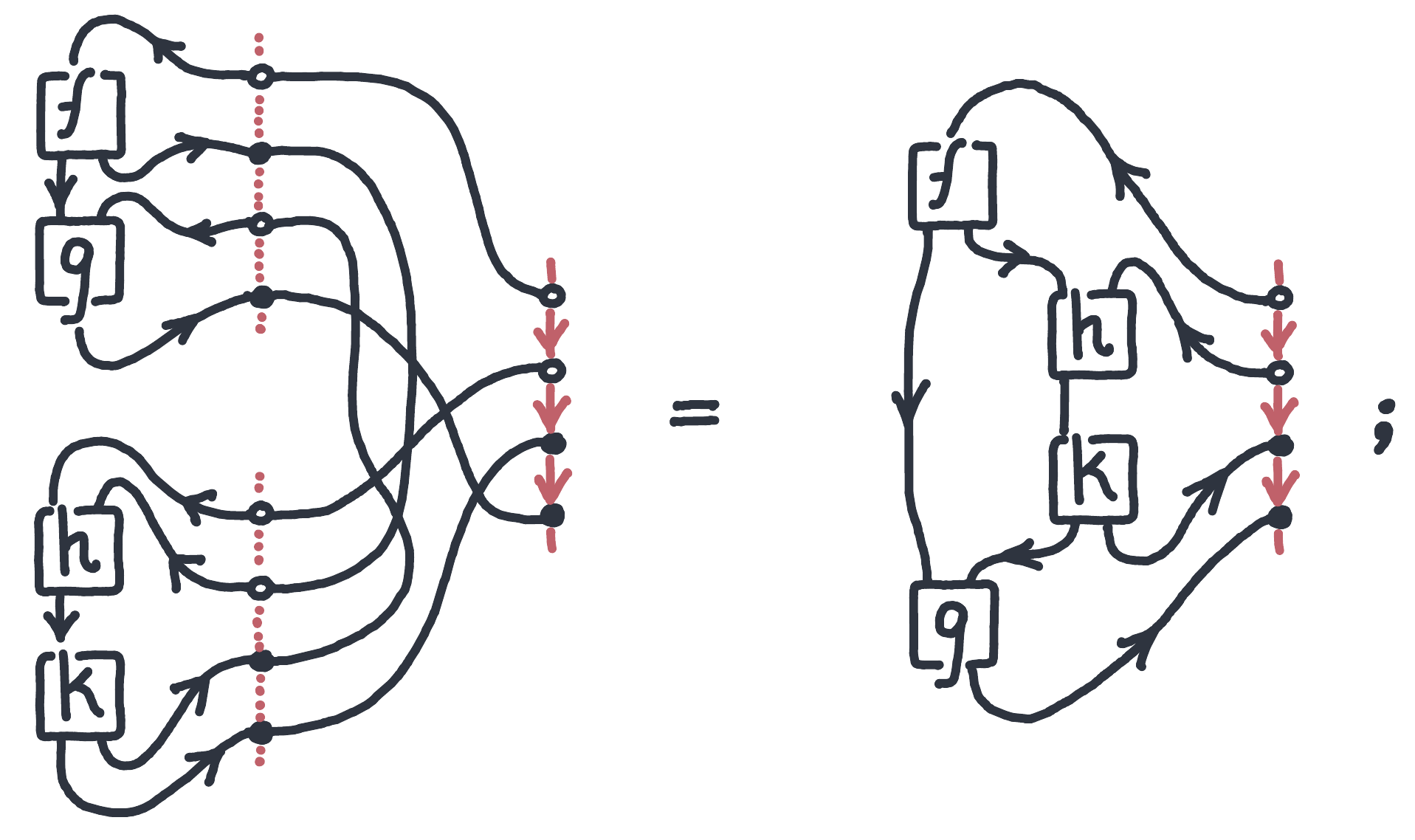}
    \caption{Two sessions compose along a polar shuffle.}
    \label{fig:polar-process-swap}
  \end{figure}
\end{example}

\subsection{Sessions versus Processes}
The final result of this section is to prove that \emph{sessions} over a \processTheory{} are the free \messageTheory{} over a \symmetricMonoidalCategory{}.

\begin{lemma}
  There exists a strict symmetric monoidal functor $$\inProc ፡ ℂ → \Proc(\Session(ℂ))$$ that includes a \monoidalCategory{} in the processes of its \messageTheory{}.
\end{lemma}
\begin{proof}
  The functor will act as the identity on objects.
  We already know that combs are sessions (\Cref{prop:combsaresessions}), and we can use this fact to construct the assignment on
  morphisms.
  \begin{align*}
    \Proc(\Session(ℂ)(A;B)) = \Session(ℂ)[A^{∘}, B^{•}] 
    ≅  \mLens(A^{∘}, B^{•}; \biobj{I}{I}) 
     ≅ ℂ(A;B).
  \end{align*}
  It only remains to show that this assignment defines a strict symmetric monoidal functor: we need to show that it preserves composition, tensoring, identities and symmetries. This is straightforward, as we only need to check that the operations that we defined for the \processTheory{} of a \messageTheory{}, $\Proc(𝕄)$, correspond to the operations of a \symmetricMonoidalCategory{}.
  Let us explicitly check composition, $\inProc(f) ⨾ \inProc(g) = \inProc(f ⨾ g)$, in \Cref{fig:polar-composition-works}, the rest follow a similar pattern.
  \begin{figure}[ht]
    \centering
    \includegraphics[scale=0.35]{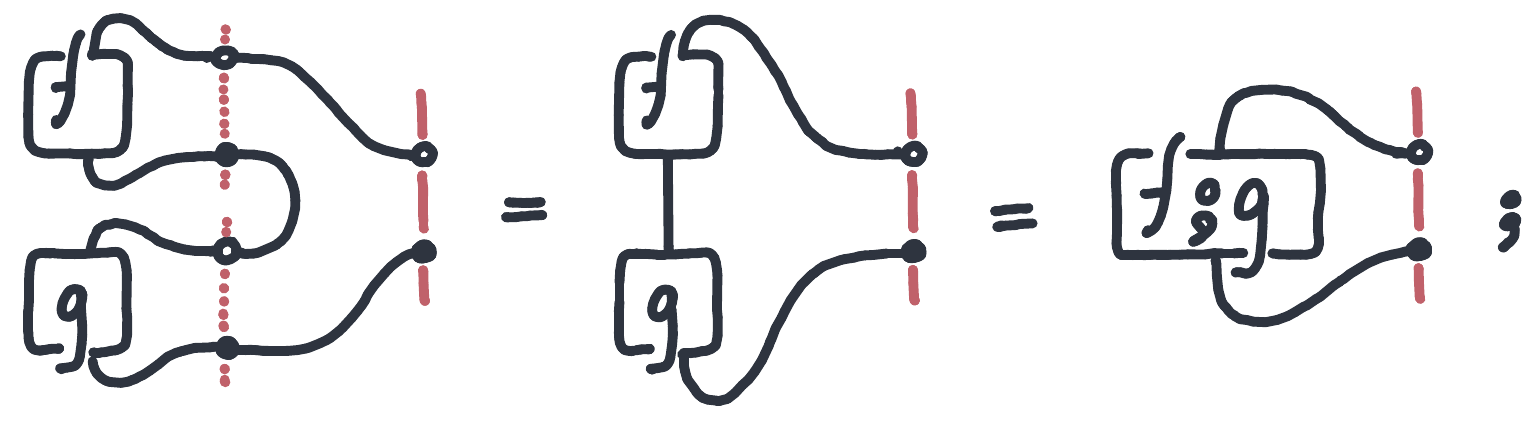}
    \caption{The inclusion of processes preserves composition.}
    \label{fig:polar-composition-works}
  \end{figure}

  Checking the rest of the cases concludes the construction.
\end{proof}

\begin{theorem}
  \label{th:sessions-vs-processes}
  Sessions and processes form an adjunction, $\Session ⊣ \Proc$; where sessions, $\Session ፡ \SymMonCatStr → \Msg$, construct the free \messageTheory{} over a \symmetricMonoidalCategory{}, and where processes, $\Proc ፡ \Msg → \SymMonCatStr$, construct the cofree \symmetricMonoidalCategory{} over a \messageTheory{}.
\end{theorem}
\begin{proof}
  Consider a strict \symmetricMonoidalCategory{}, $ℂ$, and a \messageTheory{}, $𝕄$, endowed with a strict symmetric monoidal functor $F ፡ ℂ → \Proc(𝕄)$. We will construct a message functor $F^{♯} ፡ \mLens(ℂ) → 𝕄$ and prove that it is the unique one satisfying $\inProc ⨾ \Proc(F^{♯}) = F$.

  Let us show that such a message functor, if it were to exist, would be unique. Firstly, the image on message types is already determined to be $F^{♯}(A) = F(A)$. Secondly, the image on sessions consisting on a single morphism, $\inProc(f) ፡ [A^{∘},B^{•}]$, is determined, $F^{♯}(\inProc(f)) = F(f) ፡ [FA^{∘},FB^{•}]$. We will show now that this reasoning can be extended to all sessions:
  we know that sessions of type $[X_1^{‽}, … , X_n^{‽}]$ are combs (\Cref{prop:combsaresessions}) of type
  $$(f₀ | … | fₙ) ፡ X₁^{‽} ⊲ … ⊲ X^{‽}ₙ → \pbiobj{I}{I}.$$
  This determines their image: combs can be factored as the composition, in the \messageTheory{} of multiple sessions consisting of a single morphism (\Cref{fig:comb-to-session}).
  Accordingly, their image, $F^{♯}(f₀ | … | fₙ)$, should be the composition, on the \messageTheory{}, of these pieces (\Cref{fig:comb-to-session}).
  \begin{figure}[ht]
    \centering
    \includegraphics[scale=0.35]{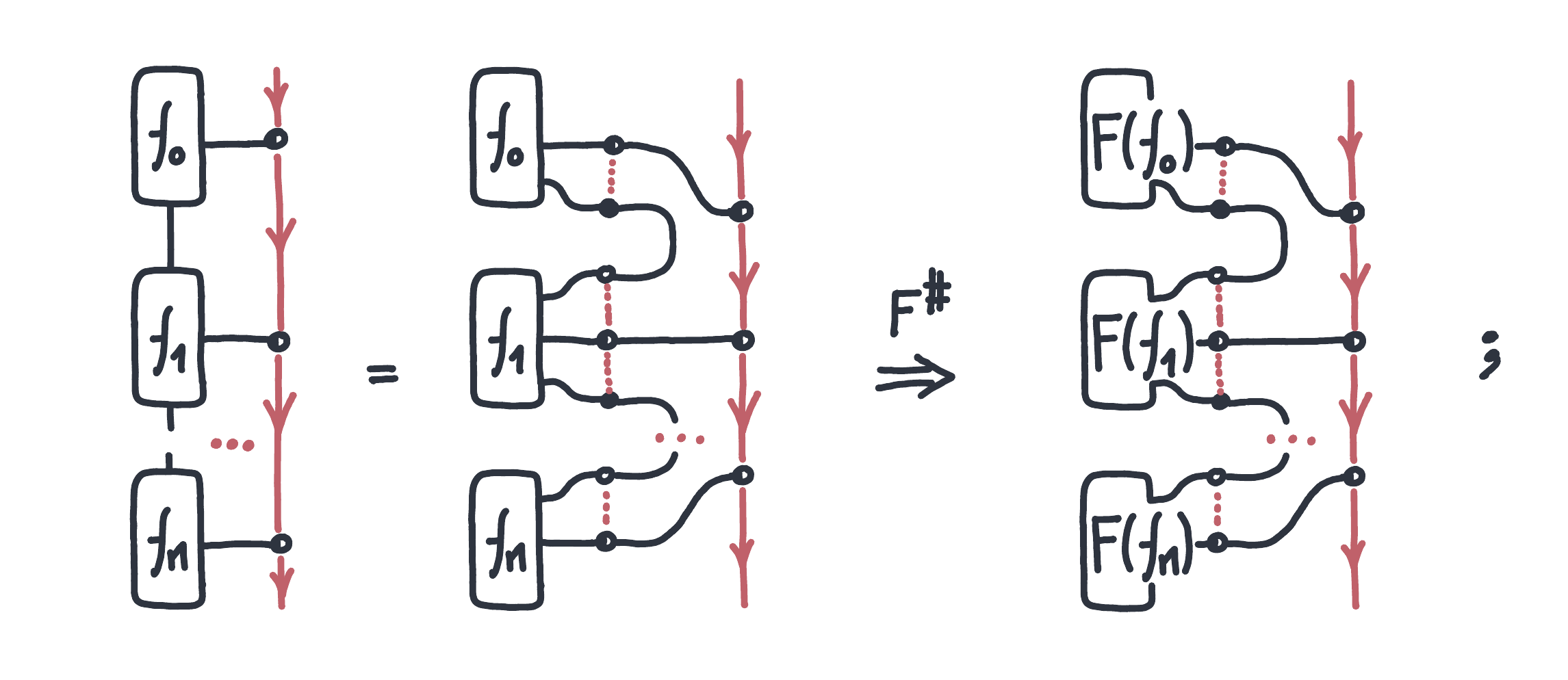}
    \caption{Image of a comb under the message functor.}
    \label{fig:comb-to-session}
\end{figure}

  Let us now show that we have constructed a well-defined assignment. Our construction should preserve the \dinaturality{} equivalence relation imposed to combs. This happens, indeed, and the proof simply checks that the images of two combs,
  $$(f₀ ⨾ (\id ⊗ h₀) | f₁ ⨾ (\id ⊗ h₁) | … | fₙ)  = (f₀ | (\id ⊗ h₀) ⨾ f₁ |  … | (\id ⊗ h_{n-1}) ⨾ fₙ),$$
  are equal (\Cref{fig:comb-assingment}).
  \begin{figure}[ht]
    \centering
    \includegraphics[scale=0.35]{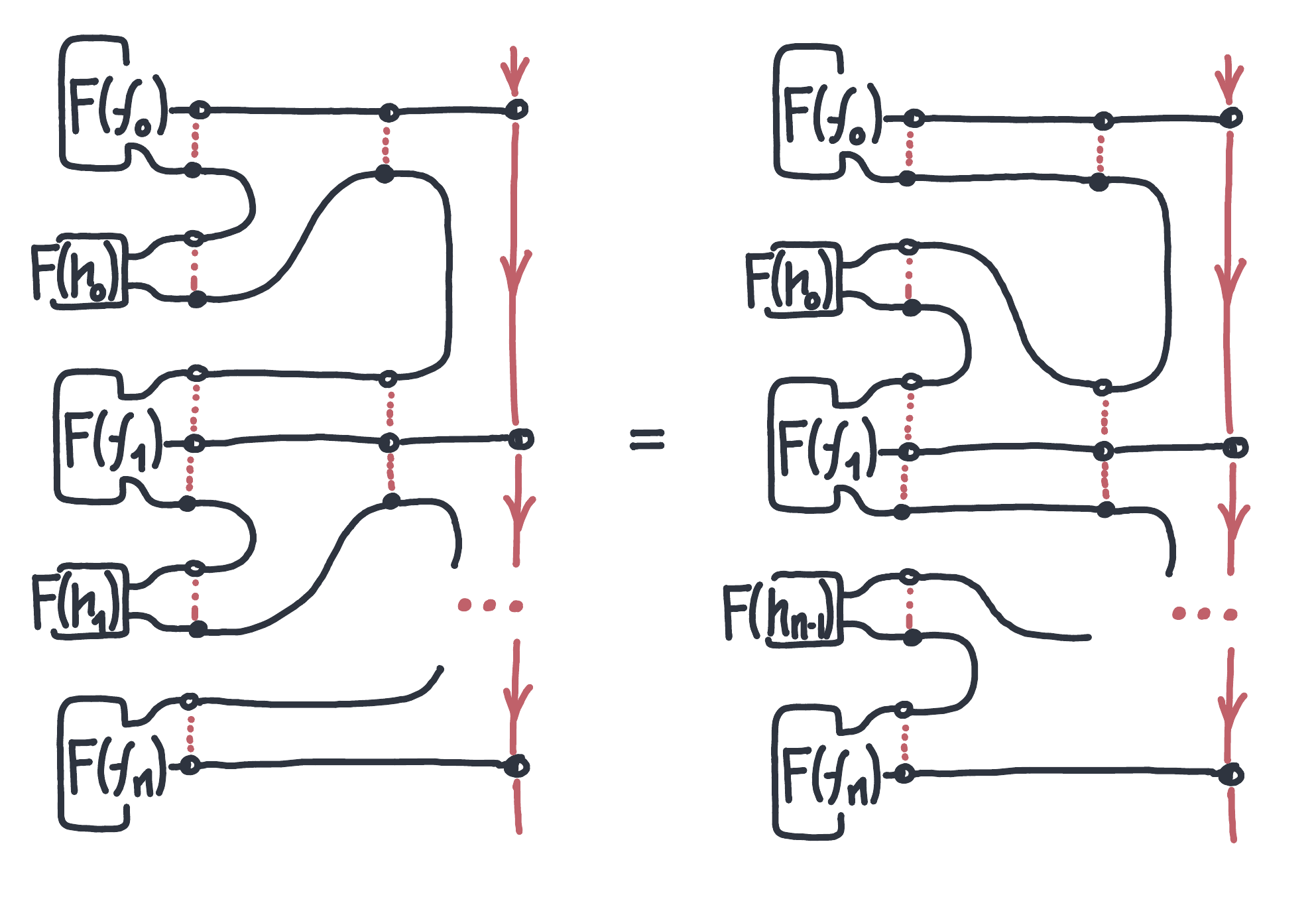}
    \caption{The assignment of combs to sessions preserves dinaturality.}
    \label{fig:comb-assingment}
  \end{figure}
  We have constructed a well-defined assignment on sessions. Finally, we need to check that $F^{♯}$ is a message functor preserving all of the operations determined by \polarShuffles{}. In the theory of combs, applying a \polarShuffle{} corresponds to a rewiring into another comb; applying the \polarShuffle{} in $𝕄$ must result in the same rewiring of the pieces forming the comb -- which, as we have already shown, is precisely the image of that comb. This forces $F^{♯}$ to preserve the application of a \polarShuffle{}.

  We have shown that $F^{♯}$ is indeed a message functor and that it is the only possible one satisfying $\inProc ⨾ \Proc(F^{♯}) = F$.
\end{proof}

\subsection{Example: One-Time Pad, as a Message Session}
  Let us come back to \Cref{ex:one-time-pad}, where we discussed a decomposition of the one-time pad. 
  We now know that there is an adjunction between \symmetricMonoidalCategories{} and \messageTheories{}, let us use it to provide semantics to the decomposition of the one-time pad example.

  The theory for the one-time pad problem can be expressed in a message theory $𝕆$ where we have a single object generator for the type of a message, $X$, and a single session generator for each one of the actors.
  \begin{enumerate}
    \item $\Stage ፡ X^{∘} ⊲ X^{•} ⊲ X^{•},$
    \item $\Bob ፡ X^{•} ⊲ X^{∘} ⊲ X^{•},$
    \item $\Alice ፡ X^{∘} ⊲ X^{∘} ⊲ X^{•},$
    \item $\Eve ፡ X^{∘} ⊲ X^{•}.$
  \end{enumerate}
  We will interpret these generators in the free \messageTheory{} over the category of finite sets and stochastic maps: thanks to the adjunction, we know that, if we could interpret each one of the components presenting the category of finite sets in any \messageTheory{}, then we can interpret this whole example inside it.

  We already gave an interpretation to each one of the components in terms of combs. We rewrite now the example explicitly separating each one of the parts that form it (\Cref{fig:opt-session}).

  \begin{proposition}
    The session describing the one-time pad protocol is equal to a session where \Alice{} and \Bob{} communicate the message directly and \Eve{} attacks a signal representing pure noise.
    \begin{figure}[ht]
      \centering
      \includegraphics[scale=0.35]{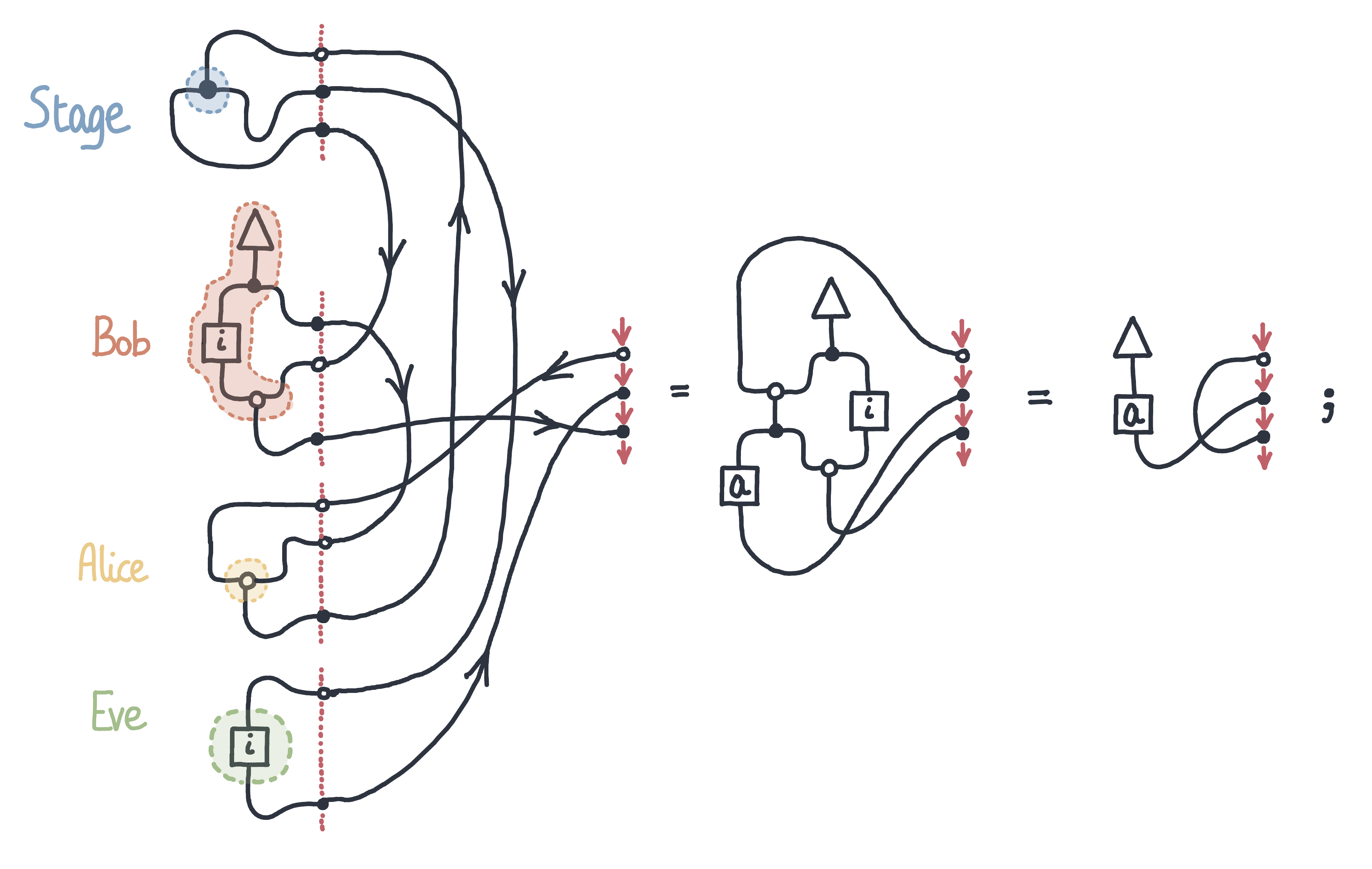}
      \caption{One-time pad, complete session.}
      \label{fig:opt-session}
  \end{figure}
  \end{proposition}
  \begin{proof}
    Evaluating the session that describes the one-time pad example using the components described before, in \Cref{ex:one-time-pad}, obtains the following \polarShuffle{} applied to multiple combs. Evaluating the \polarShuffle{}, as in \Cref{fig:opt-session}, produces the desired result.
  \end{proof}
  
  \begin{remark}
    This discussion is not restricted to the modularity of the string diagrams: it affects the modularity of the code itself. 
    Recall that we have a notation for sessions and \polarShuffles{}; we use it in \Cref{fig:otp-wiring} to write the one-time pad.
    \begin{figure}[ht] 
      \centering
      \includegraphics[scale=0.15]{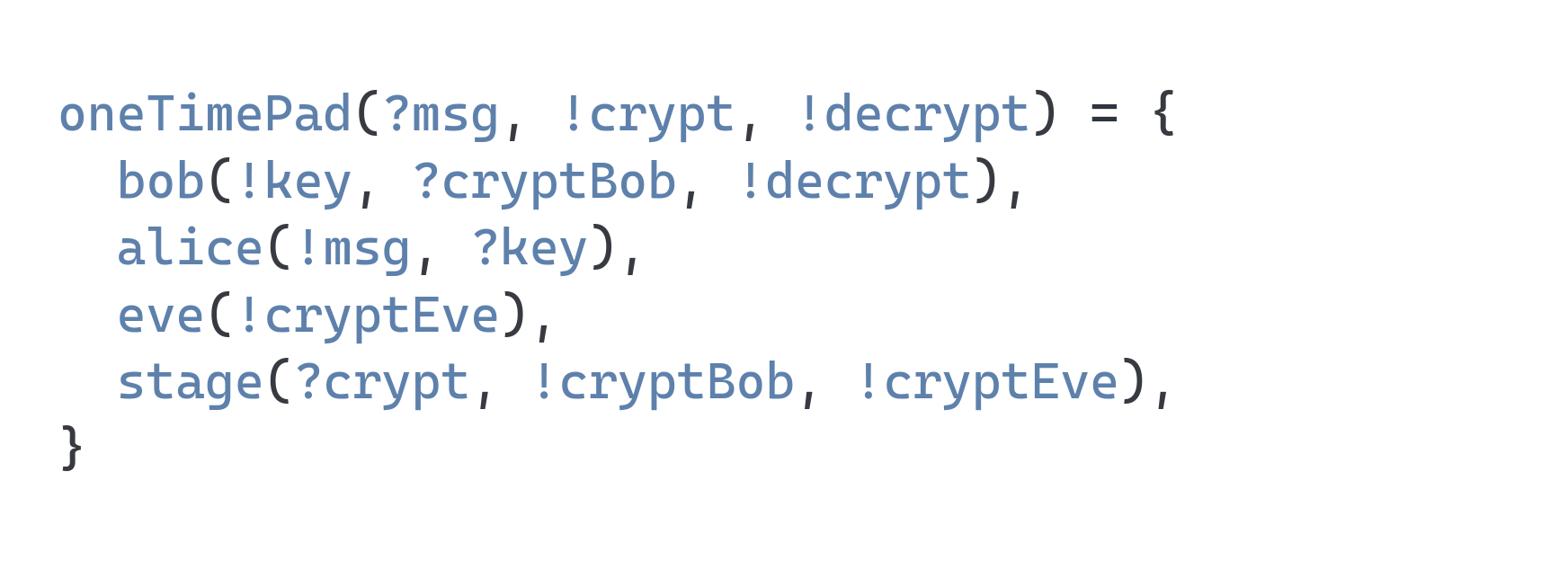}
      \caption{Notation for the one-time pad session.}
      \label{fig:otp-wiring}
    \end{figure}

    At the same time, we have second a notation for sessions: sessions are ultimately morphisms of an \effectfulCategory{}, so we can use \emph{do-notation} without the interchange axiom to represent them. The sending and receiving effects can be written as $(!/?)$ respectively. Lenses are tuples of morphisms, and they can be represented in do-notation using that exact characterization. The following \Cref{fig:otp-do-notation} shows a modular implementation of the one-time pad that separates each one of the actors into a different module.
    \begin{figure}[ht] 
      \centering
      \includegraphics[scale=0.5]{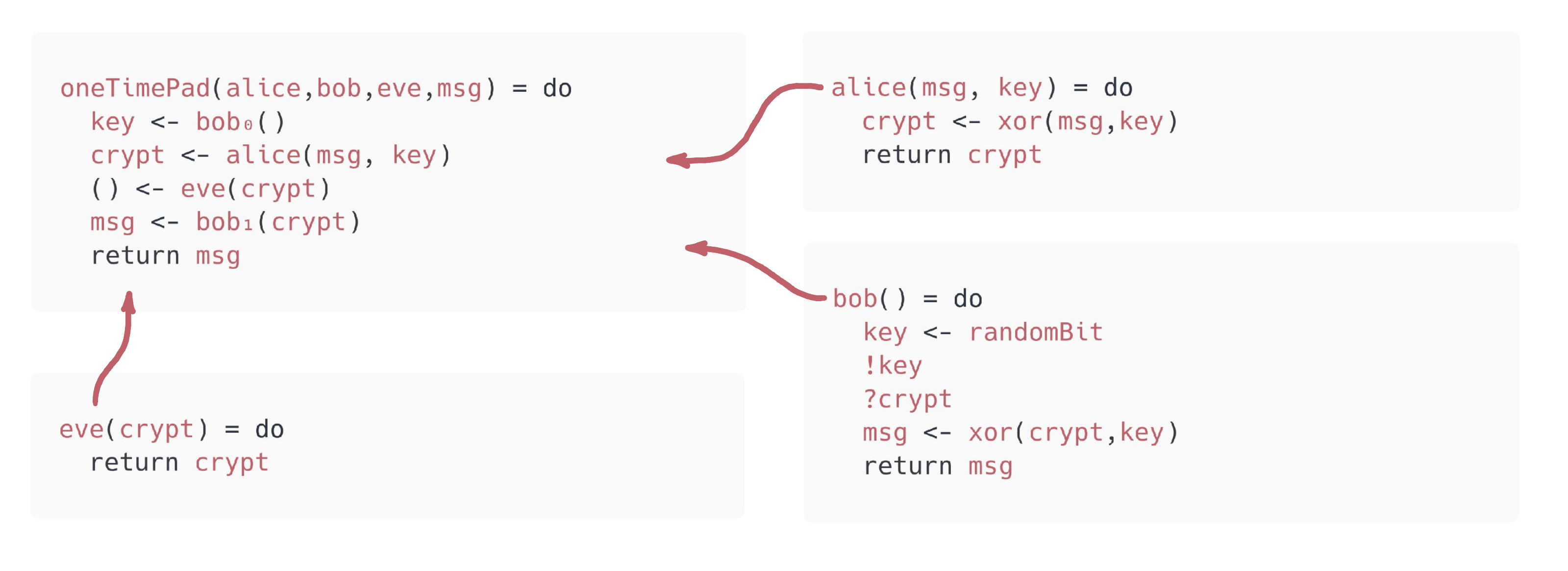}
      \caption{Do-notation for the one-time pad.}
      \label{fig:otp-do-notation}
    \end{figure}

    Which notation should we settle for? It seems that both interplay nicely together: the best way of writing a message session seems to be to write its underlying \polarShuffle{}, as in \Cref{fig:otp-wiring}, while the best way of writing processes may be the usage of do-notation as in \Cref{fig:otp-do-notation}, which is well-known and imposes a human-readable order on the operations.
  \end{remark}

\subsection{Case Study: Causal versus Evidential Decision Theories}\label{sec:casestudy}
Leibniz's dream was to see philosophical debates reduced to mathematical calculation, to have a formal language for decision theory and an algorithm to solve any dispute.
\begin{quote}
  "[...] if controversies were to arise, there would be no more need of disputation between two philosophers than between two calculators. For it would suffice for them to take their pencils in their hands and to sit down at the abacus, and say to each other (and if they so wish also to a friend called to help): \emph{calculemus} (let us calculate)."
\end{quote}

However, our modern decision theory seems far from this dream. For instance, Monty Hall's problem caused famous controversies and confident blunders of some experts \cite{vosSavant} while being relatively simple to describe. It could seem that the passage from the statement to its formal encoding is more of an art than a science.

Let us try to understand one of these debates: \emph{causal} versus \emph{evidential} decision theory on Newcomb's problem \cite{nozick69,ahmed14:evidential,yudkowsky:soares:17}. We will use \messageTheories{} to set up the scene and partial Markov categories to compute the solution.

\begin{definition}
  Newcomb’s problem \cite{nozick69} is a famous decision problem
  that sets apart Evidential and Causal Decision Theory. An
  agent (\iconAgent{}) is in front of two boxes: a transparent box filled with
  1\euro{} and an opaque box (\iconBoxes{}). The agent is given the choice between
  taking both boxes (two-boxing, $\mathsf{T}$) and taking just the opaque
  box (one-boxing, $\mathsf{O}$). However, the opaque box is controlled by
  a ``perfectly accurate'' predictor (\iconPredictor{}). The predictor placed 1000\euro{} in the
  opaque box if it predicted that the agent would one-box and left it
  empty otherwise. The agent knows this. Which action should
  the agent choose?
  \begin{figure}[ht] 
    \centering
    \includegraphics[scale=0.35]{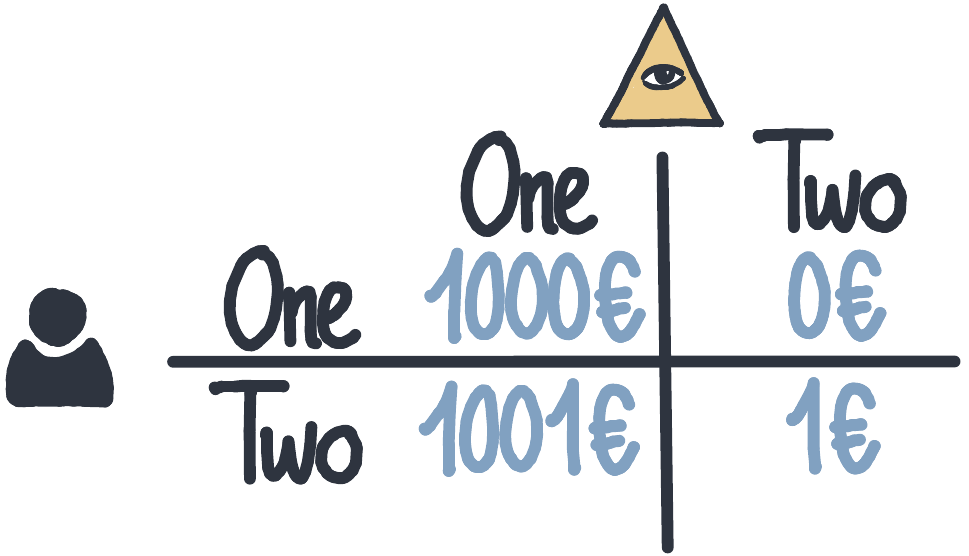}
    \caption{Newcomb's problem: table of utilities.}
    \label{fig:newcombs-problem}
  \end{figure}
\end{definition}

At the risk of oversimplifying, most philosophers are divided in two schools \cite{ahmed14:evidential,yudkowsky:soares:17}. Those that follow \emph{causal decision theory} would claim that no matter what the predictor does, the lower row of the table in \Cref{fig:newcombs-problem} contains strictly more utility; they prescribe \emph{two-boxing}. Those that follow \emph{evidential decision theory} claim that, because the predictor is omniscient, \emph{one-boxing} is the only way of ensuring the biggest prize is on the box.

The analysis of the problem starts by dividing it into different parties: \emph{(i)} the agent (\iconAgent{}) must only make a choice on whether to one-box or two-box; \emph{(ii)} the stage (\iconBoxes{}) takes the choice of the agent, the prediction of the predictor, broadcasts the choice of the agent and computes the final utility of the agent, and  \emph{(iii)} the predictor (\iconPredictor{}) sends a prediction and, only afterwards, can see the choice of the agent. Let us call $X = \{\mathsf{O}, \mathsf{T}\}$ to the set containing \emph{one-boxing} or \emph{two-boxing}; we are claiming to have three elements of a \messageTheory{}: the agent, 
$(\iconAgent{}) ፡ X^{•}$; the predictor, $(\iconPredictor) ፡ X^{•} ⊲ X^{∘}$; and the stage, $(\iconBoxes) ፡ X^{∘} ⊲ X^{∘} ⊲ X^{•} ⊲ X^{•}$.
\begin{figure}[ht] 
  \centering
  \includegraphics[scale=0.35]{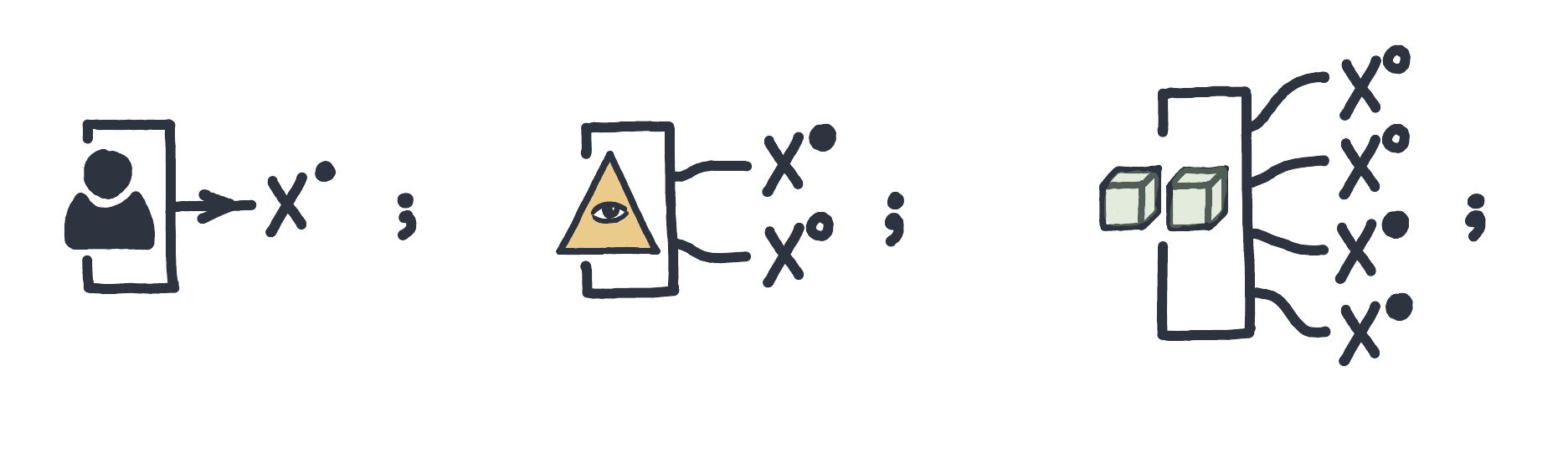}
  \caption{Newcomb's problem: components of a message theory.}
  \label{fig:newcombs-components}
\end{figure}

\emph{The Evidential Decision Theory Solution:} Let us take as an axiom that these components are constructed out of total stochastic channels; in other words, the \messageTheory{} we use is the free \messageTheory{} over the Kleisli category of the subdistribution monad, $\Session(\Kleisli(\Subd))$.

We cannot assume anything about the agent, but because of the construction of out free message theory (\Cref{def:sessions}), it must be given by a single stochastic channel $(\iconAgent{}) ፡ I → X$.  Even without assuming anything about the predictor, because of the construction of the free message theory (\Cref{def:sessions}), we know that must be constructed of two parts: the one that sends the prediction, $(\iconPredictor)_1 ፡ I → M ⊗ X$, and the one that receives the choice of the agent, $(\iconPredictor)_2 ፡ M ⊗ X → I$; of the first part we know nothing, but we have postulated that we will observe it to be perfectly accurate with the prediction, meaning that the second part will fail if it is not. Thus, we deduce it must factor as in \Cref{fig:predictor-splits}.
\begin{figure}[ht] 
  \centering
  \includegraphics[scale=0.3]{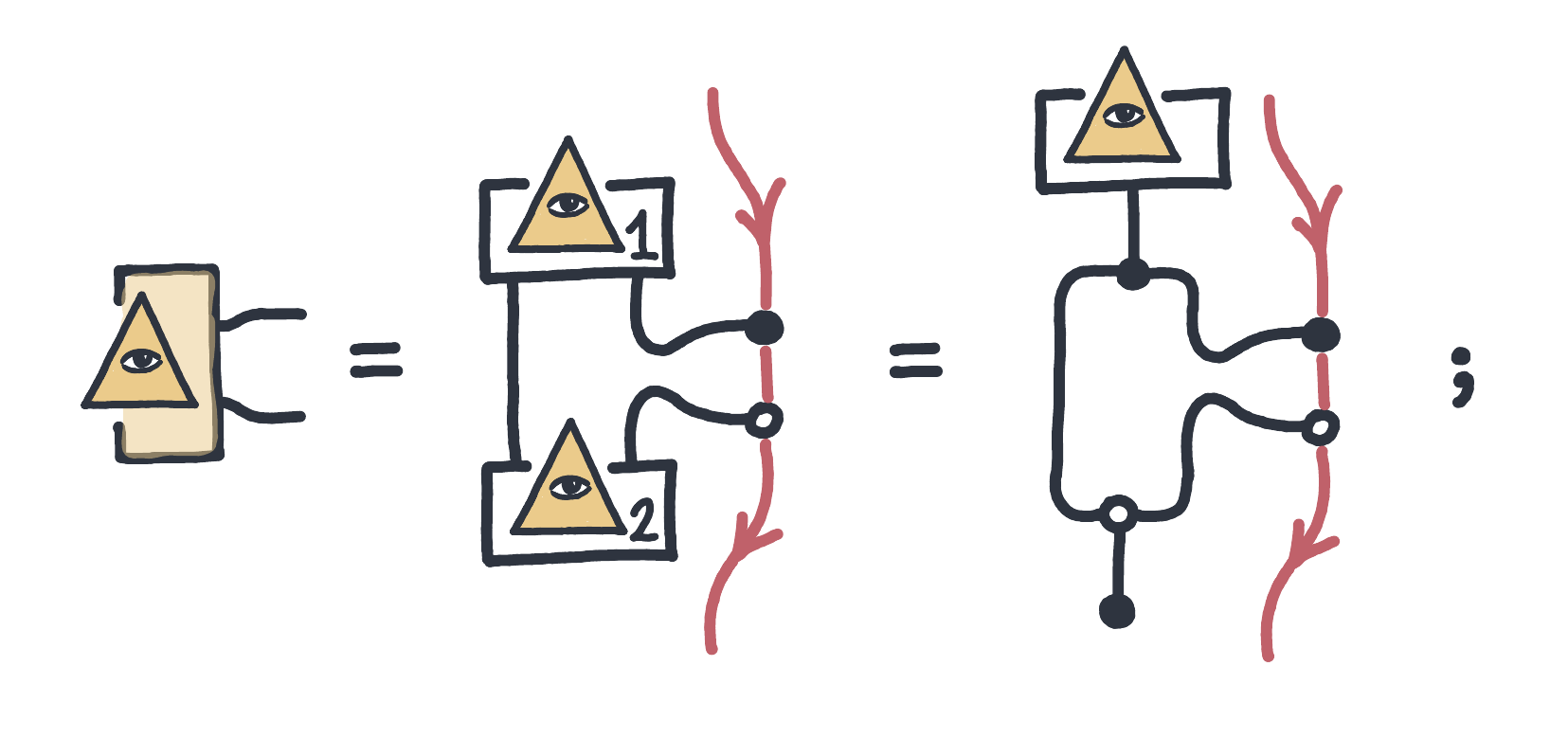}
  \caption{Evidential reading of the predictor.}
  \label{fig:predictor-splits}
\end{figure}

The wiring of the components is given by the statement: agent and predictor send choice and prediction to the stage, which answers giving back the choice to the predictor and computing the output (\Cref{fig:newcombs-evidential-solution}). We then reason \emph{(i)} computing the \polarShuffle{}; \emph{(ii)} we analyze the agent by cases, the agent two-boxes ($\mathsf{T}$) with probability $a$ or one-boxes ($\mathsf{O}$) with probability $(1-a)$; \emph{(iii)} because both cases are deterministic, they can be copied; \emph{(iv)} we analyze then the predictor, it two-boxes ($\mathsf{T}$) with probability $p$ or one-boxes ($\mathsf{O}$) with probability $(1-p)$; \emph{(v)} we compute according to \Cref{fig:newcombs-problem}, canceling the incompatible equality checks; and \emph{vi} assuming that the first term is just an order of magnitude larger than the second, we can bound it by $p \cdot 1000$\euro{}, where we take $a = 1$: the agent should one-box.
\begin{figure}[ht] 
  \centering
  \includegraphics[scale=0.25]{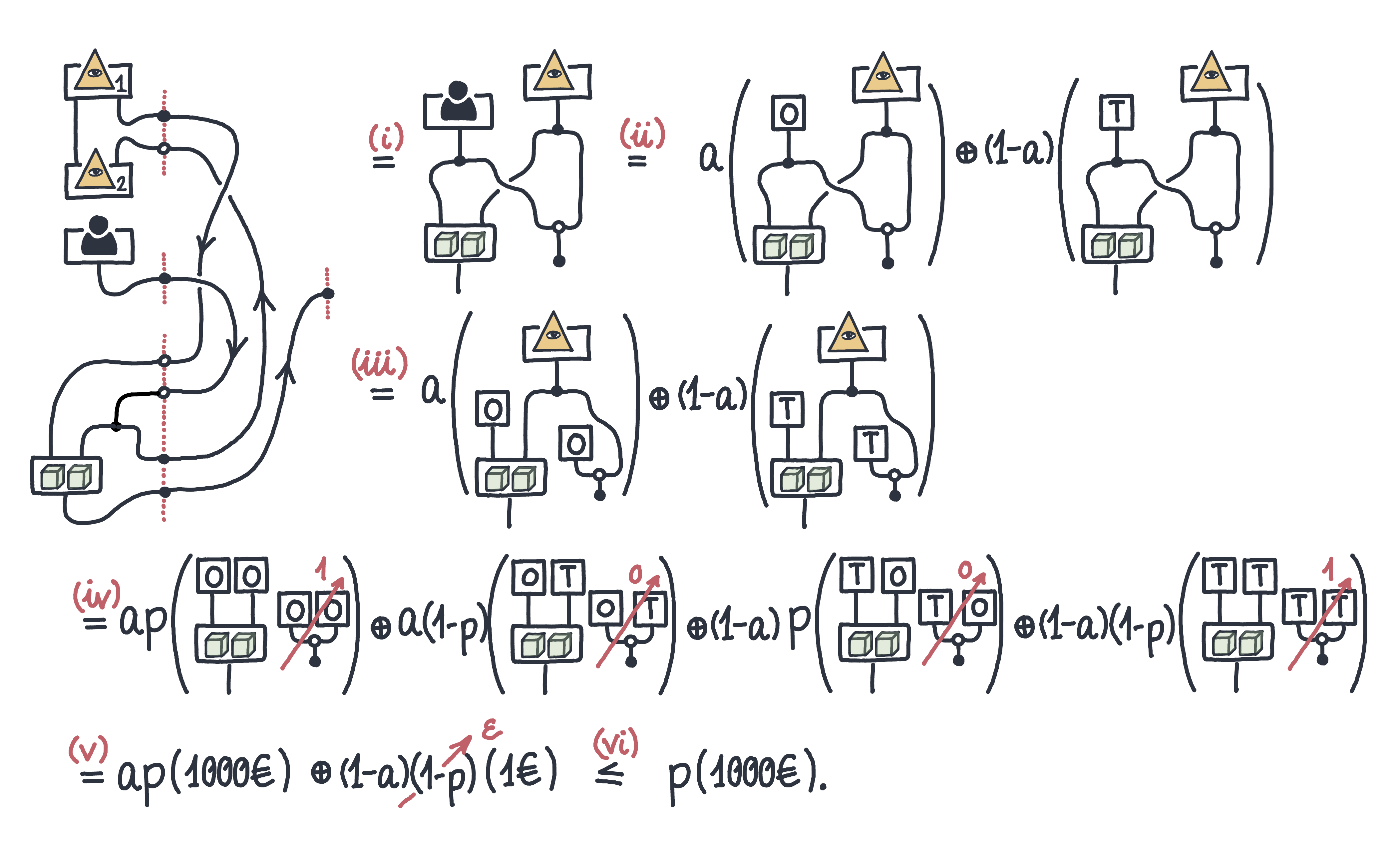}
  \caption{Newcomb's problem: the solution from Evidential Decision Theory.}
  \label{fig:newcombs-evidential-solution}
\end{figure}

\newpage
\emph{The Causal Decision Theory Solution:} Let us assume the same components (\Cref{fig:newcombs-problem}). The only hypothesis over which we will place suspicion is that the predictor can be ``perfectly accurate'' without violating causality in some way. Causal decision theory assumes that all processes are causal, or total. That means that, after receiving the news of what the agent has chosen, the predictor can do nothing: there must exist a unique total morphism $X ⊗ X → I$.
\begin{figure}[ht] 
  \centering
  \includegraphics[scale=0.4]{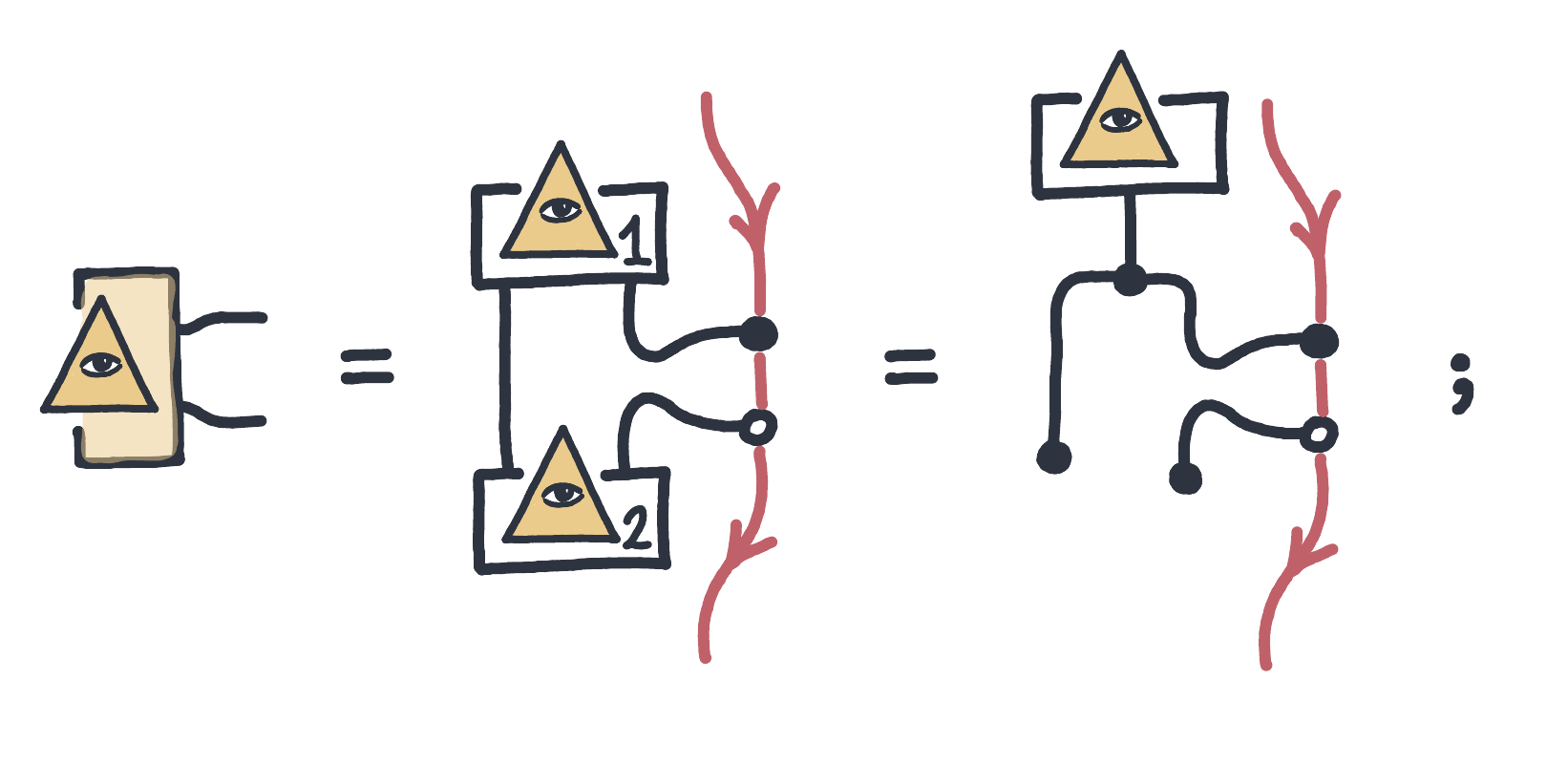}
  \caption{Causal reading of the predictor.}
  \label{fig:predictor-causal-splits}
\end{figure}

The wiring of the components is again the same, but the computation is now different: \emph{(i)} we compute the \polarShuffle{}; \emph{(ii)} we analyze the agent by cases, the agent two-boxes ($\mathsf{T}$) with probability $a$ or one-boxes ($\mathsf{O}$) with probability $(1-a)$; \emph{(iii)} we analyze in the same way the predictor; \emph{(iv)} we compute according to \Cref{fig:newcombs-problem}, canceling the incompatible equality checks; and \emph{(v)} we bound everything by the case where $a = 0$: this time, to maximize utility, the agent must two-box.
\begin{figure}[ht] 
  \centering
  \includegraphics[scale=0.3]{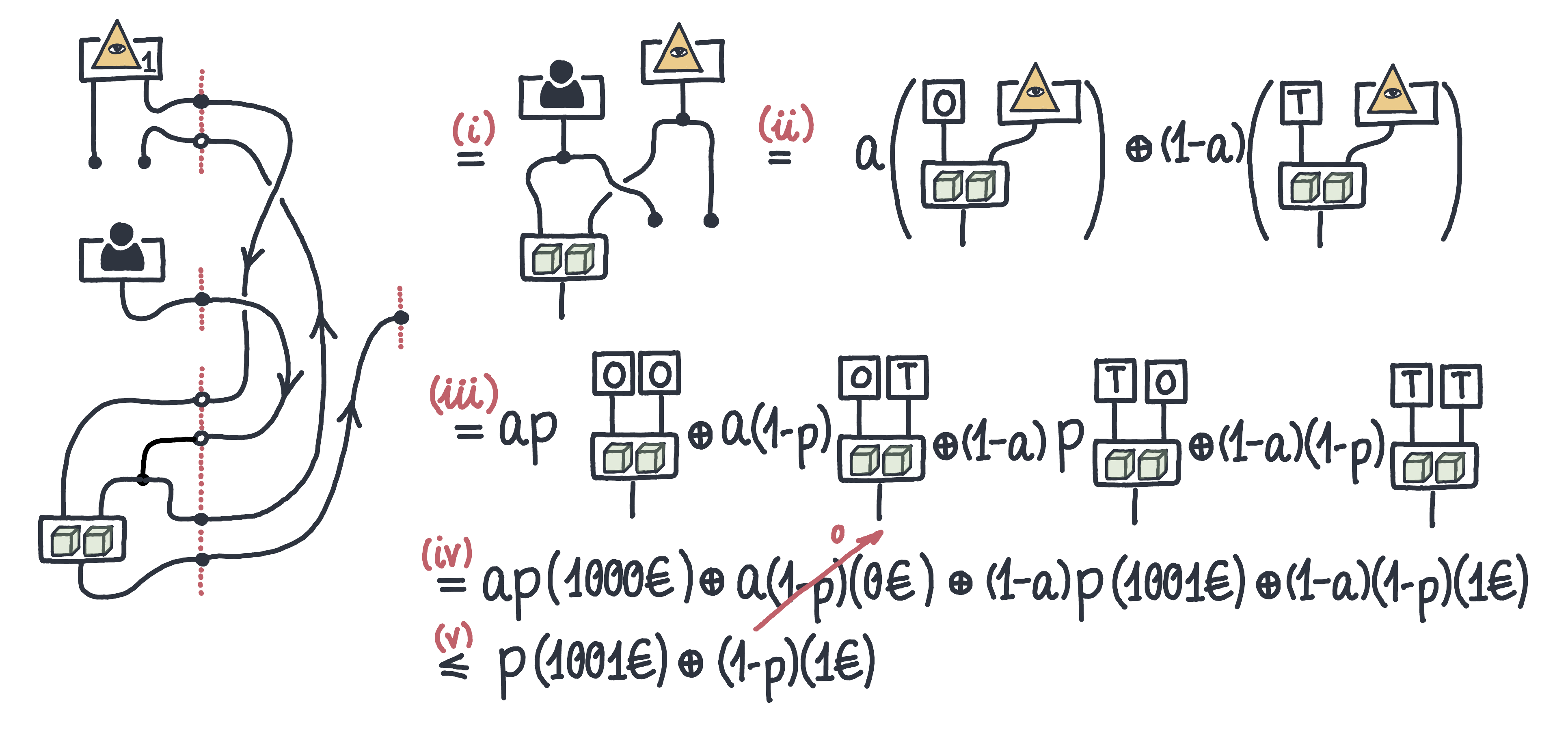}
  \caption{Newcomb's problem: the solution from Causal Decision Theory.}
  \label{fig:newcombs-evidential-solution}
\end{figure}

Was this formal analysis better than a pure discussion? We can now claim that the advantage is that the computations go from a starting diagram that represents our reading of the problem to a utility that we can maximize. We have turned most of the problem into a problem of computation: not only looking at the table of utilities (\Cref{fig:newcombs-problem}) but at the whole statement of the problem. Of course, our formalization does not solve the debate on Newcomb's paradox, but at least it moves the controversy to a more fundamental point: are we fine with using the action of the agent to reason acausally about the predictor? The algebra of partial Markov categories provides a mathematical framework where it makes sense to assume so; the algebra of \messageTheories{} takes care of the rest of the discussion.

\subsection*{Bibliography}
The idea of using send/receive effects for encoding sessions is not new. 
Message passing can be also seen as a core component of game semantics, which has a vast literature \cite{abramsky94:games,abramsky99,hyland97:game}. Game semantics has the ambition to provide the mathematical structures that describe coordination between distributed agents, starting from a duality between the Player's moves and the Opponent's moves; one of its achievements is to provide syntax-independent semantics for different extensions of PCF \cite{mccusker00:games}, including one with global state \cite{abramsky98:fully}. 

Game semantics and session types \cite{honda08:sessionTypes} have been called ``two sides of the same coin'' \cite{castellan19:two}.
Orchard and Yoshida \cite{orchard16:effects} discuss two mutual embeddings between an effectful $\lambda$-calculus (PCF) and a session $\pi$-calculus; further work also implemented the corresponding do-notation \cite{orchard17:session}. In our framework, this correspondence occurs between \premonoidalCategories{} and \messageTheories{}; and thus between do-notation and \polarShuffles{}.

Particularly relevant is Melliès' categorical approach to game semantics in the form of \emph{template games} \cite{mellies19:templategames}. The crucial difference between the present proposal and Melliès' line of work is that it starts from labelled transition systems as the basic notion. Melliès introduces \emph{asynchronous graphs} -- graphs with a set of commuting squares -- and many of the same ingredients that we use here. Asynchronous graphs explain shuffles, polarization, and a failure of interchange in the form of a Grey tensor product \cite{mellies21:asynchronous}. We would be interested to compare our approach to message passing with the framework of template games, and especially the points where they diverge: we land on \normalDuoidalCategories{} while template games are based on Girard's linear logic \cite{girard89}.

Finally, Newcomb's problem (or \emph{paradox}) was first stated by Nozick \cite{nozick69}. Evidential decision theory is defended in the work of Ahmad \cite{ahmed14:evidential}; both Everitt, Leike and Hutter \cite{everitt15:sequential}, and Yudkowsky and Soares \cite{yudkowsky:soares:17} have formalized comparisons of evidential decision theory, causal decision theory, and further variants.\clearpage{}%

\chapter{Conclusions and Further Work}
\section*{Conclusions}

\subsection*{Monoidal Context Theory}
We have universally characterized the \normalProduoidalCategory{} of \monoidalLenses{} (\Cref{th:monoidalContextsAreANormalization}) as a free normalization of the cofree \produoidalCategory{} on top of a \monoidalCategory{}. The interpretation of this result is relevant: the splice-contour adjunction (\cite{mellies22:parsing}, \Cref{th:splice-right-contour}) relates each category to its cofree \promonoidalCategory{} of incomplete terms; in the same way, we have constructed a monoidal splice-contour adjunction relating each \monoidalCategory{} to its cofree \produoidalCategory{} of incomplete processes. The category underlying this universal \produoidalCategory{} is familiar: it is the category of \monoidalLenses{}.

\MonoidalLenses{} have gained recent popularity in applications of category theory, apart from their classical counterparts in database theory \cite{johnson2012lenses}: they have spawned applications in bidirectional transformations \cite{foster07:bidirectional} but also in functional programming \cite{pickering17:profunctoroptics,ClarkeRoman20:ProfunctorOptics}, open games \cite{ghani:compositionalgametheory2018}, polynomial functors \cite{niuspivak:polynomial} and quantum combs \cite{hefford_combs}. Moreover, a different \promonoidal{} structure for lenses had been already studied in the past by Pastro and Street \cite{pastro07}. Apart from lenses, incomplete processes have appeared implicitly multiple times in recent literature. Kissinger and Uijlen \cite{kissinger:uijlen:causalstructure:lics2017} describe higher-order quantum processes using contexts with holes in compact closed monoidal categories. Ghani, Hedges, Winschel and Zahn \cite{ghani:compositionalgametheory2018} describe economic game theory in terms of \emph{lenses} and incomplete processes in cartesian monoidal categories. Bonchi, Piedeleu, \Sobocinski{} and Zanasi \cite{bonchi:graphicalaffinealgebra2019} study contextual equivalence in their monoidal category of affine signal flow graphs. Di Lavore, de Felice and Román \cite{monoidalstreams} define \emph{monoidal streams} by iterating monoidal context coalgebraically.
 This situation prompted a question: why are lenses so prevalent? why do they appear in seemingly unrelated applications?  We can now claim a conceptual answer with a mathematical justification. Lenses are the universal algebra for decomposing morphisms in process theories. The recent applications of lenses all describe incomplete processes.

Incomplete processes have two uses: on the one hand, they track the dependencies between monoidal processes; on the other hand, they allow us to split the multiple agents of a multi-party process. The former is quite useful in itself: \duoidalCategories{} provide two different tensors -- a sequential, $(⊲)$, and a parallel one $(⊗)$ -- that can track dependencies between different processes (as we saw in \Cref{sec:communication-duoidals}). Signalling and non-signalling conditions are important to the study of quantum theories, and recent work by Wilson and Chiribella \cite{wilson22:mathematical} and Hefford and Kissinger \cite{hefford_spacetime} has studied signalling using structures close to \monoidalLenses{}; we hope that the universal characterization of lenses as cofree \produoidalCategories{} may help extracting the exact structure needed for these physical frameworks. Even forgetting about dependencies explicitly, the theory of monoidal contexts allows us to pursue branches of computer science that were classically restricted to 1-dimensional syntaxes: the recent work of Eanrshaw and Sobocinski \cite{earnshaw22} studies monoidal regular languages as the natural monoidal analogue of the classical notion of regular language. 

However, it is the second application of incomplete processes the one that this author found more surprising: we can now separate the different agents of a multi-party process in an arbitrary \monoidalCategory{}. Message passing was not the main goal of this thesis, but developing it has brought interesting connections to \emph{game semantics}.

\subsection*{Monoidal Message Passing}
\SymmetricMonoidalCategories{} have two operations: sequential composition $(⨾)$ and parallel composition $(⊗)$. Naively, we would think then that there are two ways of decomposing monoidal processes: sequentially and in parallel. This is not false, and for many applications this may be the simpler way of dealing with this; after all, glueing sequentially and in parallel is the only thing we need to separate a string diagram in its atomic parts. However, this misses the rich algebra of incomplete morphisms and their compositions: \monoidalLenses{} can be composed according to any \polarShuffle{}, and that is the algebra that we are implicitly using when we cut a string diagram into pieces that do not necessarily follow the sequential and parallel divisions.
We can now argue that the algebra of message passing for process theories is that of \messageTheories{}. \MessageTheories{} try to be a minimalistic axiomatization of what it means to communicate different processes that send and receive messages: we have argued for these axioms in \Cref{sec:message-theories}, in a way that should appeal any reader not familiar with the categorical framework behind them. Their combinatorial characterization in terms of \polarShuffles{} only makes them more concrete.

We have developed a theory of context on top of \monoidalCategories{} and we have used it to develop a canonical theory of message passing in \monoidalCategories{}. 
We have argued that this a fundamental structure for concurrency, and we have characterized it universally; however, we could still argue that it does not address the problem posed by Abramsky of finding the fundamental structures of concurrency \cite{abramsky06:concurrency}. The main limitation of this framework is that it does not provide most of the features that we expect from fully-fledged session types: how can we model choice, or synchrony \cite{honda93,honda08:sessionTypes}? how can we model iteration, feedback, or other common programming constructs \cite{mccusker00:games}?
Minimalism, however, may be a good thing to separate these concerns from the fundamental structure: if we want to model choice, we can do so using \emph{distributive categories} \cite{carboni93:distributiveextensive}, or \emph{linear actegories} \cite{cockettPastro09:messagepassing}; if we want feedback and iteration, we can recall traced categories \cite{joyal96}, categories with feedback \cite{diLavore:spangraph}, and notions of monoidal automata \cite{monoidalstreams}. Precisely because our framework is minimal, it seems that it is robust enough to support these additions; we will be interested in constructing models of game semantics in further work.

\subsection*{Future Work}
Monoidal game semantics was not the original goal of this project but it became its most promising avenue; last section gives us a recipe to construct a \emph{session do-notation} calculus on top of any \monoidalCategory{}, which would include stochastic, non-deterministic and partial variants of a multi-party do-notation calculus. Levy, Power and Thielecke \cite{levy03:modelling} discuss the correspondence between \premonoidalCategories{} and call-by-value languages; we would be interested in extending this correspondence into message passing \cite{orchard17:session}. Once there, it seems plausible that we can connect this idea to the literature on game semantics for fully-fledged programming semantics \cite{mccusker00:games}. Specifically, we would like to construct a model of probabilistic programming allowing for message passing, choice and iteration; multiple developments in categorical probability (such as quasiborel spaces \cite{heunen18:semantic} or Markov categories \cite{fritz:markov2020}) make this possible.

The existence of a vast literature on message passing in terms of linear actions \cite{cockettPastro09:messagepassing} makes it particularly important to understand exactly how \duoidalCategories{} and linear logic relate. We conjecture that \physicalDuoidalCategories{} also form \emph{isomix} linearly distributive categories \cite{cockett97:prooftheory}, and this may be a categorical justification behind this connection.  

For conciseness, we have not discussed approaches to iteration in \monoidalCategories{}. Joint work of this author with Di Lavore and de Felice \cite{monoidalstreams} has shown that it is possible to reason coinductively with automata in \monoidalCategories{}, using precisely \monoidalLenses{} to describe incomplete processes. It seems plausible then, that mixing coinduction with message passing allows us to talk about networks of stochastic iterative processes: these have remarkable applications in scientific modelling, where we can write causal networks of stochastic processes in which unknown components are approximated stochastically.

\bibliographystyle{alpha}
\bibliography{./main}

\newcommand{\etalchar}[1]{$^{#1}$}
\begin{thebibliography}{DLdFR22}

\bibitem[Abr05]{abramsky06:concurrency}
Samson Abramsky.
\newblock What are the fundamental structures of concurrency?: We still don't
  know!
\newblock In Luca Aceto and Andrew~D. Gordon, editors, {\em Proceedings of the
  Workshop "Essays on Algebraic Process Calculi", {APC} 25, Bertinoro, Italy,
  August 1-5, 2005}, volume 162 of {\em Electronic Notes in Theoretical
  Computer Science}, pages 37--41. Elsevier, 2005.

\bibitem[AC09]{abramsky09:categoricalquantum}
Samson Abramsky and Bob Coecke.
\newblock Categorical quantum mechanics.
\newblock In Kurt Engesser, Dov~M. Gabbay, and Daniel Lehmann, editors, {\em
  Handbook of Quantum Logic and Quantum Structures}, pages 261--323. Elsevier,
  Amsterdam, 2009.

\bibitem[AHLF18]{aguiar18}
Marcelo Aguiar, Mariana Haim, and Ignacio L{\'o}pez~Franco.
\newblock Monads on higher monoidal categories.
\newblock {\em Applied Categorical Structures}, 26(3):413--458, Jun 2018.

\bibitem[AHM98]{abramsky98:fully}
Samson Abramsky, Kohei Honda, and Guy McCusker.
\newblock A fully abstract game semantics for general references.
\newblock In {\em Proceedings. Thirteenth Annual IEEE Symposium on Logic in
  Computer Science (Cat. No. 98CB36226)}, pages 334--344. IEEE, 1998.

\bibitem[Ahm14]{ahmed14:evidential}
Arif Ahmed.
\newblock {\em Evidence, {Decision} and {Causality}}.
\newblock Cambridge University Press, 2014.

\bibitem[AHS02]{abramsky02}
Samson Abramsky, Esfandiar Haghverdi, and Philip~J. Scott.
\newblock Geometry of interaction and linear combinatory algebras.
\newblock {\em Math. Struct. Comput. Sci.}, 12(5):625--665, 2002.

\bibitem[AJ94]{abramsky94:games}
Samson Abramsky and Radha Jagadeesan.
\newblock Games and full completeness for multiplicative linear logic.
\newblock {\em The Journal of Symbolic Logic}, 59(2):543--574, 1994.

\bibitem[Alu21]{aluffi21:algebra}
Paolo Aluffi.
\newblock {\em Algebra: {C}hapter 0}, volume 104.
\newblock American Mathematical Society, 2021.

\bibitem[AM]{aguiarEmail}
Marcelo Aguiar and Swapneel~A. Mahajan.
\newblock personal communication.

\bibitem[AM99]{abramsky99}
Samson Abramsky and Guy McCusker.
\newblock Game {{Semantics}}.
\newblock In Ulrich Berger and Helmut Schwichtenberg, editors, {\em
  Computational {{Logic}}}, {{NATO ASI Series}}, pages 1--55, {Berlin,
  Heidelberg}, 1999. {Springer}.

\bibitem[AM10]{aguiar10:monoidal}
Marcelo Aguiar and Swapneel~A. Mahajan.
\newblock {\em Monoidal functors, species and Hopf algebras}, volume~29.
\newblock American Mathematical Society Providence, RI, 2010.

\bibitem[BCST96]{blute96}
Robert~F. Blute, Robin~B. Cockett, Robert A.~G. Seely, and Todd~H. Trimble.
\newblock Natural deduction and coherence for weakly distributive categories.
\newblock {\em Journal of Pure and Applied Algebra}, 113(3):229--296, 1996.

\bibitem[BD98]{baez98:higher}
John~C. Baez and James Dolan.
\newblock Higher-dimensional algebra {III}. n-categories and the algebra of
  opetopes.
\newblock {\em Advances in Mathematics}, 135(2):145--206, 1998.

\bibitem[BE14]{erbele14:control}
John~C. Baez and Jason Erbele.
\newblock Categories in control.
\newblock {\em arXiv preprint arXiv:1405.6881}, 2014.

\bibitem[Bec23]{becerra23:strictification}
Jorge Becerra.
\newblock Strictification and non-strictification of monoidal categories.
\newblock {\em arXiv preprint arXiv:2303.16740}, 2023.

\bibitem[B{\'e}n67]{benabou67}
Jean B{\'e}nabou.
\newblock Introduction to bicategories.
\newblock In {\em Reports of the Midwest Category Seminar}, pages 1--77,
  Berlin, Heidelberg, 1967. Springer Berlin Heidelberg.

\bibitem[B{\'e}n68]{benabou68:structures}
Jean B{\'e}nabou.
\newblock Structures alg{\'e}briques dans les cat{\'e}gories.
\newblock {\em Cahiers de topologie et g{\'e}ometrie diff{\'e}rentielle},
  10(1):1--126, 1968.

\bibitem[B{\'e}n00]{benabou00}
Jean B{\'e}nabou.
\newblock Distributors at work.
\newblock {\em Lecture notes written by Thomas Streicher}, 11, 2000.

\bibitem[BG18]{boisseau2018you}
Guillaume Boisseau and Jeremy Gibbons.
\newblock What you needa know about {Y}oneda: Profunctor optics and the
  {Y}oneda lemma (functional pearl).
\newblock {\em Proceedings of the ACM on Programming Languages}, 2(ICFP):1--27,
  2018.

\bibitem[BGK{\etalchar{+}}16]{bonchi16:rewriting}
Filippo Bonchi, Fabio Gadducci, Aleks Kissinger, Pawe{\l} Soboci{\'n}ski, and
  Fabio Zanasi.
\newblock Rewriting modulo symmetric monoidal structure.
\newblock In {\em Proceedings of the 31st Annual ACM/IEEE Symposium on Logic in
  Computer Science}, pages 710--719, 2016.

\bibitem[BGMS21]{baez21:categoriesnets}
John~C. Baez, Fabrizio Genovese, Jade Master, and Michael Shulman.
\newblock Categories of nets.
\newblock In {\em 2021 36th Annual ACM/IEEE Symposium on Logic in Computer
  Science (LICS)}, pages 1--13. IEEE, 2021.

\bibitem[BK22]{broadbent22:crypto}
Anne Broadbent and Martti Karvonen.
\newblock Categorical composable cryptography.
\newblock In Patricia Bouyer and Lutz Schr{\"{o}}der, editors, {\em Foundations
  of Software Science and Computation Structures - 25th International
  Conference, {FOSSACS} 2022, Held as Part of the European Joint Conferences on
  Theory and Practice of Software, {ETAPS} 2022, Munich, Germany, April 2-7,
  2022, Proceedings}, volume 13242 of {\em Lecture Notes in Computer Science},
  pages 161--183. Springer, 2022.

\bibitem[BNR22]{boisseaunester:corneringoptics}
Guillaume Boisseau, Chad Nester, and Mario Rom{\'{a}}n.
\newblock Cornering optics.
\newblock In {\em ACT 2022}, volume abs/2205.00842, 2022.

\bibitem[BPS17]{pavlovic17}
Filippo Bonchi, Dusko Pavlovic, and Pawel Sobocinski.
\newblock Functorial semantics for relational theories.
\newblock {\em CoRR}, abs/1711.08699, 2017.

\bibitem[BPSZ19]{bonchi:graphicalaffinealgebra2019}
Filippo Bonchi, Robin Piedeleu, Pawel Sobocinski, and Fabio Zanasi.
\newblock Graphical affine algebra.
\newblock In {\em 34th Annual {ACM/IEEE} Symposium on Logic in Computer
  Science, {LICS} 2019, Vancouver, BC, Canada, June 24-27, 2019}, pages 1--12.
  {IEEE}, 2019.

\bibitem[BR23]{braithwaite23:collages}
Dylan Braithwaite and Mario Rom{\'a}n.
\newblock Collages of string diagrams.
\newblock {\em arXiv preprint arXiv:2305.02675}, 2023.

\bibitem[BS10]{baezstay10:rosetta}
John~C. Baez and Mike Stay.
\newblock Physics, topology, logic and computation: A {R}osetta stone.
\newblock In {\em New Structures for Physics}, pages 95--172. Springer Berlin
  Heidelberg, 2010.

\bibitem[BS13]{bookerstreet13}
Thomas Booker and Ross Street.
\newblock Tannaka duality and convolution for duoidal categories.
\newblock {\em Theory and Applications of Categories}, 28(6):166--205, 2013.

\bibitem[BSS18]{bonchi18}
Filippo Bonchi, Jens Seeber, and Pawel Sobocinski.
\newblock Graphical conjunctive queries.
\newblock In Dan~R. Ghica and Achim Jung, editors, {\em 27th {EACSL} Annual
  Conference on Computer Science Logic, {CSL} 2018, September 4-7, 2018,
  Birmingham, {UK}}, volume 119 of {\em LIPIcs}, pages 13:1--13:23. Schloss
  Dagstuhl - Leibniz-Zentrum f{\"{u}}r Informatik, 2018.

\bibitem[BSZ14]{bonchi14}
Filippo Bonchi, Paweł Sobociński, and Fabio Zanasi.
\newblock A categorical semantics of signal flow graphs.
\newblock In {\em International Conference on Concurrency Theory}, pages
  435--450. Springer, 2014.

\bibitem[BZ20]{blanco20:polycategories}
Nicolas Blanco and Noam Zeilberger.
\newblock Bifibrations of polycategories and classical linear logic.
\newblock In Patricia Johann, editor, {\em Proceedings of the 36th Conference
  on the Mathematical Foundations of Programming Semantics, {MFPS} 2020,
  Online, October 1, 2020}, volume 352 of {\em Electronic Notes in Theoretical
  Computer Science}, pages 29--52. Elsevier, 2020.

\bibitem[Cam19]{campbell19:strictification}
Alexander Campbell.
\newblock How strict is strictification?
\newblock {\em Journal of Pure and Applied Algebra}, 223(7):2948--2976, 2019.

\bibitem[CEG{\etalchar{+}}20]{ClarkeRoman20:ProfunctorOptics}
Bryce Clarke, Derek Elkins, Jeremy Gibbons, Fosco Loregi{\`{a}}n, Bartosz
  Milewski, Emily Pillmore, and Mario Rom{\'{a}}n.
\newblock Profunctor optics, a categorical update.
\newblock {\em CoRR}, abs/2001.07488, 2020.

\bibitem[CFS16]{coeckeFS16}
Bob Coecke, Tobias Fritz, and Robert~W. Spekkens.
\newblock A mathematical theory of resources.
\newblock {\em Inf. Comput.}, 250:59--86, 2016.

\bibitem[CGG{\etalchar{+}}22]{cruttwell22:learning}
Geoffrey S.~H. Cruttwell, Bruno Gavranovi{\'c}, Neil Ghani, Paul Wilson, and
  Fabio Zanasi.
\newblock Categorical foundations of gradient-based learning.
\newblock In {\em European Symposium on Programming}, pages 1--28. Springer,
  Cham, 2022.

\bibitem[CJ19]{cho:jacobs:disintegration2019}
Kenta Cho and Bart Jacobs.
\newblock Disintegration and {{Bayesian Inversion}} via {{String Diagrams}}.
\newblock {\em Mathematical Structures in Computer Science}, pages 1--34, March
  2019.

\bibitem[CLW93]{carboni93:distributiveextensive}
Aurelio Carboni, Stephen Lack, and Robert~F.C. Walters.
\newblock Introduction to extensive and distributive categories.
\newblock {\em Journal of Pure and Applied Algebra}, 84(2):145--158, 1993.

\bibitem[CP09]{cockettPastro09:messagepassing}
Robin~B. Cockett and Craig~A. Pastro.
\newblock The logic of message-passing.
\newblock {\em Sci. Comput. Program.}, 74(8):498--533, 2009.

\bibitem[CS97a]{cockett97:prooftheory}
Robin~B. Cockett and Robert A.~G. Seely.
\newblock Proof theory for full intuitionistic linear logic, bilinear logic,
  and mix categories.
\newblock {\em Theory and Applications of categories}, 3(5):85--131, 1997.

\bibitem[CS97b]{cockett1997}
Robin~B. Cockett and Robert A.~G. Seely.
\newblock Weakly distributive categories.
\newblock {\em Journal of Pure and Applied Algebra}, 114(2):133--173, 1997.

\bibitem[CS07]{cockett07:polarized}
Robin~B. Cockett and Robert A.~G. Seely.
\newblock Polarized category theory, modules, and game semantics.
\newblock {\em Theory and Applications of Categories}, 18(2):4--101, 2007.

\bibitem[CS09]{cruttwell09:unified}
Geoffrey S.~H. Cruttwell and Michael Shulman.
\newblock A unified framework for generalized multicategories.
\newblock {\em arXiv preprint arXiv:0907.2460}, 2009.

\bibitem[CW87]{carboni87:cartesianbicategories}
Aurelio Carboni and Robert F.~C. Walters.
\newblock Cartesian bicategories {I}.
\newblock {\em Journal of pure and applied algebra}, 49(1-2):11--32, 1987.

\bibitem[CY19]{castellan19:two}
Simon Castellan and Nobuko Yoshida.
\newblock Two sides of the same coin: session types and game semantics: a
  synchronous side and an asynchronous side.
\newblock {\em Proceedings of the ACM on Programming Languages}, 3(POPL):1--29,
  2019.

\bibitem[Day70]{day}
Brian Day.
\newblock On closed categories of functors.
\newblock In {\em Reports of the Midwest Category Seminar IV}, volume 137,
  pages 1--38, Berlin, Heidelberg, 1970. Springer Berlin Heidelberg.

\bibitem[Dd09]{dezani09}
Mariangiola Dezani{-}Ciancaglini and Ugo de'Liguoro.
\newblock Sessions and session types: An overview.
\newblock In Cosimo Laneve and Jianwen Su, editors, {\em Web Services and
  Formal Methods, 6th International Workshop, {WS-FM} 2009, Bologna, Italy,
  September 4-5, 2009, Revised Selected Papers}, volume 6194 of {\em Lecture
  Notes in Computer Science}, pages 1--28. Springer, 2009.

\bibitem[DDR11]{dumas11:cartesianEffect}
Jean{-}Guillaume Dumas, Dominique Duval, and Jean{-}Claude Reynaud.
\newblock Cartesian effect categories are {Freyd}-categories.
\newblock {\em Journal of Symbolic Computation}, 46(3):272--293, 2011.

\bibitem[DGNO10]{drinfeld10}
Vladimir Drinfeld, Shlomo Gelaki, Dmitri Nikshych, and Victor Ostrik.
\newblock On braided fusion categories {I}.
\newblock {\em Selecta Mathematica}, 16(1):1--119, 2010.

\bibitem[DLdFR22]{monoidalstreams}
Elena Di~Lavore, Giovanni de~Felice, and Mario Rom\'{a}n.
\newblock Monoidal streams for dataflow programming.
\newblock In {\em Proceedings of the 37th Annual ACM/IEEE Symposium on Logic in
  Computer Science}, LICS '22, New York, NY, USA, 2022. Association for
  Computing Machinery.

\bibitem[DLLNS21]{di2021functorial}
Ivan Di~Liberti, Fosco Loregian, Chad Nester, and Pawe{\l} Soboci{\'n}ski.
\newblock Functorial semantics for partial theories.
\newblock {\em Proceedings of the ACM on Programming Languages}, 5(POPL):1--28,
  2021.

\bibitem[DPS05]{day05:centres}
Brian Day, Elango Panchadcharam, and Ross Street.
\newblock On centres and lax centres for promonoidal categories.
\newblock In {\em Colloque International Charles Ehresmann}, volume 100, 2005.

\bibitem[DS03]{daystreet04:quantum}
Brian Day and Ross Street.
\newblock Quantum categories, star autonomy, and quantum groupoids, 2003.

\bibitem[EH61]{eckman61:structure}
Beno Eckman and Peter Hilton.
\newblock Structure maps in group theory.
\newblock {\em Fundamenta Mathematicae}, 50(2):207--221, 1961.

\bibitem[EHR23]{produoidal23}
Matt Earnshaw, James Hefford, and Mario Román.
\newblock The produoidal algebra of process decomposition.
\newblock {\em arXiv preprint arXiv:2301.11867}, 2023.

\bibitem[ELH15]{everitt15:sequential}
Tom Everitt, Jan Leike, and Marcus Hutter.
\newblock Sequential extensions of causal and evidential decision theory.
\newblock In {\em International Conference on Algorithmic Decision Theory},
  pages 205--221. Springer, 2015.

\bibitem[ES22]{earnshaw22}
Matthew Earnshaw and Pawel Soboci\'{n}ski.
\newblock {Regular Monoidal Languages}.
\newblock In Stefan Szeider, Robert Ganian, and Alexandra Silva, editors, {\em
  47th International Symposium on Mathematical Foundations of Computer Science
  (MFCS 2022)}, volume 241 of {\em Leibniz International Proceedings in
  Informatics (LIPIcs)}, pages 44:1--44:14, Dagstuhl, Germany, 2022. Schloss
  Dagstuhl -- Leibniz-Zentrum f{\"u}r Informatik.

\bibitem[FB94]{foulis94:effect}
David~J Foulis and Mary~K Bennett.
\newblock Effect algebras and unsharp quantum logics.
\newblock {\em Foundations of physics}, 24:1331--1352, 1994.

\bibitem[FGM{\etalchar{+}}07]{foster07:bidirectional}
J.~Nathan Foster, Michael~B. Greenwald, Jonathan~T. Moore, Benjamin~C. Pierce,
  and Alan Schmitt.
\newblock Combinators for bidirectional tree transformations: A linguistic
  approach to the view-update problem.
\newblock {\em ACM Transactions on Programming Languages and Systems (TOPLAS)},
  29(3):17--es, 2007.

\bibitem[FGP21]{fitz2021deFinetti}
Tobias Fritz, Tomáš Gonda, and Paolo Perrone.
\newblock {De Finetti}’s theorem in categorical probability.
\newblock {\em Journal of Stochastic Analysis}, 2(4), 2021.

\bibitem[FJ19]{fong19:lenses}
Brendan Fong and Michael Johnson.
\newblock Lenses and learners.
\newblock {\em arXiv preprint arXiv:1903.03671}, 2019.

\bibitem[Fon13]{fong_2013}
Brendan Fong.
\newblock Causal {Theories}: {A} {Categorical} {Perspective} on {Bayesian}
  {Networks}.
\newblock {\em Master's Thesis, University of Oxford. ArXiv preprint
  arXiv:1301.6201}, 2013.

\bibitem[Fox76]{fox76}
Thomas Fox.
\newblock Coalgebras and cartesian categories.
\newblock {\em Communications in Algebra}, 4(7):665--667, 1976.

\bibitem[FP19]{fritz2019probability}
Tobias Fritz and Paolo Perrone.
\newblock A probability monad as the colimit of spaces of finite samples.
\newblock {\em Theory and Applications of Categories}, 34(7):170--220, 2019.

\bibitem[FPR21]{fritz2021probability}
Tobias Fritz, Paolo Perrone, and Sharwin Rezagholi.
\newblock Probability, valuations, hyperspace: Three monads on top and the
  support as a morphism.
\newblock {\em Mathematical Structures in Computer Science}, 31(8):850--897,
  2021.

\bibitem[FR20]{fritz2020infinite}
Tobias Fritz and Eigil~Fjeldgren Rischel.
\newblock Infinite products and zero-one laws in categorical probability.
\newblock {\em Compositionality}, 2:3, 2020.

\bibitem[Fri20]{fritz:markov2020}
Tobias Fritz.
\newblock A synthetic approach to {{Markov}} kernels, conditional independence
  and theorems on sufficient statistics.
\newblock {\em Advances in Mathematics}, 370:107239, 2020.

\bibitem[FS19]{fong2019supplying}
Brendan Fong and David~I Spivak.
\newblock Supplying bells and whistles in symmetric monoidal categories.
\newblock {\em arXiv preprint arXiv:1908.02633}, 2019.

\bibitem[FV20]{vasilakopoulou:20duoidal}
Ignacio~L{\'o}pez Franco and Christina Vasilakopoulou.
\newblock Duoidal categories, measuring comonoids and enrichment.
\newblock {\em arXiv preprint arXiv:2005.01340}, 2020.

\bibitem[GF16]{garner16}
Richard Garner and Ignacio~L{\'o}pez Franco.
\newblock Commutativity.
\newblock {\em Journal of Pure and Applied Algebra}, 220(5):1707--1751, 2016.

\bibitem[GG09]{garner09low}
Richard Garner and Nick Gurski.
\newblock The low-dimensional structures formed by tricategories.
\newblock In {\em Mathematical Proceedings of the Cambridge Philosophical
  Society}, volume 146, pages 551--589. Cambridge University Press, 2009.

\bibitem[GH99]{gay99}
Simon~J. Gay and Malcolm Hole.
\newblock Types and subtypes for client-server interactions.
\newblock In S.~Doaitse Swierstra, editor, {\em Programming Languages and
  Systems, 8th European Symposium on Programming, ESOP'99, Held as Part of the
  European Joint Conferences on the Theory and Practice of Software, ETAPS'99,
  Amsterdam, The Netherlands, 22-28 March, 1999, Proceedings}, volume 1576 of
  {\em Lecture Notes in Computer Science}, pages 74--90. Springer, 1999.

\bibitem[GHWZ18]{ghani:compositionalgametheory2018}
Neil Ghani, Jules Hedges, Viktor Winschel, and Philipp Zahn.
\newblock Compositional game theory.
\newblock In Anuj Dawar and Erich Gr{\"{a}}del, editors, {\em Proceedings of
  the 33rd Annual {ACM/IEEE} Symposium on Logic in Computer Science, {LICS}
  2018, Oxford, UK, July 09-12, 2018}, pages 472--481. {ACM}, 2018.

\bibitem[Gir89]{girard89}
Jean-Yves Girard.
\newblock Geometry of {Interaction} 1: Interpretation of {System} {F}.
\newblock In R.~Ferro, C.~Bonotto, S.~Valentini, and A.~Zanardo, editors, {\em
  Logic Colloquium '88}, volume 127 of {\em Studies in Logic and the
  Foundations of Mathematics}, pages 221--260. Elsevier, 1989.

\bibitem[Gis88]{gischer88:equationalpomsets}
Jay~L. Gischer.
\newblock The equational theory of pomsets.
\newblock {\em Theoretical Computer Science}, 61(2-3):199--224, 1988.

\bibitem[Gra81]{grabowski81:partial}
Jan Grabowski.
\newblock On partial languages.
\newblock {\em Fundamenta Informaticae}, 4(2):427--498, 1981.

\bibitem[Gro85]{grothendieck85:recoltes}
Alexandre Grothendieck.
\newblock R{\'e}coltes et semailles: r{\'e}flexions et t{\'e}moignage sur un
  pass{\'e} de math{\'e}maticien.
\newblock {\em Grothendieck Circle Page}, 1985.

\bibitem[Gui80]{guitart1980tenseurs}
Ren{\'e} Guitart.
\newblock Tenseurs et machines.
\newblock {\em Cahiers de topologie et g{\'e}om{\'e}trie diff{\'e}rentielle
  cat{\'e}goriques}, 21(1):5--62, 1980.

\bibitem[Had18]{hadzihasanovic18:weak}
Amar Hadzihasanovic.
\newblock Weak units, universal cells, and coherence via universality for
  bicategories.
\newblock {\em arXiv preprint arXiv:1803.06086}, 2018.

\bibitem[Has97]{hasegawa97}
Masahito Hasegawa.
\newblock {\em Models of sharing graphs: a categorical semantics of let and
  letrec}.
\newblock PhD thesis, University of Edinburgh, {UK}, 1997.

\bibitem[HC22]{hefford_combs}
James Hefford and Cole Comfort.
\newblock Coend optics for quantum combs.
\newblock {\em arXiv preprint arXiv:2205.09027}, 2022.

\bibitem[Hed19]{hedgesblog2019}
Jules Hedges.
\newblock Folklore: Monoidal kleisli categories, Apr 2019.

\bibitem[Her00]{hermida:representable}
Claudio Hermida.
\newblock Representable multicategories.
\newblock {\em Advances in Mathematics}, 151(2):164--225, 2000.

\bibitem[Her01]{hermida01:coherent}
Claudio Hermida.
\newblock From coherent structures to universal properties.
\newblock {\em Journal of Pure and Applied Algebra}, 165(1):7--61, 2001.

\bibitem[HJ06]{heunen06:arrows}
Chris Heunen and Bart Jacobs.
\newblock Arrows, like monads, are monoids.
\newblock In Stephen~D. Brookes and Michael~W. Mislove, editors, {\em
  Proceedings of the 22nd Annual Conference on Mathematical Foundations of
  Programming Semantics, {MFPS} 2006, Genova, Italy, May 23-27, 2006}, volume
  158 of {\em Electronic Notes in Theoretical Computer Science}, pages
  219--236. Elsevier, 2006.

\bibitem[HJW{\etalchar{+}}92]{haskellreport}
Paul Hudak, Simon L.~Peyton Jones, Philip Wadler, Brian Boutel, Jon Fairbairn,
  Joseph~H. Fasel, Mar{\'{\i}}a~M. Guzm{\'{a}}n, Kevin Hammond, John Hughes,
  Thomas Johnsson, Richard~B. Kieburtz, Rishiyur~S. Nikhil, Will Partain, and
  John Peterson.
\newblock {Report} on the {Programming} {Language} {Haskell}, {A} {Non-strict},
  {Purely} {Functional} {Language}.
\newblock {\em {ACM} {SIGPLAN} Notices}, 27(5):1, 1992.

\bibitem[HK22]{hefford_spacetime}
James Hefford and Aleks Kissinger.
\newblock On the pre- and promonoidal structure of spacetime.
\newblock {\em arXiv preprint arXiv.2206.09678}, 2022.

\bibitem[HKS{\etalchar{+}}18]{heunen18:semantic}
Chris Heunen, Ohad Kammar, Sam Staton, Sean Moss, Matthijs V{\'a}k{\'a}r, Adam
  {\'S}cibior, and Hongseok Yang.
\newblock The semantic structure of quasi-borel spaces.
\newblock In {\em PPS Workshop on Probabilistic Programming Semantics}, 2018.

\bibitem[HMH14]{hoshino14}
Naohiko Hoshino, Koko Muroya, and Ichiro Hasuo.
\newblock Memoryful geometry of interaction: from coalgebraic components to
  algebraic effects.
\newblock In Thomas~A. Henzinger and Dale Miller, editors, {\em Joint Meeting
  of the Twenty-Third {EACSL} Annual Conference on Computer Science Logic
  {(CSL)} and the Twenty-Ninth Annual {ACM/IEEE} Symposium on Logic in Computer
  Science (LICS), {CSL-LICS} '14, Vienna, Austria, July 14 - 18, 2014}, pages
  52:1--52:10. {ACM}, 2014.

\bibitem[Hon93]{honda93}
Kohei Honda.
\newblock Types for dyadic interaction.
\newblock In Eike Best, editor, {\em {CONCUR} '93, 4th International Conference
  on Concurrency Theory, Hildesheim, Germany, August 23-26, 1993, Proceedings},
  volume 715 of {\em Lecture Notes in Computer Science}, pages 509--523.
  Springer, 1993.

\bibitem[HR23]{hefford23:optics-premonoidal}
James Hefford and Mario Rom{\'{a}}n.
\newblock Optics for premonoidal categories.
\newblock {\em CoRR}, abs/2305.02906, 2023.

\bibitem[HS23]{sigal23:duoidally}
Chris Heunen and Jesse Sigal.
\newblock Duoidally enriched {Freyd} categories.
\newblock In {\em International Conference on Relational and Algebraic Methods
  in Computer Science}, pages 241--257. Springer, 2023.

\bibitem[Huf54]{huffman1954synthesis}
David~A. Huffman.
\newblock The synthesis of sequential switching circuits.
\newblock {\em Journal of the {F}ranklin {I}nstitute}, 257(3):161--190, 1954.

\bibitem[Hug00]{hughes00}
John Hughes.
\newblock Generalising monads to arrows.
\newblock {\em Science of Computer Programming}, 37(1-3):67--111, 2000.

\bibitem[Hug12]{hughes12:simple}
Dominic Hughes.
\newblock Simple free star-autonomous categories and full coherence.
\newblock {\em Journal of Pure and Applied Algebra}, 216(11):2386--2410, 2012.

\bibitem[HV19]{heunenvicary19:categoriesquantum}
Chris Heunen and Jamie Vicary.
\newblock {\em Categories for Quantum Theory: An Introduction}.
\newblock Oxford University Press, 2019.

\bibitem[HYC08]{honda08:sessionTypes}
Kohei Honda, Nobuko Yoshida, and Marco Carbone.
\newblock Multiparty asynchronous session types.
\newblock In George~C. Necula and Philip Wadler, editors, {\em Proceedings of
  the 35th {ACM} {SIGPLAN-SIGACT} Symposium on Principles of Programming
  Languages, {POPL} 2008, San Francisco, California, USA, January 7-12, 2008},
  pages 273--284. {ACM}, 2008.

\bibitem[Hyl97]{hyland97:game}
Martin Hyland.
\newblock Game semantics.
\newblock {\em Semantics and logics of computation}, 14:131, 1997.

\bibitem[Jac15]{jacobs15:new}
Bart Jacobs.
\newblock New directions in categorical logic, for classical, probabilistic and
  quantum logic.
\newblock {\em Logical Methods in Computer Science}, 11, 2015.

\bibitem[Jef97a]{jeffrey1997:premonoidal}
Alan Jeffrey.
\newblock Premonoidal categories and a graphical view of programs.
\newblock {\em Preprint at ResearchGate}, 1997.

\bibitem[Jef97b]{jeffrey97}
Alan Jeffrey.
\newblock Premonoidal categories and flow graphs.
\newblock {\em Electronical Notes in Theoretical Computer Science}, 10:51,
  1997.

\bibitem[JHH09]{jacobs09}
Bart Jacobs, Chris Heunen, and Ichiro Hasuo.
\newblock Categorical semantics for arrows.
\newblock {\em J. Funct. Program.}, 19(3-4):403--438, 2009.

\bibitem[JKZ21]{jacobs2021causal}
Bart Jacobs, Aleks Kissinger, and Fabio Zanasi.
\newblock Causal inference via string diagram surgery: A diagrammatic approach
  to interventions and counterfactuals.
\newblock {\em Mathematical Structures in Computer Science}, 31(5):553--574,
  2021.

\bibitem[JRW12]{johnson2012lenses}
Michael Johnson, Robert Rosebrugh, and Richard~J. Wood.
\newblock Lenses, fibrations and universal translations.
\newblock {\em Mathematical structures in computer science}, 22(1):25--42,
  2012.

\bibitem[JS91]{joyal91:geometryOfTensorCalculus}
André Joyal and Ross Street.
\newblock The geometry of tensor calculus, {I}.
\newblock {\em Advances in Mathematics}, 88(1):55--112, 1991.

\bibitem[JS93]{joyal:braided}
Andr{\'e} Joyal and Ross Street.
\newblock Braided tensor categories.
\newblock {\em Advances in Mathematics}, 102(1):20--78, 1993.

\bibitem[JSV96]{joyal96}
Andr\'e Joyal, Ross Street, and Dominic Verity.
\newblock Traced monoidal categories.
\newblock {\em Mathematical Proceedings of the Cambridge Philosophical
  Society}, 119:447 -- 468, 04 1996.

\bibitem[JZ20]{jacobs2020logical}
Bart Jacobs and Fabio Zanasi.
\newblock The logical essentials of bayesian reasoning.
\newblock {\em Foundations of Probabilistic Programming}, pages 295--331, 2020.

\bibitem[KL80]{kelly80:coherence}
Gregory Kelly and Miguel Laplaza.
\newblock Coherence for compact closed categories.
\newblock {\em Journal of pure and applied algebra}, 19:193--213, 1980.

\bibitem[Kme12]{kmett12:lenslibrary}
Edward Kmett.
\newblock lens library, version 4.16.
\newblock {\em Hackage https://hackage. haskell. org/package/lens-4.16}, 2018,
  2012.

\bibitem[KPT96]{kobayashi96:linearitypicalculus}
Naoki Kobayashi, Benjamin~C. Pierce, and David~N. Turner.
\newblock Linearity and the pi-calculus.
\newblock In Hans{-}Juergen Boehm and Guy L.~Steele Jr., editors, {\em
  Conference Record of POPL'96: The 23rd {ACM} {SIGPLAN-SIGACT} Symposium on
  Principles of Programming Languages, Papers Presented at the Symposium, St.
  Petersburg Beach, Florida, USA, January 21-24, 1996}, pages 358--371. {ACM}
  Press, 1996.

\bibitem[KSW97]{sabadini95:bicategoriesOfProcesses}
Piergiulio Katis, Nicoletta Sabadini, and Robert F.~C. Walters.
\newblock Bicategories of processes.
\newblock {\em Journal of Pure and Applied Algebra}, 115(2):141--178, 1997.

\bibitem[KU17]{kissinger:uijlen:causalstructure:lics2017}
Aleks Kissinger and Sander Uijlen.
\newblock A categorical semantics for causal structure.
\newblock In {\em 32nd Annual {ACM/IEEE} Symposium on Logic in Computer
  Science, {LICS} 2017, Reykjavik, Iceland, June 20-23, 2017}, pages 1--12.
  {IEEE} Computer Society, 2017.

\bibitem[Lam69]{lambek:deductive}
Joachim Lambek.
\newblock Deductive systems and categories {II}: standard constructions and
  closed categories.
\newblock {\em Category Theory, Homology Theory and their Applications I},
  1969.

\bibitem[Lam86]{lambek1986a}
Joachim Lambek.
\newblock Cartesian closed categories and typed {$\lambda$}-calculi.
\newblock In Guy Cousineau, Pierre-Louis Curien, and Bernard Robinet, editors,
  {\em Combinators and {{Functional Programming Languages}}}, Lecture {{Notes}}
  in {{Computer Science}}, pages 136--175, {Berlin, Heidelberg}, 1986.
  {Springer}.

\bibitem[Lau05]{lauda05:frobenius}
Aaron~D. Lauda.
\newblock Frobenius algebras and ambidextrous adjunctions.
\newblock {\em arXiv preprint math/0502550}, 2005.

\bibitem[Law63]{lawvere63:functorial}
F.~William Lawvere.
\newblock Functorial semantics of algebraic theories.
\newblock {\em Proceedings of the National Academy of Sciences},
  50(5):869--872, 1963.

\bibitem[Lei04]{leinster04}
Tom Leinster.
\newblock {\em Higher Operads, Higher Categories}.
\newblock London Mathematical Society Lecture Note Series. Cambridge University
  Press, 2004.

\bibitem[Lev22]{levy04:callbypush}
Paul~Blain Levy.
\newblock Call-by-push-value.
\newblock {\em ACM SIGLOG News}, 9(2):7–29, may 2022.

\bibitem[Lew06]{lewis06:coherence}
Geoffrey Lewis.
\newblock Coherence for a closed functor.
\newblock In {\em Coherence in categories}, pages 148--195. Springer, 2006.

\bibitem[LGR{\etalchar{+}}21]{diLavore21:feedback}
Elena~Di Lavore, Alessandro Gianola, Mario Rom{\'{a}}n, Nicoletta Sabadini, and
  Pawel Sobocinski.
\newblock A canonical algebra of open transition systems.
\newblock In Gwen Sala{\"{u}}n and Anton Wijs, editors, {\em Formal Aspects of
  Component Software - 17th International Conference, {FACS} 2021, Virtual
  Event, October 28-29, 2021, Proceedings}, volume 13077 of {\em Lecture Notes
  in Computer Science}, pages 63--81. Springer, 2021.

\bibitem[LGR{\etalchar{+}}23]{diLavore:spangraph}
Elena~Di Lavore, Alessandro Gianola, Mario Rom{\'{a}}n, Nicoletta Sabadini, and
  Pawel Sobocinski.
\newblock Span({G}raph): a canonical feedback algebra of open transition
  systems.
\newblock {\em Softw. Syst. Model.}, 22(2):495--520, 2023.

\bibitem[Lor21]{loregian2021}
Fosco Loregiàn.
\newblock {\em (Co)end Calculus}.
\newblock London Mathematical Society Lecture Note Series. Cambridge University
  Press, 2021.

\bibitem[LPT03]{levy03:modelling}
Paul~Blain Levy, John Power, and Hayo Thielecke.
\newblock Modelling environments in call-by-value programming languages.
\newblock {\em Information and computation}, 185(2):182--210, 2003.

\bibitem[LR23]{dilavore:evidentialdecision}
Elena~Di Lavore and Mario Rom{\'{a}}n.
\newblock Evidential decision theory via partial markov categories.
\newblock In {\em {LICS}}, pages 1--14, 2023.

\bibitem[LS09]{lawvere09:conceptual}
F.~William Lawvere and Stephen~H. Schanuel.
\newblock {\em Conceptual mathematics: a first introduction to categories}.
\newblock Cambridge University Press, 2009.

\bibitem[Mac63]{maclane63:natural}
Saunders MacLane.
\newblock Natural associativity and commutativity.
\newblock {\em Rice Institute Pamphlet-Rice University Studies}, 49(4), 1963.

\bibitem[{Mac}78]{maclane78}
Saunders {Mac Lane}.
\newblock {\em Categories for the Working Mathematician}.
\newblock Graduate Texts in Mathematics. Springer New York, 1978.

\bibitem[Mar14]{marsden14:category}
Daniel Marsden.
\newblock Category theory using string diagrams.
\newblock {\em arXiv preprint arXiv:1401.7220}, 2014.

\bibitem[McC00]{mccusker00:games}
Guy McCusker.
\newblock Games and full abstraction for {FPC}.
\newblock {\em Information and Computation}, 160(1-2):1--61, 2000.

\bibitem[Mel19]{mellies19:templategames}
Paul{-}Andr{\'{e}} Melli{\`{e}}s.
\newblock Template games and differential linear logic.
\newblock In {\em 34th Annual {ACM/IEEE} Symposium on Logic in Computer
  Science, {LICS} 2019, Vancouver, BC, Canada, June 24-27, 2019}, pages 1--13.
  {IEEE}, 2019.

\bibitem[Mel21]{mellies21:asynchronous}
Paul-Andr{\'e} Melli{\`e}s.
\newblock Asynchronous template games and the gray tensor product of
  2-categories.
\newblock In {\em 2021 36th Annual ACM/IEEE Symposium on Logic in Computer
  Science (LICS)}, pages 1--13. IEEE, 2021.

\bibitem[ML71]{macLane71:workingMathematician}
Saunders Mac~Lane.
\newblock {\em Categories for the Working Mathematician}, volume~5 of {\em
  Graduate Texts in Mathematics}.
\newblock Springer Verlag, 1971.

\bibitem[MM90]{meseguer90:petriaremonoids}
Jos{\'{e}} Meseguer and Ugo Montanari.
\newblock Petri nets are monoids.
\newblock {\em Inf. Comput.}, 88(2):105--155, 1990.

\bibitem[Mog91]{moggi91}
Eugenio Moggi.
\newblock Notions of computation and monads.
\newblock {\em Inf. Comput.}, 93(1):55--92, 1991.

\bibitem[MP21]{malkiewich21:coherence}
Cary Malkiewich and Kate Ponto.
\newblock Coherence for bicategories, lax functors, and shadows, 2021.

\bibitem[MS14]{mogelberg14}
Rasmus~Ejlers M{\o}gelberg and Sam Staton.
\newblock Linear usage of state.
\newblock {\em Log. Methods Comput. Sci.}, 10(1), 2014.

\bibitem[MS18]{mellies18:game}
Paul-Andr{\'e} Melli{\`e}s and L{\'e}o Stefanesco.
\newblock A game semantics of concurrent separation logic.
\newblock {\em Electronic Notes in Theoretical Computer Science}, 336:241--256,
  2018.

\bibitem[Mye16]{myers16}
David~Jaz Myers.
\newblock String diagrams for double categories and equipments, 2016.

\bibitem[MZ22]{mellies22:parsing}
Paul-Andr{\'e} Melli{\`e}s and Noam Zeilberger.
\newblock Parsing as a {L}ifting {P}roblem and the
  {C}homsky-{S}ch{\"u}tzenberger {R}epresentation {T}heorem.
\newblock In {\em MFPS 2022-38th conference on Mathematical Foundations for
  Programming Semantics}, 2022.

\bibitem[Nes21]{nester21}
Chad Nester.
\newblock The structure of concurrent process histories.
\newblock In Ferruccio Damiani and Ornela Dardha, editors, {\em Coordination
  Models and Languages - 23rd {IFIP} {WG} 6.1 International Conference,
  {COORDINATION} 2021, Held as Part of the 16th International Federated
  Conference on Distributed Computing Techniques, DisCoTec 2021, Valletta,
  Malta, June 14-18, 2021, Proceedings}, volume 12717 of {\em Lecture Notes in
  Computer Science}, pages 209--224. Springer, 2021.

\bibitem[Noz69]{nozick69}
Robert Nozick.
\newblock Newcomb’s {Problem} and {Two} {Principles} of {Choice}.
\newblock In {\em Essays in honor of Carl G. Hempel}, pages 114--146. Springer,
  1969.

\bibitem[NS22]{niuspivak:polynomial}
Nelson Niu and David~I. Spivak.
\newblock Polynomial functors: A general theory of interaction.
\newblock {\em In preparation}, 2022.

\bibitem[NV23]{voorneveld23}
Chad Nester and Niels F.~W. Voorneveld.
\newblock Protocol choice and iteration for the free cornering.
\newblock {\em CoRR}, abs/2305.16899, 2023.

\bibitem[Ord20]{ord20:precipice}
Toby Ord.
\newblock {\em The precipice: Existential risk and the future of humanity}.
\newblock Hachette Books, 2020.

\bibitem[OY16]{orchard16:effects}
Dominic Orchard and Nobuko Yoshida.
\newblock Effects as sessions, sessions as effects.
\newblock {\em ACM SIGPLAN Notices}, 51(1):568--581, 2016.

\bibitem[OY17]{orchard17:session}
Dominic Orchard and Nobuko Yoshida.
\newblock Session types with linearity in {Haskell}.
\newblock In Simon Gay and António Ravara, editors, {\em Behavioural Types:
  from Theory to Tools}, River Publishers Series in Automation, Control and
  Robotics. River Publishers, 2017.

\bibitem[Pat01]{paterson01:arrows}
Ross Paterson.
\newblock A new notation for arrows.
\newblock In Benjamin~C. Pierce, editor, {\em Proceedings of the Sixth {ACM}
  {SIGPLAN} International Conference on Functional Programming {(ICFP} '01),
  Firenze (Florence), Italy, September 3-5, 2001}, pages 229--240. {ACM}, 2001.

\bibitem[Pat03]{paterson03:arrowscomputation}
Ross Paterson.
\newblock Arrows and computation.
\newblock {\em The Fun of Programming}, pages 201--222, 2003.

\bibitem[Pav13]{pavlovic13}
Dusko Pavlovic.
\newblock Monoidal computer {I:} basic computability by string diagrams.
\newblock {\em Inf. Comput.}, 226:94--116, 2013.

\bibitem[PC07]{lawvereinterview}
Jorge Picado and Maria~Manuel Clementino.
\newblock {\em An {I}nterview with {W}illiam {F.} {L}awvere}.
\newblock Online, \url{https://www.mat.uc.pt/~picado/lawvere/interview.pdf}.,
  2007.

\bibitem[PGW17]{pickering17:profunctoroptics}
Matthew Pickering, Jeremy Gibbons, and Nicolas Wu.
\newblock Profunctor optics: Modular data accessors.
\newblock {\em Art Sci. Eng. Program.}, 1(2):7, 2017.

\bibitem[Pow02]{power02}
John Power.
\newblock Premonoidal categories as categories with algebraic structure.
\newblock {\em Theor. Comput. Sci.}, 278(1-2):303--321, 2002.

\bibitem[PR84]{penrose:kissingerquote}
Roger Penrose and Wolfgang Rindler.
\newblock {\em Spinors and Spacetime}.
\newblock Cited by {Aleks} {Kissinger} at the {Categories} mailing list.
  Cambridge University Press, 1984.

\bibitem[PR97]{power97}
John Power and Edmund Robinson.
\newblock Premonoidal categories and notions of computation.
\newblock {\em Math. Struct. Comput. Sci.}, 7(5):453--468, 1997.

\bibitem[PS93]{pierce93:subtyping}
Benjamin~C. Pierce and Davide Sangiorgi.
\newblock Typing and subtyping for mobile processes.
\newblock In {\em Proceedings of the Eighth Annual Symposium on Logic in
  Computer Science {(LICS} '93), Montreal, Canada, June 19-23, 1993}, pages
  376--385. {IEEE} Computer Society, 1993.

\bibitem[PS07]{pastro07}
Craig Pastro and Ross Street.
\newblock Doubles for {M}onoidal {C}ategories.
\newblock {\em arXiv preprint arXiv:0711.1859}, 2007.

\bibitem[PSV21]{patterson21:wiringdiagrams}
Evan Patterson, David~I. Spivak, and Dmitry Vagner.
\newblock Wiring diagrams as normal forms for computing in symmetric monoidal
  categories.
\newblock {\em Electronic Proceedings in Theoretical Computer Science}, page
  49–64, Feb 2021.

\bibitem[PT99]{power99:freyd}
John Power and Hayo Thielecke.
\newblock Closed {Freyd}- and kappa-categories.
\newblock In Jir{\'{\i}} Wiedermann, Peter van Emde~Boas, and Mogens Nielsen,
  editors, {\em Automata, Languages and Programming, 26th International
  Colloquium, ICALP'99, Prague, Czech Republic, July 11-15, 1999, Proceedings},
  volume 1644 of {\em Lecture Notes in Computer Science}, pages 625--634.
  Springer, 1999.

\bibitem[Ril18]{riley2018categories}
Mitchell Riley.
\newblock Categories of {O}ptics.
\newblock {\em arXiv preprint arXiv:1809.00738}, 2018.

\bibitem[Rom20a]{roman2020}
Mario Rom{\'{a}}n.
\newblock Comb {Diagrams} for {Discrete}-{Time} {Feedback}.
\newblock {\em CoRR}, abs/2003.06214, 2020.

\bibitem[Rom20b]{openDiagrams}
Mario Rom{\'{a}}n.
\newblock Open diagrams via coend calculus.
\newblock In David~I. Spivak and Jamie Vicary, editors, {\em Proceedings of the
  3rd Annual International Applied Category Theory Conference 2020, {ACT} 2020,
  Cambridge, USA, 6-10th July 2020}, volume 333 of {\em {EPTCS}}, pages 65--78,
  2020.

\bibitem[Rom22]{roman:promonads-string-diagrams}
Mario Rom{\'{a}}n.
\newblock Promonads and string diagrams for effectful categories.
\newblock In Jade Master and Martha Lewis, editors, {\em Proceedings Fifth
  International Conference on Applied Category Theory, {ACT} 2022, Glasgow,
  United Kingdom, 18-22 July 2022}, volume 380 of {\em {EPTCS}}, pages
  344--361, 2022.

\bibitem[Sel10]{selinger2010survey}
Peter Selinger.
\newblock A survey of graphical languages for monoidal categories.
\newblock In {\em New structures for physics}, pages 289--355. Springer, 2010.

\bibitem[Shu16]{shulman:catlog}
Michael Shulman.
\newblock Categorical logic from a categorical point of view.
\newblock {\em Available on the web}, 2016.

\bibitem[Shu17]{nlab:duoidal}
Michael Shulman.
\newblock Duoidal category (nlab entry), section 2., 2017.
\newblock \url{https://ncatlab.org/nlab/show/duoidal+category}, Last accessed
  on 2022-12-14.

\bibitem[Shu18]{shulman20182}
Michael Shulman.
\newblock The 2-{C}hu-{D}ialectica construction and the polycategory of
  multivariable adjunctions.
\newblock {\em arXiv preprint arXiv:1806.06082}, 2018.

\bibitem[SL13]{staton13}
Sam Staton and Paul~Blain Levy.
\newblock Universal properties of impure programming languages.
\newblock In Roberto Giacobazzi and Radhia Cousot, editors, {\em The 40th
  Annual {ACM} {SIGPLAN-SIGACT} Symposium on Principles of Programming
  Languages, {POPL} '13, Rome, Italy - January 23 - 25, 2013}, pages 179--192.
  {ACM}, 2013.

\bibitem[Sob10]{sobocinski10:representationspetrinets}
Pawe{\l} Soboci{\'n}ski.
\newblock Representations of {Petri} net interactions.
\newblock In {\em International Conference on Concurrency Theory}, pages
  554--568. Springer, 2010.

\bibitem[Sob13]{sobocinski13:graphicalLinearAlgebra}
Pawel Soboci{\'n}ski.
\newblock Graphical linear algebra.
\newblock Online, personal blog, \url{https://graphicallinearalgebra.net},
  2013.

\bibitem[Spi13]{spivak13}
David~I. Spivak.
\newblock The operad of wiring diagrams: formalizing a graphical language for
  databases, recursion, and plug-and-play circuits.
\newblock {\em CoRR}, abs/1305.0297, 2013.

\bibitem[SS22]{shapiro22:duoidal}
Brandon~T. Shapiro and David~I. Spivak.
\newblock Duoidal structures for compositional dependence.
\newblock {\em arXiv preprint arXiv:2210.01962}, 2022.

\bibitem[SSV20]{schultz20:dynamical}
Patrick Schultz, David~I. Spivak, and Christina Vasilakopoulou.
\newblock Dynamical systems and sheaves.
\newblock {\em Applied Categorical Structures}, 28(1):1--57, 2020.

\bibitem[Str04]{street04:frobeniusmonads}
Ross Street.
\newblock Frobenius monads and pseudomonoids.
\newblock {\em Journal of mathematical physics}, 45(10):3930--3948, 2004.

\bibitem[Str12]{street12:linking}
Ross Street.
\newblock Monoidal categories in, and linking, geometry and algebra.
\newblock {\em Bulletin of the Belgian Mathematical Society-Simon Stevin},
  19(5):769--820, 2012.

\bibitem[SW01]{sangiorgi01:picalculus}
Davide Sangiorgi and David Walker.
\newblock {\em The Pi-Calculus - a theory of mobile processes}.
\newblock Cambridge University Press, 2001.

\bibitem[Sza75]{szabo75:polycategories}
Manfred~E. Szabo.
\newblock Polycategories.
\newblock {\em Communications in Algebra}, 3(8):663--689, 1975.

\bibitem[{Tod}10]{trimble:coherence}
{Todd Trimble}.
\newblock Coherence theorem for monoidal categories (nlab entry), section 3.
  discussion, 2010.
\newblock
  \url{https://ncatlab.org/nlab/show/coherence+theorem+for+monoidal+categories},
  Last accessed on 2022-05-10.

\bibitem[UV08]{uustalu2008comonadic}
Tarmo Uustalu and Varmo Vene.
\newblock Comonadic notions of computation.
\newblock In Jiří Ad{\'{a}}mek and Clemens Kupke, editors, {\em Proceedings
  of the Ninth Workshop on Coalgebraic Methods in Computer Science, {CMCS}
  2008, Budapest, Hungary, April 4-6, 2008}, volume 203 of {\em Electronic
  Notes in Theoretical Computer Science}, pages 263--284. Elsevier, 2008.

\bibitem[UVZ18]{uustalu18:sequent}
Tarmo Uustalu, Niccol{\`o} Veltri, and Noam Zeilberger.
\newblock The sequent calculus of skew monoidal categories.
\newblock {\em Electronic Notes in Theoretical Computer Science}, 341:345--370,
  2018.

\bibitem[VC22]{videlacapucci22}
Andr{\'{e}} Videla and Matteo Capucci.
\newblock Lenses for composable servers.
\newblock {\em CoRR}, abs/2203.15633, 2022.

\bibitem[vdW21]{deWetering21:interval}
John van~de Wetering.
\newblock A categorical construction of the real unit interval.
\newblock {\em arXiv preprint arXiv:2106.10094}, 2021.

\bibitem[vN58]{vonneumann20:computerandbrain}
John von Neumann.
\newblock {\em The {C}omputer and the {B}rain}.
\newblock Yale {U}niversity {P}ress, quoted by David Darlymple in ``A {N}ew
  {T}ype of {M}athematics'' (2018), a Transcript from a talk at Montreal, 1958.

\bibitem[vS]{vosSavant}
Marilyn vos Savant.
\newblock {Parade 16: Ask Marilyn (Archived)}.
\newblock
  \url{https://web.archive.org/web/20130121183432/http://marilynvossavant.com/game-show-problem/}.
\newblock Accessed: 2013-01-21.

\bibitem[WC22]{wilson22:mathematical}
Matt Wilson and Giulio Chiribella.
\newblock A mathematical framework for transformations of physical processes.
\newblock {\em arXiv preprint arXiv:2204.04319}, 2022.

\bibitem[YS17]{yudkowsky:soares:17}
Eliezer Yudkowsky and Nate Soares.
\newblock Functional {Decision} {Theory}: a {New} {Theory} of {Instrumental}
  {Rationality}.
\newblock {\em ArXiv preprint arXiv:1710.05060}, 2017.

\end{thebibliography}

\appendix
\chapter*{Appendix}
\clearpage{}%

\section{Coherence diagrams for a duoidal category}  
\label{sec:coherencediagramsduoidal}

\begin{figure}[ht]
  \centering
  \begin{tikzcd}
    ((A ◁ B) ⊗ (C ◁ D)) ⊗ (E ◁ F) 
    \rar{\alpha} \dar[swap]{ψ₂ ⊗ id} &
    (A ◁ B) ⊗ ((C ◁ D) ⊗ (E ◁ F))
    \dar{id ⊗ ψ₂} \\
    ((A ⊗ C) ◁ (B ⊗ D)) ⊗ (E ◁ F)
    \dar[swap]{ψ₂} &
    (A ◁ B) ⊗ ((C ⊗ E) ◁ (D ⊗ F))
    \dar{ψ₂} \\
    ((A ⊗ C) ⊗ E) ◁ ((B ⊗ D) ⊗ F)
    \rar{α ◁ α} &
    (A ⊗ (C ⊗ E)) ◁ (B ⊗ (D ⊗ F))
  \end{tikzcd}
  \begin{tikzcd}
    ((A ◁ B) ◁ C) ⊗ ((D ◁ E) ◁ F) 
    \rar{\beta ⊗ \beta} \dar[swap]{ψ₂} &
    (A ◁ (B ◁ C)) ⊗ (D ◁ (E ◁ F))
    \dar{ψ₂} \\
    ((A ◁ B) ⊗ (D ◁ E)) ◁ (C ⊗ F)
    \dar[swap]{ψ₂ ⊗ id} &
    (A ⊗ D) ◁ ((B ◁ C) ⊗ (E ◁ F))
    \dar{id ⊗ ψ₂} \\
    ((A ⊗ D) ◁ (B ⊗ E)) ◁ (C ⊗ F)
    \rar{\beta} &
    (A ⊗ D) ◁ ((B ⊗ E) ◁ (C ⊗ F))
  \end{tikzcd}
  \caption{Coherence diagrams for associativity of a duoidal category.}
  \label{cd:duoidal-coherence-assoc}
\end{figure}

\begin{figure}
  \begin{tikzcd}
    I ⊗ (A ◁ B) 
    \rar{ψ₀ ⊗ id} \dar[swap]{\lambda} &
    (I ◁ I) ⊗ (A ◁ B) 
    \dar{ψ₂} \\
    A ◁ B
    &
    (I ⊗ A) ◁ (I ⊗ B)
    \lar{\lambda ◁ \lambda}
  \end{tikzcd}
  \begin{tikzcd}
    (A ◁ B) ⊗ I
    \rar{ψ₀ ⊗ id} \dar[swap]{\rho} &
    (A ◁ B) ⊗ (I ◁ I)
    \dar{ψ₂} \\
    A ◁ B
    &
    (A ⊗ I) ◁ (B ⊗ I)
    \lar{\rho ◁ \rho}
  \end{tikzcd}
  \caption{Coherence diagrams for $⊗$-unitality of a duoidal category.}
  \label{cd:duoidal-coherence-unit}
\end{figure}

\begin{figure}
  \begin{tikzcd}
    N ◁ (A ⊗ B)
    \dar[swap]{κ} &
    (N ⊗ N) ◁ (A ⊗ B)
    \lar[swap]{\varphi_2 ◁ id} \dar{ψ₂} \\
    A ⊗ B
    &
    (N ◁ A) ⊗ (N ◁ B)
    \lar{κ ⊗ κ}
  \end{tikzcd}
  \begin{tikzcd}
    (A ⊗ B) ◁ N
    \dar[swap]{ν} &
    (A ⊗ B)  ◁ (N ⊗ N)
    \lar[swap]{id ◁ \varphi_2} \dar{ψ₂} \\
    A ⊗ B
    &
    (A ◁ N) ⊗ (B ◁ N)
    \lar{ν ⊗ ν}
  \end{tikzcd}
  \caption{Coherence diagrams for $◁$-unitality of a duoidal category.}
  \label{cd:duoidal-coherence-unit2}
\end{figure}

\begin{figure}
  \begin{tikzcd}
    (N ⊗ N) ⊗ N
    \ar{rr}{\alpha} \dar[swap]{\varphi_2 ⊗ id}&&
    N ⊗ (N ⊗ N)
    \dar{id ⊗ \varphi_2}\\
    N ⊗ N
    \rar[swap]{\varphi_2} &
    N
    &
    N ⊗ N
    \lar{\varphi_2}
  \end{tikzcd}
  \begin{tikzcd}
    I ◁ I
    \dar[swap]{\psi_0 ⊗ id} &
    I 
    \lar[swap]{\psi_0} \rar{\psi_0} &
    I ◁ I
    \dar{id ⊗ \psi_0} \\
    (I ◁ I) ◁ I
    \ar{rr}[swap]{\beta} &&
    I ◁ (I ◁ I)
  \end{tikzcd}
  \caption{Associativity and coassociativity for $N$ and $I$ in a duoidal category.}\label{cd:duoidal-coherence-bimonoids}

  \begin{tikzcd}
    N ⊗ I
    \rar{ρ} \dar[swap]{id ⊗ \varphi_0} &
    N
    \\
    N ⊗ N 
    \urar[swap]{\varphi_2} &
  \end{tikzcd}
  \begin{tikzcd}
    I ⊗ N
    \rar{\lambda} \dar[swap]{\varphi_0 ⊗ id} &
    N
    \\
    N ⊗ N 
    \urar[swap]{\varphi_2} &
  \end{tikzcd}
  \begin{tikzcd}
    I ◁ N
    \rar{id ⊗ \varphi_0} \dar[swap]{ν} &
    I ◁ I
    \\
    I 
    \urar[swap]{\psi_0} &
  \end{tikzcd}
  \begin{tikzcd}
    N ◁ I
    \rar{id ⊗ \varphi_0} \dar[swap]{κ} &
    I ◁ I
    \\
    I 
    \urar[swap]{\psi_0} &
  \end{tikzcd}
  \caption{Unitality and counitality for $N$ and $I$ in a duoidal category.}\label{cd:duoidal-coherence-nandi}
\end{figure}

\clearpage

\clearpage{}%
\clearpage{}%
\section{Polycategories}

This extra section repeats the splice-contour adjunction for \polycategories{}. It is a detour from the main text; and it is not necessary for its development: this does not seem to be the direction we want to follow to study context in categories or \monoidalCategories{}. It is, however, another proof of the resilience of the splice-contour adjunction: the duality between a category and its opposite induces a pseudofrobenius algebra on the monoidal bicategory of profunctors.

{ %
\renewcommand\UnicodeBlackboardP{\ensuremath{\mathbb{P}}} %
\subsection{Polycategories}

A \polycategory{} is like a category where every morphism has both a list of inputs and a list of outputs \cite{szabo75:polycategories}. This does not mean that its inputs and outpuss start forming a monoid, as in strict monoidal categories; morphisms really have different multiple inputs and outputs, and we need to choose a single one to compose along it.
A polycategory, $ℙ$, contains a set of objects, $ℙ_{obj}$, as categories and multicategories; but instead of a set of morphisms, it will have a set of \emph{polymorphisms},
$$ℙ(X₁,…,Xₙ; Y₁,…, Yₘ),$$
for each two lists of objects $X₁, …, Xₙ , Y₀, …, Yₘ ∈ ℙ_{obj}$. As in linear logic, we denote both sides of a derivation by two metavariables, $Γ = X₁, …, Xₙ$ and $Δ = Y₀, …, Yₘ$, and write $ℙ(Γ; Δ)$ for the set of polymorphisms.

\begin{definition}
  \defining{linkPolycategory}{}\label{def:polycategory}
  A \emph{polycategory} $ℙ$ is a collection of objects, $ℙ_{obj}$, together with a collection of \emph{polymorphisms}, $ℙ(X₀,\dots,Xₙ; Y₀,\dots,Yₘ)$, for each two lists of objects $X₀,\dots,Xₙ \in ℙ$ and $Y₀,\dots,Yₘ \in ℙ$. For each object $X \in ℙ_{obj}$, there must be an identity, $\mathrm{id}_X \in ℙ_{obj}(X;X)$; and for each pair of composable maps,
  \begin{align*}
  & f \in ℙ(\Gamma ; \Delta₁, X, \Delta₂) \mbox{ and } g \in ℙ(\Gamma₁,X,\Gamma₂; \Delta),\\
  & \qquad\mbox{ where either }
  \Delta₁\mbox{ or }\Gamma₁,\mbox{ and either }\Delta₂\mbox{ or }\Gamma₂,\mbox{ are empty,}
  \end{align*}
  there must be a composite polymorphism $f ⨾_X g \in ℙ(\Gamma₁,\Gamma,\Gamma₂; \Delta₁,\Delta,\Delta₂)$.
  This means that there are four possible types of composition (\Cref{fig:polycategorical-composition}), and they yield the same polymorphism whenever they overlap.

	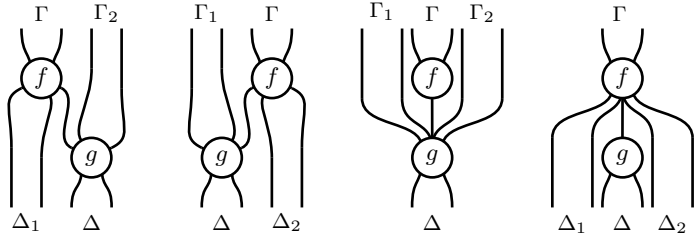
\begin{figure}[h]
	\centering

\begin{tikzpicture}[x=0.75pt,y=0.75pt,yscale=-1,xscale=1,line width=1pt]
\draw    (265.06,109.94) -- (265,80) ;
\draw    (270,110) .. controls (269.78,124.33) and (274.89,119.89) .. (275,135) ;
\draw    (260,110) .. controls (259.78,124.33) and (254.89,119.89) .. (255,135) ;
\draw    (265,105) .. controls (264.78,90.67) and (279.89,95.11) .. (280,80) ;
\draw    (265,105) .. controls (264.78,90.67) and (249.89,95.11) .. (250,80) ;
\draw    (90,115) .. controls (90.22,126.47) and (85.11,122.91) .. (85,135) ;
\draw    (100,115) .. controls (100.22,126.47) and (105.11,122.91) .. (105,135) ;
\draw    (95.03,104.97) .. controls (80.56,99.67) and (95.22,84.78) .. (95,65) ;
\draw    (70,75) .. controls (81,81) and (70.11,89.89) .. (70,110) ;
\draw    (65,70) .. controls (65.22,55.67) and (60.11,60.11) .. (60,45) ;
\draw    (75,70) .. controls (75.22,55.67) and (80.11,60.11) .. (80,45) ;
\draw    (70,70) -- (60,75) ;
\draw    (105.03,104.97) -- (95.03,109.97) ;
\draw    (94.94,109.94) -- (85,105) ;
\draw  [fill={rgb, 255:red, 255; green, 255; blue, 255 }  ,fill opacity=1 ] (104.88,109.94) .. controls (104.88,104.45) and (100.43,100) .. (94.94,100) .. controls (89.45,100) and (85,104.45) .. (85,109.94) .. controls (85,115.43) and (89.45,119.88) .. (94.94,119.88) .. controls (100.43,119.88) and (104.88,115.43) .. (104.88,109.94) -- cycle ;
\draw    (80,75) -- (70,70) ;
\draw  [fill={rgb, 255:red, 255; green, 255; blue, 255 }  ,fill opacity=1 ] (79.88,70.06) .. controls (79.88,75.55) and (75.43,80) .. (69.94,80) .. controls (64.45,80) and (60,75.55) .. (60,70.06) .. controls (60,64.57) and (64.45,60.12) .. (69.94,60.12) .. controls (75.43,60.12) and (79.88,64.57) .. (79.88,70.06) -- cycle ;
\draw    (80,75) .. controls (89.73,80.09) and (74.82,100.27) .. (85,105) ;
\draw    (59.97,75.03) .. controls (49.73,80.09) and (55,85.36) .. (55,105) ;
\draw    (55,135) -- (55,105) ;
\draw    (110,75) -- (110,45) ;
\draw    (105.03,104.97) .. controls (115.27,99.91) and (110,94.64) .. (110,75) ;
\draw  [draw opacity=0] (125,40) -- (50,40) -- (50,140) -- (125,140) -- cycle ;
\draw    (95,65) -- (95,45) ;
\draw    (70,135) -- (70,110) ;
\draw  [draw opacity=0] (55,135) -- (70,135) -- (70,150) -- (55,150) -- cycle ;
\draw  [draw opacity=0] (85,135) -- (105,135) -- (105,150) -- (85,150) -- cycle ;
\draw  [draw opacity=0] (95,30) -- (110,30) -- (110,45) -- (95,45) -- cycle ;
\draw  [draw opacity=0] (60,30) -- (80,30) -- (80,45) -- (60,45) -- cycle ;
\draw    (165,115) .. controls (164.78,126.47) and (169.89,122.91) .. (170,135) ;
\draw    (155,115) .. controls (154.78,126.47) and (149.89,122.91) .. (150,135) ;
\draw    (159.97,104.97) .. controls (174.44,99.67) and (159.78,84.78) .. (160,65) ;
\draw    (185,75) .. controls (174,81) and (184.89,89.89) .. (185,110) ;
\draw    (190,70) .. controls (189.78,55.67) and (194.89,60.11) .. (195,45) ;
\draw    (180,70) .. controls (179.78,55.67) and (174.89,60.11) .. (175,45) ;
\draw    (185,70) -- (195,75) ;
\draw    (149.97,104.97) -- (159.97,109.97) ;
\draw    (160.06,109.94) -- (170,105) ;
\draw  [fill={rgb, 255:red, 255; green, 255; blue, 255 }  ,fill opacity=1 ] (150.12,109.94) .. controls (150.12,104.45) and (154.57,100) .. (160.06,100) .. controls (165.55,100) and (170,104.45) .. (170,109.94) .. controls (170,115.43) and (165.55,119.88) .. (160.06,119.88) .. controls (154.57,119.88) and (150.12,115.43) .. (150.12,109.94) -- cycle ;
\draw    (175,75) -- (185,70) ;
\draw  [fill={rgb, 255:red, 255; green, 255; blue, 255 }  ,fill opacity=1 ] (175.12,70.06) .. controls (175.12,75.55) and (179.57,80) .. (185.06,80) .. controls (190.55,80) and (195,75.55) .. (195,70.06) .. controls (195,64.57) and (190.55,60.12) .. (185.06,60.12) .. controls (179.57,60.12) and (175.12,64.57) .. (175.12,70.06) -- cycle ;
\draw    (175,75) .. controls (165.27,80.09) and (180.18,100.27) .. (170,105) ;
\draw    (195.03,75.03) .. controls (205.27,80.09) and (200,85.36) .. (200,105) ;
\draw    (200,135) -- (200,105) ;
\draw    (145,75) -- (145,45) ;
\draw    (149.97,104.97) .. controls (139.73,99.91) and (145,94.64) .. (145,75) ;
\draw    (160,65) -- (160,45) ;
\draw    (185,135) -- (185,110) ;
\draw  [draw opacity=0] (145,30) -- (160,30) -- (160,45) -- (145,45) -- cycle ;
\draw  [draw opacity=0] (175,30) -- (195,30) -- (195,45) -- (175,45) -- cycle ;
\draw  [draw opacity=0] (150,135) -- (170,135) -- (170,150) -- (150,150) -- cycle ;
\draw  [draw opacity=0] (185,135) -- (200,135) -- (200,150) -- (185,150) -- cycle ;
\draw    (270,70) .. controls (269.78,55.67) and (274.89,60.11) .. (275,45) ;
\draw    (260,70) .. controls (259.78,55.67) and (254.89,60.11) .. (255,45) ;
\draw  [fill={rgb, 255:red, 255; green, 255; blue, 255 }  ,fill opacity=1 ] (255.12,70.06) .. controls (255.12,75.55) and (259.57,80) .. (265.06,80) .. controls (270.55,80) and (275,75.55) .. (275,70.06) .. controls (275,64.57) and (270.55,60.12) .. (265.06,60.12) .. controls (259.57,60.12) and (255.12,64.57) .. (255.12,70.06) -- cycle ;
\draw    (270,105) .. controls (269.78,90.67) and (299.89,95.11) .. (300,80) ;
\draw    (260,105) .. controls (259.78,90.67) and (229.89,95.11) .. (230,80) ;
\draw  [fill={rgb, 255:red, 255; green, 255; blue, 255 }  ,fill opacity=1 ] (255.12,109.94) .. controls (255.12,115.43) and (259.57,119.88) .. (265.06,119.88) .. controls (270.55,119.88) and (275,115.43) .. (275,109.94) .. controls (275,104.45) and (270.55,100) .. (265.06,100) .. controls (259.57,100) and (255.12,104.45) .. (255.12,109.94) -- cycle ;
\draw    (230,80) -- (230,45) ;
\draw    (250,80) -- (250,45) ;
\draw    (280,80) -- (280,45) ;
\draw    (300,80) -- (300,45) ;
\draw  [draw opacity=0] (255,30) -- (275,30) -- (275,45) -- (255,45) -- cycle ;
\draw  [draw opacity=0] (280,30) -- (300,30) -- (300,45) -- (280,45) -- cycle ;
\draw  [draw opacity=0] (230,29) -- (250,29) -- (250,44) -- (230,44) -- cycle ;
\draw  [draw opacity=0] (255,135) -- (275,135) -- (275,150) -- (255,150) -- cycle ;
\draw    (360.06,70.06) -- (360,100) ;
\draw    (365,70) .. controls (364.78,55.67) and (369.89,60.11) .. (370,45) ;
\draw    (355,70) .. controls (354.78,55.67) and (349.89,60.11) .. (350,45) ;
\draw    (360,75) .. controls (359.78,89.33) and (374.89,84.89) .. (375,100) ;
\draw    (360,75) .. controls (359.78,89.33) and (344.89,84.89) .. (345,100) ;
\draw    (365,110) .. controls (364.78,124.33) and (369.89,119.89) .. (370,135) ;
\draw    (355,110) .. controls (354.78,124.33) and (349.89,119.89) .. (350,135) ;
\draw  [fill={rgb, 255:red, 255; green, 255; blue, 255 }  ,fill opacity=1 ] (350.12,109.94) .. controls (350.12,104.45) and (354.57,100) .. (360.06,100) .. controls (365.55,100) and (370,104.45) .. (370,109.94) .. controls (370,115.43) and (365.55,119.88) .. (360.06,119.88) .. controls (354.57,119.88) and (350.12,115.43) .. (350.12,109.94) -- cycle ;
\draw    (365,75) .. controls (364.78,89.33) and (394.89,84.89) .. (395,100) ;
\draw    (355,75) .. controls (354.78,89.33) and (324.89,84.89) .. (325,100) ;
\draw  [fill={rgb, 255:red, 255; green, 255; blue, 255 }  ,fill opacity=1 ] (350.12,70.06) .. controls (350.12,64.57) and (354.57,60.12) .. (360.06,60.12) .. controls (365.55,60.12) and (370,64.57) .. (370,70.06) .. controls (370,75.55) and (365.55,80) .. (360.06,80) .. controls (354.57,80) and (350.12,75.55) .. (350.12,70.06) -- cycle ;
\draw    (325,100) -- (325,135) ;
\draw    (345,100) -- (345,135) ;
\draw    (375,100) -- (375,135) ;
\draw    (395,100) -- (395,135) ;
\draw  [draw opacity=0] (350,30) -- (370,30) -- (370,45) -- (350,45) -- cycle ;
\draw  [draw opacity=0] (325,136) -- (345,136) -- (345,151) -- (325,151) -- cycle ;
\draw  [draw opacity=0] (350,135) -- (370,135) -- (370,150) -- (350,150) -- cycle ;
\draw  [draw opacity=0] (375,136) -- (395,136) -- (395,151) -- (375,151) -- cycle ;

\draw (69.91,70.14) node  [font=\small]  {$f$};
\draw (95.03,109.97) node  [font=\small]  {$g$};
\draw (62.5,142.5) node  [font=\footnotesize]  {$\Delta _{1}$};
\draw (95,142.5) node  [font=\footnotesize]  {$\Delta $};
\draw (102.5,37.5) node  [font=\footnotesize]  {$\Gamma _{2}$};
\draw (70,37.5) node  [font=\footnotesize]  {$\Gamma $};
\draw (185.06,70.06) node  [font=\small]  {$f$};
\draw (159.97,109.97) node  [font=\small]  {$g$};
\draw (152.5,37.5) node  [font=\footnotesize]  {$\Gamma _{1}$};
\draw (185,37.5) node  [font=\footnotesize]  {$\Gamma $};
\draw (160,142.5) node  [font=\footnotesize]  {$\Delta $};
\draw (192.5,142.5) node  [font=\footnotesize]  {$\Delta _{2}$};
\draw (265.06,70.06) node  [font=\small]  {$f$};
\draw (265.06,109.94) node  [font=\small]  {$g$};
\draw (265,37.5) node  [font=\footnotesize]  {$\Gamma $};
\draw (290,37.5) node  [font=\footnotesize]  {$\Gamma _{2}$};
\draw (240,36.5) node  [font=\footnotesize]  {$\Gamma _{1}$};
\draw (265,142.5) node  [font=\footnotesize]  {$\Delta $};
\draw (360.06,109.94) node  [font=\small]  {$g$};
\draw (360.06,70.06) node  [font=\small]  {$f$};
\draw (358.5,37.5) node  [font=\footnotesize]  {$\Gamma $};
\draw (335,143.5) node  [font=\footnotesize]  {$\Delta _{1}$};
\draw (360,142.5) node  [font=\footnotesize]  {$\Delta $};
\draw (385,143.5) node  [font=\footnotesize]  {$\Delta _{2}$};

\end{tikzpicture}
 	\caption{Four planar polycategorical compositions.}
	\label{fig:polycategorical-composition}
  \end{figure}

  Moreover, polycategories must satisfy the following two unitality axioms, $f ⨾_X \mathrm{id}_X = f$ and $\mathrm{id}_X ⨾_X f = f$; two associativity axioms, $f ⨾_X (g ⨾_Y h) = (f ⨾_X g) ⨾_Y h$ and $f ⨾_X (g ⨾_Y h) = g ⨾_Y (f ⨾_X h)$; and an interchange axiom, $(f ⨾_A g) ⨾_B h = (f ⨾_B h) ⨾_A g$; whenever any of these is formally well-typed.
\end{definition}

\begin{remark}
	Asking for an identity on each object, $\mathrm{id}_X \in ℙ(X;X)$, is different from asking for an identity on each list of objects, $$\mathrm{id}_{X₀,\dots,Xₙ} \in ℙ(X₀,\dots,Xₙ;X₀,\dots,Xₙ).$$ The latter gives rise to \emph{isomix} categories \cite{cockett97:prooftheory} and we will not discuss it here.
\end{remark}

\begin{example}
	A \emph{polyfunctional relation}, $R \in \mathbf{MultiFun}(A₁,\dots, Aₙ; B₁,\dots, Bₘ)$, is a relation $R \colon A₁ × \dots × Aₙ → B₁ × \dots × Bₘ$ together with representing functions that, given an element of the relation missing exactly one element, return the element missing. Explicitly, there exist two families of functions,
	\begin{align*}
	& fⱼ ፡ A₁ × \dots × A_n × B₀ × \overset{\cancel{B_j}}\dots × B_m → B_j
	\quad\mbox{ and } \\
  & g_i \colon A₁ \times \overset{\cancel{A_i}}\dots \times Aₙ × B_0 × \dots × B_m → Aᵢ,
	\end{align*}
	such that $R(a₁,…,bₘ)$ if and only if $fⱼ(a₁,…,bₘ) = bⱼ$ and if and only if $g_i(a₁,…,bₘ) = aᵢ$ for each two indices $i$ and $j$. Polyfunctional relations form a \polycategory{} with relational composition.

	A $(1,1)$-polyfunctional relation is a pair of inverse functions. A (2,1) or (1,2)-polyfunctional relation is a triple functions $f₀ ፡ A₁ × A₂ → A₀$, $f₁ ፡ A₂ × A₀ → A₁$ and $f₂ ፡ A₀ × A₁ → A₂$ such that 
	$f₁(a_1,a_2) = a_0$ if and only if
	$f_1(a_2,a_0) = a_1$ and if and only if
	$f_2(a_0,a_1) = a_2$.
	Polyfunctional relations are a decategorification of the \emph{multivariable adjunctions} in the work of Shulman \cite{shulman20182}.
\end{example}

\subsection{The Category of Polycategories}
Analogously to the categorical and multicategorical case, the theory of \polyfunctors{} and polynatural transformations is synthetised in the 2-category $\mathbf{PolyCat}$ of polycategories, polyfunctors and polynatural transformations.

\begin{definition}
  \defining{linkPolyfunctor}{}
	A \emph{polyfunctor}, $F \colon ℙ → ℚ$, between two \polycategories{} $ℙ$ and $ℚ$, is an assignment on objects, $F_{obj} \colon ℙ_{obj} → ℚ_{obj}$ together with an assignment on polymorphisms
	$$F_{n,m} \colon ℙ(X₀, \dots, Xₙ; Y₀, \dots, Yₘ) → ℚ(F_{obj}X₀, \dots, F_{obj}Xₙ; F_{obj}Y₀, \dots, F_{obj}Yₘ).$$
	This assignment must be functorial, in that $F(f ⨾_{X_i} g) = F(f) ⨾_{X_i} F(g)$ and that $F(\mathrm{id}_X) = \mathrm{id}_{FX}$, whenever these are formally well-typed. 
\end{definition}

\begin{definition}
  A polynatural transformation $θ ፡ F → G$ between two \polyfunctors{} $F, G ፡ ℙ → ℚ$ is given by a family of polymorphisms $θ_X ∈ ℚ(FX; GX)$ such that, for each polymorphism $f ∈ ℙ(X₁, …, Xₙ; Y₁, …, Yₘ)$, the following naturality condition holds
  $$θ_{X₁} ⨾ … ⨾ θ_{Xₙ} ⨾ G(f) = F(f) ⨾ θ_{Y₁} ⨾ … ⨾ θ_{Yₘ}.$$
\end{definition}

\begin{definition}
  Polyfunctors between polycategories form a category, $\mathbf{PolyCat}$.
  This is moreover a 2-category with polynatural transformations.
\end{definition}

\begin{remark}
  In the same sense that a multifunctor from the terminal \multicategory{} picks a monoid, a polyfunctor from the terminal \polycategory{} should pick a \emph{polyoid} -- instead, we call these \emph{Frobenius monoids}.
\end{remark}

}

\subsection{Polycategorical Contour}

\begin{definition}
  \defining{linkContourPolycategory}{}
  Let $ℙ$ be a \polycategory{}. Its contour, $\Contour{ℙ}$, is the category presented by the following generators and equations:
  \begin{itemize}
    \item two polarized objects, $Xˡ$ and $Xʳ$, for each object $X ∈ ℙ_{obj}$;
    \item for each polymorphism, $f ∈ ℙ(X₁,…,Xₙ ; Y₁,…,Yₘ)$, the following generators,
    \begin{align*}
      & fʳ_1 ፡ Xʳ_1 → Xˡ_2, \quad …, \quad fʳ_{n-1} ፡ Xʳ_{n-1} → Xˡ_n, \quad fᵈ ፡ Xʳ_n → Y^r_m, \\
      & fˡ_1 ፡ Yˡ_2 → Yʳ_1, \quad …, \quad fˡ_{m-1} ፡ Yˡ_m → Yʳ_{m-1}, \quad fᵘ ፡ Yˡ_1 → Xˡ_1,
    \end{align*}
    having instead $fᵘ ፡ Yˡ_1 → Yʳ_m$ when $n = 0$, having instead $fᵈ ፡ Xʳ_n → Xˡ_1$ when $m = 0$, and using no generators for $(0,0)$-polymorphisms;
  \end{itemize}
  to which we impose equations requiring contour to preserve identities, $(\id_X)ᵘ = \id_{Xˡ}$ and $(\id_X)ᵈ = \id_{Xʳ}$;
  and requiring contour to preserve compositions, meaning that for each $f ∈ ℙ(X₁,…,Xₙ;Y₁,…,Yₘ)$ and each $g ∈ ℙ(Z₁,…,Zₚ;Q₁,…,Q_q)$ such that $Yₘ = U₁$, the contour of the composition along $Yₘ = U₁$ is defined by the following eight cases
  \begin{align*}
  & (f ⨾_{Xᵢ} g)ᵘ = fᵘ;
  && (f ⨾_{Xᵢ} g)ʳ_i = f^r_i,\mbox{ for } i = 1,…,n-1; \\
  & (f ⨾_{Yₘ} g)ʳ_n = fᵈ ⨾ gʳ_1;
  && (f ⨾_{Yₘ} g)ʳ_j = gʳ_{j-n+1},\mbox{ for } j = n+1, …, n+p-2; \\
  & (f ⨾_{Yₘ} g)ᵈ = gᵈ;
  && (f ⨾_{Yₘ} g)ˡ_i = fˡ_i,\mbox{ for } i = 1,…,m-2; \\
  & (f ⨾_{Yₘ} g)ᵈ_{m-1} = gᵘ ⨾ fˡ_{m-1};
  && (f ⨾_{Yₘ} g)ᵈ_{j} = gˡ_{j-m+1},\mbox{ for } j = m+1,…,m+q-2;
  \end{align*}
  and similar conditions for the rest of the compositions. These equations are depicted in \Cref{fig:contour-proautonomous}.

\end{definition}

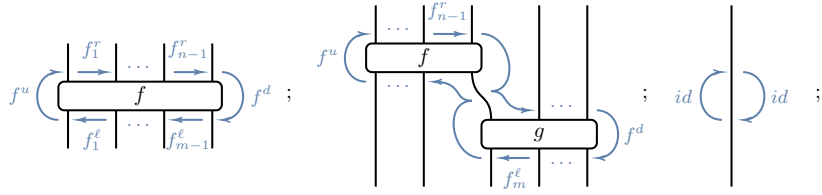
\begin{figure}[h]
  \centering
  \scalebox{0.8}{%

\tikzset{every picture/.style={line width=0.95pt}} %

\begin{tikzpicture}[x=0.75pt,y=0.75pt,yscale=-1.2,xscale=1.2]
\draw [color={rgb, 255:red, 94; green, 129; blue, 172 }  ,draw opacity=1 ]   (85,85) -- (98,85) ;
\draw [shift={(100,85)}, rotate = 180] [color={rgb, 255:red, 94; green, 129; blue, 172 }  ,draw opacity=1 ][line width=0.75]    (4.37,-1.32) .. controls (2.78,-0.56) and (1.32,-0.12) .. (0,0) .. controls (1.32,0.12) and (2.78,0.56) .. (4.37,1.32)   ;
\draw [color={rgb, 255:red, 94; green, 129; blue, 172 }  ,draw opacity=1 ]   (72.92,85.22) .. controls (60.44,87.7) and (60.89,110) .. (75,110) ;
\draw [shift={(75,85)}, rotate = 178.88] [color={rgb, 255:red, 94; green, 129; blue, 172 }  ,draw opacity=1 ][line width=0.75]    (4.37,-1.32) .. controls (2.78,-0.56) and (1.32,-0.12) .. (0,0) .. controls (1.32,0.12) and (2.78,0.56) .. (4.37,1.32)   ;
\draw  [draw opacity=0] (110,75) -- (125,75) -- (125,90) -- (110,90) -- cycle ;
\draw  [draw opacity=0] (85,70) -- (100,70) -- (100,85) -- (85,85) -- cycle ;
\draw    (80,105) -- (80,125) ;
\draw    (105,105) -- (105,125) ;
\draw    (130,105) -- (130,125) ;
\draw    (155,105) -- (155,125) ;
\draw    (80,70) -- (80,90) ;
\draw    (105,70) -- (105,90) ;
\draw    (130,70) -- (130,90) ;
\draw    (155,70) -- (155,90) ;
\draw [color={rgb, 255:red, 94; green, 129; blue, 172 }  ,draw opacity=1 ]   (135,85) -- (148,85) ;
\draw [shift={(150,85)}, rotate = 180] [color={rgb, 255:red, 94; green, 129; blue, 172 }  ,draw opacity=1 ][line width=0.75]    (4.37,-1.32) .. controls (2.78,-0.56) and (1.32,-0.12) .. (0,0) .. controls (1.32,0.12) and (2.78,0.56) .. (4.37,1.32)   ;
\draw  [draw opacity=0] (135,70) -- (150,70) -- (150,85) -- (135,85) -- cycle ;
\draw [color={rgb, 255:red, 94; green, 129; blue, 172 }  ,draw opacity=1 ]   (100,110) -- (87,110) ;
\draw [shift={(85,110)}, rotate = 360] [color={rgb, 255:red, 94; green, 129; blue, 172 }  ,draw opacity=1 ][line width=0.75]    (4.37,-1.32) .. controls (2.78,-0.56) and (1.32,-0.12) .. (0,0) .. controls (1.32,0.12) and (2.78,0.56) .. (4.37,1.32)   ;
\draw  [draw opacity=0] (110,105) -- (125,105) -- (125,120) -- (110,120) -- cycle ;
\draw  [draw opacity=0] (45,90) -- (65,90) -- (65,105) -- (45,105) -- cycle ;
\draw [color={rgb, 255:red, 94; green, 129; blue, 172 }  ,draw opacity=1 ]   (150,110) -- (137,110) ;
\draw [shift={(135,110)}, rotate = 360] [color={rgb, 255:red, 94; green, 129; blue, 172 }  ,draw opacity=1 ][line width=0.75]    (4.37,-1.32) .. controls (2.78,-0.56) and (1.32,-0.12) .. (0,0) .. controls (1.32,0.12) and (2.78,0.56) .. (4.37,1.32)   ;
\draw  [draw opacity=0] (135,110) -- (150,110) -- (150,125) -- (135,125) -- cycle ;
\draw   (75,93) .. controls (75,91.34) and (76.34,90) .. (78,90) -- (157,90) .. controls (158.66,90) and (160,91.34) .. (160,93) -- (160,102) .. controls (160,103.66) and (158.66,105) .. (157,105) -- (78,105) .. controls (76.34,105) and (75,103.66) .. (75,102) -- cycle ;
\draw  [draw opacity=0] (240,55) -- (265,55) -- (265,70) -- (240,70) -- cycle ;
\draw    (240,85) -- (240,145) ;
\draw    (265,85) -- (265,145) ;
\draw    (240,50) -- (240,70) ;
\draw    (290,50) -- (290,70) ;
\draw   (235,73) .. controls (235,71.34) and (236.34,70) .. (238,70) -- (292,70) .. controls (293.66,70) and (295,71.34) .. (295,73) -- (295,82) .. controls (295,83.66) and (293.66,85) .. (292,85) -- (238,85) .. controls (236.34,85) and (235,83.66) .. (235,82) -- cycle ;
\draw   (295,113) .. controls (295,111.34) and (296.34,110) .. (298,110) -- (352,110) .. controls (353.66,110) and (355,111.34) .. (355,113) -- (355,122) .. controls (355,123.66) and (353.66,125) .. (352,125) -- (298,125) .. controls (296.34,125) and (295,123.66) .. (295,122) -- cycle ;
\draw    (290,85) .. controls (290.67,100.83) and (300,95.17) .. (300,110) ;
\draw    (325,50) -- (325,110) ;
\draw    (350,50) -- (350,110) ;
\draw    (300,125) -- (300,145) ;
\draw    (350,125) -- (350,145) ;
\draw [color={rgb, 255:red, 94; green, 129; blue, 172 }  ,draw opacity=1 ]   (290,100) .. controls (274.67,100.17) and (279,130.17) .. (295,130) ;
\draw    (265,50) -- (265,70) ;
\draw [color={rgb, 255:red, 94; green, 129; blue, 172 }  ,draw opacity=1 ]   (270,65) -- (283,65) ;
\draw [shift={(285,65)}, rotate = 180] [color={rgb, 255:red, 94; green, 129; blue, 172 }  ,draw opacity=1 ][line width=0.75]    (4.37,-1.32) .. controls (2.78,-0.56) and (1.32,-0.12) .. (0,0) .. controls (1.32,0.12) and (2.78,0.56) .. (4.37,1.32)   ;
\draw [color={rgb, 255:red, 94; green, 129; blue, 172 }  ,draw opacity=1 ]   (290,100) .. controls (280.13,99.86) and (285.52,90.93) .. (271.82,90.06) ;
\draw [shift={(270,90)}, rotate = 0.6] [color={rgb, 255:red, 94; green, 129; blue, 172 }  ,draw opacity=1 ][line width=0.75]    (4.37,-1.32) .. controls (2.78,-0.56) and (1.32,-0.12) .. (0,0) .. controls (1.32,0.12) and (2.78,0.56) .. (4.37,1.32)   ;
\draw [color={rgb, 255:red, 94; green, 129; blue, 172 }  ,draw opacity=1 ]   (300,95) .. controls (315.33,94.83) and (311,64.83) .. (295,65) ;
\draw [color={rgb, 255:red, 94; green, 129; blue, 172 }  ,draw opacity=1 ]   (300,95) .. controls (309.87,95.14) and (304.48,104.07) .. (318.18,104.94) ;
\draw [shift={(320,105)}, rotate = 180.6] [color={rgb, 255:red, 94; green, 129; blue, 172 }  ,draw opacity=1 ][line width=0.75]    (4.37,-1.32) .. controls (2.78,-0.56) and (1.32,-0.12) .. (0,0) .. controls (1.32,0.12) and (2.78,0.56) .. (4.37,1.32)   ;
\draw  [draw opacity=0] (240,85) -- (265,85) -- (265,100) -- (240,100) -- cycle ;
\draw    (325,125) -- (325,145) ;
\draw  [draw opacity=0] (325,95) -- (350,95) -- (350,110) -- (325,110) -- cycle ;
\draw  [draw opacity=0] (325,125) -- (350,125) -- (350,140) -- (325,140) -- cycle ;
\draw [color={rgb, 255:red, 94; green, 129; blue, 172 }  ,draw opacity=1 ]   (320,130) -- (307,130) ;
\draw [shift={(305,130)}, rotate = 360] [color={rgb, 255:red, 94; green, 129; blue, 172 }  ,draw opacity=1 ][line width=0.75]    (4.37,-1.32) .. controls (2.78,-0.56) and (1.32,-0.12) .. (0,0) .. controls (1.32,0.12) and (2.78,0.56) .. (4.37,1.32)   ;
\draw [color={rgb, 255:red, 94; green, 129; blue, 172 }  ,draw opacity=1 ]   (232.92,65.22) .. controls (220.44,67.7) and (220.89,90) .. (235,90) ;
\draw [shift={(235,65)}, rotate = 178.88] [color={rgb, 255:red, 94; green, 129; blue, 172 }  ,draw opacity=1 ][line width=0.75]    (4.37,-1.32) .. controls (2.78,-0.56) and (1.32,-0.12) .. (0,0) .. controls (1.32,0.12) and (2.78,0.56) .. (4.37,1.32)   ;
\draw [color={rgb, 255:red, 94; green, 129; blue, 172 }  ,draw opacity=1 ]   (160,85) .. controls (173.92,85.27) and (174.8,107.83) .. (161.91,109.85) ;
\draw [shift={(160,110)}, rotate = 360] [color={rgb, 255:red, 94; green, 129; blue, 172 }  ,draw opacity=1 ][line width=0.75]    (4.37,-1.32) .. controls (2.78,-0.56) and (1.32,-0.12) .. (0,0) .. controls (1.32,0.12) and (2.78,0.56) .. (4.37,1.32)   ;
\draw  [draw opacity=0] (170,90) -- (190,90) -- (190,105) -- (170,105) -- cycle ;
\draw    (425,50) -- (425,145) ;
\draw [color={rgb, 255:red, 94; green, 129; blue, 172 }  ,draw opacity=1 ]   (431.04,85) .. controls (444.95,85.27) and (445.84,107.83) .. (432.95,109.85) ;
\draw [shift={(431.04,110)}, rotate = 360] [color={rgb, 255:red, 94; green, 129; blue, 172 }  ,draw opacity=1 ][line width=0.75]    (4.37,-1.32) .. controls (2.78,-0.56) and (1.32,-0.12) .. (0,0) .. controls (1.32,0.12) and (2.78,0.56) .. (4.37,1.32)   ;
\draw [color={rgb, 255:red, 94; green, 129; blue, 172 }  ,draw opacity=1 ]   (417.92,85.22) .. controls (405.44,87.7) and (405.89,110) .. (420,110) ;
\draw [shift={(420,85)}, rotate = 178.88] [color={rgb, 255:red, 94; green, 129; blue, 172 }  ,draw opacity=1 ][line width=0.75]    (4.37,-1.32) .. controls (2.78,-0.56) and (1.32,-0.12) .. (0,0) .. controls (1.32,0.12) and (2.78,0.56) .. (4.37,1.32)   ;
\draw  [draw opacity=0] (185,85) -- (205,85) -- (205,110) -- (185,110) -- cycle ;
\draw  [draw opacity=0] (370,85) -- (390,85) -- (390,110) -- (370,110) -- cycle ;
\draw  [draw opacity=0] (400,75) -- (450,75) -- (450,120) -- (400,120) -- cycle ;
\draw [color={rgb, 255:red, 94; green, 129; blue, 172 }  ,draw opacity=1 ]   (355,105) .. controls (368.92,105.27) and (369.8,127.83) .. (356.91,129.85) ;
\draw [shift={(355,130)}, rotate = 360] [color={rgb, 255:red, 94; green, 129; blue, 172 }  ,draw opacity=1 ][line width=0.75]    (4.37,-1.32) .. controls (2.78,-0.56) and (1.32,-0.12) .. (0,0) .. controls (1.32,0.12) and (2.78,0.56) .. (4.37,1.32)   ;
\draw  [draw opacity=0] (390,90) -- (410,90) -- (410,105) -- (390,105) -- cycle ;
\draw  [draw opacity=0] (440,90) -- (460,90) -- (460,105) -- (440,105) -- cycle ;
\draw  [draw opacity=0] (460,85) -- (480,85) -- (480,110) -- (460,110) -- cycle ;
\draw  [draw opacity=0] (135,70) -- (150,70) -- (150,85) -- (135,85) -- cycle ;
\draw  [draw opacity=0] (170,90) -- (185,90) -- (185,105) -- (170,105) -- cycle ;
\draw  [draw opacity=0] (135,110) -- (150,110) -- (150,125) -- (135,125) -- cycle ;
\draw  [draw opacity=0] (85,110) -- (100,110) -- (100,125) -- (85,125) -- cycle ;
\draw  [draw opacity=0] (85,110) -- (100,110) -- (100,125) -- (85,125) -- cycle ;
\draw  [draw opacity=0] (205,70) -- (225,70) -- (225,85) -- (205,85) -- cycle ;
\draw  [draw opacity=0] (365,110) -- (385,110) -- (385,125) -- (365,125) -- cycle ;
\draw  [draw opacity=0] (365,110) -- (380,110) -- (380,125) -- (365,125) -- cycle ;
\draw  [draw opacity=0] (270,50) -- (285,50) -- (285,65) -- (270,65) -- cycle ;
\draw  [draw opacity=0] (270,50) -- (285,50) -- (285,65) -- (270,65) -- cycle ;
\draw  [draw opacity=0] (305,130) -- (320,130) -- (320,145) -- (305,145) -- cycle ;
\draw  [draw opacity=0] (305,130) -- (320,130) -- (320,145) -- (305,145) -- cycle ;

\draw (117.5,97.5) node    {$f$};
\draw (117.5,82.5) node  [font=\small,color={rgb, 255:red, 94; green, 129; blue, 172 }  ,opacity=1 ]  {$\dotsc $};
\draw (92.5,73.5) node  [font=\small,color={rgb, 255:red, 94; green, 129; blue, 172 }  ,opacity=1 ]  {$f_{1}^{r}$};
\draw (117.5,112.5) node  [font=\small,color={rgb, 255:red, 94; green, 129; blue, 172 }  ,opacity=1 ]  {$\dotsc $};
\draw (55,97.5) node  [font=\small,color={rgb, 255:red, 94; green, 129; blue, 172 }  ,opacity=1 ]  {$f^{u}$};
\draw (265,77.5) node    {$f$};
\draw (252.5,62.5) node  [font=\small,color={rgb, 255:red, 94; green, 129; blue, 172 }  ,opacity=1 ]  {$\dotsc $};
\draw (325,117.5) node    {$g$};
\draw (252.5,92.5) node  [font=\small,color={rgb, 255:red, 94; green, 129; blue, 172 }  ,opacity=1 ]  {$\dotsc $};
\draw (337.5,102.5) node  [font=\small,color={rgb, 255:red, 94; green, 129; blue, 172 }  ,opacity=1 ]  {$\dotsc $};
\draw (337.5,132.5) node  [font=\small,color={rgb, 255:red, 94; green, 129; blue, 172 }  ,opacity=1 ]  {$\dotsc $};
\draw (195,97.5) node    {$;$};
\draw (380,97.5) node    {$;$};
\draw (400,97.5) node  [font=\small,color={rgb, 255:red, 94; green, 129; blue, 172 }  ,opacity=1 ]  {$id$};
\draw (450,97.5) node  [font=\small,color={rgb, 255:red, 94; green, 129; blue, 172 }  ,opacity=1 ]  {$id$};
\draw (470,97.5) node    {$;$};
\draw (142.5,73.5) node  [font=\small,color={rgb, 255:red, 94; green, 129; blue, 172 }  ,opacity=1 ]  {$f_{n-1}^{r}$};
\draw (180.5,97.5) node  [font=\small,color={rgb, 255:red, 94; green, 129; blue, 172 }  ,opacity=1 ]  {$f^{d}$};
\draw (142.5,120.5) node  [font=\small,color={rgb, 255:red, 94; green, 129; blue, 172 }  ,opacity=1 ]  {$f_{m-1}^{\ell}$};
\draw (92.5,120.5) node  [font=\small,color={rgb, 255:red, 94; green, 129; blue, 172 }  ,opacity=1 ]  {$f_{1}^{\ell}$};
\draw (215,77.5) node  [font=\small,color={rgb, 255:red, 94; green, 129; blue, 172 }  ,opacity=1 ]  {$f^{u}$};
\draw (374.5,117.5) node  [font=\small,color={rgb, 255:red, 94; green, 129; blue, 172 }  ,opacity=1 ]  {$f^{d}$};
\draw (277.5,53.5) node  [font=\small,color={rgb, 255:red, 94; green, 129; blue, 172 }  ,opacity=1 ]  {$f_{n-1}^{r}$};
\draw (312.5,140.5) node  [font=\small,color={rgb, 255:red, 94; green, 129; blue, 172 }  ,opacity=1 ]  {$f_{m}^{\ell}$};

\end{tikzpicture}
}
  \caption{Contour of a morphism, composition of contours, and identity contours.}
  \label{fig:contour-proautonomous}
\end{figure}

\begin{proposition}
  Contouring extends to a functor from the category of polycategories to the category of categories,
  $𝓒 \colon \mathbf{PolyCat} \to \mathbf{Cat}$.
\end{proposition}

\subsection{Malleable Polycategories}

A malleable polycategory is a polycategory where each morphism can be morphed uniquely into any possible shape. This
means that there exist unique factorizations of each morphism into each one of the possible shapes for composition.

\begin{definition}
  The (1,1)-polymorphisms of a \polycategory{} $ℙ$ form an underlying category $ℙᵘ$. 
  The polymorphisms form profunctors over the \polycategory{} and their composition, in its four possible forms, is dinatural with respect to the underlying category. 
  This means that the following four operations are well-defined:
  \begin{align*}
    (⨾)_1 &፡ {\textstyle \left( ∫^{X ∈ ℙᵘ} ℙ(Γ; Δ, X) × ℙ(X , Γ'; Δ') \right) → ℙ(Γ, Γ'; Δ, Δ'),} \\  
    (⨾)_2 &፡ {\textstyle \left( ∫^{X ∈ ℙᵘ} ℙ(Γ; X, Δ) × ℙ(Γ', X; Δ') \right) → ℙ(Γ, Γ'; Δ, Δ'),} \\  
    (⨾)_3 &፡ {\textstyle \left( ∫^{X ∈ ℙᵘ} ℙ(Γ; Δ₁, X, Δ₂) × ℙ(X; Δ) \right) → ℙ(Γ; Δ₁, Δ, Δ₂),} \\  
    (⨾)_4 &፡ {\textstyle \left( ∫^{X ∈ ℙᵘ} ℙ(Γ;X) × ℙ(Γ₁, X, Γ₂; Δ) \right) → ℙ(Γ₁, Γ, Γ₂; Δ).}
  \end{align*}
\end{definition}

\begin{definition}
  \defining{linkMalleablePolycategory}{}
  A \emph{malleable polycategory} is a \polycategory{} where dinatural composition, in all its four forms, is invertible.
\end{definition}

\begin{remark}
  If a \polycategory{} is malleable, we can reconstruct it up to isomorphism from its binary, cobinary, nullary and conullary maps.
  When defining a \malleablePolycategory{}, it is usually easier to provide these binary, cobinary, unary, nullary and conullary maps, and deduce from those the rest of the structure. The situation is now similar to that of linearly distributive categories: we do not need to provide all n-ary tensors in order to define a linearly distributive category, we only provide the binary $(⊗,⅋)$ and unary $(I,Z)$ tensors.
  
  This suggests that we will really work with a \emph{biased} version of \malleablePolycategories{}, one that privileges the binary and nullary tensors over the others. Biased malleable polycategories are what we will call \emph{prostar autonomous categories}.
\end{remark}

\subsection{Prostar-Autonomous Categories}

Prostar-autonomous categories provide an algebra for both coherent composition and decomposition.
Apart from the usual \emph{morphisms}, $𝕍(X; Y)$; and the \emph{joints}, $𝕍(X₀ ⊗ X₁ ; Y)$, and \emph{units}, $𝕍(⊤; Y)$, of a \promonoidalCategory{}; a prostar-autonomous category has \emph{splits}, $𝕍(X ; Y₀ ⅋ Y₁)$, and \emph{atoms}, $𝕍(X ; ⊥)$. As in the case of \multicategories{}, these compositions and decompositions must be coherent, which translates into the existence of natural isomorphisms witnessing a Frobenius rule, the \emph{Frobenius distributors}
\begin{align*}
  & \textstyle \varphi_l \colon ∫^{W} ℂ(A ; C ⅋ W) × ℂ(W ⊗ B ; D)
   \xrightarrow[]{\cong}
   ℂ(A ⊗ B; C ⅋ D),\mbox{ and } \\
  & \textstyle \varphi_r \colon ∫^{W} ℂ(A ⊗ W ; C) × ℂ(B ; W ⅋ D)
   \xrightarrow[]{\cong}
   ℂ(A ⊗ B; C ⅋ D).
\end{align*}

In summary, after this section, we will have developed the relation between malleability and profunctorial structures in an analogous way for both multicategories and polycategories.

\begin{figure}[h]
\begin{tabular}{c|c|c}  
  Multicategory & Malleable Multicategory & Promonoidal category \\ \hline
  Polycategory & Malleable Polycategory & Prostar autonomous category
\end{tabular}
\end{figure}

\begin{definition} 
    \defining{linkProstar}{}\label{def:prostar}
  Prostar-autonomous categories are the 2-Frobenius monoids of the monoidal bicategory of profunctors, which is equivalent to the following definition.
	A \emph{prostar autonomous category} is a category $ℂ$ endowed with a promonoidal structure $(ℂ,⊗,⊤)$, and procomonoidal structure $(ℂ,⅋,⊥)$, that interact as a Frobenius pseudomonoid \cite{daystreet04:quantum,lauda05:frobenius}.
	That is, it is a category endowed with four profunctors, suggestively written $ℂ(\bullet ⊗ \bullet ; \bullet)$, $ℂ(\top; \bullet)$, $ℂ(\bullet ; \bot)$ and $ℂ(\bullet ; \bullet \parr \bullet)$, as if they were representable. These profunctors form two promonoidal categories \cite{day} with coherent associators and unitors. Further, they are endowed with invertible Frobenius distributors,
	\begin{align*}
	  & \textstyle \varphi_l \colon ∫^{W} ℂ(A ; C ⅋ W) × ℂ(W ⊗ B ; D)
	   \xrightarrow[]{\cong}
	   ℂ(A ⊗ B; C ⅋ D), \\
	  & \textstyle \varphi_r \colon ∫^{W} ℂ(A ⊗ W ; C) × ℂ(B ; W ⅋ D)
	   \xrightarrow[]{\cong}
	   ℂ(A ⊗ B; C ⅋ D),
	\end{align*}
	such that every formal diagram formed of these distributors and promonoidal coherences commutes.
  \end{definition}

\ProstarAutonomousCategories{} have a canonical \emph{prostar} given by profunctors $ℂ(• ⊗ •; ⊥)$ and $ℂ(⊤; • \parr •)$. We may think of a \prostarAutonomousCategory{} as a category $ℂ$ equipped with sets of polymorphisms $ℂ(• ⊗ ... ⊗ • ; • \parr ... \parr •)$. The Frobenius isomorphisms let us decompose polymorphisms into combinations of the pro(co)monoidal structures: this decomposition is unique up to dinaturality. Informally, \prostarAutonomousCategories{} are to \polycategories{} what promonoidal categories are to (co)multicategories.

\begin{definition}
  \defining{linkProstarFunctor}{} \label{def:prostarfunctor}
  A \emph{prostar functor} $F \colon 𝕍 \to 𝕎$ is a quintuple $(F_{obj},F_{\otimes},F_{\parr},F_{\top},F_{\bot})$ where $(F_{obj},F_\otimes,F_\top)$ and $(F_{obj}, F_\parr, F_\bot)$ are promonoidal functors that together strictly preserve the Frobenius distributors, in that
    $\varphi_l ⨾ (F_\otimes × F_\parr) = (F_\parr × F_\otimes) ⨾ \varphi_l'$ and $\varphi_r ⨾ (F_\otimes × F_\parr) = (F_\otimes × F_\parr) ⨾ \varphi_r'$.
  Prostar functors between \prostarAutonomousCategories{} form a category, $\mathbf{ProStar}$.
\end{definition}

\subsection{Prostar Autonomous are Malleable Polycategories}

In this section, we show that the category of \prostarAutonomousCategories{} is equivalent to that of malleable polycategories. In this sense, the study of malleable polycategories is the study of prostar autonomous categories.

\begin{definition}[Polycategorical analogue of \Cref{def:underlyingMalleableMulticategory}]
  \label{def:underlyingMalleablePolycategory}
	Let $𝕎$ be a \prostarAutonomousCategory{}. There is a \malleablePolycategory{}, $𝕎^m$, that has the same objects but polymorphisms defined by the elements of the \prostarAutonomousCategory{}. By induction, we define
	\begin{align*}
	  𝕎^m(X₀,X₁,Γ; Δ) &= \textstyle{∫^V 𝕎(X₀ ⊗ X₁; V) × 𝕎^m(V,Γ;Δ)}, \\
	  𝕎^m(;Δ) &= \textstyle{∫^V 𝕎(⊤;V) × 𝕎(V;Δ)}, \\
	  𝕎^m(X;Y₀,Y₁,Δ) &= \textstyle{∫^V 𝕎^m(X;V,Δ) × 𝕎(V;Y₀ ⅋ Y₁)}, \\
	  𝕎^m(X;) &= \textstyle{𝕎(X;⊥)}.
	\end{align*}
	In other words, the polymorphisms are elements of the left-biased tree reductions of the promonoidal category, seen as a 2-monoid.
	The four forms of dinatural composition are then defined to be the unique map relating two tree expressions in a 2-Frobenius monoid, which exist uniquely by coherence,
  \begin{align*}
    (\mathrm{coh})_1 &፡ {\textstyle \left( ∫^{X ∈ 𝕎ᵐ} 𝕎ᵐ(Γ; Δ, X) × 𝕎ᵐ(X , Γ'; Δ') \right) → 𝕎ᵐ(Γ, Γ'; Δ, Δ'),} \\  
    (\mathrm{coh})_2 &፡ {\textstyle \left( ∫^{X ∈ 𝕎ᵐ} 𝕎ᵐ(Γ; X, Δ) × 𝕎ᵐ(Γ', X; Δ') \right) → 𝕎ᵐ(Γ, Γ'; Δ, Δ'),} \\  
    (\mathrm{coh})_3 &፡ {\textstyle \left( ∫^{X ∈ 𝕎ᵐ} 𝕎ᵐ(Γ; Δ₁, X, Δ₂) × 𝕎ᵐ(X; Δ) \right) → 𝕎ᵐ(Γ; Δ₁, Δ, Δ₂),} \\  
    (\mathrm{coh})_4 &፡ {\textstyle \left( ∫^{X ∈ 𝕎ᵐ} 𝕎ᵐ(Γ;X) × 𝕎ᵐ(Γ₁, X, Γ₂; Δ) \right) → 𝕎ᵐ(Γ₁, Γ, Γ₂; Δ).}
  \end{align*}
	Coherence maps are isomorphisms, and so dinatural composition is invertible, making the \polycategory{} malleable. 
	By coherence for pseudomonoids, composition must satisfy associativity and unitality.
  \end{definition}
  
  \begin{proposition}
    \label{prop:equivalenceProstarMalleable}
	  The category of \prostarAutonomousCategories{} and the category of \malleablePolycategories{} are equivalent with the functor $(•)ᵐ ፡ \ProStar{} → \mathbf{mPoly}$ induced by the construction of the underlying malleable polycategory of a \prostarAutonomousCategory{}.  This is the polycategorical analogue of \Cref{prop:equivalencePromonoidalMalleable}.
  \end{proposition}
  \begin{proof}
    First, let us show that a \prostarFunctor{}, $F ፡ 𝕍 → 𝕎$, induces a polyfunctor, $Fᵐ ፡ 𝕍ᵐ → 𝕎ᵐ$, between the Underlying polycategories. On objects, we define it to be the same, $Fᵐ_{obj} = F_{obj}$. On polymorphisms, we can define the binary, nullary, cobinary, conullary and unary using the \prostarFunctor{} structure,
    $$F_{2,1}^m = F_{⊗};\ F_{0,1}^m = F_⊤;\ F_{1,2}^m = F_{⅋};\ F_{1,0}^m = F_⊥; \mbox{ and } F_{1,1}^m = F.$$
    
  \end{proof}

\subsection{Splice of a Polycategory}
  \begin{definition}
	Let $ℂ$ be a category. Its \prostarAutonomousCategory{} of \emph{spliced arrows}, $Sℂ$, has underlying category $ℂ^{\text{op}} × ℂ$. Intuitively, its profunctors are defined by spliced circles of morphisms. %
	\begin{figure}[ht]
	  \centering
      \scalebox{0.8}{%

\tikzset{every picture/.style={line width=1.2pt}} %

\begin{tikzpicture}[x=0.75pt,y=0.75pt,yscale=-1.2,xscale=1.2]
\draw   (21,57) .. controls (21,40.43) and (34.43,27) .. (51,27) .. controls (67.57,27) and (81,40.43) .. (81,57) .. controls (81,73.57) and (67.57,87) .. (51,87) .. controls (34.43,87) and (21,73.57) .. (21,57) -- cycle ;
\draw  [fill={rgb, 255:red, 94; green, 129; blue, 172 }  ,fill opacity=0.2 ] (51,27) .. controls (54.69,27) and (58.23,27.67) .. (61.49,28.89) .. controls (60.21,32.01) and (59.5,35.42) .. (59.5,39) .. controls (59.5,51.5) and (68.15,62.02) .. (79.9,65.09) .. controls (78.26,70.93) and (74.91,76.05) .. (70.41,79.87) .. controls (65.44,75.01) and (58.58,72) .. (51,72) .. controls (43.42,72) and (36.56,75.01) .. (31.59,79.87) .. controls (27.09,76.05) and (23.74,70.93) .. (22.1,65.09) .. controls (33.85,62.02) and (42.5,51.5) .. (42.5,39) .. controls (42.5,35.42) and (41.79,32.01) .. (40.51,28.89) .. controls (43.77,27.67) and (47.31,27) .. (51,27) -- cycle ;
\draw    (61.79,49.43) -- (63.38,52.26) ;
\draw [shift={(64.36,54)}, rotate = 240.64] [color={rgb, 255:red, 0; green, 0; blue, 0 }  ][line width=0.75]    (4.37,-1.32) .. controls (2.78,-0.56) and (1.32,-0.12) .. (0,0) .. controls (1.32,0.12) and (2.78,0.56) .. (4.37,1.32)   ;
\draw  [draw opacity=0] (18,24) -- (84,24) -- (84,90) -- (18,90) -- cycle ;
\draw    (53.5,72) -- (50.61,72.16) ;
\draw [shift={(48.61,72.27)}, rotate = 356.82] [color={rgb, 255:red, 0; green, 0; blue, 0 }  ][line width=0.75]    (4.37,-1.32) .. controls (2.78,-0.56) and (1.32,-0.12) .. (0,0) .. controls (1.32,0.12) and (2.78,0.56) .. (4.37,1.32)   ;
\draw    (38.9,53.1) -- (39.72,51.14) ;
\draw [shift={(40.5,49.3)}, rotate = 112.83] [color={rgb, 255:red, 0; green, 0; blue, 0 }  ][line width=0.75]    (4.37,-1.32) .. controls (2.78,-0.56) and (1.32,-0.12) .. (0,0) .. controls (1.32,0.12) and (2.78,0.56) .. (4.37,1.32)   ;
\draw   (101.5,57) .. controls (101.5,73.57) and (114.93,87) .. (131.5,87) .. controls (148.07,87) and (161.5,73.57) .. (161.5,57) .. controls (161.5,40.43) and (148.07,27) .. (131.5,27) .. controls (114.93,27) and (101.5,40.43) .. (101.5,57) -- cycle ;
\draw  [fill={rgb, 255:red, 94; green, 129; blue, 172 }  ,fill opacity=0.2 ] (131.5,87) .. controls (135.19,87) and (138.73,86.33) .. (141.99,85.11) .. controls (140.71,81.99) and (140,78.58) .. (140,75) .. controls (140,62.5) and (148.65,51.98) .. (160.4,48.91) .. controls (158.76,43.07) and (155.41,37.95) .. (150.91,34.13) .. controls (145.94,38.99) and (139.08,42) .. (131.5,42) .. controls (123.92,42) and (117.06,38.99) .. (112.09,34.13) .. controls (107.59,37.95) and (104.24,43.07) .. (102.6,48.91) .. controls (114.35,51.98) and (123,62.5) .. (123,75) .. controls (123,78.58) and (122.29,81.99) .. (121.01,85.11) .. controls (124.27,86.33) and (127.81,87) .. (131.5,87) -- cycle ;
\draw    (144.75,59.81) -- (144.14,60.84) ;
\draw [shift={(143.13,62.56)}, rotate = 300.58] [color={rgb, 255:red, 0; green, 0; blue, 0 }  ][line width=0.75]    (4.37,-1.32) .. controls (2.78,-0.56) and (1.32,-0.12) .. (0,0) .. controls (1.32,0.12) and (2.78,0.56) .. (4.37,1.32)   ;
\draw    (131.5,42) -- (132.5,42) ;
\draw [shift={(134.5,42)}, rotate = 180] [color={rgb, 255:red, 0; green, 0; blue, 0 }  ][line width=0.75]    (4.37,-1.32) .. controls (2.78,-0.56) and (1.32,-0.12) .. (0,0) .. controls (1.32,0.12) and (2.78,0.56) .. (4.37,1.32)   ;
\draw    (120.5,63.29) -- (120.13,62.7) ;
\draw [shift={(119.07,61)}, rotate = 57.99] [color={rgb, 255:red, 0; green, 0; blue, 0 }  ][line width=0.75]    (4.37,-1.32) .. controls (2.78,-0.56) and (1.32,-0.12) .. (0,0) .. controls (1.32,0.12) and (2.78,0.56) .. (4.37,1.32)   ;
\draw   (182.5,57) .. controls (182.5,40.43) and (195.93,27) .. (212.5,27) .. controls (229.07,27) and (242.5,40.43) .. (242.5,57) .. controls (242.5,73.57) and (229.07,87) .. (212.5,87) .. controls (195.93,87) and (182.5,73.57) .. (182.5,57) -- cycle ;
\draw  [fill={rgb, 255:red, 94; green, 129; blue, 172 }  ,fill opacity=0.2 ] (212.5,27) .. controls (216.27,27) and (219.89,27.7) .. (223.21,28.97) .. controls (220.15,37.82) and (218.5,47.23) .. (218.5,57) .. controls (218.5,66.77) and (220.15,76.18) .. (223.21,85.03) .. controls (219.89,86.3) and (216.27,87) .. (212.5,87) .. controls (208.73,87) and (205.11,86.3) .. (201.79,85.03) .. controls (204.85,76.18) and (206.5,66.77) .. (206.5,57) .. controls (206.5,47.23) and (204.85,37.82) .. (201.79,28.97) .. controls (205.11,27.7) and (208.73,27) .. (212.5,27) -- cycle ;
\draw    (206.5,57) -- (206.5,56) ;
\draw [shift={(206.5,54)}, rotate = 90] [color={rgb, 255:red, 0; green, 0; blue, 0 }  ][line width=0.75]    (4.37,-1.32) .. controls (2.78,-0.56) and (1.32,-0.12) .. (0,0) .. controls (1.32,0.12) and (2.78,0.56) .. (4.37,1.32)   ;
\draw    (218.5,54) -- (218.5,55) ;
\draw [shift={(218.5,57)}, rotate = 270] [color={rgb, 255:red, 0; green, 0; blue, 0 }  ][line width=0.75]    (4.37,-1.32) .. controls (2.78,-0.56) and (1.32,-0.12) .. (0,0) .. controls (1.32,0.12) and (2.78,0.56) .. (4.37,1.32)   ;
\draw  [fill={rgb, 255:red, 94; green, 129; blue, 172 }  ,fill opacity=0.2 ] (294,27) .. controls (297.77,27) and (301.39,27.7) .. (304.71,28.97) .. controls (304.71,35.05) and (299.92,39.97) .. (294,39.97) .. controls (288.08,39.97) and (283.29,35.05) .. (283.29,28.97) .. controls (286.61,27.7) and (290.23,27) .. (294,27) -- cycle ;
\draw    (292.5,40.08) -- (293.84,40.02) ;
\draw [shift={(290.83,39.92)}, rotate = 0] [color={rgb, 255:red, 0; green, 0; blue, 0 }  ][line width=0.75]    (4.37,-1.32) .. controls (2.78,-0.56) and (1.32,-0.12) .. (0,0) .. controls (1.32,0.12) and (2.78,0.56) .. (4.37,1.32)   ;
\draw   (264,57) .. controls (264,40.43) and (277.43,27) .. (294,27) .. controls (310.57,27) and (324,40.43) .. (324,57) .. controls (324,73.57) and (310.57,87) .. (294,87) .. controls (277.43,87) and (264,73.57) .. (264,57) -- cycle ;
\draw  [fill={rgb, 255:red, 94; green, 129; blue, 172 }  ,fill opacity=0.2 ] (375,87) .. controls (378.77,87) and (382.39,86.3) .. (385.71,85.03) .. controls (385.71,78.95) and (380.92,74.03) .. (375,74.03) .. controls (369.08,74.03) and (364.29,78.95) .. (364.29,85.03) .. controls (367.61,86.3) and (371.23,87) .. (375,87) -- cycle ;
\draw    (376.17,73.75) -- (375.5,73.79) ;
\draw [shift={(377.5,73.92)}, rotate = 180] [color={rgb, 255:red, 0; green, 0; blue, 0 }  ][line width=0.75]    (4.37,-1.32) .. controls (2.78,-0.56) and (1.32,-0.12) .. (0,0) .. controls (1.32,0.12) and (2.78,0.56) .. (4.37,1.32)   ;
\draw   (345,57) .. controls (345,73.57) and (358.43,87) .. (375,87) .. controls (391.57,87) and (405,73.57) .. (405,57) .. controls (405,40.43) and (391.57,27) .. (375,27) .. controls (358.43,27) and (345,40.43) .. (345,57) -- cycle ;
\draw  [draw opacity=0][fill={rgb, 255:red, 0; green, 0; blue, 0 }  ,fill opacity=1 ] (39.01,28.89) .. controls (39.01,28.06) and (39.68,27.39) .. (40.51,27.39) .. controls (41.33,27.39) and (42.01,28.06) .. (42.01,28.89) .. controls (42.01,29.72) and (41.33,30.39) .. (40.51,30.39) .. controls (39.68,30.39) and (39.01,29.72) .. (39.01,28.89) -- cycle ;
\draw  [draw opacity=0][fill={rgb, 255:red, 0; green, 0; blue, 0 }  ,fill opacity=1 ] (59.99,28.89) .. controls (59.99,28.06) and (60.67,27.39) .. (61.49,27.39) .. controls (62.32,27.39) and (62.99,28.06) .. (62.99,28.89) .. controls (62.99,29.72) and (62.32,30.39) .. (61.49,30.39) .. controls (60.67,30.39) and (59.99,29.72) .. (59.99,28.89) -- cycle ;
\draw  [draw opacity=0][fill={rgb, 255:red, 0; green, 0; blue, 0 }  ,fill opacity=1 ] (78.4,65.09) .. controls (78.4,64.26) and (79.07,63.59) .. (79.9,63.59) .. controls (80.72,63.59) and (81.4,64.26) .. (81.4,65.09) .. controls (81.4,65.92) and (80.72,66.59) .. (79.9,66.59) .. controls (79.07,66.59) and (78.4,65.92) .. (78.4,65.09) -- cycle ;
\draw  [draw opacity=0][fill={rgb, 255:red, 0; green, 0; blue, 0 }  ,fill opacity=1 ] (68.91,79.87) .. controls (68.91,79.05) and (69.58,78.37) .. (70.41,78.37) .. controls (71.24,78.37) and (71.91,79.05) .. (71.91,79.87) .. controls (71.91,80.7) and (71.24,81.37) .. (70.41,81.37) .. controls (69.58,81.37) and (68.91,80.7) .. (68.91,79.87) -- cycle ;
\draw  [draw opacity=0][fill={rgb, 255:red, 0; green, 0; blue, 0 }  ,fill opacity=1 ] (30.09,79.87) .. controls (30.09,79.05) and (30.76,78.37) .. (31.59,78.37) .. controls (32.42,78.37) and (33.09,79.05) .. (33.09,79.87) .. controls (33.09,80.7) and (32.42,81.37) .. (31.59,81.37) .. controls (30.76,81.37) and (30.09,80.7) .. (30.09,79.87) -- cycle ;
\draw  [draw opacity=0][fill={rgb, 255:red, 0; green, 0; blue, 0 }  ,fill opacity=1 ] (20.6,65.09) .. controls (20.6,64.26) and (21.28,63.59) .. (22.1,63.59) .. controls (22.93,63.59) and (23.6,64.26) .. (23.6,65.09) .. controls (23.6,65.92) and (22.93,66.59) .. (22.1,66.59) .. controls (21.28,66.59) and (20.6,65.92) .. (20.6,65.09) -- cycle ;
\draw  [draw opacity=0][fill={rgb, 255:red, 0; green, 0; blue, 0 }  ,fill opacity=1 ] (110.59,34.13) .. controls (110.59,33.3) and (111.26,32.63) .. (112.09,32.63) .. controls (112.92,32.63) and (113.59,33.3) .. (113.59,34.13) .. controls (113.59,34.95) and (112.92,35.63) .. (112.09,35.63) .. controls (111.26,35.63) and (110.59,34.95) .. (110.59,34.13) -- cycle ;
\draw  [draw opacity=0][fill={rgb, 255:red, 0; green, 0; blue, 0 }  ,fill opacity=1 ] (149.41,34.13) .. controls (149.41,33.3) and (150.08,32.63) .. (150.91,32.63) .. controls (151.74,32.63) and (152.41,33.3) .. (152.41,34.13) .. controls (152.41,34.95) and (151.74,35.63) .. (150.91,35.63) .. controls (150.08,35.63) and (149.41,34.95) .. (149.41,34.13) -- cycle ;
\draw  [draw opacity=0][fill={rgb, 255:red, 0; green, 0; blue, 0 }  ,fill opacity=1 ] (158.9,48.91) .. controls (158.9,48.08) and (159.57,47.41) .. (160.4,47.41) .. controls (161.22,47.41) and (161.9,48.08) .. (161.9,48.91) .. controls (161.9,49.74) and (161.22,50.41) .. (160.4,50.41) .. controls (159.57,50.41) and (158.9,49.74) .. (158.9,48.91) -- cycle ;
\draw  [draw opacity=0][fill={rgb, 255:red, 0; green, 0; blue, 0 }  ,fill opacity=1 ] (140.49,85.11) .. controls (140.49,84.28) and (141.17,83.61) .. (141.99,83.61) .. controls (142.82,83.61) and (143.49,84.28) .. (143.49,85.11) .. controls (143.49,85.94) and (142.82,86.61) .. (141.99,86.61) .. controls (141.17,86.61) and (140.49,85.94) .. (140.49,85.11) -- cycle ;
\draw  [draw opacity=0][fill={rgb, 255:red, 0; green, 0; blue, 0 }  ,fill opacity=1 ] (119.51,85.11) .. controls (119.51,84.28) and (120.18,83.61) .. (121.01,83.61) .. controls (121.83,83.61) and (122.51,84.28) .. (122.51,85.11) .. controls (122.51,85.94) and (121.83,86.61) .. (121.01,86.61) .. controls (120.18,86.61) and (119.51,85.94) .. (119.51,85.11) -- cycle ;
\draw  [draw opacity=0][fill={rgb, 255:red, 0; green, 0; blue, 0 }  ,fill opacity=1 ] (101.1,48.91) .. controls (101.1,48.08) and (101.78,47.41) .. (102.6,47.41) .. controls (103.43,47.41) and (104.1,48.08) .. (104.1,48.91) .. controls (104.1,49.74) and (103.43,50.41) .. (102.6,50.41) .. controls (101.78,50.41) and (101.1,49.74) .. (101.1,48.91) -- cycle ;
\draw  [draw opacity=0][fill={rgb, 255:red, 0; green, 0; blue, 0 }  ,fill opacity=1 ] (200.29,28.97) .. controls (200.29,28.14) and (200.96,27.47) .. (201.79,27.47) .. controls (202.61,27.47) and (203.29,28.14) .. (203.29,28.97) .. controls (203.29,29.8) and (202.61,30.47) .. (201.79,30.47) .. controls (200.96,30.47) and (200.29,29.8) .. (200.29,28.97) -- cycle ;
\draw  [draw opacity=0][fill={rgb, 255:red, 0; green, 0; blue, 0 }  ,fill opacity=1 ] (221.71,28.97) .. controls (221.71,28.14) and (222.39,27.47) .. (223.21,27.47) .. controls (224.04,27.47) and (224.71,28.14) .. (224.71,28.97) .. controls (224.71,29.8) and (224.04,30.47) .. (223.21,30.47) .. controls (222.39,30.47) and (221.71,29.8) .. (221.71,28.97) -- cycle ;
\draw  [draw opacity=0][fill={rgb, 255:red, 0; green, 0; blue, 0 }  ,fill opacity=1 ] (200.29,85.03) .. controls (200.29,84.2) and (200.96,83.53) .. (201.79,83.53) .. controls (202.61,83.53) and (203.29,84.2) .. (203.29,85.03) .. controls (203.29,85.86) and (202.61,86.53) .. (201.79,86.53) .. controls (200.96,86.53) and (200.29,85.86) .. (200.29,85.03) -- cycle ;
\draw  [draw opacity=0][fill={rgb, 255:red, 0; green, 0; blue, 0 }  ,fill opacity=1 ] (221.71,85.03) .. controls (221.71,84.2) and (222.39,83.53) .. (223.21,83.53) .. controls (224.04,83.53) and (224.71,84.2) .. (224.71,85.03) .. controls (224.71,85.86) and (224.04,86.53) .. (223.21,86.53) .. controls (222.39,86.53) and (221.71,85.86) .. (221.71,85.03) -- cycle ;
\draw  [draw opacity=0][fill={rgb, 255:red, 0; green, 0; blue, 0 }  ,fill opacity=1 ] (281.79,28.97) .. controls (281.79,28.14) and (282.46,27.47) .. (283.29,27.47) .. controls (284.11,27.47) and (284.79,28.14) .. (284.79,28.97) .. controls (284.79,29.8) and (284.11,30.47) .. (283.29,30.47) .. controls (282.46,30.47) and (281.79,29.8) .. (281.79,28.97) -- cycle ;
\draw  [draw opacity=0][fill={rgb, 255:red, 0; green, 0; blue, 0 }  ,fill opacity=1 ] (303.21,28.97) .. controls (303.21,28.14) and (303.89,27.47) .. (304.71,27.47) .. controls (305.54,27.47) and (306.21,28.14) .. (306.21,28.97) .. controls (306.21,29.8) and (305.54,30.47) .. (304.71,30.47) .. controls (303.89,30.47) and (303.21,29.8) .. (303.21,28.97) -- cycle ;
\draw  [draw opacity=0][fill={rgb, 255:red, 0; green, 0; blue, 0 }  ,fill opacity=1 ] (362.79,85.03) .. controls (362.79,84.2) and (363.46,83.53) .. (364.29,83.53) .. controls (365.11,83.53) and (365.79,84.2) .. (365.79,85.03) .. controls (365.79,85.86) and (365.11,86.53) .. (364.29,86.53) .. controls (363.46,86.53) and (362.79,85.86) .. (362.79,85.03) -- cycle ;
\draw  [draw opacity=0][fill={rgb, 255:red, 0; green, 0; blue, 0 }  ,fill opacity=1 ] (384.21,85.03) .. controls (384.21,84.2) and (384.89,83.53) .. (385.71,83.53) .. controls (386.54,83.53) and (387.21,84.2) .. (387.21,85.03) .. controls (387.21,85.86) and (386.54,86.53) .. (385.71,86.53) .. controls (384.89,86.53) and (384.21,85.86) .. (384.21,85.03) -- cycle ;

\draw (38.51,25.49) node [anchor=south east] [inner sep=0.75pt]  [font=\scriptsize,color={rgb, 255:red, 0; green, 0; blue, 0 }  ,opacity=1 ]  {$X^{+}$};
\draw (63.49,25.49) node [anchor=south west] [inner sep=0.75pt]  [font=\scriptsize,color={rgb, 255:red, 0; green, 0; blue, 0 }  ,opacity=1 ]  {$X^{-}$};
\draw (81.9,68.49) node [anchor=north west][inner sep=0.75pt]  [font=\scriptsize,color={rgb, 255:red, 0; green, 0; blue, 0 }  ,opacity=1 ]  {$Z^{-}$};
\draw (72.41,83.27) node [anchor=north west][inner sep=0.75pt]  [font=\scriptsize,color={rgb, 255:red, 0; green, 0; blue, 0 }  ,opacity=1 ]  {$Z^{+}$};
\draw (20.1,68.49) node [anchor=north east] [inner sep=0.75pt]  [font=\scriptsize,color={rgb, 255:red, 0; green, 0; blue, 0 }  ,opacity=1 ]  {$Y^{+}$};
\draw (29.59,83.27) node [anchor=north east] [inner sep=0.75pt]  [font=\scriptsize,color={rgb, 255:red, 0; green, 0; blue, 0 }  ,opacity=1 ]  {$Y^{-}$};
\draw (100.6,45.51) node [anchor=south east] [inner sep=0.75pt]  [font=\scriptsize,color={rgb, 255:red, 0; green, 0; blue, 0 }  ,opacity=1 ]  {$X^{+}$};
\draw (110.09,30.73) node [anchor=south east] [inner sep=0.75pt]  [font=\scriptsize,color={rgb, 255:red, 0; green, 0; blue, 0 }  ,opacity=1 ]  {$X^{-}$};
\draw (152.91,30.73) node [anchor=south west] [inner sep=0.75pt]  [font=\scriptsize,color={rgb, 255:red, 0; green, 0; blue, 0 }  ,opacity=1 ]  {$Y^{+}$};
\draw (162.4,45.51) node [anchor=south west] [inner sep=0.75pt]  [font=\scriptsize,color={rgb, 255:red, 0; green, 0; blue, 0 }  ,opacity=1 ]  {$Y^{-}$};
\draw (119.01,88.51) node [anchor=north east] [inner sep=0.75pt]  [font=\scriptsize,color={rgb, 255:red, 0; green, 0; blue, 0 }  ,opacity=1 ]  {$Z^{+}$};
\draw (143.99,88.51) node [anchor=north west][inner sep=0.75pt]  [font=\scriptsize,color={rgb, 255:red, 0; green, 0; blue, 0 }  ,opacity=1 ]  {$Z^{-}$};
\draw (199.79,25.57) node [anchor=south east] [inner sep=0.75pt]  [font=\scriptsize,color={rgb, 255:red, 0; green, 0; blue, 0 }  ,opacity=1 ]  {$X^{+}$};
\draw (225.21,25.57) node [anchor=south west] [inner sep=0.75pt]  [font=\scriptsize,color={rgb, 255:red, 0; green, 0; blue, 0 }  ,opacity=1 ]  {$X^{-}$};
\draw (199.79,88.43) node [anchor=north east] [inner sep=0.75pt]  [font=\scriptsize,color={rgb, 255:red, 0; green, 0; blue, 0 }  ,opacity=1 ]  {$Y^{+}$};
\draw (225.21,88.43) node [anchor=north west][inner sep=0.75pt]  [font=\scriptsize,color={rgb, 255:red, 0; green, 0; blue, 0 }  ,opacity=1 ]  {$Y^{-}$};
\draw (281.29,25.57) node [anchor=south east] [inner sep=0.75pt]  [font=\scriptsize,color={rgb, 255:red, 0; green, 0; blue, 0 }  ,opacity=1 ]  {$X^{+}$};
\draw (306.71,25.57) node [anchor=south west] [inner sep=0.75pt]  [font=\scriptsize,color={rgb, 255:red, 0; green, 0; blue, 0 }  ,opacity=1 ]  {$X^{-}$};
\draw (362.29,88.43) node [anchor=north east] [inner sep=0.75pt]  [font=\scriptsize,color={rgb, 255:red, 0; green, 0; blue, 0 }  ,opacity=1 ]  {$Y^{+}$};
\draw (387.71,88.43) node [anchor=north west][inner sep=0.75pt]  [font=\scriptsize,color={rgb, 255:red, 0; green, 0; blue, 0 }  ,opacity=1 ]  {$Y^{-}$};
\draw (36.9,49.7) node [anchor=south east] [inner sep=0.75pt]  [font=\scriptsize]  {$f_{0}$};
\draw (57,85.6) node [anchor=south east] [inner sep=0.75pt]  [font=\scriptsize]  {$f_{2}$};
\draw (66.36,50.6) node [anchor=south west] [inner sep=0.75pt]  [font=\scriptsize]  {$f_{1}$};
\draw (137,37.6) node [anchor=south east] [inner sep=0.75pt]  [font=\scriptsize]  {$g_{1}$};
\draw (119,62.4) node [anchor=north east] [inner sep=0.75pt]  [font=\scriptsize]  {$g_{0}$};
\draw (145,62.4) node [anchor=north west][inner sep=0.75pt]  [font=\scriptsize]  {$g_{2}$};
\draw (193,53.4) node [anchor=north west][inner sep=0.75pt]  [font=\scriptsize]  {$h_{0}$};
\draw (220,53.4) node [anchor=north west][inner sep=0.75pt]  [font=\scriptsize]  {$h_{1}$};
\draw (289,44.4) node [anchor=north west][inner sep=0.75pt]  [font=\scriptsize]  {$k_{0}$};
\draw (370,60.4) node [anchor=north west][inner sep=0.75pt]  [font=\scriptsize]  {$l_{0}$};

\end{tikzpicture}
}
	\end{figure}

	Explicitly, it is defined by the following profunctors (below, left). The coherence isomorphisms are defined by glueing circles along the desired boundary and composing the relevant arrows; two compositions are isomorphic if and only if they determine the same arrows (below, right).

  \begin{remark}
  This structure appeared in Day \& Street \cite[Ex. 7.3]{daystreet04:quantum}, where it was noticed that the canonical promonoidal category induced by a small category \cite{day} has an involution. As a \multicategory{}, it was rediscovered by Melliès \& Zeilberger \cite{mellies22}. Monoidal spliced arrows were explicitly introduced and characterized as an adjunction in a joint work \cite{produoidal23}.
\end{remark}
\begin{figure}[ht]
  \begin{align*}
	\Splice{(ℂ)} \left( \biobj{X^+}{X^-} ;  \biobj{Y^+}{Y^-} ⅋ \biobj{Z^+}{Z^-}\right) & = 
      ℂ(Y^+; X^+) × ℂ(X^-; Z^-) × ℂ(Z^+; Y^-); \\
    \Splice{(ℂ)} \left( \biobj{X^+}{X^-} ⊗ \biobj{Y^+}{Y^-} ; \biobj{Z^+}{Z^-}\right) & = 
      ℂ(Z^+; X^+) × ℂ(X^-; Y^+) × ℂ(Y^-; Z^-); \\
	\Splice{(ℂ)} \left( \biobj{X^+}{X^-} ; \biobj{Y^+}{Y^-} \right) & = 
      ℂ(Y^+; X^+) × ℂ(X^-; Y^-); \\
	\Splice{(ℂ)} \left( \biobj{X^+}{X^-} ; \bot \right) & = 
      ℂ(X^-; X^+); \\
	\Splice{(ℂ)} \left(\top ; \biobj{Y^+}{Y^-}  \right) & = 
      ℂ(Y^+; Y^-). \\
  \end{align*}
  \label{fig:spliced-arrow-profrobenius}
\end{figure}
\begin{figure}
  \scalebox{0.80}{%

\tikzset{every picture/.style={line width=1pt}} %

\begin{tikzpicture}[x=0.75pt,y=0.75pt,yscale=-1.2,xscale=1.2]
\draw   (39,17.13) .. controls (55.57,17.13) and (69,30.56) .. (69,47.13) .. controls (69,49.98) and (68.6,52.73) .. (67.86,55.34) .. controls (73.11,50.85) and (79.92,48.13) .. (87.37,48.13) .. controls (103.94,48.13) and (117.37,61.56) .. (117.37,78.13) .. controls (117.37,94.7) and (103.94,108.13) .. (87.37,108.13) .. controls (70.81,108.13) and (57.37,94.7) .. (57.37,78.13) .. controls (57.37,75.28) and (57.77,72.52) .. (58.51,69.91) .. controls (53.27,74.41) and (46.45,77.13) .. (39,77.13) .. controls (22.43,77.13) and (9,63.69) .. (9,47.13) .. controls (9,30.56) and (22.43,17.13) .. (39,17.13) -- cycle ;
\draw  [fill={rgb, 255:red, 94; green, 129; blue, 172 }  ,fill opacity=0.2 ] (39,17.13) .. controls (42.69,17.13) and (46.23,17.79) .. (49.49,19.01) .. controls (48.21,22.14) and (47.5,25.55) .. (47.5,29.13) .. controls (47.5,41.62) and (56.15,52.14) .. (67.9,55.22) .. controls (67.88,55.26) and (67.87,55.3) .. (67.86,55.34) .. controls (67.89,55.31) and (67.93,55.29) .. (67.96,55.26) .. controls (72.94,60.12) and (79.8,63.13) .. (87.37,63.13) .. controls (94.95,63.13) and (101.81,60.12) .. (106.79,55.26) .. controls (111.28,59.08) and (114.64,64.2) .. (116.27,70.04) .. controls (104.53,73.11) and (95.87,83.63) .. (95.87,96.13) .. controls (95.87,99.71) and (96.58,103.12) .. (97.87,106.24) .. controls (94.6,107.46) and (91.07,108.13) .. (87.37,108.13) .. controls (83.68,108.13) and (80.15,107.46) .. (76.88,106.24) .. controls (78.17,103.12) and (78.87,99.71) .. (78.87,96.13) .. controls (78.87,83.63) and (70.22,73.11) .. (58.48,70.04) .. controls (58.49,70) and (58.5,69.95) .. (58.51,69.91) .. controls (58.48,69.94) and (58.45,69.97) .. (58.41,70) .. controls (53.44,65.13) and (46.58,62.13) .. (39,62.13) .. controls (31.42,62.13) and (24.56,65.13) .. (19.59,70) .. controls (15.09,66.18) and (11.74,61.06) .. (10.1,55.22) .. controls (21.85,52.14) and (30.5,41.62) .. (30.5,29.13) .. controls (30.5,25.55) and (29.79,22.14) .. (28.51,19.01) .. controls (31.77,17.79) and (35.31,17.13) .. (39,17.13) -- cycle ;
\draw    (49.79,39.55) -- (51.38,42.38) ;
\draw [shift={(52.36,44.13)}, rotate = 240.64] [color={rgb, 255:red, 0; green, 0; blue, 0 }  ][line width=0.75]    (4.37,-1.32) .. controls (2.78,-0.56) and (1.32,-0.12) .. (0,0) .. controls (1.32,0.12) and (2.78,0.56) .. (4.37,1.32)   ;
\draw  [draw opacity=0] (6,14.13) -- (72,14.13) -- (72,80.13) -- (6,80.13) -- cycle ;
\draw    (41.5,62.13) -- (38.61,62.29) ;
\draw [shift={(36.61,62.4)}, rotate = 356.82] [color={rgb, 255:red, 0; green, 0; blue, 0 }  ][line width=0.75]    (4.37,-1.32) .. controls (2.78,-0.56) and (1.32,-0.12) .. (0,0) .. controls (1.32,0.12) and (2.78,0.56) .. (4.37,1.32)   ;
\draw    (26.9,43.22) -- (27.72,41.27) ;
\draw [shift={(28.5,39.43)}, rotate = 112.83] [color={rgb, 255:red, 0; green, 0; blue, 0 }  ][line width=0.75]    (4.37,-1.32) .. controls (2.78,-0.56) and (1.32,-0.12) .. (0,0) .. controls (1.32,0.12) and (2.78,0.56) .. (4.37,1.32)   ;
\draw  [draw opacity=0][fill={rgb, 255:red, 0; green, 0; blue, 0 }  ,fill opacity=1 ] (27.01,19.01) .. controls (27.01,18.18) and (27.68,17.51) .. (28.51,17.51) .. controls (29.33,17.51) and (30.01,18.18) .. (30.01,19.01) .. controls (30.01,19.84) and (29.33,20.51) .. (28.51,20.51) .. controls (27.68,20.51) and (27.01,19.84) .. (27.01,19.01) -- cycle ;
\draw  [draw opacity=0][fill={rgb, 255:red, 0; green, 0; blue, 0 }  ,fill opacity=1 ] (47.99,19.01) .. controls (47.99,18.18) and (48.67,17.51) .. (49.49,17.51) .. controls (50.32,17.51) and (50.99,18.18) .. (50.99,19.01) .. controls (50.99,19.84) and (50.32,20.51) .. (49.49,20.51) .. controls (48.67,20.51) and (47.99,19.84) .. (47.99,19.01) -- cycle ;
\draw  [draw opacity=0][fill={rgb, 255:red, 0; green, 0; blue, 0 }  ,fill opacity=1 ] (66.4,55.22) .. controls (66.4,54.39) and (67.07,53.72) .. (67.9,53.72) .. controls (68.72,53.72) and (69.4,54.39) .. (69.4,55.22) .. controls (69.4,56.04) and (68.72,56.72) .. (67.9,56.72) .. controls (67.07,56.72) and (66.4,56.04) .. (66.4,55.22) -- cycle ;
\draw  [draw opacity=0][fill={rgb, 255:red, 0; green, 0; blue, 0 }  ,fill opacity=1 ] (56.91,70) .. controls (56.91,69.17) and (57.58,68.5) .. (58.41,68.5) .. controls (59.24,68.5) and (59.91,69.17) .. (59.91,70) .. controls (59.91,70.83) and (59.24,71.5) .. (58.41,71.5) .. controls (57.58,71.5) and (56.91,70.83) .. (56.91,70) -- cycle ;
\draw  [draw opacity=0][fill={rgb, 255:red, 0; green, 0; blue, 0 }  ,fill opacity=1 ] (18.09,70) .. controls (18.09,69.17) and (18.76,68.5) .. (19.59,68.5) .. controls (20.42,68.5) and (21.09,69.17) .. (21.09,70) .. controls (21.09,70.83) and (20.42,71.5) .. (19.59,71.5) .. controls (18.76,71.5) and (18.09,70.83) .. (18.09,70) -- cycle ;
\draw  [draw opacity=0][fill={rgb, 255:red, 0; green, 0; blue, 0 }  ,fill opacity=1 ] (8.6,55.22) .. controls (8.6,54.39) and (9.28,53.72) .. (10.1,53.72) .. controls (10.93,53.72) and (11.6,54.39) .. (11.6,55.22) .. controls (11.6,56.04) and (10.93,56.72) .. (10.1,56.72) .. controls (9.28,56.72) and (8.6,56.04) .. (8.6,55.22) -- cycle ;
\draw    (100.62,80.94) -- (100.02,81.97) ;
\draw [shift={(99,83.69)}, rotate = 300.58] [color={rgb, 255:red, 0; green, 0; blue, 0 }  ][line width=0.75]    (4.37,-1.32) .. controls (2.78,-0.56) and (1.32,-0.12) .. (0,0) .. controls (1.32,0.12) and (2.78,0.56) .. (4.37,1.32)   ;
\draw    (87.37,63.13) -- (88.37,63.13) ;
\draw [shift={(90.37,63.13)}, rotate = 180] [color={rgb, 255:red, 0; green, 0; blue, 0 }  ][line width=0.75]    (4.37,-1.32) .. controls (2.78,-0.56) and (1.32,-0.12) .. (0,0) .. controls (1.32,0.12) and (2.78,0.56) .. (4.37,1.32)   ;
\draw    (76.37,84.42) -- (76.01,83.83) ;
\draw [shift={(74.95,82.13)}, rotate = 57.99] [color={rgb, 255:red, 0; green, 0; blue, 0 }  ][line width=0.75]    (4.37,-1.32) .. controls (2.78,-0.56) and (1.32,-0.12) .. (0,0) .. controls (1.32,0.12) and (2.78,0.56) .. (4.37,1.32)   ;
\draw  [draw opacity=0][fill={rgb, 255:red, 0; green, 0; blue, 0 }  ,fill opacity=1 ] (66.46,55.26) .. controls (66.46,54.43) and (67.13,53.76) .. (67.96,53.76) .. controls (68.79,53.76) and (69.46,54.43) .. (69.46,55.26) .. controls (69.46,56.08) and (68.79,56.76) .. (67.96,56.76) .. controls (67.13,56.76) and (66.46,56.08) .. (66.46,55.26) -- cycle ;
\draw  [draw opacity=0][fill={rgb, 255:red, 0; green, 0; blue, 0 }  ,fill opacity=1 ] (105.29,55.26) .. controls (105.29,54.43) and (105.96,53.76) .. (106.79,53.76) .. controls (107.61,53.76) and (108.29,54.43) .. (108.29,55.26) .. controls (108.29,56.08) and (107.61,56.76) .. (106.79,56.76) .. controls (105.96,56.76) and (105.29,56.08) .. (105.29,55.26) -- cycle ;
\draw  [draw opacity=0][fill={rgb, 255:red, 0; green, 0; blue, 0 }  ,fill opacity=1 ] (114.77,70.04) .. controls (114.77,69.21) and (115.44,68.54) .. (116.27,68.54) .. controls (117.1,68.54) and (117.77,69.21) .. (117.77,70.04) .. controls (117.77,70.87) and (117.1,71.54) .. (116.27,71.54) .. controls (115.44,71.54) and (114.77,70.87) .. (114.77,70.04) -- cycle ;
\draw  [draw opacity=0][fill={rgb, 255:red, 0; green, 0; blue, 0 }  ,fill opacity=1 ] (96.37,106.24) .. controls (96.37,105.42) and (97.04,104.74) .. (97.87,104.74) .. controls (98.7,104.74) and (99.37,105.42) .. (99.37,106.24) .. controls (99.37,107.07) and (98.7,107.74) .. (97.87,107.74) .. controls (97.04,107.74) and (96.37,107.07) .. (96.37,106.24) -- cycle ;
\draw  [draw opacity=0][fill={rgb, 255:red, 0; green, 0; blue, 0 }  ,fill opacity=1 ] (75.38,106.24) .. controls (75.38,105.42) and (76.05,104.74) .. (76.88,104.74) .. controls (77.71,104.74) and (78.38,105.42) .. (78.38,106.24) .. controls (78.38,107.07) and (77.71,107.74) .. (76.88,107.74) .. controls (76.05,107.74) and (75.38,107.07) .. (75.38,106.24) -- cycle ;
\draw  [draw opacity=0][fill={rgb, 255:red, 0; green, 0; blue, 0 }  ,fill opacity=1 ] (56.98,70.04) .. controls (56.98,69.21) and (57.65,68.54) .. (58.48,68.54) .. controls (59.31,68.54) and (59.98,69.21) .. (59.98,70.04) .. controls (59.98,70.87) and (59.31,71.54) .. (58.48,71.54) .. controls (57.65,71.54) and (56.98,70.87) .. (56.98,70.04) -- cycle ;
\draw  [draw opacity=0] (6,14.13) -- (120,14.13) -- (120,110.13) -- (6,110.13) -- cycle ;
\draw   (198,89.13) .. controls (198,105.69) and (184.57,119.13) .. (168,119.13) .. controls (151.43,119.13) and (138,105.69) .. (138,89.13) .. controls (138,76.32) and (146.02,65.39) .. (157.31,61.09) .. controls (145.95,56.81) and (137.88,45.85) .. (137.88,33) .. controls (137.88,16.43) and (151.31,3) .. (167.88,3) .. controls (184.44,3) and (197.88,16.43) .. (197.88,33) .. controls (197.88,45.8) and (189.86,56.73) .. (178.57,61.04) .. controls (189.92,65.31) and (198,76.28) .. (198,89.13) -- cycle ;
\draw  [fill={rgb, 255:red, 94; green, 129; blue, 172 }  ,fill opacity=0.2 ] (178.49,61.01) .. controls (177.21,64.14) and (176.5,67.55) .. (176.5,71.13) .. controls (176.5,83.62) and (185.15,94.14) .. (196.9,97.22) .. controls (195.26,103.06) and (191.91,108.18) .. (187.41,112) .. controls (182.44,107.13) and (175.58,104.13) .. (168,104.13) .. controls (160.42,104.13) and (153.56,107.13) .. (148.59,112) .. controls (144.09,108.18) and (140.74,103.06) .. (139.1,97.22) .. controls (150.85,94.14) and (159.5,83.62) .. (159.5,71.13) .. controls (159.5,67.62) and (158.82,64.26) .. (157.58,61.19) .. controls (157.51,61.16) and (157.45,61.14) .. (157.38,61.11) .. controls (158.67,57.99) and (159.38,54.58) .. (159.38,51) .. controls (159.38,38.5) and (150.72,27.98) .. (138.98,24.91) .. controls (140.61,19.07) and (143.97,13.95) .. (148.46,10.13) .. controls (153.44,14.99) and (160.3,18) .. (167.88,18) .. controls (175.45,18) and (182.31,14.99) .. (187.29,10.13) .. controls (191.78,13.95) and (195.14,19.07) .. (196.77,24.91) .. controls (185.03,27.98) and (176.38,38.5) .. (176.38,51) .. controls (176.38,54.51) and (177.06,57.86) .. (178.3,60.94) .. controls (178.36,60.96) and (178.43,60.99) .. (178.49,61.01) -- cycle ;
\draw    (178.79,81.55) -- (180.38,84.38) ;
\draw [shift={(181.36,86.13)}, rotate = 240.64] [color={rgb, 255:red, 0; green, 0; blue, 0 }  ][line width=0.75]    (4.37,-1.32) .. controls (2.78,-0.56) and (1.32,-0.12) .. (0,0) .. controls (1.32,0.12) and (2.78,0.56) .. (4.37,1.32)   ;
\draw  [draw opacity=0] (135,56.13) -- (201,56.13) -- (201,122.13) -- (135,122.13) -- cycle ;
\draw    (170.5,104.13) -- (167.61,104.29) ;
\draw [shift={(165.61,104.4)}, rotate = 356.82] [color={rgb, 255:red, 0; green, 0; blue, 0 }  ][line width=0.75]    (4.37,-1.32) .. controls (2.78,-0.56) and (1.32,-0.12) .. (0,0) .. controls (1.32,0.12) and (2.78,0.56) .. (4.37,1.32)   ;
\draw    (155.9,85.23) -- (156.72,83.27) ;
\draw [shift={(157.5,81.43)}, rotate = 112.83] [color={rgb, 255:red, 0; green, 0; blue, 0 }  ][line width=0.75]    (4.37,-1.32) .. controls (2.78,-0.56) and (1.32,-0.12) .. (0,0) .. controls (1.32,0.12) and (2.78,0.56) .. (4.37,1.32)   ;
\draw  [draw opacity=0][fill={rgb, 255:red, 0; green, 0; blue, 0 }  ,fill opacity=1 ] (156.01,61.01) .. controls (156.01,60.18) and (156.68,59.51) .. (157.51,59.51) .. controls (158.33,59.51) and (159.01,60.18) .. (159.01,61.01) .. controls (159.01,61.84) and (158.33,62.51) .. (157.51,62.51) .. controls (156.68,62.51) and (156.01,61.84) .. (156.01,61.01) -- cycle ;
\draw  [draw opacity=0][fill={rgb, 255:red, 0; green, 0; blue, 0 }  ,fill opacity=1 ] (176.99,61.01) .. controls (176.99,60.18) and (177.67,59.51) .. (178.49,59.51) .. controls (179.32,59.51) and (179.99,60.18) .. (179.99,61.01) .. controls (179.99,61.84) and (179.32,62.51) .. (178.49,62.51) .. controls (177.67,62.51) and (176.99,61.84) .. (176.99,61.01) -- cycle ;
\draw  [draw opacity=0][fill={rgb, 255:red, 0; green, 0; blue, 0 }  ,fill opacity=1 ] (195.4,97.22) .. controls (195.4,96.39) and (196.07,95.72) .. (196.9,95.72) .. controls (197.72,95.72) and (198.4,96.39) .. (198.4,97.22) .. controls (198.4,98.04) and (197.72,98.72) .. (196.9,98.72) .. controls (196.07,98.72) and (195.4,98.04) .. (195.4,97.22) -- cycle ;
\draw  [draw opacity=0][fill={rgb, 255:red, 0; green, 0; blue, 0 }  ,fill opacity=1 ] (185.91,112) .. controls (185.91,111.17) and (186.58,110.5) .. (187.41,110.5) .. controls (188.24,110.5) and (188.91,111.17) .. (188.91,112) .. controls (188.91,112.83) and (188.24,113.5) .. (187.41,113.5) .. controls (186.58,113.5) and (185.91,112.83) .. (185.91,112) -- cycle ;
\draw  [draw opacity=0][fill={rgb, 255:red, 0; green, 0; blue, 0 }  ,fill opacity=1 ] (147.09,112) .. controls (147.09,111.17) and (147.76,110.5) .. (148.59,110.5) .. controls (149.42,110.5) and (150.09,111.17) .. (150.09,112) .. controls (150.09,112.83) and (149.42,113.5) .. (148.59,113.5) .. controls (147.76,113.5) and (147.09,112.83) .. (147.09,112) -- cycle ;
\draw  [draw opacity=0][fill={rgb, 255:red, 0; green, 0; blue, 0 }  ,fill opacity=1 ] (137.6,97.22) .. controls (137.6,96.39) and (138.28,95.72) .. (139.1,95.72) .. controls (139.93,95.72) and (140.6,96.39) .. (140.6,97.22) .. controls (140.6,98.04) and (139.93,98.72) .. (139.1,98.72) .. controls (138.28,98.72) and (137.6,98.04) .. (137.6,97.22) -- cycle ;
\draw    (181.13,35.81) -- (180.52,36.84) ;
\draw [shift={(179.5,38.56)}, rotate = 300.58] [color={rgb, 255:red, 0; green, 0; blue, 0 }  ][line width=0.75]    (4.37,-1.32) .. controls (2.78,-0.56) and (1.32,-0.12) .. (0,0) .. controls (1.32,0.12) and (2.78,0.56) .. (4.37,1.32)   ;
\draw    (167.88,18) -- (168.88,18) ;
\draw [shift={(170.88,18)}, rotate = 180] [color={rgb, 255:red, 0; green, 0; blue, 0 }  ][line width=0.75]    (4.37,-1.32) .. controls (2.78,-0.56) and (1.32,-0.12) .. (0,0) .. controls (1.32,0.12) and (2.78,0.56) .. (4.37,1.32)   ;
\draw    (156.88,39.29) -- (156.51,38.7) ;
\draw [shift={(155.45,37)}, rotate = 57.99] [color={rgb, 255:red, 0; green, 0; blue, 0 }  ][line width=0.75]    (4.37,-1.32) .. controls (2.78,-0.56) and (1.32,-0.12) .. (0,0) .. controls (1.32,0.12) and (2.78,0.56) .. (4.37,1.32)   ;
\draw  [draw opacity=0][fill={rgb, 255:red, 0; green, 0; blue, 0 }  ,fill opacity=1 ] (146.96,10.13) .. controls (146.96,9.3) and (147.64,8.63) .. (148.46,8.63) .. controls (149.29,8.63) and (149.96,9.3) .. (149.96,10.13) .. controls (149.96,10.95) and (149.29,11.63) .. (148.46,11.63) .. controls (147.64,11.63) and (146.96,10.95) .. (146.96,10.13) -- cycle ;
\draw  [draw opacity=0][fill={rgb, 255:red, 0; green, 0; blue, 0 }  ,fill opacity=1 ] (185.79,10.13) .. controls (185.79,9.3) and (186.46,8.63) .. (187.29,8.63) .. controls (188.11,8.63) and (188.79,9.3) .. (188.79,10.13) .. controls (188.79,10.95) and (188.11,11.63) .. (187.29,11.63) .. controls (186.46,11.63) and (185.79,10.95) .. (185.79,10.13) -- cycle ;
\draw  [draw opacity=0][fill={rgb, 255:red, 0; green, 0; blue, 0 }  ,fill opacity=1 ] (195.27,24.91) .. controls (195.27,24.08) and (195.94,23.41) .. (196.77,23.41) .. controls (197.6,23.41) and (198.27,24.08) .. (198.27,24.91) .. controls (198.27,25.74) and (197.6,26.41) .. (196.77,26.41) .. controls (195.94,26.41) and (195.27,25.74) .. (195.27,24.91) -- cycle ;
\draw  [draw opacity=0][fill={rgb, 255:red, 0; green, 0; blue, 0 }  ,fill opacity=1 ] (176.87,61.11) .. controls (176.87,60.28) and (177.54,59.61) .. (178.37,59.61) .. controls (179.2,59.61) and (179.87,60.28) .. (179.87,61.11) .. controls (179.87,61.94) and (179.2,62.61) .. (178.37,62.61) .. controls (177.54,62.61) and (176.87,61.94) .. (176.87,61.11) -- cycle ;
\draw  [draw opacity=0][fill={rgb, 255:red, 0; green, 0; blue, 0 }  ,fill opacity=1 ] (155.88,61.11) .. controls (155.88,60.28) and (156.55,59.61) .. (157.38,59.61) .. controls (158.21,59.61) and (158.88,60.28) .. (158.88,61.11) .. controls (158.88,61.94) and (158.21,62.61) .. (157.38,62.61) .. controls (156.55,62.61) and (155.88,61.94) .. (155.88,61.11) -- cycle ;
\draw  [draw opacity=0][fill={rgb, 255:red, 0; green, 0; blue, 0 }  ,fill opacity=1 ] (137.48,24.91) .. controls (137.48,24.08) and (138.15,23.41) .. (138.98,23.41) .. controls (139.81,23.41) and (140.48,24.08) .. (140.48,24.91) .. controls (140.48,25.74) and (139.81,26.41) .. (138.98,26.41) .. controls (138.15,26.41) and (137.48,25.74) .. (137.48,24.91) -- cycle ;
\draw  [draw opacity=0] (117,47.13) -- (144,47.13) -- (144,74.13) -- (117,74.13) -- cycle ;

\draw (24.9,39.82) node [anchor=south east] [inner sep=0.75pt]  [font=\scriptsize]  {$f_{0}$};
\draw (45,74.72) node [anchor=south east] [inner sep=0.75pt]  [font=\scriptsize]  {$f_{2}$};
\draw (54.36,40.72) node [anchor=south west] [inner sep=0.75pt]  [font=\scriptsize]  {$f_{1}$};
\draw (92.87,58.73) node [anchor=south east] [inner sep=0.75pt]  [font=\scriptsize]  {$g_{1}$};
\draw (74.87,83.53) node [anchor=north east] [inner sep=0.75pt]  [font=\scriptsize]  {$g_{0}$};
\draw (100.87,83.53) node [anchor=north west][inner sep=0.75pt]  [font=\scriptsize]  {$g_{2}$};
\draw (153.9,81.82) node [anchor=south east] [inner sep=0.75pt]  [font=\scriptsize]  {$k_{0}$};
\draw (174,116.73) node [anchor=south east] [inner sep=0.75pt]  [font=\scriptsize]  {$k_{2}$};
\draw (183.36,82.73) node [anchor=south west] [inner sep=0.75pt]  [font=\scriptsize]  {$k_{1}$};
\draw (173.38,13.6) node [anchor=south east] [inner sep=0.75pt]  [font=\scriptsize]  {$h_{1}$};
\draw (155.38,38.4) node [anchor=north east] [inner sep=0.75pt]  [font=\scriptsize]  {$h_{0}$};
\draw (181.38,38.4) node [anchor=north west][inner sep=0.75pt]  [font=\scriptsize]  {$h_{2}$};
\draw (130.5,60.63) node  [font=\small]  {$\overset{\phi_1}\sim $};

\end{tikzpicture}
}
  \caption{Holds if $f₀ = k₀ ⨾ h₀$, $f₁ ⨾ g₁ = h₁$, $g₂ = h₂ ⨾ k₁$, $g₀ ⨾ f₂ = k₂$.}
\end{figure}
\end{definition}
\begin{remark}
$\Splice{(ℂ)}$ has a representable prostar, given on objects by 
$$\left( \biobj{X^+}{X^-} \right)^{*} = \left( \biobj{X^-}{X^+} \right).$$
\end{remark}

\begin{proposition}
  Spliced arrows extend to a functor, $\Splice{} : \mathbf{Cat}{} \to \ProStarAut{}$.
\end{proposition}

\begin{theorem}
    Contour extends to a functor $\Contour{} \colon \mathbf{PolyCat} \to \Cat{}$, splice extends to a functor $\Splice \colon \Cat{} \to \ProStarAut{}$. $\Contour{}$ is left adjoint to $\Splice{}$ composed with the forgetful functor, $\Contour{} \dashv \Splice{} ⨾ \mathsf{Forget}$; and $\Contour{}$ composed with the forgetful functor is left adjoint to $\Splice{}$, meaning $\mathsf{Forget} ⨾ \Contour{} \dashv \Splice{}$.
\end{theorem}
\begin{proof}
  The proof extends our previous one \cite[Theorem 3.7]{produoidal23}.
\end{proof}

\subsection{Bibliography}

\Polycategories{} were defined by Szabo \cite{szabo75:polycategories} in the symmetric case; Cockett and Seely contributed the planar version we study here \cite{cockett1997,blanco20:polycategories}.

Street \cite{street04:frobeniusmonads} prove that Frobenius pseudomonoids in $\mathbf{Prof}$ are equivalent to what Day \& Street  \cite{daystreet04:quantum} call ``$\ast$-autonomous promonoidal categories''. The minor twist we take, ``prostar autonomous'', emphasizes that the canonical prostar may not be representable. When all of the structure including the prostar is representable, we obtain $\ast$-autonomous categories.

\newpage
\clearpage{}%

\end{document}